\documentclass{amsbook}%
\usepackage{amssymb}
\usepackage{amsmath}
\usepackage{amsfonts}
\usepackage{graphicx}%
\providecommand{\U}[1]{\protect\rule{.1in}{.1in}}
\theoremstyle{plain}

\newtheorem{corollary}{Corollary}

\newtheorem{definition}{Definition}
\newtheorem{example}{Example}

\newtheorem{lemma}{Lemma}
\newtheorem{notation}{Notation}

\newtheorem{proposition}{Proposition}
\newtheorem{remark}{Remark}

\newtheorem{theorem}{Theorem}
\numberwithin{equation}{section}
\begin{document}
\frontmatter
\title[Periodic binary signals and flows]{Binary periodic signals and flows}
\author[Serban E. Vlad]{Serban E. Vlad}
\address{Primaria Oradea, P-ta Unirii, Nr. 1\\
410100, Oradea, Romania}
\curraddr{Str. Zimbrului, Nr. 3, Ap. 11 \\
410430, Oradea, Romania}
\email{serban\_e\_vlad@yahoo.com}
\urladdr{http://www.serbanvlad.ro}
\author{}
\thanks{}
\dedicatory{ }\maketitle
\tableofcontents

\chapter*{Preface}

The boolean autonomous deterministic regular asynchronous systems have been
defined for the first time in our work \textit{Boolean dynamical systems},
ROMAI Journal, Vol. 3, Nr. 2, 2007, pp 277-324 and a deeper study of such
systems can be found in \cite{bib10}. The concept has its origin in switching
theory, the theory of modeling the switching circuits from the digital
electrical engineering. The attribute boolean vaguely refers to the Boole
algebra with two elements; autonomous means that there is no input;
determinism means the existence of a unique (state) function; and regular
indicates the existence of a function $\Phi:\{0,1\}^{n}\rightarrow
\{0,1\}^{n},\Phi=(\Phi_{1},...,\Phi_{n})$ that 'generates' the system. Time is
discrete: $\{-1,0,1,...\}$ or continuous: $\mathbf{R}$. The system, which is
analogue to the (real, usual) dynamical systems, iterates (asynchronously) on
each coordinate $i\in\{1,...,n\},$ one of

- $\Phi_{i}:$ we say that $\Phi$ is computed, at that time instant, on that coordinate;

- $\{0,1\}^{n}\ni(\mu_{1},...,\mu_{i},...,\mu_{n})\longmapsto\mu_{i}%
\in\{0,1\}:$ we use to say that $\Phi$ is not computed, at that time instant,
on that coordinate.

The flows are these that result by analogy with the dynamical systems.

The 'nice' discrete time and real time functions that the (boolean)
asynchronous systems work with are called signals and periodicity is a very
important feature in Nature.

In the first two Chapters we give the most important concepts concerning the
signals and periodicity. The periodicity properties are used to characterize
the eventually constant signals in Chapter 3 and the constant signals in
Chapter 4. Chapters 5,...,8 are dedicated to the eventually periodic points,
eventually periodic signals, periodic points and periodic signals.

Chapter 9 shows constructions that, given an (eventually) periodic point, by
changing some values of the signal, change the periodicity properties of the point.

The monograph continues with flows. Chapter 10 is dedicated to the computation
functions, i.e. to the functions that show when and how the function $\Phi$ is
iterated (asynchronously). Chapter 11 introduces the flows and Chapter 12
gives a wider point of view on the flows, which are interpreted as
deterministic asynchronous systems. Chapters 13,...,18 restate the topics from
Chapters 3,...,8 in the special case when the signals are flows and the main
interest is periodicity.

In order to point out our source of inspiration, we give the example of the
circuit from Figure \ref{fig31},
\begin{figure}
[ptb]
\begin{center}
\fbox{\includegraphics[
height=4.3267in,
width=4.0439in
]%
{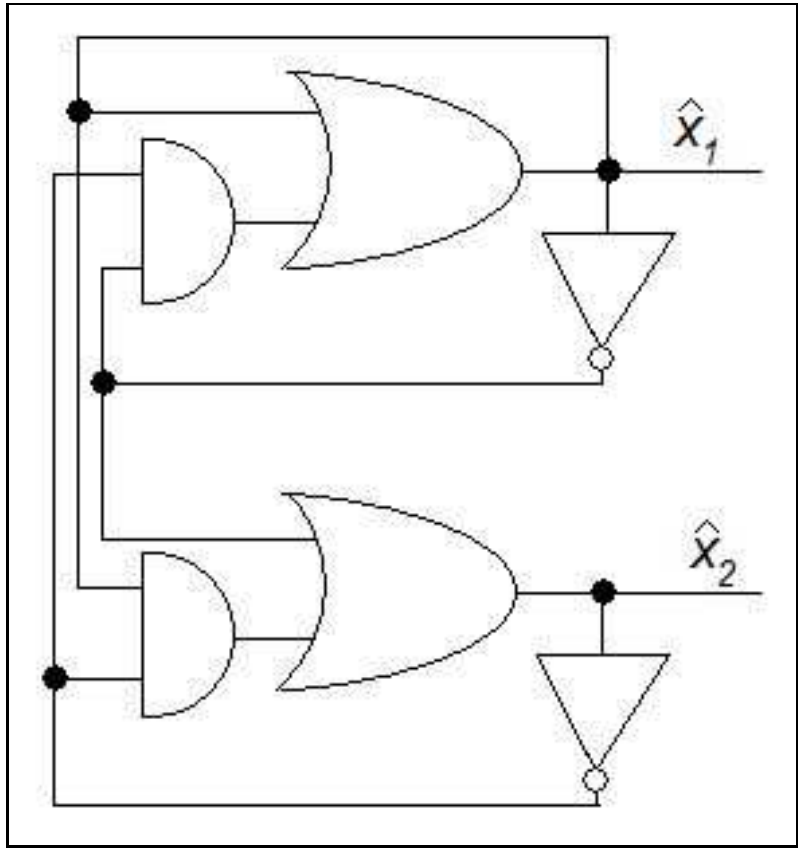}%
}\caption{Asynchronous circuit.}%
\label{fig31}%
\end{center}
\end{figure}
where $\widehat{x}:\{-1,0,1,...\}\longrightarrow\{0,1\}^{2}$ is the signal
representing the state of the system, and the initial state is $(0,0).$ The
function that generates the system is $\Phi:\{0,1\}^{2}\longrightarrow
\{0,1\}^{2},$ $\forall\mu\in\{0,1\}^{2},$%
\[
\Phi(\mu)=(\mu_{1}\cup\overline{\mu_{1}}\cdot\overline{\mu_{2}},\overline
{\mu_{1}}\cup\mu_{1}\cdot\overline{\mu_{2}}).
\]
The evolution of the system is given by its state diagram from Figure
\ref{fig1_21},%
\begin{figure}
[ptb]
\begin{center}
\fbox{\includegraphics[
height=0.8466in,
width=1.1294in
]%
{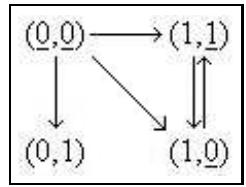}%
}\caption{The state diagram of the circuit from Figure \ref{fig31}.}%
\label{fig1_21}%
\end{center}
\end{figure}
where the arrows indicate the time increase and we have underlined these
coordinates $\mu_{i},i=\overline{1,2}$ that, by the computation of $\Phi,$
change their value: $\Phi_{i}(\mu)=\overline{\mu_{i}}.$ Let $\alpha
:\{0,1,2,...\}\longrightarrow\{0,1\}^{2}$ be the computation function whose
values $\alpha_{i}^{k}$ show that $\Phi_{i}$ is computed at the time instant
$k$ if $\alpha_{i}^{k}=1,$ respectively that it is not computed at the time
instant $k$ if $\alpha_{i}^{k}=0,$ where $i=\overline{1,2}$ and $k\in
\{0,1,2,...\}.$ The uncertainty related with the circuit, depending in general
on the technology, the temperature, etc. manifests in the fact that the order
and the time of computation of each coordinate function $\Phi_{i}$ are not
known. If the second coordinate is computed at the time instant $0,$ then
$\alpha^{0}=(0,1)$ indicates the transfer from $(0,0)$ to $(0,1),$ where the
system remains indefinitely long for any values of $\alpha^{1},\alpha
^{2},\alpha^{3},...$, since $\Phi(0,1)=(0,1).$ Such a signal $\widehat{x}$ is
called eventually constant and it corresponds to a stable system. The
eventually constant discrete time signals are eventually periodic with an
arbitrary period $p\geq1.$

Another possibility is that the first coordinate of $\Phi$ is computed at the
time instant $0,$ thus $\alpha^{0}=(1,0).$ Figure \ref{fig1_2} indicates the
transfer from $(0,0)$ in $(1,0),$ while $\alpha^{0}=(1,1)$ indicates the
transfer from $(0,0)$ to $(1,1),$ as resulted by the simultaneous computation
of $\Phi_{1}(0,0)$ and $\Phi_{2}(0,0).$ And if $\alpha^{k}=(1,1),k\in
\{0,1,2,...\},$ then $\widehat{x}$ is eventually periodic with the period
$p\in\{2,4,6,...\},$ as it switches from $(1,1)$ to $(1,0)$ and from $(1,0)$
to $(1,1)$. This last possibility represents an unstable system.

The bibliography consists in works of (real, usual) dynamical systems and we
use analogies.

The book ends with a list of notations, an index of notions and an appendix
with Lemmas. These Lemmas are frequently used in the exposure and some of them
are interesting by themselves.

The book is structured in Chapters, each Chapter consists in several Sections
and each Section is structured in paragraphs. The Chapters begin with an
abstract. The paragraphs are of the following kinds: Definitions, Notations,
Remarks, Theorems, Corollaries, Lemmas, Examples and Propositions. Each kind
of paragraph is numbered separately on the others. Inside the paragraphs, the
equations and, more generally, the most important statements are numbered
also. When we refer to the statement $(x,y)$ this means the $y-th$ statement
of the $x-th$ Section of the current Chapter.

We refer to a Definition, Theorem, Example,... by indicating its number and,
when necessary, its page. When we refer to the statement $(x,y)$ we indicate
sometimes the page where it occurs as an inferior index.

The book addresses to researchers in systems theory and computer science, but
it is also interesting to those that study periodicity itself. From this last
perspective, the binary signals may be thought of as functions with finitely
many values.

\mainmatter

\chapter{\label{Cha1}Preliminaries}

The signals from digital electrical engineering ale modeled by 'nice' discrete
time and real time functions, also called signals and their introduction is
the purpose of this Chapter. We define the left and the right limits of the
real time signals, the initial and the final values of the signals, the
initial and the final time of the signals, the forgetful function and finally
we define the orbits, the omega limit sets and the support sets.

\section{The definition of the signals}

\begin{notation}
\label{Not1}We denote by $\mathbf{B}=\{0,1\}$ the binary Boole algebra. Its
laws are the usual ones:%
\begin{align*}
&
\begin{array}
[c]{cc}%
\overline{\ \ \ } & \\
0 & 1\\
1 & 0
\end{array}
,\
\begin{array}
[c]{ccc}%
\cdot & 0 & 1\\
0 & 0 & 0\\
1 & 0 & 1
\end{array}
,\
\begin{array}
[c]{ccc}%
\cup & 0 & 1\\
0 & 0 & 1\\
1 & 1 & 1
\end{array}
,\
\begin{array}
[c]{ccc}%
\oplus & 0 & 1\\
0 & 0 & 1\\
1 & 1 & 0
\end{array}
\\
&
\;\;\;\;\;\;\;\;\;\;\;\;\;\;\;\;\;\;\;\;\;\;\;\;\;\;\;\;\;\;\;\;\;\;\;\;Table\;1
\end{align*}
and they induce laws that are denoted with the same symbols on $\mathbf{B}%
^{n},n\geq1.$
\end{notation}

\begin{definition}
Both sets $\mathbf{B}$ and $\mathbf{B}^{n}$ are organized as topological
spaces by the discrete topology.
\end{definition}

\begin{notation}
\label{Not2}$\mathbf{N},$ $\mathbf{Z},$ $\mathbf{R}$ denote the sets of the
non negative integers, of the integers and of the real numbers. $\mathbf{N}%
_{\_}=\mathbf{N}\cup\{-1\}$ is the notation of the discrete time set.
\end{notation}

\begin{notation}
\label{Not3}We denote%
\[
\widehat{Seq}=\{(k_{j})|k_{j}\in\mathbf{N}_{\_},j\in\mathbf{N}_{\_}\text{ and
}k_{-1}<k_{0}<k_{1}<...\},
\]%
\[
Seq=\{(t_{k})|t_{k}\in\mathbf{R},k\in\mathbf{N}\text{ and }t_{0}<t_{1}%
<t_{2}<...\text{ superiorly unbounded}\}.
\]

\end{notation}

\begin{example}
A typical example of element of $\widehat{Seq}$ is the sequence $k_{j}%
=j,j\in\mathbf{N}_{\_}$ and typical examples of elements of $Seq$ are given by
the sequences $z,z+1,z+2,...,z\in\mathbf{Z}$.
\end{example}

\begin{proposition}
\label{Lem3}Let $(t_{k})\in Seq$ and $t\in\mathbf{R}$ be arbitrary. Then%
\[
\exists\varepsilon>0,\{k|k\in\mathbf{N},t_{k}\in(t-\varepsilon,t+\varepsilon
)\}=\left\{
\begin{array}
[c]{c}%
\{k^{\prime}\},if\;t=t_{k^{\prime}},\\
\varnothing,if\;\forall k\in\mathbf{N},t\neq t_{k}.
\end{array}
\right.
\]

\end{proposition}

\begin{proof}
We have the following possibilities.

Case $t<t_{0};$ we take $\varepsilon\in(0,t_{0}-t),$ for which $\{k|k\in
\mathbf{N},t_{k}\in(t-\varepsilon,t+\varepsilon)\}=\varnothing.$

Case $t=t_{0};$ for $\varepsilon\in(0,t_{1}-t)$ we have $\{k|k\in
\mathbf{N},t_{k}\in(t-\varepsilon,t+\varepsilon)\}=\{t_{0}\}.$

Case $t\in(t_{k^{\prime}-1},t_{k^{\prime}}),k^{\prime}\geq1;$ $\varepsilon
\in(0,\min\{t-t_{k^{\prime}-1},t_{k^{\prime}}-t\})$ gives $\{k|k\in
\mathbf{N},t_{k}\in(t-\varepsilon,t+\varepsilon)\}=\varnothing.$

Case $t=t_{k^{\prime}},k^{\prime}\geq1;$ in this situation any $\varepsilon
\in(0,\min\{t-t_{k^{\prime}-1},t_{k^{\prime}+1}-t\})$ gives $\{k|k\in
\mathbf{N},t_{k}\in(t-\varepsilon,t+\varepsilon)\}=\{t_{k^{\prime}}\}.$
\end{proof}

\begin{remark}
The previous $\varepsilon$ obviously depends on $t$. We consider for example
the sequence $t_{k}=\frac{1}{0+1}+\frac{1}{1+1}+...+\frac{1}{k+1}%
,k\in\mathbf{N.}$ We have $(t_{k})\in Seq$ and
\[
\forall\varepsilon>0,\exists t\in\mathbf{R},card(\{k|k\in\mathbf{N},t_{k}%
\in(t-\varepsilon,t+\varepsilon)\})>1
\]
holds.
\end{remark}

\begin{notation}
\label{Not4}$\chi_{A}:\mathbf{R}\rightarrow\mathbf{B}$ is the notation of the
characteristic function of the set $A\subset\mathbf{R}:\forall t\in
\mathbf{R},$%
\[
\chi_{A}(t)=\left\{
\begin{array}
[c]{c}%
1,if\;t\in A,\\
0,otherwise
\end{array}
\right.  .
\]

\end{notation}

\begin{definition}
\label{Def2}The \textbf{discrete time signals} are by definition the functions
$\widehat{x}:\mathbf{N}_{\_}\rightarrow\mathbf{B}^{n}.$ Their set is denoted
with $\widehat{S}^{(n)}.$

The \textbf{continuous time signals} are the functions $x:\mathbf{R}%
\rightarrow\mathbf{B}^{n}$ of the form $\forall t\in\mathbf{R},$%
\begin{equation}
x(t)=\mu\cdot\chi_{(-\infty,t_{0})}(t)\oplus x(t_{0})\cdot\chi_{\lbrack
t_{0},t_{1})}(t)\oplus...\oplus x(t_{k})\cdot\chi_{\lbrack t_{k},t_{k+1}%
)}(t)\oplus... \label{per189}%
\end{equation}
where $\mu\in\mathbf{B}^{n}$ and $(t_{k})\in Seq.$ Their set is denoted by
$S^{(n)}.$
\end{definition}

\begin{example}
The constant functions $\widehat{x}\in\widehat{S}^{(1)},x\in S^{(1)}$ equal
with $\mu\in\mathbf{B}:$%
\begin{equation}
\forall k\in\mathbf{N}_{\_},\widehat{x}(k)=\mu, \label{per521}%
\end{equation}%
\begin{equation}
\forall t\in\mathbf{R},x(t)=\mu\label{per522}%
\end{equation}
are typical examples of signals. Here are some other examples:%
\begin{equation}
\forall k\in\mathbf{N}_{\_},\widehat{x}(k)=\left\{
\begin{array}
[c]{c}%
1,if\;k\;is\;odd,\\
0,\;if\;k\;is\;even
\end{array}
\right.  , \label{per518}%
\end{equation}%
\begin{equation}
\forall t\in\mathbf{R},x(t)=\chi_{\lbrack0,\infty)}(t), \label{per519}%
\end{equation}%
\begin{equation}
\forall t\in\mathbf{R},x(t)=\chi_{\lbrack0,1)}(t)\oplus\chi_{\lbrack
2,3)}(t)\oplus...\oplus\chi_{\lbrack2k,2k+1)}(t)\oplus... \label{per520}%
\end{equation}
The signal from (\ref{per519}) is called the (unitary) step function (of Heaviside).
\end{example}

\begin{remark}
At Definition \ref{Def2} a convention of notation has occurred for the first
time, namely a hat\ ' $\widehat{\ }$ ' is used to show that we have discrete
time. The hat will make the difference between, for example, the notation of
the discrete time signals $\widehat{x},\widehat{y},..$ and the notation of the
real time signals $x,y,...$
\end{remark}

\begin{remark}
The discrete time signals are sequences. The real time signals are piecewise
constant functions.
\end{remark}

\begin{remark}
As we shall see in the rest of the book, the study of the periodicity of the
signals does not use essentially the fact that they take values in
$\mathbf{B}^{n},$ but the fact that they take finitely many values. For
example, instead of using $^{\prime}\cdot^{\prime}$ and $^{\prime}%
\oplus^{\prime}$ in (\ref{per189}), we can write equivalently%
\[
x(t)=\left\{
\begin{array}
[c]{c}%
\mu,t<t_{0},\\
x(t_{0}),t\in\lbrack t_{0},t_{1}),\\
...\\
x(t_{k}),t\in\lbrack t_{k},t_{k+1}),\\
...
\end{array}
\right.
\]

\end{remark}

\begin{remark}
The signals model the electrical signals of the circuits from the digital
electrical engineering.
\end{remark}

\section{Left and right limits}

\begin{theorem}
\label{The1}For any $x\in S^{(n)}$ and any $t\in\mathbf{R},$ there exist
$x(t-0),x(t+0)\in\mathbf{B}^{n}$ with the property%
\begin{equation}
\exists\varepsilon>0,\forall\xi\in(t-\varepsilon,t),x(\xi)=x(t-0),
\label{per488}%
\end{equation}%
\begin{equation}
\exists\varepsilon>0,\forall\xi\in(t,t+\varepsilon),x(\xi)=x(t+0).
\label{per489}%
\end{equation}

\end{theorem}

\begin{proof}
We presume that $x,t$ are arbitrary and fixed and that $x$ is of the form
\begin{equation}
x(t)=\mu\cdot\chi_{(-\infty,t_{0})}(t)\oplus x(t_{0})\cdot\chi_{\lbrack
t_{0},t_{1})}(t)\oplus...\oplus x(t_{k})\cdot\chi_{\lbrack t_{k},t_{k+1}%
)}(t)\oplus... \label{p95}%
\end{equation}
with $\mu\in\mathbf{B}^{n}$ and $(t_{k})\in Seq.$ We take $\varepsilon>0$
small enough, see Proposition \ref{Lem3}, page \pageref{Lem3} such that%
\[
\{k|k\in\mathbf{N},t_{k}\in(t-\varepsilon,t+\varepsilon)\}=\left\{
\begin{array}
[c]{c}%
\{k^{\prime}\},if\;t=t_{k^{\prime}},\\
\varnothing,if\;\forall k\in\mathbf{N},t\neq t_{k}.
\end{array}
\right.
\]
We have the following possibilities:

Case $t<t_{0};$%
\[
\forall\xi\in(t-\varepsilon,t),x(\xi)=\mu,
\]%
\[
\forall\xi\in(t,t+\varepsilon),x(\xi)=\mu.
\]

Case $t=t_{0};$%
\[
\forall\xi\in(t-\varepsilon,t),x(\xi)=\mu,
\]%
\[
\forall\xi\in(t,t+\varepsilon),x(\xi)=x(t_{0}).
\]

Case $t\in(t_{k^{\prime}-1},t_{k^{\prime}}),k^{\prime}\geq1;$%
\[
\forall\xi\in(t-\varepsilon,t),x(\xi)=x(t_{k^{\prime}-1}),
\]%
\[
\forall\xi\in(t,t+\varepsilon),x(\xi)=x(t_{k^{\prime}-1}).
\]

Case $t=t_{k^{\prime}},k^{\prime}\geq1;$%
\[
\forall\xi\in(t-\varepsilon,t),x(\xi)=x(t_{k^{\prime}-1}),
\]%
\[
\forall\xi\in(t,t+\varepsilon),x(\xi)=x(t_{k^{\prime}}).
\]

\end{proof}

\begin{definition}
\label{Def8}The functions $\mathbf{R}\ni t\rightarrow x(t-0)\in\mathbf{B}%
^{n},\mathbf{R}\ni t\rightarrow x(t+0)\in\mathbf{B}^{n}$ are called the
\textbf{left limit} function of $x$ and the \textbf{right limit} function of
$x$.
\end{definition}

\begin{remark}
Theorem \ref{The1} states that the signals $x\in S^{(n)}$ have a left limit
function $x(t-0)$ and a right limit function $x(t+0)$. Moreover, if
(\ref{p95}) is true$,$ then%
\begin{equation}
x(t-0)=\mu\cdot\chi_{(-\infty,t_{0}]}(t)\oplus x(t_{0})\cdot\chi_{(t_{0}%
,t_{1}]}(t)\oplus...\oplus x(t_{k})\cdot\chi_{(t_{k},t_{k+1}]}(t)\oplus...,
\label{per425}%
\end{equation}%
\begin{equation}
x(t+0)=x(t) \label{per426}%
\end{equation}
hold, meaning in particular that $x(t-0)$ is not a signal and that $x(t+0)$
coincides with $x(t).$
\end{remark}

\begin{remark}
The property (\ref{per426}) stating in fact that the real time signals $x$ are
right continuous will be used later under the form
\begin{equation}
\forall t\in\mathbf{R},\exists\varepsilon>0,\forall\xi\in\lbrack
t,t+\varepsilon),x(\xi)=x(t). \label{per741}%
\end{equation}

\end{remark}

\section{Initial and final values, initial and final time}

\begin{definition}
\label{Def6}The \textbf{initial value} of $\widehat{x}\in\widehat{S}^{(n)}$ is
$\widehat{x}(-1)\in\mathbf{B}^{n}.$

For $x\in S^{(n)},$%
\begin{equation}
x(t)=\mu\cdot\chi_{(-\infty,t_{0})}(t)\oplus x(t_{0})\cdot\chi_{\lbrack
t_{0},t_{1})}(t)\oplus...\oplus x(t_{k})\cdot\chi_{\lbrack t_{k},t_{k+1}%
)}(t)\oplus..., \label{pre745}%
\end{equation}
where $\mu\in\mathbf{B}^{n}$ and $(t_{k})\in Seq,$ the \textbf{initial value}
is $\mu.$
\end{definition}

\begin{notation}
\label{Not14}There is no special notation for the initial value of
$\widehat{x}$.

The initial value of $x$ has two usual notations, $x(-\infty+0)$ and
$\underset{t\rightarrow-\infty}{\lim}x(t).$
\end{notation}

\begin{definition}
\label{Def29}By definition, the \textbf{initial time} (\textbf{instant}) of
$\widehat{x}$ is $k=-1.$

The \textbf{initial time} (\textbf{instant}) of $x$ is any number $t_{0}%
\in\mathbf{R}$ that fulfills%
\begin{equation}
\forall t\leq t_{0},x(t)=x(-\infty+0). \label{p121}%
\end{equation}

\end{definition}

\begin{notation}
\label{Not15}The set of the initial time instants of $x$ is denoted by
$I^{x}.$
\end{notation}

\begin{definition}
\label{Def7}The \textbf{final value} $\mu\in\mathbf{B}^{n}$ of $\widehat{x}%
\in\widehat{S}^{(n)}$ is defined by $\exists k^{\prime}\in\mathbf{N}_{\_}, $%
\begin{equation}
\forall k\geq k^{\prime},\widehat{x}(k)=\mu\label{pre746}%
\end{equation}
and the \textbf{final value} $\mu\in\mathbf{B}^{n}$ of $x\in S^{(n)}$ is
defined by $\exists t^{\prime}\in\mathbf{R},$%
\begin{equation}
\forall t\geq t^{\prime},x(t)=\mu. \label{pre747}%
\end{equation}

\end{definition}

\begin{notation}
\label{Not16}The usual notations for $\mu$ in (\ref{pre746}) are $\widehat
{x}(\infty-0)$ and $\underset{k\rightarrow\infty}{\lim}\widehat{x}(k).$

The final value $\mu$ from (\ref{pre747}) is denoted with either of
$x(\infty-0)$ and $\underset{t\rightarrow\infty}{\lim}x(t).$
\end{notation}

\begin{definition}
\label{Def44}If the final value $\mu$ of $\widehat{x}$ exists, then any
$k^{\prime}\in\mathbf{N}_{\_}$ like in (\ref{pre746}) is called \textbf{final
time} (\textbf{instant}) of $\widehat{x}$.

Similarly, if the final value $\mu$ of $x$ exists and (\ref{pre747}) holds,
then any such $t^{\prime}\in\mathbf{R}$ is called \textbf{final time}
(\textbf{instant}) of $x$.
\end{definition}

\begin{notation}
\label{Not17}The set of the final time instants of $\widehat{x}$ is denoted by
$\widehat{F}^{\widehat{x}}.$

The set of the final time instants of $x$ has the notation $F^{x}.$
\end{notation}

\begin{example}
The signals from (\ref{per521}), (\ref{per522}) fulfill $\underset
{k\rightarrow\infty}{\lim}\widehat{x}(k)=\underset{t\rightarrow-\infty}{\lim
}x(t)=\underset{t\rightarrow\infty}{\lim}x(t)=\mu,$ $\widehat{F}^{\widehat{x}%
}=\mathbf{N}_{\_},I^{x}=F^{x}=\mathbf{R}$ and the signal from (\ref{per519})
fulfills $\underset{t\rightarrow-\infty}{\lim}x(t)=0,$ $\underset
{t\rightarrow\infty}{\lim}x(t)=1,I^{x}=(-\infty,0),F^{x}=[0,\infty).$ The
signals (\ref{per518}), (\ref{per520}) have no final value: $\widehat
{F}^{\widehat{x}}=F^{x}=\varnothing.$
\end{example}

\begin{remark}
For arbitrary $\widehat{x},x$ the initial value exists and it is unique; the
initial time of $\widehat{x}$ is unique and the initial time of $x$ is not unique.

The final value of $\widehat{x},x$ might not exist, but if it exists, it is
unique. The final time of $\widehat{x},x$ might not exist, but if it exists,
it is not unique.
\end{remark}

\begin{theorem}
\label{The143}a) Let $\widehat{x}\in\widehat{S}^{(n)}$ and $k_{0}\in
\mathbf{N}_{\_}.$ The following equivalencies hold:%
\begin{equation}
\left\{
\begin{array}
[c]{c}%
\forall k\geq k_{0},\widehat{x}(k)=\widehat{x}(\infty-0),\\
k_{0}\geq0\Longrightarrow\widehat{x}(k_{0}-1)\neq\widehat{x}(\infty-0)
\end{array}
\right.  \Longleftrightarrow\widehat{F}^{\widehat{x}}=\{k_{0},k_{0}%
+1,k_{0}+2,...\}, \label{p257}%
\end{equation}%
\begin{equation}
\forall k\in\mathbf{N}_{\_},\widehat{x}(k)=\widehat{x}(\infty
-0)\Longleftrightarrow\widehat{F}^{\widehat{x}}=\mathbf{N}_{\_}. \label{p258}%
\end{equation}

b) Let $x\in S^{(n)},t_{0}\in\mathbf{R}.$ The following equivalencies take
place:%
\begin{equation}
\left\{
\begin{array}
[c]{c}%
\forall t<t_{0},x(t)=x(-\infty+0),\\
x(t_{0})\neq x(-\infty+0)
\end{array}
\right.  \Longleftrightarrow I^{x}=(-\infty,t_{0}), \label{p259}%
\end{equation}%
\begin{equation}
\left\{
\begin{array}
[c]{c}%
\forall t\geq t_{0},x(t)=x(\infty-0),\\
x(t_{0}-0)\neq x(t_{0})
\end{array}
\right.  \Longleftrightarrow F^{x}=[t_{0},\infty), \label{p260}%
\end{equation}%
\begin{equation}
\forall t\in\mathbf{R},x(t)=x(-\infty+0)\Longleftrightarrow I^{x}=\mathbf{R},
\label{p261}%
\end{equation}%
\begin{equation}
\forall t\in\mathbf{R},x(t)=x(\infty-0)\Longleftrightarrow F^{x}=\mathbf{R}.
\label{p262}%
\end{equation}

\end{theorem}

\begin{proof}
a) Two possibilities exist.

Case $k_{0}=-1$

The statements $\forall k\in\mathbf{N}_{\_},\widehat{x}(k)=\widehat{x}%
(\infty-0)$ and $\{k^{\prime}|\forall k\geq k^{\prime},\widehat{x}%
(k)=\widehat{x}(\infty-0)\}=\mathbf{N}_{\_}$ are equivalent indeed and this
special case of (\ref{p257}) coincides with (\ref{p258})$.$

Case $k_{0}\geq0$

We have that $(\forall k\geq k_{0},\widehat{x}(k)=\widehat{x}(\infty-0)$ and
$\widehat{x}(k_{0}-1)\neq\widehat{x}(\infty-0))$ is equivalent with
$\{k^{\prime}|\forall k\geq k^{\prime},\widehat{x}(k)=\widehat{x}%
(\infty-0)\}=\{k_{0},k_{0}+1,k_{0}+2,...\}.$

b) The statement $(\forall t<t_{0},x(t)=x(-\infty+0)$ and $x(t_{0})\neq
x(-\infty+0))$ is equivalent with $\{t^{\prime}|\forall t\leq t^{\prime
},x(t)=x(-\infty+0)\}=(-\infty,t_{0}).$ This coincides with (\ref{p259}). The
statement $\forall t\in\mathbf{R},x(t)=x(-\infty+0)$ is equivalent with
$\{t^{\prime}|\forall t\leq t^{\prime},x(t)=x(-\infty+0)\}=\mathbf{R},$ giving
the truth of (\ref{p261}). (\ref{p260}) and (\ref{p262}) are obvious now.
\end{proof}

\begin{remark}
Versions of Theorem \ref{The143} exist, stating that $\widehat{x}$ is constant
iff $\widehat{F}^{\widehat{x}}=\mathbf{N}_{\_}$ and non constant otherwise,
respectively the statements:

i) $x$ is not constant;

ii) $t_{0}\in\mathbf{R}$ exists with
\[
\forall t<t_{0},x(t)=x(-\infty+0),
\]%
\[
x(t_{0})\neq x(-\infty+0);
\]

iii) $t_{0}\in\mathbf{R}$ exists with $I^{x}=(-\infty,t_{0})$

are equivalent etc.
\end{remark}

\begin{theorem}
\label{The92}Let the signal $x\in S^{(n)}$ from (\ref{pre745}). We define
$\widehat{x}\in\widehat{S}^{(n)}$ by%
\[
\widehat{x}(-1)=\mu,
\]%
\[
\forall k\in\mathbf{N},\widehat{x}(k)=x(t_{k}).
\]
The following statements hold.

a) $\underset{k\rightarrow\infty}{\lim}\widehat{x}(k)$ exists if and only if
$\underset{t\rightarrow\infty}{\lim}x(t)$ exists and in case that the previous
limits exist we have $\underset{k\rightarrow\infty}{\lim}\widehat
{x}(k)=\underset{t\rightarrow\infty}{\lim}x(t).$

b) We suppose that $\underset{k\rightarrow\infty}{\lim}\widehat{x}%
(k),\underset{t\rightarrow\infty}{\lim}x(t)$ exist. Then $-1$ is final time of
$\widehat{x}$ if and only if any $t^{\prime}<t_{0}$ is final time of $x$ and
$\forall k^{\prime}\geq0,$ $k^{\prime}$ is final time of $\widehat{x}$ if and
only if $t_{k^{\prime}}$ is final time of $x.$
\end{theorem}

\begin{proof}
a) From the hypothesis we infer that for any $k^{\prime}\in\mathbf{N}$ we can
write%
\begin{equation}
\{\widehat{x}(k)|k\geq k^{\prime}\}=\{x(t)|t\geq t_{k^{\prime}}\}.
\label{pre973}%
\end{equation}
Then%
\[
\underset{k\rightarrow\infty}{\lim}\widehat{x}(k)\text{ exists}%
\Longleftrightarrow\exists k^{\prime}\in\mathbf{N},card(\{\widehat{x}(k)|k\geq
k^{\prime}\})=1\Longleftrightarrow
\]%
\[
\Longleftrightarrow\exists k^{\prime}\in\mathbf{N},card(\{x(t)|t\geq
t_{k^{\prime}}\})=1\Longleftrightarrow\underset{t\rightarrow\infty}{\lim
}x(t)\text{ exists}%
\]
and if one of the previous equivalent statements is true, we obtain the
existence of $\mu\in\mathbf{B}^{n},k^{\prime}\in\mathbf{N}$ such that%
\[
\{\widehat{x}(k)|k\geq k^{\prime}\}=\{\mu\}=\{x(t)|t\geq t_{k^{\prime}}\}
\]
i.e.
\begin{equation}
\underset{k\rightarrow\infty}{\lim}\widehat{x}(k)=\mu=\underset{t\rightarrow
\infty}{\lim}x(t). \label{pre974}%
\end{equation}

b) Let us presume that (\ref{pre974}) holds$.$ We have%
\[
-1\in\widehat{F}^{\widehat{x}}\Longleftrightarrow\forall k\in\mathbf{N}%
_{\_},\widehat{x}(k)=\mu\Longleftrightarrow\forall t\in\mathbf{R}%
,x(t)=\mu\Longleftrightarrow
\]%
\[
\Longleftrightarrow\forall t^{\prime}<t_{0},\forall t\geq t^{\prime}%
,x(t)=\mu\Longleftrightarrow\forall t^{\prime}<t_{0},t^{\prime}\in F^{x}%
\]
and similarly for any $k^{\prime}\geq0,$%
\[
k^{\prime}\in\widehat{F}^{\widehat{x}}\Longleftrightarrow\forall k\geq
k^{\prime},\widehat{x}(k)=\mu\Longleftrightarrow\forall t\geq t_{k^{\prime}%
},x(t)=\mu\Longleftrightarrow t_{k^{\prime}}\in F^{x}.
\]

\end{proof}

\section{The forgetful function}

\begin{definition}
\label{Def26}The \textbf{discrete time} \textbf{forgetful function}
$\widehat{\sigma}^{k^{\prime}}:\widehat{S}^{(n)}\rightarrow\widehat{S}^{(n)}$
is defined for $k^{\prime}\in\mathbf{N}$ by
\[
\forall\widehat{x}\in\widehat{S}^{(n)},\forall k\in\mathbf{N}_{\_}%
,\widehat{\sigma}^{k^{\prime}}(\widehat{x})(k)=\widehat{x}(k+k^{\prime})
\]
and the\textbf{\ real time forgetful function} $\sigma^{t^{\prime}}%
:S^{(n)}\rightarrow S^{(n)}$ is defined for $t^{\prime}\in\mathbf{R}$ in the
following manner%
\[
\forall x\in S^{(n)},\forall t\in\mathbf{R},\sigma^{t^{\prime}}(x)(t)=\left\{
\begin{array}
[c]{c}%
x(t),t\geq t^{\prime},\\
x(t^{\prime}-0),t<t^{\prime}%
\end{array}
\right.  .
\]

\end{definition}

\begin{theorem}
\label{Pro2}The signals $\widehat{x}\in\widehat{S}^{(n)},x\in S^{(n)}$ are
given. The following statements hold:

a) $\widehat{\sigma}^{0}(\widehat{x})=\widehat{x};$ if $I^{x}=\mathbf{R},$
then $\forall t^{\prime}\in\mathbf{R},\sigma^{t^{\prime}}(x)=x$ and if
$\exists t_{0}\in\mathbf{R},$ $I^{x}=(-\infty,t_{0}),$ then $\forall
t^{\prime}\leq t_{0},\sigma^{t^{\prime}}(x)=x;$

b) for $k^{\prime},k^{\prime\prime}\in\mathbf{N}$ we have $(\widehat{\sigma
}^{k^{\prime}}\circ\widehat{\sigma}^{k^{\prime\prime}})(\widehat{x}%
)=\widehat{\sigma}^{k^{\prime}+k^{\prime\prime}}(\widehat{x});$ for any
$t^{\prime},t^{\prime\prime}\in\mathbf{R}$ we have $(\sigma^{t^{\prime}}%
\circ\sigma^{t^{\prime\prime}})(x)=\sigma^{\max\{t^{\prime},t^{\prime\prime
}\}}(x).$
\end{theorem}

\begin{proof}
a) The discrete time statement is obvious. In order to prove the real time
statement, we notice that $I^{x}=\mathbf{R}$ is true if $x$ is constant, see
Theorem \ref{The143}, page \pageref{The143}, so that we can suppose now that
$x$ is not constant and some $t_{0}$ exists such that $I^{x}=(-\infty,t_{0}):$%
\[
\forall t<t_{0},x(t)=x(-\infty+0),
\]%
\[
x(t_{0})\neq x(-\infty+0).
\]
Let $t^{\prime}\leq t_{0}$ arbitrary. We have $\forall t\in\mathbf{R},$%
\[
\sigma^{t^{\prime}}(x)(t)=\left\{
\begin{array}
[c]{c}%
x(t),t\geq t^{\prime}\\
x(t^{\prime}-0),t<t^{\prime}%
\end{array}
\right.  =\left\{
\begin{array}
[c]{c}%
x(t),t\geq t^{\prime}\\
x(-\infty+0),t<t^{\prime}%
\end{array}
\right.  =x(t).
\]

b) We fix arbitrarily $k^{\prime},k^{\prime\prime}\in\mathbf{N}$. We can write
for any $k\in\mathbf{N}$ that%
\[
(\widehat{\sigma}^{k^{\prime}}\circ\widehat{\sigma}^{k^{\prime\prime}%
})(\widehat{x})(k)=\widehat{x}(k+k^{\prime}+k^{\prime\prime})=\widehat{\sigma
}^{k^{\prime}+k^{\prime\prime}}(\widehat{x})(k).
\]
Let us take now $t^{\prime},t^{\prime\prime}\in\mathbf{R}$ arbitrarily. We get
the existence of the next possibilities.

Case $t^{\prime\prime}\leq t^{\prime}$

For any $t\in\mathbf{R}$ we infer%
\[
(\sigma^{t^{\prime}}\circ\sigma^{t^{\prime\prime}})(x)(t)=\sigma^{t^{\prime}%
}(\sigma^{t^{\prime\prime}}(x))(t)=\left\{
\begin{array}
[c]{c}%
\sigma^{t^{\prime\prime}}(x)(t),t\geq t^{\prime}\\
\sigma^{t^{\prime\prime}}(x)(t^{\prime}-0),t<t^{\prime}%
\end{array}
\right.
\]%
\[
=\left\{
\begin{array}
[c]{c}%
x(t),t\geq t^{\prime}\\
x(t^{\prime}-0),t<t^{\prime}%
\end{array}
\right.  =\sigma^{t^{\prime}}(x)(t).
\]

Case $t^{\prime\prime}>t^{\prime}$

We get for arbitrary $t\in\mathbf{R}$ that%
\[
(\sigma^{t^{\prime}}\circ\sigma^{t^{\prime\prime}})(x)(t)=\sigma^{t^{\prime}%
}(\sigma^{t^{\prime\prime}}(x))(t)=\left\{
\begin{array}
[c]{c}%
\sigma^{t^{\prime\prime}}(x)(t),t\geq t^{\prime}\\
\sigma^{t^{\prime\prime}}(x)(t^{\prime}-0),t<t^{\prime}%
\end{array}
\right.
\]%
\[
=\left\{
\begin{array}
[c]{c}%
\sigma^{t^{\prime\prime}}(x)(t),t\geq t^{\prime}\\
x(t^{\prime\prime}-0),t<t^{\prime}%
\end{array}
\right.  =\sigma^{t^{\prime\prime}}(x)(t).
\]

\end{proof}

\begin{remark}
Let us give $\widehat{x}$ by its values $\widehat{x}=x^{-1},x^{0},x^{1},...$
where $x^{k}\in\mathbf{B}^{n},k\in\mathbf{N}_{\_}.$ Then $\widehat{\sigma}%
^{1}(\widehat{x})=x^{0},x^{1},...$ i.e. $\widehat{x}$ has forgotten its first
value. Furthermore, $\widehat{\sigma}^{0}(\widehat{x})$ makes $\widehat{x}$
forget nothing and $\widehat{\sigma}^{k^{\prime}}(\widehat{x}) $ makes
$\widehat{x}$ forget its first $k^{\prime}\geq1$ values.
\end{remark}

\begin{remark}
$\sigma^{t^{\prime}}(x)$ makes $x$ forget its values prior to $t^{\prime}: $
no value if $\forall t<t^{\prime},x(t)=x(-\infty+0)$ and some values otherwise.
\end{remark}

\section{Orbits, omega limit sets and support sets}

\begin{definition}
\label{Def9}The \textbf{orbits} of $\widehat{x}\in\widehat{S}^{(n)},x\in
S^{(n)}$ are the sets of the values of these functions:%
\[
\widehat{Or}(\widehat{x})=\{\widehat{x}(k)|k\in\mathbf{N}_{\_}\},
\]%
\[
Or(x)=\{x(t)|t\in\mathbf{R}\}.
\]

\end{definition}

\begin{definition}
\label{Def10_}The \textbf{omega limit set} $\widehat{\omega}(\widehat{x})$ of
$\widehat{x}$ is defined as%
\[
\widehat{\omega}(\widehat{x})=\{\mu|\mu\in\mathbf{B}^{n},\exists(k_{j}%
)\in\widehat{Seq},\forall j\in\mathbf{N}_{\_},\widehat{x}(k_{j})=\mu\}
\]
and the \textbf{omega limit set} $\omega(x)$ of $x$ is defined by%
\[
\omega(x)=\{\mu|\mu\in\mathbf{B}^{n},\exists(t_{k})\in Seq,\forall
k\in\mathbf{N},x(t_{k})=\mu\}.
\]
The points of $\widehat{\omega}(\widehat{x}),\omega(x)$ are called
\textbf{omega limit points}.\footnote{In a real time construction, in
\cite{bib10}, when $x$ represents the state of a (control, nondeterministic,
asynchronous) system, the value $\mu$ of $x $ is called (accessible) recurrent
if $\forall t_{0}\in\mathbf{R},\exists t>t_{0},x(t)=\mu,$ i.e. if $\mu
\in\omega(x).$}
\end{definition}

\begin{example}
We define $\widehat{x}\in\widehat{S}^{(2)}$ by%
\[
\widehat{x}(k)=\left\{
\begin{array}
[c]{c}%
(0,0),k=-1,\\
(0,1),k=3k^{\prime}+1,k^{\prime}\geq0,\\
(1,0),k=3k^{\prime}+2,k^{\prime}\geq0,\\
(1,1),k=3k^{\prime},k^{\prime}\geq0
\end{array}
\right.
\]
and $x\in S^{(2)}$ by%
\[
x(t)=\widehat{x}(-1)\cdot\chi_{(-\infty,0)}(t)\oplus\widehat{x}(0)\cdot
\chi_{\lbrack0,1)}(t)\oplus...\oplus\widehat{x}(k)\cdot\chi_{\lbrack
k,k+1)}(t)\oplus...
\]
We see that $\widehat{Or}(\widehat{x})=Or(x)=\mathbf{B}^{2}$ and
$\widehat{\omega}(\widehat{x})=\omega(x)=\{(0,1),(1,0),$ $(1,1)\}.$
\end{example}

\begin{definition}
\label{Not5}For $\widehat{x}\in\widehat{S}^{(n)},x\in S^{(n)}$ and $\mu
\in\mathbf{B}^{n},$ we define the \textbf{support sets} of $\mu$ by%
\[
\widehat{\mathbf{T}}_{\mu}^{\widehat{x}}=\{k|k\in\mathbf{N}_{\_},\widehat
{x}(k)=\mu\},
\]%
\[
\mathbf{T}_{\mu}^{x}=\{t|t\in\mathbf{R},x(t)=\mu\}.
\]

\end{definition}

\begin{remark}
The previous Definition allows us to express the fact that $t$ is an initial
time instant of $x$, $t\in I^{x}$ under the equivalent form $(-\infty
,t]\subset\mathbf{T}_{x(-\infty+0)}^{x}.$ We shall use sometimes this
possibility in the rest of the exposure.
\end{remark}

\begin{theorem}
\label{The12}Let $\widehat{x}\in\widehat{S}^{(n)},x\in S^{(n)}$. We have that

a) $\widehat{\omega}(\widehat{x})=\{\mu|\mu\in\mathbf{B}^{n},\widehat
{\mathbf{T}}_{\mu}^{\widehat{x}}$ is infinite$\},$ $\omega(x)=\{\mu|\mu
\in\mathbf{B}^{n},\mathbf{T}_{\mu}^{x}$ is unbounded from above$\};$

b) $\widehat{\omega}(\widehat{x})\neq\varnothing,\omega(x)\neq\varnothing;$

c) for any $\widetilde{k}\in\mathbf{N},\widetilde{t}\in\mathbf{R}$ the
following diagrams commute%
\[%
\begin{array}
[c]{ccc}%
\widehat{Or}(\widehat{\sigma}^{\widetilde{k}}(\widehat{x})) & \subset &
\widehat{Or}(\widehat{x})\\
\cup &  & \cup\\
\widehat{\omega}(\widehat{\sigma}^{\widetilde{k}}(\widehat{x})) & = &
\widehat{\omega}(\widehat{x})
\end{array}
,\;%
\begin{array}
[c]{ccc}%
Or(\sigma^{\widetilde{t}}(x)) & \subset & Or(x)\\
\cup &  & \cup\\
\omega(\sigma^{\widetilde{t}}(x)) & = & \omega(x)
\end{array}
\]

\end{theorem}

\begin{proof}
a) Indeed, for any $\mu\in\mathbf{B}^{n},$ the fact that $\mu\in
\widehat{\omega}(\widehat{x})$ is equivalent with any of:%
\[
\text{a sequence }k_{-1}<k_{0}<k_{1}<...\text{ exists such that }\forall
j\in\mathbf{N}_{\_},\widehat{x}(k_{j})=\mu,
\]%
\[
\text{the set }\{k|k\in\mathbf{N}_{\_},\widehat{x}(k)=\mu\}\text{ is infinite}%
\]
and the fact that $\mu\in\omega(x)$ is equivalent with any of%
\[
\text{an unbounded from above sequence }t_{0}<t_{1}<t_{2}<...\text{ exists
with }\forall j\in\mathbf{N},x(t_{j})=\mu,
\]%
\[
\text{the set }\{t|t\in\mathbf{R},x(t)=\mu\}\text{ is unbounded from above.}%
\]

b) The sets $\widehat{\mathbf{T}}_{\mu}^{\widehat{x}},$ $\mu\in\mathbf{B}^{n}$
are either empty, or finite non-empty, or infinite. We put $\mathbf{B}^{n}$
under the form $\mathbf{B}^{n}=\{\mu^{1},\mu^{2},...,\mu^{2^{n}}\}.$ Because
in the equation%
\[
\widehat{\mathbf{T}}_{\mu^{1}}^{\widehat{x}}\cup...\cup\widehat{\mathbf{T}%
}_{\mu^{2^{n}}}^{\widehat{x}}=\mathbf{N}_{\_}%
\]
the right hand set is infinite, we infer that infinite sets $\widehat
{\mathbf{T}}_{\mu^{i}}^{\widehat{x}}$ always exist, let them be, without
loosing the generality, $\widehat{\mathbf{T}}_{\mu^{1}}^{\widehat{x}%
},...,\widehat{\mathbf{T}}_{\mu^{p}}^{\widehat{x}}.$ We have from a)%
\[
\widehat{\omega}(\widehat{x})=\{\mu^{1},...,\mu^{p}\}.
\]

Similarly, we consider the equation%
\[
\mathbf{T}_{\mu^{1}}^{x}\cup...\cup\mathbf{T}_{\mu^{2^{n}}}^{x}=\mathbf{R}%
\]
where the right hand set is unbounded from above. We infer that the left hand
term contains sets $\mathbf{T}_{\mu^{i}}^{x}$ which are unbounded from above
and let them be, without loosing the generality, $\mathbf{T}_{\mu^{1}}%
^{x},...,\mathbf{T}_{\mu^{p}}^{x}.$ We infer from a) that%
\[
\omega(x)=\{\mu^{1},...,\mu^{p}\}.
\]

c) We prove that $\widehat{\omega}(\widehat{x})\subset\widehat{Or}(x).$ Some
sets $\widehat{\mathbf{T}}_{\mu^{i}}^{\widehat{x}}$ may exist which are finite
non-empty, let them be without loosing the generality $\widehat{\mathbf{T}%
}_{\mu^{p+1}}^{\widehat{x}},...,\widehat{\mathbf{T}}_{\mu^{s}}^{\widehat{x}},$
where $p\leq s\leq2^{n}.$ Then%
\[
\widehat{\omega}(\widehat{x})=\{\mu^{1},...,\mu^{p}\}\subset\{\mu^{1}%
,...,\mu^{s}\}=\widehat{Or}(\widehat{x}).
\]
The previous inclusion is true as equality if finite non-empty sets
$\widehat{\mathbf{T}}_{\mu^{i}}^{\widehat{x}}$ do not exist and $p=s.$

The proof of $\omega(x)\subset Or(x)$ is similar, we presume that the
non-empty, bounded sets $\mathbf{T}_{\mu^{i}}^{x}$ are $\mathbf{T}_{\mu^{p+1}%
}^{x},...,\mathbf{T}_{\mu^{s}}^{x},$ with $p\leq s\leq2^{n}.$ Then%
\[
\omega(x)=\{\mu^{1},...,\mu^{p}\}\subset\{\mu^{1},...,\mu^{s}\}=Or(x)
\]
and the previous inclusion holds as equality in the situation when all the
non-empty sets $\mathbf{T}_{\mu^{i}}^{x}$ are unbounded from above, i.e. when
$p=s.$

$\widehat{\omega}(\widehat{\sigma}^{\widetilde{k}}(\widehat{x}))=\widehat
{\omega}(\widehat{x})$ is a consequence of the fact that for any $\mu
\in\mathbf{B}^{n},$ the sets $\widehat{\mathbf{T}}_{\mu}^{\widehat{x}}$ and
$\widehat{\mathbf{T}}_{\mu}^{\widehat{\sigma}^{\widetilde{k}}(\widehat{x}%
)}=\widehat{\mathbf{T}}_{\mu}^{\widehat{x}}\cap\{\widetilde{k}-1,\widetilde
{k},\widetilde{k}+1,...\}$ are both finite (the empty sets are in this
situation) or infinite.

$\omega(\sigma^{\widetilde{t}}(x))=\omega(x)$ results from the fact that for
any $\mu\in\mathbf{B}^{n},$ the sets $\mathbf{T}_{\mu}^{x}$ and $\mathbf{T}%
_{\mu}^{\sigma^{\widetilde{t}}(x)}\supset\mathbf{T}_{\mu}^{x}\cap
\lbrack\widetilde{t},\infty)\footnote{If $\mu=x(\widetilde{t}-0)$ then the the
inclusion $\mathbf{T}_{\mu}^{\sigma^{\widetilde{t}}(x)}\supset\mathbf{T}_{\mu
}^{x}\cap\lbrack\widetilde{t},\infty)$ is strict, otherwise it takes place as
equality.}$ are both superiorly bounded (including the empty sets, that are
considered to have this property) or superiorly unbounded.

We prove $\widehat{Or}(\widehat{\sigma}^{\widetilde{k}}(\widehat{x}%
))\subset\widehat{Or}(x),Or(\sigma^{\widetilde{t}}(x))\subset Or(x)$ in the
following way:%
\[
\widehat{Or}(\widehat{\sigma}^{\widetilde{k}}(\widehat{x}))=\{\widehat{\sigma
}^{\widetilde{k}}(\widehat{x})(k)|k\in\mathbf{N}_{\_}\}=\{\widehat
{x}(k+\widetilde{k})|k\in\mathbf{N}_{\_}\}=
\]%
\[
=\{\widehat{x}(k)|k\geq\widetilde{k}-1\}\subset\{\widehat{x}(k)|k\in
\mathbf{N}_{\_}\}=\widehat{Or}(x)
\]
and on the other hand let $\varepsilon>0$ with $\forall\xi\in(\widetilde
{t}-\varepsilon,\widetilde{t}),x(\xi)=x(\widetilde{t}-0);$ then%
\[
Or(\sigma^{\widetilde{t}}(x))=\{\sigma^{\widetilde{t}}(x)(t)|t\in
\mathbf{R}\}=\{x(t)|t>\widetilde{t}-\varepsilon\}\subset\{x(t)|t\in
\mathbf{R}\}=Or(x).
\]

\end{proof}

\begin{theorem}
\label{The114}The signals $\widehat{x},x$ are given and we suppose that the
sequence $(t_{k})\in Seq$ exists such that%
\begin{equation}
x(t)=\widehat{x}(-1)\cdot\chi_{(-\infty,t_{0})}(t)\oplus\widehat{x}%
(0)\cdot\chi_{\lbrack t_{0},t_{1})}(t)\oplus...\oplus\widehat{x}(k)\cdot
\chi_{\lbrack t_{k},t_{k+1})}(t)\oplus... \label{per102}%
\end{equation}

a) We have $\widehat{Or}(\widehat{x})=Or(x)$ and $\widehat{\omega}(\widehat
{x})=\omega(x)$.

b) For any $\widetilde{k}\in\mathbf{N},\widetilde{t}\in\mathbf{R}$ we infer
$\widehat{\omega}(\widehat{\sigma}^{\widetilde{k}}(\widehat{x}))=\omega
(\sigma^{\widetilde{t}}(x));$ if either $\widetilde{k}=0,\widetilde{t}\leq
t_{0},$ or $\widetilde{k}\geq1,$ $\widetilde{t}\in(t_{\widetilde{k}%
-1},t_{\widetilde{k}}]$, then $\widehat{Or}(\widehat{\sigma}^{\widetilde{k}%
}(\widehat{x}))=Or(\sigma^{\widetilde{t}}(x))$.
\end{theorem}

\begin{proof}
a) We have%
\[
\widehat{Or}(\widehat{x})=\{\widehat{x}(k)|k\in\mathbf{N}_{\_}\}\overset
{(\ref{per102})}{=}\{x(t)|t\in\mathbf{R}\}=Or(x).
\]

In order to prove the second equality, let some arbitrary $\mu\in
\widehat{\omega}(\widehat{x}),$ thus the sequence $(k_{j})\in\widehat{Seq}$
exists with the property that%
\[
\forall j\in\mathbf{N}_{\_},\widehat{x}(k_{j})=\mu.
\]
For $x$ given by (\ref{per102}), we can define the unbounded from above
sequence%
\[
\forall j\in\mathbf{N}_{\_},t_{j+1}^{\prime}\overset{def}{=}t_{k_{j}},
\]
for which we get%
\[
\forall j\in\mathbf{N}_{\_},x(t_{j+1}^{\prime})=x(t_{k_{j}})=\widehat{x}%
(k_{j})=\mu,
\]
thus $\mu\in\omega(x)$ and $\widehat{\omega}(\widehat{x})\subset\omega(x).$
The inverse inclusion is proved similarly.

b) We fix $\widetilde{k}\in\mathbf{N},\widetilde{t}\in\mathbf{R}$ arbitrarily.
The first statement results from%
\[
\widehat{\omega}(\widehat{\sigma}^{\widetilde{k}}(\widehat{x}))\overset
{Theorem\;\ref{The12}}{=}\widehat{\omega}(\widehat{x})\overset{a)}{=}%
\omega(x)\overset{Theorem\;\ref{The12}}{=}\omega(\sigma^{\widetilde{t}}(x)).
\]

We prove the second statement. If $\widetilde{k}=0,\widetilde{t}\leq t_{0},$
then $\widehat{\sigma}^{\widetilde{k}}(\widehat{x})=\widehat{x}$ and
$\sigma^{\widetilde{t}}(x)=x,$ thus
\[
\widehat{Or}(\widehat{\sigma}^{\widetilde{k}}(\widehat{x}))=\widehat
{Or}(\widehat{x})\overset{a)}{=}Or(x)=Or(\sigma^{\widetilde{t}}(x)).
\]
We suppose from this moment that $\widetilde{k}\geq1,$ $\widetilde{t}%
\in(t_{\widetilde{k}-1},t_{\widetilde{k}}]$ hold. We conclude that
\[
\widehat{Or}(\widehat{\sigma}^{\widetilde{k}}(\widehat{x}))=\{\widehat
{x}(k)|k\geq\widetilde{k}-1\}\overset{(\ref{per102})}{=}\{x(t)|t\geq
t_{\widetilde{k}-1}\}=\{\sigma^{\widetilde{t}}(x)(t)|t\in\mathbf{R}%
\}=Or(\sigma^{\widetilde{t}}(x)).
\]

\end{proof}

\begin{theorem}
\label{The12_}For any $\widehat{x}\in\widehat{S}^{(n)},x\in S^{(n)}$ we have%
\[
\exists k^{\prime}\in\mathbf{N}_{\_},\forall k^{\prime\prime}\geq k^{\prime
},\widehat{\omega}(\widehat{x})=\{\widehat{x}(k)|k\geq k^{\prime\prime}\},
\]%
\[
\exists t^{\prime}\in\mathbf{R},\forall t^{\prime\prime}\geq t^{\prime}%
,\omega(x)=\{x(t)|t\geq t^{\prime\prime}\}.
\]

\end{theorem}

\begin{proof}
We denote once again the elements of $\mathbf{B}^{n}$ with $\mu^{1}%
,...,\mu^{2^{n}}.$ From the proof of Theorem \ref{The12}, if $\widehat
{\mathbf{T}}_{\mu^{1}}^{\widehat{x}},...,\widehat{\mathbf{T}}_{\mu^{p}%
}^{\widehat{x}}$ are the infinite sets $\widehat{\mathbf{T}}_{\mu^{i}%
}^{\widehat{x}},i=\overline{1,2^{n}}$ then $\widehat{\omega}(\widehat
{x})=\{\mu^{1},...,\mu^{p}\}.$ The number
\[
k^{\prime}=\left\{
\begin{array}
[c]{c}%
1+\max\{k|k\in\mathbf{N}_{\_},\widehat{x}(k)\in\widehat{Or}(\widehat
{x})\smallsetminus\widehat{\omega}(\widehat{x})\},\;if\;\widehat{Or}%
(\widehat{x})\smallsetminus\widehat{\omega}(\widehat{x})\neq\varnothing\\
-1,\;otherwise
\end{array}
\right.
\]
satisfies the property that $\forall i\in\{1,...,p\},\forall k^{\prime\prime
}\geq k^{\prime},\{k^{\prime\prime},k^{\prime\prime}+1,k^{\prime\prime
}+2,...\}\cap\widehat{\mathbf{T}}_{\mu^{i}}^{\widehat{x}}\neq\varnothing,$
thus $\forall k^{\prime\prime}\geq k^{\prime},$%
\[
\{\mu^{1},...,\mu^{p}\}=\{\widehat{x}(k)|k\in\widehat{\mathbf{T}}_{\mu^{1}%
}^{\widehat{x}}\cup...\cup\widehat{\mathbf{T}}_{\mu^{p}}^{\widehat{x}}\}
\]%
\[
=\{\widehat{x}(k)|k\in(\widehat{\mathbf{T}}_{\mu^{1}}^{\widehat{x}}\cup
...\cup\widehat{\mathbf{T}}_{\mu^{p}}^{\widehat{x}})\cap\{k^{\prime\prime
},k^{\prime\prime}+1,k^{\prime\prime}+2,...\}\}=\{\widehat{x}(k)|k\geq
k^{\prime\prime}\}.
\]

Similarly, if $\mathbf{T}_{\mu^{1}}^{x},...,\mathbf{T}_{\mu^{p}}^{x}$ are the
unbounded from above sets $\mathbf{T}_{\mu^{i}}^{x},i=\overline{1,2^{n}} $
then $\omega(x)=\{\mu^{1},...,\mu^{p}\}$ and, with the notation
\[
t^{\prime}=\left\{
\begin{array}
[c]{c}%
\sup\{t|t\in\mathbf{R},x(t)\in Or(x)\smallsetminus\omega
(x)\},\;if\;Or(x)\smallsetminus\omega(x)\neq\varnothing\\
0,\;otherwise
\end{array}
\right.
\]
we have that $\forall i\in\{1,...,p\},\forall t^{\prime\prime}\geq t^{\prime
},[t^{\prime\prime},\infty)\cap\mathbf{T}_{\mu^{i}}^{x}\neq\varnothing,$ thus
$\forall t^{\prime\prime}\geq t^{\prime},$%
\[
\{\mu^{1},...,\mu^{p}\}=\{x(t)|t\in\mathbf{T}_{\mu^{1}}^{x}\cup...\cup
\mathbf{T}_{\mu^{p}}^{x}\}
\]%
\[
=\{x(t)|t\in(\mathbf{T}_{\mu^{1}}^{x}\cup...\cup\mathbf{T}_{\mu^{p}}^{x}%
)\cap\lbrack t^{\prime\prime},\infty)\}=\{x(t)|t\geq t^{\prime\prime}\}.
\]

\end{proof}

\begin{remark}
Let the signals $\widehat{x}$ and $x$. If $\widehat{Or}(\widehat{x}%
)\neq\widehat{\omega}(\widehat{x}),$ the time instant $k^{\prime}\in
\mathbf{N}_{\_}$ exists that determines two time intervals for $\widehat{x}:$
$\{-1,0,...,k^{\prime}\}$ when $\widehat{x}$ can take values in any of
$\widehat{Or}(\widehat{x})\smallsetminus\widehat{\omega}(\widehat{x}%
),\widehat{\omega}(\widehat{x})$ and $\{k^{\prime}+1,k^{\prime}+2,...\}$ when
$\widehat{x}$ takes values in $\widehat{\omega}(\widehat{x})$ only. Similarly
for $x$, if $Or(x)\neq\omega(x),$ the time instant $t^{\prime}\in\mathbf{R}$
exists that determines two time intervals for $x:$ $(-\infty,t^{\prime})$ when
$x$ can take values in both sets $Or(x)\smallsetminus\omega(x),\omega(x)$ and
$[t^{\prime},\infty)$ when $x$ takes values in $\omega(x)$ only.
\end{remark}

\chapter{The main definitions on periodicity}

In this Chapter we list the main definitions on periodicity, that are
necessary in order to understand the rest of the exposure: the eventually
periodic points and the eventually periodic signals, the periodic points and
the periodic signals.

\section{Eventually periodic points}

\begin{definition}
\label{Def19}In case that, for $\mu\in\widehat{Or}(\widehat{x}),$ $p\geq1, $
some $k^{\prime}\in\mathbf{N}_{\_}$ exists such that we have%
\begin{equation}
\left\{
\begin{array}
[c]{c}%
\widehat{\mathbf{T}}_{\mu}^{\widehat{x}}\cap\{k^{\prime},k^{\prime
}+1,k^{\prime}+2,...\}\neq\varnothing\text{ and }\forall k\in\widehat
{\mathbf{T}}_{\mu}^{\widehat{x}}\cap\{k^{\prime},k^{\prime}+1,k^{\prime
}+2,...\},\\
\{k+zp|z\in\mathbf{Z}\}\cap\{k^{\prime},k^{\prime}+1,k^{\prime}+2,...\}\subset
\widehat{\mathbf{T}}_{\mu}^{\widehat{x}},
\end{array}
\right.  \label{pre738}%
\end{equation}
then $\mu$ is said to be \textbf{eventually periodic} (an \textbf{eventually
periodic point of} $\widehat{x},$ or \textbf{of} $\widehat{Or}(\widehat{x})$)
with the \textbf{period} $p$ and with the \textbf{limit of periodicity}
$k^{\prime}.$

Let $\mu\in Or(x)$ and $T>0$ such that $t^{\prime}\in\mathbf{R}$ exists with%
\begin{equation}
\mathbf{T}_{\mu}^{x}\cap\lbrack t^{\prime},\infty)\neq\varnothing\text{ and
}\forall t\in\mathbf{T}_{\mu}^{x}\cap\lbrack t^{\prime},\infty),\{t+zT|z\in
\mathbf{Z}\}\cap\lbrack t^{\prime},\infty)\subset\mathbf{T}_{\mu}^{x}.
\label{pre740}%
\end{equation}
Then $\mu$ is said to be \textbf{eventually periodic} (an \textbf{eventually
periodic point of} $x,$ or \textbf{of} $Or(x)$) with the \textbf{period} $T$
and with the \textbf{limit of periodicity} $t^{\prime}.$
\end{definition}

\begin{definition}
\label{Def45}The least $p,T$ that fulfill (\ref{pre738}), (\ref{pre740}) are
called \textbf{prime periods} (of $\mu$). For any $p,T,$ the least $k^{\prime
},t^{\prime}$ that fulfill (\ref{pre738}), (\ref{pre740}) are called
\textbf{prime limits of periodicity} (of $\mu$).
\end{definition}

\begin{notation}
\label{Not12}We use the notation $\widehat{P}_{\mu}^{\widehat{x}}$ for the set
of the periods of $\mu\in\widehat{Or}(\widehat{x}):$%
\[
\widehat{P}_{\mu}^{\widehat{x}}=\{p|p\geq1,\exists k^{\prime}\in
\mathbf{N}_{\_},(\ref{pre738})\text{ holds}\}\text{.}%
\]

The notation $P_{\mu}^{x}$ is used for the analogue set of the periods of
$\mu\in Or(x):$%
\[
P_{\mu}^{x}=\{T|T>0,\exists t^{\prime}\in\mathbf{R},(\ref{pre740})\text{
holds}\}.
\]

\end{notation}

\begin{notation}
\label{Not18}We denote with $\widehat{L}_{\mu}^{\widehat{x}}$ the set of the
limits of periodicity of $\mu\in\widehat{Or}(\widehat{x}):$%
\[
\widehat{L}_{\mu}^{\widehat{x}}=\{k^{\prime}|k^{\prime}\in\mathbf{N}%
_{\_},\exists p\geq1,(\ref{pre738})\text{ holds}\}
\]
and $L_{\mu}^{x}$ denotes the set of the limits of periodicity of $\mu\in
Or(x):$%
\[
L_{\mu}^{x}=\{t^{\prime}|t^{\prime}\in\mathbf{R},\exists T>0,(\ref{pre740}%
)\text{ is true}\}.
\]

\end{notation}

\begin{remark}
The eventual periodicity of $\mu\in\widehat{Or}(\widehat{x})$ with the period
$p$ and the limit of periodicity $k^{\prime}$ means a periodic behavior that
starts from $k^{\prime}$: for any $k\in\widehat{\mathbf{T}}_{\mu}^{\widehat
{x}}\cap\{k^{\prime},k^{\prime}+1,k^{\prime}+2,...\},$ we can go upwards and
downwards with multiples of $p$ to $k+zp,z\in\mathbf{Z}$ without getting out
of the 'final' time set $\{k^{\prime},k^{\prime}+1,k^{\prime}+2,...\}$ and we
still remain in $\widehat{\mathbf{T}}_{\mu}^{\widehat{x}}.$ In other words%
\[
\mu=\widehat{x}(k)=\widehat{x}(k-p)=\widehat{x}(k-2p)=...=\widehat{x}%
(k-k_{1}p),
\]
where $k_{1}\in\mathbf{N},$ fulfills $k-k_{1}p\geq k^{\prime},$ $k-(k_{1}%
+1)p<k^{\prime}$ and%
\[
\mu=\widehat{x}(k)=\widehat{x}(k+p)=\widehat{x}(k+2p)=...
\]

\end{remark}

\begin{remark}
The requirement $\widehat{\mathbf{T}}_{\mu}^{\widehat{x}}\cap\{k^{\prime
},k^{\prime}+1,k^{\prime}+2,...\}\neq\varnothing$ is one of non-triviality. It
is necessary, because for any point $\mu\in\widehat{Or}(\widehat
{x})\smallsetminus\widehat{\omega}(\widehat{x}),$ the set $\widehat
{\mathbf{T}}_{\mu}^{\widehat{x}}$ is finite, some $k^{\prime}\in
\mathbf{N}_{\_}$ exists such that $\widehat{\mathbf{T}}_{\mu}^{\widehat{x}%
}\cap\{k^{\prime},k^{\prime}+1,k^{\prime}+2,...\}=\varnothing$ and%
\[
\left\{
\begin{array}
[c]{c}%
\forall k\in\widehat{\mathbf{T}}_{\mu}^{\widehat{x}}\cap\{k^{\prime}%
,k^{\prime}+1,k^{\prime}+2,...\},\\
\{k+zp|z\in\mathbf{Z}\}\cap\{k^{\prime},k^{\prime}+1,k^{\prime}+2,...\}\subset
\widehat{\mathbf{T}}_{\mu}^{\widehat{x}}%
\end{array}
\right.  ,
\]
equivalent with%
\[
\forall k,k\in\varnothing\Longrightarrow\{k+zp|z\in\mathbf{Z}\}\cap
\{k^{\prime},k^{\prime}+1,k^{\prime}+2,...\}\subset\widehat{\mathbf{T}}_{\mu
}^{\widehat{x}},
\]
is true, $\forall p\geq1.$
\end{remark}

\begin{remark}
The eventually periodic points $\mu\in\widehat{Or}(\widehat{x})$ are omega
limit points $\mu\in\widehat{\omega}(\widehat{x})$ because the set
$\widehat{\mathbf{T}}_{\mu}^{\widehat{x}}$ is necessarily infinite.
\end{remark}

\begin{remark}
Definition \ref{Def19} avoids the triviality expressed by the possibility
$\widehat{\mathbf{T}}_{\mu}^{\widehat{x}}\cap\{k^{\prime},k^{\prime
}+1,k^{\prime}+2,...\}=\varnothing$, but a way of obtaining the same result is
to ask $\mu\in\widehat{\omega}(\widehat{x})$ instead of $\mu\in\widehat
{Or}(\widehat{x}),$ see Lemma \ref{Lem37}, page \pageref{Lem37}, since in that
case we have that $\widehat{\mathbf{T}}_{\mu}^{\widehat{x}}$ is infinite and
$\forall k^{\prime}\in\mathbf{N}_{\_},\widehat{\mathbf{T}}_{\mu}^{\widehat{x}%
}\cap\{k^{\prime},k^{\prime}+1,k^{\prime}+2,...\}\neq\varnothing.$ With this
note, the discrete time part of Definition \ref{Def19} becomes, equivalently:
$\mu\in\widehat{\omega}(\widehat{x})$ is eventually periodic with the period
$p$ and the limit of periodicity $k^{\prime}$ if%
\[
\forall k\in\widehat{\mathbf{T}}_{\mu}^{\widehat{x}}\cap\{k^{\prime}%
,k^{\prime}+1,k^{\prime}+2,...\},\{k+zp|z\in\mathbf{Z}\}\cap\{k^{\prime
},k^{\prime}+1,k^{\prime}+2,...\}\subset\widehat{\mathbf{T}}_{\mu}%
^{\widehat{x}}.
\]

\end{remark}

\begin{remark}
The eventual periodicity of $\mu\in Or(x)$ with the period $T$ and the limit
of periodicity $t^{\prime}$ means periodicity that starts from $t^{\prime}%
\in\mathbf{R}$: for any $t\in\mathbf{T}_{\mu}^{x}\cap\lbrack t^{\prime}%
,\infty)$ we can go arbitrarily upwards and downwards with multiples of $T,$
to $t+zT,z\in\mathbf{Z}$ without leaving the 'final' time set $[t^{\prime
},\infty)$ and we still remain in $\mathbf{T}_{\mu}^{x}.$
\end{remark}

\begin{remark}
The requirement $\mathbf{T}_{\mu}^{x}\cap\lbrack t^{\prime},\infty
)\neq\varnothing$ in (\ref{pre740}) is one of non-triviality. An equivalent
way of obtaining non-triviality is to ask $\mu\in\omega(x)$ and to replace
(\ref{pre740}) with%
\[
\forall t\in\mathbf{T}_{\mu}^{x}\cap\lbrack t^{\prime},\infty),\{t+zT|z\in
\mathbf{Z}\}\cap\lbrack t^{\prime},\infty)\subset\mathbf{T}_{\mu}^{x}.
\]

\end{remark}

\begin{remark}
The eventual periodicity of $\mu\in Or(x)$ obviously implies that $\mu
\in\omega(x),$ because the set $\mathbf{T}_{\mu}^{x}$ is superiorly unbounded.
\end{remark}

\begin{remark}
We have $\widehat{P}_{\mu}^{\widehat{x}}\neq\varnothing\Longleftrightarrow
\widehat{L}_{\mu}^{\widehat{x}}\neq\varnothing$ and $P_{\mu}^{x}%
\neq\varnothing\Longleftrightarrow L_{\mu}^{x}\neq\varnothing.$
\end{remark}

\begin{example}
The signal $\widehat{x}\in\widehat{S}^{(2)}$ with $\widehat{\mathbf{T}%
}_{(1,1)}^{\widehat{x}}=\{1,3,5,...\}$ fulfills the property that $(1,1)$ is
eventually periodic with the period $2$ and the limit of periodicity
$k^{\prime}=0.$
\end{example}

\begin{example}
Let $\widehat{x}\in\widehat{S}^{(2)}$ arbitrary with $(1,1)\notin\widehat
{Or}(\widehat{x})$ and $\widehat{x}(-1)\neq\widehat{x}(0).$ $y\in S^{(2)}$ is
defined like this:%
\[
y(t)=\widehat{x}(-1)\cdot\chi_{(-\infty,0)}(t)\oplus\widehat{x}(0)\cdot
\chi_{\lbrack0,1)}(t)\oplus(1,1)\cdot\chi_{\lbrack1,2)}(t)\oplus\widehat
{x}(2)\cdot\chi_{\lbrack2,3)}(t)\oplus
\]%
\[
\oplus(1,1)\cdot\chi_{\lbrack3,4)}(t)\oplus\widehat{x}(4)\cdot\chi
_{\lbrack4,5)}(t)\oplus(1,1)\cdot\chi_{\lbrack5,6)}(t)\oplus...
\]
The point $(1,1)$ is an eventually periodic point of $y$ with the period $T=2
$ and any $t^{\prime}\in\lbrack0,\infty)$ is a limit of periodicity. The
situation $\widehat{x}(-1)=\widehat{x}(0)$ generates a special case called
periodicity, that will be analyzed later and the situation $(1,1)\in
\widehat{Or}(\widehat{x})$ might generate several possibilities, for example
$y$ has the period $p=1$ or $y$ changes its limit of periodicity.
\end{example}

\section{Eventually periodic signals}

\begin{definition}
\label{Def28}For $p\geq1$ and $k^{\prime}\in\mathbf{N}_{\_},$ if%
\begin{equation}
\forall k\geq k^{\prime},\widehat{x}(k)=\widehat{x}(k+p), \label{pre742}%
\end{equation}
we say that $\widehat{x}$ is \textbf{eventually periodic} with the
\textbf{period} $p$ and the \textbf{limit of periodicity} $k^{\prime}.$

Let $T>0.$ If $t^{\prime}\in\mathbf{R}$ exists such that%
\begin{equation}
\forall t\geq t^{\prime},x(t)=x(t+T) \label{pre744}%
\end{equation}
is true, we say that $x$ is \textbf{eventually periodic} with the
\textbf{period} $T$ and the \textbf{limit of periodicity} $t^{\prime}.$
\end{definition}

\begin{definition}
\label{Def46}The least $p,T$ that fulfill (\ref{pre742}), (\ref{pre744}) are
called \textbf{prime periods} (of $\widehat{x},x$) and the least $k^{\prime
},t^{\prime}$ that fulfill (\ref{pre742}), (\ref{pre744}) are called
\textbf{prime limits of periodicity} (of $\widehat{x},x).$
\end{definition}

\begin{notation}
\label{Not19}We use the notation $\widehat{P}^{\widehat{x}}$ for the set of
the periods of $\widehat{x}:$%
\[
\widehat{P}^{\widehat{x}}=\{p|p\geq1,\exists k^{\prime}\in\mathbf{N}%
_{\_},(\ref{pre742})\text{ holds}\}
\]
and also the notation $P^{x}$ for the set of the periods of $x:$%
\[
P^{x}=\{T|T>0,\exists t^{\prime}\in\mathbf{R},(\ref{pre744})\text{ holds}\}.
\]

\end{notation}

\begin{notation}
\label{Not20}We use the notations%
\[
\widehat{L}^{\widehat{x}}=\{k^{\prime}|k^{\prime}\in\mathbf{N}_{\_},\exists
p\geq1,(\ref{pre742})\text{ holds}\},
\]%
\[
L^{x}=\{t^{\prime}|t^{\prime}\in\mathbf{R},\exists T>0,(\ref{pre744})\text{
holds}\}.
\]

\end{notation}

\begin{remark}
The eventual periodicity of $\widehat{x}$ with the period $p$ and the limit of
periodicity $k^{\prime}$ means that all the values $\mu\in\widehat{\omega
}(\widehat{x})$ are eventually periodic with the same period $p$ and with the
same limit of periodicity $k^{\prime}.$
\end{remark}

\begin{remark}
The signal $x$ is eventually periodic with the period $T$ and the limit of
periodicity $t^{\prime}$ if all the values $\mu\in\omega(x)$ are eventually
periodic with the same period $T$ and with the same limit of periodicity
$t^{\prime}.$
\end{remark}

\begin{remark}
We see that $\widehat{P}^{\widehat{x}}\neq\varnothing\Longleftrightarrow
\widehat{L}^{\widehat{x}}\neq\varnothing$ and $P^{x}\neq\varnothing
\Longleftrightarrow L^{x}\neq\varnothing.$
\end{remark}

\begin{example}
\label{Exa17}The signal $\widehat{x}\in\widehat{S}^{(1)}$ defined by
$\widehat{x}=0,1,1,1,...$ is eventually constant with $\widehat{F}%
^{\widehat{x}}=\mathbf{N}$. It is eventually periodic with the period $p=1$
and the limit of periodicity $k^{\prime}=0.$
\end{example}

\begin{example}
\label{Exa18}The real time analogue of the previous example is given by $x\in
S^{(1)},x(t)=\chi_{\lbrack0,\infty)}(t).$ The signal $x$ is eventually
constant and eventually periodic, with the arbitrary period $T>0.$ We have
$I^{x}=(-\infty,0)$ and $F^{x}=L^{x}=[0,\infty).$
\end{example}

\section{Periodic points}

\begin{definition}
\label{Def47}We consider the signals $\widehat{x}\in\widehat{S}^{(n)},x\in
S^{(n)}.$

Let $\mu\in\widehat{Or}(\widehat{x})$ and $p\geq1.$ If%
\begin{equation}
\forall k\in\widehat{\mathbf{T}}_{\mu}^{\widehat{x}},\{k+zp|z\in
\mathbf{Z}\}\cap\mathbf{N}_{\_}\subset\widehat{\mathbf{T}}_{\mu}^{\widehat{x}%
}, \label{pre737}%
\end{equation}
we say that $\mu$ is \textbf{periodic} (a \textbf{periodic point of}
$\widehat{x}$, or \textbf{of} $\widehat{Or}(\widehat{x})$) with the
\textbf{period} $p$.

Let $\mu\in Or(x)$ and $T>0$ such that $t^{\prime}\in I^{x}$ exists with%
\begin{equation}
\forall t\in\mathbf{T}_{\mu}^{x}\cap\lbrack t^{\prime},\infty),\{t+zT|z\in
\mathbf{Z}\}\cap\lbrack t^{\prime},\infty)\subset\mathbf{T}_{\mu}^{x}\text{.}
\label{pre739}%
\end{equation}
Then $\mu~$is called \textbf{periodic} (a \textbf{periodic point of} $x$, or
\textbf{of} $Or(x)$) with the \textbf{period} $T.$
\end{definition}

\begin{remark}
\label{Rem18}The periodicity of $\mu\in\widehat{Or}(\widehat{x})$ with the
period $p\geq1$ means eventual periodicity that starts at the limit of
periodicity $k^{\prime}=-1.$ The property is non-trivial since $\mu\in
\widehat{Or}(\widehat{x})$ implies $\varnothing\neq\widehat{\mathbf{T}}_{\mu
}^{\widehat{x}}=\widehat{\mathbf{T}}_{\mu}^{\widehat{x}}\cap\{-1,0,1,...\}.$
\end{remark}

\begin{remark}
The periodicity of $\mu\in Or(x)$ with the period $T>0$ means eventual
periodicity with the property that the limit of periodicity $t^{\prime}$ is an
initial time instant of $x$ also. The property is non-trivial as far as
$\mathbf{T}_{\mu}^{x}\cap\lbrack t^{\prime},\infty)\neq\varnothing$ results
from Lemma \ref{Lem36}, page \pageref{Lem36}.
\end{remark}

\begin{remark}
Because the periodicity of $\mu$ is a special case of eventual periodicity,
the concepts of prime period, prime limit of periodicity and the notations
$\widehat{P}_{\mu}^{\widehat{x}},P_{\mu}^{x},\widehat{L}_{\mu}^{\widehat{x}%
},L_{\mu}^{x}$ are used for the periodic points also, with the remark that
$\widehat{L}_{\mu}^{\widehat{x}}=\mathbf{N}_{\_},$ $L_{\mu}^{x}\cap I^{x}%
\neq\varnothing.$
\end{remark}

\begin{remark}
\label{Rem27}The periodic points are omega limit points. On one hand even if
there is a periodic point, omega limit points might exist that are not
periodic and on the other hand when stating periodicity we must not ask
$\mu\in\widehat{\omega}(\widehat{x}),\mu\in\omega(x)$ because triviality is impossible.
\end{remark}

\begin{remark}
Mentioning the limit of periodicity in case of periodicity is not necessary:
in the discrete time case because $k^{\prime}=-1$ is always clear and in the
real time case because the property of periodicity does not depend on the
choice of $t^{\prime},$ as we shall see later.
\end{remark}

\begin{example}
Let $\widehat{x}\in\widehat{S}^{(n)},$ $\mu\in\widehat{Or}(\widehat{x})$ with
$\widehat{\mathbf{T}}_{\mu}^{\widehat{x}}=\{-1,1,3,5,...\},$ thus the point
$\mu$ is periodic with the period $p=2.$

We define $x\in S^{(n)}$ by%
\[
x(t)=\widehat{x}(-1)\cdot\chi_{(-\infty,0)}(t)\oplus\widehat{x}(0)\cdot
\chi_{\lbrack0,1)}(t)\oplus...\oplus\widehat{x}(k)\cdot\chi_{\lbrack
k,k+1)}(t)\oplus...
\]
We have $x(-\infty+0)=\mu$ and $\mathbf{T}_{x(-\infty+0)}^{x}=\mathbf{T}_{\mu
}^{x}=(-\infty,0)\cup\lbrack1,2)\cup\lbrack3,4)\cup...$ For any $t^{\prime}%
\in\lbrack-1,0),$ we infer the truth of $(-\infty,t^{\prime}]\subset
\mathbf{T}_{x(-\infty+0)}^{x},$ $\mathbf{T}_{\mu}^{x}\cap\lbrack t^{\prime
},\infty)=[t^{\prime},0)\cup\lbrack1,2)\cup\lbrack3,4)\cup...$ and
\[
\forall t\in\lbrack t^{\prime},0)\cup\lbrack1,2)\cup\lbrack3,4)\cup...,
\]%
\[
\{t+z2|z\in\mathbf{Z}\}\cap\lbrack t^{\prime},\infty)\subset(-\infty
,0)\cup\lbrack1,2)\cup\lbrack3,4)\cup...
\]
$\mu$ has the period $T=2$.
\end{example}

\section{Periodic signals}

\begin{definition}
\label{Def18}Let $\widehat{x}\in\widehat{S}^{(n)},x\in S^{(n)}$ and
$p\geq1,T>0.$

If%
\begin{equation}
\forall k\in\mathbf{N}_{\_},\widehat{x}(k)=\widehat{x}(k+p), \label{pre741}%
\end{equation}
we say that $\widehat{x}$ is \textbf{periodic} with the \textbf{period} $p$.

In case that $\exists t^{\prime}\in I^{x}$,%
\begin{equation}
\forall t\geq t^{\prime},x(t)=x(t+T) \label{pre743}%
\end{equation}
holds, we say that $x$ is \textbf{periodic} with the \textbf{period} $T.$
\end{definition}

\begin{remark}
If $\widehat{x}$ is periodic with the period $p$ then all its values $\mu
\in\widehat{Or}(\widehat{x})$ are periodic with the period $p.$ This means in
particular that the periodicity of $\widehat{x}$ implies $\widehat
{Or}(\widehat{x})=\widehat{\omega}(\widehat{x}).$
\end{remark}

\begin{remark}
If the signal $x$ is periodic with the period $T$ then all the values $\mu\in
Or(x)$ are periodic with the same period $T$. Note that $Or(x)=\omega(x).$
\end{remark}

\begin{remark}
The periodic signals are special cases of eventually periodic signals when
$k^{\prime}=-1$ instead of $k^{\prime}\in\widehat{L}^{\widehat{x}},$
respectively when $t^{\prime}\in I^{x}\cap L^{x},$ instead of $t^{\prime}\in
L^{x}.$ In particular the concepts of prime period, prime limit of periodicity
and the notations $\widehat{P}^{\widehat{x}},P^{x},\widehat{L}^{\widehat{x}%
},L^{x}$ are used for the periodic signals too. We have $\widehat{L}%
^{\widehat{x}}=\mathbf{N}_{\_},L^{x}\cap I^{x}\neq\varnothing.$
\end{remark}

\begin{remark}
Mentioning the limit of periodicity $k^{\prime},t^{\prime}$ in Definition
\ref{Def18} is not necessary, since the property itself does not depend on the
choice of $k^{\prime},t^{\prime}.$
\end{remark}

\begin{example}
The signal $\widehat{x}\in\widehat{S}^{(1)}$ given by $\widehat{x}%
=1,0,1,0,1,...$ is periodic with the period $2$. $\widehat{Or}(\widehat
{x})=\{0,1\}$ and both points $0,1$ are periodic with the period $2$.
\end{example}

\begin{example}
The signal $x\in S^{(1)}$ that is defined in the following way:%
\[
x(t)=\chi_{(-\infty,0)}(t)\oplus\chi_{\lbrack1,2)}(t)\oplus\chi_{\lbrack
3,4)}(t)\oplus...
\]
has the period $2$ if we take the initial time=limit of periodicity
$t^{\prime}\in\lbrack-1,0).$ If we take $t^{\prime}<-1$ then (\ref{pre743})
does not hold, i.e. $t^{\prime}$ is not limit of periodicity; if we take
$t^{\prime}\geq0,$ then $t^{\prime}$ is not initial time.
\end{example}

\chapter{\label{Cha3}Eventually constant signals}

The purpose of the Chapter is that of giving properties that are equivalent
with the eventual constancy of the signals, a concept that is anticipated in
Chapter 1, Definition \ref{Def7}, page \pageref{Def7} and the following
paragraphs and in Chapter 2, Example \ref{Exa17} and Example \ref{Exa18}, page
\pageref{Exa17}. The importance of eventual constancy is that of being related
with the stability of the asynchronous systems\footnote{It is not the purpose
of this monograph to address the stability of the systems.}.

The first group of eventual constancy properties of Section 1 does not involve
periodicity. The groups 2 and 3 are related with the eventual periodicity of
the points and they are introduced in Sections 3, 4 and 5. The group 4 of
eventual constancy properties is related with the eventual periodicity of the
signals and it is introduced in Section 6. Section 7 shows the connection
between discrete time and continuous time as far as eventual constancy is
concerned and Section 8 contains a discussion.

\section{The first group of eventual constancy properties}

\begin{remark}
The first group of eventual constancy properties of the signals contains these
properties that are not related with periodicity.
\end{remark}

\begin{theorem}
\label{The15}Let the signals $\widehat{x}\in\widehat{S}^{(n)},x\in S^{(n)}.$

a) The statements%
\begin{equation}
\exists\mu\in\mathbf{B}^{n},\exists k^{\prime}\in\mathbf{N}_{\_},\forall k\geq
k^{\prime},\widehat{x}(k)=\mu, \label{per48}%
\end{equation}%
\begin{equation}
\exists\mu\in\mathbf{B}^{n},\exists k^{\prime}\in\mathbf{N}_{\_},\{k^{\prime
},k^{\prime}+1,k^{\prime}+2,...\}\subset\widehat{\mathbf{T}}_{\mu}%
^{\widehat{x}}, \label{per127}%
\end{equation}%
\begin{equation}
\exists\mu\in\mathbf{B}^{n},\widehat{\omega}(\widehat{x})=\{\mu\}
\label{per80}%
\end{equation}
are equivalent.

b) The statements%
\begin{equation}
\exists\mu\in\mathbf{B}^{n},\exists t^{\prime}\in\mathbf{R},\forall t\geq
t^{\prime},x(t)=\mu, \label{per49}%
\end{equation}%
\begin{equation}
\exists\mu\in\mathbf{B}^{n},\exists t^{\prime}\in\mathbf{R},[t^{\prime}%
,\infty)\subset\mathbf{T}_{\mu}^{x}, \label{per128}%
\end{equation}%
\begin{equation}
\exists\mu\in\mathbf{B}^{n},\omega(x)=\{\mu\} \label{per79}%
\end{equation}
are also equivalent.
\end{theorem}

\begin{proof}
a) (\ref{per48})$\Longrightarrow$(\ref{per127}) $\mu\in\mathbf{B}^{n}$ and
$k^{\prime}\in\mathbf{N}_{\_}$ exist with the property%
\[
\forall k\geq k^{\prime},\widehat{x}(k)=\mu.
\]
Then
\begin{equation}
\{k^{\prime},k^{\prime}+1,k^{\prime}+2,...\}\subset\{k|k\in\mathbf{N}%
_{\_},\widehat{x}(k)=\mu\} \label{pre516}%
\end{equation}
holds.

(\ref{per127})$\Longrightarrow$(\ref{per80}) $\mu\in\mathbf{B}^{n}$ and
$k^{\prime}\in\mathbf{N}_{\_}$ exist such that (\ref{pre516}) holds. We
suppose, see Theorem \ref{The12_}, page \pageref{The12_}, that $k^{\prime
\prime}\in\mathbf{N}_{\_}$ fulfills%
\begin{equation}
\{\widehat{x}(k)|k\geq k^{\prime\prime}\}=\widehat{\omega}(\widehat{x}).
\label{pre517}%
\end{equation}
For $k_{1}=\max\{k^{\prime},k^{\prime\prime}\}$ we can write that%
\[
\{\mu\}\overset{(\ref{pre516})}{=}\{\widehat{x}(k)|k\geq k_{1}\}\overset
{(\ref{pre517})}{=}\widehat{\omega}(\widehat{x}).
\]

(\ref{per80})$\Longrightarrow$(\ref{per48}) From (\ref{per80}) and Theorem
\ref{The12_}, page \pageref{The12_} we have the existence of $k^{\prime}%
\in\mathbf{N}$ such that%
\[
\{\mu\}=\widehat{\omega}(\widehat{x})=\{\widehat{x}(k)|k\geq k^{\prime}\},
\]
wherefrom the truth of (\ref{per48}).

b) (\ref{per49})$\Longrightarrow$(\ref{per128}) We suppose that $\mu
\in\mathbf{B}^{n}$ and $t^{\prime}\in\mathbf{R}$ exist such that $\forall
t\geq t^{\prime},x(t)=\mu.$ Then $[t^{\prime},\infty)\subset\{t|t\in
\mathbf{R},x(t)=\mu\}.$

(\ref{per128})$\Longrightarrow$(\ref{per79}) Some $t_{1}\in\mathbf{R}$ exists
satisfying $\{x(t)|t\geq t_{1}\}=\omega(x)$ and, from the hypothesis, $\mu
\in\mathbf{B}^{n}$ and $t^{\prime}\in\mathbf{R}$ exist such that $[t^{\prime
},\infty)\subset\{t|t\in\mathbf{R},x(t)=\mu\}.$ We use the notation
$t^{\prime\prime}=\max\{t_{1},t^{\prime}\}$ and we have%
\[
\{\mu\}=\{x(t)|t\geq t^{\prime\prime}\}=\omega(x).
\]

(\ref{per79})$\Longrightarrow$(\ref{per49}) The hypothesis (\ref{per79}) and
Theorem \ref{The12_} show the existence of $\mu\in\mathbf{B}^{n},t^{\prime}%
\in\mathbf{R}$ with%
\[
\{\mu\}=\omega(x)=\{x(t)|t\geq t^{\prime}\},
\]
wherefrom the truth of (\ref{per49}).
\end{proof}

\section{Eventual constancy}

\begin{definition}
\label{Def15}If $\widehat{x}\in\widehat{S}^{(n)}$ fulfills one of
(\ref{per48}),...,(\ref{per80}), it is called \textbf{eventually constant} and
if $x\in S^{(n)}$ fulfills one of (\ref{per49}),...,(\ref{per79}), it is
called \textbf{eventually constant. }In (\ref{per48}), (\ref{per127}),
$k^{\prime}\in\mathbf{N}_{\_}$ is called the \textbf{limit of constancy}, or
\textbf{limit of equilibrium}, or \textbf{final time} of $\widehat{x}$.
Similarly in (\ref{per49}), (\ref{per128}), $t^{\prime}\in\mathbf{R}$ is
called the \textbf{limit of constancy}, or \textbf{limit of equilibrium}, or
\textbf{final time} of $x$.
\end{definition}

\begin{theorem}
a) If $\widehat{x}$ is constant, it is eventually constant.

b) If $x$ is constant, it is eventually constant.
\end{theorem}

\begin{proof}
a) The constancy of $\widehat{x}$ means the eventual constancy of $\widehat
{x}$ with the limit of constancy $k^{\prime}=-1.$

b) The constancy of $x$ is its eventual constancy with the limit of constancy
$t^{\prime}\in I^{x}.$
\end{proof}

\begin{remark}
The eventual constancy of a signal coincides with the existence of the final
value, Definition \ref{Def7}, page \pageref{Def7}. This is the reason why in
Definition \ref{Def15} $k^{\prime}$ and $t^{\prime}$ are also called final time.
\end{remark}

\begin{remark}
Eventual constancy is important in systems theory since it is associated with
stability: if modeling is deterministic and the signal is an asynchronous
flow, then stability means exactly the eventual constancy of that flow; and if
modeling is non-deterministic and we have a set of deterministic flows, then
stability means the eventual constancy of all these flows.
\end{remark}

\section{The second group of eventual constancy properties}

\begin{remark}
This group of eventual constancy properties of the signals involves eventual
periodicity of all the points $\mu$ of the orbit, i.e. in (\ref{pre781}%
),...,(\ref{pre538}), (\ref{pre553}),...,(\ref{pre600}) to follow we ask
$\forall\mu\in\widehat{Or}(\widehat{x}),\forall\mu\in Or(x).$
\end{remark}

\begin{remark}
In order to understand better the way that these properties were written, to
be noticed the existence of the following symmetries:

- (\ref{pre781})-\{(\ref{pre553}),(\ref{pre597})\}; (\ref{pre514}%
)-\{(\ref{pre554}),(\ref{pre598})\}; (\ref{pre524})-\{(\ref{pre555}%
),(\ref{pre599})\}; (\ref{pre538})-\{(\ref{pre556}), (\ref{pre600})\};

- (\ref{pre781})-(\ref{pre514}), (\ref{pre524})-(\ref{pre538}) and
(\ref{pre553})-(\ref{pre597})-(\ref{pre554})-(\ref{pre598}), (\ref{pre555}%
)-(\ref{pre599})-(\ref{pre556})-(\ref{pre600});

- (\ref{pre781})-(\ref{pre524}), (\ref{pre514})-(\ref{pre538}) and
(\ref{pre553})-(\ref{pre597})-(\ref{pre555})-(\ref{pre599}), (\ref{pre554}%
)-(\ref{pre598})-(\ref{pre556})-(\ref{pre600}).
\end{remark}

\begin{theorem}
\label{The97}Let the signals $\widehat{x}\in\widehat{S}^{(n)},x\in S^{(n)}.$

a) The following statements are equivalent with the eventual constancy of
$\widehat{x}:$%
\begin{equation}
\left\{
\begin{array}
[c]{c}%
\forall p\geq1,\forall\mu\in\widehat{Or}(\widehat{x}),\exists k^{\prime}%
\in\mathbf{N}_{\_},\forall k\in\widehat{\mathbf{T}}_{\mu}^{\widehat{x}}%
\cap\{k^{\prime},k^{\prime}+1,k^{\prime}+2,...\},\\
\{k+zp|z\in\mathbf{Z}\}\cap\{k^{\prime},k^{\prime}+1,k^{\prime}+2,...\}\subset
\widehat{\mathbf{T}}_{\mu}^{\widehat{x}},
\end{array}
\right.  \label{pre781}%
\end{equation}%
\begin{equation}
\left\{
\begin{array}
[c]{c}%
\forall p\geq1,\forall\mu\in\widehat{Or}(\widehat{x}),\exists k^{\prime\prime
}\in\mathbf{N},\forall k\in\widehat{\mathbf{T}}_{\mu}^{\widehat{\sigma
}^{k^{\prime\prime}}(\widehat{x})},\\
\{k+zp|z\in\mathbf{Z}\}\cap\mathbf{N}_{\_}\subset\widehat{\mathbf{T}}_{\mu
}^{\widehat{\sigma}^{k^{\prime\prime}}(\widehat{x})},
\end{array}
\right.  \label{pre514}%
\end{equation}%
\begin{equation}
\left\{
\begin{array}
[c]{c}%
\forall p\geq1,\forall\mu\in\widehat{Or}(\widehat{x}),\exists k^{\prime}%
\in\mathbf{N}_{\_},\forall k\geq k^{\prime},\widehat{x}(k)=\mu\Longrightarrow
\\
\Longrightarrow(\widehat{x}(k)=\widehat{x}(k+p)\text{ and }k-p\geq k^{\prime
}\Longrightarrow\widehat{x}(k)=\widehat{x}(k-p)),
\end{array}
\right.  \label{pre524}%
\end{equation}%
\begin{equation}
\left\{
\begin{array}
[c]{c}%
\forall p\geq1,\forall\mu\in\widehat{Or}(\widehat{x}),\exists k^{\prime\prime
}\in\mathbf{N},\forall k\in\mathbf{N}_{\_},\widehat{\sigma}^{k^{\prime\prime}%
}(\widehat{x})(k)=\mu\Longrightarrow\\
\Longrightarrow(\widehat{\sigma}^{k^{\prime\prime}}(\widehat{x})(k)=\widehat
{\sigma}^{k^{\prime\prime}}(\widehat{x})(k+p)\text{ and }\\
\text{and }k-p\geq-1\Longrightarrow\widehat{\sigma}^{k^{\prime\prime}%
}(\widehat{x})(k)=\widehat{\sigma}^{k^{\prime\prime}}(\widehat{x})(k-p)).
\end{array}
\right.  \label{pre538}%
\end{equation}

b) The following statements are equivalent with the eventual constancy of $x$:%
\begin{equation}
\left\{
\begin{array}
[c]{c}%
\forall T>0,\forall\mu\in Or(x),\exists t^{\prime}\in I^{x},\\
\exists t_{1}^{\prime}\geq t^{\prime},\forall t\in\mathbf{T}_{\mu}^{x}%
\cap\lbrack t_{1}^{\prime},\infty),\{t+zT|z\in\mathbf{Z}\}\cap\lbrack
t_{1}^{\prime},\infty)\subset\mathbf{T}_{\mu}^{x},
\end{array}
\right.  \label{pre553}%
\end{equation}%
\begin{equation}
\left\{
\begin{array}
[c]{c}%
\forall T>0,\forall\mu\in Or(x),\exists t_{1}^{\prime}\in\mathbf{R},\\
\forall t\in\mathbf{T}_{\mu}^{x}\cap\lbrack t_{1}^{\prime},\infty
),\{t+zT|z\in\mathbf{Z}\}\cap\lbrack t_{1}^{\prime},\infty)\subset
\mathbf{T}_{\mu}^{x},
\end{array}
\right.  \label{pre597}%
\end{equation}%
\begin{equation}
\left\{
\begin{array}
[c]{c}%
\forall T>0,\forall\mu\in Or(x),\exists t^{\prime\prime}\in\mathbf{R},\exists
t^{\prime}\in I^{\sigma^{t^{\prime\prime}}(x)},\\
\forall t\in\mathbf{T}_{\mu}^{\sigma^{t^{\prime\prime}}(x)}\cap\lbrack
t^{\prime},\infty),\{t+zT|z\in\mathbf{Z}\}\cap\lbrack t^{\prime}%
,\infty)\subset\mathbf{T}_{\mu}^{\sigma^{t^{\prime\prime}}(x)},
\end{array}
\right.  \label{pre554}%
\end{equation}%
\begin{equation}
\left\{
\begin{array}
[c]{c}%
\forall T>0,\forall\mu\in Or(x),\exists t^{\prime\prime}\in\mathbf{R},\exists
t^{\prime}\in\mathbf{R},\\
\forall t\in\mathbf{T}_{\mu}^{\sigma^{t^{\prime\prime}}(x)}\cap\lbrack
t^{\prime},\infty),\{t+zT|z\in\mathbf{Z}\}\cap\lbrack t^{\prime}%
,\infty)\subset\mathbf{T}_{\mu}^{\sigma^{t^{\prime\prime}}(x)},
\end{array}
\right.  \label{pre598}%
\end{equation}%
\begin{equation}
\left\{
\begin{array}
[c]{c}%
\forall T>0,\forall\mu\in Or(x),\exists t^{\prime}\in I^{x},\exists
t_{1}^{\prime}\geq t^{\prime},\forall t\geq t_{1}^{\prime},x(t)=\mu
\Longrightarrow\\
\Longrightarrow(x(t)=x(t+T)\text{ and }t-T\geq t_{1}^{\prime}\Longrightarrow
x(t)=x(t-T)),
\end{array}
\right.  \label{pre555}%
\end{equation}%
\begin{equation}
\left\{
\begin{array}
[c]{c}%
\forall T>0,\forall\mu\in Or(x),\exists t_{1}^{\prime}\in\mathbf{R},\forall
t\geq t_{1}^{\prime},x(t)=\mu\Longrightarrow\\
\Longrightarrow(x(t)=x(t+T)\text{ and }t-T\geq t_{1}^{\prime}\Longrightarrow
x(t)=x(t-T)),
\end{array}
\right.  \label{pre599}%
\end{equation}%
\begin{equation}
\left\{
\begin{array}
[c]{c}%
\forall T>0,\forall\mu\in Or(x),\exists t^{\prime\prime}\in\mathbf{R},\exists
t^{\prime}\in I^{\sigma^{t^{\prime\prime}}(x)},\\
\forall t\geq t^{\prime},\sigma^{t^{\prime\prime}}(x)(t)=\mu\Longrightarrow
(\sigma^{t^{\prime\prime}}(x)(t)=\sigma^{t^{\prime\prime}}(x)(t+T)\text{
and}\\
\text{and }t-T\geq t^{\prime}\Longrightarrow\sigma^{t^{\prime\prime}%
}(x)(t)=\sigma^{t^{\prime\prime}}(x)(t-T)),
\end{array}
\right.  \label{pre556}%
\end{equation}%
\begin{equation}
\left\{
\begin{array}
[c]{c}%
\forall T>0,\forall\mu\in Or(x),\exists t^{\prime\prime}\in\mathbf{R},\exists
t^{\prime}\in\mathbf{R},\forall t\geq t^{\prime},\sigma^{t^{\prime\prime}%
}(x)(t)=\mu\Longrightarrow\\
\Longrightarrow(\sigma^{t^{\prime\prime}}(x)(t)=\sigma^{t^{\prime\prime}%
}(x)(t+T)\text{ and}\\
\text{and }t-T\geq t^{\prime}\Longrightarrow\sigma^{t^{\prime\prime}%
}(x)(t)=\sigma^{t^{\prime\prime}}(x)(t-T)).
\end{array}
\right.  \label{pre600}%
\end{equation}

\end{theorem}

\begin{proof}
a) (\ref{per80})$\Longrightarrow$(\ref{pre781}) From Theorem \ref{The12_},
page \pageref{The12_}, $k^{\prime}\in\mathbf{N}_{\_}$ exists such that
$\{\widehat{x}(k)|k\geq k^{\prime}\}=\widehat{\omega}(\widehat{x})$ and, if we
take into account (\ref{per80}) also, $\mu\in\mathbf{B}^{n}$ exists with%
\begin{equation}
\{\widehat{x}(k)|k\geq k^{\prime}\}=\{\mu\}(=\widehat{\omega}(\widehat{x})).
\label{pre518}%
\end{equation}
Let $p\geq1$ and $\mu^{\prime}\in\widehat{Or}(\widehat{x})$ arbitrary. We have
two possibilities.

Case $\mu^{\prime}\neq\mu$

This corresponds to the situation when $\mu^{\prime}\in\widehat{Or}%
(\widehat{x})\setminus\widehat{\omega}(\widehat{x})$ and $\widehat{\mathbf{T}%
}_{\mu^{\prime}}^{\widehat{x}}\cap\{k^{\prime},k^{\prime}+1,k^{\prime
}+2,...\}=\varnothing.$ The statement%
\begin{equation}
\left\{
\begin{array}
[c]{c}%
\forall k\in\widehat{\mathbf{T}}_{\mu^{\prime}}^{\widehat{x}}\cap\{k^{\prime
},k^{\prime}+1,k^{\prime}+2,...\},\\
\{k+zp|z\in\mathbf{Z}\}\cap\{k^{\prime},k^{\prime}+1,k^{\prime}+2,...\}\subset
\widehat{\mathbf{T}}_{\mu^{\prime}}^{\widehat{x}}%
\end{array}
\right.
\end{equation}
takes place trivially.

Case $\mu^{\prime}=\mu$

In this case $\widehat{\mathbf{T}}_{\mu}^{\widehat{x}}\cap\{k^{\prime
},k^{\prime}+1,k^{\prime}+2,...\}\neq\varnothing$ and let $k\in\widehat
{\mathbf{T}}_{\mu}^{\widehat{x}}\cap\{k^{\prime},k^{\prime}+1,k^{\prime
}+2,...\},$ $z\in\mathbf{Z}$ arbitrary such that $k+zp\geq k^{\prime}.$ Then
from (\ref{pre518}) we get $\widehat{x}(k+zp)=\mu,$ thus $k+zp\in
\widehat{\mathbf{T}}_{\mu}^{\widehat{x}}.$

(\ref{pre781})$\Longrightarrow$(\ref{pre514}) Let $p\geq1,\mu\in\widehat
{Or}(\widehat{x})$ be arbitrary. From (\ref{pre781}) we have the existence of
$k^{\prime}\in\mathbf{N}_{\_}$ with%
\begin{equation}
\left\{
\begin{array}
[c]{c}%
\forall k\in\widehat{\mathbf{T}}_{\mu}^{\widehat{x}}\cap\{k^{\prime}%
,k^{\prime}+1,k^{\prime}+2,...\},\\
\{k+zp|z\in\mathbf{Z}\}\cap\{k^{\prime},k^{\prime}+1,k^{\prime}+2,...\}\subset
\widehat{\mathbf{T}}_{\mu}^{\widehat{x}}.
\end{array}
\right.  \label{pre681}%
\end{equation}
We define $k^{\prime\prime}=k^{\prime}+1$ and there are two possibilities.

Case $\widehat{\mathbf{T}}_{\mu}^{\widehat{\sigma}^{k^{\prime\prime}}%
(\widehat{x})}=\varnothing$

This situation occurs because $\widehat{Or}(\widehat{\sigma}^{k^{\prime\prime
}}(\widehat{x}))\subset\widehat{Or}(\widehat{x}),$ in the situation when
$\mu\in\widehat{Or}(\widehat{x})\setminus\widehat{Or}(\widehat{\sigma
}^{k^{\prime\prime}}(\widehat{x})).$ The statement%
\begin{equation}
\forall k\in\widehat{\mathbf{T}}_{\mu}^{\widehat{\sigma}^{k^{\prime\prime}%
}(\widehat{x})},\{k+zp|z\in\mathbf{Z}\}\cap\mathbf{N}_{\_}\subset
\widehat{\mathbf{T}}_{\mu}^{\widehat{\sigma}^{k^{\prime\prime}}(\widehat{x})}%
\end{equation}
takes place trivially.

Case $\widehat{\mathbf{T}}_{\mu}^{\widehat{\sigma}^{k^{\prime\prime}}%
(\widehat{x})}\neq\varnothing$

In this case $\mu\in\widehat{Or}(\widehat{\sigma}^{k^{\prime\prime}}%
(\widehat{x})).$ We take $k\in\widehat{\mathbf{T}}_{\mu}^{\widehat{\sigma
}^{k^{\prime\prime}}(\widehat{x})},z\in\mathbf{Z}$ arbitrary such that
$k+zp\geq-1.$ We conclude%
\[
\widehat{\sigma}^{k^{\prime\prime}}(\widehat{x})(k)=\mu=\widehat
{x}(k+k^{\prime\prime})=\widehat{x}(k+k^{\prime}+1),
\]
in other words $k+k^{\prime}+1\in\widehat{\mathbf{T}}_{\mu}^{\widehat{x}%
},k+k^{\prime}+1\geq k^{\prime}.$ Furthermore, $k+zp+k^{\prime}+1\geq
-1+k^{\prime}+1=k^{\prime},$ thus we can apply (\ref{pre681}), wherefrom
$k+zp+k^{\prime}+1\in\widehat{\mathbf{T}}_{\mu}^{\widehat{x}}.$ This means
that
\[
\widehat{x}(k+zp+k^{\prime}+1)=\mu=\widehat{\sigma}^{k^{\prime}+1}(\widehat
{x})(k+zp)=\widehat{\sigma}^{k^{\prime\prime}}(\widehat{x})(k+zp)
\]
i.e. $k+zp\in\widehat{\mathbf{T}}_{\mu}^{\widehat{\sigma}^{k^{\prime\prime}%
}(\widehat{x})}.$

(\ref{pre514})$\Longrightarrow$(\ref{pre524}) Let $p\geq1,\mu\in\widehat
{Or}(\widehat{x})$ arbitrary. From (\ref{pre514}) we have the existence of
$k^{\prime\prime}\in\mathbf{N}$ such that%
\begin{equation}
\forall k\in\widehat{\mathbf{T}}_{\mu}^{\widehat{\sigma}^{k^{\prime\prime}%
}(\widehat{x})},\{k+zp|z\in\mathbf{Z}\}\cap\mathbf{N}_{\_}\subset
\widehat{\mathbf{T}}_{\mu}^{\widehat{\sigma}^{k^{\prime\prime}}(\widehat{x})}.
\label{pre682}%
\end{equation}
We define $k^{\prime}=k^{\prime\prime}-1$ and we have two possibilities.

Case $\forall k\geq k^{\prime},\widehat{x}(k)\neq\mu$

Then $\widehat{\mathbf{T}}_{\mu}^{\widehat{\sigma}^{k^{\prime\prime}}%
(\widehat{x})}=\varnothing$ and (\ref{pre682}) is trivially fulfilled, as well
as the statement%
\[
\left\{
\begin{array}
[c]{c}%
\forall k\geq k^{\prime},\widehat{x}(k)=\mu\Longrightarrow\\
\Longrightarrow(\widehat{x}(k)=\widehat{x}(k+p)\text{ and }k-p\geq k^{\prime
}\Longrightarrow\widehat{x}(k)=\widehat{x}(k-p)).
\end{array}
\right.
\]

Case $\exists k\geq k^{\prime},\widehat{x}(k)=\mu$

We take $k\geq k^{\prime}$ arbitrary, such that $\widehat{x}(k)=\mu.$ Then
$k-k^{\prime\prime}=k-k^{\prime}-1\geq-1$ and%
\[
\widehat{\sigma}^{k^{\prime\prime}}(\widehat{x})(k-k^{\prime\prime}%
)=\widehat{x}(k-k^{\prime\prime}+k^{\prime\prime})=\widehat{x}(k)=\mu,
\]
thus $k-k^{\prime\prime}\in\widehat{\mathbf{T}}_{\mu}^{\widehat{\sigma
}^{k^{\prime\prime}}(\widehat{x})}.$ Furthermore%
\[
k-k^{\prime\prime}+p\in\{k-k^{\prime\prime}+zp|z\in\mathbf{Z}\}\cap
\mathbf{N}_{\_}\overset{(\ref{pre682})}{\subset}\widehat{\mathbf{T}}_{\mu
}^{\widehat{\sigma}^{k^{\prime\prime}}(\widehat{x})},
\]
meaning that
\[
\widehat{\sigma}^{k^{\prime\prime}}(\widehat{x})(k-k^{\prime\prime}%
+p)=\mu=\widehat{x}(k-k^{\prime\prime}+p+k^{\prime\prime})=\widehat{x}(k+p).
\]
If $k-p\geq k^{\prime},$ then $k-p-k^{\prime\prime}=k-p-k^{\prime}-1\geq-1,$
thus%
\[
k-k^{\prime\prime}-p\in\{k-k^{\prime\prime}+zp|z\in\mathbf{Z}\}\cap
\mathbf{N}_{\_}\overset{(\ref{pre682})}{\subset}\widehat{\mathbf{T}}_{\mu
}^{\widehat{\sigma}^{k^{\prime\prime}}(\widehat{x})}%
\]
and finally
\[
\widehat{\sigma}^{k^{\prime\prime}}(\widehat{x})(k-k^{\prime\prime}%
-p)=\mu=\widehat{x}(k-k^{\prime\prime}-p+k^{\prime\prime})=\widehat{x}(k-p).
\]

(\ref{pre524})$\Longrightarrow$(\ref{pre538}) We take $p\geq1,\mu\in
\widehat{Or}(\widehat{x})$ arbitrarily and we infer from (\ref{pre524}) that
$k^{\prime}\in\mathbf{N}_{\_}$ exists with%
\begin{equation}
\forall k\geq k^{\prime},\widehat{x}(k)=\mu\Longrightarrow\widehat
{x}(k)=\widehat{x}(k+p), \label{pre683}%
\end{equation}%
\begin{equation}
\forall k\geq k^{\prime},(\widehat{x}(k)=\mu\text{ and }k-p\geq k^{\prime
})\Longrightarrow\widehat{x}(k)=\widehat{x}(k-p). \label{pre684}%
\end{equation}
We define $k^{\prime\prime}=k^{\prime}+1$ and there are two possibilities.

Case $\forall k\in\mathbf{N}_{\_},\widehat{\sigma}^{k^{\prime\prime}}%
(\widehat{x})(k)\neq\mu$

This corresponds to the situation when $\mu\in\widehat{Or}(\widehat
{x})\setminus\widehat{Or}(\widehat{\sigma}^{k^{\prime\prime}}(\widehat{x})). $
The statement%
\[
\left\{
\begin{array}
[c]{c}%
\forall k\in\mathbf{N}_{\_},\widehat{\sigma}^{k^{\prime\prime}}(\widehat
{x})(k)=\mu\Longrightarrow\\
\Longrightarrow(\widehat{\sigma}^{k^{\prime\prime}}(\widehat{x})(k)=\widehat
{\sigma}^{k^{\prime\prime}}(\widehat{x})(k+p)\text{ and }\\
\text{and }k-p\geq-1\Longrightarrow\widehat{\sigma}^{k^{\prime\prime}%
}(\widehat{x})(k)=\widehat{\sigma}^{k^{\prime\prime}}(\widehat{x})(k-p))
\end{array}
\right.
\]
is true in a trivial manner.

Case $\exists k\in\mathbf{N}_{\_},\widehat{\sigma}^{k^{\prime\prime}}%
(\widehat{x})(k)=\mu$

We take $k\in\mathbf{N}_{\_}$ arbitrarily with $\widehat{\sigma}%
^{k^{\prime\prime}}(\widehat{x})(k)=\mu,$ thus $\widehat{x}(k+k^{\prime\prime
})=\mu=\widehat{x}(k+k^{\prime}+1).$ We have $k+k^{\prime}+1\geq k^{\prime}$
and then%
\[
\widehat{\sigma}^{k^{\prime\prime}}(\widehat{x})(k)=\widehat{x}(k+k^{\prime
\prime})=\widehat{x}(k+k^{\prime}+1)\overset{(\ref{pre683})}{=}\widehat
{x}(k+k^{\prime}+1+p)=
\]%
\[
=\widehat{x}(k+k^{\prime\prime}+p)=\widehat{\sigma}^{k^{\prime\prime}%
}(\widehat{x})(k+p).
\]
If in addition $k-p\geq-1,$ as $k-p+k^{\prime}+1\geq k^{\prime},$ we can write
that%
\[
\widehat{\sigma}^{k^{\prime\prime}}(\widehat{x})(k)=\widehat{x}(k+k^{\prime
\prime})=\widehat{x}(k+k^{\prime}+1)\overset{(\ref{pre684})}{=}\widehat
{x}(k-p+k^{\prime}+1)=
\]%
\[
=\widehat{x}(k-p+k^{\prime\prime})=\widehat{\sigma}^{k^{\prime\prime}%
}(\widehat{x})(k-p).
\]

(\ref{pre538})$\Longrightarrow$(\ref{per48}) We write (\ref{pre538}) for $p=1$
and for an arbitrary $\mu\in\widehat{\omega}(\widehat{x})$ (we have
$\widehat{\omega}(\widehat{x})\neq\varnothing$). Some $k^{\prime\prime}%
\in\mathbf{N}$ exists then with%
\begin{equation}
\left\{
\begin{array}
[c]{c}%
\forall k\in\mathbf{N}_{\_},\widehat{x}(k+k^{\prime\prime})=\mu\Longrightarrow
(\widehat{x}(k+k^{\prime\prime})=\widehat{x}(k+k^{\prime\prime}+1)\text{ and
}\\
\text{and }k\geq0\Longrightarrow\widehat{x}(k+k^{\prime\prime})=\widehat
{x}(k+k^{\prime\prime}-1))
\end{array}
\right.  \label{pre824}%
\end{equation}
and, whichever $k^{\prime\prime}$ might be, some $k\in\mathbf{N}_{\_}$ exists
such that $\widehat{x}(k+k^{\prime\prime})=\mu$ (from the hypothesis that
$\mu\in\widehat{\omega}(\widehat{x})$). We get from (\ref{pre824}) that%
\[
\mu=\widehat{x}(k^{\prime\prime}-1)=\widehat{x}(k^{\prime\prime})=\widehat
{x}(k^{\prime\prime}+1)=...
\]
i.e. (\ref{per48}) holds with $k^{\prime}=k^{\prime\prime}-1$.

b) (\ref{per79})$\Longrightarrow$(\ref{pre553}) We have the existence of
$\mu\in\mathbf{B}^{n}$ and $t_{1}\in\mathbf{R}$ with $\{x(t)|t\geq
t_{1}\}=\omega(x)=\{\mu\}$ and let $T>0,\mu^{\prime}\in Or(x)$
arbitrary\footnote{The fact that we can take $t^{\prime}\in I^{x}$ arbitrary
shows that we prove at this moment a statement that is stronger than
(\ref{pre553}).}. Let $t^{\prime}\in I^{x}$ arbitrary$.$ We take
$t_{1}^{\prime}\geq\max\{t^{\prime},t_{1}\}$ arbitrarily also and we have two possibilities.

Case $\mu^{\prime}\neq\mu$

Then $\mathbf{T}_{\mu^{\prime}}^{x}\cap\lbrack t_{1}^{\prime},\infty
)=\varnothing$ and the statement%
\[
\forall t\in\mathbf{T}_{\mu^{\prime}}^{x}\cap\lbrack t_{1}^{\prime}%
,\infty),\{t+zT|z\in\mathbf{Z}\}\cap\lbrack t_{1}^{\prime},\infty
)\subset\mathbf{T}_{\mu^{\prime}}^{x}%
\]
takes place trivially.

Case $\mu^{\prime}=\mu$

We have $\mathbf{T}_{\mu}^{x}\cap\lbrack t_{1}^{\prime},\infty)=[t_{1}%
^{\prime},\infty)$ thus let $t\in\mathbf{T}_{\mu}^{x}\cap\lbrack t_{1}%
^{\prime},\infty)$ and $z\in\mathbf{Z}$ with the property that $t+zT\geq
t_{1}^{\prime}.$ As $t+zT\geq t_{1},$ we have $t+zT\in\mathbf{T}_{\mu}^{x}.$

(\ref{pre553})$\Longrightarrow$(\ref{pre597}) Obvious.

(\ref{pre597})$\Longrightarrow$(\ref{pre554}) Let $T>0,\mu\in Or(x)$
arbitrary. From (\ref{pre597}) we have the existence of $t_{1}^{\prime}%
\in\mathbf{R}$ with%
\begin{equation}
\forall t\in\mathbf{T}_{\mu}^{x}\cap\lbrack t_{1}^{\prime},\infty
),\{t+zT|z\in\mathbf{Z}\}\cap\lbrack t_{1}^{\prime},\infty)\subset
\mathbf{T}_{\mu}^{x}. \label{pre701}%
\end{equation}
We take $t^{\prime\prime}>t_{1}^{\prime}$ arbitrary and let $\varepsilon>0$
having the property that%
\[
\forall t\in(t^{\prime\prime}-\varepsilon,t^{\prime\prime}),x(t)=x(t^{\prime
\prime}-0).
\]
We take $t^{\prime}\in(t^{\prime\prime}-\varepsilon,t^{\prime\prime}%
)\cap\lbrack t_{1}^{\prime},\infty)$ arbitrary and we notice that%
\begin{equation}
\sigma^{t^{\prime\prime}}(x)(t)=\left\{
\begin{array}
[c]{c}%
x(t),t>t^{\prime\prime}-\varepsilon,\\
x(t^{\prime\prime}-0),t<t^{\prime\prime}.
\end{array}
\right.  \label{pre702}%
\end{equation}
Obviously $t^{\prime}\in(-\infty,t^{\prime\prime})\subset I^{\sigma
^{t^{\prime\prime}}(x)}$ and we have also%
\begin{equation}
\forall t\in\mathbf{T}_{\mu}^{x}\cap\lbrack t^{\prime},\infty),\{t+zT|z\in
\mathbf{Z}\}\cap\lbrack t^{\prime},\infty)\subset\mathbf{T}_{\mu}^{x},
\label{pre703}%
\end{equation}
from Lemma \ref{Lem30}, page \pageref{Lem30}, (\ref{pre701}) and taking into
account the fact that $t^{\prime}\geq t_{1}^{\prime}$.

The truth of%
\begin{equation}
\forall t\in\mathbf{T}_{\mu}^{\sigma^{t^{\prime\prime}}(x)}\cap\lbrack
t^{\prime},\infty),\{t+zT|z\in\mathbf{Z}\}\cap\lbrack t^{\prime}%
,\infty)\subset\mathbf{T}_{\mu}^{\sigma^{t^{\prime\prime}}(x)}%
\end{equation}
results from (\ref{pre703}) and from the fact that $\forall t\geq t^{\prime
},\sigma^{t^{\prime\prime}}(x)(t)=x(t),$ see (\ref{pre702}).

(\ref{pre554})$\Longrightarrow$(\ref{pre598}) Obvious.

(\ref{pre598})$\Longrightarrow$(\ref{pre555}) We take arbitrarily $T>0,\mu\in
Or(x).$ From (\ref{pre598}) we have the existence of $t^{\prime\prime}%
\in\mathbf{R}$ and $t^{\prime\prime\prime}\in\mathbf{R}$ such that%
\begin{equation}
\forall t\in\mathbf{T}_{\mu}^{\sigma^{t^{\prime\prime}}(x)}\cap\lbrack
t^{\prime\prime\prime},\infty),\{t+zT|z\in\mathbf{Z}\}\cap\lbrack
t^{\prime\prime\prime},\infty)\subset\mathbf{T}_{\mu}^{\sigma^{t^{\prime
\prime}}(x)}. \label{pre705}%
\end{equation}
Let $t^{\prime}\in I^{x}$ arbitrary\footnote{The statement that we prove is
stronger than (\ref{pre555}).} and we take $t_{1}^{\prime}\geq\max\{t^{\prime
},t^{\prime\prime},t^{\prime\prime\prime}\}$ arbitrarily also. From
$t_{1}^{\prime}\geq t^{\prime\prime\prime},$ from (\ref{pre705}) and from
Lemma \ref{Lem30} we infer%
\begin{equation}
\forall t\in\mathbf{T}_{\mu}^{\sigma^{t^{\prime\prime}}(x)}\cap\lbrack
t_{1}^{\prime},\infty),\{t+zT|z\in\mathbf{Z}\}\cap\lbrack t_{1}^{\prime
},\infty)\subset\mathbf{T}_{\mu}^{\sigma^{t^{\prime\prime}}(x)}.
\label{pre706}%
\end{equation}
Let $t\geq t_{1}^{\prime}$ arbitrary and we have two possibilities.

Case $\forall t\geq t_{1}^{\prime},x(t)\neq\mu$

Then the implication%
\begin{equation}
\forall t\geq t_{1}^{\prime},x(t)=\mu\Longrightarrow(x(t)=x(t+T)\text{ and
}t-T\geq t_{1}^{\prime}\Longrightarrow x(t)=x(t-T))
\end{equation}
is trivially true.

Case $\exists t\geq t_{1}^{\prime},x(t)=\mu$

We take $t\geq t_{1}^{\prime}$ arbitrarily such that $x(t)=\mu.$ Because
$t\geq t^{\prime\prime},$ we have $\sigma^{t^{\prime\prime}}(x)(t)=x(t)=\mu,$
thus $t\in\mathbf{T}_{\mu}^{\sigma^{t^{\prime\prime}}(x)}\cap\lbrack
t_{1}^{\prime},\infty).$ We have%
\[
t+T\in\{t+zT|z\in\mathbf{Z}\}\cap\lbrack t_{1}^{\prime},\infty)\overset
{(\ref{pre706})}{\subset}\mathbf{T}_{\mu}^{\sigma^{t^{\prime\prime}}(x)},
\]
i.e. $\sigma^{t^{\prime\prime}}(x)(t+T)=\mu.$ On the other hand $\sigma
^{t^{\prime\prime}}(x)(t+T)=x(t+T),$ thus $x(t+T)=\mu=x(t).$

We suppose now that we have in addition $t-T\geq t_{1}^{\prime}.$ In a similar
way with the previous situation,%
\[
t-T\in\{t+zT|z\in\mathbf{Z}\}\cap\lbrack t_{1}^{\prime},\infty)\overset
{(\ref{pre706})}{\subset}\mathbf{T}_{\mu}^{\sigma^{t^{\prime\prime}}(x)},
\]
i.e. $\sigma^{t^{\prime\prime}}(x)(t-T)=\mu.$ As $\sigma^{t^{\prime\prime}%
}(x)(t-T)=x(t-T),$ we have obtained that $x(t-T)=\mu=x(t).$

(\ref{pre555})$\Longrightarrow$(\ref{pre599}) Obvious.

(\ref{pre599})$\Longrightarrow$(\ref{pre556}) Let $T>0,\mu\in Or(x)$
arbitrary. From (\ref{pre599}) we have the existence of $t_{1}^{\prime}%
\in\mathbf{R}$ such that the property%
\begin{equation}
\forall t\geq t_{1}^{\prime},x(t)=\mu\Longrightarrow(x(t)=x(t+T)\text{ and
}t-T\geq t_{1}^{\prime}\Longrightarrow x(t)=x(t-T)) \label{pre708}%
\end{equation}
holds. We take $t^{\prime\prime}>t_{1}^{\prime}$ arbitrary. Some
$\varepsilon>0$ exists with $\forall t\in(t^{\prime\prime}-\varepsilon
,t^{\prime\prime}),x(t)=x(t^{\prime\prime}-0).$ We take $t^{\prime}%
\in(t^{\prime\prime}-\varepsilon,t^{\prime\prime})\cap\lbrack t_{1}^{\prime
},\infty)$ arbitrary, for which obviously%
\begin{equation}
\sigma^{t^{\prime\prime}}(x)(t)=\left\{
\begin{array}
[c]{c}%
x(t),t>t^{\prime\prime}-\varepsilon,\\
x(t^{\prime\prime}-0),t<t^{\prime\prime}.
\end{array}
\right.  \label{pre709}%
\end{equation}
We have $t^{\prime}\in(-\infty,t^{\prime\prime})\subset I^{\sigma
^{t^{\prime\prime}}(x)}.$ On the other hand $t^{\prime}\geq t_{1}^{\prime},$
(\ref{pre708}) and Lemma \ref{Lem30}, page \pageref{Lem30} imply the truth of%
\begin{equation}
\forall t\geq t^{\prime},x(t)=\mu\Longrightarrow(x(t)=x(t+T)\text{ and
}t-T\geq t^{\prime}\Longrightarrow x(t)=x(t-T)). \label{pre710}%
\end{equation}
As $\forall t\geq t^{\prime},\sigma^{t^{\prime\prime}}(x)(t)=x(t),$
(\ref{pre710}) implies that%
\[
\left\{
\begin{array}
[c]{c}%
\forall t\geq t^{\prime},\sigma^{t^{\prime\prime}}(x)(t)=\mu\Longrightarrow
(\sigma^{t^{\prime\prime}}(x)(t)=\sigma^{t^{\prime\prime}}(x)(t+T)\text{
and}\\
\text{and }t-T\geq t^{\prime}\Longrightarrow\sigma^{t^{\prime\prime}%
}(x)(t)=\sigma^{t^{\prime\prime}}(x)(t-T))
\end{array}
\right.
\]
is true.

(\ref{pre556})$\Longrightarrow$(\ref{pre600}) Obvious.

(\ref{pre600})$\Longrightarrow$(\ref{per79}) We suppose against all reason
that $\mu,\mu^{\prime}\in\omega(x)$ exist, $\mu\neq\mu^{\prime}.$ We write
(\ref{pre600}) for an arbitrary $T>0,$ thus $t_{1}^{\prime},t_{2}^{\prime
},t_{1}^{\prime\prime},t_{2}^{\prime\prime}\in\mathbf{R} $ exist such that%
\begin{equation}
\left\{
\begin{array}
[c]{c}%
\forall t\geq t_{1}^{\prime},\sigma^{t_{1}^{\prime\prime}}(x)(t)=\mu
\Longrightarrow(\sigma^{t_{1}^{\prime\prime}}(x)(t)=\sigma^{t_{1}%
^{\prime\prime}}(x)(t+T)\text{ and}\\
\text{and }t-T\geq t_{1}^{\prime}\Longrightarrow\sigma^{t_{1}^{\prime\prime}%
}(x)(t)=\sigma^{t_{1}^{\prime\prime}}(x)(t-T)),
\end{array}
\right.  \label{pre825}%
\end{equation}%
\begin{equation}
\left\{
\begin{array}
[c]{c}%
\forall t\geq t_{2}^{\prime},\sigma^{t_{2}^{\prime\prime}}(x)(t)=\mu^{\prime
}\Longrightarrow(\sigma^{t_{2}^{\prime\prime}}(x)(t)=\sigma^{t_{2}%
^{\prime\prime}}(x)(t+T)\text{ and}\\
\text{and }t-T\geq t_{2}^{\prime}\Longrightarrow\sigma^{t_{2}^{\prime\prime}%
}(x)(t)=\sigma^{t_{2}^{\prime\prime}}(x)(t-T)).
\end{array}
\right.  \label{pre826}%
\end{equation}
Let $t_{3}^{\prime}\geq\max\{t_{1}^{\prime},t_{2}^{\prime},t_{1}^{\prime
\prime},t_{2}^{\prime\prime}\}.$ From (\ref{pre825}), (\ref{pre826}) and Lemma
\ref{Lem30}, page \pageref{Lem30} we get%
\begin{equation}
\forall t\geq t_{3}^{\prime},x(t)=\mu\Longrightarrow x(t)=x(t+T),
\label{pre827}%
\end{equation}%
\begin{equation}
\forall t\geq t_{3}^{\prime},x(t)=\mu^{\prime}\Longrightarrow x(t)=x(t+T).
\label{pre828}%
\end{equation}
As $\mu,\mu^{\prime}\in\omega(x),$ some $t_{1}\geq t_{3}^{\prime}$ exists such
that $x(t_{1})=\mu$ and some $t_{2}\geq t_{3}^{\prime}$ also exists such that
$x(t_{2}^{\prime})=\mu^{\prime}$ and, from (\ref{pre827}), (\ref{pre828}) we
infer%
\begin{equation}
\mu=x(t_{1})=x(t_{1}+T)=x(t_{1}+2T)=... \label{pre829}%
\end{equation}%
\begin{equation}
\mu^{\prime}=x(t_{2})=x(t_{2}+T)=x(t_{2}+2T)=... \label{pre830}%
\end{equation}
We suppose without loss that $t_{1}<t_{2}.$ We write (\ref{pre600}) for
$T^{\prime}=t_{2}-t_{1},$ thus $t_{4}^{\prime},t_{4}^{\prime\prime}%
\in\mathbf{R}$ exist with%
\begin{equation}
\left\{
\begin{array}
[c]{c}%
\forall t\geq t_{4}^{\prime},\sigma^{t_{4}^{\prime\prime}}(x)(t)=\mu
\Longrightarrow(\sigma^{t_{4}^{\prime\prime}}(x)(t)=\sigma^{t_{4}%
^{\prime\prime}}(x)(t+T^{\prime})\text{ and}\\
\text{and }t-T^{\prime}\geq t_{4}^{\prime}\Longrightarrow\sigma^{t_{4}%
^{\prime\prime}}(x)(t)=\sigma^{t_{4}^{\prime\prime}}(x)(t-T^{\prime})).
\end{array}
\right.  \label{pre831}%
\end{equation}
For $t_{5}^{\prime}\geq\max\{t_{4}^{\prime},t_{4}^{\prime\prime}\}$ we infer
from (\ref{pre831}) that%
\begin{equation}
\forall t\geq t_{5}^{\prime},x(t)=\mu\Longrightarrow x(t)=x(t+T^{\prime}).
\label{pre832}%
\end{equation}
Let now $k\in\mathbf{N}$ having the property that $t_{1}+kT\geq t_{5}^{\prime
}.$ As $t_{2}+kT-T^{\prime}=t_{1}+kT,$ we have%
\[
\mu\overset{(\ref{pre829})}{=}x(t_{1}+kT)=x(t_{2}+kT-T^{\prime})\overset
{(\ref{pre832})}{=}x(t_{2}+kT-T^{\prime}+T^{\prime})=x(t_{2}+kT)\overset
{(\ref{pre830})}{=}\mu^{\prime},
\]
contradiction. We have obtained that $\omega(x)$ has a single point $\mu$,
thus (\ref{per79}) is true.
\end{proof}

\section{The third group of eventual constancy properties}

\begin{remark}
The third group of eventual constancy properties involves eventual periodicity
properties of some point $\mu$ of the orbit. These properties result one by
one from the properties of the second group, by the replacement of $\forall
\mu\in\widehat{Or}(\widehat{x}),\forall\mu\in Or(x)$ with $\exists\mu
\in\widehat{Or}(\widehat{x}),\exists\mu\in Or(x). $ We notice that we have
avoided each time the trivialities by requests of the kind $\widehat
{\mathbf{T}}_{\mu}^{\widehat{x}}\cap\{k^{\prime},k^{\prime}+1,k^{\prime
}+2,...\}\neq\varnothing,\mathbf{T}_{\mu}^{x}\cap\lbrack t_{1}^{\prime}%
,\infty)\neq\varnothing.$ The possibility of replacing the universal
quantifier with the existential quantifier when passing from Theorem
\ref{The97}, page \pageref{The97} to Theorem \ref{The98} is given by the fact
that the final value, if it exists, is unique.
\end{remark}

\begin{theorem}
\label{The98}Let the signals $\widehat{x}\in\widehat{S}^{(n)},x\in S^{(n)}.$

a) The following statements are equivalent with the eventual constancy of
$\widehat{x}:$%
\begin{equation}
\left\{
\begin{array}
[c]{c}%
\forall p\geq1,\exists\mu\in\widehat{Or}(\widehat{x}),\exists k^{\prime}%
\in\mathbf{N}_{\_},\\
\widehat{\mathbf{T}}_{\mu}^{\widehat{x}}\cap\{k^{\prime},k^{\prime
}+1,k^{\prime}+2,...\}\neq\varnothing\text{ and }\forall k\in\widehat
{\mathbf{T}}_{\mu}^{\widehat{x}}\cap\{k^{\prime},k^{\prime}+1,k^{\prime
}+2,...\},\\
\{k+zp|z\in\mathbf{Z}\}\cap\{k^{\prime},k^{\prime}+1,k^{\prime}+2,...\}\subset
\widehat{\mathbf{T}}_{\mu}^{\widehat{x}},
\end{array}
\right.  \label{pre785}%
\end{equation}%
\begin{equation}
\left\{
\begin{array}
[c]{c}%
\forall p\geq1,\exists\mu\in\widehat{Or}(\widehat{x}),\exists k^{\prime\prime
}\in\mathbf{N},\widehat{\mathbf{T}}_{\mu}^{\widehat{\sigma}^{k^{\prime\prime}%
}(\widehat{x})}\neq\varnothing\text{ and }\forall k\in\widehat{\mathbf{T}%
}_{\mu}^{\widehat{\sigma}^{k^{\prime\prime}}(\widehat{x})},\\
\{k+zp|z\in\mathbf{Z}\}\cap\mathbf{N}_{\_}\subset\widehat{\mathbf{T}}_{\mu
}^{\widehat{\sigma}^{k^{\prime\prime}}(\widehat{x})},
\end{array}
\right.  \label{pre782}%
\end{equation}%
\begin{equation}
\left\{
\begin{array}
[c]{c}%
\forall p\geq1,\exists\mu\in\widehat{Or}(\widehat{x}),\exists k^{\prime}%
\in\mathbf{N}_{\_},\exists k_{1}\geq k^{\prime},\widehat{x}(k_{1})=\mu\text{
and }\\
\text{and }\forall k\geq k^{\prime},\widehat{x}(k)=\mu\Longrightarrow\\
\Longrightarrow(\widehat{x}(k)=\widehat{x}(k+p)\text{ and }k-p\geq k^{\prime
}\Longrightarrow\widehat{x}(k)=\widehat{x}(k-p)),
\end{array}
\right.  \label{pre783}%
\end{equation}%
\begin{equation}
\left\{
\begin{array}
[c]{c}%
\forall p\geq1,\exists\mu\in\widehat{Or}(\widehat{x}),\exists k^{\prime\prime
}\in\mathbf{N},\exists k_{1}\in\mathbf{N}_{\_},\widehat{\sigma}^{k^{\prime
\prime}}(\widehat{x})(k_{1})=\mu\text{ and}\\
\text{and }\forall k\in\mathbf{N}_{\_},\widehat{\sigma}^{k^{\prime\prime}%
}(\widehat{x})(k)=\mu\Longrightarrow(\widehat{\sigma}^{k^{\prime\prime}%
}(\widehat{x})(k)=\widehat{\sigma}^{k^{\prime\prime}}(\widehat{x})(k+p)\text{
and }\\
\text{and }k-p\geq-1\Longrightarrow\widehat{\sigma}^{k^{\prime\prime}%
}(\widehat{x})(k)=\widehat{\sigma}^{k^{\prime\prime}}(\widehat{x})(k-p)).
\end{array}
\right.  \label{pre784}%
\end{equation}

b) The following statements are equivalent with the eventual constancy of $x$:%
\begin{equation}
\left\{
\begin{array}
[c]{c}%
\forall T>0,\exists\mu\in Or(x),\exists t^{\prime}\in I^{x},\exists
t_{1}^{\prime}\geq t^{\prime},\mathbf{T}_{\mu}^{x}\cap\lbrack t_{1}^{\prime
},\infty)\neq\varnothing\text{ and}\\
\text{and }\forall t\in\mathbf{T}_{\mu}^{x}\cap\lbrack t_{1}^{\prime}%
,\infty),\{t+zT|z\in\mathbf{Z}\}\cap\lbrack t_{1}^{\prime},\infty
)\subset\mathbf{T}_{\mu}^{x},
\end{array}
\right.  \label{pre786}%
\end{equation}%
\begin{equation}
\left\{
\begin{array}
[c]{c}%
\forall T>0,\exists\mu\in Or(x),\exists t_{1}^{\prime}\in\mathbf{R}%
,\mathbf{T}_{\mu}^{x}\cap\lbrack t_{1}^{\prime},\infty)\neq\varnothing\text{
and }\\
\text{and }\forall t\in\mathbf{T}_{\mu}^{x}\cap\lbrack t_{1}^{\prime}%
,\infty),\{t+zT|z\in\mathbf{Z}\}\cap\lbrack t_{1}^{\prime},\infty
)\subset\mathbf{T}_{\mu}^{x},
\end{array}
\right.  \label{pre787}%
\end{equation}%
\begin{equation}
\left\{
\begin{array}
[c]{c}%
\forall T>0,\exists\mu\in Or(x),\exists t^{\prime\prime}\in\mathbf{R},\exists
t^{\prime}\in I^{\sigma^{t^{\prime\prime}}(x)},\mathbf{T}_{\mu}^{\sigma
^{t^{\prime\prime}}(x)}\cap\lbrack t^{\prime},\infty)\neq\varnothing\text{
and}\\
\text{and }\forall t\in\mathbf{T}_{\mu}^{\sigma^{t^{\prime\prime}}(x)}%
\cap\lbrack t^{\prime},\infty),\{t+zT|z\in\mathbf{Z}\}\cap\lbrack t^{\prime
},\infty)\subset\mathbf{T}_{\mu}^{\sigma^{t^{\prime\prime}}(x)},
\end{array}
\right.  \label{pre788}%
\end{equation}%
\begin{equation}
\left\{
\begin{array}
[c]{c}%
\forall T>0,\exists\mu\in Or(x),\exists t^{\prime\prime}\in\mathbf{R},\exists
t^{\prime}\in\mathbf{R},\\
\mathbf{T}_{\mu}^{\sigma^{t^{\prime\prime}}(x)}\cap\lbrack t^{\prime}%
,\infty)\neq\varnothing\text{ and}\\
\text{and }\forall t\in\mathbf{T}_{\mu}^{\sigma^{t^{\prime\prime}}(x)}%
\cap\lbrack t^{\prime},\infty),\{t+zT|z\in\mathbf{Z}\}\cap\lbrack t^{\prime
},\infty)\subset\mathbf{T}_{\mu}^{\sigma^{t^{\prime\prime}}(x)},
\end{array}
\right.  \label{pre789}%
\end{equation}%
\begin{equation}
\left\{
\begin{array}
[c]{c}%
\forall T>0,\exists\mu\in Or(x),\exists t^{\prime}\in I^{x},\\
\exists t_{1}^{\prime}\geq t^{\prime},\exists t_{2}^{\prime}\geq t_{1}%
^{\prime},x(t_{2}^{\prime})=\mu\text{ and }\forall t\geq t_{1}^{\prime
},x(t)=\mu\Longrightarrow\\
\Longrightarrow(x(t)=x(t+T)\text{ and }t-T\geq t_{1}^{\prime}\Longrightarrow
x(t)=x(t-T)),
\end{array}
\right.  \label{pre790}%
\end{equation}%
\begin{equation}
\left\{
\begin{array}
[c]{c}%
\forall T>0,\exists\mu\in Or(x),\exists t_{1}^{\prime}\in\mathbf{R},\exists
t_{2}^{\prime}\geq t_{1}^{\prime},x(t_{2}^{\prime})=\mu\text{ and }\\
\text{and }\forall t\geq t_{1}^{\prime},x(t)=\mu\Longrightarrow\\
\Longrightarrow(x(t)=x(t+T)\text{ and }t-T\geq t_{1}^{\prime}\Longrightarrow
x(t)=x(t-T)),
\end{array}
\right.  \label{pre791}%
\end{equation}%
\begin{equation}
\left\{
\begin{array}
[c]{c}%
\forall T>0,\exists\mu\in Or(x),\exists t^{\prime\prime}\in\mathbf{R},\exists
t^{\prime}\in I^{\sigma^{t^{\prime\prime}}(x)},\exists t^{\prime\prime\prime
}\geq t^{\prime},\\
\sigma^{t^{\prime\prime}}(x)(t^{\prime\prime\prime})=\mu\text{ and }\forall
t\geq t^{\prime},\sigma^{t^{\prime\prime}}(x)(t)=\mu\Longrightarrow\\
\Longrightarrow(\sigma^{t^{\prime\prime}}(x)(t)=\sigma^{t^{\prime\prime}%
}(x)(t+T)\text{ and }\\
\text{and }t-T\geq t^{\prime}\Longrightarrow\sigma^{t^{\prime\prime}%
}(x)(t)=\sigma^{t^{\prime\prime}}(x)(t-T)),
\end{array}
\right.  \label{pre792}%
\end{equation}%
\begin{equation}
\left\{
\begin{array}
[c]{c}%
\forall T>0,\exists\mu\in Or(x),\exists t^{\prime\prime}\in\mathbf{R},\exists
t^{\prime}\in\mathbf{R},\exists t^{\prime\prime\prime}\geq t^{\prime}%
,\sigma^{t^{\prime\prime}}(x)(t^{\prime\prime\prime})=\mu\text{ and}\\
\text{and }\forall t\geq t^{\prime},\sigma^{t^{\prime\prime}}(x)(t)=\mu
\Longrightarrow(\sigma^{t^{\prime\prime}}(x)(t)=\sigma^{t^{\prime\prime}%
}(x)(t+T)\text{ and}\\
\text{and }t-T\geq t^{\prime}\Longrightarrow\sigma^{t^{\prime\prime}%
}(x)(t)=\sigma^{t^{\prime\prime}}(x)(t-T)).
\end{array}
\right.  \label{pre793}%
\end{equation}

\end{theorem}

\begin{proof}
a) (\ref{per48})$\Longrightarrow$(\ref{pre785}) Let $p\geq1$ arbitrary. From
(\ref{per48}) we have the existence of $\mu\in\mathbf{B}^{n}$ and $k^{\prime
}\in\mathbf{N}_{\_}$ with the property%
\begin{equation}
\forall k\geq k^{\prime},\widehat{x}(k)=\mu. \label{pre833}%
\end{equation}
We have that $\widehat{\mathbf{T}}_{\mu}^{\widehat{x}}\cap\{k^{\prime
},k^{\prime}+1,k^{\prime}+2,...\}\neq\varnothing$ and let $k\in\widehat
{\mathbf{T}}_{\mu}^{\widehat{x}},z\in\mathbf{Z}$ arbitrary such that $k\geq
k^{\prime},k+zp\geq k^{\prime}.$ We get from (\ref{pre833}) that
$k+zp\in\widehat{\mathbf{T}}_{\mu}^{\widehat{x}}.$

(\ref{pre785})$\Longrightarrow$(\ref{pre782}) Let $p\geq1$ arbitrary.
(\ref{pre785}) shows the existence of $\mu\in\widehat{Or}(\widehat{x})$ and
$k^{\prime}\in\mathbf{N}_{\_}$ such that%
\begin{equation}
\widehat{\mathbf{T}}_{\mu}^{\widehat{x}}\cap\{k^{\prime},k^{\prime
}+1,k^{\prime}+2,...\}\neq\varnothing, \label{pre834}%
\end{equation}%
\begin{equation}
\left\{
\begin{array}
[c]{c}%
\forall k\in\widehat{\mathbf{T}}_{\mu}^{\widehat{x}}\cap\{k^{\prime}%
,k^{\prime}+1,k^{\prime}+2,...\},\\
\{k+zp|z\in\mathbf{Z}\}\cap\{k^{\prime},k^{\prime}+1,k^{\prime}+2,...\}\subset
\widehat{\mathbf{T}}_{\mu}^{\widehat{x}}.
\end{array}
\right.  \label{pre835}%
\end{equation}
We put $k^{\prime\prime}=k^{\prime}+1.$ The existence from (\ref{pre834}) of
some $k\in\widehat{\mathbf{T}}_{\mu}^{\widehat{x}}\cap\{k^{\prime},k^{\prime
}+1,k^{\prime}+2,...\}$ means that $\widehat{x}(k)=\mu$ and $k\geq k^{\prime
},$ thus $k\geq k^{\prime\prime}-1;$ the number $k_{1}=k-k^{\prime\prime}$ is
$\geq-1$ and it fulfills $\widehat{\sigma}^{k^{\prime\prime}}(\widehat
{x})(k_{1})=\widehat{x}(k_{1}+k^{\prime\prime})=\widehat{x}(k)=\mu,$ thus
$k_{1}\in\widehat{\mathbf{T}}_{\mu}^{\widehat{\sigma}^{k^{\prime\prime}%
}(\widehat{x})}$ and $\widehat{\mathbf{T}}_{\mu}^{\widehat{\sigma}%
^{k^{\prime\prime}}(\widehat{x})}\neq\varnothing.$ Let now $k\in
\widehat{\mathbf{T}}_{\mu}^{\widehat{\sigma}^{k^{\prime\prime}}(\widehat{x})}$
and $z\in\mathbf{Z}$ arbitrary such that $k+zp\geq-1.$ We have%
\[
\mu=\widehat{\sigma}^{k^{\prime\prime}}(\widehat{x})(k)=\widehat
{x}(k+k^{\prime\prime}),
\]
where $k\geq-1$ means that $k+k^{\prime\prime}=k+k^{\prime}+1\geq k^{\prime}$
and on the other hand $k+k^{\prime\prime}+zp\geq k^{\prime\prime}-1=k^{\prime
},$ thus we can apply (\ref{pre835}). We have:%
\[
\widehat{\sigma}^{k^{\prime\prime}}(\widehat{x})(k+zp)=\widehat{x}%
(k+k^{\prime\prime}+zp)\overset{(\ref{pre835})}{=}\widehat{x}(k+k^{\prime
\prime})=\mu,
\]
in other words $k+zp\in\widehat{\mathbf{T}}_{\mu}^{\widehat{\sigma}%
^{k^{\prime\prime}}(\widehat{x})}.$

(\ref{pre782})$\Longrightarrow$(\ref{pre783}) Let $p\geq1.$ (\ref{pre782})
states the existence of $\mu\in\widehat{Or}(\widehat{x})$ and $k^{\prime
\prime}\in\mathbf{N}$ such that%
\begin{equation}
\widehat{\mathbf{T}}_{\mu}^{\widehat{\sigma}^{k^{\prime\prime}}(\widehat{x}%
)}\neq\varnothing, \label{pre836}%
\end{equation}%
\begin{equation}
\forall k\in\widehat{\mathbf{T}}_{\mu}^{\widehat{\sigma}^{k^{\prime\prime}%
}(\widehat{x})},\{k+zp|z\in\mathbf{Z}\}\cap\mathbf{N}_{\_}\subset
\widehat{\mathbf{T}}_{\mu}^{\widehat{\sigma}^{k^{\prime\prime}}(\widehat{x})}.
\label{pre837}%
\end{equation}
We define $k^{\prime}=k^{\prime\prime}-1.$ (\ref{pre836}) shows the existence
of $k\in\widehat{\mathbf{T}}_{\mu}^{\widehat{\sigma}^{k^{\prime\prime}%
}(\widehat{x})},$ thus $\widehat{x}(k+k^{\prime\prime})=\mu.$ With the
notation $k_{1}=k+k^{\prime\prime}$ we have $k_{1}=k+k^{\prime}+1\geq
k^{\prime}$ (because $k\geq-1$).

Let now $k\geq k^{\prime}$ arbitrary such that $\widehat{x}(k)=\mu.$ The
number $k-k^{\prime\prime}=k-k^{\prime}-1\geq-1$ satisfies $\widehat{\sigma
}^{k^{\prime\prime}}(\widehat{x})(k-k^{\prime\prime})=\widehat{x}(k)=\mu,$
thus $k-k^{\prime\prime}\in\widehat{\mathbf{T}}_{\mu}^{\widehat{\sigma
}^{k^{\prime\prime}}(\widehat{x})}.$ We infer%
\[
k-k^{\prime\prime}+p\in\{k-k^{\prime\prime}+zp|z\in\mathbf{Z}\}\cap
\mathbf{N}_{\_}\overset{(\ref{pre837})}{\subset}\widehat{\mathbf{T}}_{\mu
}^{\widehat{\sigma}^{k^{\prime\prime}}(\widehat{x})},
\]
thus $\mu=\widehat{\sigma}^{k^{\prime\prime}}(\widehat{x})(k-k^{\prime\prime
}+p)=\widehat{x}(k+p).$ Moreover, if $k-p\geq k^{\prime},$ then $k-k^{\prime
\prime}-p=k-k^{\prime}-1-p\geq-1$ and%
\[
k-k^{\prime\prime}-p\in\{k-k^{\prime\prime}+zp|z\in\mathbf{Z}\}\cap
\mathbf{N}_{\_}\overset{(\ref{pre837})}{\subset}\widehat{\mathbf{T}}_{\mu
}^{\widehat{\sigma}^{k^{\prime\prime}}(\widehat{x})},
\]
thus $\mu=\widehat{\sigma}^{k^{\prime\prime}}(\widehat{x})(k-k^{\prime\prime
}-p)=\widehat{x}(k-p).$

(\ref{pre783})$\Longrightarrow$(\ref{pre784}) Let $p\geq1$ arbitrary. From
(\ref{pre783}) we infer the existence of $\mu\in\widehat{Or}(\widehat{x})$ and
$k^{\prime}\in\mathbf{N}_{\_}$ such that%
\begin{equation}
\exists k_{1}\geq k^{\prime},\widehat{x}(k_{1})=\mu,
\end{equation}%
\begin{equation}
\left\{
\begin{array}
[c]{c}%
\forall k\geq k^{\prime},\widehat{x}(k)=\mu\Longrightarrow\\
\Longrightarrow(\widehat{x}(k)=\widehat{x}(k+p)\text{ and }k-p\geq k^{\prime
}\Longrightarrow\widehat{x}(k)=\widehat{x}(k-p)).
\end{array}
\right.  \label{pre838}%
\end{equation}
We define $k^{\prime\prime}=k^{\prime}+1$ and let $k_{1}\geq k^{\prime}$ such
that $\widehat{x}(k_{1})=\mu.$ The number $k_{1}^{\prime}=k_{1}-k^{\prime
\prime}$ belongs to $\mathbf{N}_{\_}$ and fulfills $\widehat{\sigma
}^{k^{\prime\prime}}(\widehat{x})(k_{1}^{\prime})=\widehat{x}(k_{1}^{\prime
}+k^{\prime\prime})=\widehat{x}(k_{1})=\mu,$ in other words%
\begin{equation}
\exists k_{1}^{\prime}\in\mathbf{N}_{\_},\widehat{\sigma}^{k^{\prime\prime}%
}(\widehat{x})(k_{1}^{\prime})=\mu.
\end{equation}
Let now $k\in\mathbf{N}_{\_}$ arbitrary, with the property that $\widehat
{\sigma}^{k^{\prime\prime}}(\widehat{x})(k)=\mu,$ thus $\widehat
{x}(k+k^{\prime\prime})=\mu.$ In this situation we have $k+k^{\prime\prime
}=k+k^{\prime}+1\geq k^{\prime}$ and we can apply (\ref{pre838}), resulting%
\[
\widehat{\sigma}^{k^{\prime\prime}}(\widehat{x})(k)=\widehat{x}(k+k^{\prime
\prime})\overset{(\ref{pre838})}{=}\widehat{x}(k+k^{\prime\prime}%
+p)=\widehat{\sigma}^{k^{\prime\prime}}(\widehat{x})(k+p).
\]
In the case when in addition $k-p\geq-1,$ we have $k+k^{\prime\prime}-p\geq
k^{\prime\prime}-1=k^{\prime},$ thus we can apply (\ref{pre838}) again, with
the result%
\[
\widehat{\sigma}^{k^{\prime\prime}}(\widehat{x})(k)=\widehat{x}(k+k^{\prime
\prime})\overset{(\ref{pre838})}{=}\widehat{x}(k+k^{\prime\prime}%
-p)=\widehat{\sigma}^{k^{\prime\prime}}(\widehat{x})(k-p).
\]

(\ref{pre784})$\Longrightarrow$(\ref{per48}) We put $p=1$ in (\ref{pre784});
then $\mu\in\widehat{Or}(\widehat{x})$ and $k^{\prime\prime}\in\mathbf{N}$
exist such that%
\begin{equation}
\exists k_{1}\in\mathbf{N}_{\_},\widehat{\sigma}^{k^{\prime\prime}}%
(\widehat{x})(k_{1})=\mu,
\end{equation}%
\begin{equation}
\left\{
\begin{array}
[c]{c}%
\forall k\in\mathbf{N}_{\_},\widehat{\sigma}^{k^{\prime\prime}}(\widehat
{x})(k)=\mu\Longrightarrow(\widehat{\sigma}^{k^{\prime\prime}}(\widehat
{x})(k)=\widehat{\sigma}^{k^{\prime\prime}}(\widehat{x})(k+1)\text{ and }\\
\text{and }k\geq0\Longrightarrow\widehat{\sigma}^{k^{\prime\prime}}%
(\widehat{x})(k)=\widehat{\sigma}^{k^{\prime\prime}}(\widehat{x})(k-1)),
\end{array}
\right.
\end{equation}
thus%
\begin{equation}
\exists k_{1}\in\mathbf{N}_{\_},\widehat{x}(k_{1}+k^{\prime\prime})=\mu,
\label{pre839}%
\end{equation}%
\begin{equation}
\left\{
\begin{array}
[c]{c}%
\forall k\in\mathbf{N}_{\_},\widehat{x}(k+k^{\prime\prime})=\mu\Longrightarrow
(\widehat{x}(k+k^{\prime\prime})=\widehat{x}(k+k^{\prime\prime}+1)\text{ and
}\\
\text{and }k\geq0\Longrightarrow\widehat{x}(k+k^{\prime\prime})=\widehat
{x}(k+k^{\prime\prime}-1)).
\end{array}
\right.  \label{pre840}%
\end{equation}
We define $k^{\prime}=k^{\prime\prime}-1.$ With the notation $k_{2}%
=k+k^{\prime\prime},$ where $k\in\mathbf{N}_{\_},$ we get $k_{2}=k+k^{\prime
}+1\geq k^{\prime}$ and (\ref{pre839}), (\ref{pre840}) become%
\begin{equation}
\exists k_{2}\geq k^{\prime},\widehat{x}(k_{2})=\mu, \label{pre841}%
\end{equation}%
\begin{equation}
\left\{
\begin{array}
[c]{c}%
\forall k_{2}\geq k^{\prime},\widehat{x}(k_{2})=\mu\Longrightarrow(\widehat
{x}(k_{2})=\widehat{x}(k_{2}+1)\text{ and }\\
\text{and }k_{2}\geq k^{\prime}+1\Longrightarrow\widehat{x}(k_{2})=\widehat
{x}(k_{2}-1)).
\end{array}
\right.  \label{pre842}%
\end{equation}
From (\ref{pre841}), (\ref{pre842}) we infer%
\[
\mu=\widehat{x}(k^{\prime})=\widehat{x}(k^{\prime}+1)=\widehat{x}(k^{\prime
}+2)=...
\]
thus (\ref{per48}) holds.

b) (\ref{per49})$\Longrightarrow$(\ref{pre786}) Let $T>0$ arbitrary. From
(\ref{per49}) some $\mu\in\mathbf{B}^{n}$ and $t_{1}^{\prime}\in\mathbf{R} $
exist such that%
\begin{equation}
\forall t\geq t_{1}^{\prime},x(t)=\mu. \label{pre843}%
\end{equation}
Let $t^{\prime}\in I^{x}$ that we can choose, without restricting the
generality, $t^{\prime}\leq t_{1}^{\prime}.$ From (\ref{pre843}) we have
$\mathbf{T}_{\mu}^{x}\cap\lbrack t_{1}^{\prime},\infty)=[t_{1}^{\prime}%
,\infty)\neq\varnothing$ and, on the other hand, let $t\in\mathbf{T}_{\mu}%
^{x}\cap\lbrack t_{1}^{\prime},\infty),z\in\mathbf{Z}$ be arbitrary with
$t+zT\geq t_{1}^{\prime}.$ Then, from (\ref{pre843}), $x(t+zT)=\mu,$ i.e.
$t+zT\in\mathbf{T}_{\mu}^{x}.$ (\ref{pre786}) results.

(\ref{pre786})$\Longrightarrow$(\ref{pre787}) Obvious.

(\ref{pre787})$\Longrightarrow$(\ref{pre788}) Let $T>0$ arbitrary. From
(\ref{pre787}) we have the existence of $\mu\in Or(x)$ and $t_{1}^{\prime}%
\in\mathbf{R}$ such that%
\begin{equation}
\mathbf{T}_{\mu}^{x}\cap\lbrack t_{1}^{\prime},\infty)\neq\varnothing,
\label{pre844}%
\end{equation}%
\begin{equation}
\forall t\in\mathbf{T}_{\mu}^{x}\cap\lbrack t_{1}^{\prime},\infty
),\{t+zT|z\in\mathbf{Z}\}\cap\lbrack t_{1}^{\prime},\infty)\subset
\mathbf{T}_{\mu}^{x}. \label{pre845}%
\end{equation}
Because some $t\in\mathbf{T}_{\mu}^{x}\cap\lbrack t_{1}^{\prime},\infty) $
exists (from (\ref{pre844})) and then $\{t,t+T,t+2T,...\}\subset
\mathbf{T}_{\mu}^{x}$ (from (\ref{pre845})) we infer $\mu\in\omega(x).$

We take $t^{\prime\prime}>t_{1}^{\prime}$ arbitrary. Some $\varepsilon>0$
exists with $\forall\xi\in(t^{\prime\prime}-\varepsilon,t^{\prime\prime
}),x(\xi)=x(t^{\prime\prime}-0).$ We take $t^{\prime}\in(\max\{t_{1}^{\prime
},t^{\prime\prime}-\varepsilon\},t^{\prime\prime})$ arbitrarily and we have%
\[
\sigma^{t^{\prime\prime}}(x)(t)=\left\{
\begin{array}
[c]{c}%
x(t),t\geq t^{\prime}\\
x(t^{\prime\prime}-0),t<t^{\prime\prime}.
\end{array}
\right.
\]
We infer the truth of $t^{\prime}\in(-\infty,t^{\prime\prime})\subset
I^{\sigma^{t^{\prime\prime}}(x)}.$ The fact that $\mathbf{T}_{\mu}%
^{\sigma^{t^{\prime\prime}}(x)}\cap\lbrack t^{\prime},\infty)\neq\varnothing$
results from the remark that $\mu\in\omega(x).$ Let us take now some
$t\in\mathbf{T}_{\mu}^{\sigma^{t^{\prime\prime}}(x)}\cap\lbrack t^{\prime
},\infty)$ and $z\in\mathbf{Z}$ arbitrarily such that $t+zT\geq t^{\prime}.$
Obviously $\mathbf{T}_{\mu}^{\sigma^{t^{\prime\prime}}(x)}\cap\lbrack
t^{\prime},\infty)=\mathbf{T}_{\mu}^{x}\cap\lbrack t^{\prime},\infty).$ As far
as $t\in\mathbf{T}_{\mu}^{x}\cap\lbrack t_{1}^{\prime},\infty),t+zT\geq
t_{1}^{\prime},$ we can apply (\ref{pre845}) and we have that $t+zT\in
\mathbf{T}_{\mu}^{x},$ i.e. $x(t+zT)=\mu.$ As $t+zT\geq t^{\prime},$ and
consequently $x(t+zT)=\sigma^{t^{\prime\prime}}(x)(t+zT),$ we conclude that
$t+zT\in\mathbf{T}_{\mu}^{\sigma^{t^{\prime\prime}}(x)}.$

(\ref{pre788})$\Longrightarrow$(\ref{pre789}) Obvious.

(\ref{pre789})$\Longrightarrow$(\ref{pre790}) Let $T>0.$ From (\ref{pre789})
we have the existence of $\mu\in Or(x),t^{\prime\prime}\in\mathbf{R}$ and
$t^{\prime\prime\prime}\in\mathbf{R}$ such that%
\begin{equation}
\mathbf{T}_{\mu}^{\sigma^{t^{\prime\prime}}(x)}\cap\lbrack t^{\prime
\prime\prime},\infty)\neq\varnothing, \label{pre846}%
\end{equation}%
\begin{equation}
\forall t\in\mathbf{T}_{\mu}^{\sigma^{t^{\prime\prime}}(x)}\cap\lbrack
t^{\prime\prime\prime},\infty),\{t+zT|z\in\mathbf{Z}\}\cap\lbrack
t^{\prime\prime\prime},\infty)\subset\mathbf{T}_{\mu}^{\sigma^{t^{\prime
\prime}}(x)}. \label{pre847}%
\end{equation}
From (\ref{pre846}) some $t\in\mathbf{T}_{\mu}^{\sigma^{t^{\prime\prime}}%
(x)}\cap\lbrack t^{\prime\prime\prime},\infty)$ exists and from (\ref{pre847})
$\{t,t+T,t+2T,...\}\subset\mathbf{T}_{\mu}^{\sigma^{t^{\prime\prime}}(x)},$
thus $\mu\in\omega(x).$

We define $t_{1}^{\prime}=\max\{t^{\prime\prime},t^{\prime\prime\prime}\}.$
$t^{\prime}\in\mathbf{R}$ is chosen without loss $\leq t_{1}^{\prime}$ such
that $t^{\prime}\in I^{x}.$ Since $\mu\in\omega(x)$ we get $\exists
t_{2}^{\prime}\geq t_{1}^{\prime},x(t_{2}^{\prime})=\mu.$

Let now $t\geq t_{1}^{\prime}$ such that $x(t)=\mu,$ thus $t\in\mathbf{T}%
_{\mu}^{\sigma^{t^{\prime\prime}}(x)}\cap\lbrack t^{\prime\prime\prime}%
,\infty).$ We have%
\[
t+T\in\{t+zT|z\in\mathbf{Z}\}\cap\lbrack t^{\prime\prime\prime},\infty
)\overset{(\ref{pre847})}{\subset}\mathbf{T}_{\mu}^{\sigma^{t^{\prime\prime}%
}(x)}%
\]
and, because $t+T\geq t_{1}^{\prime},\mu=\sigma^{t^{\prime\prime}%
}(x)(t+T)=x(t+T).$ Similarly, if $t-T\geq t_{1}^{\prime},$ we can apply
(\ref{pre847}) again and we get $\mu=\sigma^{t^{\prime\prime}}%
(x)(t-T)=x(t-T).$

(\ref{pre790})$\Longrightarrow$(\ref{pre791}) Obvious.

(\ref{pre791})$\Longrightarrow$(\ref{pre792}) Let $T>0$ arbitrary. Then
$\mu\in Or(x),$ $t_{1}^{\prime}\in\mathbf{R}$ exist such that%
\begin{equation}
\exists t_{2}^{\prime}\geq t_{1}^{\prime},x(t_{2}^{\prime})=\mu,
\label{pre848}%
\end{equation}%
\begin{equation}
\left\{
\begin{array}
[c]{c}%
\forall t\geq t_{1}^{\prime},x(t)=\mu\Longrightarrow\\
\Longrightarrow(x(t)=x(t+T)\text{ and }t-T\geq t_{1}^{\prime}\Longrightarrow
x(t)=x(t-T)).
\end{array}
\right.  \label{pre849}%
\end{equation}
From (\ref{pre848}), (\ref{pre849}) we infer $\{t_{2}^{\prime},t_{2}^{\prime
}+T,t_{2}^{\prime}+2T,...\}\subset\mathbf{T}_{\mu}^{x},$ thus $\mu\in
\omega(x).$

Let $\varepsilon>0$ with $\forall\xi\in\lbrack t_{1}^{\prime},t_{1}^{\prime
}+\varepsilon),x(\xi)=x(t_{1}^{\prime})$ and we take $t^{\prime}%
=t^{\prime\prime}\in(t_{1}^{\prime},t_{1}^{\prime}+\varepsilon)$ arbitrarily.
We have%
\begin{equation}
\sigma^{t^{\prime\prime}}(x)(t)=\left\{
\begin{array}
[c]{c}%
x(t),t\geq t^{\prime},\\
x(t_{1}^{\prime}),t\leq t^{\prime},
\end{array}
\right.  \label{pre850}%
\end{equation}
wherefrom%
\[
\sigma^{t^{\prime\prime}}(x)(-\infty+0)=x(t_{1}^{\prime})=x(t^{\prime\prime
}-0),
\]
meaning that $(-\infty,t^{\prime}]\subset\mathbf{T}_{x(t^{\prime\prime}%
-0)}^{\sigma^{t^{\prime\prime}}(x)}$ holds, in other words $t^{\prime}\in
I^{\sigma^{t^{\prime\prime}}(x)}$. As $\mu\in\omega(x),$ we obtain the
existence of $t^{\prime\prime\prime}>t^{\prime}$ with $\sigma^{t^{\prime
\prime}}(x)(t^{\prime\prime\prime})=x(t^{\prime\prime\prime})=\mu.$ The truth
of%
\begin{equation}
\left\{
\begin{array}
[c]{c}%
\forall t\geq t^{\prime},\sigma^{t^{\prime\prime}}(x)(t)=\mu\Longrightarrow
(\sigma^{t^{\prime\prime}}(x)(t)=\sigma^{t^{\prime\prime}}(x)(t+T)\text{
and}\\
\text{and }t-T\geq t^{\prime}\Longrightarrow\sigma^{t^{\prime\prime}%
}(x)(t)=\sigma^{t^{\prime\prime}}(x)(t-T))
\end{array}
\right.
\end{equation}
results from (\ref{pre849}), $t^{\prime}\geq t_{1}^{\prime},$ (\ref{pre850})
and Lemma \ref{Lem30}, page \pageref{Lem30}.

(\ref{pre792})$\Longrightarrow$(\ref{pre793}) Obvious.

(\ref{pre793})$\Longrightarrow$(\ref{per79}) Let $T>0$ arbitrary. Then $\mu\in
Or(x),$ $t^{\prime\prime}\in\mathbf{R}$ and $t^{\prime}\in\mathbf{R} $ exist
such that%
\begin{equation}
\exists t^{\prime\prime\prime}\geq t^{\prime},\sigma^{t^{\prime\prime}%
}(x)(t^{\prime\prime\prime})=\mu, \label{pre851}%
\end{equation}%
\begin{equation}
\left\{
\begin{array}
[c]{c}%
\forall t\geq t^{\prime},\sigma^{t^{\prime\prime}}(x)(t)=\mu\Longrightarrow
(\sigma^{t^{\prime\prime}}(x)(t)=\sigma^{t^{\prime\prime}}(x)(t+T)\text{
and}\\
\text{and }t-T\geq t^{\prime}\Longrightarrow\sigma^{t^{\prime\prime}%
}(x)(t)=\sigma^{t^{\prime\prime}}(x)(t-T)).
\end{array}
\right.  \label{pre852}%
\end{equation}
From (\ref{pre852}), Lemma \ref{Lem30}, page \pageref{Lem30} and from the fact
that $\forall t\geq\max\{t^{\prime},t^{\prime\prime}\},\sigma^{t^{\prime
\prime}}(x)(t)=x(t),$ we have
\begin{equation}
\left\{
\begin{array}
[c]{c}%
\forall t\geq\max\{t^{\prime},t^{\prime\prime}\},x(t)=\mu\Longrightarrow\\
\Longrightarrow(x(t)=x(t+T)\text{ and }t-T\geq\max\{t^{\prime},t^{\prime
\prime}\}\Longrightarrow x(t)=x(t-T)).
\end{array}
\right.
\end{equation}
On the other hand from (\ref{pre851}), (\ref{pre852}) we have the existence of
$t^{\prime\prime\prime}\geq t^{\prime}$ with%
\[
\{t^{\prime\prime\prime},t^{\prime\prime\prime}+T,t^{\prime\prime\prime
}+2T,...\}\subset\mathbf{T}_{\mu}^{\sigma^{t^{\prime\prime}}(x)},
\]
meaning that $\mu\in\omega(\sigma^{t^{\prime\prime}}(x))=\omega(x)$ (Theorem
\ref{The12}, page \pageref{The12}). If $\omega(x)=\{\mu\}$ then the
implication is proved, so let us suppose against all reason that this is not
true. Some $t_{1},t_{2}\in\mathbf{R}$ exist with the property $\max
\{t^{\prime},t^{\prime\prime}\}<t_{1}<t_{2},$ $[t_{1},t_{2})\subset
\mathbf{T}_{\mu}^{x},$ $x(t_{1}-0)\neq\mu,$ $x(t_{2})\neq\mu.$ This shows from
Lemma \ref{Lem28}, page \pageref{Lem28} that
\begin{equation}
\lbrack t_{1},t_{2})\cup\lbrack t_{1}+T,t_{2}+T)\cup\lbrack t_{1}%
+2T,t_{2}+2T)\cup...\subset\mathbf{T}_{\mu}^{x} \label{pre960}%
\end{equation}
and from Lemma \ref{Lem25}, page \pageref{Lem25} that $\forall k\in
\mathbf{N},$%
\begin{equation}
x(t_{1}+kT-0)\neq\mu, \label{pre956}%
\end{equation}%
\begin{equation}
x(t_{2}+kT)\neq\mu. \label{pre957}%
\end{equation}

Let us write now (\ref{pre793}) for $T^{\prime}\in(0,t_{2}-t_{1}).$ There
exist $\mu^{\prime}\in Or(x),$ $t_{1}^{\prime\prime}\in\mathbf{R}$ and
$t_{1}^{\prime}\in\mathbf{R}$ with%
\begin{equation}
\exists t_{1}^{\prime\prime\prime}\geq t_{1}^{\prime},\sigma^{t_{1}%
^{\prime\prime}}(x)(t_{1}^{\prime\prime\prime})=\mu^{\prime},
\end{equation}%
\begin{equation}
\left\{
\begin{array}
[c]{c}%
\forall t\geq t_{1}^{\prime},\sigma^{t_{1}^{\prime\prime}}(x)(t)=\mu^{\prime
}\Longrightarrow(\sigma^{t_{1}^{\prime\prime}}(x)(t)=\sigma^{t_{1}%
^{\prime\prime}}(x)(t+T^{\prime})\text{ and}\\
\text{and }t-T^{\prime}\geq t_{1}^{\prime}\Longrightarrow\sigma^{t_{1}%
^{\prime\prime}}(x)(t)=\sigma^{t_{1}^{\prime\prime}}(x)(t-T^{\prime})).
\end{array}
\right.
\end{equation}
We infer like before the existence of $t_{3},t_{4}\in\mathbf{R}$ having the
property that $\max\{t_{1}^{\prime},t_{1}^{\prime\prime}\}<t_{3}%
<t_{4},\;[t_{3},t_{4})\subset\mathbf{T}_{\mu^{\prime}}^{x},\;x(t_{3}-0)\neq
\mu^{\prime},\;x(t_{4})\neq\mu^{\prime}$ and
\begin{equation}
\lbrack t_{3},t_{4})\cup\lbrack t_{3}+T^{\prime},t_{4}+T^{\prime})\cup\lbrack
t_{3}+2T^{\prime},t_{4}+2T^{\prime})\cup...\subset\mathbf{T}_{\mu^{\prime}%
}^{x}%
\end{equation}
and for any $k\in\mathbf{N}$ we have%
\begin{equation}
x(t_{3}+kT^{\prime}-0)\neq\mu^{\prime}, \label{pre958}%
\end{equation}%
\begin{equation}
x(t_{4}+kT^{\prime})\neq\mu^{\prime}. \label{pre959}%
\end{equation}
From the fact that $T^{\prime}<t_{2}-t_{1}$ and from Lemma \ref{Lem29}, page
\pageref{Lem29} we have that%
\[
\varnothing\neq([t_{1},t_{2})\cup\lbrack t_{1}+T,t_{2}+T)\cup\lbrack
t_{1}+2T,t_{2}+2T)\cup...)\cap
\]%
\[
\cap([t_{3},t_{4})\cup\lbrack t_{3}+T^{\prime},t_{4}+T^{\prime})\cup\lbrack
t_{3}+2T^{\prime},t_{4}+2T^{\prime})\cup...)\subset\mathbf{T}_{\mu}^{x}%
\cap\mathbf{T}_{\mu^{\prime}}^{x}%
\]
wherefrom $\mu=\mu^{\prime}.$

Let now $k_{2},k_{3}\in\mathbf{N}$ with the property that $[t_{1}+k_{2}%
T,t_{2}+k_{2}T)\cap\lbrack t_{3}+k_{3}T^{\prime},t_{4}+k_{3}T^{\prime}%
)\neq\varnothing.$ We have the following non-exclusive cases, that cover all
the possibilities.

Case $t_{1}+k_{2}T=t_{3}+k_{3}T^{\prime}$

As $T^{\prime}<t_{2}-t_{1},$ we have $t_{3}+(k_{3}+1)T^{\prime}\in(t_{1}%
+k_{2}T,t_{2}+k_{2}T),$ contradiction with (\ref{pre958}).

Case $t_{1}+k_{2}T\in(t_{3}+k_{3}T^{\prime},t_{4}+k_{3}T^{\prime}),$

contradiction with (\ref{pre956}).

Case $t_{2}+k_{2}T\in(t_{3}+k_{3}T^{\prime},t_{4}+k_{3}T^{\prime}),$

contradiction with (\ref{pre957}).

Case $t_{3}+k_{3}T^{\prime}\in(t_{1}+k_{2}T,t_{2}+k_{2}T),$

contradiction with (\ref{pre958}).

Case $t_{4}+k_{3}T^{\prime}\in(t_{1}+k_{2}T,t_{2}+k_{2}T),$

contradiction with (\ref{pre959}).

The fact that we have obtained in all these cases a contradiction shows the
falsity of (\ref{pre960}), with $x(t_{1}-0)\neq\mu,$ $x(t_{2})\neq\mu.$ These
should be replaced by an inclusion of the form $[t_{1},\infty)\subset
\mathbf{T}_{\mu}^{x}.$ We have proved the truth of (\ref{per128}), thus
(\ref{per79}) holds.
\end{proof}

\section{The third group of eventual constancy properties, version}

\begin{remark}
These properties are a version of the properties of the third group from the
previous Section. To be noticed that the universal quantifier $\forall\mu
\in\widehat{Or}(\widehat{x}),\forall\mu\in Or(x)$ in Theorem \ref{The97}, page
\pageref{The97} can be replaced by the existential quantifier in two different
ways; the first possibility expressed at (\ref{pre785}), (\ref{pre786}) is:
$\exists\mu\in\widehat{Or}(\widehat{x}),...,\widehat{\mathbf{T}}_{\mu
}^{\widehat{x}}\cap\{k^{\prime},k^{\prime}+1,k^{\prime}+2,...\}\neq
\varnothing,\exists\mu\in Or(x),...,\mathbf{T}_{\mu}^{x}\cap\lbrack
t_{1}^{\prime},\infty)\neq\varnothing$ and the second possibility expressed at
(\ref{pre795}), (\ref{pre799}) to follow is: $\exists\mu\in\widehat{\omega
}(\widehat{x}),\exists\mu\in\omega(x)$, when the previous non-triviality
conditions $\widehat{\mathbf{T}}_{\mu}^{\widehat{x}}\cap\{k^{\prime}%
,k^{\prime}+1,k^{\prime}+2,...\}\neq\varnothing,\mathbf{T}_{\mu}^{x}%
\cap\lbrack t_{1}^{\prime},\infty)\neq\varnothing$ are fulfilled see also
Lemma \ref{Lem37}, page \pageref{Lem37}.
\end{remark}

\begin{theorem}
\label{The99}Let the signals $\widehat{x}\in\widehat{S}^{(n)},x\in S^{(n)}.$

a) The following statements are equivalent with the eventual constancy of
$\widehat{x}:$%
\begin{equation}
\left\{
\begin{array}
[c]{c}%
\forall p\geq1,\exists\mu\in\widehat{\omega}(\widehat{x}),\exists k^{\prime
}\in\mathbf{N}_{\_},\forall k\in\widehat{\mathbf{T}}_{\mu}^{\widehat{x}}%
\cap\{k^{\prime},k^{\prime}+1,k^{\prime}+2,...\},\\
\{k+zp|z\in\mathbf{Z}\}\cap\{k^{\prime},k^{\prime}+1,k^{\prime}+2,...\}\subset
\widehat{\mathbf{T}}_{\mu}^{\widehat{x}},
\end{array}
\right.  \label{pre795}%
\end{equation}%
\begin{equation}
\left\{
\begin{array}
[c]{c}%
\forall p\geq1,\exists\mu\in\widehat{\omega}(\widehat{x}),\exists
k^{\prime\prime}\in\mathbf{N},\forall k\in\widehat{\mathbf{T}}_{\mu}%
^{\widehat{\sigma}^{k^{\prime\prime}}(\widehat{x})},\\
\{k+zp|z\in\mathbf{Z}\}\cap\mathbf{N}_{\_}\subset\widehat{\mathbf{T}}_{\mu
}^{\widehat{\sigma}^{k^{\prime\prime}}(\widehat{x})},
\end{array}
\right.  \label{pre796}%
\end{equation}%
\begin{equation}
\left\{
\begin{array}
[c]{c}%
\forall p\geq1,\exists\mu\in\widehat{\omega}(\widehat{x}),\exists k^{\prime
}\in\mathbf{N}_{\_},\forall k\geq k^{\prime},\widehat{x}(k)=\mu\Longrightarrow
\\
\Longrightarrow(\widehat{x}(k)=\widehat{x}(k+p)\text{ and }k-p\geq k^{\prime
}\Longrightarrow\widehat{x}(k)=\widehat{x}(k-p)),
\end{array}
\right.  \label{pre797}%
\end{equation}%
\begin{equation}
\left\{
\begin{array}
[c]{c}%
\forall p\geq1,\exists\mu\in\widehat{\omega}(\widehat{x}),\exists
k^{\prime\prime}\in\mathbf{N},\forall k\in\mathbf{N}_{\_},\widehat{\sigma
}^{k^{\prime\prime}}(\widehat{x})(k)=\mu\Longrightarrow\\
\Longrightarrow(\widehat{\sigma}^{k^{\prime\prime}}(\widehat{x})(k)=\widehat
{\sigma}^{k^{\prime\prime}}(\widehat{x})(k+p)\text{ and }\\
\text{and }k-p\geq-1\Longrightarrow\widehat{\sigma}^{k^{\prime\prime}%
}(\widehat{x})(k)=\widehat{\sigma}^{k^{\prime\prime}}(\widehat{x})(k-p)).
\end{array}
\right.  \label{pre798}%
\end{equation}

b) The following statements are equivalent with the eventual constancy of $x$:%
\begin{equation}
\left\{
\begin{array}
[c]{c}%
\forall T>0,\exists\mu\in\omega(x),\exists t^{\prime}\in I^{x},\\
\exists t_{1}^{\prime}\geq t^{\prime},\forall t\in\mathbf{T}_{\mu}^{x}%
\cap\lbrack t_{1}^{\prime},\infty),\{t+zT|z\in\mathbf{Z}\}\cap\lbrack
t_{1}^{\prime},\infty)\subset\mathbf{T}_{\mu}^{x}),
\end{array}
\right.  \label{pre799}%
\end{equation}%
\begin{equation}
\left\{
\begin{array}
[c]{c}%
\forall T>0,\exists\mu\in\omega(x),\exists t_{1}^{\prime}\in\mathbf{R},\\
\forall t\in\mathbf{T}_{\mu}^{x}\cap\lbrack t_{1}^{\prime},\infty
),\{t+zT|z\in\mathbf{Z}\}\cap\lbrack t_{1}^{\prime},\infty)\subset
\mathbf{T}_{\mu}^{x},
\end{array}
\right.  \label{pre800}%
\end{equation}%
\begin{equation}
\left\{
\begin{array}
[c]{c}%
\forall T>0,\exists\mu\in\omega(x),\exists t^{\prime\prime}\in\mathbf{R}%
,\exists t^{\prime}\in I^{\sigma^{t^{\prime\prime}}(x)},\\
\forall t\in\mathbf{T}_{\mu}^{\sigma^{t^{\prime\prime}}(x)}\cap\lbrack
t^{\prime},\infty),\{t+zT|z\in\mathbf{Z}\}\cap\lbrack t^{\prime}%
,\infty)\subset\mathbf{T}_{\mu}^{\sigma^{t^{\prime\prime}}(x)},
\end{array}
\right.  \label{pre801}%
\end{equation}%
\begin{equation}
\left\{
\begin{array}
[c]{c}%
\forall T>0,\exists\mu\in\omega(x),\exists t^{\prime\prime}\in\mathbf{R}%
,\exists t^{\prime}\in\mathbf{R},\\
\forall t\in\mathbf{T}_{\mu}^{\sigma^{t^{\prime\prime}}(x)}\cap\lbrack
t^{\prime},\infty),\{t+zT|z\in\mathbf{Z}\}\cap\lbrack t^{\prime}%
,\infty)\subset\mathbf{T}_{\mu}^{\sigma^{t^{\prime\prime}}(x)},
\end{array}
\right.  \label{pre802}%
\end{equation}%
\begin{equation}
\left\{
\begin{array}
[c]{c}%
\forall T>0,\exists\mu\in\omega(x),\exists t^{\prime}\in I^{x},\exists
t_{1}^{\prime}\geq t^{\prime},\forall t\geq t_{1}^{\prime},x(t)=\mu
\Longrightarrow\\
\Longrightarrow(x(t)=x(t+T)\text{ and }t-T\geq t_{1}^{\prime}\Longrightarrow
x(t)=x(t-T)),
\end{array}
\right.  \label{pre803}%
\end{equation}%
\begin{equation}
\left\{
\begin{array}
[c]{c}%
\forall T>0,\exists\mu\in\omega(x),\exists t_{1}^{\prime}\in\mathbf{R},\forall
t\geq t_{1}^{\prime},x(t)=\mu\Longrightarrow\\
\Longrightarrow(x(t)=x(t+T)\text{ and }t-T\geq t_{1}^{\prime}\Longrightarrow
x(t)=x(t-T)),
\end{array}
\right.  \label{pre804}%
\end{equation}%
\begin{equation}
\left\{
\begin{array}
[c]{c}%
\forall T>0,\exists\mu\in\omega(x),\exists t^{\prime\prime}\in\mathbf{R}%
,\exists t^{\prime}\in I^{\sigma^{t^{\prime\prime}}(x)},\\
\forall t\geq t^{\prime},\sigma^{t^{\prime\prime}}(x)(t)=\mu\Longrightarrow
(\sigma^{t^{\prime\prime}}(x)(t)=\sigma^{t^{\prime\prime}}(x)(t+T)\text{
and}\\
\text{and }t-T\geq t^{\prime}\Longrightarrow\sigma^{t^{\prime\prime}%
}(x)(t)=\sigma^{t^{\prime\prime}}(x)(t-T)),
\end{array}
\right.  \label{pre805}%
\end{equation}%
\begin{equation}
\left\{
\begin{array}
[c]{c}%
\forall T>0,\exists\mu\in\omega(x),\exists t^{\prime\prime}\in\mathbf{R}%
,\exists t^{\prime}\in\mathbf{R},\forall t\geq t^{\prime},\sigma
^{t^{\prime\prime}}(x)(t)=\mu\Longrightarrow\\
\Longrightarrow(\sigma^{t^{\prime\prime}}(x)(t)=\sigma^{t^{\prime\prime}%
}(x)(t+T)\text{ and}\\
\text{and }t-T\geq t^{\prime}\Longrightarrow\sigma^{t^{\prime\prime}%
}(x)(t)=\sigma^{t^{\prime\prime}}(x)(t-T)).
\end{array}
\right.  \label{pre806}%
\end{equation}

\end{theorem}

\section{The fourth group of eventual constancy properties}

\begin{remark}
This group of eventual constancy properties involves the eventual periodicity
of the signals.
\end{remark}

\begin{theorem}
\label{The100}Let the signals $\widehat{x}\in\widehat{S}^{(n)},x\in S^{(n)}. $

a) The following statements are equivalent with the eventual constancy of
$\widehat{x}:$%
\begin{equation}
\forall p\geq1,\exists k^{\prime}\in\mathbf{N}_{\_},\forall k\geq k^{\prime
},\widehat{x}(k)=\widehat{x}(k+p), \label{per83}%
\end{equation}%
\begin{equation}
\forall p\geq1,\exists k^{\prime\prime}\in\mathbf{N},\forall k\in
\mathbf{N}_{\_},\widehat{\sigma}^{k^{\prime\prime}}(\widehat{x})(k)=\widehat
{\sigma}^{k^{\prime\prime}}(\widehat{x})(k+p). \label{pre515}%
\end{equation}

b) The following statements are equivalent with the eventual constancy of $x:
$%
\begin{equation}
\forall T>0,\exists t^{\prime}\in I^{x},\exists t_{1}^{\prime}\geq t^{\prime
},\forall t\geq t_{1}^{\prime},x(t)=x(t+T), \label{pre557}%
\end{equation}%
\begin{equation}
\forall T>0,\exists t_{1}^{\prime}\in\mathbf{R},\forall t\geq t_{1}^{\prime
},x(t)=x(t+T), \label{pre601}%
\end{equation}%
\begin{equation}
\forall T>0,\exists t^{\prime\prime}\in\mathbf{R},\exists t^{\prime}\in
I^{\sigma^{t^{\prime\prime}}(x)},\forall t\geq t^{\prime},\sigma
^{t^{\prime\prime}}(x)(t)=\sigma^{t^{\prime\prime}}(x)(t+T), \label{pre558}%
\end{equation}%
\begin{equation}
\forall T>0,\exists t^{\prime\prime}\in\mathbf{R},\exists t^{\prime}%
\in\mathbf{R},\forall t\geq t^{\prime},\sigma^{t^{\prime\prime}}%
(x)(t)=\sigma^{t^{\prime\prime}}(x)(t+T). \label{pre602}%
\end{equation}

\end{theorem}

\begin{proof}
a) (\ref{per48})$_{page\;\pageref{per48}}\Longrightarrow$(\ref{per83}) Let
$p\geq1$ arbitrary. We have from (\ref{per48}) the existence of $\mu
\in\mathbf{B}^{n}$ and $k^{\prime}\in\mathbf{N}_{\_}$ such that%
\[
\forall k\geq k^{\prime},\widehat{x}(k)=\mu.
\]
Then for any $k\geq k^{\prime}$ we have $\widehat{x}(k+p)=\mu,$ thus
(\ref{per83}) holds.

(\ref{per83})$\Longrightarrow$(\ref{pre515}) Let $p\geq1.$ From (\ref{per83}),
some $k^{\prime}\in\mathbf{N}_{\_}$ exists with%
\begin{equation}
\forall k\geq k^{\prime},\widehat{x}(k)=\widehat{x}(k+p). \label{pre686}%
\end{equation}
We define $k^{\prime\prime}=k^{\prime}+1$ and let $k\in\mathbf{N}_{\_}$
arbitrary. As $k+k^{\prime\prime}=k+k^{\prime}+1\geq k^{\prime},$ we can write
that%
\[
\widehat{\sigma}^{k^{\prime\prime}}(\widehat{x})(k)=\widehat{x}(k+k^{\prime
\prime})=\widehat{x}(k+k^{\prime}+1)\overset{(\ref{pre686})}{=}\widehat
{x}(k+k^{\prime}+1+p)=
\]%
\[
=\widehat{x}(k+k^{\prime\prime}+p)=\widehat{\sigma}^{k^{\prime\prime}%
}(\widehat{x})(k+p).
\]

(\ref{pre515})$\Longrightarrow$(\ref{per48}) We write (\ref{pre515}) for $p=1$
under the form: $k^{\prime\prime}\in\mathbf{N}$ exists with%
\begin{equation}
\forall k\in\mathbf{N}_{\_},\widehat{\sigma}^{k^{\prime\prime}}(\widehat
{x})(k)=\widehat{\sigma}^{k^{\prime\prime}}(\widehat{x})(k+1), \label{pre687}%
\end{equation}
i.e. $\widehat{\sigma}^{k^{\prime\prime}}(\widehat{x})$ is constant. We denote
with $\mu$ the value of this constant, for which we have from (\ref{pre687}):%
\begin{equation}
\forall k\in\mathbf{N}_{\_},\widehat{x}(k+k^{\prime\prime})=\mu.
\label{pre688}%
\end{equation}
We denote \ $k^{\prime}=k^{\prime\prime}-1,k^{\prime}\in\mathbf{N}_{\_}. $ As
$k+k^{\prime\prime}=k+k^{\prime}+1\geq k^{\prime},$ (\ref{pre688}) shows that%
\[
\forall k\geq k^{\prime},\widehat{x}(k)=\mu.
\]

b) (\ref{per49})$_{page\;\pageref{per49}}\Longrightarrow$(\ref{pre557}) Let
$T>0$ arbitrary. Some $\mu\in\mathbf{B}^{n}$ and some $t_{1}^{\prime}%
\in\mathbf{R}$ exist from (\ref{per49}) with%
\begin{equation}
\forall t\geq t_{1}^{\prime},x(t)=\mu. \label{pre856}%
\end{equation}
There also exists an initial time instant $t^{\prime}\in I^{x}$ that can be
chosen without loss $\leq t_{1}^{\prime}.$

We fix in (\ref{pre856}) an arbitrary $t\geq t_{1}^{\prime}.$ We have
$x(t+T)=\mu,$ thus (\ref{pre557}) holds.

(\ref{pre557})$\Longrightarrow$(\ref{pre601}) Obvious.

(\ref{pre601})$\Longrightarrow$(\ref{pre558}) Let $T>0$ arbitrary. Some
$t_{1}^{\prime}\in\mathbf{R}$ exists from (\ref{pre601}) such that%
\begin{equation}
\forall t\geq t_{1}^{\prime},x(t)=x(t+T). \label{pre718}%
\end{equation}
We take $t^{\prime\prime}>t_{1}^{\prime}$ arbitrary. Some $\varepsilon>0$
exists then with $\forall t\in(t^{\prime\prime}-\varepsilon,t^{\prime\prime
}),x(t)=x(t^{\prime\prime}-0).$ We also take $t^{\prime}\in(t^{\prime\prime
}-\varepsilon,t^{\prime\prime})\cap\lbrack t_{1}^{\prime},\infty)$ arbitrarily
and on the other hand we have%
\begin{equation}
\sigma^{t^{\prime\prime}}(x)(t)=\left\{
\begin{array}
[c]{c}%
x(t),t>t^{\prime\prime}-\varepsilon,\\
x(t^{\prime\prime}-0),t<t^{\prime\prime}.
\end{array}
\right.  \label{pre719}%
\end{equation}
The fact that $\forall t\leq t^{\prime},\sigma^{t^{\prime\prime}%
}(x)(t)=x(t^{\prime\prime}-0)$ is obvious, wherefrom $t^{\prime}\in
I^{\sigma^{t^{\prime\prime}}(x)}$. For any $t\geq t^{\prime}$ we have%
\[
\sigma^{t^{\prime\prime}}(x)(t)\overset{(\ref{pre719})}{=}x(t)\overset
{(\ref{pre718})}{=}x(t+T)\overset{(\ref{pre719})}{=}\sigma^{t^{\prime\prime}%
}(x)(t+T).
\]

(\ref{pre558})$\Longrightarrow$(\ref{pre602}) Obvious.

(\ref{pre602})$\Longrightarrow$(\ref{per79})$_{page\;\pageref{per79}}$ We
suppose against all reason that (\ref{per79}) is false, meaning that
$\mu^{\prime},\mu^{\prime\prime}\in\omega(x)$ exist, with $\mu^{\prime}\neq
\mu^{\prime\prime}.$ Let $T>0$ be arbitrary. From (\ref{pre602}) we have the
existence of $t^{\prime\prime}\in\mathbf{R},t^{\prime}\in\mathbf{R}$ such that%
\begin{equation}
\forall t\geq t^{\prime},\sigma^{t^{\prime\prime}}(x)(t)=\sigma^{t^{\prime
\prime}}(x)(t+T), \label{pre721}%
\end{equation}
wherefrom%
\begin{equation}
\forall t\geq\max\{t^{\prime},t^{\prime\prime}\},x(t)=x(t+T). \label{pre961}%
\end{equation}
Then $t_{0}\geq\max\{t^{\prime},t^{\prime\prime}\}$ and $t_{1}\geq
\max\{t^{\prime},t^{\prime\prime}\}$ exist such that $x(t_{0})=\mu^{\prime
},x(t_{1})=\mu^{\prime\prime}$ thus, from (\ref{pre961}),%
\begin{equation}
\forall k\in\mathbf{N},x(t_{0}+kT)=\mu^{\prime}, \label{pre722}%
\end{equation}%
\begin{equation}
\forall k\in\mathbf{N},x(t_{1}+kT)=\mu^{\prime\prime}. \label{pre723}%
\end{equation}
Obviously $t_{0}\neq t_{1}$ and, in order to make a choice, we suppose that
$t_{0}<t_{1}.$

We write now (\ref{pre602}) for $T^{\prime}=t_{1}-t_{0}$ and we get the
existence of $t_{1}^{\prime\prime}\in\mathbf{R},t_{1}^{\prime}\in\mathbf{R}$
with%
\begin{equation}
\forall t\geq t_{1}^{\prime},\sigma^{t_{1}^{\prime\prime}}(x)(t)=\sigma
^{t_{1}^{\prime\prime}}(x)(t+t_{1}-t_{0}). \label{pre724}%
\end{equation}
Let $k_{1}\in\mathbf{N}$ satisfying $t_{0}+k_{1}T\geq\max\{t_{1}^{\prime
},t_{1}^{\prime\prime}\}.$ We have $t_{1}+k_{1}T>t_{0}+k_{1}T\geq\max
\{t_{1}^{\prime},t_{1}^{\prime\prime}\},$ wherefrom:%
\[
\mu^{\prime}\overset{(\ref{pre722})}{=}x(t_{0}+k_{1}T)=\sigma^{t_{1}%
^{\prime\prime}}(x)(t_{0}+k_{1}T)\overset{(\ref{pre724})}{=}\sigma
^{t_{1}^{\prime\prime}}(x)(t_{0}+k_{1}T+t_{1}-t_{0})
\]%
\[
=\sigma^{t_{1}^{\prime\prime}}(x)(t_{1}+k_{1}T)=x(t_{1}+k_{1}T)\overset
{(\ref{pre723})}{=}\mu^{\prime\prime},
\]
representing a contradiction with our supposition that $\mu^{\prime}\neq
\mu^{\prime\prime}.$ We infer the truth of (\ref{per79}).
\end{proof}

\section{Discrete time vs real time}

\begin{theorem}
\label{The113}We suppose that $(t_{k})\in Seq$ exists with%
\begin{equation}
x(t)=\widehat{x}(-1)\cdot\chi_{(-\infty,t_{0})}(t)\oplus\widehat{x}%
(0)\cdot\chi_{\lbrack t_{0},t_{1})}(t)\oplus...\oplus\widehat{x}(k)\cdot
\chi_{\lbrack t_{k},t_{k+1})}(t)\oplus... \label{per119}%
\end{equation}
Then the eventual constancies of $\widehat{x}$ and $x$ are equivalent.
\end{theorem}

\begin{proof}
If (\ref{per119}) is true, then (\ref{per48})$_{page\;\pageref{per48}}$ and
(\ref{per49})$_{page\;\pageref{per49}}$ are obviously equivalent, with
$k^{\prime}=\left\{
\begin{array}
[c]{c}%
-1,\;if\;t^{\prime}<t_{0},\\
k,\text{ }if\;t^{\prime}\in\lbrack t_{k},t_{k+1}),k\geq0
\end{array}
\right.  ,$ and $t^{\prime}=\left\{
\begin{array}
[c]{c}%
t_{k^{\prime}},if\;k^{\prime}\geq0\\
t_{0}-\varepsilon,if\;k^{\prime}=-1
\end{array}
\right.  ,$ where $\varepsilon>0$ is arbitrary.
\end{proof}

\section{Discussion}

\begin{remark}
The statements of Theorem \ref{The15}, page \pageref{The15},..., Theorem
\ref{The100}, page \pageref{The100}, are structured in discrete time - real
time analogue properties and we notice that to a discrete time statement there
correspond either one or (in Theorems \ref{The97}, \ref{The98}, \ref{The99},
\ref{The100}) two real time statements. This is principially based on the fact
that we may omit in these requirements to state that an initial time exists,
since this is contained in the definition of the signals.
\end{remark}

\begin{remark}
Theorem \ref{The113} is a restatement of Theorem \ref{The92}, page
\pageref{The92}.
\end{remark}

\begin{remark}
The properties (\ref{per48}),..., (\ref{per80}) and (\ref{per49}),...,
(\ref{per79}) do not involve periodicity. The other properties that are
equivalent with eventual constancy are divided into two groups:

- (\ref{pre781}),...,(\ref{pre538}) and (\ref{pre553}),...,(\ref{pre600});
(\ref{pre785}),...,(\ref{pre784}) and (\ref{pre786}),...,(\ref{pre793});
(\ref{pre795}),..., (\ref{pre798}) and (\ref{pre799}),...,(\ref{pre806}) are
of eventual periodicity of the points, and

- (\ref{per83}), (\ref{pre515}) and (\ref{pre557}),..., (\ref{pre602}) are of
eventual periodicity of the signals.
\end{remark}

\begin{remark}
The common point, of intersection of the previous groups of periodicity
properties is the one that the eventual periodicity of a signal exists if all
the points of the orbit are eventually periodic, with the same period and the
same limit of periodicity.
\end{remark}

\begin{remark}
The statements (\ref{pre514}), (\ref{pre538}) and (\ref{pre554}),
(\ref{pre556}); (\ref{pre782}), (\ref{pre784}) and (\ref{pre788}),
(\ref{pre792}); (\ref{pre796}), (\ref{pre798}) and (\ref{pre801}),
(\ref{pre805}); (\ref{pre515}) and (\ref{pre558}) are of periodicity of
$\widehat{\sigma}^{k^{\prime\prime}}(\widehat{x}),\sigma^{t^{\prime\prime}%
}(x).$ The eventual periodicity of $\widehat{x},x$ results from the fact that
$\widehat{\sigma}^{k^{\prime\prime}}(\widehat{x}),\sigma^{t^{\prime\prime}%
}(x)$ ignore the first values of $\widehat{x},x.$
\end{remark}

\begin{remark}
\label{Rem24}We ask that, in order that the eventual periodicity be equivalent
with the eventual constancy, it should take place with any period
$p\geq1,T>0.$
\end{remark}

\begin{remark}
\label{Rem25}In (\ref{per48}),..., (\ref{per80}) and (\ref{per49}),...,
(\ref{per79}) the existence of a unique $\mu$ is asked and we have
$\mu=\underset{k\rightarrow\infty}{\lim}\widehat{x}(k),\mu=\underset
{t\rightarrow\infty}{\lim}x(t).$
\end{remark}

\begin{remark}
The eventually constant signals $\widehat{x},x$ fulfill $\widehat{\omega
}(\widehat{x})=\{\mu\},\omega(x)=\{\mu\}$ like the constant signals, but
$\widehat{Or}(\widehat{x}),Or(x)$ contain also other points than $\mu,$ in
general. The points of $\widehat{Or}(\widehat{x})\setminus\widehat{\omega
}(\widehat{x}),$ $Or(x)\setminus\omega(x)$ are some 'first values' of these signals.
\end{remark}

\chapter{\label{Cha2}Constant signals}

The Chapter presents properties that are equivalent with the constancy of the
signals and that are related, most of them, with periodicity. The key aspect
is that periodicity must hold with any period in order to be equivalent with constancy.

We have gathered these properties in four groups, in order to analyze them
better and make them be better understood. Section 1 presents the first group
of constancy properties, gathering these properties that are not related with
periodicity. Sections 2, respectively 3 present the groups of constancy
properties of the signals involving periodicity and eventual periodicity
properties of all the points of the orbit, respectively of some point of the
orbit. The fourth group of constancy properties, involving the periodicity and
the eventual periodicity of the signals, is introduced in Section 4. Section 5
relates the constancy of the discrete time and the real time signals. The last
Section contains the interpretation of the constancy properties.

\section{The first group of constancy properties}

\begin{remark}
The first group of constancy properties of the signals contains these
properties that are not related with periodicity. These properties are
inspired one by one by the properties of eventual constancy from Theorem
\ref{The15}, page \pageref{The15}.
\end{remark}

\begin{theorem}
\label{The13}We consider the signals $\widehat{x}\in\widehat{S}^{(n)},x\in
S^{(n)}.$

a) The following requirements stating the constancy of $\widehat{x}$ are
equivalent%
\begin{equation}
\exists\mu\in\mathbf{B}^{n},\forall k\in\mathbf{N}_{\_},\widehat{x}(k)=\mu,
\label{per46}%
\end{equation}%
\begin{equation}
\exists\mu\in\mathbf{B}^{n},\widehat{\mathbf{T}}_{\mu}^{\widehat{x}%
}=\mathbf{N}_{\_}, \label{per129}%
\end{equation}%
\begin{equation}
\exists\mu\in\mathbf{B}^{n},\widehat{Or}(\widehat{x})=\{\mu\}. \label{per77}%
\end{equation}

b) The following requirements stating the constancy of $x$ are also equivalent%
\begin{equation}
\exists\mu\in\mathbf{B}^{n},\forall t\in\mathbf{R},x(t)=\mu, \label{per47}%
\end{equation}%
\begin{equation}
\exists\mu\in\mathbf{B}^{n},\mathbf{T}_{\mu}^{x}=\mathbf{R}, \label{per126}%
\end{equation}%
\begin{equation}
\exists\mu\in\mathbf{B}^{n},Or(x)=\{\mu\}. \label{per78}%
\end{equation}

\end{theorem}

\begin{proof}
a) (\ref{per46})$\Longrightarrow$(\ref{per129}) If $\mu\in\mathbf{B}^{n}$
exists such that $\forall k\in\mathbf{N}_{\_},\widehat{x}(k)=\mu,$ then
$\{k|k\in\mathbf{N}_{\_},\widehat{x}(k)=\mu\}=\mathbf{N}_{\_}.$

(\ref{per129})$\Longrightarrow$(\ref{per77}) If $\mu\in\mathbf{B}^{n}$ exists
such that $\{k|k\in\mathbf{N}_{\_},\widehat{x}(k)=\mu\}=\mathbf{N}_{\_},$ then
$\{\widehat{x}(k)|k\in\mathbf{N}_{\_}\}=\{\mu\}.$

(\ref{per77})$\Longrightarrow$(\ref{per46}) The existence of $\mu\in
\mathbf{B}^{n}$ such that $\{\widehat{x}(k)|k\in\mathbf{N}_{\_}\}=\{\mu\}$
implies $\forall k\in\mathbf{N}_{\_},\widehat{x}(k)=\mu.$

b) (\ref{per47})$\Longrightarrow$(\ref{per126}) If $\mu\in\mathbf{B}^{n}$
exists such that $\forall t\in\mathbf{R},x(t)=\mu,$ then $\{t|t\in
\mathbf{R},x(t)=\mu\}=\mathbf{R}.$

(\ref{per126})$\Longrightarrow$(\ref{per78}) If $\mu\in\mathbf{B}^{n}$ exists
such that $\{t|t\in\mathbf{R},x(t)=\mu\}=\mathbf{R}$, then $\{x(t)|t\in
\mathbf{R}\}=\{\mu\}.$

(\ref{per78})$\Longrightarrow$(\ref{per47}) The existence of $\mu$ with
$\{x(t)|t\in\mathbf{R}\}=\{\mu\}$ implies that $\forall t\in\mathbf{R}%
,x(t)=\mu$ is true.
\end{proof}

\section{The second group of constancy properties}

\begin{remark}
This group of constancy properties of the signals involves periodicity and
eventual periodicity properties of all the points $\mu$ of the orbit, i.e. in
(\ref{per185}),...,(\ref{pre537}), (\ref{per186}),...,(\ref{pre552}) to follow
we ask $\forall\mu\in\widehat{Or}(\widehat{x}),\forall\mu\in Or(x).$
\end{remark}

\begin{remark}
In order to understand better the way that these properties were written, to
be noticed the existence of the couples and triples:

- (\ref{per185})-(\ref{per186}), (\ref{pre506})-(\ref{pre548}%
),...,(\ref{pre537})-(\ref{pre552});

- (\ref{per185})-(\ref{pre506})-(\ref{pre507}), (\ref{pre522})-(\ref{pre523}%
)-(\ref{pre537}) and (\ref{per186})-(\ref{pre548})-(\ref{pre549}),
(\ref{pre550})-(\ref{pre551})-(\ref{pre552});

- (\ref{per185})-(\ref{pre522}),(\ref{pre506})-(\ref{pre523}),(\ref{pre507}%
)-(\ref{pre537}) and (\ref{per186})-(\ref{pre550}), (\ref{pre548}%
)-(\ref{pre551}), (\ref{pre549})-(\ref{pre552}).
\end{remark}

\begin{remark}
These properties are inspired by the equivalent properties of Theorem
\ref{The97}, page \pageref{The97}.\ Note that:

- (\ref{per185}) and (\ref{pre506}) (the last contains $\forall k^{\prime}%
\in\mathbf{N}_{\_}$) are inspired by (\ref{pre781})$_{page\;\pageref{pre781}}$
(containing $\exists k^{\prime}\in\mathbf{N}_{\_}$);

- (\ref{pre507}) (containing $\forall k^{\prime\prime}\in\mathbf{N}$) is
inspired by (\ref{pre514})$_{page\;\pageref{pre514}}$ (containing $\exists
k^{\prime\prime}\in\mathbf{N}$);

- (\ref{pre522}) and (\ref{pre523}) (the last contains $\forall k^{\prime}%
\in\mathbf{N}_{\_}$) are inspired by (\ref{pre524})$_{page\;\pageref{pre524}}$
(containing $\exists k^{\prime}\in\mathbf{N}_{\_}$);

- (\ref{pre537}) (containing $\forall k^{\prime\prime}\in\mathbf{N}$) is
inspired by (\ref{pre538})$_{page\;\pageref{pre538}}$ (containing $\exists
k^{\prime\prime}\in\mathbf{N}$);

- (\ref{per186}) and (\ref{pre548}) (the last contains $\forall t_{1}^{\prime
}\geq t^{\prime}$) are inspired by (\ref{pre553})$_{page\;\pageref{pre553}}$
and (\ref{pre597})$_{page\;\pageref{pre597}}$ (containing $\exists
t_{1}^{\prime}\geq t^{\prime}$ and $\exists t_{1}^{\prime}\in\mathbf{R}$);

- (\ref{pre549}) (containing $\forall t^{\prime\prime}\in\mathbf{R}$) is
inspired by (\ref{pre554})$_{page\;\pageref{pre554}}$ and (\ref{pre598}%
)$_{page\;\pageref{pre598}}$ (containing $\exists t^{\prime\prime}%
\in\mathbf{R}$);

- (\ref{pre550}) and (\ref{pre551}) (the last contains $\forall t_{1}^{\prime
}\geq t^{\prime}$) are inspired by (\ref{pre555})$_{page\;\pageref{pre555}}$
and (\ref{pre599})$_{page\;\pageref{pre599}}$ (containing $\exists
t_{1}^{\prime}\geq t^{\prime}$ and $\exists t_{1}^{\prime}\in\mathbf{R}$);

- (\ref{pre552}) (containing $\forall t^{\prime\prime}\in\mathbf{R}$) is
inspired by (\ref{pre556})$_{page\;\pageref{pre556}}$ and (\ref{pre600}%
)$_{page\;\pageref{pre600}}$ (containing $\exists t^{\prime\prime}%
\in\mathbf{R}$).
\end{remark}

\begin{theorem}
\label{The95}a) Any of the following properties is equivalent with the
constancy of $\widehat{x}\in\widehat{S}^{(n)}$:%
\begin{equation}
\forall p\geq1,\forall\mu\in\widehat{Or}(\widehat{x}),\forall k\in
\widehat{\mathbf{T}}_{\mu}^{\widehat{x}},\{k+zp|z\in\mathbf{Z}\}\cap
\mathbf{N}_{\_}\subset\widehat{\mathbf{T}}_{\mu}^{\widehat{x}}, \label{per185}%
\end{equation}%
\begin{equation}
\left\{
\begin{array}
[c]{c}%
\forall p\geq1,\forall\mu\in\widehat{Or}(\widehat{x}),\forall k^{\prime}%
\in\mathbf{N}_{\_},\forall k\in\widehat{\mathbf{T}}_{\mu}^{\widehat{x}}%
\cap\{k^{\prime},k^{\prime}+1,k^{\prime}+2,...\},\\
\{k+zp|z\in\mathbf{Z}\}\cap\{k^{\prime},k^{\prime}+1,k^{\prime}+2,...\}\subset
\widehat{\mathbf{T}}_{\mu}^{\widehat{x}},
\end{array}
\right.  \label{pre506}%
\end{equation}%
\begin{equation}
\left\{
\begin{array}
[c]{c}%
\forall p\geq1,\forall\mu\in\widehat{Or}(\widehat{x}),\forall k^{\prime\prime
}\in\mathbf{N},\forall k\in\widehat{\mathbf{T}}_{\mu}^{\widehat{\sigma
}^{k^{\prime\prime}}(\widehat{x})},\\
\{k+zp|z\in\mathbf{Z}\}\cap\mathbf{N}_{\_}\subset\widehat{\mathbf{T}}_{\mu
}^{\widehat{\sigma}^{k^{\prime\prime}}(\widehat{x})},
\end{array}
\right.  \label{pre507}%
\end{equation}%
\begin{equation}
\left\{
\begin{array}
[c]{c}%
\forall p\geq1,\forall\mu\in\widehat{Or}(\widehat{x}),\forall k\in
\mathbf{N}_{\_},\widehat{x}(k)=\mu\Longrightarrow\\
\Longrightarrow(\widehat{x}(k)=\widehat{x}(k+p)\text{ and }k-p\geq
-1\Longrightarrow\widehat{x}(k)=\widehat{x}(k-p)),
\end{array}
\right.  \label{pre522}%
\end{equation}%
\begin{equation}
\left\{
\begin{array}
[c]{c}%
\forall p\geq1,\forall\mu\in\widehat{Or}(\widehat{x}),\forall k^{\prime}%
\in\mathbf{N}_{\_},\forall k\geq k^{\prime},\widehat{x}(k)=\mu\Longrightarrow
\\
\Longrightarrow(\widehat{x}(k)=\widehat{x}(k+p)\text{ and }k-p\geq k^{\prime
}\Longrightarrow\widehat{x}(k)=\widehat{x}(k-p)),
\end{array}
\right.  \label{pre523}%
\end{equation}%
\begin{equation}
\left\{
\begin{array}
[c]{c}%
\forall p\geq1,\forall\mu\in\widehat{Or}(\widehat{x}),\forall k^{\prime\prime
}\in\mathbf{N},\forall k\in\mathbf{N}_{\_},\widehat{\sigma}^{k^{\prime\prime}%
}(\widehat{x})(k)=\mu\Longrightarrow\\
\Longrightarrow(\widehat{\sigma}^{k^{\prime\prime}}(\widehat{x})(k)=\widehat
{\sigma}^{k^{\prime\prime}}(\widehat{x})(k+p)\text{ and }\\
\text{and }k-p\geq-1\Longrightarrow\widehat{\sigma}^{k^{\prime\prime}%
}(\widehat{x})(k)=\widehat{\sigma}^{k^{\prime\prime}}(\widehat{x})(k-p)).
\end{array}
\right.  \label{pre537}%
\end{equation}

b) Any of the following properties is equivalent with the constancy of $x\in
S^{(n)}$:%
\begin{equation}
\left\{
\begin{array}
[c]{c}%
\forall T>0,\forall\mu\in Or(x),\exists t^{\prime}\in I^{x},\\
\forall t\in\mathbf{T}_{\mu}^{x}\cap\lbrack t^{\prime},\infty),\{t+zT|z\in
\mathbf{Z}\}\cap\lbrack t^{\prime},\infty)\subset\mathbf{T}_{\mu}^{x},
\end{array}
\right.  \label{per186}%
\end{equation}%
\begin{equation}
\left\{
\begin{array}
[c]{c}%
\forall T>0,\forall\mu\in Or(x),\exists t^{\prime}\in I^{x},\\
\forall t_{1}^{\prime}\geq t^{\prime},\forall t\in\mathbf{T}_{\mu}^{x}%
\cap\lbrack t_{1}^{\prime},\infty),\{t+zT|z\in\mathbf{Z}\}\cap\lbrack
t_{1}^{\prime},\infty)\subset\mathbf{T}_{\mu}^{x},
\end{array}
\right.  \label{pre548}%
\end{equation}%
\begin{equation}
\left\{
\begin{array}
[c]{c}%
\forall T>0,\forall\mu\in Or(x),\forall t^{\prime\prime}\in\mathbf{R},\exists
t^{\prime}\in I^{\sigma^{t^{\prime\prime}}(x)},\\
\forall t\in\mathbf{T}_{\mu}^{\sigma^{t^{\prime\prime}}(x)}\cap\lbrack
t^{\prime},\infty),\{t+zT|z\in\mathbf{Z}\}\cap\lbrack t^{\prime}%
,\infty)\subset\mathbf{T}_{\mu}^{\sigma^{t^{\prime\prime}}(x)},
\end{array}
\right.  \label{pre549}%
\end{equation}%
\begin{equation}
\left\{
\begin{array}
[c]{c}%
\forall T>0,\forall\mu\in Or(x),\exists t^{\prime}\in I^{x},\forall t\geq
t^{\prime},\\
x(t)=\mu\Longrightarrow(x(t)=x(t+T)\text{ and }t-T\geq t^{\prime
}\Longrightarrow x(t)=x(t-T)),
\end{array}
\right.  \label{pre550}%
\end{equation}%
\begin{equation}
\left\{
\begin{array}
[c]{c}%
\forall T>0,\forall\mu\in Or(x),\exists t^{\prime}\in I^{x},\forall
t_{1}^{\prime}\geq t^{\prime},\forall t\geq t_{1}^{\prime},x(t)=\mu
\Longrightarrow\\
\Longrightarrow(x(t)=x(t+T)\text{ and }t-T\geq t_{1}^{\prime}\Longrightarrow
x(t)=x(t-T)),
\end{array}
\right.  \label{pre551}%
\end{equation}%
\begin{equation}
\left\{
\begin{array}
[c]{c}%
\forall T>0,\forall\mu\in Or(x),\forall t^{\prime\prime}\in\mathbf{R},\exists
t^{\prime}\in I^{\sigma^{t^{\prime\prime}}(x)},\\
\forall t\geq t^{\prime},\sigma^{t^{\prime\prime}}(x)(t)=\mu\Longrightarrow
(\sigma^{t^{\prime\prime}}(x)(t)=\sigma^{t^{\prime\prime}}(x)(t+T)\text{
and}\\
\text{and }t-T\geq t^{\prime}\Longrightarrow\sigma^{t^{\prime\prime}%
}(x)(t)=\sigma^{t^{\prime\prime}}(x)(t-T)).
\end{array}
\right.  \label{pre552}%
\end{equation}

\end{theorem}

\begin{proof}
a) (\ref{per77})$\Longrightarrow$(\ref{per185}) We suppose that $\mu
\in\mathbf{B}^{n}$ exists with $\{\widehat{x}(k)|k\in\mathbf{N}_{\_}\}=\{\mu\}
$ and let $p\geq1,$ $k\in\widehat{\mathbf{T}}_{\mu}^{\widehat{x}},$
$z\in\mathbf{Z}$ arbitrary such that $k+zp\geq-1.$ Obviously $\widehat
{x}(k+zp)=\mu,$ thus $k+zp\in\widehat{\mathbf{T}}_{\mu}^{\widehat{x}}.$

(\ref{per185})$\Longrightarrow$(\ref{pre506}) We take $p\geq1,$ $\mu
\in\widehat{Or}(\widehat{x}),$ $k^{\prime}\in\mathbf{N}_{\_}$ arbitrary. If
$\widehat{\mathbf{T}}_{\mu}^{\widehat{x}}\cap\{k^{\prime},k^{\prime
}+1,k^{\prime}+2,...\}=\varnothing,$ then the statement%
\[
\forall k\in\widehat{\mathbf{T}}_{\mu}^{\widehat{x}}\cap\{k^{\prime}%
,k^{\prime}+1,k^{\prime}+2,...\},\{k+zp|z\in\mathbf{Z}\}\cap\{k^{\prime
},k^{\prime}+1,k^{\prime}+2,...\}\subset\widehat{\mathbf{T}}_{\mu}%
^{\widehat{x}}%
\]
holds trivially, thus we can suppose from now that $\widehat{\mathbf{T}}_{\mu
}^{\widehat{x}}\cap\{k^{\prime},k^{\prime}+1,k^{\prime}+2,...\}\neq
\varnothing$ and let $k\in\widehat{\mathbf{T}}_{\mu}^{\widehat{x}},$
$z\in\mathbf{Z}$ arbitrary, fixed, such that $k\geq k^{\prime}$ and $k+zp\geq
k^{\prime}.$ As $k+zp\geq-1,$ we have from (\ref{per185}) that $k+zp\in
\widehat{\mathbf{T}}_{\mu}^{\widehat{x}}.$

(\ref{pre506})$\Longrightarrow$(\ref{pre507}) Let $p\geq1,$ $\mu\in
\widehat{Or}(\widehat{x}),$ $k^{\prime\prime}\in\mathbf{N}$ arbitrary. If
$\widehat{\mathbf{T}}_{\mu}^{\widehat{\sigma}^{k^{\prime\prime}}(\widehat{x}%
)}=\varnothing,$ the statement%
\[
\forall k\in\widehat{\mathbf{T}}_{\mu}^{\widehat{\sigma}^{k^{\prime\prime}%
}(\widehat{x})},\{k+zp|z\in\mathbf{Z}\}\cap\mathbf{N}_{\_}\subset
\widehat{\mathbf{T}}_{\mu}^{\widehat{\sigma}^{k^{\prime\prime}}(\widehat{x})}%
\]
is trivially fulfilled, so that we can suppose from now that $\widehat
{\mathbf{T}}_{\mu}^{\widehat{\sigma}^{k^{\prime\prime}}(\widehat{x})}%
\neq\varnothing.$ Let $k\in\widehat{\mathbf{T}}_{\mu}^{\widehat{\sigma
}^{k^{\prime\prime}}(\widehat{x})},$ $z\in\mathbf{Z}$ arbitrary such that
$k+zp\geq-1.$ We have $\widehat{x}(k+k^{\prime\prime})=\mu$ or, if we denote
$k^{\prime}=k^{\prime\prime}-1,$ then $\widehat{x}(k+k^{\prime}+1)=\mu,$ where
$k^{\prime}\in\mathbf{N}_{\_}.$ Of course that $k+k^{\prime}+1\geq k^{\prime
},$ thus $k+k^{\prime}+1\in\widehat{\mathbf{T}}_{\mu}^{\widehat{x}}%
\cap\{k^{\prime},k^{\prime}+1,k^{\prime}+2,...\}$ and, on the other hand,
$k+k^{\prime}+1+zp\geq k^{\prime}+1-1,$ thus we can apply (\ref{pre506}),
wherefrom $\widehat{x}(k+k^{\prime}+1+zp)=\mu.$ It has resulted that
$\widehat{\sigma}^{k^{\prime\prime}}(\widehat{x})(k+zp)=\widehat
{x}(k+k^{\prime\prime}+zp)=\mu,$ in other words $k+zp\in\widehat{\mathbf{T}%
}_{\mu}^{\widehat{\sigma}^{k^{\prime\prime}}(\widehat{x})}.$

(\ref{pre507})$\Longrightarrow$(\ref{pre522}) Let $p\geq1,\mu\in\widehat
{Or}(\widehat{x})$ and $k\in\mathbf{N}_{\_}$ such that $\widehat{x}(k)=\mu.$
(\ref{pre507}) written for $k^{\prime\prime}=0$ gives $\{k+zp|z\in
\mathbf{Z}\}\cap\mathbf{N}_{\_}\subset\widehat{\mathbf{T}}_{\mu}^{\widehat{x}%
},$ thus%
\[
k+p\in\{k+zp|z\in\mathbf{Z}\}\cap\mathbf{N}_{\_}\overset{(\ref{pre507}%
)}{\subset}\widehat{\mathbf{T}}_{\mu}^{\widehat{x}},
\]
wherefrom $\widehat{x}(k+p)=\mu=\widehat{x}(k).$

If in addition $k-p\geq-1,$ then%
\[
k-p\in\{k+zp|z\in\mathbf{Z}\}\cap\mathbf{N}_{\_}\overset{(\ref{pre507}%
)}{\subset}\widehat{\mathbf{T}}_{\mu}^{\widehat{x}},
\]
wherefrom $\widehat{x}(k-p)=\mu=\widehat{x}(k).$

(\ref{pre522})$\Longrightarrow$(\ref{pre523}) We take $p\geq1,$ $\mu
\in\widehat{Or}(\widehat{x}),$ $k^{\prime}\in\mathbf{N}_{\_}$ arbitrary. If
$\forall k\geq k^{\prime},\widehat{x}(k)\neq\mu$, then the statement%
\[
\forall k\geq k^{\prime},\widehat{x}(k)=\mu\Longrightarrow(\widehat
{x}(k)=\widehat{x}(k+p)\text{ and }k-p\geq k^{\prime}\Longrightarrow
\widehat{x}(k)=\widehat{x}(k-p))
\]
is trivially fulfilled, thus we can take $k\geq k^{\prime}$ arbitrarily with
$\widehat{x}(k)=\mu.$ From (\ref{pre522}) we have that $\widehat
{x}(k)=\widehat{x}(k+p).$ In the case that in addition $k-p\geq k^{\prime},$
as $k-p\geq-1,$ we can apply (\ref{pre522}) again in order to infer that
$\widehat{x}(k)=\widehat{x}(k-p).$

(\ref{pre523})$\Longrightarrow$(\ref{pre537}) Let $p\geq1,$ $\mu\in
\widehat{Or}(\widehat{x}),$ $k^{\prime\prime}\in\mathbf{N}$ arbitrary. If
$\forall k\in\mathbf{N}_{\_},\widehat{\sigma}^{k^{\prime\prime}}(\widehat
{x})(k)\neq\mu$ then the statement%
\[
\left\{
\begin{array}
[c]{c}%
\forall k\in\mathbf{N}_{\_},\widehat{\sigma}^{k^{\prime\prime}}(\widehat
{x})(k)=\mu\Longrightarrow(\widehat{\sigma}^{k^{\prime\prime}}(\widehat
{x})(k)=\widehat{\sigma}^{k^{\prime\prime}}(\widehat{x})(k+p)\text{ and }\\
\text{and }k-p\geq-1\Longrightarrow\widehat{\sigma}^{k^{\prime\prime}%
}(\widehat{x})(k)=\widehat{\sigma}^{k^{\prime\prime}}(\widehat{x})(k-p))
\end{array}
\right.
\]
is trivially true, thus we take $k\in\mathbf{N}_{\_}$ arbitrary such that
$\widehat{\sigma}^{k^{\prime\prime}}(\widehat{x})(k)=\widehat{x}%
(k+k^{\prime\prime})=\mu.$ We denote $k^{\prime}=k^{\prime\prime}-1$ and we
see that $\widehat{x}(k+k^{\prime}+1)=\mu,$ where $k+k^{\prime}+1\geq
k^{\prime}.$ We can apply (\ref{pre523}) and we infer that%
\[
\widehat{\sigma}^{k^{\prime\prime}}(\widehat{x})(k)=\widehat{x}(k+k^{\prime
\prime})=\widehat{x}(k+k^{\prime}+1)\overset{(\ref{pre523})}{=}\widehat
{x}(k+k^{\prime}+1+p)=
\]%
\[
=\widehat{x}(k+k^{\prime\prime}+p)=\widehat{\sigma}^{k^{\prime\prime}%
}(\widehat{x})(k+p).
\]
We suppose now that in addition $k-p\geq-1,$ thus $k+k^{\prime}+1-p\geq
k^{\prime}$ and we can apply again (\ref{pre523}) in order to obtain%
\[
\widehat{\sigma}^{k^{\prime\prime}}(\widehat{x})(k)=\widehat{x}(k+k^{\prime
\prime})=\widehat{x}(k+k^{\prime}+1)\overset{(\ref{pre523})}{=}\widehat
{x}(k+k^{\prime}+1-p)=
\]%
\[
=\widehat{x}(k+k^{\prime\prime}-p)=\widehat{\sigma}^{k^{\prime\prime}%
}(\widehat{x})(k-p).
\]

(\ref{pre537})$\Longrightarrow$(\ref{per46}) The statement (\ref{pre537})
written for $p=1$ and $k^{\prime\prime}=0$ becomes:%
\begin{equation}
\left\{
\begin{array}
[c]{c}%
\forall\mu\in\widehat{Or}(\widehat{x}),\forall k\in\mathbf{N}_{\_},\widehat
{x}(k)=\mu\Longrightarrow\\
\Longrightarrow(\widehat{x}(k)=\widehat{x}(k+1)\text{ and }k\geq
0\Longrightarrow\widehat{x}(k)=\widehat{x}(k-1)).
\end{array}
\right.  \label{pre748}%
\end{equation}
Let $\mu\in\widehat{Or}(\widehat{x})$ arbitrary, thus $k_{1}\in\mathbf{N}%
_{\_}$ exists with $\widehat{x}(k_{1})=\mu.$ From (\ref{pre748}) we infer:%
\[
\widehat{x}(k_{1})=\widehat{x}(k_{1}-1)=\widehat{x}(k_{1}-2)=...=\widehat
{x}(-1),
\]%
\[
\widehat{x}(k_{1})=\widehat{x}(k_{1}+1)=\widehat{x}(k_{1}+2)=...
\]
We have obtained that (\ref{per46}) holds.

b) (\ref{per78})$\Longrightarrow$(\ref{per186}) The hypothesis states the
existence of $\mu\in\mathbf{B}^{n}$ such that $\{x(t)|t\in\mathbf{R}%
\}=\{\mu\}.$ Let us take $T>0$ and $t^{\prime}\in I^{x}$ arbitrarily. Let
furthermore $t\in\mathbf{T}_{\mu}^{x}\cap\lbrack t^{\prime},\infty
)$(=$[t^{\prime},\infty)$) and $z\in\mathbf{Z}$ having the property that
$t+zT\geq t^{\prime}.$ We have $x(t+zT)=\mu,$ thus $t+zT\in\mathbf{T}_{\mu
}^{x}.$ These imply the truth of (\ref{per186}).

(\ref{per186})$\Longrightarrow$(\ref{pre548}) Let $T>0$ and $\mu\in Or(x)$
arbitrary. (\ref{per186}) shows the existence of $t^{\prime}\in I^{x}$ with
the property%
\begin{equation}
\forall t\in\mathbf{T}_{\mu}^{x}\cap\lbrack t^{\prime},\infty),\{t+zT|z\in
\mathbf{Z}\}\cap\lbrack t^{\prime},\infty)\subset\mathbf{T}_{\mu}^{x}.
\label{pre657}%
\end{equation}
Let us take $t_{1}^{\prime}\geq t^{\prime}$ arbitrary. If $\mathbf{T}_{\mu
}^{x}\cap\lbrack t_{1}^{\prime},\infty)=\varnothing,$ then the statement%
\[
\forall t\in\mathbf{T}_{\mu}^{x}\cap\lbrack t_{1}^{\prime},\infty
),\{t+zT|z\in\mathbf{Z}\}\cap\lbrack t_{1}^{\prime},\infty)\subset
\mathbf{T}_{\mu}^{x}%
\]
is trivially true, so we suppose $\mathbf{T}_{\mu}^{x}\cap\lbrack
t_{1}^{\prime},\infty)\neq\varnothing$ and let $t\in\mathbf{T}_{\mu}^{x}%
\cap\lbrack t_{1}^{\prime},\infty),z\in\mathbf{Z}$ arbitrary such that
$t+zT\geq t_{1}^{\prime}.$ We have $t\in\mathbf{T}_{\mu}^{x}\cap\lbrack
t^{\prime},\infty),t+zT\geq t^{\prime}$ and we can apply (\ref{pre657}). We
infer $t+zT\in\mathbf{T}_{\mu}^{x}.$

(\ref{pre548})$\Longrightarrow$(\ref{pre549}) Let $T>0,\mu\in Or(x)$
arbitrary. (\ref{pre548}) shows the existence of $t^{\prime\prime\prime}\in
I^{x}$ with%
\begin{equation}
\forall t\in\mathbf{T}_{\mu}^{x}\cap\lbrack t^{\prime\prime\prime}%
,\infty),\{t+zT|z\in\mathbf{Z}\}\cap\lbrack t^{\prime\prime\prime}%
,\infty)\subset\mathbf{T}_{\mu}^{x} \label{pre658}%
\end{equation}
true. When writing (\ref{pre658}) we have taken in (\ref{pre548})
$t_{1}^{\prime}=t^{\prime\prime\prime}(=t^{\prime}).$

Let $t^{\prime\prime}\in\mathbf{R}$ arbitrary. We have the following possibilities.

Case $t^{\prime\prime}\leq t^{\prime\prime\prime}$

Then, from $t^{\prime\prime\prime}\in I^{x}$, $\sigma^{t^{\prime\prime}}(x)=x$
and (\ref{pre658}) we have the truth of (\ref{pre549}) with $t^{\prime
}=t^{\prime\prime\prime}.$

Case $t^{\prime\prime}>t^{\prime\prime\prime}$

Some $\varepsilon>0$ exists with the property that $\forall t\in
(t^{\prime\prime}-\varepsilon,t^{\prime\prime}),x(t)=x(t^{\prime\prime}-0).$
We take $t^{\prime}\in(t^{\prime\prime}-\varepsilon,t^{\prime\prime}%
)\cap(t^{\prime\prime\prime},t^{\prime\prime})$ arbitrarily. The fact that%
\[
(-\infty,t^{\prime}]\subset\mathbf{T}_{x(t^{\prime\prime}-0)}^{\sigma
^{t^{\prime\prime}}(x)}%
\]
is obvious. If $\mathbf{T}_{\mu}^{\sigma^{t^{\prime\prime}}(x)}\cap\lbrack
t^{\prime},\infty)=\varnothing,$ then the property%
\begin{equation}
\forall t\in\mathbf{T}_{\mu}^{\sigma^{t^{\prime\prime}}(x)}\cap\lbrack
t^{\prime},\infty),\{t+zT|z\in\mathbf{Z}\}\cap\lbrack t^{\prime}%
,\infty)\subset\mathbf{T}_{\mu}^{\sigma^{t^{\prime\prime}}(x)}%
\end{equation}
is true, thus we can suppose that $\mathbf{T}_{\mu}^{\sigma^{t^{\prime\prime}%
}(x)}\cap\lbrack t^{\prime},\infty)\neq\varnothing$ and let $t\in
\mathbf{T}_{\mu}^{\sigma^{t^{\prime\prime}}(x)}\cap\lbrack t^{\prime},\infty)$
arbitrary, fixed. We notice that $\forall t\geq t^{\prime},\sigma
^{t^{\prime\prime}}(x)(t)=x(t),$ thus $\mathbf{T}_{\mu}^{\sigma^{t^{\prime
\prime}}(x)}\cap\lbrack t^{\prime},\infty)=\mathbf{T}_{\mu}^{x}\cap\lbrack
t^{\prime},\infty).$ We take $z\in\mathbf{Z}$ arbitrary with $t+zT\geq
t^{\prime}.$ Because in this situation $t\in\mathbf{T}_{\mu}^{x}\cap\lbrack
t^{\prime\prime\prime},\infty)$ and $t+zT\geq t^{\prime\prime\prime},$ we can
apply (\ref{pre658}) and we infer $t+zT\in\mathbf{T}_{\mu}^{x},$ i.e.
$x(t+zT)=\mu=\sigma^{t^{\prime\prime}}(x)(t+zT)$ and finally $t+zT\in
\mathbf{T}_{\mu}^{\sigma^{t^{\prime\prime}}(x)}.$

(\ref{pre549})$\Longrightarrow$(\ref{pre550}) Let $T>0,\mu\in Or(x)$
arbitrary, fixed. The existence of $x(-\infty+0)$ shows that in (\ref{pre549})
we can choose $t^{\prime\prime}\in\mathbf{R}$ sufficiently small such that
$\sigma^{t^{\prime\prime}}(x)=x$ (see Theorem \ref{Pro2} a), page
\pageref{Pro2}). For that choice of $t^{\prime\prime},$ (\ref{pre549}) shows
the existence of $t^{\prime}\in I^{\sigma^{t^{\prime\prime}}(x)}$ with%
\begin{equation}
\forall t\in\mathbf{T}_{\mu}^{x}\cap\lbrack t^{\prime},\infty),\{t+zT|z\in
\mathbf{Z}\}\cap\lbrack t^{\prime},\infty)\subset\mathbf{T}_{\mu}^{x}
\label{pre660}%
\end{equation}
true. If $\forall t\geq t^{\prime},x(t)\neq\mu,$ then the statement
\[
\forall t\geq t^{\prime},x(t)=\mu\Longrightarrow(x(t)=x(t+T)\text{ and
}t-T\geq t^{\prime}\Longrightarrow x(t)=x(t-T))
\]
is trivially true, so we can suppose that $\mathbf{T}_{\mu}^{x}\cap\lbrack
t^{\prime},\infty)\neq\varnothing.$ Let $t\geq t^{\prime}$ arbitrary such that
$x(t)=\mu,$ in other words $t\in\mathbf{T}_{\mu}^{x}\cap\lbrack t^{\prime
},\infty).$ As far as%
\[
t+T\in\{t+zT|z\in\mathbf{Z}\}\cap\lbrack t^{\prime},\infty)\overset
{(\ref{pre660})}{\subset}\mathbf{T}_{\mu}^{x},
\]
we conclude that $t+T\in\mathbf{T}_{\mu}^{x},$ i.e. $x(t+T)=\mu=x(t).$ If in
addition $t-T\geq t^{\prime},$ then%
\[
t-T\in\{t+zT|z\in\mathbf{Z}\}\cap\lbrack t^{\prime},\infty)\overset
{(\ref{pre660})}{\subset}\mathbf{T}_{\mu}^{x},
\]
wherefrom $t-T\in\mathbf{T}_{\mu}^{x},$ i.e. $x(t-T)=\mu=x(t).$

(\ref{pre550})$\Longrightarrow$(\ref{pre551}) We take $T>0$ and $\mu\in Or(x)$
arbitrary, for which the truth of (\ref{pre550}) shows the existence of
$t^{\prime}\in I^{x}$ such that%
\begin{equation}
\forall t\geq t^{\prime},x(t)=\mu\Longrightarrow(x(t)=x(t+T)\text{ and
}t-T\geq t^{\prime}\Longrightarrow x(t)=x(t-T)). \label{pre662}%
\end{equation}
Let now $t_{1}^{\prime}\geq t^{\prime}$ arbitrary. If $\forall t\geq
t_{1}^{\prime},x(t)\neq\mu,$ the statement%
\[
\forall t\geq t_{1}^{\prime},x(t)=\mu\Longrightarrow(x(t)=x(t+T)\text{ and
}t-T\geq t_{1}^{\prime}\Longrightarrow x(t)=x(t-T))
\]
is trivially true, so we suppose that we can take $t\geq t_{1}^{\prime}$
arbitrarily with $x(t)=\mu.$ As $t\geq t^{\prime},$ we conclude from
(\ref{pre662}) that
\begin{equation}
x(t)=x(t+T)\text{ and }t-T\geq t^{\prime}\Longrightarrow x(t)=x(t-T)
\label{pre663}%
\end{equation}
holds. If $t-T\geq t_{1}^{\prime},$ then $t-T\geq t^{\prime}$ and from
(\ref{pre663}) we have that $x(t)=x(t-T).$

(\ref{pre551})$\Longrightarrow$(\ref{pre552}) Let $T>0,\mu\in Or(x)$
arbitrary, fixed. (\ref{pre551}) shows the existence of $t^{\prime\prime
\prime}\in I^{x}$ such that, in the special case when $t_{1}^{\prime}\geq
t^{\prime\prime\prime}$ holds as equality$,$ we have%
\begin{equation}
\forall t\geq t^{\prime\prime\prime},x(t)=\mu\Longrightarrow(x(t)=x(t+T)\text{
and }t-T\geq t^{\prime\prime\prime}\Longrightarrow x(t)=x(t-T)).
\label{pre665}%
\end{equation}
Let $t^{\prime\prime}\in\mathbf{R}$ arbitrary, fixed. We have the following possibilities.

Case $t^{\prime\prime}\leq t^{\prime\prime\prime}$

We have $\sigma^{t^{\prime\prime}}(x)=x$ and, from (\ref{pre665}) we have the
truth of (\ref{pre552}), with $t^{\prime}=t^{\prime\prime\prime}.$

Case $t^{\prime\prime}>t^{\prime\prime\prime}$

Some $\varepsilon>0$ exists with the property that $\forall t\in
(t^{\prime\prime}-\varepsilon,t^{\prime\prime}),x(t)=x(t^{\prime\prime}-0).$
We take $t^{\prime}\in(t^{\prime\prime}-\varepsilon,t^{\prime\prime}%
)\cap(t^{\prime\prime\prime},t^{\prime\prime})$ arbitrary and, from
(\ref{pre665}) and Lemma \ref{Lem30}, page \pageref{Lem30}, we infer%
\begin{equation}
\forall t\geq t^{\prime},x(t)=\mu\Longrightarrow(x(t)=x(t+T)\text{ and
}t-T\geq t^{\prime}\Longrightarrow x(t)=x(t-T)). \label{pre979}%
\end{equation}
We notice that $(-\infty,t^{\prime}]\subset\mathbf{T}_{x(t^{\prime\prime}%
-0)}^{\sigma^{t^{\prime\prime}}(x)},$ thus $t^{\prime}\in I^{\sigma
^{t^{\prime\prime}}(x)}.$ If $\forall t\geq t^{\prime},\sigma^{t^{\prime
\prime}}(x)(t)\neq\mu,$ then
\[
\left\{
\begin{array}
[c]{c}%
\forall t\geq t^{\prime},\sigma^{t^{\prime\prime}}(x)(t)=\mu\Longrightarrow
(\sigma^{t^{\prime\prime}}(x)(t)=\sigma^{t^{\prime\prime}}(x)(t+T)\text{
and}\\
\text{and }t-T\geq t^{\prime}\Longrightarrow\sigma^{t^{\prime\prime}%
}(x)(t)=\sigma^{t^{\prime\prime}}(x)(t-T))
\end{array}
\right.
\]
is trivially true, so let $t\geq t^{\prime}$ arbitrary with $\sigma
^{t^{\prime\prime}}(x)(t)=\mu.$ As for $t\geq t^{\prime},\sigma^{t^{\prime
\prime}}(x)(t)=x(t)$ (irrespective of the fact that $t<t^{\prime\prime}$ or
$t\geq t^{\prime\prime}$), we can apply (\ref{pre979}).

(\ref{pre552})$\Longrightarrow$(\ref{per78}) We suppose against all reason
that $\mu,\mu^{\prime}\in Or(x)$ exist, $\mu\neq\mu^{\prime},$ meaning the
existence of $t_{1},t_{2}\in\mathbf{R}$ with $x(t_{1})=\mu,x(t_{2}%
)=\mu^{\prime}.$ We can suppose without loss that $t_{1}>t_{2}.$ We write
(\ref{pre552}) for the period $T^{\prime}=t_{1}-t_{2}>0,$ for $\mu^{\prime}$
and for $t^{\prime\prime}\in\mathbf{R}$ sufficiently small such that
$\sigma^{t^{\prime\prime}}(x)=x.$ Some $t^{\prime}\in I^{x}$ exists with%
\begin{equation}
\forall t\geq t^{\prime},x(t)=\mu^{\prime}\Longrightarrow(x(t)=x(t+T^{\prime
})\text{ and }t-T^{\prime}\geq t^{\prime}\Longrightarrow x(t)=x(t-T^{\prime
})). \label{pre751}%
\end{equation}
We use the fact that $Or(x)=\{x(t)|t\geq t^{\prime}\},$ thus $t^{\prime}\leq
t_{2}<t_{1}$ and we have%
\[
\mu^{\prime}=x(t_{2})\overset{(\ref{pre751})}{=}x(t_{2}+T^{\prime}%
)=x(t_{2}+t_{1}-t_{2})=x(t_{1})=\mu,
\]
contradiction with the supposition that $\mu\neq\mu^{\prime}.$ (\ref{per78}) holds.
\end{proof}

\section{The third group of constancy properties}

\begin{remark}
The third group of constancy properties involves periodicity and eventual
periodicity properties of some point $\mu$ of the orbit. The constancy
properties result from (\ref{per185}),...,(\ref{pre537}) and (\ref{per186}%
),...,(\ref{pre552}) of the second group of properties, by the replacement of
$\forall\mu\in\widehat{Or}(\widehat{x}),\forall\mu\in Or(x)$ with $\exists
\mu\in\widehat{Or}(\widehat{x}),\exists\mu\in Or(x).$ The proofs of the
implications are similar, most of them, with the proofs of Theorem
\ref{The95}, page \pageref{The95}.
\end{remark}

\begin{remark}
\label{Rem17}The properties are also inspired by the eventual constancy
properties of Theorem \ref{The98}, page \pageref{The98}. Note that:

- (\ref{pre752}), (\ref{pre753}) (the last contains $\forall k^{\prime}%
\in\mathbf{N}_{\_}$) are inspired by (\ref{pre785})$_{page\;\pageref{pre785}}$
(containing $\exists k^{\prime}\in\mathbf{N}_{\_},\widehat{\mathbf{T}}_{\mu
}^{\widehat{x}}\cap\{k^{\prime},k^{\prime}+1,k^{\prime}+2,...\}\neq
\varnothing$);

- (\ref{pre754}) (containing $\forall k^{\prime\prime}\in\mathbf{N}$) is
inspired by (\ref{pre782})$_{page\;\pageref{pre782}}$ (containing $\exists
k^{\prime\prime}\in\mathbf{N},\widehat{\mathbf{T}}_{\mu}^{\widehat{\sigma
}^{k^{\prime\prime}}(\widehat{x})}\neq\varnothing$);

- (\ref{pre755}), (\ref{pre756}) (the last contains $\forall k^{\prime}%
\in\mathbf{N}_{\_}$) are inspired by (\ref{pre783})$_{page\;\pageref{pre783}}$
(containing $\exists k^{\prime}\in\mathbf{N}_{\_},\exists k_{1}\geq k^{\prime
},\widehat{x}(k_{1})=\mu$);

- (\ref{pre757}) (containing $\forall k^{\prime\prime}\in\mathbf{N}$) is
inspired by (\ref{pre784})$_{page\;\pageref{pre784}}$ (containing $\exists
k^{\prime\prime}\in\mathbf{N},\exists k_{1}\in\mathbf{N}_{\_},\widehat{\sigma
}^{k^{\prime\prime}}(\widehat{x})(k_{1})=\mu$);

- (\ref{pre758}) and (\ref{pre759}) (the last contains $\forall t_{1}^{\prime
}\geq t^{\prime}$) are inspired by (\ref{pre786})$_{page\;\pageref{pre786}}$
(containing $\exists t_{1}^{\prime}\geq t^{\prime},\mathbf{T}_{\mu}^{x}%
\cap\lbrack t_{1}^{\prime},\infty)\neq\varnothing$) and (\ref{pre787}%
)$_{page\;\pageref{pre787}}$ (containing $\exists t_{1}^{\prime}\in
\mathbf{R},\mathbf{T}_{\mu}^{x}\cap\lbrack t_{1}^{\prime},\infty
)\neq\varnothing$);

- (\ref{pre760}) (containing $\forall t^{\prime\prime}\in\mathbf{R}$) is
inspired by (\ref{pre788})$_{page\;\pageref{pre788}}$ and (\ref{pre789}%
)$_{page\;\pageref{pre789}}$ (containing $\exists t^{\prime\prime}%
\in\mathbf{R}$...$\mathbf{T}_{\mu}^{\sigma^{t^{\prime\prime}}(x)}\cap\lbrack
t^{\prime},\infty)\neq\varnothing$);

- (\ref{pre761}) and (\ref{pre762}) (the last contains $\forall t_{1}^{\prime
}\geq t^{\prime}$) are inspired by (\ref{pre790})$_{page\;\pageref{pre790}}$
(containing $\exists t_{1}^{\prime}\geq t^{\prime},\exists t_{2}^{\prime}\geq
t_{1}^{\prime},x(t_{2}^{\prime})=\mu$) and (\ref{pre791}%
)$_{page\;\pageref{pre791}}$ (containing $\exists t_{1}^{\prime}\in
\mathbf{R},\exists t_{2}^{\prime}\geq t_{1}^{\prime},x(t_{2}^{\prime})=\mu$);

- (\ref{pre763}) (containing $\forall t^{\prime\prime}\in\mathbf{R}$) is
inspired by (\ref{pre792})$_{page\;\pageref{pre792}}$ and (\ref{pre793}%
)$_{page\;\pageref{pre793}}$ (containing $\exists t^{\prime\prime}%
\in\mathbf{R}$...$\exists t^{\prime\prime\prime}\geq t^{\prime},\sigma
^{t^{\prime\prime}}(x)(t^{\prime\prime\prime})=\mu$).
\end{remark}

\begin{remark}
We refer also to Theorem \ref{The99}, page \pageref{The99} that makes use of
the eventual periodicity of some points of the omega limit set. Here are the differences:

- (\ref{pre752}), (\ref{pre753}) (the last contains $\exists\mu\in\widehat
{Or}(\widehat{x}),\forall k^{\prime}\in\mathbf{N}_{\_}$) are inspired by
(\ref{pre795})$_{page\;\pageref{pre795}}$ (containing $\exists\mu\in
\widehat{\omega}(\widehat{x}),\exists k^{\prime}\in\mathbf{N}_{\_} $);

- (\ref{pre754}) (containing $\exists\mu\in\widehat{Or}(\widehat{x}),\forall
k^{\prime\prime}\in\mathbf{N}$) is inspired by (\ref{pre796}%
)$_{page\;\pageref{pre796}}$ (containing $\exists\mu\in\widehat{\omega
}(\widehat{x}),\exists k^{\prime\prime}\in\mathbf{N}$);

- (\ref{pre755}), (\ref{pre756}) (the last contains $\exists\mu\in\widehat
{Or}(\widehat{x}),\forall k^{\prime}\in\mathbf{N}_{\_}$) are inspired by
(\ref{pre797})$_{page\;\pageref{pre797}}$ (containing $\exists\mu\in
\widehat{\omega}(\widehat{x}),\exists k^{\prime}\in\mathbf{N}_{\_} $);

- (\ref{pre757}) (containing $\exists\mu\in\widehat{Or}(\widehat{x}),\forall
k^{\prime\prime}\in\mathbf{N}$) is inspired by (\ref{pre798}%
)$_{page\;\pageref{pre798}}$ (containing $\exists\mu\in\widehat{\omega
}(\widehat{x}),\exists k^{\prime\prime}\in\mathbf{N}$);

- (\ref{pre758}) and (\ref{pre759}) (the last contains $\exists\mu\in
Or(x),...,\forall t_{1}^{\prime}\geq t^{\prime}$) are inspired by
(\ref{pre799})$_{page\;\pageref{pre799}}$ (containing $\exists\mu\in
\omega(x),...,\exists t_{1}^{\prime}\geq t^{\prime}$) and (\ref{pre800}%
)$_{page\;\pageref{pre800}}$ (containing $\exists\mu\in\omega(x),\exists
t_{1}^{\prime}\in\mathbf{R}$);

- (\ref{pre760}) (containing $\exists\mu\in Or(x),\forall t^{\prime\prime}%
\in\mathbf{R}$) is inspired by (\ref{pre801})$_{page\;\pageref{pre801}}$ and
(\ref{pre802})$_{page\;\pageref{pre802}}$ (containing $\exists\mu\in
\omega(x),\exists t^{\prime\prime}\in\mathbf{R}$);

- (\ref{pre761}) and (\ref{pre762}) (the last contains $\exists\mu\in
Or(x),...,$ $\forall t_{1}^{\prime}\geq t^{\prime}$) are inspired by
(\ref{pre803})$_{page\;\pageref{pre803}}$ (containing $\exists\mu\in\omega(x),
$...$,$ $\exists t_{1}^{\prime}\geq t^{\prime}$) and (\ref{pre804}%
)$_{page\;\pageref{pre804}}$ (containing $\exists\mu\in\omega(x),\exists
t_{1}^{\prime}\in\mathbf{R}$);

- (\ref{pre763}) (containing $\exists\mu\in Or(x),\forall t^{\prime\prime}%
\in\mathbf{R}$) is inspired by (\ref{pre805})$_{page\;\pageref{pre805}}$ and
(\ref{pre806})$_{page\;\pageref{pre806}}$ (containing $\exists\mu\in
\omega(x),\exists t^{\prime\prime}\in\mathbf{R}$).
\end{remark}

\begin{theorem}
\label{The94}Let the signals $\widehat{x}\in\widehat{S}^{(n)},x\in S^{(n)}.$

a) The following properties are equivalent with the constancy of $\widehat
{x}\in\widehat{S}^{(n)}$:%
\begin{equation}
\forall p\geq1,\exists\mu\in\widehat{Or}(\widehat{x}),\forall k\in
\widehat{\mathbf{T}}_{\mu}^{\widehat{x}},\{k+zp|z\in\mathbf{Z}\}\cap
\mathbf{N}_{\_}\subset\widehat{\mathbf{T}}_{\mu}^{\widehat{x}}, \label{pre752}%
\end{equation}%
\begin{equation}
\left\{
\begin{array}
[c]{c}%
\forall p\geq1,\exists\mu\in\widehat{Or}(\widehat{x}),\forall k^{\prime}%
\in\mathbf{N}_{\_},\forall k\in\widehat{\mathbf{T}}_{\mu}^{\widehat{x}}%
\cap\{k^{\prime},k^{\prime}+1,k^{\prime}+2,...\},\\
\{k+zp|z\in\mathbf{Z}\}\cap\{k^{\prime},k^{\prime}+1,k^{\prime}+2,...\}\subset
\widehat{\mathbf{T}}_{\mu}^{\widehat{x}},
\end{array}
\right.  \label{pre753}%
\end{equation}%
\begin{equation}
\left\{
\begin{array}
[c]{c}%
\forall p\geq1,\exists\mu\in\widehat{Or}(\widehat{x}),\forall k^{\prime\prime
}\in\mathbf{N},\forall k\in\widehat{\mathbf{T}}_{\mu}^{\widehat{\sigma
}^{k^{\prime\prime}}(\widehat{x})},\\
\{k+zp|z\in\mathbf{Z}\}\cap\mathbf{N}_{\_}\subset\widehat{\mathbf{T}}_{\mu
}^{\widehat{\sigma}^{k^{\prime\prime}}(\widehat{x})},
\end{array}
\right.  \label{pre754}%
\end{equation}%
\begin{equation}
\left\{
\begin{array}
[c]{c}%
\forall p\geq1,\exists\mu\in\widehat{Or}(\widehat{x}),\forall k\in
\mathbf{N}_{\_},\widehat{x}(k)=\mu\Longrightarrow\\
\Longrightarrow(\widehat{x}(k)=\widehat{x}(k+p)\text{ and }k-p\geq
-1\Longrightarrow\widehat{x}(k)=\widehat{x}(k-p)),
\end{array}
\right.  \label{pre755}%
\end{equation}%
\begin{equation}
\left\{
\begin{array}
[c]{c}%
\forall p\geq1,\exists\mu\in\widehat{Or}(\widehat{x}),\forall k^{\prime}%
\in\mathbf{N}_{\_},\forall k\geq k^{\prime},\widehat{x}(k)=\mu\Longrightarrow
\\
\Longrightarrow(\widehat{x}(k)=\widehat{x}(k+p)\text{ and }k-p\geq k^{\prime
}\Longrightarrow\widehat{x}(k)=\widehat{x}(k-p)),
\end{array}
\right.  \label{pre756}%
\end{equation}%
\begin{equation}
\left\{
\begin{array}
[c]{c}%
\forall p\geq1,\exists\mu\in\widehat{Or}(\widehat{x}),\forall k^{\prime\prime
}\in\mathbf{N},\forall k\in\mathbf{N}_{\_},\widehat{\sigma}^{k^{\prime\prime}%
}(\widehat{x})(k)=\mu\Longrightarrow\\
\Longrightarrow(\widehat{\sigma}^{k^{\prime\prime}}(\widehat{x})(k)=\widehat
{\sigma}^{k^{\prime\prime}}(\widehat{x})(k+p)\text{ and }\\
\text{and }k-p\geq-1\Longrightarrow\widehat{\sigma}^{k^{\prime\prime}%
}(\widehat{x})(k)=\widehat{\sigma}^{k^{\prime\prime}}(\widehat{x})(k-p)).
\end{array}
\right.  \label{pre757}%
\end{equation}

b) The following properties are equivalent with the constancy of $x\in
S^{(n)}$:%
\begin{equation}
\left\{
\begin{array}
[c]{c}%
\forall T>0,\exists\mu\in Or(x),\exists t^{\prime}\in I^{x},\\
\forall t\in\mathbf{T}_{\mu}^{x}\cap\lbrack t^{\prime},\infty),\{t+zT|z\in
\mathbf{Z}\}\cap\lbrack t^{\prime},\infty)\subset\mathbf{T}_{\mu}^{x},
\end{array}
\right.  \label{pre758}%
\end{equation}%
\begin{equation}
\left\{
\begin{array}
[c]{c}%
\forall T>0,\exists\mu\in Or(x),\exists t^{\prime}\in I^{x},\\
\forall t_{1}^{\prime}\geq t^{\prime},\forall t\in\mathbf{T}_{\mu}^{x}%
\cap\lbrack t_{1}^{\prime},\infty),\{t+zT|z\in\mathbf{Z}\}\cap\lbrack
t_{1}^{\prime},\infty)\subset\mathbf{T}_{\mu}^{x},
\end{array}
\right.  \label{pre759}%
\end{equation}%
\begin{equation}
\left\{
\begin{array}
[c]{c}%
\forall T>0,\exists\mu\in Or(x),\forall t^{\prime\prime}\in\mathbf{R},\exists
t^{\prime}\in I^{\sigma^{t^{\prime\prime}}(x)},\\
\forall t\in\mathbf{T}_{\mu}^{\sigma^{t^{\prime\prime}}(x)}\cap\lbrack
t^{\prime},\infty),\{t+zT|z\in\mathbf{Z}\}\cap\lbrack t^{\prime}%
,\infty)\subset\mathbf{T}_{\mu}^{\sigma^{t^{\prime\prime}}(x)},
\end{array}
\right.  \label{pre760}%
\end{equation}%
\begin{equation}
\left\{
\begin{array}
[c]{c}%
\forall T>0,\exists\mu\in Or(x),\exists t^{\prime}\in I^{x},\forall t\geq
t^{\prime},\\
x(t)=\mu\Longrightarrow(x(t)=x(t+T)\text{ and }t-T\geq t^{\prime
}\Longrightarrow x(t)=x(t-T)),
\end{array}
\right.  \label{pre761}%
\end{equation}%
\begin{equation}
\left\{
\begin{array}
[c]{c}%
\forall T>0,\exists\mu\in Or(x),\exists t^{\prime}\in I^{x},\forall
t_{1}^{\prime}\geq t^{\prime},\forall t\geq t_{1}^{\prime},x(t)=\mu
\Longrightarrow\\
\Longrightarrow(x(t)=x(t+T)\text{ and }t-T\geq t_{1}^{\prime}\Longrightarrow
x(t)=x(t-T)),
\end{array}
\right.  \label{pre762}%
\end{equation}%
\begin{equation}
\left\{
\begin{array}
[c]{c}%
\forall T>0,\exists\mu\in Or(x),\forall t^{\prime\prime}\in\mathbf{R},\exists
t^{\prime}\in I^{\sigma^{t^{\prime\prime}}(x)},\\
\forall t\geq t^{\prime},\sigma^{t^{\prime\prime}}(x)(t)=\mu\Longrightarrow
(\sigma^{t^{\prime\prime}}(x)(t)=\sigma^{t^{\prime\prime}}(x)(t+T)\text{
and}\\
\text{and }t-T\geq t^{\prime}\Longrightarrow\sigma^{t^{\prime\prime}%
}(x)(t)=\sigma^{t^{\prime\prime}}(x)(t-T)).
\end{array}
\right.  \label{pre763}%
\end{equation}

\end{theorem}

\begin{proof}
a) (\ref{per46})$\Longrightarrow$(\ref{pre752}) Let us prove first that
(\ref{per46}) implies%
\begin{equation}
\exists\mu\in\widehat{Or}(\widehat{x}),\forall p\geq1,\forall k\in
\widehat{\mathbf{T}}_{\mu}^{\widehat{x}},\{k+zp|z\in\mathbf{Z}\}\cap
\mathbf{N}_{\_}\subset\widehat{\mathbf{T}}_{\mu}^{\widehat{x}}. \label{pre927}%
\end{equation}
Indeed, the hypothesis states the existence of $\mu\in\widehat{Or}(\widehat
{x})$ with $\forall k\in\mathbf{N}_{\_},\widehat{x}(k)=\mu$ and let
$p\geq1,k\in\widehat{\mathbf{T}}_{\mu}^{\widehat{x}},z\in\mathbf{Z}$ arbitrary
such that $k+zp\geq-1.$ Then $\widehat{x}(k+zp)=\mu,$ thus $k+zp\in
\widehat{\mathbf{T}}_{\mu}^{\widehat{x}}$ and (\ref{pre927}) holds.
(\ref{pre927}) obviously implies (\ref{pre752}).

(\ref{pre752})$\Longrightarrow$(\ref{pre753}) We take $p\geq1$ arbitrarily.
The truth of (\ref{pre752}) shows the existence of $\mu\in\widehat
{Or}(\widehat{x})$ with%
\begin{equation}
\forall k\in\widehat{\mathbf{T}}_{\mu}^{\widehat{x}},\{k+zp|z\in
\mathbf{Z}\}\cap\mathbf{N}_{\_}\subset\widehat{\mathbf{T}}_{\mu}^{\widehat{x}%
}. \label{pre928}%
\end{equation}
Let now $k^{\prime}\in\mathbf{N}_{\_}$ arbitrary. If $\widehat{\mathbf{T}%
}_{\mu}^{\widehat{x}}\cap\{k^{\prime},k^{\prime}+1,k^{\prime}%
+2,...\}=\varnothing,$ then
\[
\forall k\in\widehat{\mathbf{T}}_{\mu}^{\widehat{x}}\cap\{k^{\prime}%
,k^{\prime}+1,k^{\prime}+2,...\},\{k+zp|z\in\mathbf{Z}\}\cap\{k^{\prime
},k^{\prime}+1,k^{\prime}+2,...\}\subset\widehat{\mathbf{T}}_{\mu}%
^{\widehat{x}}%
\]
is trivially true, so we suppose $\widehat{\mathbf{T}}_{\mu}^{\widehat{x}}%
\cap\{k^{\prime},k^{\prime}+1,k^{\prime}+2,...\}\neq\varnothing$ and let
$k\in\widehat{\mathbf{T}}_{\mu}^{\widehat{x}},$ $z\in\mathbf{Z}$ arbitrary,
fixed, such that $k\geq k^{\prime}$ and $k+zp\geq k^{\prime}.$ As
$k+zp\geq-1,$ we have from (\ref{pre928}) that $k+zp\in\widehat{\mathbf{T}%
}_{\mu}^{\widehat{x}}.$

(\ref{pre753})$\Longrightarrow$(\ref{pre754}) Let an arbitrary $p\geq1.$ We
have from (\ref{pre753}) the existence of $\mu\in\widehat{Or}(\widehat{x})$
such that%
\begin{equation}
\left\{
\begin{array}
[c]{c}%
\forall k^{\prime}\in\mathbf{N}_{\_},\forall k\in\widehat{\mathbf{T}}_{\mu
}^{\widehat{x}}\cap\{k^{\prime},k^{\prime}+1,k^{\prime}+2,...\},\\
\{k+zp|z\in\mathbf{Z}\}\cap\{k^{\prime},k^{\prime}+1,k^{\prime}+2,...\}\subset
\widehat{\mathbf{T}}_{\mu}^{\widehat{x}}%
\end{array}
\right.  \label{pre929}%
\end{equation}
holds and we take $k^{\prime\prime}\in\mathbf{N}$ arbitrary. If $\widehat
{\mathbf{T}}_{\mu}^{\widehat{\sigma}^{k^{\prime\prime}}(\widehat{x}%
)}=\varnothing,$ then
\[
\forall k\in\widehat{\mathbf{T}}_{\mu}^{\widehat{\sigma}^{k^{\prime\prime}%
}(\widehat{x})},\{k+zp|z\in\mathbf{Z}\}\cap\mathbf{N}_{\_}\subset
\widehat{\mathbf{T}}_{\mu}^{\widehat{\sigma}^{k^{\prime\prime}}(\widehat{x})}%
\]
is trivially true, thus we can suppose $\widehat{\mathbf{T}}_{\mu}%
^{\widehat{\sigma}^{k^{\prime\prime}}(\widehat{x})}\neq\varnothing$ and let
$k\in\widehat{\mathbf{T}}_{\mu}^{\widehat{\sigma}^{k^{\prime\prime}}%
(\widehat{x})},$ $z\in\mathbf{Z}$ arbitrary such that $k+zp\geq-1.$ We have
$\widehat{x}(k+k^{\prime\prime})=\mu$ or, if we denote $k^{\prime}%
=k^{\prime\prime}-1,$ then $\widehat{x}(k+k^{\prime}+1)=\mu,$ where
$k^{\prime}\in\mathbf{N}_{\_}.$ Of course that $k+k^{\prime}+1\geq k^{\prime
},$ thus $k+k^{\prime}+1\in\widehat{\mathbf{T}}_{\mu}^{\widehat{x}}%
\cap\{k^{\prime},k^{\prime}+1,k^{\prime}+2,...\}$ and, on the other hand,
$k+k^{\prime}+1+zp\geq k^{\prime}+1-1,$ resulting that we can apply
(\ref{pre929}), wherefrom $\widehat{x}(k+k^{\prime}+1+zp)=\mu.$ It has
resulted that $\widehat{\sigma}^{k^{\prime\prime}}(\widehat{x})(k+zp)=\widehat
{x}(k+k^{\prime\prime}+zp)=\mu,$ in other words $k+zp\in\widehat{\mathbf{T}%
}_{\mu}^{\widehat{\sigma}^{k^{\prime\prime}}(\widehat{x})}.$

(\ref{pre754})$\Longrightarrow$(\ref{pre755}) Let $p\geq1$ arbitrary.
(\ref{pre754}) shows the existence of $\mu\in\widehat{Or}(\widehat{x})$ with%
\begin{equation}
\forall k\in\widehat{\mathbf{T}}_{\mu}^{\widehat{x}},\{k+zp|z\in
\mathbf{Z}\}\cap\mathbf{N}_{\_}\subset\widehat{\mathbf{T}}_{\mu}^{\widehat{x}}
\label{pre930}%
\end{equation}
true (for $k^{\prime\prime}=0$ and $\widehat{\sigma}^{k^{\prime\prime}%
}(\widehat{x})=\widehat{x}$) and let $k\in\mathbf{N}_{\_}$ such that
$\widehat{x}(k)=\mu.$ We obtain%
\[
k+p\in\{k+zp|z\in\mathbf{Z}\}\cap\mathbf{N}_{\_}\overset{(\ref{pre930}%
)}{\subset}\widehat{\mathbf{T}}_{\mu}^{\widehat{x}},
\]
wherefrom $\widehat{x}(k+p)=\mu=\widehat{x}(k).$

If in addition $k-p\geq-1,$ then%
\[
k-p\in\{k+zp|z\in\mathbf{Z}\}\cap\mathbf{N}_{\_}\overset{(\ref{pre930}%
)}{\subset}\widehat{\mathbf{T}}_{\mu}^{\widehat{x}},
\]
wherefrom $\widehat{x}(k-p)=\mu=\widehat{x}(k).$

(\ref{pre755})$\Longrightarrow$(\ref{pre756}) We take an arbitrary $p\geq1$
and we have from (\ref{pre755}) the existence of $\mu\in\widehat{Or}%
(\widehat{x}),$ with%
\begin{equation}
\left\{
\begin{array}
[c]{c}%
\forall k\in\mathbf{N}_{\_},\widehat{x}(k)=\mu\Longrightarrow\\
\Longrightarrow(\widehat{x}(k)=\widehat{x}(k+p)\text{ and }k-p\geq
-1\Longrightarrow\widehat{x}(k)=\widehat{x}(k-p))
\end{array}
\right.  \label{pre931}%
\end{equation}
fulfilled. We take $k^{\prime}\in\mathbf{N}_{\_}$ arbitrary. If $\forall k\geq
k^{\prime},\widehat{x}(k)\neq\mu$ then
\[
\forall k\geq k^{\prime},\widehat{x}(k)=\mu\Longrightarrow(\widehat
{x}(k)=\widehat{x}(k+p)\text{ and }k-p\geq k^{\prime}\Longrightarrow
\widehat{x}(k)=\widehat{x}(k-p))
\]
is trivially true, thus we can take $k\geq k^{\prime}$ arbitrarily such that
$\widehat{x}(k)=\mu.$ From (\ref{pre931}) we have that $\widehat
{x}(k)=\widehat{x}(k+p).$ In case that $k-p\geq k^{\prime},$ as $k-p\geq-1,$
we can apply (\ref{pre931}) once again in order to infer that $\widehat
{x}(k)=\widehat{x}(k-p).$

(\ref{pre756})$\Longrightarrow$(\ref{pre757}) For an arbitrary $p\geq1,$ the
hypothesis states the existence of $\mu\in\widehat{Or}(\widehat{x})$ with%
\begin{equation}
\left\{
\begin{array}
[c]{c}%
\forall k^{\prime}\in\mathbf{N}_{\_},\forall k\geq k^{\prime},\widehat
{x}(k)=\mu\Longrightarrow\\
\Longrightarrow(\widehat{x}(k)=\widehat{x}(k+p)\text{ and }k-p\geq k^{\prime
}\Longrightarrow\widehat{x}(k)=\widehat{x}(k-p))
\end{array}
\right.  \label{pre932}%
\end{equation}
true. We take $k^{\prime\prime}\in\mathbf{N}$ arbitrary. If $\forall
k\in\mathbf{N}_{\_},\widehat{\sigma}^{k^{\prime\prime}}(\widehat{x})(k)\neq
\mu,$ then
\[
\left\{
\begin{array}
[c]{c}%
\forall k\in\mathbf{N}_{\_},\widehat{\sigma}^{k^{\prime\prime}}(\widehat
{x})(k)=\mu\Longrightarrow(\widehat{\sigma}^{k^{\prime\prime}}(\widehat
{x})(k)=\widehat{\sigma}^{k^{\prime\prime}}(\widehat{x})(k+p)\text{ and }\\
\text{and }k-p\geq-1\Longrightarrow\widehat{\sigma}^{k^{\prime\prime}%
}(\widehat{x})(k)=\widehat{\sigma}^{k^{\prime\prime}}(\widehat{x})(k-p))
\end{array}
\right.
\]
is trivially true, thus let $k\in\mathbf{N}_{\_}$ arbitrary with
$\widehat{\sigma}^{k^{\prime\prime}}(\widehat{x})(k)=\widehat{x}%
(k+k^{\prime\prime})=\mu.$ We denote $k^{\prime}=k^{\prime\prime}-1$ and we
see that $\widehat{x}(k+k^{\prime}+1)=\mu,$ where $k+k^{\prime}+1\geq
k^{\prime}.$ We can apply (\ref{pre932}) and we infer that%
\[
\widehat{\sigma}^{k^{\prime\prime}}(\widehat{x})(k)=\widehat{x}(k+k^{\prime
\prime})=\widehat{x}(k+k^{\prime}+1)\overset{(\ref{pre932})}{=}\widehat
{x}(k+k^{\prime}+1+p)=
\]%
\[
=\widehat{x}(k+k^{\prime\prime}+p)=\widehat{\sigma}^{k^{\prime\prime}%
}(\widehat{x})(k+p).
\]
We suppose now that in addition we have $k-p\geq-1,$ thus $k+k^{\prime
}+1-p\geq k^{\prime}$ and we can apply again (\ref{pre932}) in order to obtain%
\[
\widehat{\sigma}^{k^{\prime\prime}}(\widehat{x})(k)=\widehat{x}(k+k^{\prime
\prime})=\widehat{x}(k+k^{\prime}+1)\overset{(\ref{pre932})}{=}\widehat
{x}(k+k^{\prime}+1-p)=
\]%
\[
=\widehat{x}(k+k^{\prime\prime}-p)=\widehat{\sigma}^{k^{\prime\prime}%
}(\widehat{x})(k-p).
\]

(\ref{pre757})$\Longrightarrow$(\ref{per46}) The hypothesis written for $p=1 $
shows the existence of $\mu\in\widehat{Or}(\widehat{x})$ such that, in the
special case when $k^{\prime\prime}=0,$%
\begin{equation}
\left\{
\begin{array}
[c]{c}%
\forall k\in\mathbf{N}_{\_},\widehat{x}(k)=\mu\Longrightarrow\\
\Longrightarrow(\widehat{x}(k)=\widehat{x}(k+1)\text{ and }k\geq
0\Longrightarrow\widehat{x}(k)=\widehat{x}(k-1))
\end{array}
\right.  \label{pre764}%
\end{equation}
is fulfilled. Some $k_{1}\in\mathbf{N}_{\_}$ exists with $\widehat{x}%
(k_{1})=\mu$ and from (\ref{pre764}) we get:%
\[
\widehat{x}(k_{1})=\widehat{x}(k_{1}-1)=\widehat{x}(k_{1}-2)=...=\widehat
{x}(-1),
\]%
\[
\widehat{x}(k_{1})=\widehat{x}(k_{1}+1)=\widehat{x}(k_{1}+2)=...
\]
i.e. (\ref{per46}) holds.

b) (\ref{per126})$\Longrightarrow$(\ref{pre758}) Let $T>0$ arbitrary. The
hypothesis states the existence of $\mu\in\mathbf{B}^{n}$ such that
$\mathbf{T}_{\mu}^{x}=\{t|t\in\mathbf{R},x(t)=\mu\}=\mathbf{R},$ in particular
$\mu=x(-\infty+0).$ We take $t^{\prime}\in I^{x}$ arbitrarily and let
$t\in\mathbf{T}_{\mu}^{x}\cap\lbrack t^{\prime},\infty)=[t^{\prime}%
,\infty),z\in\mathbf{Z}$ with $t+zT\geq t^{\prime}.$ We conclude that
$t+zT\in\mathbf{T}_{\mu}^{x}.$ These imply the truth of (\ref{pre758}).

(\ref{pre758})$\Longrightarrow$(\ref{pre759}) Let $T>0$ arbitrary. The
hypothesis states the existence of $\mu\in Or(x),t^{\prime}\in I^{x}$ with the
property%
\begin{equation}
\forall t\in\mathbf{T}_{\mu}^{x}\cap\lbrack t^{\prime},\infty),\{t+zT|z\in
\mathbf{Z}\}\cap\lbrack t^{\prime},\infty)\subset\mathbf{T}_{\mu}^{x}.
\label{pre933}%
\end{equation}
Let us take $t_{1}^{\prime}\geq t^{\prime}$ arbitrary. If $\mathbf{T}_{\mu
}^{x}\cap\lbrack t_{1}^{\prime},\infty)=\varnothing,$ then the statement%
\[
\forall t\in\mathbf{T}_{\mu}^{x}\cap\lbrack t_{1}^{\prime},\infty
),\{t+zT|z\in\mathbf{Z}\}\cap\lbrack t_{1}^{\prime},\infty)\subset
\mathbf{T}_{\mu}^{x}%
\]
is trivially true, thus we can suppose from now that $\mathbf{T}_{\mu}^{x}%
\cap\lbrack t_{1}^{\prime},\infty)\neq\varnothing.$ We take $t\in
\mathbf{T}_{\mu}^{x}\cap\lbrack t_{1}^{\prime},\infty),z\in\mathbf{Z}$
arbitrarily such that $t+zT\geq t_{1}^{\prime}.$ We have $t\in\mathbf{T}_{\mu
}^{x}\cap\lbrack t^{\prime},\infty),t+zT\geq t^{\prime}$ and we can apply
(\ref{pre933}). We infer $t+zT\in\mathbf{T}_{\mu}^{x}.$

(\ref{pre759})$\Longrightarrow$(\ref{pre760}) Let $T>0$ arbitrary. From
(\ref{pre759}) we get the existence of $\mu\in Or(x)$ and $t^{\prime
\prime\prime}\in I^{x}$ with%
\begin{equation}
\forall t\in\mathbf{T}_{\mu}^{x}\cap\lbrack t^{\prime\prime\prime}%
,\infty),\{t+zT|z\in\mathbf{Z}\}\cap\lbrack t^{\prime\prime\prime}%
,\infty)\subset\mathbf{T}_{\mu}^{x} \label{pre934}%
\end{equation}
true. When writing (\ref{pre934}) we have taken in (\ref{pre759})
$t_{1}^{\prime}=t^{\prime\prime\prime}(=t^{\prime}).$ As $Or(x)=\{x(t)|t\geq
t^{\prime\prime\prime}\},$ we have $\mathbf{T}_{\mu}^{x}\cap\lbrack
t^{\prime\prime\prime},\infty)\neq\varnothing$ and (\ref{pre934}) shows that
$\mu\in\omega(x).$

Let $t^{\prime\prime}\in\mathbf{R}$ arbitrary. We have the following possibilities.

Case $t^{\prime\prime}\leq t^{\prime\prime\prime}$

Then, since $t^{\prime\prime}\in I^{x}$, we get $\sigma^{t^{\prime\prime}%
}(x)=x$ thus from (\ref{pre934}) we have the truth of (\ref{pre760}) with
$t^{\prime}=t^{\prime\prime\prime}.$

Case $t^{\prime\prime}>t^{\prime\prime\prime}$

Some $\varepsilon>0$ exists with the property that $\forall t\in
(t^{\prime\prime}-\varepsilon,t^{\prime\prime}),x(t)=x(t^{\prime\prime}-0).$
We take $t^{\prime}\in(t^{\prime\prime}-\varepsilon,t^{\prime\prime}%
)\cap(t^{\prime\prime\prime},t^{\prime\prime})$ arbitrarily and we have%
\[
\sigma^{t^{\prime\prime}}(x)(t)=\left\{
\begin{array}
[c]{c}%
x(t),t\geq t^{\prime}\\
x(t^{\prime\prime}-0),t<t^{\prime\prime}.
\end{array}
\right.
\]
The fact that $t^{\prime}\in(-\infty,t^{\prime\prime})\subset I^{\sigma
^{t^{\prime\prime}}(x)}$ is obvious. As far as $\mu\in\omega(x),$ we have
$\mathbf{T}_{\mu}^{\sigma^{t^{\prime\prime}}(x)}\cap\lbrack t^{\prime}%
,\infty)\neq\varnothing.$ Let $t\in\mathbf{T}_{\mu}^{\sigma^{t^{\prime\prime}%
}(x)}\cap\lbrack t^{\prime},\infty)$ arbitrary, fixed and we notice that
$\mathbf{T}_{\mu}^{\sigma^{t^{\prime\prime}}(x)}\cap\lbrack t^{\prime}%
,\infty)=\mathbf{T}_{\mu}^{x}\cap\lbrack t^{\prime},\infty).$ We take also
$z\in\mathbf{Z}$ arbitrary with $t+zT\geq t^{\prime}.$ Because in this
situation $t\in\mathbf{T}_{\mu}^{x}\cap\lbrack t^{\prime\prime\prime},\infty)$
and $t+zT\geq t^{\prime\prime\prime},$ we can apply (\ref{pre934}) and we
infer $t+zT\in\mathbf{T}_{\mu}^{x},$ i.e. $x(t+zT)=\mu=\sigma^{t^{\prime
\prime}}(x)(t+zT)$ and finally $t+zT\in\mathbf{T}_{\mu}^{\sigma^{t^{\prime
\prime}}(x)}.$

(\ref{pre760})$\Longrightarrow$(\ref{pre761}) Let $T>0$ arbitrary, fixed. From
(\ref{pre760}), some $\mu\in Or(x)$ exists such that%
\begin{equation}
\left\{
\begin{array}
[c]{c}%
\forall t^{\prime\prime}\in\mathbf{R},\exists t^{\prime}\in I^{\sigma
^{t^{\prime\prime}}(x)}\mathbf{,}\text{ }\\
\forall t\in\mathbf{T}_{\mu}^{\sigma^{t^{\prime\prime}}(x)}\cap\lbrack
t^{\prime},\infty),\{t+zT|z\in\mathbf{Z}\}\cap\lbrack t^{\prime}%
,\infty)\subset\mathbf{T}_{\mu}^{\sigma^{t^{\prime\prime}}(x)}.
\end{array}
\right.  \label{pre936}%
\end{equation}
The existence of $x(-\infty+0)$ shows that in (\ref{pre936}) we can choose
$t^{\prime\prime}\in\mathbf{R}$ sufficiently small such that $\sigma
^{t^{\prime\prime}}(x)=x.$ For this choice of $t^{\prime\prime},$
(\ref{pre936}) shows the existence of $t^{\prime}\in I^{x}$ with%
\begin{equation}
\forall t\in\mathbf{T}_{\mu}^{x}\cap\lbrack t^{\prime},\infty),\{t+zT|z\in
\mathbf{Z}\}\cap\lbrack t^{\prime},\infty)\subset\mathbf{T}_{\mu}^{x}
\label{pre937}%
\end{equation}
true. We have $Or(x)=\{x(t)|t\geq t^{\prime}\},$ wherefrom we get
$\mathbf{T}_{\mu}^{x}\cap\lbrack t^{\prime},\infty)\neq\varnothing.$ Let
$t\geq t^{\prime}$ arbitrary with $x(t)=\mu,$ in other words $t\in
\mathbf{T}_{\mu}^{x}\cap\lbrack t^{\prime},\infty).$ As far as%
\[
t+T\in\{t+zT|z\in\mathbf{Z}\}\cap\lbrack t^{\prime},\infty)\overset
{(\ref{pre937})}{\subset}\mathbf{T}_{\mu}^{x},
\]
we conclude that $t+T\in\mathbf{T}_{\mu}^{x},$ i.e. $x(t+T)=\mu=x(t).$ If in
addition $t-T\geq t^{\prime},$ then%
\[
t-T\in\{t+zT|z\in\mathbf{Z}\}\cap\lbrack t^{\prime},\infty)\overset
{(\ref{pre937})}{\subset}\mathbf{T}_{\mu}^{x},
\]
wherefrom $t-T\in\mathbf{T}_{\mu}^{x},$ i.e. $x(t-T)=\mu=x(t).$

(\ref{pre761})$\Longrightarrow$(\ref{pre762}) We take $T>0$ arbitrarily, for
which the truth of (\ref{pre761}) shows the existence of $\mu\in Or(x)$ and
$t^{\prime}\in I^{x}$ such that%
\begin{equation}
\forall t\geq t^{\prime},x(t)=\mu\Longrightarrow(x(t)=x(t+T)\text{ and
}t-T\geq t^{\prime}\Longrightarrow x(t)=x(t-T)). \label{pre938}%
\end{equation}
Let now $t_{1}^{\prime}\geq t^{\prime}$ arbitrary. As $\mu\in Or(x)$ we get
the existence of $t\geq t^{\prime}$ with $x(t)=\mu,$ wherefrom taking into
account (\ref{pre938}) we infer $\mu\in\omega(x).$ Let $t\geq t_{1}^{\prime}$
arbitrarily such that $x(t)=\mu$ (we can take such $t^{\prime}$s because
$\mu\in\omega(x)$ and $\mathbf{T}_{\mu}^{x}$ is superiorly unbounded). As
$t\geq t^{\prime},$ we conclude from (\ref{pre938}) that $x(t)=x(t+T)$ holds.
If in addition $t-T\geq t_{1}^{\prime},$ then $t-T\geq t^{\prime}$ and from
(\ref{pre938}) we have that $x(t)=x(t-T).$

(\ref{pre762})$\Longrightarrow$(\ref{pre763}) Let $T>0$ arbitrary, fixed.
(\ref{pre762}) shows the existence of $\mu\in Or(x)$ and $t^{\prime
\prime\prime}\in I^{x}$ such that, in the special case when $t_{1}^{\prime
}\geq t^{\prime\prime\prime}$ is true as equality, we have%
\begin{equation}
\forall t\geq t^{\prime\prime\prime},x(t)=\mu\Longrightarrow(x(t)=x(t+T)\text{
and }t-T\geq t^{\prime\prime\prime}\Longrightarrow x(t)=x(t-T)) \label{pre941}%
\end{equation}
and in particular we notice that $\mu\in\omega(x)$ and $\mathbf{T}_{\mu}^{x}$
is superiorly unbounded.

Let $t^{\prime\prime}\in\mathbf{R}$ arbitrary, fixed. We have the following possibilities.

Case $t^{\prime\prime}\leq t^{\prime\prime\prime}$

We have $\sigma^{t^{\prime\prime}}(x)=x$ and, from (\ref{pre941}) we have the
truth of (\ref{pre763}), with $t^{\prime}=t^{\prime\prime\prime}.$

Case $t^{\prime\prime}>t^{\prime\prime\prime}$

Some $\varepsilon>0$ exists with the property that $\forall t\in
(t^{\prime\prime}-\varepsilon,t^{\prime\prime}),x(t)=x(t^{\prime\prime}-0).$
We take $t^{\prime}\in(t^{\prime\prime}-\varepsilon,t^{\prime\prime}%
)\cap(t^{\prime\prime\prime},t^{\prime\prime})$ arbitrary. We notice that%
\[
\sigma^{t^{\prime\prime}}(x)(t)=\left\{
\begin{array}
[c]{c}%
x(t),t\geq t^{\prime},\\
x(t^{\prime\prime}-0),t<t^{\prime\prime},
\end{array}
\right.
\]
$t^{\prime}\in(-\infty,t^{\prime\prime})\subset I^{\sigma^{t^{\prime\prime}%
}(x)}$ hold and let now $t\geq t^{\prime}$ arbitrary with $\sigma
^{t^{\prime\prime}}(x)(t)=\mu.$ Such a choice of $t$ is possible since
$\mathbf{T}_{\mu}^{x}$ is superiorly unbounded. We infer from (\ref{pre941})
and Lemma \ref{Lem30}, page \pageref{Lem30} that%
\begin{equation}
\forall t\geq t^{\prime},x(t)=\mu\Longrightarrow(x(t)=x(t+T)\text{ and
}t-T\geq t^{\prime}\Longrightarrow x(t)=x(t-T)). \label{pre980}%
\end{equation}
As for $t\geq t^{\prime},\sigma^{t^{\prime\prime}}(x)(t)=x(t)$, we can apply
(\ref{pre980}) in order to conclude the truth of (\ref{pre763}).

(\ref{pre763})$\Longrightarrow$(\ref{per47}) We suppose against all reason
that $x$ is not constant, thus $t_{0}\in\mathbf{R}$ exists with%
\begin{equation}
\forall t<t_{0},x(t)=x(-\infty+0), \label{pre766}%
\end{equation}%
\begin{equation}
x(t_{0})\neq x(-\infty+0). \label{pre767}%
\end{equation}
Let $T>0$ arbitrary. Some $\mu\in Or(x)$ exists from the hypothesis
(\ref{pre763}) such that%
\begin{equation}
\left\{
\begin{array}
[c]{c}%
\forall t^{\prime\prime}\in\mathbf{R},\exists t^{\prime}\in I^{\sigma
^{t^{\prime\prime}}(x)}\text{,}\\
\forall t\geq t^{\prime},\sigma^{t^{\prime\prime}}(x)(t)=\mu\Longrightarrow
(\sigma^{t^{\prime\prime}}(x)(t)=\sigma^{t^{\prime\prime}}(x)(t+T)\text{
and}\\
\text{and }t-T\geq t^{\prime}\Longrightarrow\sigma^{t^{\prime\prime}%
}(x)(t)=\sigma^{t^{\prime\prime}}(x)(t-T)).
\end{array}
\right.  \label{pre765}%
\end{equation}
We take in (\ref{pre765}) $t^{\prime\prime}\leq t_{0},$ for which we have
$\sigma^{t^{\prime\prime}}(x)=x$ and from (\ref{pre766}), (\ref{pre767}),
(\ref{pre765}), $t^{\prime}<t_{0}$ exists with $t^{\prime}\in I^{x},$%
\begin{equation}
\forall t\geq t^{\prime},x(t)=\mu\Longrightarrow(x(t)=x(t+T)\text{ and
}t-T\geq t^{\prime}\Longrightarrow x(t)=x(t-T)). \label{pre777}%
\end{equation}
As $\mu\in Or(x)=\{x(t)|t\geq t^{\prime}\},$ some $t^{\prime\prime\prime}\geq
t^{\prime}$ exists with $x(t^{\prime\prime\prime})=\mu$ and, from
(\ref{pre777}), $\mu\in\omega(x).$ In both situations: $t^{\prime\prime\prime
}\in\lbrack t^{\prime},t_{0})$ and $t^{\prime\prime\prime}\geq t_{0},$ we have
the existence of $t_{1},t_{2}\in\mathbf{R}$ with the properties $t^{\prime
}\leq t_{1}<t_{2},$ $[t_{1},t_{2})\subset\mathbf{T}_{\mu}^{x}$ and at least
one of $x(t_{1}-0)\neq\mu,x(t_{2})\neq\mu$ is true. We note from Lemma
\ref{Lem28}, page \pageref{Lem28} that in this situation the inclusion%
\begin{equation}
\lbrack t_{1},t_{2})\cup\lbrack t_{1}+T,t_{2}+T)\cup\lbrack t_{1}%
+2T,t_{2}+2T)\cup...\subset\mathbf{T}_{\mu}^{x} \label{pre778}%
\end{equation}
holds. Moreover, Lemma \ref{Lem25}, page \pageref{Lem25} shows that $\forall
k\in\mathbf{N},$ one of $x(t_{1}+kT-0)\neq\mu,x(t_{2}+kT)\neq\mu$ is also fulfilled.

Let now $T^{\prime}\in(0,t_{2}-t_{1}).$ From the hypothesis (\ref{pre763}),
$\mu^{\prime}\in Or(x)$ exists such that%
\begin{equation}
\left\{
\begin{array}
[c]{c}%
\forall t^{\prime\prime}\in\mathbf{R},\exists t^{\prime}\in I^{\sigma
^{t^{\prime\prime}}(x)}\text{,}\\
\forall t\geq t^{\prime},\sigma^{t^{\prime\prime}}(x)(t)=\mu^{\prime
}\Longrightarrow(\sigma^{t^{\prime\prime}}(x)(t)=\sigma^{t^{\prime\prime}%
}(x)(t+T^{\prime})\text{ and}\\
\text{and }t-T^{\prime}\geq t^{\prime}\Longrightarrow\sigma^{t^{\prime\prime}%
}(x)(t)=\sigma^{t^{\prime\prime}}(x)(t-T^{\prime})).
\end{array}
\right.  \label{pre779}%
\end{equation}
For $t^{\prime\prime}\leq t_{0},$ as $\sigma^{t^{\prime\prime}}(x)=x,$
(\ref{pre766}), (\ref{pre767}), (\ref{pre779}) imply the existence of
$t_{0}^{\prime}<t_{0}$ such that $t_{0}^{\prime}\in I^{x},$%
\begin{equation}
\left\{
\begin{array}
[c]{c}%
\forall t\geq t_{0}^{\prime},x(t)=\mu^{\prime}\Longrightarrow\\
\Longrightarrow(x(t)=x(t+T^{\prime})\text{ and }t-T^{\prime}\geq t_{0}%
^{\prime}\Longrightarrow x(t)=x(t-T^{\prime})).
\end{array}
\right.  \label{pre780}%
\end{equation}
But $\mu^{\prime}\in Or(x)=\{x(t)|t\geq t_{0}^{\prime}\}$ and, like before,
$t_{1}^{\prime},t_{2}^{\prime}\in\mathbf{R}$ exist such that $t_{0}^{\prime
}\leq t_{1}^{\prime}<t_{2}^{\prime},$ $[t_{1}^{\prime},t_{2}^{\prime}%
)\subset\mathbf{T}_{\mu^{\prime}}^{x}$ and%
\begin{equation}
\lbrack t_{1}^{\prime},t_{2}^{\prime})\cup\lbrack t_{1}^{\prime}+T^{\prime
},t_{2}^{\prime}+T^{\prime})\cup\lbrack t_{1}^{\prime}+2T^{\prime}%
,t_{2}^{\prime}+2T^{\prime})\cup...\subset\mathbf{T}_{\mu^{\prime}}^{x}.
\end{equation}
The fact that $T^{\prime}<t_{2}-t_{1}$ implies however from Lemma \ref{Lem29},
page \pageref{Lem29} that%
\[
\varnothing\neq([t_{1},t_{2})\cup\lbrack t_{1}+T,t_{2}+T)\cup\lbrack
t_{1}+2T,t_{2}+2T)\cup...)\cap
\]%
\[
\cap([t_{1}^{\prime},t_{2}^{\prime})\cup\lbrack t_{1}^{\prime}+T^{\prime
},t_{2}^{\prime}+T^{\prime})\cup\lbrack t_{1}^{\prime}+2T^{\prime}%
,t_{2}^{\prime}+2T^{\prime})\cup...)\subset\mathbf{T}_{\mu}^{x}\cap
\mathbf{T}_{\mu^{\prime}}^{x},
\]
thus $\mu=\mu^{\prime}.$ As we have already mentioned, two possibilities exist.

Case $x(t_{1}-0)\neq\mu.$

Let $k\in\mathbf{N}$ with $t_{1}+kT>t_{0}^{\prime}.$ Some $\varepsilon>0$
exists with $\forall\xi\in(t_{1}+kT-\varepsilon,t_{1}+kT),x(\xi)=x(t_{1}%
+kT-0)$ and $t_{1}+kT-\varepsilon\geq t_{0}^{\prime}.$ But then $t\in
(t_{1}+kT-\varepsilon,t_{1}+kT)$ exists such that $t+T^{\prime}\in\lbrack
t_{1}+kT,t_{2}+kT)\footnote{Proving that $max\{t_{1}+kT-\varepsilon
,t_{1}+kT-T^{\prime}\}<\min\{t_{1}+kT,t_{2}+kT-T^{\prime}\}$ is easy and we
take $t\in(max\{t_{1}+kT-\varepsilon,t_{1}+kT-T^{\prime}\}<\min\{t_{1}%
+kT,t_{2}+kT-T^{\prime}\})$ arbitrarily.}$ and we have%
\[
\mu\overset{Lemma\;\ref{Lem25}}{\neq}x(t_{1}+kT-0)=x(t),
\]%
\[
\mu\overset{(\ref{pre778})}{=}x(t+T^{\prime})\overset{(\ref{pre780})\text{
with }\mu=\mu^{\prime}}{=}x(t),
\]
contradiction.

Case $x(t_{2})\neq\mu.$

Let $k\in\mathbf{N}$ such that $t_{1}+kT>t_{0}^{\prime}$ and $t\in\lbrack
t_{1}+kT,t_{2}+kT)$ such that $t+T^{\prime}=t_{2}+kT\footnote{Such a $t$
exists since $t_{1}+kT\leq t_{2}+kT-T^{\prime}<t_{2}+kT$ holds.}. $ We have%
\[
\mu\overset{(\ref{pre778})}{=}x(t)\overset{(\ref{pre780})\text{ with }\mu
=\mu^{\prime}}{=}x(t+T^{\prime})=x(t_{2}+kT)\overset{(\ref{pre777})}{=}%
x(t_{2})\overset{Lemma\;\ref{Lem25}}{\neq}\mu,
\]
contradiction.

We have obtained that $x$ is constant.
\end{proof}

\section{The fourth group of constancy properties}

\begin{remark}
The constancy properties to follow have their origin in the eventual constancy
properties from Theorem \ref{The100}, page \pageref{The100} and they use the
periodicity and the eventual periodicity of the signals. We see that:

- (\ref{per75}) and (\ref{pre508}) (the last contains $\forall k^{\prime}%
\in\mathbf{N}_{\_}$) have their origin in (\ref{per83}%
)$_{page\;\pageref{per83}}$ (containing $\exists k^{\prime}\in\mathbf{N}_{\_}$);

- (\ref{pre509}) (containing $\forall k^{\prime\prime}\in\mathbf{N}$) has its
origin in (\ref{pre515})$_{page\;\pageref{pre515}}$ (containing $\exists
k^{\prime\prime}\in\mathbf{N}$);

- (\ref{per76}) and (\ref{pre512}) (the last contains $\forall t_{1}^{\prime
}\geq t^{\prime}$) have their origin in (\ref{pre557}%
)$_{page\;\pageref{pre557}}$ (containing $\exists t_{1}^{\prime}\geq
t^{\prime}$) and (\ref{pre601})$_{page\;\pageref{pre601}}$ (containing
$\exists t_{1}^{\prime}\in\mathbf{R}$);

- (\ref{pre513}) (containing $\forall t^{\prime\prime}\in\mathbf{R}$) has its
origin in (\ref{pre558})$_{page\;\pageref{pre558}}$ and (\ref{pre602}%
)$_{page\;\pageref{pre602}}$ (containing both $\exists t^{\prime\prime}%
\in\mathbf{R}$).
\end{remark}

\begin{theorem}
\label{The96}a) The following properties are equivalent with the constancy of
the signal $\widehat{x}\in\widehat{S}^{(n)}$:%
\begin{equation}
\forall p\geq1,\forall k\in\mathbf{N}_{\_},\widehat{x}(k)=\widehat{x}(k+p),
\label{per75}%
\end{equation}%
\begin{equation}
\forall p\geq1,\forall k^{\prime}\in\mathbf{N}_{\_},\forall k\geq k^{\prime
},\widehat{x}(k)=\widehat{x}(k+p), \label{pre508}%
\end{equation}%
\begin{equation}
\forall p\geq1,\forall k^{\prime\prime}\in\mathbf{N},\forall k\in
\mathbf{N}_{\_},\widehat{\sigma}^{k^{\prime\prime}}(\widehat{x})(k)=\widehat
{\sigma}^{k^{\prime\prime}}(\widehat{x})(k+p). \label{pre509}%
\end{equation}

b) The following properties are equivalent with the constancy of $x\in
S^{(n)}$:%
\begin{equation}
\forall T>0,\exists t^{\prime}\in I^{x},\forall t\geq t^{\prime},x(t)=x(t+T),
\label{per76}%
\end{equation}%
\begin{equation}
\forall T>0,\exists t^{\prime}\in I^{x},\forall t_{1}^{\prime}\geq t^{\prime
},\forall t\geq t_{1}^{\prime},x(t)=x(t+T), \label{pre512}%
\end{equation}%
\begin{equation}
\forall T>0,\forall t^{\prime\prime}\in\mathbf{R},\exists t^{\prime}\in
I^{\sigma^{t^{\prime\prime}}(x)},\forall t\geq t^{\prime},\sigma
^{t^{\prime\prime}}(x)(t)=\sigma^{t^{\prime\prime}}(x)(t+T). \label{pre513}%
\end{equation}

\end{theorem}

\begin{proof}
a) (\ref{per46})$\Longrightarrow$(\ref{per75}) We suppose that $\mu
\in\mathbf{B}^{n}$ exists with $\forall k\in\mathbf{N}_{\_},\widehat{x}%
(k)=\mu$ and let $p\geq1,k\in\mathbf{N}_{\_}$ arbitrary. We have
\[
\widehat{x}(k)=\mu=\widehat{x}(k+p),
\]
making (\ref{per75}) true.

(\ref{per75})$\Longrightarrow$(\ref{pre508}) Let $p\geq1,k^{\prime}%
\in\mathbf{N}_{\_}$,$k\geq k^{\prime}$ arbitrary$.$ From (\ref{per75}) we
infer that $\widehat{x}(k)=\widehat{x}(k+p).$

(\ref{pre508})$\Longrightarrow$(\ref{pre509}) We take $p\geq1,k^{\prime\prime
}\in\mathbf{N},k\in\mathbf{N}_{\_}$ arbitrarily. We denote $k^{\prime
}=k^{\prime\prime}-1$ and we notice that $k+k^{\prime\prime}=k+k^{\prime
}+1\geq k^{\prime},$ thus we can apply (\ref{pre508}) and we obtain%
\[
\widehat{\sigma}^{k^{\prime\prime}}(\widehat{x})(k)=\widehat{x}(k+k^{\prime
\prime})=\widehat{x}(k+k^{\prime}+1)\overset{(\ref{pre508})}{=}\widehat
{x}(k+k^{\prime}+1+p)=
\]%
\[
=\widehat{x}(k+k^{\prime\prime}+p)=\widehat{\sigma}^{k^{\prime\prime}%
}(\widehat{x})(k+p).
\]

(\ref{pre509})$\Longrightarrow$(\ref{per46}) We write (\ref{pre509}) for $p=1$
and $k^{\prime\prime}=0,$ when $\widehat{\sigma}^{k^{\prime\prime}}%
(\widehat{x})(k)=\widehat{x}(k),$ with $k=-1,0,1,...$ and we get%
\[
\widehat{x}(-1)=\widehat{x}(0)=\widehat{x}(1)=...
\]
We denote with $\mu$ the common value of $\widehat{x}(-1),\widehat
{x}(0),\widehat{x}(1),...$ (\ref{per46}) holds.

b) (\ref{per47})$\Longrightarrow$(\ref{per76}) If $\mu\in\mathbf{B}^{n}$
exists with $\forall t\in\mathbf{R},x(t)=\mu,$ then for arbitrary $T>0$ and
$t^{\prime}\in\mathbf{R},$
\[
\forall t\leq t^{\prime},x(t)=\mu,
\]%
\[
\forall t\geq t^{\prime},x(t)=x(t+T)=\mu
\]
hold. We have that $t^{\prime}\in I^{x}$ and (\ref{per76}) is true.

(\ref{per76})$\Longrightarrow$(\ref{pre512}) Let $T>0$ arbitrary.
(\ref{per76}) shows that $t^{\prime}\in I^{x}$ exists such that%
\begin{equation}
\forall t\geq t^{\prime},x(t)=x(t+T). \label{pre671}%
\end{equation}
We take $t_{1}^{\prime}\geq t^{\prime}$ and $t\geq t_{1}^{\prime}$
arbitrarily. From the fact that $t\geq t^{\prime},$ the statement
(\ref{pre671}) gives $x(t)=x(t+T),$ i.e. (\ref{pre512}) is true.

(\ref{pre512})$\Longrightarrow$(\ref{pre513}) Let $T>0$ arbitrary.
(\ref{pre512}) shows the existence of $t^{\prime\prime\prime}\in I^{x}$ such
that, in the special case when $t_{1}^{\prime}\geq t^{\prime\prime\prime} $
holds as equality$,$ the statement%
\begin{equation}
\forall t\geq t^{\prime\prime\prime},x(t)=x(t+T) \label{pre673}%
\end{equation}
is fulfilled. We suppose that an arbitrary $t^{\prime\prime}\in\mathbf{R}$ is
given and we have the following possibilities.

Case $t^{\prime\prime}\leq t^{\prime\prime\prime}$

From $\forall t\leq t^{\prime\prime\prime},x(t)=x(-\infty+0)$ we infer that
$\sigma^{t^{\prime\prime}}(x)=x$ and, taking into account (\ref{pre673})
also$,$ we get that (\ref{pre513}) is true with $t^{\prime}=t^{\prime
\prime\prime}.$

Case $t^{\prime\prime}>t^{\prime\prime\prime}$

Some $\varepsilon>0$ exists with the property that $\forall t\in
(t^{\prime\prime}-\varepsilon,t^{\prime\prime}),x(t)=x(t^{\prime\prime}-0).$
We take arbitrarily a $t^{\prime}\in(t^{\prime\prime}-\varepsilon
,t^{\prime\prime})\cap(t^{\prime\prime\prime},t^{\prime\prime})$ and we get
that $\forall t\leq t^{\prime},\sigma^{t^{\prime\prime}}(x)(t)=x(t^{\prime
\prime}-0)$ is true. We notice that for any $t\geq t^{\prime}$ we have
$\sigma^{t^{\prime\prime}}(x)(t)=x(t),$ irrespective of the fact that
$t<t^{\prime\prime}$ or $t\geq t^{\prime\prime}$ and let us fix an arbitrary
$t\geq t^{\prime}.$ We have%
\[
\sigma^{t^{\prime\prime}}(x)(t)=x(t)\overset{(\ref{pre673})}{=}x(t+T)=\sigma
^{t^{\prime\prime}}(x)(t+T).
\]

(\ref{pre513})$\Longrightarrow$(\ref{per47}) Let us suppose against all reason
that (\ref{per47}) is false, meaning that $t_{0}<t_{1}$ exist with the
property%
\begin{equation}
\forall t<t_{0},x(t)=x(-\infty+0), \label{pre675}%
\end{equation}%
\begin{equation}
\forall t\in\lbrack t_{0},t_{1}),x(t)\neq x(-\infty+0). \label{pre676}%
\end{equation}
We write (\ref{pre513}) for $T\in(0,t_{1}-t_{0})$ and $t^{\prime\prime}$
sufficiently small such that $\sigma^{t^{\prime\prime}}(x)=x$ and we obtain
the existence of $t^{\prime}\in I^{x}$ with%
\begin{equation}
\forall t\geq t^{\prime},x(t)=x(t+T). \label{pre982}%
\end{equation}
From (\ref{pre675}), (\ref{pre676}) we infer $t^{\prime}<t_{0}.$

Let now $t\in\lbrack t^{\prime},t_{0})\cap\lbrack t_{0}-T,t_{0})$ fixed. We
have $t+T\in\lbrack t_{0},t_{0}+T)\subset\lbrack t_{0},t_{1}),$ thus%
\[
x(-\infty+0)=x(t)\overset{(\ref{pre982})}{=}x(t+T)\overset{(\ref{pre676}%
)}{\neq}x(-\infty+0),
\]
contradiction. We conclude that (\ref{per47}) holds.
\end{proof}

\section{Discrete time vs real time}

\begin{theorem}
\label{The112}Let us suppose that the sequence $(t_{k})\in Seq$ exists such
that%
\begin{equation}
x(t)=\widehat{x}(-1)\cdot\chi_{(-\infty,t_{0})}(t)\oplus\widehat{x}%
(0)\cdot\chi_{\lbrack t_{0},t_{1})}(t)\oplus...\oplus\widehat{x}(k)\cdot
\chi_{\lbrack t_{k},t_{k+1})}(t)\oplus...
\end{equation}
Then the constancy of $\widehat{x}$ is equivalent with the constancy of $x$.
\end{theorem}

\begin{proof}
Obvious, but let us take a look at Theorem \ref{The114} a) also, page
\pageref{The114}, stating that the hypothesis implies $\widehat{Or}%
(\widehat{x})=Or(x).$ We infer then the equivalence between (\ref{per77}%
)$_{page\;\pageref{per77}}$ and (\ref{per78})$_{page\;\pageref{per78}}$:%
\[
\exists\mu\in\mathbf{B}^{n},\widehat{Or}(\widehat{x})=\{\mu
\}\Longleftrightarrow\exists\mu\in\mathbf{B}^{n},Or(x)=\{\mu\}.
\]

\end{proof}

\section{Discussion}

\begin{remark}
The statements from Theorems \ref{The13}, \ref{The95}, \ref{The94},
\ref{The96} are present in discrete time - real time couples: (\ref{per46}%
)-(\ref{per47}), (\ref{per129})-(\ref{per126}),... This continues the previous
style of organizing the exposure, corresponding to our intuition that strong
analogies work between the discrete time and the real time properties of the
signals. Theorem \ref{The112} gives the relation between the discrete time and
the real time constancy of the signals.
\end{remark}

\begin{remark}
A common point, of intersection of the three groups 2,3,4 of properties of
periodicity exists, in the sense that the periodicity of a signal is present
i.e. all the points of its orbit are periodic, with the same period.
\end{remark}

\begin{remark}
\label{Rem2}The key request in all these periodicity properties in order to be
equivalent with constancy is that they hold for any period $p\geq1,T>0.$
\end{remark}

\begin{remark}
\label{Rem15}a) In (\ref{per46}),..., (\ref{per77}) the existence of a unique
$\mu\in\mathbf{B}^{n}$ is stated. Similarly, in (\ref{per47}),...,
(\ref{per78}) $\exists\mu\in\mathbf{B}^{n}$ must be understood as $\exists
!\mu\in\mathbf{B}^{n}.$
\end{remark}

\begin{remark}
The statement%
\[
\forall k\in\widehat{\mathbf{T}}_{\mu}^{\widehat{x}},\{k+zp|z\in
\mathbf{Z}\}\cap\mathbf{N}_{\_}\subset\widehat{\mathbf{T}}_{\mu}^{\widehat{x}}%
\]
from (\ref{per185}) is one of periodicity of $\mu$ with the period $p$ and the
statement%
\[
\exists t^{\prime}\in I^{x},\forall t\in\mathbf{T}_{\mu}^{x}\cap\lbrack
t^{\prime},\infty),\{t+zT|z\in\mathbf{Z}\}\cap\lbrack t^{\prime}%
,\infty)\subset\mathbf{T}_{\mu}^{x}%
\]
from (\ref{per186}) is one of periodicity of $\mu$ with the period $T.$ Both
these requirements are related with an initial time=limit of periodicity,
which is $-1$ and $t^{\prime}.$ Their demand is that if $\widehat{x}%
(k)=\mu,x(t)=\mu,$ then right translations along the time axis are allowed
giving the same value $\mu$ of $\widehat{x},x:$ $\widehat{x}(k+zp)=\mu
,z\geq0,x(t+zT)=\mu,z\geq0$ and left translations along the time axis are also
allowed as long as the argument still exceeds the limit of periodicity- and
they give the same value $\mu$ of $\widehat{x},x:$ $\widehat{x}(k+zp)=\mu
,z<0,x(t+zT)=\mu,z<0.$ And this should happen for all the periods $p\geq1,T>0$
and all the points of the orbit $\mu\in\widehat{Or}(\widehat{x}),\mu\in
Or(x).$
\end{remark}

\begin{remark}
The properties%
\[
\left\{
\begin{array}
[c]{c}%
\forall k^{\prime}\in\mathbf{N}_{\_},\forall k\in\widehat{\mathbf{T}}_{\mu
}^{\widehat{x}}\cap\{k^{\prime},k^{\prime}+1,k^{\prime}+2,...\},\\
\{k+zp|z\in\mathbf{Z}\}\cap\{k^{\prime},k^{\prime}+1,k^{\prime}+2,...\}\subset
\widehat{\mathbf{T}}_{\mu}^{\widehat{x}},
\end{array}
\right.
\]%
\[
\exists t^{\prime}\in I^{x},\forall t_{1}^{\prime}\geq t^{\prime},\forall
t\in\mathbf{T}_{\mu}^{x}\cap\lbrack t_{1}^{\prime},\infty),\{t+zT|z\in
\mathbf{Z}\}\cap\lbrack t_{1}^{\prime},\infty)\subset\mathbf{T}_{\mu}^{x}%
\]
from (\ref{pre506}), (\ref{pre548}) are of eventual periodicity of $\mu$ with
the period $p,T.$ Here the periodicity of $\mu$ starts not from the very
beginning $-1,t^{\prime}$ like previously, but from a time instant $k^{\prime
}\in\mathbf{N}_{\_},t_{1}^{\prime}\geq t^{\prime}.$ In order to have
periodicity, we ask that such properties hold for any $k^{\prime}%
,t_{1}^{\prime}.$ And in order to rediscover constancy, they should hold for
any $p,T$ and any $\mu\in\widehat{Or}(\widehat{x}),\mu\in Or(x).$
\end{remark}

\begin{remark}
\label{Rem26}Let us fix in (\ref{pre507}) $p\geq1,\mu\in\widehat{Or}%
(\widehat{x})$ and $k^{\prime\prime}\in\mathbf{N}.$ In general we have
$\widehat{Or}(\widehat{\sigma}^{k^{\prime\prime}}(\widehat{x}))\subset
\widehat{Or}(\widehat{x})$ (we have shown this at Theorem \ref{The12} c), page
\pageref{The12}) and the points $\mu\in\widehat{Or}(\widehat{x})\smallsetminus
\widehat{Or}(\widehat{\sigma}^{k^{\prime\prime}}(\widehat{x}))$ may satisfy
$\widehat{\mathbf{T}}_{\mu}^{\widehat{\sigma}^{k^{\prime\prime}}(\widehat{x}%
)}=\varnothing,$ when%
\[
\forall k\in\widehat{\mathbf{T}}_{\mu}^{\widehat{\sigma}^{k^{\prime\prime}%
}(\widehat{x})},\{k+zp|z\in\mathbf{Z}\}\cap\mathbf{N}_{\_}\subset
\widehat{\mathbf{T}}_{\mu}^{\widehat{\sigma}^{k^{\prime\prime}}(\widehat{x})}%
\]
is trivially fulfilled. This is not the case if the previous property takes
place for any $k^{\prime\prime}\in\mathbf{N},$ including the case
$k^{\prime\prime}=0,$ when $\widehat{Or}(\widehat{\sigma}^{k^{\prime\prime}%
}(\widehat{x}))=\widehat{Or}(\widehat{x}).$ This discussion is in principle
the same for (\ref{pre549}).
\end{remark}

\begin{remark}
The requests%
\[
\forall k\in\widehat{\mathbf{T}}_{\mu}^{\widehat{\sigma}^{k^{\prime\prime}%
}(\widehat{x})},\{k+zp|z\in\mathbf{Z}\}\cap\mathbf{N}_{\_}\subset
\widehat{\mathbf{T}}_{\mu}^{\widehat{\sigma}^{k^{\prime\prime}}(\widehat{x}%
)},
\]%
\[
\left\{
\begin{array}
[c]{c}%
\forall k\in\mathbf{N}_{\_},\widehat{\sigma}^{k^{\prime\prime}}(\widehat
{x})(k)=\mu\Longrightarrow(\widehat{\sigma}^{k^{\prime\prime}}(\widehat
{x})(k)=\widehat{\sigma}^{k^{\prime\prime}}(\widehat{x})(k+p)\text{ and }\\
\text{and }k-p\geq-1\Longrightarrow\widehat{\sigma}^{k^{\prime\prime}%
}(\widehat{x})(k)=\widehat{\sigma}^{k^{\prime\prime}}(\widehat{x})(k-p))
\end{array}
\right.
\]
derived from (\ref{pre507}), (\ref{pre537}) and the requests%
\[
\exists t^{\prime}\in I^{\sigma^{t^{\prime\prime}}(x)},\forall t\in
\mathbf{T}_{\mu}^{\sigma^{t^{\prime\prime}}(x)}\cap\lbrack t^{\prime}%
,\infty),\{t+zT|z\in\mathbf{Z}\}\cap\lbrack t^{\prime},\infty)\subset
\mathbf{T}_{\mu}^{\sigma^{t^{\prime\prime}}(x)},
\]%
\[
\left\{
\begin{array}
[c]{c}%
\exists t^{\prime}\in I^{\sigma^{t^{\prime\prime}}(x)},\forall t\geq
t^{\prime},\sigma^{t^{\prime\prime}}(x)(t)=\mu\Longrightarrow(\sigma
^{t^{\prime\prime}}(x)(t)=\sigma^{t^{\prime\prime}}(x)(t+T)\text{ and}\\
\text{and }t-T\geq t^{\prime}\Longrightarrow\sigma^{t^{\prime\prime}%
}(x)(t)=\sigma^{t^{\prime\prime}}(x)(t-T))
\end{array}
\right.
\]
derived from (\ref{pre549}), (\ref{pre552}) are of eventual periodicity of
$\mu$ with the period $p,T.$ In such requests, the fact that periodicity might
not start from the very beginning is indicated by working with the signals
$\widehat{\sigma}^{k^{\prime\prime}}(\widehat{x}),\sigma^{t^{\prime\prime}%
}(x)$ that have forgotten their first values. Note that all these properties
are of periodicity of $\mu$ -related to $\widehat{\sigma}^{k^{\prime\prime}%
}(\widehat{x}),\sigma^{t^{\prime\prime}}(x)$ instead of $\widehat{x},x$- and
that they must hold, for constancy, for all $p,T,$ all $\mu\in\widehat
{Or}(\widehat{x}),\mu\in Or(x)$ and all $k^{\prime\prime}\in\mathbf{N}%
,t^{\prime\prime}\in\mathbf{R}.$
\end{remark}

\begin{remark}
The statements%
\[
\left\{
\begin{array}
[c]{c}%
\forall k\in\mathbf{N}_{\_},\widehat{x}(k)=\mu\Longrightarrow\\
\Longrightarrow(\widehat{x}(k)=\widehat{x}(k+p)\text{ and }k-p\geq
-1\Longrightarrow\widehat{x}(k)=\widehat{x}(k-p)),
\end{array}
\right.
\]%
\[
\exists t^{\prime}\in I^{x},\forall t\geq t^{\prime},x(t)=\mu\Longrightarrow
(x(t)=x(t+T)\text{ and }t-T\geq t^{\prime}\Longrightarrow x(t)=x(t-T))
\]
from (\ref{pre522}), (\ref{pre550}) refer also to the periodicity of $\mu$
with the period $p,T.$ The difference from the previous property consists in
the fact that the translations along the time axis are with one period only,
and the general case is rediscovered by iterating these translations. We must
have periodicity with any period $p,T$, of all the points $\mu\in\widehat
{Or}(\widehat{x}),\mu\in Or(x)$ for constancy.
\end{remark}

\begin{remark}
The case of (\ref{pre523}) is similar with that of (\ref{pre507}) (see Remark
\ref{Rem26}). Points $\mu\in\widehat{Or}(\widehat{x})$ might exist for which
$\widehat{x}(k)=\mu$ is false if $k\geq k^{\prime}$ and then
\begin{equation}
\left\{
\begin{array}
[c]{c}%
\forall k\geq k^{\prime},\widehat{x}(k)=\mu\Longrightarrow\\
\Longrightarrow(\widehat{x}(k)=\widehat{x}(k+p)\text{ and }k-p\geq k^{\prime
}\Longrightarrow\widehat{x}(k)=\widehat{x}(k-p))
\end{array}
\right.  \label{p124}%
\end{equation}
is trivially true. This is not the case, because the truth of (\ref{p124})
includes the value $k^{\prime}=-1.$ The remark holds also for%
\begin{equation}
\left\{
\begin{array}
[c]{c}%
\forall t\geq t_{1}^{\prime},x(t)=\mu\Longrightarrow\\
\Longrightarrow(x(t)=x(t+T)\text{ and }t-T\geq t_{1}^{\prime}\Longrightarrow
x(t)=x(t-T))
\end{array}
\right.  \label{p125}%
\end{equation}
and (\ref{pre551}).
\end{remark}

\begin{remark}
The requests (\ref{p124}), (\ref{p125}) derived from (\ref{pre523}),
(\ref{pre551}) are also of eventual periodicity of $\mu,$ i.e. periodicity
starting at $k^{\prime}\in\mathbf{N}_{\_}$ and $t_{1}^{\prime}\geq t^{\prime
}\in I^{x}.$ The difference with the previous situation is given by the
translations along the time axis with one period. The requests must be
fulfilled for all $k^{\prime},t_{1}^{\prime}\geq t^{\prime},$ all $p,T$ and
all $\mu.$
\end{remark}

\begin{remark}
In (\ref{pre537}), (\ref{pre552}) we have%
\[
\left\{
\begin{array}
[c]{c}%
\forall k\in\mathbf{N}_{\_},\widehat{\sigma}^{k^{\prime\prime}}(\widehat
{x})(k)=\mu\Longrightarrow(\widehat{\sigma}^{k^{\prime\prime}}(\widehat
{x})(k)=\widehat{\sigma}^{k^{\prime\prime}}(\widehat{x})(k+p)\text{ and }\\
\text{and }k-p\geq-1\Longrightarrow\widehat{\sigma}^{k^{\prime\prime}%
}(\widehat{x})(k)=\widehat{\sigma}^{k^{\prime\prime}}(\widehat{x})(k-p)),
\end{array}
\right.
\]%
\[
\left\{
\begin{array}
[c]{c}%
\exists t^{\prime}\in I^{\sigma^{t^{\prime\prime}}(x)},\forall t\geq
t^{\prime},\sigma^{t^{\prime\prime}}(x)(t)=\mu\Longrightarrow(\sigma
^{t^{\prime\prime}}(x)(t)=\sigma^{t^{\prime\prime}}(x)(t+T)\text{ and}\\
\text{and }t-T\geq t^{\prime}\Longrightarrow\sigma^{t^{\prime\prime}%
}(x)(t)=\sigma^{t^{\prime\prime}}(x)(t-T))
\end{array}
\right.
\]
i.e. the periodicity of $\widehat{x},x$ after having forgotten some first values.
\end{remark}

\begin{remark}
The third group of constancy properties repeats the statements of the second
group, by replacing $\forall\mu\in\widehat{Or}(\widehat{x}),\forall\mu\in
Or(x)$ with $\exists\mu\in\widehat{Or}(\widehat{x}),\exists\mu\in Or(x).$ This
is possible since constancy means that the orbits have exactly one point,
$\widehat{Or}(\widehat{x})=\{\mu\},Or(x)=\{\mu\}.$ The proofs of the
implications are in general similar with those of the second group.
\end{remark}

\begin{remark}
The properties%
\[
\forall k\in\mathbf{N}_{\_},\widehat{x}(k)=\widehat{x}(k+p),
\]%
\[
\exists t^{\prime}\in I^{x},\forall t\geq t^{\prime},x(t)=x(t+T)
\]
from (\ref{per75}), (\ref{per76}) are of periodicity of the signals
$\widehat{x},x$ with the period $p,T.$ In order to get constancy, these
properties must hold for any period $p,T.$
\end{remark}

\begin{remark}
The next properties%
\[
\forall k\geq k^{\prime},\widehat{x}(k)=\widehat{x}(k+p),
\]%
\[
\exists t^{\prime}\in I^{x},\forall t\geq t_{1}^{\prime},x(t)=x(t+T)
\]
that occur in (\ref{pre508}), (\ref{pre512}) are of eventual periodicity of
the signals $\widehat{x},x$ with the periods $p,T.$ It is asked that
periodicity starts at any limit of periodicity $k^{\prime},t_{1}^{\prime}\geq
t^{\prime}$ (we have periodicity so far) and that it holds for any period
$p,T$ for constancy.
\end{remark}

\begin{remark}
The properties%
\[
\forall k\in\mathbf{N}_{\_},\widehat{\sigma}^{k^{\prime\prime}}(\widehat
{x})(k)=\widehat{\sigma}^{k^{\prime\prime}}(\widehat{x})(k+p),
\]%
\[
\exists t^{\prime}\in I^{\sigma^{t^{\prime\prime}}(x)},\forall t\geq
t^{\prime},\sigma^{t^{\prime\prime}}(x)(t)=\sigma^{t^{\prime\prime}}(x)(t+T)
\]
from (\ref{pre509}), (\ref{pre513}) are of eventual periodicity of
$\widehat{x},x.$ The properties are asked to hold for any $k^{\prime\prime}%
\in\mathbf{N},t^{\prime\prime}\in\mathbf{R}$ for periodicity and any
$p\geq1,T>0$ for constancy.
\end{remark}

\chapter{\label{Cha6}Eventually periodic points}

We give first some statements that are equivalent with the eventual
periodicity of the points and a discussion on their properties.

Section 3 shows that an eventually periodic point is accessed for time
instants greater than the limit of periodicity at least once in a time
interval with the length of a period. This fundamental result will be used
frequently later.

The bound of the limit of periodicity and the independence of eventual
periodicity on the choice of the limit of periodicity are treated in Section 4.

The property of eventual constancy that follows in Section 5 is used in
Section 6 to establish the relation between the discrete time and the
continuous time eventual periodicity of the points.

Section 7 highlights the relation between the support sets $\widehat
{\mathbf{T}}_{\mu}^{\widehat{x}},\mathbf{T}_{\mu}^{x}$ and the sets of the
periods $\widehat{P}_{\mu}^{\widehat{x}},P_{\mu}^{x}$.

The fact that the sum, the difference and the multiples of the periods are
periods is treated in Section 8.

In Section 9 we show which is the form of $\widehat{P}_{\mu}^{\widehat{x}%
},P_{\mu}^{x}$ and in particular we address the issue of the existence of the
prime period.

Sections 10 and 11 give necessary and sufficient conditions of eventual
periodicity and a special case of eventually periodic point is treated in
Section 12, where the prime period is known.

Section 13 gives a result relating the eventually periodic points with the
eventually constant signals.

\section{Equivalent properties with the eventual periodicity of a point}

\begin{remark}
The properties that are equivalent with the eventual periodicity of the points
were already used in Chapter 3 dedicated to the eventually constant signals at
Theorem \ref{The97}, page \pageref{The97} (see also Theorem \ref{The98}, page
\pageref{The98}, and Theorem \ref{The99}, page \pageref{The99}). To be
compared (\ref{pre154}),...,(\ref{pre159}) with (\ref{pre781}%
)$_{page\;\pageref{pre781}}$,...,(\ref{pre538})$_{page\;\pageref{pre538}}$ and
(\ref{pre572}),...,(\ref{pre606}) with (\ref{pre553})$_{page\;\pageref{pre553}%
}$,...,(\ref{pre600})$_{page\;\pageref{pre600}}.$ We make also the
associations (\ref{pre153})-(\ref{per83})$_{page\;\pageref{per83}}$,
(\ref{pre159})-(\ref{pre515})$_{page\;\pageref{pre515}}$ and (\ref{pre574}%
)-(\ref{pre557})$_{page\;\pageref{pre557}}$, (\ref{pre605})-(\ref{pre601}%
)$_{page\;\pageref{pre601}}$, (\ref{pre575})-(\ref{pre558}%
)$_{page\;\pageref{pre558}}$, (\ref{pre606})-(\ref{pre602}%
)$_{page\;\pageref{pre602}}$ with the properties of eventual constancy of the
signals from group four, see Theorem \ref{The100}, page \pageref{The100}.
\end{remark}

\begin{theorem}
\label{The67}We consider the signals $\widehat{x}\in\widehat{S}^{(n)},x\in
S^{(n)}$.

a) The following statements are equivalent for any $p\geq1$ and any $\mu
\in\widehat{\omega}(\widehat{x}):$%
\begin{equation}
\left\{
\begin{array}
[c]{c}%
\exists k^{\prime}\in\mathbf{N}_{\_},\forall k\in\widehat{\mathbf{T}}_{\mu
}^{\widehat{x}}\cap\{k^{\prime},k^{\prime}+1,k^{\prime}+2,...\},\\
\{k+zp|z\in\mathbf{Z}\}\cap\{k^{\prime},k^{\prime}+1,k^{\prime}+2,...\}\subset
\widehat{\mathbf{T}}_{\mu}^{\widehat{x}},
\end{array}
\right.  \label{pre154}%
\end{equation}%
\begin{equation}
\exists k^{\prime\prime}\in\mathbf{N},\forall k\in\widehat{\mathbf{T}}_{\mu
}^{\widehat{\sigma}^{k^{\prime\prime}}(\widehat{x})},\{k+zp|z\in
\mathbf{Z}\}\cap\mathbf{N}_{\_}\subset\widehat{\mathbf{T}}_{\mu}%
^{\widehat{\sigma}^{k^{\prime\prime}}(\widehat{x})}, \label{pre158}%
\end{equation}%
\begin{equation}
\left\{
\begin{array}
[c]{c}%
\exists k^{\prime}\in\mathbf{N}_{\_},\forall k\geq k^{\prime},\widehat
{x}(k)=\mu\Longrightarrow\\
\Longrightarrow(\widehat{x}(k)=\widehat{x}(k+p)\text{ and }k-p\geq k^{\prime
}\Longrightarrow\widehat{x}(k)=\widehat{x}(k-p)),
\end{array}
\right.  \label{pre153}%
\end{equation}%
\begin{equation}
\left\{
\begin{array}
[c]{c}%
\exists k^{\prime\prime}\in\mathbf{N},\forall k\in\mathbf{N}_{\_}%
,\widehat{\sigma}^{k^{\prime\prime}}(\widehat{x})(k)=\mu\Longrightarrow\\
\Longrightarrow(\widehat{\sigma}^{k^{\prime\prime}}(\widehat{x})(k)=\widehat
{\sigma}^{k^{\prime\prime}}(\widehat{x})(k+p)\text{ and }\\
\text{and }k-p\geq-1\Longrightarrow\widehat{\sigma}^{k^{\prime\prime}%
}(\widehat{x})(k)=\widehat{\sigma}^{k^{\prime\prime}}(\widehat{x})(k-p)).
\end{array}
\right.  \label{pre159}%
\end{equation}

b) The following statements are also equivalent for any $T>0$ and $\mu
\in\omega(x)$:%
\begin{equation}
\exists t^{\prime}\in I^{x},\exists t_{1}^{\prime}\geq t^{\prime},\forall
t\in\mathbf{T}_{\mu}^{x}\cap\lbrack t_{1}^{\prime},\infty),\{t+zT|z\in
\mathbf{Z}\}\cap\lbrack t_{1}^{\prime},\infty)\subset\mathbf{T}_{\mu}^{x},
\label{pre572}%
\end{equation}%
\begin{equation}
\exists t_{1}^{\prime}\in\mathbf{R},\forall t\in\mathbf{T}_{\mu}^{x}%
\cap\lbrack t_{1}^{\prime},\infty),\{t+zT|z\in\mathbf{Z}\}\cap\lbrack
t_{1}^{\prime},\infty)\subset\mathbf{T}_{\mu}^{x}, \label{pre603}%
\end{equation}%
\begin{equation}
\left\{
\begin{array}
[c]{c}%
\exists t^{\prime\prime}\in\mathbf{R},\exists t^{\prime}\in I^{\sigma
^{t^{\prime\prime}}(x)},\\
\forall t\in\mathbf{T}_{\mu}^{\sigma^{t^{\prime\prime}}(x)}\cap\lbrack
t^{\prime},\infty),\{t+zT|z\in\mathbf{Z}\}\cap\lbrack t^{\prime}%
,\infty)\subset\mathbf{T}_{\mu}^{\sigma^{t^{\prime\prime}}(x)},
\end{array}
\right.  \label{pre573}%
\end{equation}%
\begin{equation}
\exists t^{\prime\prime}\in\mathbf{R},\exists t^{\prime}\in\mathbf{R},\forall
t\in\mathbf{T}_{\mu}^{\sigma^{t^{\prime\prime}}(x)}\cap\lbrack t^{\prime
},\infty),\{t+zT|z\in\mathbf{Z}\}\cap\lbrack t^{\prime},\infty)\subset
\mathbf{T}_{\mu}^{\sigma^{t^{\prime\prime}}(x)}, \label{pre604}%
\end{equation}%
\begin{equation}
\left\{
\begin{array}
[c]{c}%
\exists t^{\prime}\in I^{x},\exists t_{1}^{\prime}\geq t^{\prime},\forall
t\geq t_{1}^{\prime},x(t)=\mu\Longrightarrow\\
\Longrightarrow(x(t)=x(t+T)\text{ and }t-T\geq t_{1}^{\prime}\Longrightarrow
x(t)=x(t-T)),
\end{array}
\right.  \label{pre574}%
\end{equation}%
\begin{equation}
\left\{
\begin{array}
[c]{c}%
\exists t_{1}^{\prime}\in\mathbf{R},\forall t\geq t_{1}^{\prime}%
,x(t)=\mu\Longrightarrow\\
\Longrightarrow(x(t)=x(t+T)\text{ and }t-T\geq t_{1}^{\prime}\Longrightarrow
x(t)=x(t-T)),
\end{array}
\right.  \label{pre605}%
\end{equation}%
\begin{equation}
\left\{
\begin{array}
[c]{c}%
\exists t^{\prime\prime}\in\mathbf{R},\exists t^{\prime}\in I^{\sigma
^{t^{\prime\prime}}(x)},\\
\forall t\geq t^{\prime},\sigma^{t^{\prime\prime}}(x)(t)=\mu\Longrightarrow
(\sigma^{t^{\prime\prime}}(x)(t)=\sigma^{t^{\prime\prime}}(x)(t+T)\text{
and}\\
\text{and }t-T\geq t^{\prime}\Longrightarrow\sigma^{t^{\prime\prime}%
}(x)(t)=\sigma^{t^{\prime\prime}}(x)(t-T)),
\end{array}
\right.  \label{pre575}%
\end{equation}%
\begin{equation}
\left\{
\begin{array}
[c]{c}%
\exists t^{\prime\prime}\in\mathbf{R},\exists t^{\prime}\in\mathbf{R},\forall
t\geq t^{\prime},\sigma^{t^{\prime\prime}}(x)(t)=\mu\Longrightarrow\\
\Longrightarrow(\sigma^{t^{\prime\prime}}(x)(t)=\sigma^{t^{\prime\prime}%
}(x)(t+T)\text{ and}\\
\text{and }t-T\geq t^{\prime}\Longrightarrow\sigma^{t^{\prime\prime}%
}(x)(t)=\sigma^{t^{\prime\prime}}(x)(t-T)).
\end{array}
\right.  \label{pre606}%
\end{equation}

\end{theorem}

\begin{proof}
a) The proof of the implications%
\[
(\ref{pre154})\Longrightarrow(\ref{pre158})\Longrightarrow(\ref{pre153}%
)\Longrightarrow(\ref{pre159})
\]
follows from the proof of Theorem \ref{The97}, page \pageref{The97}.

(\ref{pre159})$\Longrightarrow$(\ref{pre154}) From (\ref{pre159}),
$k^{\prime\prime}\in\mathbf{N}$ exists making%
\begin{equation}
\left\{
\begin{array}
[c]{c}%
\forall k\in\mathbf{N}_{\_},\widehat{\sigma}^{k^{\prime\prime}}(\widehat
{x})(k)=\mu\Longrightarrow(\widehat{\sigma}^{k^{\prime\prime}}(\widehat
{x})(k)=\widehat{\sigma}^{k^{\prime\prime}}(\widehat{x})(k+p)\text{ and }\\
\text{and }k-p\geq-1\Longrightarrow\widehat{\sigma}^{k^{\prime\prime}%
}(\widehat{x})(k)=\widehat{\sigma}^{k^{\prime\prime}}(\widehat{x})(k-p))
\end{array}
\right.  \label{pre887}%
\end{equation}
true. We define $k^{\prime}=k^{\prime\prime}-1.$ The fact that $\mu\in
\widehat{\omega}(\widehat{x})$ implies that $\widehat{\mathbf{T}}_{\mu
}^{\widehat{x}}$ is infinite, thus $\widehat{\mathbf{T}}_{\mu}^{\widehat{x}%
}\cap\{k^{\prime},k^{\prime}+1,k^{\prime}+2,...\}\neq\varnothing.$ Let
$k\in\widehat{\mathbf{T}}_{\mu}^{\widehat{x}}\cap\{k^{\prime},k^{\prime
}+1,k^{\prime}+2,...\},z\in\mathbf{Z}$ arbitrary such that $k+zp\geq
k^{\prime}.$ The number $k-k^{\prime\prime}$ satisfies $k-k^{\prime\prime
}=k-k^{\prime}-1\geq-1,\widehat{x}(k-k^{\prime\prime}+k^{\prime\prime
})=\widehat{\sigma}^{k^{\prime\prime}}(\widehat{x})(k-k^{\prime\prime})$ and
the number $k-k^{\prime\prime}+zp$ satisfies $k-k^{\prime\prime}%
+zp=k-k^{\prime}-1+zp\geq-1,$ thus we can apply (\ref{pre887}). We have the
following possibilities:

Case $z>0,$%
\[
\mu=\widehat{x}(k)=\widehat{\sigma}^{k^{\prime\prime}}(\widehat{x}%
)(k-k^{\prime\prime})\overset{(\ref{pre887})}{=}\widehat{\sigma}%
^{k^{\prime\prime}}(\widehat{x})(k-k^{\prime\prime}+p)\overset{(\ref{pre887}%
)}{=}%
\]%
\[
\overset{(\ref{pre887})}{=}\widehat{\sigma}^{k^{\prime\prime}}(\widehat
{x})(k-k^{\prime\prime}+2p)\overset{(\ref{pre887})}{=}...\overset
{(\ref{pre887})}{=}\widehat{\sigma}^{k^{\prime\prime}}(\widehat{x}%
)(k-k^{\prime\prime}+zp)=\widehat{x}(k+zp);
\]

Case $z=0,$%
\[
\mu=\widehat{x}(k)=\widehat{x}(k+zp);
\]

Case $z<0,$%
\[
\mu=\widehat{x}(k)=\widehat{\sigma}^{k^{\prime\prime}}(\widehat{x}%
)(k-k^{\prime\prime})\overset{(\ref{pre887})}{=}\widehat{\sigma}%
^{k^{\prime\prime}}(\widehat{x})(k-k^{\prime\prime}-p)\overset{(\ref{pre887}%
)}{=}%
\]%
\[
\overset{(\ref{pre887})}{=}\widehat{\sigma}^{k^{\prime\prime}}(\widehat
{x})(k-k^{\prime\prime}-2p)\overset{(\ref{pre887})}{=}...\overset
{(\ref{pre887})}{=}\widehat{\sigma}^{k^{\prime\prime}}(\widehat{x}%
)(k-k^{\prime\prime}+zp)=\widehat{x}(k+zp).
\]
It has resulted that, in all the three situations, $k+zp\in\widehat
{\mathbf{T}}_{\mu}^{\widehat{x}},$ thus (\ref{pre154}) holds.

b) The proof of the implications%
\[
(\ref{pre572})\Longrightarrow(\ref{pre603})\Longrightarrow(\ref{pre573}%
)\Longrightarrow(\ref{pre604})\Longrightarrow(\ref{pre574})\Longrightarrow
(\ref{pre605})\Longrightarrow(\ref{pre575})\Longrightarrow(\ref{pre606})
\]
follows from the proof of Theorem \ref{The97}, page \pageref{The97}.

(\ref{pre606})$\Longrightarrow$(\ref{pre572}) From (\ref{pre606}) we get the
existence of $t^{\prime\prime}\in\mathbf{R}$ and $t^{\prime}\in\mathbf{R}$
with%
\begin{equation}
\left\{
\begin{array}
[c]{c}%
\forall t\geq t^{\prime},\sigma^{t^{\prime\prime}}(x)(t)=\mu\Longrightarrow
(\sigma^{t^{\prime\prime}}(x)(t)=\sigma^{t^{\prime\prime}}(x)(t+T)\text{
and}\\
\text{and }t-T\geq t^{\prime}\Longrightarrow\sigma^{t^{\prime\prime}%
}(x)(t)=\sigma^{t^{\prime\prime}}(x)(t-T))
\end{array}
\right.  \label{pre896}%
\end{equation}
and on the other hand we take arbitrarily some $t^{\prime\prime\prime}\in
I^{x}\footnote{From this moment we prove the truth of a statement which is
stronger than (\ref{pre572}).}.$ Let $t_1^\prime\geq\max\{t^\prime
,t^\prime\prime,t^\prime\prime\prime\}$ arbitrary also. We have%
\begin{equation}
\forall t\geq t_{1}^{\prime},\sigma^{t^{\prime\prime}}(x)(t)=x(t)
\label{pre898}%
\end{equation}
and, taking into account (\ref{pre896}), (\ref{pre898}) and Lemma \ref{Lem30},
page \pageref{Lem30} we infer%
\begin{equation}
\forall t\geq t_{1}^{\prime},x(t)=\mu\Longrightarrow(x(t)=x(t+T)\text{ and
}t-T\geq t_{1}^{\prime}\Longrightarrow x(t)=x(t-T)). \label{pre899}%
\end{equation}
As $\mu\in\omega(x),$ $\textbf{T}_\mu^x$ is superiorly unbounded thus
$\textbf{T}_\mu^x\cap\lbrack t_1^\prime,\infty)\neq\varnothing.$ Let us take
now $t\in\textbf{T}_\mu^x\cap\lbrack t_1^\prime,\infty)$ and $z\in\textbf{Z}$
arbitrarily such that $t+zT\geq t_1^\prime.$ The following possibilities exist:

Case $z>0,$%
\[
\mu=x(t)\overset{(\ref{pre899})}{=}x(t+T)\overset{(\ref{pre899})}%
{=}x(t+2T)\overset{(\ref{pre899})}{=}...\overset{(\ref{pre899})}{=}x(t+zT);
\]

Case $z=0,$%
\[
\mu=x(t)=x(t+zT);
\]

Case $z<0,$%
\[
\mu=x(t)\overset{(\ref{pre899})}{=}x(t-T)\overset{(\ref{pre899})}%
{=}x(t-2T)\overset{(\ref{pre899})}{=}...\overset{(\ref{pre899})}{=}x(t+zT).
\]
It has resulted that in all these situations $x(t+zT)=\mu,$ thus
$t+zT\in\mathbf{T}_{\mu}^{x}.$
\end{proof}

\begin{example}
For the signal $x\in S^{(1)},\forall t\in\mathbf{R},$%
\[
x(t)=\chi_{(-\infty,0)}(t)\oplus\chi_{\lbrack3,4)}(t)\oplus\chi_{\lbrack
5,6)}(t)\oplus\chi_{\lbrack7,8)}(t)\oplus...
\]
neither of $0,1\in Or(x)$ is periodic, but for any $t^{\prime}\in
\lbrack2,\infty)$ we get%
\[
\forall t\in\mathbf{T}_{0}^{x}\cap\lbrack t^{\prime},\infty),\{t+z2|z\in
\mathbf{Z}\}\cap\lbrack t^{\prime},\infty)\subset\mathbf{T}_{0}^{x},
\]%
\[
\forall t\in\mathbf{T}_{1}^{x}\cap\lbrack t^{\prime},\infty),\{t+z2|z\in
\mathbf{Z}\}\cap\lbrack t^{\prime},\infty)\subset\mathbf{T}_{1}^{x},
\]
thus $0,1$ are eventually periodic with $P_{0}^{x}=P_{1}^{x}=\{2,4,6,...\}.$
\end{example}

\section{Discussion}

\begin{remark}
The properties (\ref{pre154}), (\ref{pre153}) are of eventual periodicity of
$\mu$, meaning that the periodicity starts at a limit of periodicity
$k^{\prime}$ which is in general bigger than the initial time $-1$.
Equivalently, the properties (\ref{pre158}), (\ref{pre159}) are of periodicity
of $\mu$ (starting at the initial time $-1$), however not the periodicity
referring to $\widehat{x},$ but the periodicity referring to $\widehat{\sigma
}^{k^{\prime\prime}}(\widehat{x}),$ $k^{\prime\prime}\geq0,$ meaning that
$\widehat{x}$ might have forgotten some of its first values. The real time
equivalent statements are interpreted similarly.
\end{remark}

\begin{remark}
Note that in the statement of Theorem \ref{The67} we have asked $\mu
\in\widehat{\omega}(\widehat{x}),\mu\in\omega(x)$ instead of $\mu\in
\widehat{Or}(\widehat{x}),\mu\in Or(x),$ the usual demand of periodicity of
$\mu.$ This avoids stating further requests of non-triviality $\widehat
{\mathbf{T}}_{\mu}^{\widehat{x}}\cap\{k^{\prime},k^{\prime}+1,k^{\prime
}+2,...\}\neq\varnothing$ and $\mathbf{T}_{\mu}^{x}\cap\lbrack t_{1}^{\prime
},\infty)\neq\varnothing$ that are necessary in eventual periodicity, see also
Lemma \ref{Lem37}, page \pageref{Lem37}.
\end{remark}

\begin{remark}
The prime period of the eventually periodic point $\mu\in\widehat{\omega
}(\widehat{x})$ always exists, but the prime period of the eventually periodic
point $\mu\in\omega(x)$ might not exist, for example if $x$ is eventually
constant and equal with $\mu$ (i.e. if $\underset{t\rightarrow\infty}{\lim
}x(t)=\mu$), when $P_{\mu}^{x}=(0,\infty).$ We shall prove in Theorem
\ref{The70}, page \pageref{The70} that this is the only situation when the
eventually periodic point $\mu\in\omega(x)$ has no prime period.
\end{remark}

\section{The accessibility of the eventually periodic points}

\begin{theorem}
\label{Lem1}a) Let $\widehat{x}$ and $\mu\in\widehat{\omega}(\widehat{x})$
that is eventually periodic, with the period $p\geq1$ and the limit of
periodicity $k^{\prime}\in\mathbf{N}_{\_}.$ For any $k\geq k^{\prime}$ we have
$\widehat{\mathbf{T}}_{\mu}^{\widehat{x}}\cap\{k,k+1,...,k+p-1\}\neq
\varnothing.$

b) Let $x$ and $\mu\in\omega(x)$ that is eventually periodic with the period
$T>0$ and the limit of periodicity $t^{\prime}\in\mathbf{R.}$ For any $t\geq
t^{\prime},$ we have $\mathbf{T}_{\mu}^{x}\cap\lbrack t,t+T)\neq\varnothing.$
\end{theorem}

\begin{proof}
a) The hypothesis implies the truth of%
\begin{equation}
\widehat{\mathbf{T}}_{\mu}^{\widehat{x}}\cap\{k^{\prime},k^{\prime
}+1,k^{\prime}+2,...\}\neq\varnothing, \label{p59}%
\end{equation}%
\begin{equation}
\forall k\in\widehat{\mathbf{T}}_{\mu}^{\widehat{x}}\cap\{k^{\prime}%
,k^{\prime}+1,k^{\prime}+2,...\},\{k+zp|z\in\mathbf{Z}\}\cap\{k^{\prime
},k^{\prime}+1,k^{\prime}+2,...\}\subset\widehat{\mathbf{T}}_{\mu}%
^{\widehat{x}}. \label{p60}%
\end{equation}
The truth of (\ref{p59}) allows us to define $k^{\prime\prime}=\min
\widehat{\mathbf{T}}_{\mu}^{\widehat{x}}\cap\{k^{\prime},k^{\prime
}+1,k^{\prime}+2,...\}$ and we prove that $k^{\prime\prime}\in\widehat
{\mathbf{T}}_{\mu}^{\widehat{x}}\cap\{k^{\prime},k^{\prime}+1,...,k^{\prime
}+p-1\}.$ If, against all reason, this would not be true, then we would have
$k^{\prime\prime}\geq k^{\prime}+p$ and%
\[
k^{\prime\prime}-p\in\{k^{\prime\prime}+zp|z\in\mathbf{Z}\}\cap\{k^{\prime
},k^{\prime}+1,k^{\prime}+2,...\}\overset{(\ref{p60})}{\subset}\widehat
{\mathbf{T}}_{\mu}^{\widehat{x}},
\]
representing a contradiction with the definition of $k^{\prime\prime}.$

From (\ref{p60}) we infer that $\{k^{\prime\prime},k^{\prime\prime
}+p,k^{\prime\prime}+2p,...\}\subset\widehat{\mathbf{T}}_{\mu}^{\widehat{x}%
}\cap\{k^{\prime},k^{\prime}+1,k^{\prime}+2,...\},$ meaning that $\forall
k\geq k^{\prime},$ $\widehat{\mathbf{T}}_{\mu}^{\widehat{x}}\cap
\{k,k+1,...,k+p-1\}\neq\varnothing.$

b) We have from the hypothesis that%
\begin{equation}
\mathbf{T}_{\mu}^{x}\cap\lbrack t^{\prime},\infty)\neq\varnothing,
\label{pre68}%
\end{equation}%
\begin{equation}
\forall t\in\mathbf{T}_{\mu}^{x}\cap\lbrack t^{\prime},\infty),\{t+zT|z\in
\mathbf{Z}\}\cap\lbrack t^{\prime},\infty)\subset\mathbf{T}_{\mu}^{x}
\label{pre69}%
\end{equation}
are fulfilled. The request (\ref{pre68}) allows defining $t^{\prime\prime
}=\min\mathbf{T}_{\mu}^{x}\cap\lbrack t^{\prime},\infty).$ We show that
$t^{\prime\prime}\in\mathbf{T}_{\mu}^{x}\cap\lbrack t^{\prime},t^{\prime}+T).$
If, against all reason, this would not be true, then we would have
$t^{\prime\prime}\geq t^{\prime}+T.$ This means that $t^{\prime\prime}-T\geq
t^{\prime},$ thus%
\[
t^{\prime\prime}-T\in\{t^{\prime\prime}+zT|z\in\mathbf{Z}\}\cap\lbrack
t^{\prime},\infty)\overset{(\ref{pre69})}{\subset}\mathbf{T}_{\mu}^{x},
\]
contradiction with the definition of $t^{\prime\prime}.$

By using (\ref{pre69}) we get $\{t^{\prime\prime},t^{\prime\prime}%
+T,t^{\prime\prime}+2T,...\}\subset\mathbf{T}_{\mu}^{x}\cap\lbrack t^{\prime
},\infty).$ The statement of the Theorem holds.
\end{proof}

\section{The limit of periodicity}

\begin{theorem}
\label{The137}a) $\widehat{x}\in\widehat{S}^{(n)},\mu\in\widehat{\omega
}(\widehat{x}),p\geq1,p^{\prime}\geq1,k^{\prime}\in\mathbf{N}_{\_}%
,k^{\prime\prime}\in\mathbf{N}_{\_}$ are given. If%
\begin{equation}
\left\{
\begin{array}
[c]{c}%
\forall k\in\widehat{\mathbf{T}}_{\mu}^{\widehat{x}}\cap\{k^{\prime}%
,k^{\prime}+1,k^{\prime}+2,...\},\\
\{k+zp|z\in\mathbf{Z}\}\cap\{k^{\prime},k^{\prime}+1,k^{\prime}+2,...\}\subset
\widehat{\mathbf{T}}_{\mu}^{\widehat{x}},
\end{array}
\right.
\end{equation}%
\begin{equation}
\left\{
\begin{array}
[c]{c}%
\forall k\in\widehat{\mathbf{T}}_{\mu}^{\widehat{x}}\cap\{k^{\prime\prime
},k^{\prime\prime}+1,k^{\prime\prime}+2,...\},\\
\{k+zp^{\prime}|z\in\mathbf{Z}\}\cap\{k^{\prime\prime},k^{\prime\prime
}+1,k^{\prime\prime}+2,...\}\subset\widehat{\mathbf{T}}_{\mu}^{\widehat{x}}%
\end{array}
\right.
\end{equation}
hold, then
\begin{equation}
\left\{
\begin{array}
[c]{c}%
\forall k\in\widehat{\mathbf{T}}_{\mu}^{\widehat{x}}\cap\{k^{\prime}%
,k^{\prime}+1,k^{\prime}+2,...\},\\
\{k+zp^{\prime}|z\in\mathbf{Z}\}\cap\{k^{\prime},k^{\prime}+1,k^{\prime
}+2,...\}\subset\widehat{\mathbf{T}}_{\mu}^{\widehat{x}}%
\end{array}
\right.
\end{equation}
is true.

b) Let $x\in S^{(n)},\mu\in\omega(x),T>0,T^{\prime}>0,t^{\prime}\in
\mathbf{R},t^{\prime\prime}\in\mathbf{R}$. Then%
\begin{equation}
\forall t\in\mathbf{T}_{\mu}^{x}\cap\lbrack t^{\prime},\infty),\{t+zT|z\in
\mathbf{Z}\}\cap\lbrack t^{\prime},\infty)\subset\mathbf{T}_{\mu}^{x},
\label{p236}%
\end{equation}%
\begin{equation}
\forall t\in\mathbf{T}_{\mu}^{x}\cap\lbrack t^{\prime\prime},\infty
),\{t+zT^{\prime}|z\in\mathbf{Z}\}\cap\lbrack t^{\prime\prime},\infty
)\subset\mathbf{T}_{\mu}^{x} \label{p237}%
\end{equation}
imply%
\begin{equation}
\forall t\in\mathbf{T}_{\mu}^{x}\cap\lbrack t^{\prime},\infty),\{t+zT^{\prime
}|z\in\mathbf{Z}\}\cap\lbrack t^{\prime},\infty)\subset\mathbf{T}_{\mu}^{x}.
\label{p238}%
\end{equation}

\end{theorem}

\begin{proof}
b) Let $t\in\mathbf{T}_{\mu}^{x},z\in\mathbf{Z}$ arbitrary such that $t\geq
t^{\prime}$ and $t+zT^{\prime}\geq t^{\prime}.$ We have the following possibilities.

Case $t^{\prime}\geq t^{\prime\prime}$

Then $t\geq t^{\prime\prime}$ and $t+zT^{\prime}\geq t^{\prime\prime},$ thus
$t+zT^{\prime}\overset{(\ref{p237})}{\in}\mathbf{T}_{\mu}^{x}.$

Case $t^{\prime}<t^{\prime\prime}$

$k\in\mathbf{N}$ exists with $t+kT\geq t^{\prime\prime},t+zT^{\prime}+kT\geq
t^{\prime\prime}.$ Obviously $t+kT\geq t^{\prime}$ and we can write%
\[
\mu=x(t)\overset{(\ref{p236})}{=}x(t+kT)\overset{(\ref{p237})}{=}%
x(t+zT^{\prime}+kT)\overset{(\ref{p236})}{=}x(t+zT^{\prime}),
\]
in other words $t+zT^{\prime}\in\mathbf{T}_{\mu}^{x}.$
\end{proof}

\begin{remark}
The previous Theorem states that the set of the limits of periodicity does not
depend on the period. In particular, this justifies the notations $\widehat
{L}_{\mu}^{\widehat{x}},L_{\mu}^{x}$ where the period is missing.
\end{remark}

\begin{example}
Let the signal $x\in S^{(1)},$%
\[
x(t)=\chi_{\lbrack0,1)}(t)\oplus\chi_{\lbrack4,5)}(t)\oplus\chi_{\lbrack
6,7)}(t)\oplus\chi_{\lbrack8,9)}(t)\oplus\chi_{\lbrack10,11)}(t)\oplus...
\]
where $2,4\in P_{1}^{x}.$ We might be tempted to think that the eventual
periodicity of $1$ with the period $T=4$ has the prime limit of periodicity
$t^{\prime}$ different from its eventual periodicity with $T^{\prime}=2$ that
has the prime limit of periodicity $t^{\prime\prime}=3.$ This is not the case,
the fact that $[2,3)\subset\mathbf{T}_{1}^{x}$ is false shows that both prime
limits of periodicity are $t^{\prime}=3.$
\end{example}

\begin{theorem}
\label{The117}a) Let $\widehat{x}\in\widehat{S}^{(n)}$ and $\mu\in
\widehat{\omega}(\widehat{x})$ with the property that $\widehat{L}_{\mu
}^{\widehat{x}}\neq\varnothing.$ Then $k^{\prime}\in\mathbf{N}_{\_}$ exists
with $\widehat{L}_{\mu}^{\widehat{x}}=\{k^{\prime},k^{\prime}+1,k^{\prime
}+2,...\}.$

b) Let $x\in S^{(n)}$ non constant and $\mu\in\omega(x)$ having the property
that $L_{\mu}^{x}\neq\varnothing.$ Then $t^{\prime}\in\mathbf{R}$ exists such
that $L_{\mu}^{x}=[t^{\prime},\infty).$
\end{theorem}

\begin{proof}
a) The statement is a consequence of Lemma \ref{Lem30}, page \pageref{Lem30}.

b) Because $x$ is not constant, $t_{0}\in\mathbf{R}$ exists with
$I^{x}=(-\infty,t_{0}).$ Let $T\in P_{\mu}^{x}.$

b.i) We show first that $t_{0}-T\notin L_{\mu}^{x}$ and we suppose against all
reason that $t_{0}-T\in L_{\mu}^{x}.$ We have two possibilities.

Case $\mu=x(-\infty+0)$

The hypothesis $t_{0}-T\in L_{\mu}^{x}$ implies, as far as $t_{0}%
-T\in\mathbf{T}_{\mu}^{x},$%
\[
\mu=x(t_{0}-T)=x(t_{0}),
\]
representing a contradiction with the fact that $t_{0}\notin I^{x}.$

Case $\mu\neq x(-\infty+0)$

We infer from Theorem \ref{Lem1}, page \pageref{Lem1} that $\mathbf{T}_{\mu
}^{x}\cap\lbrack t_{0}-T,t_{0})\neq\varnothing,$ wherefrom $\mu=x(-\infty+0),$
representing a contradiction.

b.ii) From b.i) and from Lemma \ref{Lem30}, we draw the conclusion that
$L_{\mu}^{x}$ has one of the forms $L_{\mu}^{x}=(t^{\prime},\infty),L_{\mu
}^{x}=[t^{\prime},\infty),$ where $t^{\prime}>t_{0}-T.$ We show that the first
possibility cannot take place, thus we suppose against all reason that
$t^{\prime}$ exists with $L_{\mu}^{x}=(t^{\prime},\infty).$ We have the
existence of $\varepsilon^{\prime}>0,\varepsilon^{\prime\prime}>0$ such that%
\begin{equation}
\forall t\in(t^{\prime},t^{\prime}+\varepsilon^{\prime}),x(t)=x(t^{\prime}),
\label{p223}%
\end{equation}%
\begin{equation}
\forall t\in(t^{\prime}+T,t^{\prime}+T+\varepsilon^{\prime\prime
}),x(t)=x(t^{\prime}+T) \label{p224}%
\end{equation}
and let $\varepsilon\in(0,\min\{\varepsilon^{\prime},\varepsilon^{\prime
\prime}\}).$ Two possibilities exist.

Case $x(t^{\prime})=\mu$

We have $t^{\prime}\notin L_{\mu}^{x},$ thus $x(t^{\prime}+T)\neq\mu$ and
$(t^{\prime},t^{\prime}+\varepsilon)\subset L_{\mu}^{x}$ means that
\begin{equation}
\forall t\in(t^{\prime},t^{\prime}+\varepsilon),x(t)=\mu\Longrightarrow
x(t)=x(t+T). \label{p225}%
\end{equation}
Let $t\in(t^{\prime},t^{\prime}+\varepsilon)$ arbitrary. We can write%
\[
\mu=x(t^{\prime})\overset{(\ref{p223})}{=}x(t)\overset{(\ref{p225})}%
{=}x(t+T)\overset{(\ref{p224})}{=}x(t^{\prime}+T),
\]
contradiction.

Case $x(t^{\prime})\neq\mu$

In this case two possibilities exist. The case $x(t^{\prime}+T)=\mu$ when
$(t^{\prime},t^{\prime}+\varepsilon)\subset L_{\mu}^{x}$ means the truth of
(\ref{p225}). Let $t\in(t^{\prime},t^{\prime}+\varepsilon)$ arbitrary. We
conclude%
\[
\mu=x(t^{\prime}+T)\overset{(\ref{p224})}{=}x(t+T)\overset{(\ref{p225})}%
{=}x(t)\overset{(\ref{p223})}{=}x(t^{\prime}),
\]
representing a contradiction. And the case $x(t^{\prime}+T)\neq\mu$ when
$\forall k\in\mathbf{N},x(t+kT)\neq\mu.$ As for any $t\in\mathbf{T}_{\mu}%
^{x}\cap(t^{\prime},\infty)=\mathbf{T}_{\mu}^{x}\cap\lbrack t^{\prime}%
,\infty),$ we have $\{t+zT|z\in\mathbf{Z}\}\cap(t^{\prime},\infty
)=\{t+zT|z\in\mathbf{Z}\}\cap\lbrack t^{\prime},\infty),$ the conclusion is
$t^{\prime}\in L_{\mu}^{x},$ contradiction.

It has resulted that the existence of $t^{\prime}>t_{0}-T$ with $L_{\mu}%
^{x}=[t^{\prime},\infty)$ is the only possibility.
\end{proof}

\section{A property of eventual constancy}

\begin{theorem}
\label{The19}We consider the signals $\widehat{x},x.$

a) Let $\mu\in\widehat{\omega}(\widehat{x}).$ If $k^{\prime}\in\mathbf{N}%
_{\_}$ exists making%
\begin{equation}
\forall k\in\widehat{\mathbf{T}}_{\mu}^{\widehat{x}}\cap\{k^{\prime}%
,k^{\prime}+1,k^{\prime}+2,...\},\{k+zp|z\in\mathbf{Z}\}\cap\{k^{\prime
},k^{\prime}+1,k^{\prime}+2,...\}\subset\widehat{\mathbf{T}}_{\mu}%
^{\widehat{x}} \label{per372}%
\end{equation}
true for $p=1,$ then%
\begin{equation}
\forall k\geq k^{\prime},\widehat{x}(k)=\mu\label{per374}%
\end{equation}
and (\ref{per372}) holds for any $p\geq1.$

b) Let $\mu\in\omega(x)$ and we suppose that $t_{0}\in\mathbf{R},h>0$ exist
such that $x$ is of the form%
\begin{equation}%
\begin{array}
[c]{c}%
x(t)=x(-\infty+0)\cdot\chi_{(-\infty,t_{0})}(t)\oplus x(t_{0})\cdot
\chi_{\lbrack t_{0},t_{0}+h)}(t)\oplus...\\
...\oplus x(t_{0}+kh)\cdot\chi_{\lbrack t_{0}+kh,t_{0}+(k+1)h)}(t)\oplus...
\end{array}
\label{per486}%
\end{equation}
If $t^{\prime}\in\mathbf{R,}$ $T\in(0,h)\cup(h,2h)\cup...\cup(qh,(q+1)h)\cup
...$ exist making%
\begin{equation}
\forall t\in\mathbf{T}_{\mu}^{x}\cap\lbrack t^{\prime},\infty),\{t+zT|z\in
\mathbf{Z}\}\cap\lbrack t^{\prime},\infty)\subset\mathbf{T}_{\mu}^{x}
\label{per373}%
\end{equation}
true, then%
\begin{equation}
\forall t\geq t^{\prime},x(t)=\mu\label{per375}%
\end{equation}
is true and in this case (\ref{per373}) holds for any $T>0.$

c) We ask that (\ref{per486}) is fulfilled under the form%
\begin{equation}%
\begin{array}
[c]{c}%
x(t)=\widehat{x}(-1)\cdot\chi_{(-\infty,t_{0})}(t)\oplus\widehat{x}%
(0)\cdot\chi_{\lbrack t_{0},t_{0}+h)}(t)\oplus...\\
...\oplus\widehat{x}(k)\cdot\chi_{\lbrack t_{0}+kh,t_{0}+(k+1)h)}(t)\oplus...
\end{array}
\label{per371}%
\end{equation}
and let $\mu\in\widehat{\omega}(\widehat{x})=\omega(x)\footnote{The fact that
(\ref{per371}) implies $\widehat{\omega}(\widehat{x})=\omega(x)$ was proved at
Theorem \ref{The114}, page \pageref{The114}.}$ be arbitrary. The following
statements hold:

c.1) If $k^{\prime}\in\mathbf{N}_{\_}$ exists making (\ref{per372}) true for
$p=1$, then (\ref{per374}) is true, (\ref{per372}) holds for any $p\geq1 $ and
$t^{\prime}\in\mathbf{R}$ exists such that (\ref{per375}) holds and
(\ref{per373}) is also true for any $T>0.$

c.2) If $t^{\prime}\in\mathbf{R,}$ $T\in(0,h)\cup(h,2h)\cup...\cup
(qh,(q+1)h)\cup...$ exist making (\ref{per373}) true, then (\ref{per375}) is
true, (\ref{per373}) is true for any $T>0$ and $k^{\prime}\in\mathbf{N}_{\_}$
exists such that (\ref{per374}) is true and (\ref{per372}) is also true for
any $p\geq1$.
\end{theorem}

\begin{proof}
a) Some $k^{\prime}\in\mathbf{N}_{\_}$ exists with (\ref{per372}) fulfilled
for $p=1$, meaning that%
\[
\forall k,\forall z,(k\in\widehat{\mathbf{T}}_{\mu}^{\widehat{x}}\text{ and
}z\in\mathbf{Z}\text{ and }k\geq k^{\prime}\text{ and }k+z\geq k^{\prime
})\Longrightarrow k+z\in\widehat{\mathbf{T}}_{\mu}^{\widehat{x}}%
\]
holds. $\mu\in\widehat{\omega}(\widehat{x})$ implies that $\widehat
{\mathbf{T}}_{\mu}^{\widehat{x}}$ is infinite, thus some $k\in\widehat
{\mathbf{T}}_{\mu}^{\widehat{x}},k\geq k^{\prime}$ exists indeed. We infer%
\[
\{k^{\prime},k^{\prime}+1,k^{\prime}+2,...\}\subset\widehat{\mathbf{T}}_{\mu
}^{\widehat{x}}.
\]
We have obtained the truth of (\ref{per374}). In these circumstances
(\ref{per372}) holds for any $p\geq1$.

b) We suppose that $t_{0}\in\mathbf{R},h>0$ exist such that (\ref{per486})
holds and also that $t^{\prime}\in\mathbf{R,}$ $T\in(0,h)\cup(h,2h)\cup
...\cup(qh,(q+1)h)\cup...$ exist such that (\ref{per373}) is true.
Furthermore, $\mu\in\omega(x)$ implies $\mathbf{T}_{\mu}^{x}\cap\lbrack
t^{\prime},\infty)\neq\varnothing,$ since $\mathbf{T}_{\mu}^{x}$ is unbounded
from above.

We show first the existence of $\overline{t}\in\mathbf{R}$ such that%
\begin{equation}
\forall t\geq\overline{t},x(t)=\mu\label{per377}%
\end{equation}
is true.

We have the existence of $k^{\prime}\in\mathbf{N}_{\_}$ such that
$t_{0}+k^{\prime}h\geq t^{\prime},$ $x(t_{0}+k^{\prime}h)=\mu$ and%
\begin{equation}
\forall t\in\lbrack t_{0}+k^{\prime}h,t_{0}+(k^{\prime}+1)h),x(t)=\mu.
\label{per378}%
\end{equation}

Case $T\in(0,h),$ when%
\begin{equation}
t_{0}+k^{\prime}h<t_{0}+(k^{\prime}+1)h-T<t_{0}+(k^{\prime}+1)h,
\label{per389}%
\end{equation}%
\begin{equation}
\mu\overset{(\ref{per378})}{=}x(t_{0}+k^{\prime}h)\overset{(\ref{per389})}%
{=}x(t_{0}+(k^{\prime}+1)h-T)\overset{(\ref{per373})}{=}x(t_{0}+(k^{\prime
}+1)h), \label{per381}%
\end{equation}%
\begin{equation}
\forall t\in\lbrack t_{0}+(k^{\prime}+1)h,t_{0}+(k^{\prime}+2)h),x(t)\overset
{(\ref{per381})}{=}\mu; \label{per382}%
\end{equation}%
\begin{equation}
t_{0}+(k^{\prime}+1)h<t_{0}+(k^{\prime}+2)h-T<t_{0}+(k^{\prime}+2)h,
\label{per390}%
\end{equation}%
\begin{equation}
\mu\overset{(\ref{per381})}{=}x(t_{0}+(k^{\prime}+1)h)\overset{(\ref{per390}%
)}{=}x(t_{0}+(k^{\prime}+2)h-T)\overset{(\ref{per373})}{=}x(t_{0}+(k^{\prime
}+2)h), \label{per383}%
\end{equation}%
\begin{equation}
\forall t\in\lbrack t_{0}+(k^{\prime}+2)h,t_{0}+(k^{\prime}+3)h),x(t)\overset
{(\ref{per383})}{=}\mu; \label{per427}%
\end{equation}%
\[
...
\]
Thus the statement (\ref{per377}) holds, from (\ref{per378}), (\ref{per382}),
(\ref{per427}),... for $\overline{t}=t_{0}+k^{\prime}h.$

Case $T\in(h,2h),$ when%
\begin{equation}
t_{0}+(k^{\prime}+1)h<t_{0}+k^{\prime}h+T<t_{0}+(k^{\prime}+2)h,
\label{per402}%
\end{equation}%
\begin{equation}
\mu\overset{(\ref{per378})}{=}x(t_{0}+k^{\prime}h)\overset{(\ref{per373})}%
{=}x(t_{0}+k^{\prime}h+T)\overset{(\ref{per402})}{=}x(t_{0}+(k^{\prime}+1)h),
\label{per384}%
\end{equation}%
\begin{equation}
\forall t\in\lbrack t_{0}+(k^{\prime}+1)h,t_{0}+(k^{\prime}+2)h),x(t)\overset
{(\ref{per384})}{=}\mu; \label{per385}%
\end{equation}%
\begin{equation}
t_{0}+(k^{\prime}+2)h<t_{0}+(k^{\prime}+1)h+T<t_{0}+(k^{\prime}+3)h,
\label{per403}%
\end{equation}%
\begin{equation}
\mu\overset{(\ref{per384})}{=}x(t_{0}+(k^{\prime}+1)h)\overset{(\ref{per373}%
)}{=}x(t_{0}+(k^{\prime}+1)h+T)\overset{(\ref{per403})}{=}x(t_{0}+(k^{\prime
}+2)h), \label{per386}%
\end{equation}%
\begin{equation}
\forall t\in\lbrack t_{0}+(k^{\prime}+2)h,t_{0}+(k^{\prime}+3)h),x(t)\overset
{(\ref{per386})}{=}\mu; \label{per387}%
\end{equation}%
\[
...
\]
The statement (\ref{per377}) holds, from (\ref{per378}), (\ref{per385}),
(\ref{per387}),... for $\overline{t}=t_{0}+k^{\prime}h.$

Case $T\in(2h,3h).$ In this situation%
\begin{equation}
t_{0}+(k^{\prime}+2)h<t_{0}+k^{\prime}h+T<t_{0}+(k^{\prime}+3)h,
\label{per404}%
\end{equation}%
\begin{equation}
\mu\overset{(\ref{per378})}{=}x(t_{0}+k^{\prime}h)\overset{(\ref{per373})}%
{=}x(t_{0}+k^{\prime}h+T)\overset{(\ref{per404})}{=}x(t_{0}+(k^{\prime}+2)h),
\label{per405}%
\end{equation}%
\begin{equation}
\forall t\in\lbrack t_{0}+(k^{\prime}+2)h,t_{0}+(k^{\prime}+3)h),x(t)\overset
{(\ref{per405})}{=}\mu; \label{per406}%
\end{equation}%
\begin{equation}
t_{0}+(k^{\prime}+4)h<t_{0}+(k^{\prime}+2)h+T<t_{0}+(k^{\prime}+5)h,
\label{per391}%
\end{equation}%
\begin{equation}
\mu\overset{(\ref{per405})}{=}x(t_{0}+(k^{\prime}+2)h)\overset{(\ref{per373}%
)}{=}x(t_{0}+(k^{\prime}+2)h+T)\overset{(\ref{per391})}{=}x(t_{0}+(k^{\prime
}+4)h), \label{per407}%
\end{equation}%
\begin{equation}
\forall t\in\lbrack t_{0}+(k^{\prime}+4)h,t_{0}+(k^{\prime}+5)h),x(t)\overset
{(\ref{per407})}{=}\mu; \label{per408}%
\end{equation}%
\begin{equation}
t_{0}+(k^{\prime}+6)h<t_{0}+(k^{\prime}+4)h+T<t_{0}+(k^{\prime}+7)h,
\label{per409}%
\end{equation}%
\[
...
\]%
\begin{equation}
\forall j\in\mathbf{N},\forall t\in\lbrack t_{0}+(k^{\prime}+2j)h,t_{0}%
+(k^{\prime}+2j+1)h),x(t)=\mu. \label{per410}%
\end{equation}

Furthermore%
\begin{equation}
t_{0}+(k^{\prime}+1)h<t_{0}+(k^{\prime}+4)h-T<t_{0}+(k^{\prime}+2)h,
\label{per411}%
\end{equation}%
\begin{equation}
\mu\overset{(\ref{per410})}{=}x(t_{0}+(k^{\prime}+4)h)\overset{(\ref{per373}%
)}{=}x(t_{0}+(k^{\prime}+4)h-T)\overset{(\ref{per411})}{=}x(t_{0}+(k^{\prime
}+1)h), \label{per412}%
\end{equation}%
\begin{equation}
\forall t\in\lbrack t_{0}+(k^{\prime}+1)h,t_{0}+(k^{\prime}+2)h),x(t)\overset
{(\ref{per412})}{=}\mu; \label{per413}%
\end{equation}%
\begin{equation}
t_{0}+(k^{\prime}+3)h<t_{0}+(k^{\prime}+1)h+T<t_{0}+(k^{\prime}+4)h,
\end{equation}%
\[
...
\]
and we prove that%
\begin{equation}
\forall j\in\mathbf{N},\forall t\in\lbrack t_{0}+(k^{\prime}+2j+1)h,t_{0}%
+(k^{\prime}+2j+2)h),x(t)=\mu\label{per414}%
\end{equation}
similarly with (\ref{per404}),...,(\ref{per410}), starting from $x(t_{0}%
+(k^{\prime}+1)h)\overset{(\ref{per413})}{=}\mu$ instead of $x(t_{0}%
+k^{\prime}h)\overset{(\ref{per378})}{=}\mu.$ From (\ref{per410}),
(\ref{per414}) we infer that the statement (\ref{per377}) is true for
$\overline{t}=t_{0}+k^{\prime}h.$

Case $T\in(3h,4h),$%
\begin{equation}
t_{0}+(k^{\prime}+3)h<t_{0}+k^{\prime}h+T<t_{0}+(k^{\prime}+4)h,
\label{per415}%
\end{equation}%
\begin{equation}
\mu\overset{(\ref{per378})}{=}x(t_{0}+k^{\prime}h)\overset{(\ref{per373})}%
{=}x(t_{0}+k^{\prime}h+T)\overset{(\ref{per415})}{=}x(t_{0}+(k^{\prime}+3)h),
\label{per416}%
\end{equation}%
\begin{equation}
\forall t\in\lbrack t_{0}+(k^{\prime}+3)h,t_{0}+(k^{\prime}+4)h),x(t)\overset
{(\ref{per416})}{=}\mu; \label{per393}%
\end{equation}%
\begin{equation}
t_{0}+(k^{\prime}+6)h<t_{0}+(k^{\prime}+3)h+T<t_{0}+(k^{\prime}+7)h,
\label{per392}%
\end{equation}%
\begin{equation}
\mu\overset{(\ref{per416})}{=}x(t_{0}+(k^{\prime}+3)h)\overset{(\ref{per373}%
)}{=}x(t_{0}+(k^{\prime}+3)h+T)\overset{(\ref{per392})}{=}x(t_{0}+(k^{\prime
}+6)h), \label{per417}%
\end{equation}%
\begin{equation}
\forall t\in\lbrack t_{0}+(k^{\prime}+6)h,t_{0}+(k^{\prime}+7)h),x(t)\overset
{(\ref{per417})}{=}\mu; \label{per394}%
\end{equation}%
\begin{equation}
t_{0}+(k^{\prime}+9)h<t_{0}+(k^{\prime}+6)h+T<t_{0}+(k^{\prime}+10)h,
\end{equation}%
\[
...
\]%
\begin{equation}
\forall j\in\mathbf{N},\forall t\in\lbrack t_{0}+(k^{\prime}+3j)h,t_{0}%
+(k^{\prime}+3j+1)h),x(t)=\mu. \label{per418}%
\end{equation}
Furthermore,%
\begin{equation}
t_{0}+(k^{\prime}+2)h<t_{0}+(k^{\prime}+6)h-T<t_{0}+(k^{\prime}+3)h,
\label{per419}%
\end{equation}%
\begin{equation}
\mu\overset{(\ref{per418})}{=}x(t_{0}+(k^{\prime}+6)h)\overset{(\ref{per373}%
)}{=}x(t_{0}+(k^{\prime}+6)h-T)\overset{(\ref{per419})}{=}x(t_{0}+(k^{\prime
}+2)h) \label{per420}%
\end{equation}%
\[
...
\]
and we remake the reasoning (\ref{per415}),...,(\ref{per418}) starting from
$x(t_{0}+(k^{\prime}+2)h)\overset{(\ref{per419})}{=}\mu$ instead of
$x(t_{0}+k^{\prime}h)\overset{(\ref{per378})}{=}\mu.$ We obtain:%
\begin{equation}
\forall j\in\mathbf{N},\forall t\in\lbrack t_{0}+(k^{\prime}+3j+2)h,t_{0}%
+(k^{\prime}+3j+3)h),x(t)=\mu\label{per421}%
\end{equation}
and we also have%
\begin{equation}
t_{0}+(k^{\prime}+1)h<t_{0}+(k^{\prime}+5)h-T<t_{0}+(k^{\prime}+2)h,
\label{per422}%
\end{equation}%
\begin{equation}
\mu\overset{(\ref{per421})}{=}x(t_{0}+(k^{\prime}+5)h)\overset{(\ref{per373}%
)}{=}x(t_{0}+(k^{\prime}+5)h-T)\overset{(\ref{per422})}{=}x(t_{0}+(k^{\prime
}+1)h) \label{per423}%
\end{equation}%
\[
...
\]
We remake the reasoning (\ref{per415}),...,(\ref{per418}) starting from
$x(t_{0}+(k^{\prime}+1)h)\overset{(\ref{per423})}{=}\mu$ instead of
$x(t_{0}+k^{\prime}h)\overset{(\ref{per378})}{=}\mu.$ We get:%
\begin{equation}
\forall j\in\mathbf{N},\forall t\in\lbrack t_{0}+(k^{\prime}+3j+1)h,t_{0}%
+(k^{\prime}+3j+2)h),x(t)=\mu. \label{per424}%
\end{equation}
From (\ref{per418}), (\ref{per421}), (\ref{per424}) we have the truth of
(\ref{per377}) for $\overline{t}=t_{0}+k^{\prime}h.$

In the general case $T\in(qh,(q+1)h),$ $q\geq2$ we prove in succession the
truth of%
\[
\forall j\in\mathbf{N},\forall t\in\lbrack t_{0}+(k^{\prime}+qj)h,t_{0}%
+(k^{\prime}+qj+1)h),x(t)=\mu,
\]%
\[
\forall j\in\mathbf{N},\forall t\in\lbrack t_{0}+(k^{\prime}+qj+q-1)h,t_{0}%
+(k^{\prime}+qj+q)h),x(t)=\mu,
\]%
\[
...
\]%
\[
\forall j\in\mathbf{N},\forall t\in\lbrack t_{0}+(k^{\prime}+qj+1)h,t_{0}%
+(k^{\prime}+qj+2)h),x(t)=\mu,
\]
wherefrom the truth of (\ref{per377}) follows for $\overline{t}=t_{0}%
+k^{\prime}h.$

We prove now that in (\ref{per375}) we can take $\overline{t}=t^{\prime}.$ Let
us suppose, against all reason, that this is not true, i.e. $\overline
{t}>t^{\prime}$\footnote{if $\overline{t}<t^{\prime},$ the other way of
negating $\overline{t}=t^{\prime},$ then from $\forall t\geq\overline
{t},x(t)=\mu,$ we can write $\forall t\geq t^{\prime},x(t)=\mu,$ i.e. finally
we can take $\overline{t}=t^{\prime}.$} and some $t^{\prime\prime}\in\lbrack
t^{\prime},\overline{t})$ exists with $x(t^{\prime\prime})\neq\mu.$ Let
$q\geq1$ with the property that $t^{\prime\prime}+qT\geq\overline{t},$ in
other words $t^{\prime\prime}+qT\in\mathbf{T}_{\mu}^{x}\cap\lbrack t^{\prime
},\infty).$ Then
\[
t^{\prime\prime}+qT-qT\in\{t^{\prime\prime}+qT+zT|z\in\mathbf{Z}\}\cap\lbrack
t^{\prime},\infty)\subset\mathbf{T}_{\mu}^{x}%
\]
and we infer that $x(t^{\prime\prime})=\mu,$ contradiction. (\ref{per375}) is
proved and obviously (\ref{per373}) holds for any $T>0.$

c) This is a consequence of a) and b).
\end{proof}

\section{Discrete time vs real time}

\begin{theorem}
\label{The24}We consider the signals $\widehat{x}\in\widehat{S}^{(n)},x\in
S^{(n)}$ which are not eventually constant and we suppose that%
\begin{equation}%
\begin{array}
[c]{c}%
x(t)=\widehat{x}(-1)\cdot\chi_{(-\infty,t_{0})}(t)\oplus\widehat{x}%
(0)\cdot\chi_{\lbrack t_{0},t_{0}+h)}(t)\oplus...\\
...\oplus\widehat{x}(k)\cdot\chi_{\lbrack t_{0}+kh,t_{0}+(k+1)h)}(t)\oplus...
\end{array}
\label{per141_}%
\end{equation}
is true, where $t_{0}\in\mathbf{R},h>0.$ Let $\mu\in\widehat{\omega}%
(\widehat{x})=\omega(x).$

a) If $p\geq1$ and $k^{\prime}\in\mathbf{N}_{\_}$ exist such that%
\begin{equation}
\forall k\in\widehat{\mathbf{T}}_{\mu}^{\widehat{x}}\cap\{k^{\prime}%
,k^{\prime}+1,k^{\prime}+2,...\},\{k+zp|z\in\mathbf{Z}\}\cap\{k^{\prime
},k^{\prime}+1,k^{\prime}+2,...\}\subset\widehat{\mathbf{T}}_{\mu}%
^{\widehat{x}}, \label{per144_}%
\end{equation}
then $t^{\prime}\in\mathbf{R}$ exists with
\begin{equation}
\forall t\in\mathbf{T}_{\mu}^{x}\cap\lbrack t^{\prime},\infty),\{t+zT|z\in
\mathbf{Z}\}\cap\lbrack t^{\prime},\infty)\subset\mathbf{T}_{\mu}^{x}
\label{per145_}%
\end{equation}
true for $T=ph.$

b) If $T>0$ and $t^{\prime}\in\mathbf{R}$ exist for which (\ref{per145_}) is
true, then $\frac{T}{h}\in\{1,2,3,...\}$ and $k^{\prime}\in\mathbf{N}_{\_}$
exists such that (\ref{per144_}) is true for $p=\frac{T}{h}.$
\end{theorem}

\begin{proof}
a) The hypothesis states the existence of $t_{0}\in\mathbf{R},h>0$ such that
(\ref{per141_}) is true and also that, given $\mu\in\widehat{\omega}%
(\widehat{x})=\omega(x),$ $p\geq1$ and $k^{\prime}\in\mathbf{N}_{\_}$ exist
with (\ref{per144_}) fulfilled.

We define $T=ph,t^{\prime}=t_{0}+k^{\prime}h.$ Let $t\in\mathbf{T}_{\mu}%
^{x},z\in\mathbf{Z}$ be arbitrary with the property that $t\geq t^{\prime
},t+zT\geq t^{\prime}$ ($\mu\in\omega(x)$ implies that $\mathbf{T}_{\mu}%
^{x}\cap\lbrack t^{\prime},\infty)\neq\varnothing$). Some $k\geq k^{\prime}$
exists then such that $t\in\lbrack t_{0}+kh,t_{0}+(k+1)h)$ and we can write%
\[
t+zT\in\lbrack t_{0}+kh+zT,t_{0}+(k+1)h+zT)=[t_{0}+(k+zp)h,t_{0}+(k+1+zp)h).
\]
Obviously $t_{0}+(k+zp)h\geq t_{0}+k^{\prime}h=t^{\prime}$ implies $k+zp\geq
k^{\prime}.$ We infer%
\[
\mu=x(t)=\widehat{x}(k)\overset{(\ref{per144_})}{=}\widehat{x}(k+zp)=x(t+zT),
\]
in other words (\ref{per145_}) holds.

b) Some $t_{0}\in\mathbf{R}$ and $h>0$ exist from the hypothesis such that
(\ref{per141_}) is true and, given $\mu,$ some $T>0,t^{\prime}\in\mathbf{R} $
exist also such that (\ref{per145_}) holds. If $T\in(0,h)\cup(h,2h)\cup
...\cup(qh,(q+1)h)\cup...$ then from Theorem \ref{The19} b), page
\pageref{The19}, we have that $\underset{t\rightarrow\infty}{\lim
}x(t)=\underset{k\rightarrow\infty}{\lim}\widehat{x}(k)=\mu,$ contradiction
with the hypothesis, thus $T\in\{h,2h,3h,...\}$ for which we define
$p=\frac{T}{h},p\geq1.$ As $\mu\in\omega(x),\mathbf{T}_{\mu}^{x}$ is unbounded
from above and $\mathbf{T}_{\mu}^{x}\cap\lbrack t^{\prime},\infty
)\neq\varnothing$ is true for any $t^{\prime}.$ We can suppose, by making use
of Lemma \ref{Lem30}, page \pageref{Lem30} that in (\ref{per145_}) we have
$t^{\prime}\geq t_{0}-h$ and we denote by $k^{\prime}\in\mathbf{N}_{\_}$ the
number for which $t^{\prime}\in\lbrack t_{0}+k^{\prime}h,t_{0}+(k^{\prime
}+1)h).$

Let now $k\in\widehat{\mathbf{T}}_{\mu}^{\widehat{x}}$ and $z\in\mathbf{Z} $
be arbitrary with $k\geq k^{\prime}$ and $k+zp\geq k^{\prime}$ ($\widehat
{\mathbf{T}}_{\mu}^{\widehat{x}}\cap\{k^{\prime},k^{\prime}+1,k^{\prime
}+2,...\}\neq\varnothing$ because $\mu\in\widehat{\omega}(\widehat{x}%
)=\omega(x)$ and $\widehat{\mathbf{T}}_{\mu}^{\widehat{x}}$ is infinite). Then%
\begin{equation}
t^{\prime}+(k-k^{\prime})h\geq t^{\prime}, \label{per750}%
\end{equation}%
\begin{equation}
t^{\prime}+(k-k^{\prime}+zp)h\geq t^{\prime} \label{per751}%
\end{equation}
and on the other hand%
\begin{equation}
t_{0}+kh\leq t^{\prime}+(k-k^{\prime})h<t_{0}+(k+1)h, \label{per544}%
\end{equation}%
\begin{equation}
t_{0}+(k+zp)h\leq t^{\prime}+(k-k^{\prime}+zp)h<t_{0}+(k+zp+1)h \label{per545}%
\end{equation}
are true. We conclude%
\[
\mu=\widehat{x}(k)\overset{(\ref{per544})}{=}x(t^{\prime}+(k-k^{\prime
})h)\overset{(\ref{per145_}),(\ref{per750}),(\ref{per751})}{=}%
\]%
\[
=x(t^{\prime}+(k-k^{\prime})h+zT)=x(t^{\prime}+(k-k^{\prime}+zp)h)\overset
{(\ref{per545})}{=}\widehat{x}(k+zp).
\]
(\ref{per144_}) holds.
\end{proof}

\begin{example}
Let the signal $\widehat{x}\in\widehat{S}^{(1)}$ such that%
\[
\widehat{x}=0,0,0,0,1,\widehat{x}(4),\widehat{x}(5),1,\widehat{x}%
(7),\widehat{x}(8),1,\widehat{x}(10),...
\]
Then $1$ is an eventually periodic point of $\widehat{x}$, $p=3$ is its period
and any $k^{\prime}=1$ is the prime limit of periodicity. If
\[
\widehat{x}(4)=\widehat{x}(5)=\widehat{x}(7)=\widehat{x}(8)=...=0
\]
then $3$ is its prime period and if%
\[
\widehat{x}(4)=\widehat{x}(5)=\widehat{x}(7)=\widehat{x}(8)=...=1
\]
then $1$ is its prime period.
\end{example}

\section{Support sets vs sets of periods}

\begin{remark}
Let $x,y\in S^{(n)}$ be two signals and $\mu\in\omega(x)\cap\omega(y).$ One
might be tempted to think that implications of the kind%
\begin{equation}
\mathbf{T}_{\mu}^{x}=\mathbf{T}_{\mu}^{y}\Longrightarrow P_{\mu}^{x}=P_{\mu
}^{y}, \label{pre27}%
\end{equation}%
\begin{equation}
P_{\mu}^{x}=P_{\mu}^{y}\Longrightarrow\mathbf{T}_{\mu}^{x}=\mathbf{T}_{\mu
}^{y} \label{pre28}%
\end{equation}
hold and the purpose of this Section is that of understanding them better. We
give real time examples, keeping in mind that the same statements hold in
discrete time too.
\end{remark}

\begin{example}
\label{Exa14}We suppose that $Or(x)=Or(y)=\{\mu,\mu^{\prime},\mu^{\prime
\prime}\}$ and let%
\[
x(t)=\mu^{\prime}\cdot\chi_{(-\infty,2)}(t)\oplus\mu^{\prime\prime}\cdot
\chi_{\lbrack2,3)}(t)\oplus\mu\cdot\chi_{\lbrack3,4)}(t)\oplus\mu
^{\prime\prime}\cdot\chi_{\lbrack4,6)}(t)
\]%
\[
\oplus\mu\cdot\chi_{\lbrack6,7)}(t)\oplus\mu^{\prime\prime}\cdot\chi
_{\lbrack7,9)}(t)\oplus\mu\cdot\chi_{\lbrack9,10)}(t)\oplus...
\]%
\[
y(t)=\mu^{\prime}\cdot\chi_{(-\infty,0)}(t)\oplus\mu^{\prime\prime}\cdot
\chi_{\lbrack0,3)}(t)\oplus\mu\cdot\chi_{\lbrack3,4)}(t)\oplus\mu
^{\prime\prime}\cdot\chi_{\lbrack4,6)}(t)
\]%
\[
\oplus\mu\cdot\chi_{\lbrack6,7)}(t)\oplus\mu^{\prime\prime}\cdot\chi
_{\lbrack7,9)}(t)\oplus\mu\cdot\chi_{\lbrack9,10)}(t)\oplus...
\]
We see that $I^{x}=(-\infty,2),$ $I^{y}=(-\infty,0),$ $\mathbf{T}_{\mu}%
^{x}=\mathbf{T}_{\mu}^{y}=[3,4)\cup\lbrack6,7)\cup\lbrack9,10)\cup...$,
$P_{\mu}^{x}=P_{\mu}^{y}=\{3,6,9,...\}$ and $L_{\mu}^{x}=L_{\mu}^{y}%
=[1,\infty).$ The fact that $\mu$ is a periodic point of $x$ is expressed by
the non-empty intersection $I^{x}\cap L_{\mu}^{x}=[1,2)$ and the fact that
$\mu$ is an eventually periodic point of $y$ only follows from $I^{y}\cap
L_{\mu}^{y}=\varnothing.$ The interpretation of (\ref{pre27}) according to
this Example is: the implication $\mathbf{T}_{\mu}^{x}=\mathbf{T}_{\mu}%
^{y}\Longrightarrow P_{\mu}^{x}=P_{\mu}^{y}$ takes place, however $\mu$ may be
a periodic point of $x$ and an eventually periodic point of $y$.
\end{example}

\begin{example}
\label{Exa15}We take%
\[
\mathbf{T}_{\mu}^{x}=(-\infty,2)\cup\lbrack4,5)\cup\lbrack9,10)\cup
\lbrack14,15)\cup...
\]%
\[
\mathbf{T}_{\mu}^{y}=(-\infty,1)\cup\lbrack2,3)\cup\lbrack4,5)\cup
\lbrack7,8)\cup\lbrack9,10)\cup\lbrack12,13)\cup...
\]
$\mu$ is an eventually periodic point of both $x,y$ with $P_{\mu}^{x}=P_{\mu
}^{y}=\{5,10,15,...\}$ and $L_{\mu}^{x}=[2,\infty),$ $L_{\mu}^{y}=[1,\infty).$
The difference between the two signals $x,y$ consists in the fact that in
$\mathbf{T}_{\mu}^{x}$ the interval $[4,5)$ repeats within a period and in
$\mathbf{T}_{\mu}^{y}$ the intervals $[2,3),[4,5)$ repeat within a period. The
periods $T$ coincide for $x$ and $y$ and (\ref{pre28}) is false.
\end{example}

\section{Sums, differences and multiples of periods}

\begin{theorem}
\label{The68}The signals $\widehat{x},x$ are considered.

a) Let $p,p^{\prime}\geq1,k^{\prime}\in\mathbf{N}_{\_}$,$\mu\in\widehat
{\omega}(\widehat{x})$ and we ask that%
\begin{equation}
\forall k\in\widehat{\mathbf{T}}_{\mu}^{\widehat{x}}\cap\{k^{\prime}%
,k^{\prime}+1,k^{\prime}+2,...\},\{k+zp|z\in\mathbf{Z}\}\cap\{k^{\prime
},k^{\prime}+1,k^{\prime}+2,...\}\subset\widehat{\mathbf{T}}_{\mu}%
^{\widehat{x}}, \label{p228}%
\end{equation}%
\begin{equation}
\forall k\in\widehat{\mathbf{T}}_{\mu}^{\widehat{x}}\cap\{k^{\prime}%
,k^{\prime}+1,k^{\prime}+2,...\},\{k+zp^{\prime}|z\in\mathbf{Z}\}\cap
\{k^{\prime},k^{\prime}+1,k^{\prime}+2,...\}\subset\widehat{\mathbf{T}}_{\mu
}^{\widehat{x}} \label{p229}%
\end{equation}
hold. We have $p+p^{\prime}\geq1,$%
\begin{equation}
\left\{
\begin{array}
[c]{c}%
\forall k\in\widehat{\mathbf{T}}_{\mu}^{\widehat{x}}\cap\{k^{\prime}%
,k^{\prime}+1,k^{\prime}+2,...\},\\
\{k+z(p+p^{\prime})|z\in\mathbf{Z}\}\cap\{k^{\prime},k^{\prime}+1,k^{\prime
}+2,...\}\subset\widehat{\mathbf{T}}_{\mu}^{\widehat{x}}%
\end{array}
\right.  \label{p230}%
\end{equation}
and if $p>p^{\prime},$ then $p-p^{\prime}\geq1,$%
\begin{equation}
\left\{
\begin{array}
[c]{c}%
\forall k\in\widehat{\mathbf{T}}_{\mu}^{\widehat{x}}\cap\{k^{\prime}%
,k^{\prime}+1,k^{\prime}+2,...\},\\
\{k+z(p-p^{\prime})|z\in\mathbf{Z}\}\cap\{k^{\prime},k^{\prime}+1,k^{\prime
}+2,...\}\subset\widehat{\mathbf{T}}_{\mu}^{\widehat{x}}%
\end{array}
\right.  \label{p231}%
\end{equation}
hold.

b) Let $T,T^{\prime}>0,$ $t^{\prime}\in\mathbf{R,}$ $\mu\in\omega(x)$ be
arbitrary with%
\begin{equation}
\forall t\in\mathbf{T}_{\mu}^{x}\cap\lbrack t^{\prime},\infty),\{t+zT|z\in
\mathbf{Z}\}\cap\lbrack t^{\prime},\infty)\subset\mathbf{T}_{\mu}^{x},
\label{p232}%
\end{equation}%
\begin{equation}
\forall t\in\mathbf{T}_{\mu}^{x}\cap\lbrack t^{\prime},\infty),\{t+zT^{\prime
}|z\in\mathbf{Z}\}\cap\lbrack t^{\prime},\infty)\subset\mathbf{T}_{\mu}^{x}
\label{p233}%
\end{equation}
fulfilled. We have on one hand that $T+T^{\prime}>0$ and%
\begin{equation}
\forall t\in\mathbf{T}_{\mu}^{x}\cap\lbrack t^{\prime},\infty
),\{t+z(T+T^{\prime})|z\in\mathbf{Z}\}\cap\lbrack t^{\prime},\infty
)\subset\mathbf{T}_{\mu}^{x} \label{p234}%
\end{equation}
are true and on the other hand that $T>T^{\prime}$ implies $T-T^{\prime}>0$
and%
\begin{equation}
\forall t\in\mathbf{T}_{\mu}^{x}\cap\lbrack t^{\prime},\infty
),\{t+z(T-T^{\prime})|z\in\mathbf{Z}\}\cap\lbrack t^{\prime},\infty
)\subset\mathbf{T}_{\mu}^{x}. \label{p235}%
\end{equation}

\end{theorem}

\begin{proof}
a) We prove the second implication. We take some arbitrary, fixed
$k\in\widehat{\mathbf{T}}_{\mu}^{\widehat{x}},$ $z\in\mathbf{Z}$ such that
$k\geq k^{\prime},k+z(p-p^{\prime})\geq k^{\prime}$ and we have the following possibilities:

Case $z<0$

We obtain in succession $k-zp^{\prime}\geq k^{\prime},$ $k-zp^{\prime}%
\overset{(\ref{p229})}{\in}\widehat{\mathbf{T}}_{\mu}^{\widehat{x}},$
$k-zp^{\prime}+zp\overset{hyp}{\geq}k^{\prime},$ $k+z(p-p^{\prime}%
)\overset{(\ref{p228})}{\in}\widehat{\mathbf{T}}_{\mu}^{\widehat{x}}.$

Case $z=0$

$k=k+z(p-p^{\prime})\in\widehat{\mathbf{T}}_{\mu}^{\widehat{x}}$ trivially.

Case $z>0$

We have $k+zp\geq k^{\prime},$ $k+zp\overset{(\ref{p228})}{\in}\widehat
{\mathbf{T}}_{\mu}^{\widehat{x}},$ $k+zp-zp^{\prime}\overset{hyp}{\geq
}k^{\prime},$ $k+z(p-p^{\prime})\overset{(\ref{p229})}{\in}\widehat
{\mathbf{T}}_{\mu}^{\widehat{x}}.$

b) We prove the first implication and let $t\in\mathbf{T}_{\mu}^{x}\cap\lbrack
t^{\prime},\infty),$ $z\in\mathbf{Z}$ be arbitrary, fixed such that
$t+z(T+T^{\prime})\geq t^{\prime}.$

Case $z<0$

We have in succession $t+zT\geq t+z(T+T^{\prime})\overset{hyp}{\geq}t^{\prime
},$ $t+zT\overset{(\ref{p232})}{\in}\mathbf{T}_{\mu}^{x},$ $t+z(T+T^{\prime
})\overset{(\ref{p233})}{\in}\mathbf{T}_{\mu}^{x}.$

Case $z=0$

We infer $t=t+z(T+T^{\prime})\in\mathbf{T}_{\mu}^{x}.$

Case $z>0$

We have $t+zT\geq t\geq t^{\prime},$ $t+zT\overset{(\ref{p232})}{\in
}\mathbf{T}_{\mu}^{x},$ $t+z(T+T^{\prime})\overset{hyp}{\geq}t^{\prime},$
$t+z(T+T^{\prime})\overset{(\ref{p233})}{\in}\mathbf{T}_{\mu}^{x}.$
\end{proof}

\begin{theorem}
\label{The25}a) Let $p,k_{1}\geq1,$ $k^{\prime}\in\mathbf{N}_{\_}$ and $\mu
\in\widehat{\omega}(\widehat{x}).$ Then $p^{\prime}=k_{1}p$ fulfills
$p^{\prime}\geq1$ and%
\begin{equation}
\forall k\in\widehat{\mathbf{T}}_{\mu}^{\widehat{x}}\cap\{k^{\prime}%
,k^{\prime}+1,k^{\prime}+2,...\},\{k+zp|z\in\mathbf{Z}\}\cap\{k^{\prime
},k^{\prime}+1,k^{\prime}+2,...\}\subset\widehat{\mathbf{T}}_{\mu}%
^{\widehat{x}} \label{per609}%
\end{equation}
implies%
\begin{equation}
\forall k\in\widehat{\mathbf{T}}_{\mu}^{\widehat{x}}\cap\{k^{\prime}%
,k^{\prime}+1,k^{\prime}+2,...\},\{k+zp^{\prime}|z\in\mathbf{Z}\}\cap
\{k^{\prime},k^{\prime}+1,k^{\prime}+2,...\}\subset\widehat{\mathbf{T}}_{\mu
}^{\widehat{x}}. \label{per610}%
\end{equation}

b) Let $T>0,t^{\prime}\in\mathbf{R,}$ $k_{1}\geq1$ and $\mu\in\omega(x) $ be
arbitrary. Then $T^{\prime}=k_{1}T$ fulfills $T^{\prime}>0$ and%
\begin{equation}
\forall t\in\mathbf{T}_{\mu}^{x}\cap\lbrack t^{\prime},\infty),\{t+zT|z\in
\mathbf{Z}\}\cap\lbrack t^{\prime},\infty)\subset\mathbf{T}_{\mu}^{x}
\label{per611}%
\end{equation}
implies%
\begin{equation}
\forall t\in\mathbf{T}_{\mu}^{x}\cap\lbrack t^{\prime},\infty),\{t+zT^{\prime
}|z\in\mathbf{Z}\}\cap\lbrack t^{\prime},\infty)\subset\mathbf{T}_{\mu}^{x}.
\label{per612}%
\end{equation}

\end{theorem}

\begin{proof}
This is a consequence of Theorem \ref{The68}.
\end{proof}

\begin{corollary}
\label{Cor6}a) For any $\widehat{x},\mu\in\widehat{\omega}(\widehat{x})$ and
$p\geq1,$ if $p\in\widehat{P}_{\mu}^{\widehat{x}},$ then $\{p,2p,3p,...\}$
$\subset\widehat{P}_{\mu}^{\widehat{x}};$

b) for any $x,\mu\in\omega(x)$ and $T>0,$ $T\in P_{\mu}^{x}$ implies
$\{T,2T,3T,...\}\subset P_{\mu}^{x}.$
\end{corollary}

\begin{proof}
The Corollary is a direct consequence of Theorem \ref{The25}.
\end{proof}

\section{The set of the periods}

\begin{theorem}
\label{The70}a) Let $\widehat{x}\in\widehat{S}^{(n)}$ and $\mu\in
\widehat{\omega}(\widehat{x}).$ We ask that $\mu$ is an eventually periodic
point of $\widehat{x}$.~Then $\widetilde{p}\geq1$ exists such that%
\[
\widehat{P}_{\mu}^{\widehat{x}}=\{\widetilde{p},2\widetilde{p},3\widetilde
{p},...\}.
\]

b) We suppose that the signal $x\in S^{(n)}$ is not eventually constant and
let $\mu\in\omega(x).$ We ask that $\mu~$is an eventually periodic point of
$x$. Then $\widetilde{T}>0$ exists such that%
\[
P_{\mu}^{x}=\{\widetilde{T},2\widetilde{T},3\widetilde{T},...\}.
\]

\end{theorem}

\begin{proof}
a) We denote with $\widetilde{p}$ the least element of $\widehat{P}_{\mu
}^{\widehat{x}}.$ From Corollary \ref{Cor6}, page \pageref{Cor6} we have the
inclusion $\{\widetilde{p},2\widetilde{p},3\widetilde{p},...\}\subset
\widehat{P}_{\mu}^{\widehat{x}}.$ We show that $\widehat{P}_{\mu}^{\widehat
{x}}\subset\{\widetilde{p},2\widetilde{p},3\widetilde{p},...\}.$ We presume
against all reason that this is not true, i.e. that some $p^{\prime}%
\in\widehat{P}_{\mu}^{\widehat{x}}\smallsetminus\{\widetilde{p},2\widetilde
{p},3\widetilde{p},...\}$ exists. In these circumstances we have the existence
of $k_{1}\geq1$ with $k_{1}\widetilde{p}<p^{\prime}<(k_{1}+1)\widetilde{p}.$
We infer that $1\leq p^{\prime}-k_{1}\widetilde{p}<\widetilde{p}$ and, from
Theorems \ref{The68}, \ref{The25}, page \pageref{The68} we conclude that
$p^{\prime}-k_{1}\widetilde{p}\in\widehat{P}_{\mu}^{\widehat{x}}.$ We have
obtained a contradiction with the fact that $\widetilde{p}$ is the least
element of $\widehat{P}_{\mu}^{\widehat{x}}.$

b) The proof is made in two steps.

b.1) We show first that $\min P_{\mu}^{x}$ exists. We suppose against all
reason that this is not true, namely that a strictly decreasing sequence
$T_{k}\in P_{\mu}^{x},k\in\mathbf{N}$ exists that is convergent to $T=\inf
P_{\mu}^{x}.$ As $x$ is not eventually constant, the following property is
true:%
\begin{equation}
\forall t\in\mathbf{R},\exists t^{\prime\prime}>t,x(t^{\prime\prime}-0)\neq
x(t^{\prime\prime})=\mu, \label{p4_}%
\end{equation}
see Lemma \ref{Lem38}, page \pageref{Lem38}. The hypothesis states the
existence $\forall k\in\mathbf{N},$ of $t_{k}^{\prime}\in\mathbf{R}$ with%
\begin{equation}
\forall t\in\mathbf{T}_{\mu}^{x}\cap\lbrack t_{k}^{\prime},\infty
),\{t+zT_{k}|z\in\mathbf{Z}\}\cap\lbrack t_{k}^{\prime},\infty)\subset
\mathbf{T}_{\mu}^{x}.
\end{equation}
We can suppose, as $t_{k}^{\prime}$ do not depend on $T_{k},$ that they have a
common value $t^{\prime}.$ From (\ref{p4_}) we infer that we can take some
$t^{\prime\prime}>t^{\prime}$ with $x(t^{\prime\prime}-0)\neq x(t^{\prime
\prime})=\mu$ and, since $\mu\in\omega(x),$ we can apply Lemma \ref{Lem10},
page \pageref{Lem10} stating%
\begin{equation}
\forall k\in\mathbf{N},x(t^{\prime\prime}+T_{k}-0)\neq x(t^{\prime\prime
}+T_{k})=\mu. \label{per946}%
\end{equation}
We infer from Lemma \ref{Lem4_}, page \pageref{Lem4_} that $N\in\mathbf{N}$
exists with $\forall k\geq N,$%
\[
x(t^{\prime\prime}+T_{k}-0)=x(t^{\prime\prime}+T_{k})=x(t^{\prime\prime}+T),
\]
contradiction with (\ref{per946}). It has resulted that such a sequence
$T_{k},k\in\mathbf{N}$ does not exist, thus $P_{\mu}^{x}$ has a minimum that
we denote by $\widetilde{T}.$

b.2) The inclusion $\{\widetilde{T},2\widetilde{T},3\widetilde{T},...\}\subset
P_{\mu}^{x}$ results from Corollary \ref{Cor6}, we prove the inclusion
$P_{\mu}^{x}\subset\{\widetilde{T},2\widetilde{T},3\widetilde{T},...\}.$ We
suppose against all reason that some $T^{\prime}\in P_{\mu}^{x}\setminus
\{\widetilde{T},2\widetilde{T},3\widetilde{T},...\}$ exists and let $k_{1}%
\geq1$ with the property $T^{\prime}\in(k_{1}\widetilde{T},(k_{1}%
+1)\widetilde{T}).$ We infer that $0<T^{\prime}-k_{1}\widetilde{T}%
<\widetilde{T}$ and, from Theorems \ref{The68}, \ref{The25}, we get
$T^{\prime}-k_{1}\widetilde{T}\in P_{\mu}^{x}.$ We have obtained a
contradiction, since $\widetilde{T}$ was defined to be the minimum of $P_{\mu
}^{x}.$ $P_{\mu}^{x}=\{\widetilde{T},2\widetilde{T},3\widetilde{T},...\}$ holds.
\end{proof}

\begin{theorem}
We suppose that the relation between $\widehat{x}$ and $x$ is given by%
\[%
\begin{array}
[c]{c}%
x(t)=\widehat{x}(-1)\cdot\chi_{(-\infty,t_{0})}(t)\oplus\widehat{x}%
(0)\cdot\chi_{\lbrack t_{0},t_{0}+h)}(t)\oplus\widehat{x}(1)\cdot\chi_{\lbrack
t_{0}+h,t_{0}+2h)}(t)\oplus...\\
...\oplus\widehat{x}(k)\cdot\chi_{\lbrack t_{0}+kh,t_{0}+(k+1)h)}(t)\oplus...
\end{array}
\]
where $t_{0}\in\mathbf{R}$ and $h>0$ and that $\mu\in\widehat{\omega}%
(\widehat{x})=\omega(x)$ is an eventually periodic point of any of
$\widehat{x},x.$ Then two possibilities exist:

a) $\widehat{x},$ $x$ are both eventually constant, $\widehat{P}_{\mu
}^{\widehat{x}}=\{1,2,3,...\}$ and $P_{\mu}^{x}=(0,\infty);$

b) none of $\widehat{x},$ $x$ is eventually constant, $\min\widehat{P}_{\mu
}^{\widehat{x}}=p>1$ and $\min P_{\mu}^{x}=T=ph$.
\end{theorem}

\begin{proof}
We see that $\widehat{x},x$ are simultaneously eventually constant or not. We
suppose that they are not eventually constant and we prove b). From Theorem
\ref{The24}, page \pageref{The24} we know that $p\in\widehat{P}_{\mu
}^{\widehat{x}}\Longrightarrow T=ph\in P_{\mu}^{x}$ and conversely, $T\in
P_{\mu}^{x}\Longrightarrow p=\frac{T}{h}\in\widehat{P}_{\mu}^{\widehat{x}}.$
From Theorem \ref{The70} we get $\widehat{P}_{\mu}^{\widehat{x}}%
=\{p,2p,3p,...\}$ and $P_{\mu}^{x}=\{T,2T,3T,...\}$, thus $T=ph.$
\end{proof}

\section{Necessity conditions of eventual periodicity}

\begin{theorem}
\label{The69}Let $\widehat{x}\in\widehat{S}^{(n)}$ be not eventually constant.
For $\mu\in\widehat{\omega}(\widehat{x}),$ $p\geq1$ and $k^{\prime}%
\in\mathbf{N}_{\_}$ we suppose that%
\begin{equation}
\forall k\in\widehat{\mathbf{T}}_{\mu}^{\widehat{x}}\cap\{k^{\prime}%
,k^{\prime}+1,k^{\prime}+2,...\},\{k+zp|z\in\mathbf{Z}\}\cap\{k^{\prime
},k^{\prime}+1,k^{\prime}+2,...\}\subset\widehat{\mathbf{T}}_{\mu}%
^{\widehat{x}} \label{pre164}%
\end{equation}
holds. Then $n_{1},n_{2},...,n_{k_{1}}\in\{k^{\prime},k^{\prime}%
+1,...,k^{\prime}+p-1\},k_{1}\geq1$ exist such that%
\begin{equation}
\widehat{\mathbf{T}}_{\mu}^{\widehat{x}}\cap\{k^{\prime},k^{\prime
}+1,k^{\prime}+2,...\}=\underset{k\in\mathbf{N}}{%
{\displaystyle\bigcup}
}\{n_{1}+kp,n_{2}+kp,...,n_{k_{1}}+kp\}. \label{pre166}%
\end{equation}

\end{theorem}

\begin{proof}
We apply Theorem \ref{Lem1}, page \pageref{Lem1} written for $k=k^{\prime}$
and we obtain that $\widehat{\mathbf{T}}_{\mu}^{\widehat{x}}\cap\{k^{\prime
},k^{\prime}+1,...,k^{\prime}+p-1\}\neq\varnothing,$ wherefrom we have the
existence of $n_{1},n_{2},...,n_{k_{1}},k_{1}\geq1$ with%
\begin{equation}
\widehat{\mathbf{T}}_{\mu}^{\widehat{x}}\cap\{k^{\prime},k^{\prime
}+1,...,k^{\prime}+p-1\}=\{n_{1},n_{2},...,n_{k_{1}}\}. \label{p36}%
\end{equation}

We prove $\widehat{\mathbf{T}}_{\mu}^{\widehat{x}}\cap\{k^{\prime},k^{\prime
}+1,k^{\prime}+2,...\}\subset\underset{k\in\mathbf{N}}{%
{\displaystyle\bigcup}
}\{n_{1}+kp,n_{2}+kp,...,n_{k_{1}}+kp\}$ and let $k^{\prime\prime}\in
\widehat{\mathbf{T}}_{\mu}^{\widehat{x}}\cap\{k^{\prime},k^{\prime
}+1,k^{\prime}+2,...\}$ arbitrary. We get from (\ref{pre164}) the existence of
a finite sequence $k^{\prime\prime},k^{\prime\prime}-p,...,k^{\prime\prime
}-\overline{k}p\in\widehat{\mathbf{T}}_{\mu}^{\widehat{x}},\overline{k}%
\in\mathbf{N}$ with the property that $k^{\prime\prime}-\overline{k}%
p\in\{k^{\prime},k^{\prime}+1,...,k^{\prime}+p-1\},$ thus we have from
(\ref{p36}) the existence of $j\in\{1,...,k_{1}\}$ with $k^{\prime\prime
}-\overline{k}p=n_{j}.$ This means that $k^{\prime\prime}=n_{j}+\overline
{k}p\in\underset{k\in\mathbf{N}}{%
{\displaystyle\bigcup}
}\{n_{1}+kp,n_{2}+kp,...,n_{k_{1}}+kp\}.$

We prove that $\underset{k\in\mathbf{N}}{%
{\displaystyle\bigcup}
}\{n_{1}+kp,n_{2}+kp,...,n_{k_{1}}+kp\}\subset\widehat{\mathbf{T}}_{\mu
}^{\widehat{x}}\cap\{k^{\prime},k^{\prime}+1,k^{\prime}+2,...\}.$ Let
$k^{\prime\prime}\in\underset{k\in\mathbf{N}}{%
{\displaystyle\bigcup}
}\{n_{1}+kp,n_{2}+kp,...,n_{k_{1}}+kp\}$ arbitrary, thus $j\in\{1,...,k_{1}\}
$ and $k\in\mathbf{N}$ exist such that $k^{\prime\prime}=n_{j}+kp.$ As
$n_{j}\in\widehat{\mathbf{T}}_{\mu}^{\widehat{x}}\cap\{k^{\prime},k^{\prime
}+1,k^{\prime}+2,...\},$ we have $n_{j}+kp\geq k^{\prime}$ thus we can apply
(\ref{pre164}) and we get $k^{\prime\prime}\in\widehat{\mathbf{T}}_{\mu
}^{\widehat{x}}.$
\end{proof}

\begin{remark}
The hypothesis of the previous Theorem avoids the situation when $\widehat{x}
$ is eventually constant. In that case $\widehat{\omega}(\widehat{x}%
)=\{\mu\},p=1,k_{1}=1,n_{1}=k^{\prime}$ and (\ref{pre166}) takes the form
$\widehat{\mathbf{T}}_{\mu}^{\widehat{x}}\cap\{k^{\prime},k^{\prime
}+1,k^{\prime}+2,...\}=\underset{k\in\mathbf{N}}{%
{\displaystyle\bigcup}
}\{k^{\prime}+k\}=\{k^{\prime},k^{\prime}+1,k^{\prime}+2,...\}.$
\end{remark}

\begin{remark}
Lemma \ref{Lem35}, page \pageref{Lem35} shows that we can replace (\ref{p36})
and (\ref{pre166}) with%
\[
\widehat{\mathbf{T}}_{\mu}^{\widehat{x}}\cap\{k^{\prime\prime},k^{\prime
\prime}+1,...,k^{\prime\prime}+p-1\}=\{n_{1}^{\prime},n_{2}^{\prime
},...,n_{k_{1}}^{\prime}\},
\]%
\[%
\begin{array}
[c]{c}%
\widehat{\mathbf{T}}_{\mu}^{\widehat{x}}\cap\{k^{\prime},k^{\prime
}+1,k^{\prime}+2,...\}=\underset{k\in\mathbf{N}}{%
{\displaystyle\bigcup}
}\{n_{1}+kp,n_{2}+kp,...,n_{k_{1}}+kp\}\\
=\underset{z\in\mathbf{Z}}{%
{\displaystyle\bigcup}
}\{n_{1}^{\prime}+zp,n_{2}^{\prime}+zp,...,n_{k_{1}}^{\prime}+zp\}\cap
\{k^{\prime},k^{\prime}+1,k^{\prime}+2,...\}\\
\supset\underset{k\in\mathbf{N}}{%
{\displaystyle\bigcup}
}\{n_{1}^{\prime}+kp,n_{2}^{\prime}+kp,...,n_{k_{1}}^{\prime}+kp\}=\widehat
{\mathbf{T}}_{\mu}^{\widehat{x}}\cap\{k^{\prime\prime},k^{\prime\prime
}+1,k^{\prime\prime}+2,...\},
\end{array}
\]
where $k^{\prime\prime}\geq k^{\prime}$ is arbitrary.
\end{remark}

\begin{theorem}
\label{The71}The signal $x\in S^{(n)}$ is not eventually constant and let the
point $\mu\in\omega(x),$ as well as $T>0,t^{\prime}\in\mathbf{R}$ with%
\begin{equation}
\forall t\in\mathbf{T}_{\mu}^{x}\cap\lbrack t^{\prime},\infty),\{t+zT|z\in
\mathbf{Z}\}\cap\lbrack t^{\prime},\infty)\subset\mathbf{T}_{\mu}^{x}.
\label{pre174}%
\end{equation}
Then $a_{1},b_{1},a_{2},b_{2},...,a_{k_{1}},b_{k_{1}}\in\mathbf{R,}$
$k_{1}\geq1$ exist such that%
\begin{equation}
t^{\prime}\leq a_{1}<b_{1}<a_{2}<b_{2}<...<a_{k_{1}}<b_{k_{1}}\leq t^{\prime
}+T, \label{p50}%
\end{equation}%
\begin{equation}%
\begin{array}
[c]{c}%
\mathbf{T}_{\mu}^{x}\cap\lbrack t^{\prime},\infty)=\\
\underset{k\in\mathbf{N}}{%
{\displaystyle\bigcup}
}([a_{1}+kT,b_{1}+kT)\cup\lbrack a_{2}+kT,b_{2}+kT)\cup...\cup\lbrack
a_{k_{1}}+kT,b_{k_{1}}+kT))
\end{array}
\label{pre178}%
\end{equation}
hold.
\end{theorem}

\begin{proof}
We define the intervals $[a_{1},b_{1}),[a_{2},b_{2}),...,[a_{k_{1}},b_{k_{1}%
})$ such that (\ref{p50}) and%
\begin{equation}
\mathbf{T}_{\mu}^{x}\cap\lbrack t^{\prime},t^{\prime}+T)=[a_{1},b_{1}%
)\cup\lbrack a_{2},b_{2})\cup...\cup\lbrack a_{k_{1}},b_{k_{1}}) \label{p71}%
\end{equation}
are fulfilled, by taking into account (\ref{pre174}) and Theorem \ref{Lem1},
page \pageref{Lem1}, written for $t=t^{\prime}$.

We prove $\mathbf{T}_{\mu}^{x}\cap\lbrack t^{\prime},\infty)\subset
\underset{k\in\mathbf{N}}{%
{\displaystyle\bigcup}
}([a_{1}+kT,b_{1}+kT)\cup\lbrack a_{2}+kT,b_{2}+kT)\cup...\cup\lbrack
a_{k_{1}}+kT,b_{k_{1}}+kT))$ and let $t\in\mathbf{T}_{\mu}^{x}\cap\lbrack
t^{\prime},\infty)$ arbitrary. A finite sequence $t,t-T,t-2T,...,t-\overline
{k}T\in\mathbf{T}_{\mu}^{x}$ exists, from (\ref{pre174}), such that
$t-\overline{k}T\in\lbrack t^{\prime},t^{\prime}+T),$ where $\overline{k}%
\in\mathbf{N}.$ This implies the existence of $j\in\{1,...,k_{1}\}$ such that
$t-\overline{k}T\in\lbrack a_{j},b_{j}),$ i.e. $t\in\lbrack a_{j}+\overline
{k}T,b_{j}+\overline{k}T)\subset\underset{k\in\mathbf{N}}{%
{\displaystyle\bigcup}
}([a_{1}+kT,b_{1}+kT)\cup\lbrack a_{2}+kT,b_{2}+kT)\cup...\cup\lbrack
a_{k_{1}}+kT,b_{k_{1}}+kT)).$

We prove $\underset{k\in\mathbf{N}}{%
{\displaystyle\bigcup}
}([a_{1}+kT,b_{1}+kT)\cup\lbrack a_{2}+kT,b_{2}+kT)\cup...\cup\lbrack
a_{k_{1}}+kT,b_{k_{1}}+kT))\subset\mathbf{T}_{\mu}^{x}\cap\lbrack t^{\prime
},\infty)$ and let $t\in\underset{k\in\mathbf{N}}{%
{\displaystyle\bigcup}
}([a_{1}+kT,b_{1}+kT)\cup\lbrack a_{2}+kT,b_{2}+kT)\cup...\cup\lbrack
a_{k_{1}}+kT,b_{k_{1}}+kT))$ arbitrary$.$ Some $j\in\{1,...,k_{1}\}$ and some
$k\in\mathbf{N}$ exist such that $t\in\lbrack a_{j}+kT,b_{j}+kT),$ wherefrom
$t-kT\in\lbrack a_{j},b_{j})\subset\mathbf{T}_{\mu}^{x}\cap\lbrack t^{\prime
},\infty).$ We can write%
\[
t\in\{t-kT+zT|z\in\mathbf{Z}\}\cap\lbrack t^{\prime},\infty)\overset
{(\ref{pre174})}{\subset}\mathbf{T}_{\mu}^{x}.
\]
Since $t\geq t^{\prime},$ (\ref{pre178}) is proved.
\end{proof}

\begin{remark}
Let us see what happens if in the hypothesis of the previous Theorem $x$ would
have been eventually constant; in this case $\omega(x)=\{\mu\},k_{1}%
=1,[a_{1},b_{1})=[t^{\prime},t^{\prime}+T)$ and (\ref{pre178}) becomes
$\mathbf{T}_{\mu}^{x}\cap\lbrack t^{\prime},\infty)=[t^{\prime},\infty).$
\end{remark}

\begin{remark}
Let $t^{\prime\prime}\geq t^{\prime}$ arbitrary. We get from Lemma
\ref{Lem35}, page \pageref{Lem35} that we can replace (\ref{p71}) and
(\ref{pre178}) with%
\[
\mathbf{T}_{\mu}^{x}\cap\lbrack t^{\prime\prime},t^{\prime\prime}%
+T)=[a_{1}^{\prime},b_{1}^{\prime})\cup\lbrack a_{2}^{\prime},b_{2}^{\prime
})\cup...\cup\lbrack a_{p_{1}}^{\prime},b_{p_{1}}^{\prime})
\]
and%
\[
\mathbf{T}_{\mu}^{x}\cap\lbrack t^{\prime},\infty)=\underset{k\in\mathbf{N}}{%
{\displaystyle\bigcup}
}([a_{1}+kT,b_{1}+kT)\cup\lbrack a_{2}+kT,b_{2}+kT)\cup...\cup\lbrack
a_{k_{1}}+kT,b_{k_{1}}+kT))
\]%
\[
=\underset{z\in\mathbf{Z}}{%
{\displaystyle\bigcup}
}([a_{1}^{\prime}+zT,b_{1}^{\prime}+zT)\cup\lbrack a_{2}^{\prime}%
+zT,b_{2}^{\prime}+zT)\cup...\cup\lbrack a_{p_{1}}^{\prime}+zT,b_{p_{1}%
}^{\prime}+zT))\cap\lbrack t^{\prime},\infty)
\]%
\[
\supset\underset{k\in\mathbf{N}}{%
{\displaystyle\bigcup}
}([a_{1}^{\prime}+kT,b_{1}^{\prime}+kT)\cup\lbrack a_{2}^{\prime}%
+kT,b_{2}^{\prime}+kT)\cup...\cup\lbrack a_{p_{1}}^{\prime}+kT,b_{p_{1}%
}^{\prime}+kT))=\mathbf{T}_{\mu}^{x}\cap\lbrack t^{\prime\prime},\infty).
\]

\end{remark}

\section{Sufficiency conditions of eventual periodicity}

\begin{theorem}
\label{The72}Let $\widehat{x}\in\widehat{S}^{(n)}$, $\mu\in\widehat{\omega
}(\widehat{x}),$ $p\geq1,k^{\prime}\in\mathbf{N}_{\_}$ and $n_{1}%
,n_{2},...,n_{k_{1}}\in\{k^{\prime},k^{\prime}+1,...,k^{\prime}+p-1\}$,
$k_{1}\geq1$ such that%
\begin{equation}
\widehat{\mathbf{T}}_{\mu}^{\widehat{x}}\cap\{k^{\prime},k^{\prime
}+1,k^{\prime}+2,...\}=\underset{k\in\mathbf{N}}{%
{\displaystyle\bigcup}
}\{n_{1}+kp,n_{2}+kp,...,n_{k_{1}}+kp\}. \label{pre180}%
\end{equation}
In such circumstances%
\begin{equation}
\forall k\in\widehat{\mathbf{T}}_{\mu}^{\widehat{x}}\cap\{k^{\prime}%
,k^{\prime}+1,k^{\prime}+2,...\},\{k+zp|z\in\mathbf{Z}\}\cap\{k^{\prime
},k^{\prime}+1,k^{\prime}+2,...\}\subset\widehat{\mathbf{T}}_{\mu}%
^{\widehat{x}}. \label{pre181}%
\end{equation}

\end{theorem}

\begin{proof}
Let $k^{\prime\prime}\in\widehat{\mathbf{T}}_{\mu}^{\widehat{x}}%
\cap\{k^{\prime},k^{\prime}+1,k^{\prime}+2,...\}$ and $z_{1}\in\mathbf{Z}$
arbitrary such that $k^{\prime\prime}+z_{1}p\geq k^{\prime}.$ Then
$j\in\{1,...,k_{1}\}$ and $\overline{k}\in\mathbf{N}$ exist with
$k^{\prime\prime}=n_{j}+\overline{k}p$ and we have $k^{\prime\prime}%
+z_{1}p=n_{j}+(z_{1}+\overline{k})p\geq k^{\prime}.$ This means the existence
of $z^{\prime}\in Z$ with $n_{j}+z^{\prime}p\geq k^{\prime}$ and, as
$n_{j}-p\leq k^{\prime}-1,$ we get $z^{\prime}\geq0.$ In this situation
$n_{j}+z^{\prime}p\overset{(\ref{pre180})}{\in}\widehat{\mathbf{T}}_{\mu
}^{\widehat{x}}\cap\{k^{\prime},k^{\prime}+1,k^{\prime}+2,...\},$ thus
(\ref{pre181}) holds.
\end{proof}

\begin{theorem}
\label{The76}Let $x,$ $\mu\in\omega(x),$ $T>0,t^{\prime}\in\mathbf{R}$ and the
numbers $a_{1},b_{1},a_{2},b_{2},...,$ $a_{k_{1}},b_{k_{1}}\in\mathbf{R},$
$k_{1}\geq1$ such that%
\begin{equation}
t^{\prime}\leq a_{1}<b_{1}<a_{2}<b_{2}<...<a_{k_{1}}<b_{k_{1}}\leq t^{\prime
}+T,
\end{equation}%
\begin{equation}%
\begin{array}
[c]{c}%
\mathbf{T}_{\mu}^{x}\cap\lbrack t^{\prime},\infty)=\\
\underset{k\in\mathbf{N}}{%
{\displaystyle\bigcup}
}([a_{1}+kT,b_{1}+kT)\cup\lbrack a_{2}+kT,b_{2}+kT)\cup...\cup\lbrack
a_{k_{1}}+kT,b_{k_{1}}+kT))
\end{array}
\label{p48}%
\end{equation}
hold. We infer%
\begin{equation}
\forall t\in\mathbf{T}_{\mu}^{x}\cap\lbrack t^{\prime},\infty),\{t+zT|z\in
\mathbf{Z}\}\cap\lbrack t^{\prime},\infty)\subset\mathbf{T}_{\mu}^{x}.
\label{pre191}%
\end{equation}

\end{theorem}

\begin{proof}
Let $t^{\prime\prime}\in\mathbf{T}_{\mu}^{x}\cap\lbrack t^{\prime},\infty)$
and $z_{1}\in\mathbf{Z}$ arbitrary with $t^{\prime\prime}+z_{1}T\geq
t^{\prime}.$ From (\ref{p48}) we have the existence of $j\in\{1,...,k_{1}\}$
and $\overline{k}\in\mathbf{N}$ with $t^{\prime\prime}\in\lbrack
a_{j}+\overline{k}T,b_{j}+\overline{k}T).$ We obtain that $t^{\prime\prime
}+z_{1}T\in\lbrack a_{j}+(z_{1}+\overline{k})T,b_{j}+(z_{1}+\overline
{k})T)\subset\lbrack t^{\prime},\infty).$ We get the existence of $z^{\prime
}\in\mathbf{Z}$ with $t^{\prime\prime}+z_{1}T\in\lbrack a_{j}+z^{\prime
}T,b_{j}+z^{\prime}T)\subset\lbrack t^{\prime},\infty)$ and, since
$b_{j}-T<t^{\prime},$ we infer $z^{\prime}\geq0.$ We have obtained that
$t^{\prime\prime}+z_{1}T\overset{(\ref{p48})}{\in}\mathbf{T}_{\mu}^{x}%
\cap\lbrack t^{\prime},\infty),$ thus (\ref{pre191}) holds.
\end{proof}

\section{A special case}

\begin{theorem}
\label{The122}Let $\widehat{x}\in\widehat{S}^{(n)}$, $\mu\in\widehat{\omega
}(\widehat{x}),$ $p\geq1,k^{\prime}\in\mathbf{N}_{\_}$ and $n_{1}%
\in\{k^{\prime},k^{\prime}+1,...,k^{\prime}+p-1\}$ such that%
\begin{equation}
\widehat{\mathbf{T}}_{\mu}^{\widehat{x}}\cap\{k^{\prime},k^{\prime
}+1,k^{\prime}+2,...\}=\{n_{1},n_{1}+p,n_{1}+2p,n_{1}+3p,...\}. \label{p14}%
\end{equation}
Then

a) $\mu$ is an eventually periodic point of $\widehat{x}$ with the period $p:$%
\begin{equation}
\forall k\in\widehat{\mathbf{T}}_{\mu}^{\widehat{x}}\cap\{k^{\prime}%
,k^{\prime}+1,k^{\prime}+2,...\},\{k+zp|z\in\mathbf{Z}\}\cap\{k^{\prime
},k^{\prime}+1,k^{\prime}+2,...\}\subset\widehat{\mathbf{T}}_{\mu}%
^{\widehat{x}}; \label{p72}%
\end{equation}

b) $p$ is the prime period of $\mu:$
\begin{equation}
\widehat{P}_{\mu}^{\widehat{x}}=\{p,2p,3p,...\}.
\end{equation}

\end{theorem}

\begin{proof}
a) This is a special case of Theorem \ref{The72}, page \pageref{The72},
written for $k_{1}=1.$

b) We suppose against all reason that $p^{\prime}\in\widehat{P}_{\mu
}^{\widehat{x}}$ exists with $p^{\prime}<p.$ As $n_{1}\in\widehat{\mathbf{T}%
}_{\mu}^{\widehat{x}}\cap\{k^{\prime},k^{\prime}+1,k^{\prime}+2,...\},$ we
obtain from (\ref{p72}) that $n_{1}+p^{\prime}\in\widehat{\mathbf{T}}_{\mu
}^{\widehat{x}}\cap\{k^{\prime},k^{\prime}+1,k^{\prime}+2,...\},$
contradiction with (\ref{p14}). Thus any $p^{\prime}\in\widehat{P}_{\mu
}^{\widehat{x}}$ fulfills $p^{\prime}\geq p.$ We apply Theorem \ref{The70},
page \pageref{The70}.
\end{proof}

\begin{theorem}
\label{The116}Let $x,$ $\mu\in\omega(x),$ $T>0,$ $t^{\prime}\in\mathbf{R}$ and
the interval $[a,b)\subset\lbrack t^{\prime},t^{\prime}+T)$ such that%
\begin{equation}
\mathbf{T}_{\mu}^{x}\cap\lbrack t^{\prime},\infty)=[a,b)\cup\lbrack
a+T,b+T)\cup\lbrack a+2T,b+2T)\cup... \label{p49}%
\end{equation}
holds. We have

a) $\mu$ is an eventually periodic point of $x$ with the period $T:$
\begin{equation}
\forall t\in\mathbf{T}_{\mu}^{x}\cap\lbrack t^{\prime},\infty),\{t+zT|z\in
\mathbf{Z}\}\cap\lbrack t^{\prime},\infty)\subset\mathbf{T}_{\mu}^{x};
\end{equation}

b) if $x$ is not eventually constant, then $T$ is the prime period of $x:$%
\begin{equation}
P_{\mu}^{x}=\{T,2T,3T,...\}.
\end{equation}

\end{theorem}

\begin{proof}
a) This is a special case of Theorem \ref{The76}, page \pageref{The76},
written for $k_{1}=1.$

b) We notice first of all that $b<a+T,$ otherwise $\mathbf{T}_{\mu}^{x}%
\cap\lbrack t^{\prime},\infty)=[a,\infty)$ and $x$ is eventually constant,
representing a contradiction with the hypothesis.

Let us suppose now against all reason that $T$ is not the prime period of
$\mu,$ i.e. $T^{\prime}\in P_{\mu}^{x}$ exists with $T^{\prime}<T.$ We see
that%
\[
\max\{a,b-T^{\prime}\}<\min\{b,a+T-T^{\prime}\}
\]
holds, because $a<b,a<a+T-T^{\prime},b-T^{\prime}<b,b-T^{\prime}%
<a+T-T^{\prime}$ are all true. We take $t\in\lbrack\max\{a,b-T^{\prime}%
\},\min\{b,a+T-T^{\prime}\})$ and we have%
\[
a\leq\max\{a,b-T^{\prime}\}\leq t<\min\{b,a+T-T^{\prime}\}\leq b,
\]%
\[
b\leq\max\{a+T^{\prime},b\}\leq t+T^{\prime}<\min\{b+T^{\prime},a+T\}\leq
a+T.
\]
We have obtained%
\[
\mu=x(t)=x(t+T^{\prime})\overset{(\ref{p49})}{\neq}\mu,
\]
contradiction. We conclude that any $T^{\prime}\in P_{\mu}^{x}$ fulfills
$T^{\prime}\geq T.$ We apply Theorem \ref{The70}, page \pageref{The70}.
\end{proof}

\section{Eventually periodic points vs. eventually constant signals}

\begin{theorem}
\label{The74_}a) Let the signal $\widehat{x}\in\widehat{S}^{(n)}$ and the
point $\mu\in\widehat{\omega}(\widehat{x}).$ We suppose that $p\geq1$ and
$k^{\prime},k_{1},k_{2}\in\mathbf{N}_{\_}$ exist such that%
\begin{equation}
k^{\prime}\leq k_{1}<k_{2}, \label{per913}%
\end{equation}%
\begin{equation}
\forall k\in\widehat{\mathbf{T}}_{\mu}^{\widehat{x}}\cap\{k^{\prime}%
,k^{\prime}+1,k^{\prime}+2,...\},\{k+zp|z\in\mathbf{Z}\}\cap\{k^{\prime
},k^{\prime}+1,k^{\prime}+2,...\}\subset\widehat{\mathbf{T}}_{\mu}%
^{\widehat{x}}, \label{per914}%
\end{equation}%
\begin{equation}
\{k_{1},k_{1}+1,...,k_{2}\}\subset\widehat{\mathbf{T}}_{\mu}^{\widehat{x}},
\label{per915}%
\end{equation}%
\begin{equation}
k_{1}+p\leq k_{2} \label{per916}%
\end{equation}
are true. Then $\{k^{\prime},k^{\prime}+1,k^{\prime}+2,...\}\subset
\widehat{\mathbf{T}}_{\mu}^{\widehat{x}}$ holds$.$

b) The signal $x\in S^{(n)}$ and the point $\mu\in\omega(x)$ are given. We
suppose that $T>0$ and $t^{\prime},t_{1},t_{2}\in\mathbf{R}$ exist such that%
\begin{equation}
t^{\prime}\leq t_{1}<t_{2}, \label{per919}%
\end{equation}%
\begin{equation}
\forall t\in\mathbf{T}_{\mu}^{x}\cap\lbrack t^{\prime},\infty),\{t+zT|z\in
\mathbf{Z}\}\cap\lbrack t^{\prime},\infty)\subset\mathbf{T}_{\mu}^{x},
\label{per496}%
\end{equation}%
\begin{equation}
\lbrack t_{1},t_{2})\subset\mathbf{T}_{\mu}^{x}, \label{per497}%
\end{equation}%
\begin{equation}
t_{1}+T\leq t_{2} \label{per920}%
\end{equation}
hold. Then $[t^{\prime},\infty)\subset\mathbf{T}_{\mu}^{x}.$
\end{theorem}

\begin{proof}
b) Let $t\geq t^{\prime}$ be arbitrary. From (\ref{per920}) we have the
existence of $z^{\prime}\in\mathbf{Z}$ with $t+z^{\prime}T\in\lbrack
t_{1},t_{2}).$ We infer:%
\begin{equation}
t+z^{\prime}T\overset{(\ref{per919}),(\ref{per497})}{\in}\mathbf{T}_{\mu}%
^{x}\cap\lbrack t^{\prime},\infty), \label{per921}%
\end{equation}%
\[
t\in\{t+z^{\prime}T+zT|z\in\mathbf{Z}\}\cap\lbrack t^{\prime},\infty
)\overset{(\ref{per496}),(\ref{per921})}{\subset}\mathbf{T}_{\mu}^{x}.
\]
As $t$ was arbitrary, we get the statement of the Theorem.
\end{proof}

\chapter{\label{Cha7}Eventually periodic signals}

In the first two Sections we give properties that are equivalent with the
eventual periodicity of the signals.

In Section 3 we show the property that, for time instants greater than the
limit of periodicity, each omega limit point is accessed in a time interval
with the length of at most a period.

The bound of the limit of periodicity issue is addressed in Section 4.

Sections 5 and 6 refer to a property of eventual constancy that is used in
Section 7 to relate the discrete time with the real time eventually periodic signals.

The fact that the sums, the differences and the multiples of periods are
periods is shown in Section 8.

Section 9 draws conclusions concerning the form of the sets $\widehat
{P}^{\widehat{x}},P^{x}$ and in particular the existence of the prime period
is proved.

Sections 10, 11, 12 give necessity and sufficiency properties of eventual
periodicity and a special case, when the prime period is known.

The issue of changing the order of the quantifiers in stating eventual
periodicity properties is addressed in Section 13. Since the problem is not
solved so far, we state in Section 14 the hypothesis $P$ stating basically
that if all the points of the omega limit set are eventually periodic, then
the signal is eventually periodic.

\section{The first group of eventual periodicity properties}

\begin{remark}
These properties involve the eventual periodicity request of all the omega
limit points $\mu\in\widehat{\omega}(\widehat{x}),\mu\in\omega(x),$ with a
common period $p\geq1,T>0$ and a common limit of periodicity $k^{\prime}%
\in\mathbf{N},$ $t^{\prime}\in\mathbf{R}$. This way, we notice the
associations (\ref{per164})-(\ref{pre781})$_{page\;\pageref{pre781}}$,...,
(\ref{pre547})-(\ref{pre538})$_{page\;\pageref{pre538}}$, (\ref{pre576}%
)-(\ref{pre553})$_{page\;\pageref{pre553}}$,..., (\ref{pre610})-(\ref{pre600}%
)$_{page\;\pageref{pre600}}$ with the statements of Theorem \ref{The97}, page
\pageref{The97}, where eventual periodicity was used to characterize eventual
constancy. We make also the associations (\ref{per164})-(\ref{pre154}%
)$_{page\;\pageref{pre154}}$,...,(\ref{pre547})-(\ref{pre159}%
)$_{page\;\pageref{pre159}}$ and (\ref{pre576})-(\ref{pre572}%
)$_{page\;\pageref{pre572}}$, ...,(\ref{pre610})-(\ref{pre606}%
)$_{page\;\pageref{pre606}}$ with the statements of Theorem \ref{The67}, page
\pageref{The67}, referring to the eventual periodicity of the points.
\end{remark}

\begin{remark}
The statements (\ref{per164}),..., (\ref{pre547}), (\ref{pre576}),...,
(\ref{pre610}) from Theorem \ref{The28} are called of eventual periodicity of
$\widehat{x},x$ due to their equivalence with Definition \ref{Def28}, page
\pageref{Def28} that will be proved in the following Section, in Theorem
\ref{The109}.
\end{remark}

\begin{theorem}
\label{The28}The signals $\widehat{x}\in\widehat{S}^{(n)},x\in S^{(n)}$ are given.

a) The following statements are equivalent for any $p\geq1:$%
\begin{equation}
\left\{
\begin{array}
[c]{c}%
\forall\mu\in\widehat{\omega}(\widehat{x}),\exists k^{\prime}\in
\mathbf{N}_{\_},\forall k\in\widehat{\mathbf{T}}_{\mu}^{\widehat{x}}%
\cap\{k^{\prime},k^{\prime}+1,k^{\prime}+2,...\},\\
\{k+zp|z\in\mathbf{Z}\}\cap\{k^{\prime},k^{\prime}+1,k^{\prime}+2,...\}\subset
\widehat{\mathbf{T}}_{\mu}^{\widehat{x}},
\end{array}
\right.  \label{per164}%
\end{equation}%
\begin{equation}
\left\{
\begin{array}
[c]{c}%
\forall\mu\in\widehat{\omega}(\widehat{x}),\exists k^{\prime\prime}%
\in\mathbf{N},\forall k\in\widehat{\mathbf{T}}_{\mu}^{\widehat{\sigma
}^{k^{\prime\prime}}(\widehat{x})},\\
\{k+zp|z\in\mathbf{Z}\}\cap\mathbf{N}_{\_}\subset\widehat{\mathbf{T}}_{\mu
}^{\widehat{\sigma}^{k^{\prime\prime}}(\widehat{x})},
\end{array}
\right.  \label{pre340}%
\end{equation}%
\begin{equation}
\left\{
\begin{array}
[c]{c}%
\forall\mu\in\widehat{\omega}(\widehat{x}),\exists k^{\prime}\in
\mathbf{N}_{\_},\forall k\geq k^{\prime},\widehat{x}(k)=\mu\Longrightarrow\\
\Longrightarrow(\widehat{x}(k)=\widehat{x}(k+p)\text{ and }k-p\geq k^{\prime
}\Longrightarrow\widehat{x}(k)=\widehat{x}(k-p)),
\end{array}
\right.  \label{pre546}%
\end{equation}%
\begin{equation}
\left\{
\begin{array}
[c]{c}%
\forall\mu\in\widehat{\omega}(\widehat{x}),\exists k^{\prime\prime}%
\in\mathbf{N},\forall k\in\mathbf{N}_{\_},\widehat{\sigma}^{k^{\prime\prime}%
}(\widehat{x})(k)=\mu\Longrightarrow\\
\Longrightarrow(\widehat{\sigma}^{k^{\prime\prime}}(\widehat{x})(k)=\widehat
{\sigma}^{k^{\prime\prime}}(\widehat{x})(k+p)\text{ and }\\
\text{and }k-p\geq-1\Longrightarrow\widehat{\sigma}^{k^{\prime\prime}%
}(\widehat{x})(k)=\widehat{\sigma}^{k^{\prime\prime}}(\widehat{x})(k-p)).
\end{array}
\right.  \label{pre547}%
\end{equation}

b) The following statements are also equivalent for any $T>0:$%
\begin{equation}
\left\{
\begin{array}
[c]{c}%
\forall\mu\in\omega(x),\exists t^{\prime}\in I^{x},\\
\exists t_{1}^{\prime}\geq t^{\prime},\forall t\in\mathbf{T}_{\mu}^{x}%
\cap\lbrack t_{1}^{\prime},\infty),\{t+zT|z\in\mathbf{Z}\}\cap\lbrack
t_{1}^{\prime},\infty)\subset\mathbf{T}_{\mu}^{x},
\end{array}
\right.  \label{pre576}%
\end{equation}%
\begin{equation}
\left\{
\begin{array}
[c]{c}%
\forall\mu\in\omega(x),\exists t_{1}^{\prime}\in\mathbf{R},\\
\forall t\in\mathbf{T}_{\mu}^{x}\cap\lbrack t_{1}^{\prime},\infty
),\{t+zT|z\in\mathbf{Z}\}\cap\lbrack t_{1}^{\prime},\infty)\subset
\mathbf{T}_{\mu}^{x},
\end{array}
\right.  \label{pre607}%
\end{equation}%
\begin{equation}
\left\{
\begin{array}
[c]{c}%
\forall\mu\in\omega(x),\exists t^{\prime\prime}\in\mathbf{R},\exists
t^{\prime}\in I^{\sigma^{t^{\prime\prime}}(x)},\\
\forall t\in\mathbf{T}_{\mu}^{\sigma^{t^{\prime\prime}}(x)}\cap\lbrack
t^{\prime},\infty),\{t+zT|z\in\mathbf{Z}\}\cap\lbrack t^{\prime}%
,\infty)\subset\mathbf{T}_{\mu}^{\sigma^{t^{\prime\prime}}(x)},
\end{array}
\right.  \label{pre577}%
\end{equation}%
\begin{equation}
\left\{
\begin{array}
[c]{c}%
\forall\mu\in\omega(x),\exists t^{\prime\prime}\in\mathbf{R},\exists
t^{\prime}\in\mathbf{R},\\
\forall t\in\mathbf{T}_{\mu}^{\sigma^{t^{\prime\prime}}(x)}\cap\lbrack
t^{\prime},\infty),\{t+zT|z\in\mathbf{Z}\}\cap\lbrack t^{\prime}%
,\infty)\subset\mathbf{T}_{\mu}^{\sigma^{t^{\prime\prime}}(x)},
\end{array}
\right.  \label{pre608}%
\end{equation}%
\begin{equation}
\left\{
\begin{array}
[c]{c}%
\forall\mu\in\omega(x),\exists t^{\prime}\in I^{x},\exists t_{1}^{\prime}\geq
t^{\prime},\forall t\geq t_{1}^{\prime},x(t)=\mu\Longrightarrow\\
\Longrightarrow(x(t)=x(t+T)\text{ and }t-T\geq t_{1}^{\prime}\Longrightarrow
x(t)=x(t-T)),
\end{array}
\right.  \label{pre578}%
\end{equation}%
\begin{equation}
\left\{
\begin{array}
[c]{c}%
\forall\mu\in\omega(x),\exists t_{1}^{\prime}\in\mathbf{R},\forall t\geq
t_{1}^{\prime},x(t)=\mu\Longrightarrow\\
\Longrightarrow(x(t)=x(t+T)\text{ and }t-T\geq t_{1}^{\prime}\Longrightarrow
x(t)=x(t-T)),
\end{array}
\right.  \label{pre609}%
\end{equation}%
\begin{equation}
\left\{
\begin{array}
[c]{c}%
\forall\mu\in\omega(x),\exists t^{\prime\prime}\in\mathbf{R},\exists
t^{\prime}\in I^{\sigma^{t^{\prime\prime}}(x)},\\
\forall t\geq t^{\prime},\sigma^{t^{\prime\prime}}(x)(t)=\mu\Longrightarrow
(\sigma^{t^{\prime\prime}}(x)(t)=\sigma^{t^{\prime\prime}}(x)(t+T)\text{
and}\\
\text{and }t-T\geq t^{\prime}\Longrightarrow\sigma^{t^{\prime\prime}%
}(x)(t)=\sigma^{t^{\prime\prime}}(x)(t-T)),
\end{array}
\right.  \label{pre579}%
\end{equation}%
\begin{equation}
\left\{
\begin{array}
[c]{c}%
\forall\mu\in\omega(x),\exists t^{\prime\prime}\in\mathbf{R},\exists
t^{\prime}\in\mathbf{R},\forall t\geq t^{\prime},\sigma^{t^{\prime\prime}%
}(x)(t)=\mu\Longrightarrow\\
\Longrightarrow(\sigma^{t^{\prime\prime}}(x)(t)=\sigma^{t^{\prime\prime}%
}(x)(t+T)\text{ and}\\
\text{and }t-T\geq t^{\prime}\Longrightarrow\sigma^{t^{\prime\prime}%
}(x)(t)=\sigma^{t^{\prime\prime}}(x)(t-T)).
\end{array}
\right.  \label{pre610}%
\end{equation}

\end{theorem}

\begin{proof}
a) The proof of the implications%
\[
(\ref{per164})\Longrightarrow(\ref{pre340})\Longrightarrow(\ref{pre546}%
)\Longrightarrow(\ref{pre547})
\]
follows from Theorem \ref{The97}, page \pageref{The97}. We prove
(\ref{pre547})$\Longrightarrow$(\ref{per164}).

Let $\mu\in\widehat{\omega}(\widehat{x})$ arbitrary, fixed. (\ref{pre547})
shows the existence of $k^{\prime\prime}\in\mathbf{N}$ such that%
\begin{equation}
\left\{
\begin{array}
[c]{c}%
\forall k\in\mathbf{N}_{\_},\widehat{\sigma}^{k^{\prime\prime}}(\widehat
{x})(k)=\mu\Longrightarrow(\widehat{\sigma}^{k^{\prime\prime}}(\widehat
{x})(k)=\widehat{\sigma}^{k^{\prime\prime}}(\widehat{x})(k+p)\text{ and }\\
\text{and }k-p\geq-1\Longrightarrow\widehat{\sigma}^{k^{\prime\prime}%
}(\widehat{x})(k)=\widehat{\sigma}^{k^{\prime\prime}}(\widehat{x})(k-p)).
\end{array}
\right.  \label{p20}%
\end{equation}

We denote $k^{\prime}=k^{\prime\prime}-1,$ where $k^{\prime}\geq-1.$ We also
denote $k^{\prime\prime\prime}=k+k^{\prime}+1,$ where $k^{\prime\prime\prime
}\geq k^{\prime}.$ With these notations, (\ref{p20}) becomes%
\begin{equation}
\left\{
\begin{array}
[c]{c}%
\forall k^{\prime\prime\prime}\geq k^{\prime},\widehat{x}(k^{\prime
\prime\prime})=\mu\Longrightarrow(\widehat{x}(k^{\prime\prime\prime}%
)=\widehat{x}(k^{\prime\prime\prime}+p)\text{ and }\\
\text{and }k^{\prime\prime\prime}-p\geq k^{\prime}\Longrightarrow\widehat
{x}(k^{\prime\prime\prime})=\widehat{x}(k^{\prime\prime\prime}-p)).
\end{array}
\right.  \label{p21}%
\end{equation}

Let now $k\in\widehat{\mathbf{T}}_{\mu}^{\widehat{x}}$ and $z\in\mathbf{Z} $
arbitrary such that $k\geq k^{\prime}$ and $k+zp\geq k^{\prime}.$ We have the
following possibilities.

Case $z>0,$%
\[
\mu=\widehat{x}(k)\overset{(\ref{p21})}{=}\widehat{x}(k+p)\overset
{(\ref{p21})}{=}\widehat{x}(k+2p)\overset{(\ref{p21})}{=}...\overset
{(\ref{p21})}{=}\widehat{x}(k+zp);
\]

Case $z=0,$%
\[
\mu=\widehat{x}(k)=\widehat{x}(k+zp);
\]

Case $z<0,$%
\[
\mu=\widehat{x}(k)\overset{(\ref{p21})}{=}\widehat{x}(k-p)\overset
{(\ref{p21})}{=}\widehat{x}(k-2p)\overset{(\ref{p21})}{=}...\overset
{(\ref{p21})}{=}\widehat{x}(k+zp).
\]

We have obtained in all these situations that $k+zp\in\widehat{\mathbf{T}%
}_{\mu}^{\widehat{x}}$ holds, i.e. (\ref{per164}) is true.

b) The proof of the implications%
\[
(\ref{pre576})\Longrightarrow(\ref{pre607})\Longrightarrow(\ref{pre577}%
)\Longrightarrow(\ref{pre608})\Longrightarrow(\ref{pre578})\Longrightarrow
(\ref{pre609})\Longrightarrow(\ref{pre579})\Longrightarrow(\ref{pre610})
\]
follows from Theorem \ref{The97}, page \pageref{The97}. We prove
(\ref{pre610})$\Longrightarrow$(\ref{pre576}).

Let $\mu\in\omega(x)$ arbitrary. From (\ref{pre610}) we get the existence of
$t^{\prime\prime}\in\mathbf{R}$ and $t^{\prime}\in\mathbf{R}$ such that%
\begin{equation}
\left\{
\begin{array}
[c]{c}%
\forall t\geq t^{\prime},\sigma^{t^{\prime\prime}}(x)(t)=\mu\Longrightarrow
(\sigma^{t^{\prime\prime}}(x)(t)=\sigma^{t^{\prime\prime}}(x)(t+T)\text{
and}\\
\text{and }t-T\geq t^{\prime}\Longrightarrow\sigma^{t^{\prime\prime}%
}(x)(t)=\sigma^{t^{\prime\prime}}(x)(t-T)).
\end{array}
\right.  \label{p22}%
\end{equation}

Let $t_{1}^{\prime}=\max\{t^{\prime},t^{\prime\prime}\}.$ On one hand
(\ref{p22}) is still true if we replace $t^{\prime}$ with $t_{1}^{\prime},$
from Lemma \ref{Lem30}, page \pageref{Lem30}. On the other hand, in this case
$\sigma^{t^{\prime\prime}}(x)=x,$ thus (\ref{p22}) becomes%
\begin{equation}
\forall t\geq t_{1}^{\prime},x(t)=\mu\Longrightarrow(x(t)=x(t+T)\text{ and
}t-T\geq t_{1}^{\prime}\Longrightarrow x(t)=x(t-T)). \label{p23}%
\end{equation}

We take arbitrarily some $t^{\prime\prime\prime}\in I^{x}\cap(-\infty
,t_{1}^{\prime}].$ Let $t\in\mathbf{T}_{\mu}^{x}$ and $z\in\mathbf{Z}$
arbitrary with $t\geq t_{1}^{\prime}$ and $t+zT\geq t_{1}^{\prime}.$ We prove
in all the three cases $z>0,z=0,z<0$ that (\ref{p23}) implies $t+zT\in
\mathbf{T}_{\mu}^{x}.$
\end{proof}

\begin{example}
Let $x\in S^{(2)},$
\[%
\begin{array}
[c]{c}%
x(t)=(0,1)\cdot\chi_{(-\infty,-\frac{1}{2})}(t)\oplus(0,1)\cdot\chi
_{\lbrack0,1)}(t)\oplus(1,0)\cdot\chi_{\lbrack1,2)}(t)\oplus(1,1)\cdot
\chi_{\lbrack2,3)}(t)\\
\oplus(0,1)\cdot\chi_{\lbrack3,4)}(t)\oplus(1,0)\cdot\chi_{\lbrack
4,5)}(t)\oplus(1,1)\cdot\chi_{\lbrack5,6)}(t)\oplus(0,1)\cdot\chi
_{\lbrack6,7)}(t)\oplus...
\end{array}
\]
$x$ is eventually periodic and it fulfills%
\begin{equation}
\forall t\geq0,x(t)=x(t+3), \label{per547}%
\end{equation}
since all of $(0,1),(1,0),(1,1)\in\omega(x)$ are eventually periodic with the
period $T=3$ and the limit of periodicity $t^{\prime}=0.$
\end{example}

\begin{example}
\label{Exa6}The signal $\widehat{x}\in\widehat{S}^{(1)},$%
\[
\widehat{x}=\underset{1}{\underbrace{0}},\underset{1}{\underbrace{1}%
},\underset{2}{\underbrace{0,0}},\underset{2}{\underbrace{1,1}},\underset
{3}{\underbrace{0,0,0}},\underset{3}{\underbrace{1,1,1}},\underset
{4}{\underbrace{0,0,0,0}},...
\]
is not eventually periodic, because none of $0,1$ is eventually periodic.
\end{example}

\section{The second group of eventual periodicity properties}

\begin{remark}
This group of properties refers to signals, not to their values, and they were
presented previously in Theorem \ref{The100}, page \pageref{The100}, as useful
in characterizing the eventual constancy. We notice the associations
(\ref{per163})-(\ref{per83})$_{page\;\pageref{per83}},$ (\ref{pre341}%
)-(\ref{pre515})$_{page\;\pageref{pre515}}$ and (\ref{pre580})-(\ref{pre557}%
)$_{page\;\pageref{pre557}},...,$(\ref{pre612})-(\ref{pre602}%
)$_{page\;\pageref{pre602}}.$
\end{remark}

\begin{theorem}
\label{The109}The signals $\widehat{x},x$ are given.

a) For any $p\geq1,$ the following statements are equivalent with the eventual
periodicity of $\widehat{x}$:%
\begin{equation}
\exists k^{\prime}\in\mathbf{N}_{\_},\forall k\geq k^{\prime},\widehat
{x}(k)=\widehat{x}(k+p), \label{per163}%
\end{equation}%
\begin{equation}
\exists k^{\prime\prime}\in\mathbf{N},\forall k\in\mathbf{N}_{\_}%
,\widehat{\sigma}^{k^{\prime\prime}}(\widehat{x})(k)=\widehat{\sigma
}^{k^{\prime\prime}}(\widehat{x})(k+p). \label{pre341}%
\end{equation}

b) For any $T>0,$ the following statements are also equivalent with the
eventual periodicity of $x$:%
\begin{equation}
\exists t^{\prime}\in I^{x},\exists t_{1}^{\prime}\geq t^{\prime},\forall
t\geq t_{1}^{\prime},x(t)=x(t+T), \label{pre580}%
\end{equation}%
\begin{equation}
\exists t_{1}^{\prime}\in\mathbf{R},\forall t\geq t_{1}^{\prime},x(t)=x(t+T),
\label{pre611}%
\end{equation}%
\begin{equation}
\exists t^{\prime\prime}\in\mathbf{R},\exists t^{\prime}\in I^{\sigma
^{t^{\prime\prime}}(x)},\forall t\geq t^{\prime},\sigma^{t^{\prime\prime}%
}(x)(t)=\sigma^{t^{\prime\prime}}(x)(t+T), \label{pre581}%
\end{equation}%
\begin{equation}
\exists t^{\prime\prime}\in\mathbf{R},\exists t^{\prime}\in\mathbf{R},\forall
t\geq t^{\prime},\sigma^{t^{\prime\prime}}(x)(t)=\sigma^{t^{\prime\prime}%
}(x)(t+T). \label{pre612}%
\end{equation}

\end{theorem}

\begin{proof}
a) The implication (\ref{per163})$\Longrightarrow$(\ref{pre341}) results from
Theorem \ref{The100}. We prove (\ref{per164})$\Longrightarrow$(\ref{per163}).

We suppose that $\widehat{\omega}(\widehat{x})=\{\mu^{1},...,\mu^{s}\}.$ For
any $i\in\{1,...,s\},$ some $k_{i}^{\prime}\in\mathbf{N}_{\_}$ exists with the
property%
\begin{equation}
\left\{
\begin{array}
[c]{c}%
\forall k\in\widehat{\mathbf{T}}_{\mu^{i}}^{\widehat{x}}\cap\{k_{i}^{\prime
},k_{i}^{\prime}+1,k_{i}^{\prime}+2,...\},\\
\{k+zp|z\in\mathbf{Z}\}\cap\{k_{i}^{\prime},k_{i}^{\prime}+1,k_{i}^{\prime
}+2,...\}\subset\widehat{\mathbf{T}}_{\mu^{i}}^{\widehat{x}}.
\end{array}
\right.
\end{equation}

Let $\widetilde{k}\in\mathbf{N}_{\_}$ be a time instant that fulfills
$\widehat{\omega}(\widehat{x})=\{\widehat{x}(k)|k\geq\widetilde{k}\}.$ With
$k^{\prime}=\max\{\widetilde{k},k_{1}^{\prime},...,k_{s}^{\prime}\},$ from
Lemma \ref{Lem30}, page \pageref{Lem30} we have
\begin{equation}
\left\{
\begin{array}
[c]{c}%
\forall k\in\widehat{\mathbf{T}}_{\mu^{i}}^{\widehat{x}}\cap\{k^{\prime
},k^{\prime}+1,k^{\prime}+2,...\},\\
\{k+zp|z\in\mathbf{Z}\}\cap\{k^{\prime},k^{\prime}+1,k^{\prime}+2,...\}\subset
\widehat{\mathbf{T}}_{\mu^{i}}^{\widehat{x}},
\end{array}
\right.  \label{p31}%
\end{equation}
and this statement is true for all $i\in\{1,...,s\}.$

Let now $k\geq k^{\prime}$ arbitrary, for which $i$ exists with $\widehat
{x}(k)=\mu^{i}.$ We infer%
\begin{equation}
k+p\in\{k+zp|z\in\mathbf{Z}\}\cap\{k^{\prime},k^{\prime}+1,k^{\prime
}+2,...\}\overset{(\ref{p31})}{\subset}\widehat{\mathbf{T}}_{\mu^{i}%
}^{\widehat{x}},
\end{equation}
thus $\widehat{x}(k+p)=\mu^{i}=\widehat{x}(k).$

(\ref{pre341})$\Longrightarrow$(\ref{per164}) Let $\mu\in\widehat{\omega
}(\widehat{x})$ arbitrary. Some $k^{\prime\prime}\in\mathbf{N}$ exists with
the property that%
\begin{equation}
\forall k\in\mathbf{N}_{\_},\widehat{x}(k+k^{\prime\prime})=\widehat
{x}(k+k^{\prime\prime}+p). \label{pre905}%
\end{equation}
We denote $k^{\prime}=k^{\prime\prime}-1,$ where $k^{\prime}\in\mathbf{N}%
_{\_}.$ We also denote $k^{\prime\prime\prime}=k+k^{\prime}+1,$ where
$k^{\prime\prime\prime}\geq k^{\prime}.$ With these notations (\ref{pre905})
becomes%
\begin{equation}
\forall k^{\prime\prime\prime}\geq k^{\prime},\widehat{x}(k^{\prime
\prime\prime})=\widehat{x}(k^{\prime\prime\prime}+p). \label{p30}%
\end{equation}
Let now $k\in\widehat{\mathbf{T}}_{\mu}^{\widehat{x}}$ and $z\in\mathbf{Z} $
arbitrary such that $k\geq k^{\prime}$ and $k+zp\geq k^{\prime}.$ The
following possibilities exist.

Case $z>0,$%
\[
\mu=\widehat{x}(k)\overset{(\ref{p30})}{=}\widehat{x}(k+p)\overset
{(\ref{p30})}{=}\widehat{x}(k+2p)\overset{(\ref{p30})}{=}...\overset
{(\ref{p30})}{=}\widehat{x}(k+zp);
\]

Case $z=0,$%
\[
\mu=\widehat{x}(k)=\widehat{x}(k+zp);
\]

Case $z<0,$%
\[
\widehat{x}(k+zp)\overset{(\ref{p30})}{=}\widehat{x}(k+(z+1)p)\overset
{(\ref{p30})}{=}\widehat{x}(k+(z+2)p)\overset{(\ref{p30})}{=}...
\]%
\[
...\overset{(\ref{p30})}{=}\widehat{x}(k-p)\overset{(\ref{p30})}{=}\widehat
{x}(k)=\mu.
\]
In all these cases $\widehat{x}(k+zp)=\mu,$ thus $k+zp\in\widehat{\mathbf{T}%
}_{\mu}^{\widehat{x}}.$ (\ref{per164}) is proved.

b) The implications%
\[
(\ref{pre580})\Longrightarrow(\ref{pre611})\Longrightarrow(\ref{pre581}%
)\Longrightarrow(\ref{pre612})
\]
result from Theorem \ref{The100}, page \pageref{The100}.

(\ref{pre576})$\Longrightarrow$(\ref{pre580}) We suppose that $\omega
(x)=\{\mu^{1},...,\mu^{s}\}$ and let $i\in\{1,...,s\}$ arbitrary. From
(\ref{pre576}) we have the existence of $t^{i}\in I^{x}$ and $t_{1}^{i}\geq
t^{i}$ with%
\begin{equation}
\forall t\in\mathbf{T}_{\mu^{i}}^{x}\cap\lbrack t_{1}^{i},\infty
),\{t+zT|z\in\mathbf{Z}\}\cap\lbrack t_{1}^{i},\infty)\subset\mathbf{T}%
_{\mu^{i}}^{x} \label{pre907}%
\end{equation}
fulfilled.

We denote $t^{\prime}=\max\{t^{1},...,t^{s}\}$ and we notice that $t^{\prime
}\in I^{x}$ holds, since $t^{\prime}$ coincides with one of $t^{1},...,t^{s}.$
We put $\widetilde{t}\in\mathbf{R}$ for the time instant that fulfills%
\begin{equation}
\omega(x)=\{x(t)|t\geq\widetilde{t}\}. \label{p33}%
\end{equation}
Let $i\in\{1,...,s\}$ arbitrary and fixed. The fact that for $t_{1}^{\prime
}=\max\{\widetilde{t},t^{\prime},t_{1}^{1},...,t_{1}^{s}\}$ the statement%
\begin{equation}
\forall t\in\mathbf{T}_{\mu^{i}}^{x}\cap\lbrack t_{1}^{\prime},\infty
),\{t+zT|z\in\mathbf{Z}\}\cap\lbrack t_{1}^{\prime},\infty)\subset
\mathbf{T}_{\mu^{i}}^{x} \label{p32}%
\end{equation}
holds is a consequence of (\ref{pre907}) and Lemma \ref{Lem30} b), page
\pageref{Lem30}.

Let now $t\geq t_{1}^{\prime}$ arbitrary. Some $i\in\{1,...,s\}$ exists with
$x(t)=\mu^{i},$ thus we can write $t\in\mathbf{T}_{\mu^{i}}^{x}\cap\lbrack
t_{1}^{\prime},\infty)$ and%
\[
t+T\in\{t+zT|z\in\mathbf{Z}\}\cap\lbrack t_{1}^{\prime},\infty)\overset
{(\ref{p32})}{\subset}\mathbf{T}_{\mu^{i}}^{x},
\]
wherefrom $x(t+T)=\mu^{i}=x(t).$

(\ref{pre612})$\Longrightarrow$(\ref{pre576}) We denote with $\widetilde{t}%
\in\mathbf{R}$ the time instant that fulfills (\ref{p33}). The hypothesis
shows the existence of $t^{\prime\prime}\in\mathbf{R}$ and $t^{\prime}%
\in\mathbf{R}$ such that%
\begin{equation}
\forall t\geq t^{\prime},\sigma^{t^{\prime\prime}}(x)(t)=\sigma^{t^{\prime
\prime}}(x)(t+T) \label{pre910}%
\end{equation}
and let $t_{1}^{\prime}=\max\{\widetilde{t},t^{\prime\prime},t^{\prime}\}$
arbitrary. We have from (\ref{pre910}):%
\begin{equation}
\forall t\geq t_{1}^{\prime},x(t)=x(t+T). \label{pre911}%
\end{equation}
Let now $\mu\in\omega(x)$ arbitrary and $t^{\prime\prime\prime}\in I^{x}$. We
suppose that $t^{\prime\prime\prime}\leq t_{1}^{\prime},$ as this is always
possible. Let $t\in\mathbf{T}_{\mu}^{x}\cap\lbrack t_{1}^{\prime},\infty)$
arbitrary and let us take $z\in\mathbf{Z}$ arbitrary itself with $t+zT\geq
t_{1}^{\prime}.$ We have:

Case $z>0,$%
\[
\mu=x(t)\overset{(\ref{pre911})}{=}x(t+T)\overset{(\ref{pre911})}%
{=}x(t+2T)\overset{(\ref{pre911})}{=}...\overset{(\ref{pre911})}{=}x(t+zT);
\]

Case $z=0,$%
\[
\mu=x(t)=x(t+zT);
\]

Case $z<0,$%
\[
x(t+zT)\overset{(\ref{pre911})}{=}...\overset{(\ref{pre911})}{=}%
x(t-2T)\overset{(\ref{pre911})}{=}x(t-T)\overset{(\ref{pre911})}{=}x(t)=\mu.
\]
We have obtained that in all these situations $x(t+zT)=\mu,$ i.e.
$t+zT\in\mathbf{T}_{\mu}^{x}.$
\end{proof}

\begin{remark}
The eventual periodicity of the signals highlights the existence of two time
instants, $\widetilde{k}\in\mathbf{N}_{\_}$ and $k^{\prime}\in\mathbf{N}_{\_}$
given by%
\begin{equation}
\widehat{\omega}(\widehat{x})=\{\widehat{x}(k)|k\geq\widetilde{k}\},
\label{p34}%
\end{equation}%
\begin{equation}
\forall k\geq k^{\prime},\widehat{x}(k)=\widehat{x}(k+p). \label{p35}%
\end{equation}
None of $\widetilde{k},k^{\prime}$ is unique, in the sense that (\ref{p34}),
(\ref{p35}) may be rewritten for any $\widetilde{k}_{1}\geq\widetilde{k}%
,k_{1}^{\prime}\geq k^{\prime}$ but if $\widetilde{k},k^{\prime}$ are chosen
to be the least such that (\ref{p34}), (\ref{p35}) hold, then $\widetilde
{k}\leq k^{\prime}.$

The situation is also true in the real time case, with the remark that if
$Or(x)=\omega(x),$ then the least $\widetilde{t}$ such that%
\[
\omega(x)=\{x(t)|t\geq\widetilde{t}\}
\]
holds does not exist.
\end{remark}

\begin{remark}
The statements (\ref{per164}),..., (\ref{pre547}) and (\ref{pre576}),...,
(\ref{pre610}) refer to left-and-right time shifts, while the statements
(\ref{per163}), (\ref{pre341}) and (\ref{pre580}),..., (\ref{pre612}) refer to
right time shifts only.
\end{remark}

\section{The accessibility of the omega limit set}

\begin{theorem}
\label{The140}a) If $\widehat{x}\in\widehat{S}^{(n)},$ then
\begin{equation}
\underset{\mu\in\widehat{\omega}(\widehat{x})}{%
{\displaystyle\bigcap}
}\widehat{P}_{\mu}^{\widehat{x}}\neq\varnothing\Longrightarrow\forall
k^{\prime}\in\underset{\mu\in\widehat{\omega}(\widehat{x})}{%
{\displaystyle\bigcap}
}\widehat{L}_{\mu}^{\widehat{x}},\text{ }\widehat{\omega}(\widehat
{x})=\{\widehat{x}(k)|k\geq k^{\prime}\}, \label{p254}%
\end{equation}%
\begin{equation}
\widehat{P}^{\widehat{x}}\neq\varnothing\Longrightarrow\forall k^{\prime}%
\in\widehat{L}^{\widehat{x}},\text{ }\widehat{\omega}(\widehat{x}%
)=\{\widehat{x}(k)|k\geq k^{\prime}\} \label{p253}%
\end{equation}
hold.

b) For $x\in S^{(n)},$ we have the truth of
\begin{equation}
\underset{\mu\in\omega(x)}{%
{\displaystyle\bigcap}
}P_{\mu}^{x}\neq\varnothing\Longrightarrow\forall t^{\prime}\in\underset
{\mu\in\omega(x)}{%
{\displaystyle\bigcap}
}L_{\mu}^{x},\text{ }\omega(x)=\{x(t)|t\geq t^{\prime}\}, \label{p255}%
\end{equation}%
\begin{equation}
P^{x}\neq\varnothing\Longrightarrow\forall t^{\prime}\in L^{x},\text{ }%
\omega(x)=\{x(t)|t\geq t^{\prime}\}. \label{p256}%
\end{equation}

\end{theorem}

\begin{proof}
a) (\ref{p254}). The hypothesis states that $\underset{\mu\in\widehat{\omega
}(\widehat{x})}{%
{\displaystyle\bigcap}
}\widehat{P}_{\mu}^{\widehat{x}}\neq\varnothing.$ If $\underset{\mu\in
\widehat{\omega}(\widehat{x})}{%
{\displaystyle\bigcap}
}\widehat{L}_{\mu}^{\widehat{x}}=\varnothing$ then the statement is trivially
true, thus we suppose that $\underset{\mu\in\widehat{\omega}(\widehat{x})}{%
{\displaystyle\bigcap}
}\widehat{L}_{\mu}^{\widehat{x}}\neq\varnothing$ and let $k^{\prime}%
\in\underset{\mu\in\widehat{\omega}(\widehat{x})}{%
{\displaystyle\bigcap}
}\widehat{L}_{\mu}^{\widehat{x}}$ arbitrary. We prove $\widehat{\omega
}(\widehat{x})\subset\{\widehat{x}(k)|k\geq k^{\prime}\}.$ Some $\widetilde
{k}\in\mathbf{N}_{\_}$ exists such that $\widehat{\omega}(\widehat
{x})=\{\widehat{x}(k)|k\geq\widetilde{k}\}$ and we have the following possibilities.

Case $k^{\prime}<\widetilde{k}$

In this case $\widehat{\omega}(\widehat{x})\subset\{\widehat{x}(k)|k\geq
k^{\prime}\}.$

Case $k^{\prime}\geq\widetilde{k}$

If so, we have from Theorem \ref{The12_}, page \pageref{The12_} that
$\widehat{\omega}(\widehat{x})=\{\widehat{x}(k)|k\geq k^{\prime}\}.$

We prove $\{\widehat{x}(k)|k\geq k^{\prime}\}\subset\widehat{\omega}%
(\widehat{x}).$ For this we take arbitrarily $k\geq k^{\prime}$ and
$p\in\underset{\mu\in\widehat{\omega}(\widehat{x})}{%
{\displaystyle\bigcap}
}\widehat{P}_{\mu}^{\widehat{x}}.$ We have $\widehat{x}(k)=\widehat
{x}(k+p)=\widehat{x}(k+2p)=...,$ thus $\widehat{\mathbf{T}}_{\widehat{x}%
(k)}^{\widehat{x}}$ is infinite and $\widehat{x}(k)\in\widehat{\omega
}(\widehat{x}).$

b) (\ref{p256}). We show that $\omega(x)\subset\{x(t)|t\geq t^{\prime}\}.$
From Theorem \ref{The12_}, page \pageref{The12_} we know that some
$\widetilde{t}\in\mathbf{R}$ exists with $\omega(x)=\{x(t)|t\geq\widetilde
{t}\}.$ There are two possibilities.

Case $t^{\prime}<\widetilde{t}$

If so, then $\omega(x)\subset\{x(t)|t\geq t^{\prime}\}.$

Case $t^{\prime}\geq\widetilde{t}$

In this case, see Theorem \ref{The12_}, $\omega(x)=\{x(t)|t\geq t^{\prime}\}.$

We show now that $\{x(t)|t\geq t^{\prime}\}\subset\omega(x)$ holds and let
$t\geq t^{\prime},T\in P^{x}$ arbitrary. The hypothesis shows that
$x(t)=x(t+T)=x(t+2T)=...,$ i.e. $\mathbf{T}_{x(t)}^{x}$ is superiorly
unbounded. This means that $x(t)\in\omega(x).$
\end{proof}

\begin{theorem}
\label{The125}a) If $\widehat{x}$ is eventually periodic with the period
$p\geq1$ and the limit of periodicity $k^{\prime}\in\mathbf{N}_{\_}:$%
\begin{equation}
\forall k\geq k^{\prime},\widehat{x}(k)=\widehat{x}(k+p),
\end{equation}
then
\begin{equation}
\forall k\geq k^{\prime},\widehat{\omega}(\widehat{x})=\{\widehat{x}%
(i)|i\in\{k,k+1,...,k+p-1\}\}.
\end{equation}

b) If $x$ is eventually periodic with the period $T>0$ and the limit of
periodicity $t^{\prime}\in\mathbf{R:}$%
\begin{equation}
\forall t\geq t^{\prime},x(t)=x(t+T),
\end{equation}
then%
\begin{equation}
\forall t\geq t^{\prime},\omega(x)=\{x(\xi)|\xi\in\lbrack t,t+T)\}.
\end{equation}

\end{theorem}

\begin{proof}
a) We know from Theorem \ref{The140} that $\widehat{\omega}(\widehat
{x})=\{\widehat{x}(k)|k\geq k^{\prime}\}.$ Let $k\geq k^{\prime}$ and $\mu
\in\widehat{\omega}(\widehat{x})$ arbitrary, fixed. As $\mu~$is eventually
periodic with the period $p$, we have from Theorem \ref{Lem1}, page
\pageref{Lem1} that $\widehat{\mathbf{T}}_{\mu}^{\widehat{x}}\cap
\{k,k+1,...,k+p-1\}\neq\varnothing.$ We get the existence of $i\in
\widehat{\mathbf{T}}_{\mu}^{\widehat{x}}\cap\{k,k+1,...,k+p-1\}$ thus
$\mu=\widehat{x}(i).$ We have proved that $\widehat{\omega}(\widehat
{x})\subset\{\widehat{x}(i)|i\in\{k,k+1,...,k+p-1\}\}.$ The inverse inclusion
is obvious, since any eventually periodic value of $\widehat{x}$ is an omega
limit point.

b) Theorem \ref{The140} shows that $\omega(x)=\{x(t)|t\geq t^{\prime}\}.$ Let
us fix arbitrarily $t\geq t^{\prime}$ and $\mu\in\omega(x).$ As $\mu$ is
eventually periodic with the period $T,$ we infer fromTheorem \ref{Lem1} that
$\mathbf{T}_{\mu}^{x}\cap\lbrack t,t+T)\neq\varnothing$ and let $\xi
\in\mathbf{T}_{\mu}^{x}\cap\lbrack t,t+T)$ thus $\mu=x(\xi).$ We have shown
the inclusion $\omega(x)\subset\{x(\xi)|\xi\in\lbrack t,t+T)\}.$ The inclusion
$\{x(\xi)|\xi\in\lbrack t,t+T)\}\subset\omega(x)$ is obvious, since any point
of the left hand set is eventually periodic and omega limit.
\end{proof}

\begin{remark}
The previous Theorem states the property that, in the case of the eventually
periodic signals, all the omega limit points are accessible in a time interval
with the length of a period.
\end{remark}

\section{The limit of periodicity}

\begin{theorem}
\label{The141}a) $\widehat{x}\in\widehat{S}^{(n)},p\geq1,p^{\prime}%
\geq1,k^{\prime}\in\mathbf{N}_{\_},k^{\prime\prime}\in\mathbf{N}_{\_}$ are
given such that%
\begin{equation}
\forall k\geq k^{\prime},\widehat{x}(k)=\widehat{x}(k+p), \label{p247}%
\end{equation}%
\begin{equation}
\forall k\geq k^{\prime\prime},\widehat{x}(k)=\widehat{x}(k+p^{\prime})
\label{p248}%
\end{equation}
hold. We have%
\begin{equation}
\forall k\geq k^{\prime},\widehat{x}(k)=\widehat{x}(k+p^{\prime}).
\label{p249}%
\end{equation}

b) We consider the signal $x\in S^{(n)},$ together with $T>0,T^{\prime
}>0,t^{\prime}\in\mathbf{R},t^{\prime\prime}\in\mathbf{R}$ and we ask that%
\begin{equation}
\forall t\geq t^{\prime},x(t)=x(t+T),
\end{equation}%
\begin{equation}
\forall t\geq t^{\prime\prime},x(t)=x(t+T^{\prime})
\end{equation}
are fulfilled. Then%
\begin{equation}
\forall t\geq t^{\prime},x(t)=x(t+T^{\prime})
\end{equation}
is true.
\end{theorem}

\begin{proof}
a) Let $k\geq k^{\prime}$ arbitrary, fixed. We have two possibilities.

Case $k^{\prime}\geq k^{\prime\prime}$

In this situation $k\geq k^{\prime\prime},$ thus we can write%
\[
\widehat{x}(k)\overset{(\ref{p248})}{=}\widehat{x}(k+p^{\prime}).
\]

Case $k^{\prime}<k^{\prime\prime}$

Let us take $k_{1}\in\mathbf{N}$ with the property that $k+k_{1}p\geq
k^{\prime\prime}.$ We can write:%
\[
\widehat{x}(k)\overset{(\ref{p247})}{=}\widehat{x}(k+k_{1}p)\overset
{(\ref{p248})}{=}\widehat{x}(k+k_{1}p+p^{\prime})\overset{(\ref{p247})}%
{=}\widehat{x}(k+p^{\prime}).
\]

\end{proof}

\begin{remark}
The previous Theorem states the fact that, if $\widehat{x},x$ are eventually
periodic, then $\widehat{L}^{\widehat{x}},L^{x}$ do not depend on the choice
of $p\in\widehat{P}^{\widehat{x}},T\in P^{x}.$
\end{remark}

\begin{theorem}
\label{The139}a) If $\widehat{x}$ is eventually periodic, then%
\[
\widehat{L}^{\widehat{x}}=\underset{\mu\in\widehat{\omega}(\widehat{x})}{%
{\displaystyle\bigcap}
}\widehat{L}_{\mu}^{\widehat{x}};
\]

b) if $x$ is eventually periodic, we have%
\[
L^{x}=\underset{\mu\in\omega(x)}{%
{\displaystyle\bigcap}
}L_{\mu}^{x}.
\]

\end{theorem}

\begin{proof}
a) The hypothesis states $\widehat{P}^{\widehat{x}}\neq\varnothing$ and let
$p\in\widehat{P}^{\widehat{x}}.$ We show that $\widehat{L}^{\widehat{x}%
}\subset\underset{\mu\in\widehat{\omega}(\widehat{x})}{%
{\displaystyle\bigcap}
}\widehat{L}_{\mu}^{\widehat{x}}$ and we take for this $k^{\prime}\in
\widehat{L}^{\widehat{x}}$ arbitrary, thus%
\begin{equation}
\forall k\geq k^{\prime},\widehat{x}(k)=\widehat{x}(k+p). \label{p303}%
\end{equation}
Starting from Theorem \ref{The109}, page \pageref{The109}, the proof of
(\ref{pre341})$_{page\;\pageref{pre341}}\Longrightarrow$(\ref{per164}%
)$_{page\;\pageref{per164}}$ it is shown that (\ref{p303}) implies
\begin{equation}
\left\{
\begin{array}
[c]{c}%
\forall\mu\in\widehat{\omega}(\widehat{x}),\forall k\in\widehat{\mathbf{T}%
}_{\mu}^{\widehat{x}}\cap\{k^{\prime},k^{\prime}+1,k^{\prime}+2,...\},\\
\{k+zp|z\in\mathbf{Z}\}\cap\{k^{\prime},k^{\prime}+1,k^{\prime}+2,...\}\subset
\widehat{\mathbf{T}}_{\mu}^{\widehat{x}},
\end{array}
\right.  \label{p304}%
\end{equation}
wherefrom we have that $k^{\prime}\in\underset{\mu\in\widehat{\omega}%
(\widehat{x})}{%
{\displaystyle\bigcap}
}\widehat{L}_{\mu}^{\widehat{x}}.$

We show that $\underset{\mu\in\widehat{\omega}(\widehat{x})}{%
{\displaystyle\bigcap}
}\widehat{L}_{\mu}^{\widehat{x}}\subset\widehat{L}^{\widehat{x}}$ and let for
this $k^{\prime}\underset{\mu\in\widehat{\omega}(\widehat{x})}{\in%
{\displaystyle\bigcap}
}\widehat{L}_{\mu}^{\widehat{x}},$ i.e. (\ref{p304}) holds. Starting from the
implication (\ref{per164})$_{page\;\pageref{per164}}\Longrightarrow
$(\ref{pre341})$_{page\;\pageref{pre341}}$ of Theorem \ref{The109}, page
\pageref{The109}, it is shown the truth of (\ref{p303}), in other words
$k^{\prime}\in\widehat{L}^{\widehat{x}}.$
\end{proof}

\begin{theorem}
\label{The77}a) Let $\widehat{x}\in\widehat{S}^{(n)}$ eventually periodic.
Then $k^{\prime}\in\mathbf{N}_{\_}$ exists with $\widehat{L}^{\widehat{x}%
}=\{k^{\prime},k^{\prime}+1,k^{\prime}+2,...\}.$

b) Let $x\in S^{(n)}$ be eventually periodic and not constant$.$ Then
$t^{\prime}\in\mathbf{R}$ exists such that $L^{x}=[t^{\prime},\infty).$
\end{theorem}

\begin{proof}
a) We put $\widehat{\omega}(\widehat{x})$ under the form $\widehat{\omega
}(\widehat{x})=\{\mu^{1},...,\mu^{s}\},s\geq1.$ Theorem \ref{The117}, page
\pageref{The117} shows the existence of $k_{i}^{\prime}\in\mathbf{N}_{\_}$
that fulfill%
\[
\widehat{L}_{\mu^{i}}^{\widehat{x}}=\{k_{i}^{\prime},k_{i}^{\prime}%
+1,k_{i}^{\prime}+2,...\},i=\overline{1,s}.
\]
We apply Theorem \ref{The139} and we get%
\[
\widehat{L}^{\widehat{x}}=\widehat{L}_{\mu^{1}}^{\widehat{x}}\cap
...\cap\widehat{L}_{\mu^{s}}^{\widehat{x}}=\{k^{\prime},k^{\prime}%
+1,k^{\prime}+2,...\},
\]
where $k^{\prime}=\max\{k_{1}^{\prime},...,k_{s}^{\prime}\}.$
\end{proof}

\section{A property of eventual constancy}

\begin{theorem}
\label{The20}We consider the signals $\widehat{x},x$.

a) If $k^{\prime}\in\mathbf{N}_{\_}$ exists making%
\begin{equation}
\forall k\geq k^{\prime},\widehat{x}(k)=\widehat{x}(k+p) \label{per343}%
\end{equation}
true for $p=1,$ then $\mu\in\widehat{\omega}(\widehat{x})$ exists with%
\begin{equation}
\forall k\geq k^{\prime},\widehat{x}(k)=\mu\label{per346}%
\end{equation}
fulfilled and in this case (\ref{per343}) holds for any $p\geq1.$

b) We suppose that
\begin{equation}%
\begin{array}
[c]{c}%
x(t)=x(-\infty+0)\cdot\chi_{(-\infty,t_{0})}(t)\oplus x(t_{0})\cdot
\chi_{\lbrack t_{0},t_{0}+h)}(t)\oplus...\\
...\oplus x(t_{0}+kh)\cdot\chi_{\lbrack t_{0}+kh,t_{0}+(k+1)h)}(t)\oplus...
\end{array}
\label{per487}%
\end{equation}
is true for $t_{0}\in\mathbf{R}$ and $h>0.$ If%
\begin{equation}
\forall t\geq t^{\prime},x(t)=x(t+T) \label{per344}%
\end{equation}
holds for $t^{\prime}\in\mathbf{R,}$ $T\in(0,h)\cup(h,2h)\cup...\cup
(qh,(q+1)h)\cup...$ then some $\mu\in\omega(x)$ exists such that%
\begin{equation}
\forall t\geq t^{\prime},x(t)=\mu\label{per347}%
\end{equation}
and in this case (\ref{per344}) is true for any $T>0.$

c) We presume that (\ref{per487}) takes the form%
\begin{equation}%
\begin{array}
[c]{c}%
x(t)=\widehat{x}(-1)\cdot\chi_{(-\infty,t_{0})}(t)\oplus\widehat{x}%
(0)\cdot\chi_{\lbrack t_{0},t_{0}+h)}(t)\oplus...\\
...\oplus\widehat{x}(k)\cdot\chi_{\lbrack t_{0}+kh,t_{0}+(k+1)h)}(t)\oplus...
\end{array}
\label{per345}%
\end{equation}
and let $\mu\in\widehat{\omega}(\widehat{x})=\omega(x)$ be an arbitrary point.

c.1) If $k^{\prime}\in\mathbf{N}_{\_}$ exists such that (\ref{per343}) is true
for $p=1,$ then (\ref{per346}) is fulfilled and $t^{\prime}\in\mathbf{R}$
exists also such that (\ref{per347}) is true. In this case (\ref{per343})
holds for any $p\geq1$ and (\ref{per344}) holds for any $T>0.$

c.2) If $t^{\prime}\in\mathbf{R,}$ $T\in(0,h)\cup(h,2h)\cup...\cup
(qh,(q+1)h)\cup...$ exist making (\ref{per344}) true, then $k^{\prime}%
\in\mathbf{N}_{\_}$ exists such that (\ref{per346}) holds and (\ref{per347})
holds too. Moreover, in this situation (\ref{per343}) is true for any $p\geq1$
and (\ref{per344}) is true for any $T>0.$
\end{theorem}

\begin{proof}
a) Let $k^{\prime}\in\mathbf{N}_{\_}$ be with the property that (\ref{per343})
holds for $p=1,$ i.e.
\begin{equation}
\forall k\geq k^{\prime},\widehat{x}(k)=\widehat{x}(k^{\prime}).
\label{per433}%
\end{equation}
We denote $\widehat{x}(k^{\prime})$ with $\mu$ and this obviously implies that
$\mu\in\widehat{\omega}(\widehat{x}).$ Equation (\ref{per433}) may be
rewritten under the form (\ref{per346}) and%
\[
\forall k\geq k^{\prime},\widehat{x}(k)=\mu=\widehat{x}(k+p)
\]
holds for any $p\geq1.$

b) The hypothesis states the existence of $t_{0}\in\mathbf{R},h>0$ such that
(\ref{per487}) holds and also the existence of $t^{\prime}\in\mathbf{R}$ and
$T\in(0,h)\cup(h,2h)\cup...\cup(qh,(q+1)h)\cup...$ such that (\ref{per344})
holds. We denote $x(t^{\prime})$ with $\mu.$

Let $T\in(0,h)$ be arbitrary. If, against all reason, $x$ does not fulfill
(\ref{per347}), the time instant $t_{0}^{\prime}>t^{\prime}$ exists such that%
\begin{equation}
\forall t\in\lbrack t^{\prime},t_{0}^{\prime}),x(t)=\mu, \label{per326_}%
\end{equation}%
\begin{equation}
x(t_{0}^{\prime})\neq\mu. \label{per327_}%
\end{equation}
Since obviously $t_{0}^{\prime}\geq t_{0},$ we have the existence of $k_{0}%
\in\mathbf{N}$ such that $t_{0}^{\prime}\in\lbrack t_{0}+k_{0}h,t_{0}%
+(k_{0}+1)h).$ As $\forall t\in\lbrack t_{0}+k_{0}h,t_{0}+(k_{0}+1)h),$
$x(t)=x(t_{0}+k_{0}h),$ we get $t_{0}^{\prime}=t_{0}+k_{0}h.$ With the
notation $\widetilde{t}=\max\{t_{0}^{\prime}-T,t^{\prime}\},$ we infer
$\widetilde{t}<t_{0}^{\prime}$ and for any $t^{\prime\prime}\in(\widetilde
{t},t_{0}^{\prime}),$ we have%
\begin{equation}
t^{\prime}<t^{\prime\prime}<t_{0}^{\prime}<t^{\prime\prime}+T<t_{0}^{\prime
}+T<t_{0}^{\prime}+h. \label{per743}%
\end{equation}
We deduce%
\begin{equation}
x(t^{\prime\prime}+T)=x(t_{0}^{\prime}+T), \label{per328_}%
\end{equation}
as far as both previous terms are equal with $x(t_{0}^{\prime}),$ and%
\[
\mu\overset{(\ref{per326_}),(\ref{per743})}{=}x(t^{\prime\prime}%
)\overset{(\ref{per344}),(\ref{per743})}{=}x(t^{\prime\prime}+T)\overset
{(\ref{per328_})}{=}x(t_{0}^{\prime}+T)\overset{(\ref{per344}),(\ref{per743}%
)}{=}x(t_{0}^{\prime})\overset{(\ref{per327_})}{\neq}\mu,
\]
contradiction showing that a $t_{0}^{\prime}$ that makes true (\ref{per326_}),
(\ref{per327_}) does not exist.

The case when for $q\geq1,$ we have that $T\in(qh,(q+1)h)$ is similar with the
previous one. (\ref{per344}) continues to be true for some $t^{\prime}%
\in\mathbf{R}$ and if, against all reason, $x$ does not fulfill (\ref{per347}%
), we get that $t_{q}^{\prime}>t^{\prime}$ exists with
\begin{equation}
\forall t\in\lbrack t^{\prime},t_{q}^{\prime}),x(t)=\mu, \label{per329_}%
\end{equation}%
\begin{equation}
x(t_{q}^{\prime})\neq\mu. \label{per330_}%
\end{equation}
Thus $k_{q}\in\mathbf{N}$ exists such that $t_{q}^{\prime}\in\lbrack
t_{0}+k_{q}h,t_{0}+(k_{q}+1)h)$ and, from the fact that $\forall t\in\lbrack
t_{0}+k_{q}h,t_{0}+(k_{q}+1)h),$ we get $x(t)=x(t_{0}+k_{q}h),$ the conclusion
is $t_{q}^{\prime}=t_{0}+k_{q}h.$ With the notation $\widetilde{t}=\max
\{t_{q}^{\prime}+qh-T,t^{\prime}\},$ we obtain $\widetilde{t}<t_{q}^{\prime}$
and for any $t^{\prime\prime}\in(\widetilde{t},t_{q}^{\prime})$ we have%
\begin{equation}
t^{\prime}<t^{\prime\prime}<t_{q}^{\prime}<t_{q}^{\prime}+qh<t^{\prime\prime
}+T<t_{q}^{\prime}+T<t_{q}^{\prime}+(q+1)h. \label{per744}%
\end{equation}
We infer%
\begin{equation}
x(t^{\prime\prime}+T)=x(t_{q}^{\prime}+T), \label{per331_}%
\end{equation}
because both previous terms are equal with $x(t_{q}^{\prime}+qh)$ and%
\[
\mu\overset{(\ref{per329_}),(\ref{per744})}{=}x(t^{\prime\prime}%
)\overset{(\ref{per344}),(\ref{per744})}{=}x(t^{\prime\prime}+T)\overset
{(\ref{per331_})}{=}x(t_{q}^{\prime}+T)\overset{(\ref{per344}),(\ref{per744}%
)}{=}x(t_{q}^{\prime})\overset{(\ref{per330_})}{\neq}\mu
\]
contradiction, in other words a $t_{q}^{\prime}\in\mathbf{R}$ that makes
(\ref{per329_}), (\ref{per330_}) true does not exist. Thus $x$ fulfills
(\ref{per347}) and in such circumstances (\ref{per344}) is true for any $T>0.$
\end{proof}

\section{Discussion on eventual constancy}

\begin{remark}
The point is that Theorem \ref{The19}, page \pageref{The19} and Theorem
\ref{The20}, page \pageref{The20} express the same idea, meaning that in the
situation when $\widehat{x},x$ are related by%
\[%
\begin{array}
[c]{c}%
x(t)=\widehat{x}(-1)\cdot\chi_{(-\infty,t_{0})}(t)\oplus\widehat{x}%
(0)\cdot\chi_{\lbrack t_{0},t_{0}+h)}(t)\oplus...\\
...\oplus\widehat{x}(k)\cdot\chi_{\lbrack t_{0}+kh,t_{0}+(k+1)h)}(t)\oplus...
\end{array}
\]
any of a)%
\begin{equation}
\left\{
\begin{array}
[c]{c}%
\forall\mu\in\widehat{\omega}(\widehat{x}),\forall k\in\widehat{\mathbf{T}%
}_{\mu}^{\widehat{x}}\cap\{k^{\prime},k^{\prime}+1,k^{\prime}+2,...\},\\
\{k+zp|z\in\mathbf{Z}\}\cap\{k^{\prime},k^{\prime}+1,k^{\prime}+2,...\}\subset
\widehat{\mathbf{T}}_{\mu}^{\widehat{x}},
\end{array}
\right.  \label{p97}%
\end{equation}
or%
\begin{equation}
\forall k\geq k^{\prime},\widehat{x}(k)=\widehat{x}(k+p) \label{p98}%
\end{equation}
true for $p=1$ and some $k^{\prime}\in\mathbf{N}_{\_},$

b)%
\begin{equation}
\forall\mu\in\omega(x),\forall t\in\mathbf{T}_{\mu}^{x}\cap\lbrack t^{\prime
},\infty),\{t+zT|z\in\mathbf{Z}\}\cap\lbrack t^{\prime},\infty)\subset
\mathbf{T}_{\mu}^{x}, \label{p99}%
\end{equation}
or%
\begin{equation}
\forall t\geq t^{\prime},x(t)=x(t+T) \label{p100}%
\end{equation}
true for $T\in(0,h)\cup(h,2h)\cup...\cup(qh,(q+1)h)\cup...$ and some
$t^{\prime}\in\mathbf{R}$

implies the truth of%
\begin{equation}
\forall k\geq k^{\prime},\widehat{x}(k)=\mu, \label{p101}%
\end{equation}%
\begin{equation}
\forall t\geq t^{\prime},x(t)=\mu\label{p102}%
\end{equation}
meaning in particular that $\widehat{x},x$ are eventually equal with the same
constant $\mu$. However Theorem \ref{The109}, page \pageref{The109} states the
equivalence, for any $p\geq1,k^{\prime}\in\mathbf{N}_{\_}$ between (\ref{p97})
and (\ref{p98}) and also the equivalence, for any $T>0,t^{\prime}\in
\mathbf{R}$ between (\ref{p99}) and (\ref{p100}), thus the fact that Theorems
\ref{The19} and \ref{The20} give the same conclusion is natural.
\end{remark}

\section{Discrete time vs real time}

\begin{theorem}
\label{The21}We suppose that $\widehat{x},x$ are related by%
\begin{equation}%
\begin{array}
[c]{c}%
x(t)=\widehat{x}(-1)\cdot\chi_{(-\infty,t_{0})}(t)\oplus\widehat{x}%
(0)\cdot\chi_{\lbrack t_{0},t_{0}+h)}(t)\oplus...\\
...\oplus\widehat{x}(k)\cdot\chi_{\lbrack t_{0}+kh,t_{0}+(k+1)h)}(t)\oplus...
\end{array}
\label{per106}%
\end{equation}
where $t_{0}\in\mathbf{R},h>0.$ The existence of $p\geq1$ and $k^{\prime}%
\in\mathbf{N}_{\_}$ such that%
\begin{equation}
\forall k\geq k^{\prime},\widehat{x}(k)=\widehat{x}(k+p), \label{per563}%
\end{equation}
implies the existence of $t^{\prime}\in\mathbf{R}$ such that%
\begin{equation}
\forall t\geq t^{\prime},x(t)=x(t+T) \label{per565}%
\end{equation}
is true for $T=ph.$
\end{theorem}

\begin{proof}
The equation (\ref{per106}) is true for some $t_{0}\in\mathbf{R},h>0$ and
$p\geq1,k^{\prime}\in\mathbf{N}_{\_}$ exist having the property that
(\ref{per563}) holds. We use the notations $T=ph,t^{\prime}=t_{0}+k^{\prime}h$
and let $t\geq t^{\prime}$ be arbitrary, fixed. Some $k\geq k^{\prime}$ exists
with the property $t\in\lbrack t_{0}+kh,t_{0}+(k+1)h),$ wherefrom
$t+T\in\lbrack t_{0}+(k+p)h,t_{0}+(k+1+p)h)$ and we finally infer that%
\[
x(t)=\widehat{x}(k)\overset{(\ref{per563})}{=}\widehat{x}(k+p)=x(t+T).
\]
Because $t\geq t^{\prime}$ was arbitrarily chosen, we have inferred the truth
of (\ref{per565}).
\end{proof}

\begin{theorem}
\label{The118}If $\widehat{x},x$ are not eventually constant, (\ref{per106})
holds for $t_{0}\in\mathbf{R},h>0$ and $T>0,t^{\prime}\in\mathbf{R}$ exist
such that $x$ fulfills (\ref{per565}), then $\frac{T}{h}\in\{1,2,3,...\}$ and
$k^{\prime}\in\mathbf{N}_{\_}$ exists such that (\ref{per563}) is true for
$p=\frac{T}{h}.$
\end{theorem}

\begin{proof}
Some $t_{0}\in R,h>0$ exist with (\ref{per106}) fulfilled and $T>0,t^{\prime
}\in\mathbf{R}$ exist also with (\ref{per565}) true. If in (\ref{per565}) we
have $T\in(0,h)\cup(h,2h)\cup...\cup(qh,(q+1)h)\cup...$ then, from Theorem
\ref{The20}, page \pageref{The20}, $\mu\in\widehat{\omega}(\widehat{x}%
)=\omega(x)$ and $k^{\prime}\in\mathbf{N}_{\_}$ exist such that $\forall k\geq
k^{\prime},\widehat{x}(k)=\mu$ resulting a contradiction with the hypothesis,
stating that $\widehat{x},x$ are not eventually constant. We suppose from now
that $T\in\{h,2h,3h,...\}.$ We denote $p=\frac{T}{h},$ $p\geq1.$ As far as for
any $t^{\prime\prime}\geq t^{\prime}$ we have%
\begin{equation}
\forall t\geq t^{\prime\prime},x(t)=x(t+T), \label{per443}%
\end{equation}
we can suppose without loosing the generality the existence of $k^{\prime}%
\in\mathbf{N}_{\_}$ with $t^{\prime\prime}=t_{0}+k^{\prime}h.$ In this
situation for any $k\geq k^{\prime}$ and any $t\geq t^{\prime\prime}$ with
$t\in\lbrack t_{0}+kh,t_{0}+(k+1)h)$ we have%
\[
t+T\in\lbrack t_{0}+kh+ph,t_{0}+(k+1)h+ph)=[t_{0}+(k+p)h,t_{0}+(k+1+p)h)
\]
and we can write%
\[
\widehat{x}(k)=x(t)\overset{(\ref{per443})}{=}x(t+T)=\widehat{x}(k+p),
\]
thus (\ref{per563}) is true.
\end{proof}

\begin{example}
We define $\widehat{x}\in\widehat{S}^{(1)}$ by
\[
\forall k\in\mathbf{N}_{\_},\widehat{x}(k)=\left\{
\begin{array}
[c]{c}%
1,if\;k\in\{-1,2,4,6,8,...\}\\
0,otherwise
\end{array}
\right.
\]
and $x\in S^{(1)}$ respectively by%
\[
x(t)=\widehat{x}(-1)\cdot\chi_{(-\infty,-4)}(t)\oplus\widehat{x}(0)\cdot
\chi_{\lbrack-4,-2)}(t)\oplus
\]%
\[
\oplus\widehat{x}(1)\cdot\chi_{\lbrack-2,0)}(t)\oplus\widehat{x}(2)\cdot
\chi_{\lbrack0,2)}(t)\oplus...
\]
We have%
\[
\forall t\geq-2,x(t)=x(t+4),
\]%
\[
\forall k\geq1,\widehat{x}(k)=\widehat{x}(k+2)
\]
thus (\ref{per565}) is fulfilled with $T=4,t^{\prime}=-2$ and (\ref{per563})
is true with $p=2,k^{\prime}=1.$ Furthermore, in this example $h=2.$
\end{example}

\section{Sums, differences and multiples of periods}

\begin{theorem}
\label{The79}Let the signals $\widehat{x},x.$

a) We suppose that $\widehat{x}$ has the periods $p,p^{\prime}\geq1$ and the
limit of periodicity $k^{\prime}\in\mathbf{N}_{\_}:$%
\begin{equation}
\forall k\geq k^{\prime},\widehat{x}(k)=\widehat{x}(k+p), \label{p239}%
\end{equation}%
\begin{equation}
\forall k\geq k^{\prime},\widehat{x}(k)=\widehat{x}(k+p^{\prime}).
\label{p240}%
\end{equation}
Then $p+p^{\prime}\geq1$, $\widehat{x}$ has the period $p+p^{\prime}$ and the
limit of periodicity $k^{\prime}$%
\begin{equation}
\forall k\geq k^{\prime},\widehat{x}(k)=\widehat{x}(k+p+p^{\prime})
\label{p241}%
\end{equation}
and if $p>p^{\prime},$ then $p-p^{\prime}\geq1,$ $\widehat{x}$ has the period
$p-p^{\prime}$ and the limit of periodicity $k^{\prime}$%
\begin{equation}
\forall k\geq k^{\prime},\widehat{x}(k)=\widehat{x}(k+p-p^{\prime}).
\label{p242}%
\end{equation}

b) Let $T,T^{\prime}>0,$ $t^{\prime}\in\mathbf{R}$ be arbitrary with%
\begin{equation}
\forall t\geq t^{\prime},x(t)=x(t+T), \label{p243}%
\end{equation}%
\begin{equation}
\forall t\geq t^{\prime},x(t)=x(t+T^{\prime}) \label{p244}%
\end{equation}
fulfilled. We have on one hand that $T+T^{\prime}>0$ and%
\begin{equation}
\forall t\geq t^{\prime},x(t)=x(t+T+T^{\prime}), \label{p245}%
\end{equation}
and on the other hand that $T>T^{\prime}$ implies $T-T^{\prime}>0$ and%
\begin{equation}
\forall t\geq t^{\prime},x(t)=x(t+T-T^{\prime}). \label{p246}%
\end{equation}

\end{theorem}

\begin{proof}
a) Let $k\geq k^{\prime}$ be arbitrary and fixed. Then%
\[
\widehat{x}(k)\overset{(\ref{p239})}{=}\widehat{x}(k+p)\overset{(\ref{p240}%
)}{=}\widehat{x}(k+p+p^{\prime}).
\]
We suppose now that $p>p^{\prime},$ thus $k+p-p^{\prime}\geq k^{\prime}.$ We
can write that%
\[
\widehat{x}(k+p-p^{\prime})\overset{(\ref{p240})}{=}\widehat{x}(k+p)\overset
{(\ref{p239})}{=}\widehat{x}(k).
\]

\end{proof}

\begin{theorem}
\label{The33}We consider the signals $\widehat{x},x.$

a) Let $p,k_{1}\geq1$ and $k^{\prime}\in\mathbf{N}_{\_}.$ Then $p^{\prime
}=k_{1}p$ fulfills $p^{\prime}\geq1$ and%
\begin{equation}
\forall k\geq k^{\prime},\widehat{x}(k)=\widehat{x}(k+p) \label{per556}%
\end{equation}
implies%
\begin{equation}
\forall k\geq k^{\prime},\widehat{x}(k)=\widehat{x}(k+p^{\prime}).
\label{per557}%
\end{equation}

b) Let $T>0,k_{1}\geq1$ and $t^{\prime}\in\mathbf{R}.$ Then $T^{\prime}%
=k_{1}T$ fulfills $T^{\prime}>0$ and%
\begin{equation}
\forall t\geq t^{\prime},x(t)=x(t+T) \label{per560}%
\end{equation}
implies%
\begin{equation}
\forall t\geq t^{\prime},x(t)=x(t+T^{\prime}). \label{per561}%
\end{equation}

\end{theorem}

\begin{proof}
This is a consequence of Theorem \ref{The79}.
\end{proof}

\begin{corollary}
\label{Cor7}a) If $p\in\widehat{P}^{\widehat{x}},$ then
$\{p,2p,3p,...\}\subset\widehat{P}^{\widehat{x}}.$

b) If $T\in P^{x},$ then $\{T,2T,3T,...\}\subset P^{x}.$
\end{corollary}

\begin{proof}
This follows from Theorem \ref{The33}.
\end{proof}

\section{The set of the periods}

\begin{theorem}
\label{The124}a) We suppose that for $\widehat{x}\in\widehat{S}^{(n)},$ the
set $\widehat{P}^{\widehat{x}}$ is not empty. Some $\widetilde{p}\geq1$ exists
then with the property%
\begin{equation}
\widehat{P}^{\widehat{x}}=\{\widetilde{p},2\widetilde{p},3\widetilde{p},...\}.
\label{pre233}%
\end{equation}

b) Let $x\in S^{(n)}$ be not eventually constant and we suppose that the set
$P^{x}$ is not empty. Then $\widetilde{T}>0$ exists such that%
\begin{equation}
P^{x}=\{\widetilde{T},2\widetilde{T},3\widetilde{T},...\}. \label{pre234}%
\end{equation}

\end{theorem}

\begin{proof}
a) We have $\widehat{P}^{\widehat{x}}\neq\varnothing$ and we denote with
$\widetilde{p}\geq1$ its minimum$.$ The inclusion $\{\widetilde{p}%
,2\widetilde{p},3\widetilde{p},...\}\subset\widehat{P}^{\widehat{x}}$ was
stated in Corollary \ref{Cor7}$.$ In order to prove that $\widehat
{P}^{\widehat{x}}\subset\{\widetilde{p},2\widetilde{p},3\widetilde{p},...\},$
we suppose against all reason that $p^{\prime}\in\widehat{P}^{\widehat{x}}$
exists with the property that $p^{\prime}\notin\{\widetilde{p},2\widetilde
{p},3\widetilde{p},...\}$ and consequently $k\geq1$ exists such that
$k\widetilde{p}<p^{\prime}<(k+1)\widetilde{p}.$ We infer that $1\leq
p^{\prime}-k\widetilde{p}<\widetilde{p}$ and, from Theorem \ref{The79}, page
\pageref{The79} and \ref{The33}, page \pageref{The33} that $p^{\prime
}-k\widetilde{p}\in\widehat{P}^{\widehat{x}}.$ This fact is in contradiction
however with the supposition that $\widetilde{p}=\min\widehat{P}^{\widehat{x}%
}.$

b) We proceed in two steps. At b.i) we prove that $\min P^{x}$ exists and at
b.ii) we prove that the only elements of $P^{x}$ are the multiples of $\min
P^{x}.$

b.i) We suppose against all reason that $\min P^{x}$ does not exist, namely
that a strictly decreasing sequence $T_{k}\in P^{x},k\in\mathbf{N}$ exists
that is convergent to $T=\inf P^{x}.$ As $x$ is not eventually constant, the
following property%
\begin{equation}
\forall t\in\mathbf{R},\exists t^{\prime\prime}>t,x(t^{\prime\prime}-0)\neq
x(t^{\prime\prime}) \label{p65_}%
\end{equation}
is true, from Lemma \ref{Lem38}, page \pageref{Lem38}. The hypothesis states
the existence $\forall k\in\mathbf{N},$ of $t_{k}^{\prime}\in\mathbf{R}$ with%
\begin{equation}
\forall t\geq t_{k}^{\prime},x(t)=x(t+T_{k}). \label{p67_}%
\end{equation}
As $t_{k}^{\prime}$ do not depend on $T_{k},$ see Theorem \ref{The141}, page
\pageref{The141}, we can suppose that they are all equal with some $t^{\prime
}\in\mathbf{R}.$ Property (\ref{p65_}) implies that we can take a
$t^{\prime\prime}>t^{\prime}$ such that $x(t^{\prime\prime}-0)\neq
x(t^{\prime\prime})$ and we can apply now Lemma \ref{Lem10}, page
\pageref{Lem10} giving
\begin{equation}
\forall k\in\mathbf{N},x(t^{\prime\prime}+T_{k}-0)\neq x(t^{\prime\prime
}+T_{k}). \label{p68_}%
\end{equation}
We infer from Lemma \ref{Lem4_}, page \pageref{Lem4_} that $N\in\mathbf{N}$
exists with $\forall k\geq N,$%
\[
x(t^{\prime\prime}+T_{k}-0)=x(t^{\prime\prime}+T_{k})=x(t^{\prime\prime}+T),
\]
contradiction with (\ref{p68_}). It has resulted that such a sequence
$T_{k},k\in\mathbf{N}$ does not exist, thus $P^{x},$ which is non-empty, has a
minimum $\widetilde{T}>0.$

b.ii) We have from Corollary \ref{Cor7} that $\{\widetilde{T},2\widetilde
{T},3\widetilde{T},...\}\subset P^{x}.$ We prove that $\forall k\geq1,\forall
T^{\prime}\in(k\widetilde{T},(k+1)\widetilde{T}),T^{\prime}\notin P^{x}.$ We
suppose against all reason that $k\geq1$ and $T^{\prime}\in(k\widetilde
{T},(k+1)\widetilde{T})$ exist such that $T^{\prime}\in P^{x}.$ This means,
from Theorem \ref{The79}, page \pageref{The79} and Theorem \ref{The33}, page
\pageref{The33}, that $T^{\prime}-k\widetilde{T}\in P^{x}$ and since
$T^{\prime}-k\widetilde{T}<\widetilde{T},$ we have obtained a contradiction
with the fact that $\widetilde{T}=\min P^{x}.$
\end{proof}

\begin{theorem}
\label{The87}We suppose that the relation between $\widehat{x}$ and $x$ is
given by%
\[%
\begin{array}
[c]{c}%
x(t)=\widehat{x}(-1)\cdot\chi_{(-\infty,t_{0})}(t)\oplus\widehat{x}%
(0)\cdot\chi_{\lbrack t_{0},t_{0}+h)}(t)\oplus\widehat{x}(1)\cdot\chi_{\lbrack
t_{0}+h,t_{0}+2h)}(t)\oplus...\\
...\oplus\widehat{x}(k)\cdot\chi_{\lbrack t_{0}+kh,t_{0}+(k+1)h)}(t)\oplus...
\end{array}
\]
where $t_{0}\in\mathbf{R}$ and $h>0.$ If $\widehat{x},$ $x$ are eventually
periodic, two possibilities exist:

a) $\widehat{x},$ $x$ are both eventually constant, $p=1$ is the prime period
of $\widehat{x}$ and $x$ has no prime period;

b) neither of $\widehat{x},$ $x$ is eventually constant, $\widetilde{p}>1$ is
the prime period of $\widehat{x}$ and $\widetilde{T}=\widetilde{p}h$ is the
prime period of $x$.
\end{theorem}

\begin{proof}
$\widehat{x},x$ are simultaneously eventually constant or not. Let us suppose
that they are not eventually constant and we prove b). Theorem \ref{The21},
page \pageref{The21} shows that $p\in\widehat{P}^{\widehat{x}}\Longrightarrow
T=ph\in P^{x}$ and conversely, Theorem \ref{The118}, page \pageref{The118}
shows that $T\in P^{x}\Longrightarrow p=\frac{T}{h}\in\widehat{P}^{\widehat
{x}}.$ From Theorem \ref{The124} we have that $\widehat{P}^{\widehat{x}%
}=\{\widetilde{p},2\widetilde{p},3\widetilde{p},...\},P^{x}=\{\widetilde
{T},2\widetilde{T},3\widetilde{T},...\},$ thus $\widetilde{T}=\widetilde{p}h.$
\end{proof}

\begin{theorem}
\label{The110}a) If $\widehat{x}$ is eventually periodic and $\widehat{\omega
}(\widehat{x})=\{\mu^{1},...,\mu^{s}\}$, then%
\[
\widehat{P}^{\widehat{x}}=\widehat{P}_{\mu^{1}}^{\widehat{x}}\cap
...\cap\widehat{P}_{\mu^{s}}^{\widehat{x}}.
\]

b) We suppose that $x$ is eventually periodic and $\omega(x)=\{\mu^{1}%
,...,\mu^{s}\}$ holds. In this case%
\[
P^{x}=P_{\mu^{1}}^{x}\cap...\cap P_{\mu^{s}}^{x}.
\]

\end{theorem}

\begin{proof}
a) In order to prove that $\widehat{P}^{\widehat{x}}\subset\widehat{P}%
_{\mu^{1}}^{\widehat{x}}\cap...\cap\widehat{P}_{\mu^{s}}^{\widehat{x}},$ we
take an arbitrary $p\in\widehat{P}^{\widehat{x}},$ thus%
\begin{equation}
\left\{
\begin{array}
[c]{c}%
\forall i\in\{1,...,n\},\exists k_{i}^{\prime}\in\mathbf{N}_{\_},\forall
k\in\widehat{\mathbf{T}}_{\mu^{i}}^{\widehat{x}}\cap\{k_{i}^{\prime}%
,k_{i}^{\prime}+1,k_{i}^{\prime}+2,...\},\\
\{k+zp|z\in\mathbf{Z}\}\cap\{k_{i}^{\prime},k_{i}^{\prime}+1,k_{i}^{\prime
}+2,...\}\subset\widehat{\mathbf{T}}_{\mu^{i}}^{\widehat{x}}%
\end{array}
\right.
\end{equation}
holds. This means that $\forall i\in\{1,...,n\},$ $p\in\widehat{P}_{\mu^{i}%
}^{\widehat{x}}.$

The inclusion $\widehat{P}_{\mu^{1}}^{\widehat{x}}\cap...\cap\widehat{P}%
_{\mu^{s}}^{\widehat{x}}\subset\widehat{P}^{\widehat{x}}$ is obvious too.
\end{proof}

\section{Necessity conditions of eventual periodicity}

\begin{theorem}
\label{The132}Let $\widehat{x}\in\widehat{S}^{(n)}$ with $\widehat{\omega
}(\widehat{x})=\{\mu^{1},...,\mu^{s}\}.$ We suppose that $\widehat{x}$ is
eventually periodic with the period $p\geq1$ and the limit of periodicity
$k^{\prime}\in\mathbf{N}_{\_}.$ Then $n_{1}^{i},n_{2}^{i},...,n_{k_{i}}^{i}%
\in\{k^{\prime},k^{\prime}+1,...,k^{\prime}+p-1\}$ exist$,$ $k_{i}\geq1, $
such that%
\begin{equation}
\widehat{\mathbf{T}}_{\mu^{i}}^{\widehat{x}}\cap\{k^{\prime},k^{\prime
}+1,k^{\prime}+2,...\}=\underset{k\in\mathbf{N}}{%
{\displaystyle\bigcup}
}\{n_{1}^{i}+kp,n_{2}^{i}+zp,...,n_{k_{i}}^{i}+kp\}
\end{equation}
for $i\in\{1,...,s\}.$
\end{theorem}

\begin{proof}
If $\widehat{x}$ is eventually periodic with the period $p$ and the limit of
periodicity $k^{\prime},$ then every $\mu^{i}\in\widehat{\omega}(\widehat{x})$
is eventually periodic with the period $p$ and the limit of periodicity
$k^{\prime},i\in\{1,...,s\}$ and we apply Theorem \ref{The69}, page
\pageref{The69}.
\end{proof}

\begin{theorem}
\label{The138}We consider the non eventually constant signal $x\in S^{(n)}$
and we put the omega limit set under the form $\omega(x)=\{\mu^{1},...,\mu
^{s}\},$ $s\geq2.$ We suppose that $x$ is eventually periodic with the period
$T>0$ and the limit of periodicity $t^{\prime}\in\mathbf{R.}$ Then $a_{1}%
^{i},b_{1}^{i},a_{2}^{i},b_{2}^{i},...,a_{k_{i}}^{i},b_{k_{i}}^{i}%
\in\mathbf{R}$ exist,
\begin{equation}
t^{\prime}\leq a_{1}^{i}<b_{1}^{i}<a_{2}^{i}<b_{2}^{i}<...<a_{k_{i}}%
^{i}<b_{k_{i}}^{i}\leq t^{\prime}+T, \label{pre237}%
\end{equation}
with $k_{i}\geq1,$ $i\in\{1,...,s\},$ such that%
\begin{equation}
\lbrack a_{1}^{i},b_{1}^{i})\cup\lbrack a_{2}^{i},b_{2}^{i})\cup...\cup\lbrack
a_{k_{i}}^{i},b_{k_{i}}^{i})=\mathbf{T}_{\mu^{i}}^{x}\cap\lbrack t^{\prime
},t^{\prime}+T),
\end{equation}%
\begin{equation}%
\begin{array}
[c]{c}%
\mathbf{T}_{\mu^{i}}^{x}\cap\lbrack t^{\prime},\infty)=\\
\underset{k\in\mathbf{N}}{%
{\displaystyle\bigcup}
}([a_{1}^{i}+kT,b_{1}^{i}+kT)\cup\lbrack a_{2}^{i}+kT,b_{2}^{i}+kT)\cup
...\cup\lbrack a_{k_{i}}^{i}+kT,b_{k_{i}}^{i}+kT))
\end{array}
\label{pre238}%
\end{equation}
hold for $i\in\{1,...,s\}.$
\end{theorem}

\begin{proof}
a) $x$ is eventually periodic, with the period $T$ and the limit of
periodicity $t^{\prime},$ thus $\forall i\in\{1,...,s\},\mu^{i}$ is eventually
periodic with the period $T$ and the limit of periodicity $t^{\prime}.$ We
apply Theorem \ref{The71}, page \pageref{The71}.
\end{proof}

\begin{example}
The eventually periodic signal $x\in S^{(1)},$%
\[
x(t)=\chi_{(-\infty,0)}(t)\oplus\chi_{\lbrack1,5)}(t)\oplus\chi_{\lbrack
6,7)}(t)\oplus\chi_{\lbrack8,10)}(t)\oplus\chi_{\lbrack11,12)}(t)\oplus
\chi_{\lbrack13,15)}(t)\oplus...
\]
fulfills $\mu^{1}=1,\mu^{2}=0,k_{1}=k_{2}=2,T=5,t^{\prime}=3,a_{1}^{1}%
=3,b_{1}^{1}=5,a_{2}^{1}=6,b_{2}^{1}=7,a_{1}^{2}=5,b_{1}^{2}=6,a_{2}%
^{2}=7,b_{2}^{2}=8. $
\end{example}

\section{Sufficiency conditions of eventual periodicity}

\begin{theorem}
\label{The119}Let $\widehat{x}\in\widehat{S}^{(n)}$, $\widehat{\omega
}(\widehat{x})=\{\mu^{1},...,\mu^{s}\}$ and $p\geq1,k^{\prime}\in
\mathbf{N}_{\_}.$ We ask that for any $i\in\{1,...,s\},$ the numbers
$n_{1}^{i},n_{2}^{i},...,n_{k_{i}}^{i}\in\{k^{\prime},k^{\prime}%
+1,...,k^{\prime}+p-1\}$ exist, $k_{i}\geq1,$ making%
\begin{equation}
\widehat{\mathbf{T}}_{\mu^{i}}^{\widehat{x}}\cap\{k^{\prime},k^{\prime
}+1,k^{\prime}+2,...\}=\underset{k\in\mathbf{N}}{%
{\displaystyle\bigcup}
}\{n_{1}^{i}+kp,n_{2}^{i}+kp,...,n_{k_{i}}^{i}+kp\} \label{pre240}%
\end{equation}
true. Then $\widehat{x}$ is eventually periodic with the period $p\geq1$ and
the limit of periodicity $k^{\prime}\in\mathbf{N}_{\_}:\forall i\in
\{1,...,s\},$%
\begin{equation}
\forall k\in\widehat{\mathbf{T}}_{\mu^{i}}^{\widehat{x}}\cap\{k^{\prime
},k^{\prime}+1,k^{\prime}+2,...\},\{k+zp|z\in\mathbf{Z}\}\cap\{k^{\prime
},k^{\prime}+1,k^{\prime}+2,...\}\subset\widehat{\mathbf{T}}_{\mu^{i}%
}^{\widehat{x}}. \label{pre241}%
\end{equation}

\end{theorem}

\begin{proof}
We suppose that $\forall i\in\{1,...,s\},k_{i}\geq1$ and $n_{1}^{i},n_{2}%
^{i},...,n_{k_{i}}^{i}\in\{k^{\prime},k^{\prime}+1,...,k^{\prime}+p-1\}$ exist
such that (\ref{pre240}) holds. We infer from Theorem \ref{The72}, page
\pageref{The72} that $\mu^{1},...,\mu^{s}$ are all eventually periodic with
the period $p$ and the limit of periodicity $k^{\prime},$ i.e. $\widehat{x}$
is eventually periodic with the period $p$ and the limit of periodicity
$k^{\prime}$, the equivalence between (\ref{per163})$_{page\;\pageref{per163}%
}$ and (\ref{per164})$_{page\;\pageref{per164}}$ was proved at Theorem
\ref{The109}, page \pageref{The109}.
\end{proof}

\begin{theorem}
\label{The133}Let the signal $x\in S^{(n)},\omega(x)=\{\mu^{1},...,\mu^{s}\},$
$s\geq2$ and $T>0,t^{\prime}\in\mathbf{R}.$ For all $i\in\{1,...,s\},$ the
numbers $a_{1}^{i},b_{1}^{i},a_{2}^{i},b_{2}^{i},...,a_{k_{i}}^{i},b_{k_{i}%
}^{i}\in\mathbf{R},$ $k_{i}\geq1$ are given with the property that%
\begin{equation}
t^{\prime}\leq a_{1}^{i}<b_{1}^{i}<a_{2}^{i}<b_{2}^{i}<...<a_{k_{i}}%
^{i}<b_{k_{i}}^{i}\leq t^{\prime}+T, \label{pre242}%
\end{equation}%
\begin{equation}%
\begin{array}
[c]{c}%
\mathbf{T}_{\mu^{i}}^{x}\cap\lbrack t^{\prime},\infty)=\\
\underset{k\in\mathbf{N}}{%
{\displaystyle\bigcup}
}([a_{1}^{i}+kT,b_{1}^{i}+kT)\cup\lbrack a_{2}^{i}+kT,b_{2}^{i}+kT)\cup
...\cup\lbrack a_{k_{i}}^{i}+kT,b_{k_{i}}^{i}+kT))
\end{array}
\label{pre231}%
\end{equation}
hold. Then $x$ is eventually periodic with the period $T$ and the limit of
periodicity $t^{\prime}:\forall i\in\{1,...,s\},$%
\begin{equation}
\forall t\in\mathbf{T}_{\mu^{i}}^{x}\cap\lbrack t^{\prime},\infty
),\{t+zT|z\in\mathbf{Z}\}\cap\lbrack t^{\prime},\infty)\subset\mathbf{T}%
_{\mu^{i}}^{x}. \label{pre243}%
\end{equation}

\end{theorem}

\begin{proof}
This is a consequence of Theorem \ref{The76}, page \pageref{The76}.
\end{proof}

\section{A special case}

\begin{theorem}
\label{The134}Let the signal $\widehat{x}\in\widehat{S}^{(n)}$, $\widehat
{\omega}(\widehat{x})=\{\mu^{1},...,\mu^{s}\}$ and $p\geq1,k^{\prime}%
\in\mathbf{N}_{\_}.$ We ask that $\forall i\in\{1,...,s\},$ $n^{i}%
\in\{k^{\prime},k^{\prime}+1,...,k^{\prime}+p-1\}$ exists such that%
\begin{equation}
\widehat{\mathbf{T}}_{\mu^{i}}^{\widehat{x}}\cap\{k^{\prime},k^{\prime
}+1,k^{\prime}+2,...\}=\{n^{i},n^{i}+p,n^{i}+2p,...\}.
\end{equation}

a) We have: $\forall i\in\{1,...,s\},$%
\begin{equation}
\forall k\in\widehat{\mathbf{T}}_{\mu^{i}}^{\widehat{x}}\cap\{k^{\prime
},k^{\prime}+1,k^{\prime}+2,...\},\{k+zp|z\in\mathbf{Z}\}\cap\{k^{\prime
},k^{\prime}+1,k^{\prime}+2,...\}\subset\widehat{\mathbf{T}}_{\mu^{i}%
}^{\widehat{x}}.
\end{equation}

b) $p$ is the prime period of $\widehat{x}:$ for any $p^{\prime}$ and
$k^{\prime\prime}$ with $\forall i\in\{1,...,s\},$%
\begin{equation}%
\begin{array}
[c]{c}%
\forall k\in\widehat{\mathbf{T}}_{\mu^{i}}^{\widehat{x}}\cap\{k^{\prime\prime
},k^{\prime\prime}+1,k^{\prime\prime}+2,...\},\\
\{k+zp^{\prime}|z\in\mathbf{Z}\}\cap\{k^{\prime\prime},k^{\prime\prime
}+1,k^{\prime\prime}+2,...\}\subset\widehat{\mathbf{T}}_{\mu^{i}}^{\widehat
{x}},
\end{array}
\end{equation}
we have $p^{\prime}\in\{p,2p,3p,...\}.$
\end{theorem}

\begin{proof}
This follows from Theorem \ref{The122}, page \pageref{The122}.
\end{proof}

\begin{theorem}
\label{The135}We consider the signal $x$ with $\omega(x)=\{\mu^{1},...,\mu
^{s}\}$ and $T>0,t^{\prime}\in\mathbf{R}.$ For all $i\in\{1,...,s\},$ the
intervals $[a^{i},b^{i})\subset\lbrack t^{\prime},t^{\prime}+T)$ are given
with%
\begin{equation}
\mathbf{T}_{\mu^{i}}^{x}\cap\lbrack t^{\prime},\infty)=[a^{i},b^{i}%
)\cup\lbrack a^{i}+T,b^{i}+T)\cup\lbrack a^{i}+2T,b^{i}+2T)\cup....
\end{equation}
true. Then

a) $x$ is eventually periodic with the period $T$ and the limit of periodicity
$t^{\prime}:\forall i\in\{1,...,s\},$%
\begin{equation}
\forall t\in\mathbf{T}_{\mu^{i}}^{x}\cap\lbrack t^{\prime},\infty
),\{t+zT|z\in\mathbf{Z}\}\cap\lbrack t^{\prime},\infty)\subset\mathbf{T}%
_{\mu^{i}}^{x}.
\end{equation}

b) if $x$ is not eventually constant, $T$ is the prime period of $x,$ i.e. for
any $T^{\prime}$ and $t^{\prime\prime}$ with $\forall i\in\{1,...,s\},$%
\begin{equation}
\forall t\in\mathbf{T}_{\mu^{i}}^{x}\cap\lbrack t^{\prime\prime}%
,\infty),\{t+zT^{\prime}|z\in\mathbf{Z}\}\cap\lbrack t^{\prime\prime}%
,\infty)\subset\mathbf{T}_{\mu^{i}}^{x},
\end{equation}
we infer $T^{\prime}\in\{T,2T,3T,...\}.$
\end{theorem}

\begin{proof}
This follows from Theorem \ref{The116}, page \pageref{The116}.
\end{proof}

\section{Changing the order of the quantifiers}

\begin{theorem}
\label{The136}\footnote{This Theorem is partially without proof.}a) The
statements%
\begin{equation}
\left\{
\begin{array}
[c]{c}%
\exists p\geq1,\exists k^{\prime}\in\mathbf{N}_{\_},\forall\mu\in
\widehat{\omega}(\widehat{x}),\forall k\in\widehat{\mathbf{T}}_{\mu}%
^{\widehat{x}}\cap\{k^{\prime},k^{\prime}+1,k^{\prime}+2,...\},\\
\{k+zp|z\in\mathbf{Z}\}\cap\{k^{\prime},k^{\prime}+1,k^{\prime}+2,...\}\subset
\widehat{\mathbf{T}}_{\mu}^{\widehat{x}},
\end{array}
\right.  \label{p103}%
\end{equation}%
\begin{equation}
\left\{
\begin{array}
[c]{c}%
\exists p\geq1,\forall\mu\in\widehat{\omega}(\widehat{x}),\exists k^{\prime
}\in\mathbf{N}_{\_},\forall k\in\widehat{\mathbf{T}}_{\mu}^{\widehat{x}}%
\cap\{k^{\prime},k^{\prime}+1,k^{\prime}+2,...\},\\
\{k+zp|z\in\mathbf{Z}\}\cap\{k^{\prime},k^{\prime}+1,k^{\prime}+2,...\}\subset
\widehat{\mathbf{T}}_{\mu}^{\widehat{x}},
\end{array}
\right.  \label{per168}%
\end{equation}%
\begin{equation}
\left\{
\begin{array}
[c]{c}%
\exists k^{\prime}\in\mathbf{N}_{\_},\forall\mu\in\widehat{\omega}(\widehat
{x}),\exists p\geq1,\forall k\in\widehat{\mathbf{T}}_{\mu}^{\widehat{x}}%
\cap\{k^{\prime},k^{\prime}+1,k^{\prime}+2,...\},\\
\{k+zp|z\in\mathbf{Z}\}\cap\{k^{\prime},k^{\prime}+1,k^{\prime}+2,...\}\subset
\widehat{\mathbf{T}}_{\mu}^{\widehat{x}},
\end{array}
\right.  \label{per739}%
\end{equation}%
\begin{equation}
\left\{
\begin{array}
[c]{c}%
\forall\mu\in\widehat{\omega}(\widehat{x}),\exists p\geq1,\exists k^{\prime
}\in\mathbf{N}_{\_},\forall k\in\widehat{\mathbf{T}}_{\mu}^{\widehat{x}}%
\cap\{k^{\prime},k^{\prime}+1,k^{\prime}+2,...\},\\
\{k+zp|z\in\mathbf{Z}\}\cap\{k^{\prime},k^{\prime}+1,k^{\prime}+2,...\}\subset
\widehat{\mathbf{T}}_{\mu}^{\widehat{x}}%
\end{array}
\right.  \label{per148_}%
\end{equation}
are equivalent.

b) The real time statements%
\begin{equation}
\left\{
\begin{array}
[c]{c}%
\exists T>0,\exists t^{\prime}\in\mathbf{R},\forall\mu\in\omega(x),\\
\forall t\in\mathbf{T}_{\mu}^{x}\cap\lbrack t^{\prime},\infty),\{t+zT|z\in
\mathbf{Z}\}\cap\lbrack t^{\prime},\infty)\subset\mathbf{T}_{\mu}^{x}),
\end{array}
\right.  \label{p104}%
\end{equation}
\begin{equation}
\left\{
\begin{array}
[c]{c}%
\exists T>0,\forall\mu\in\omega(x),\exists t^{\prime}\in\mathbf{R},\\
\forall t\in\mathbf{T}_{\mu}^{x}\cap\lbrack t^{\prime},\infty),\{t+zT|z\in
\mathbf{Z}\}\cap\lbrack t^{\prime},\infty)\subset\mathbf{T}_{\mu}^{x},
\end{array}
\right.  \label{per154_}%
\end{equation}%
\begin{equation}
\left\{
\begin{array}
[c]{c}%
\exists t^{\prime}\in\mathbf{R},\forall\mu\in\omega(x),\exists T>0,\\
\forall t\in\mathbf{T}_{\mu}^{x}\cap\lbrack t^{\prime},\infty),\{t+zT|z\in
\mathbf{Z}\}\cap\lbrack t^{\prime},\infty)\subset\mathbf{T}_{\mu}^{x},
\end{array}
\right.  \label{per740}%
\end{equation}%
\begin{equation}
\left\{
\begin{array}
[c]{c}%
\forall\mu\in\omega(x),\exists T>0,\exists t^{\prime}\in\mathbf{R},\\
\forall t\in\mathbf{T}_{\mu}^{x}\cap\lbrack t^{\prime},\infty),\{t+zT|z\in
\mathbf{Z}\}\cap\lbrack t^{\prime},\infty)\subset\mathbf{T}_{\mu}^{x}%
\end{array}
\right.  \label{per149_}%
\end{equation}
are equivalent.
\end{theorem}

\begin{proof}
a) The implications (\ref{p103}) $\Longrightarrow$ (\ref{per168})
$\Longrightarrow$ (\ref{per148_}), (\ref{p103}) $\Longrightarrow$
(\ref{per739}) $\Longrightarrow$ (\ref{per148_}) are obvious, thus we give the
proof of (\ref{per148_})$\Longrightarrow$(\ref{p103}). Let $\widehat{\omega
}(\widehat{x})=\{\mu^{1},...,\mu^{s}\}.$ (\ref{per148_}) states that for an
arbitrary $i\in\{1,...,s\},$ some $p_{i}\geq1$ and $k_{i}^{\prime}%
\in\mathbf{N}_{\_}$ exist such that%
\begin{equation}
\forall k\in\widehat{\mathbf{T}}_{\mu^{i}}^{\widehat{x}}\cap\{k_{i}^{\prime
},k_{i}^{\prime}+1,k_{i}^{\prime}+2,...\},\{k+zp_{i}|z\in\mathbf{Z}%
\}\cap\{k_{i}^{\prime},k_{i}^{\prime}+1,k_{i}^{\prime}+2,...\}\subset
\widehat{\mathbf{T}}_{\mu^{i}}^{\widehat{x}}. \label{per649}%
\end{equation}
We denote $p=p_{1}\cdot...\cdot p_{s}\geq1$ and $k^{\prime}=\max
\{k_{1}^{\prime},...,k_{s}^{\prime}\}.$ From (\ref{per649}) and from Lemma
\ref{Lem30}, page \pageref{Lem30} we infer that%
\begin{equation}
\forall k\in\widehat{\mathbf{T}}_{\mu^{i}}^{\widehat{x}}\cap\{k^{\prime
},k^{\prime}+1,k^{\prime}+2,...\},\{k+zp_{i}|z\in\mathbf{Z}\}\cap\{k^{\prime
},k^{\prime}+1,k^{\prime}+2,...\}\subset\widehat{\mathbf{T}}_{\mu^{i}%
}^{\widehat{x}}, \label{p131}%
\end{equation}
$i\in\{1,...,s\}.$ Let in (\ref{p103}) $\mu=\mu^{i}\in\{\mu^{1},...,\mu^{s}\}$
and $k\in\widehat{\mathbf{T}}_{\mu^{i}}^{\widehat{x}}\cap\{k^{\prime
},k^{\prime}+1,k^{\prime}+2,...\}$ arbitrary. We have:%
\[
\{k+zp|z\in\mathbf{Z}\}\cap\{k^{\prime},k^{\prime}+1,k^{\prime}+2,...\}\subset
\]%
\[
\subset\{k+zp_{i}|z\in\mathbf{Z}\}\cap\{k^{\prime},k^{\prime}+1,k^{\prime
}+2,...\}\overset{(\ref{p131})}{\subset}\widehat{\mathbf{T}}_{\mu^{i}%
}^{\widehat{x}}.
\]

b) The implications (\ref{p104})$\Longrightarrow$(\ref{per154_}%
)$\Longrightarrow$(\ref{per149_}), (\ref{p104})$\Longrightarrow$%
(\ref{per740})$\Longrightarrow$(\ref{per149_}) are obvious.
\end{proof}

\begin{remark}
Stating periodicity properties may depend in general on the order of the
quantifiers. This issue is trivial when quantifiers of the same kind occur
($\exists,\exists$ or $\forall,\forall$) and it is not trivial when
quantifiers of different kinds occur ($\exists,\forall$ or $\forall,\exists$).
Our aim in the previous Theorem is to show that the eventual periodicity
properties are independent on the order of the quantifiers. However the fact
that any of (\ref{per154_}), (\ref{per740}), (\ref{per149_}) implies
(\ref{p104}) could not be proved so far. Such a proof would be important,
since we are tempted to define the eventual periodicity of the signals by
(\ref{per148_}), (\ref{per149_}) (each point of the omega limit set is
eventually periodic) and to use (\ref{p103}) or (\ref{p104}) instead (a common
period exists for all the points of the orbit).
\end{remark}

\begin{remark}
\label{Rem22}Let $\widehat{\omega}(\widehat{x})=\{\mu^{1},...,\mu^{s}\}.$ The
implication (\ref{per148_})$\Longrightarrow$(\ref{p103}) of Theorem
\ref{The136} showed that if $\mu^{1},...,\mu^{s}$ are all eventually periodic:
$\widehat{P}_{\mu^{1}}^{\widehat{x}}\neq\varnothing$ and ... and $\widehat
{P}_{\mu^{s}}^{\widehat{x}}\neq\varnothing$ then $\widehat{P}_{\mu^{1}%
}^{\widehat{x}}\cap...\cap\widehat{P}_{\mu^{s}}^{\widehat{x}}\neq\varnothing.$
Since the equality%
\[
\widehat{P}^{\widehat{x}}=\widehat{P}_{\mu^{1}}^{\widehat{x}}\cap
...\cap\widehat{P}_{\mu^{s}}^{\widehat{x}}%
\]
is always true, even when the left hand term and the right hand term are both
empty, we conclude that the eventual periodicity of $\widehat{x}$ expressed by
(\ref{p103}) (or $\widehat{P}^{\widehat{x}}\neq\varnothing$) and the eventual
periodicity of all the points of the orbit expressed by (\ref{per148_}) (or
$\widehat{P}_{\mu^{1}}^{\widehat{x}}\neq\varnothing$ and ... and $\widehat
{P}_{\mu^{s}}^{\widehat{x}}\neq\varnothing$) are equivalent. We could not
prove that this is true in the real time case, even if, for $\omega
(x)=\{\mu^{1},...,\mu^{s}\},$
\[
P^{x}=P_{\mu^{1}}^{x}\cap...\cap P_{\mu^{s}}^{x}%
\]
is true too, see Theorem \ref{The110}, page \pageref{The110}.
\end{remark}

\begin{remark}
\label{Rem23}From the previous Remark we infer that we have, in particular,
the property%
\[
\forall\mu\in\widehat{\omega}(\widehat{x}),\forall\mu^{\prime}\in
\widehat{\omega}(\widehat{x}),(\widehat{P}_{\mu}^{\widehat{x}}\neq
\varnothing\text{ and }\widehat{P}_{\mu^{\prime}}^{\widehat{x}}\neq
\varnothing)\Longrightarrow\widehat{P}_{\mu}^{\widehat{x}}\cap\widehat{P}%
_{\mu^{\prime}}^{\widehat{x}}\neq\varnothing,
\]
while the truth of the implication%
\[
\forall\mu\in\omega(x),\forall\mu^{\prime}\in\omega(x),(P_{\mu}^{x}%
\neq\varnothing\text{ and }P_{\mu^{\prime}}^{x}\neq\varnothing)\Longrightarrow
P_{\mu}^{x}\cap P_{\mu^{\prime}}^{x}\neq\varnothing
\]
was not proved so far.
\end{remark}

\section{\label{Sec1}The hypothesis P}

\begin{definition}
\label{Def48}We consider the signal $x.$ If
\[
(\forall\mu\in\omega(x),P_{\mu}^{x}\neq\varnothing)\Longrightarrow
\underset{\mu\in\omega(x)}{%
{\displaystyle\bigcap}
}P_{\mu}^{x}\neq\varnothing,
\]
we say that $x$ \textbf{fulfills the hypothesis} $P.$
\end{definition}

\begin{theorem}
\label{The26}a) We suppose that the signal $\widehat{x}\in\widehat{S}^{(n)}$
is not eventually constant, we denote $\widehat{\omega}(\widehat{x})=\{\mu
^{1},...,\mu^{s}\}$ and we ask that $\forall i\in\{1,...,s\},$ the set
$\widehat{P}_{\mu^{i}}^{\widehat{x}}$ is not empty. We denote with
$\widetilde{p}_{i}\geq1,\widetilde{p}\geq1$ the numbers that fulfill%
\begin{equation}
\widehat{P}_{\mu^{i}}^{\widehat{x}}=\{\widetilde{p}_{i},2\widetilde{p}%
_{i},3\widetilde{p}_{i},...\}, \label{per955}%
\end{equation}%
\begin{equation}
\widehat{P}^{\widehat{x}}=\{\widetilde{p},2\widetilde{p},3\widetilde{p},...\}.
\label{per911}%
\end{equation}
Then%
\begin{equation}
\widetilde{p}=n_{1}\widetilde{p}_{1}=...=n_{s}\widetilde{p}_{s},
\end{equation}
where $n_{1}\geq1,...,n_{s}\geq1$ are relatively prime ($\widetilde{p}$ is the
least common multiple of $\widetilde{p}_{1},...,\widetilde{p}_{s}$).

b) We suppose that the signal $x\in S^{(n)}$ is not eventually constant and
that it fulfills the hypothesis $P$. We denote $\omega(x)=\{\mu^{1}%
,...,\mu^{s}\}$ and we ask that $\forall i\in\{1,...,s\},$ the set $P_{\mu
^{i}}^{x}$ is not empty. We denote with $\widetilde{T}_{i}>0,\widetilde{T}>0$
the numbers that satisfy%
\begin{equation}
P_{\mu^{i}}^{x}=\{\widetilde{T}_{i},2\widetilde{T}_{i},3\widetilde{T}%
_{i},...\}, \label{per956}%
\end{equation}%
\begin{equation}
P^{x}=\{\widetilde{T},2\widetilde{T},3\widetilde{T},...\}. \label{per912}%
\end{equation}
We have%
\begin{equation}
\widetilde{T}=n_{1}\widetilde{T}_{1}=...=n_{s}\widetilde{T}_{s},
\end{equation}
where $n_{1}\geq1,...,n_{s}\geq1$ are relatively prime.
\end{theorem}

\begin{proof}
a) Any $p\in\widehat{P}^{\widehat{x}}$ belongs to $\widehat{P}_{\mu^{1}%
}^{\widehat{x}}\cap...\cap\widehat{P}_{\mu^{s}}^{\widehat{x}},$ thus
$n_{1}^{\prime}\geq1,p_{1}\in\widehat{P}_{\mu^{1}}^{\widehat{x}}%
,...,n_{s}^{\prime}\geq1,p_{s}\in\widehat{P}_{\mu^{s}}^{\widehat{x}}$ exist
such that%
\begin{equation}
p=n_{1}^{\prime}p_{1}=...=n_{s}^{\prime}p_{s}. \label{p4}%
\end{equation}
But each $p_{i}$ is a multiple of $\widetilde{p}_{i},$ thus $n_{1}%
^{\prime\prime}\geq1,...,n_{s}^{\prime\prime}\geq1$ exist with%
\begin{equation}
p_{1}=n_{1}^{\prime\prime}\widetilde{p}_{1},...,p_{s}=n_{s}^{\prime\prime
}\widetilde{p}_{s}. \label{p5}%
\end{equation}
We replace the equations (\ref{p5}) in (\ref{p4}) and we get%
\begin{equation}
p=n_{1}\widetilde{p}_{1}=...=n_{s}\widetilde{p}_{s},
\end{equation}
where $n_{1}=n_{1}^{\prime}n_{1}^{\prime\prime},...,n_{s}=n_{s}^{\prime}%
n_{s}^{\prime\prime}.$ When $n_{1},...,n_{s}$ are relatively prime,
$p=\min\widehat{P}^{\widehat{x}}.$

b) As $x$ fulfills the hypothesis $P$ and $P_{\mu^{1}}^{x}\neq\varnothing
,...,P_{\mu^{s}}^{x}\neq\varnothing,$ we have that in the equation%
\[
P^{x}=P_{\mu^{1}}^{x}\cap...\cap P_{\mu^{s}}^{x}%
\]
both terms are non-empty. From this moment the reasoning is the same like at a).
\end{proof}

\chapter{\label{Cha4}Periodic points}

First we give in Section 1 several properties that are equivalent with the
periodicity of a point. These properties were previously used to characterize
the constancy of the signals. A discussion of these properties is made in
Section 2.

Section 3 shows that the periodic points are accessed at least once in a time
interval with the length of a period.

The independence of the real time periodicity of $\mu$ on the initial time
$t^{\prime}$ of $x$ $=$ limit of periodicity of $\mu$ and also the bounds of
$t^{\prime}$ are the topics of Section 4.

The property of constancy from Section 5 is interesting by itself, but it is
also useful in treating the discrete time vs the real time periodic points,
representing the topic of Section 6.

One might be tempted to think that the relation between $\widehat{\mathbf{T}%
}_{\mu}^{\widehat{x}},\mathbf{T}_{\mu}^{x}$ and $\widehat{P}_{\mu}%
^{\widehat{x}},P_{\mu}^{x}$ is closer than it really is. Some examples and
comments on this relation are given in Section 7.

The fact that the sums, the differences and the multiples of the periods are
periods is formalized in Section 8.

The important topic of existence of the prime period is treated in Section 9,
together with the form of $\widehat{P}_{\mu}^{\widehat{x}},P_{\mu}^{x}.$

Necessary conditions, respectively sufficient conditions of periodicity of
$\mu,$ related with the form of $\widehat{\mathbf{T}}_{\mu}^{\widehat{x}%
},\mathbf{T}_{\mu}^{x}$ are given in Sections 10, respectively 11.

Section 12 deals with a special case of periodicity, applying results from
Section 10 and Section 11. The point is that in this special case we know the
precise value of the prime period.

In Section 13 we show that by forgetting some first values of $\widehat{x},x$
we get the same sets of periods $\widehat{P}_{\mu}^{\widehat{x}},P_{\mu}^{x}.$
This natural observation connects the periodicity of $\mu$ with its eventual periodicity.

Some ideas concerning further research on the periodic points are presented in
Section 14.

\section{Equivalent properties with the periodicity of a point}

\begin{remark}
The properties of periodicity of the points were present in the second group
of constancy properties of the signals from Theorem \ref{The95}, page
\pageref{The95} (and the third group, Theorem \ref{The94}, page
\pageref{The94}), thus (\ref{per144}),...,(\ref{pre536}) will be compared with
(\ref{per185})$_{page\;\pageref{per185}},...,$(\ref{pre537}%
)$_{page\;\pageref{pre537}}$ and (\ref{per145}),...,(\ref{pre562}) will be
compared with (\ref{per186})$_{page\;\pageref{per186}}$,..., (\ref{pre552}%
)$_{page\;\pageref{pre552}}.$ We make also the associations (\ref{pre49}%
)-(\ref{per75})$_{page\;\pageref{per75}},...,$ (\ref{pre536})-(\ref{pre509}%
)$_{page\;\pageref{pre509}}$ and (\ref{pre50})-(\ref{per76}%
)$_{page\;\pageref{per76}},...,$(\ref{pre562})-(\ref{pre513}%
)$_{page\;\pageref{pre513}}$ with the fourth group of constancy properties
from Theorem \ref{The96}, page \pageref{The96}.
\end{remark}

\begin{theorem}
\label{The49}We consider the signals $\widehat{x}\in\widehat{S}^{(n)},x\in
S^{(n)}$.

a) The following statements are equivalent for any $p\geq1$ and $\mu
\in\widehat{Or}(\widehat{x}):$%
\begin{equation}
\forall k\in\widehat{\mathbf{T}}_{\mu}^{\widehat{x}},\{k+zp|z\in
\mathbf{Z}\}\cap\mathbf{N}_{\_}\subset\widehat{\mathbf{T}}_{\mu}^{\widehat{x}%
}, \label{per144}%
\end{equation}%
\begin{equation}
\left\{
\begin{array}
[c]{c}%
\forall k^{\prime}\in\mathbf{N}_{\_},\forall k\in\widehat{\mathbf{T}}_{\mu
}^{\widehat{x}}\cap\{k^{\prime},k^{\prime}+1,k^{\prime}+2,...\},\\
\{k+zp|z\in\mathbf{Z}\}\cap\{k^{\prime},k^{\prime}+1,k^{\prime}+2,...\}\subset
\widehat{\mathbf{T}}_{\mu}^{\widehat{x}},
\end{array}
\right.  \label{pre525}%
\end{equation}%
\begin{equation}
\forall k^{\prime\prime}\in\mathbf{N},\forall k\in\widehat{\mathbf{T}}_{\mu
}^{\widehat{\sigma}^{k^{\prime\prime}}(\widehat{x})},\{k+zp|z\in
\mathbf{Z}\}\cap\mathbf{N}_{\_}\subset\widehat{\mathbf{T}}_{\mu}%
^{\widehat{\sigma}^{k^{\prime\prime}}(\widehat{x})}, \label{pre527}%
\end{equation}%
\begin{equation}
\left\{
\begin{array}
[c]{c}%
\forall k\in\mathbf{N}_{\_},\widehat{x}(k)=\mu\Longrightarrow\\
\Longrightarrow(\widehat{x}(k)=\widehat{x}(k+p)\text{ and }k-p\geq
-1\Longrightarrow\widehat{x}(k)=\widehat{x}(k-p)),
\end{array}
\right.  \label{pre49}%
\end{equation}%
\begin{equation}
\left\{
\begin{array}
[c]{c}%
\forall k^{\prime}\in\mathbf{N}_{\_},\forall k\geq k^{\prime},\widehat
{x}(k)=\mu\Longrightarrow\\
\Longrightarrow(\widehat{x}(k)=\widehat{x}(k+p)\text{ and }k-p\geq k^{\prime
}\Longrightarrow\widehat{x}(k)=\widehat{x}(k-p)),
\end{array}
\right.  \label{pre526}%
\end{equation}%
\begin{equation}
\left\{
\begin{array}
[c]{c}%
\forall k^{\prime\prime}\in\mathbf{N},\forall k\in\mathbf{N}_{\_}%
,\widehat{\sigma}^{k^{\prime\prime}}(\widehat{x})(k)=\mu\Longrightarrow\\
\Longrightarrow(\widehat{\sigma}^{k^{\prime\prime}}(\widehat{x})(k)=\widehat
{\sigma}^{k^{\prime\prime}}(\widehat{x})(k+p)\text{ and }\\
\text{and }k-p\geq-1\Longrightarrow\widehat{\sigma}^{k^{\prime\prime}%
}(\widehat{x})(k)=\widehat{\sigma}^{k^{\prime\prime}}(\widehat{x})(k-p)).
\end{array}
\right.  \label{pre536}%
\end{equation}

b) The following statements are also equivalent for any $T>0$ and $\mu\in
Or(x)$:%
\begin{equation}
\exists t^{\prime}\in I^{x},\forall t\in\mathbf{T}_{\mu}^{x}\cap\lbrack
t^{\prime},\infty),\{t+zT|z\in\mathbf{Z}\}\cap\lbrack t^{\prime}%
,\infty)\subset\mathbf{T}_{\mu}^{x}, \label{per145}%
\end{equation}%
\begin{equation}
\exists t^{\prime}\in I^{x},\forall t_{1}^{\prime}\geq t^{\prime},\forall
t\in\mathbf{T}_{\mu}^{x}\cap\lbrack t_{1}^{\prime},\infty),\{t+zT|z\in
\mathbf{Z}\}\cap\lbrack t_{1}^{\prime},\infty)\subset\mathbf{T}_{\mu}^{x},
\label{pre559}%
\end{equation}%
\begin{equation}
\left\{
\begin{array}
[c]{c}%
\forall t^{\prime\prime}\in\mathbf{R},\exists t^{\prime}\in I^{\sigma
^{t^{\prime\prime}}(x)},\\
\forall t\in\mathbf{T}_{\mu}^{\sigma^{t^{\prime\prime}}(x)}\cap\lbrack
t^{\prime},\infty),\{t+zT|z\in\mathbf{Z}\}\cap\lbrack t^{\prime}%
,\infty)\subset\mathbf{T}_{\mu}^{\sigma^{t^{\prime\prime}}(x)},
\end{array}
\right.  \label{pre560}%
\end{equation}%
\begin{equation}
\left\{
\begin{array}
[c]{c}%
\exists t^{\prime}\in I^{x},\forall t\geq t^{\prime},\\
x(t)=\mu\Longrightarrow(x(t)=x(t+T)\text{ and }t-T\geq t^{\prime
}\Longrightarrow x(t)=x(t-T)),
\end{array}
\right.  \label{pre50}%
\end{equation}%
\begin{equation}
\left\{
\begin{array}
[c]{c}%
\exists t^{\prime}\in I^{x},\forall t_{1}^{\prime}\geq t^{\prime},\forall
t\geq t_{1}^{\prime},x(t)=\mu\Longrightarrow\\
\Longrightarrow(x(t)=x(t+T)\text{ and }t-T\geq t_{1}^{\prime}\Longrightarrow
x(t)=x(t-T)),
\end{array}
\right.  \label{pre561}%
\end{equation}%
\begin{equation}
\left\{
\begin{array}
[c]{c}%
\forall t^{\prime\prime}\in\mathbf{R},\exists t^{\prime}\in I^{\sigma
^{t^{\prime\prime}}(x)},\\
\forall t\geq t^{\prime},\sigma^{t^{\prime\prime}}(x)(t)=\mu\Longrightarrow
(\sigma^{t^{\prime\prime}}(x)(t)=\sigma^{t^{\prime\prime}}(x)(t+T)\text{
and}\\
\text{and }t-T\geq t^{\prime}\Longrightarrow\sigma^{t^{\prime\prime}%
}(x)(t)=\sigma^{t^{\prime\prime}}(x)(t-T)).
\end{array}
\right.  \label{pre562}%
\end{equation}

\end{theorem}

\begin{proof}
The proof of the implications%
\[
(\ref{per144})\Longrightarrow(\ref{pre525})\Longrightarrow(\ref{pre527}%
)\Longrightarrow(\ref{pre49})\Longrightarrow(\ref{pre526})\Longrightarrow
(\ref{pre536})
\]
follows from Theorem \ref{The95}, page \pageref{The95}.

(\ref{pre536})$\Longrightarrow$(\ref{per144}) We can use (\ref{pre49}) that is
a special case of (\ref{pre536}) when $k^{\prime\prime}=0.$ Let $k\in
\widehat{\mathbf{T}}_{\mu}^{\widehat{x}}$ and $z\in\mathbf{Z}$ with
$k+zp\geq-1$ and we have the following possibilities.

Case $z>0,$%
\[
\mu=\widehat{x}(k)\overset{(\ref{pre49})}{=}\widehat{x}(k+p)\overset
{(\ref{pre49})}{=}\widehat{x}(k+2p)\overset{(\ref{pre49})}{=}...\overset
{(\ref{pre49})}{=}\widehat{x}(k+zp);
\]

Case $z=0,$%
\[
\mu=\widehat{x}(k)=\widehat{x}(k+zp);
\]

Case $z<0,$%
\[
\mu=\widehat{x}(k)\overset{(\ref{pre49})}{=}\widehat{x}(k-p)\overset
{(\ref{pre49})}{=}\widehat{x}(k-2p)\overset{(\ref{pre49})}{=}...\overset
{(\ref{pre49})}{=}\widehat{x}(k+zp).
\]
In all these cases $k+zp\in\widehat{\mathbf{T}}_{\mu}^{\widehat{x}}.$

b) The proof of the implications%
\[
(\ref{per145})\Longrightarrow(\ref{pre559})\Longrightarrow(\ref{pre560}%
)\Longrightarrow(\ref{pre50})\Longrightarrow(\ref{pre561})\Longrightarrow
(\ref{pre562})
\]
follows from Theorem \ref{The95}, page \pageref{The95}.

(\ref{pre562})$\Longrightarrow$(\ref{per145}) We write (\ref{pre562}) in the
special case when $t^{\prime\prime}$ fulfills $\forall t\leq t^{\prime\prime
},x(t)=x(-\infty+0)$ and consequently $\sigma^{t^{\prime\prime}}(x)=x,$
$t^{\prime}\in I^{x}$ and%
\begin{equation}
\forall t\geq t^{\prime},x(t)=\mu\Longrightarrow(x(t)=x(t+T)\text{ and
}t-T\geq t^{\prime}\Longrightarrow x(t)=x(t-T)) \label{pre873}%
\end{equation}
hold. We have $\mathbf{T}_{\mu}^{x}\cap\lbrack t^{\prime},\infty
)\neq\varnothing,$ so let $t\in\mathbf{T}_{\mu}^{x}\cap\lbrack t^{\prime
},\infty)$ and $z\in\mathbf{Z}$ arbitrary with $t+zT\geq t^{\prime}.$ We get
the following possibilities.

Case $z>0,$%
\[
\mu=x(t)\overset{(\ref{pre873})}{=}x(t+T)\overset{(\ref{pre873})}%
{=}x(t+2T)\overset{(\ref{pre873})}{=}...\overset{(\ref{pre873})}{=}x(t+zT);
\]

Case $z=0,$%
\[
\mu=x(t)=x(t+zT);
\]

Case $z<0,$%
\[
\mu=x(t)\overset{(\ref{pre873})}{=}x(t-T)\overset{(\ref{pre873})}%
{=}x(t-2T)\overset{(\ref{pre873})}{=}...\overset{(\ref{pre873})}{=}x(t+zT)
\]
and consequently in all these situations $t+zT\in\mathbf{T}_{\mu}^{x}.$
(\ref{per145}) is true.
\end{proof}

\begin{example}
\label{Exa5}We give the following example of signal $\widehat{x}\in\widehat
{S}^{(1)}$ where none of $0,1\in\widehat{Or}(\widehat{x})$ is periodic or
eventually periodic:%
\[
\widehat{x}=0,\underset{1}{\underbrace{1}},0,\underset{2}{\underbrace{1,1}%
},0,\underset{3}{\underbrace{1,1,1}},0,\underset{4}{\underbrace{1,1,1,1}%
},0,...
\]
This signal is similar with that of Example \ref{Exa6}, page \pageref{Exa6}.
\end{example}

\begin{example}
Let the signal $\widehat{x}\in\widehat{S}^{(1)}$ and we presume that%
\[
\widehat{x}=0,\widehat{x}(0),\widehat{x}(1),0,\widehat{x}(3),\widehat
{x}(4),0,\widehat{x}(6),...
\]
Then $0$ is a periodic point of $\widehat{Or}(\widehat{x})$ and it has the
period $3$. In particular if $\widehat{x}(0),\widehat{x}(1),\widehat
{x}(3),\widehat{x}(4),\widehat{x}(6),...$ are all equal with $1$ then $3$ is
the prime period of $0$ and if they are all equal with $0$ then $1$ is the
prime period of $0$.
\end{example}

\begin{example}
The signal $x\in S^{(1)},$%
\[
x(t)=\chi_{(-\infty,0)}(t)\oplus\chi_{\lbrack1,3)}(t)\oplus\chi_{\lbrack
7,9)}(t)\oplus\chi_{\lbrack10,12)}(t)\oplus\chi_{\lbrack16,18)}(t)\oplus
\chi_{\lbrack19,21)}(t)\oplus...
\]
fulfills $-1\in I^{x}$ and%
\[
\forall t\in\mathbf{T}_{1}^{x}\cap\lbrack-1,\infty),\{t+z9|z\in\mathbf{Z}%
\}\cap\lbrack-1,\infty)\subset\mathbf{T}_{1}^{x},
\]
i.e. the point $1$ has the period $9$. To be noticed how the couple
$[1,3),[7,9)$ of intervals 'generates' periodicity.
\end{example}

\section{Discussion}

\begin{remark}
From the periodicity properties (\ref{per144}), (\ref{per145}), we have that
$\widehat{\mathbf{T}}_{\mu}^{\widehat{x}}$ is infinite and $\mathbf{T}_{\mu
}^{x}$ is superiorly unbounded, thus the periodic points $\mu\in\widehat
{Or}(\widehat{x}),\mu\in Or(x)$ satisfy in fact $\mu\in\widehat{\omega
}(\widehat{x}),\mu\in\omega(x).$ This was noticed since the first introduction
of the periodic points, see Remark \ref{Rem27} from page \pageref{Rem27}, and
several times afterwards.
\end{remark}

\begin{remark}
The prime period of the periodic point $\mu\in\widehat{Or}(\widehat{x})$
always exists, but the prime period of the periodic point $\mu\in Or(x)$ may
not exist, for example if $x$ is constant and equal with $\mu$, see Theorem
\ref{The95}, page \pageref{The95} where $P_{\mu}^{x}=(0,\infty).$ We shall
prove later (Theorem \ref{The26_}, page \pageref{The26_}) that this is the
only situation when the periodic point $\mu\in Or(x)$ has no prime period.
\end{remark}

\begin{remark}
Two more compact forms of writing (\ref{pre49}) and (\ref{pre50}) are%
\begin{equation}
\exists p\in\mathbf{Z}^{\ast},\forall k\in\mathbf{N}_{\_},(\widehat{x}%
(k)=\mu\text{ and }k+p\geq-1)\Longrightarrow\widehat{x}(k)=\widehat{x}(k+p),
\end{equation}%
\begin{equation}
\exists T\in\mathbf{R}^{\ast},\exists t^{\prime}\in I^{x},\forall t\geq
t^{\prime},(x(t)=\mu\text{ and }t+T\geq t^{\prime})\Longrightarrow
x(t)=x(t+T),
\end{equation}
where we have denoted $\mathbf{Z}^{\ast}=\mathbf{Z}\smallsetminus\{0\}$ and
$\mathbf{R}^{\ast}=\mathbf{R}\smallsetminus\{0\}.$ Such statements accept the
existence of negative periods. We shall suppose in the rest of our
presentation that $p\geq1,T>0.$
\end{remark}

\begin{remark}
A temptation exists to write (\ref{pre49}) and (\ref{pre50}) in a wrong way,
recalling the periodicity (\ref{pre741})$_{page\;\pageref{pre741}}$,
(\ref{pre743})$_{page\;\pageref{pre743}}$ of the signals, under the form%
\begin{equation}
\exists p\geq1,\forall k\in\mathbf{N}_{\_},\widehat{x}(k)=\mu\Longrightarrow
\widehat{x}(k)=\widehat{x}(k+p), \label{pre56}%
\end{equation}%
\begin{equation}
\exists T>0,\exists t^{\prime}\in I^{x},\forall t\geq t^{\prime}%
,x(t)=\mu\Longrightarrow x(t)=x(t+T), \label{pre962}%
\end{equation}
that accepts only right time shifts in the definition of periodicity. We give
the discrete time example of%
\[
\widehat{x}=0,0,1,0,1,0,1,0,1,...
\]
that fulfills (\ref{pre56}) with $\mu=1,p=2.$ For this signal $\mu$ is not
periodic and the left time shift requirement $\widehat{x}(1)=1\Longrightarrow
\widehat{x}(1-2)=1$ shows where the problem is. In fact, if $\mu\in
\widehat{Or}(\widehat{x}),\mu\in Or(x)$ then (\ref{pre56}), (\ref{pre962}) are
requirements of eventual periodicity, not of periodicity.
\end{remark}

\begin{remark}
\label{Rem20}Let $\widehat{x},$ $\mu\in\widehat{Or}(\widehat{x})$ and
$p\geq1.$ We have $\widehat{\mathbf{T}}_{\mu}^{\widehat{x}}\neq\varnothing$
and if%
\[
\forall k\in\widehat{\mathbf{T}}_{\mu}^{\widehat{x}},\{k+zp|z\in
\mathbf{Z}\}\cap\mathbf{N}_{\_}\subset\widehat{\mathbf{T}}_{\mu}^{\widehat{x}%
},
\]
then we deduce from Theorem \ref{Lem1}, page \pageref{Lem1} that for any
$k\in\mathbf{N}_{\_}$ we have $\widehat{\mathbf{T}}_{\mu}^{\widehat{x}}%
\cap\{k,k+1,...,k+p-1\}\neq\varnothing.$ Similarly, $x,$ $\mu\in Or(x),$ $T>0,
$ $t^{\prime}\in I^{x}$ are given. If%
\[
\forall t\in\mathbf{T}_{\mu}^{x}\cap\lbrack t^{\prime},\infty),\{t+zT|z\in
\mathbf{Z}\}\cap\lbrack t^{\prime},\infty)\subset\mathbf{T}_{\mu}^{x},
\]
as far as $\mathbf{T}_{\mu}^{x}\cap\lbrack t^{\prime},\infty)\neq\varnothing$
(from Lemma \ref{Lem36}, page \pageref{Lem36})$,$ we can use Theorem
\ref{Lem1} again and get $\forall t\geq t^{\prime},$ $\mathbf{T}_{\mu}^{x}%
\cap\lbrack t,t+T)\neq\varnothing.$
\end{remark}

\section{The accessibility of the periodic points}

\begin{remark}
\label{The142}From Theorem \ref{Lem1}, page \pageref{Lem1} we get that
if.$\mu\in\widehat{Or}(\widehat{x})$ is a periodic point of $\widehat{x}$ with
the period $p\geq1,$ then $\widehat{\mathbf{T}}_{\mu}^{\widehat{x}}%
\cap\{k,k+1,...,k+p-1\}\neq\varnothing$ holds for any $k\in\mathbf{N}_{\_}$
\ From the same Theorem we similarly get that if $\mu\in Or(x)$ is a periodic
point of $x$ with the period $T>0$, then $t^{\prime}\in I^{x}$ exists such
that for any $t\geq t^{\prime},$ we have $\mathbf{T}_{\mu}^{x}\cap\lbrack
t,t+T)\neq\varnothing.$
\end{remark}

\section{The limit of periodicity}

\begin{example}
\label{Exa10}We consider $x,$ $\mu=x(-\infty+0),$ $T>0,$ $t_{0},t_{1}%
,t_{2},t_{3}\in\mathbf{R}$ and we suppose that%
\[
t_{0}<t_{1}<t_{2}<t_{3}<t_{0}+T,
\]%
\[
\mathbf{T}_{\mu}^{x}=(-\infty,t_{0})\cup\lbrack t_{1},t_{2})\cup\lbrack
t_{3},t_{0}+T)\cup\lbrack t_{1}+T,t_{2}+T)\cup\lbrack t_{3}+T,t_{0}+2T)\cup
\]%
\[
\cup\lbrack t_{1}+2T,t_{2}+2T)\cup\lbrack t_{3}+2T,t_{0}+3T)\cup...
\]
hold. For $t^{\prime}\in\lbrack t_{3}-T,t_{0})$ we have $t^{\prime}\in I^{x}$
and%
\begin{equation}
\forall t\in\mathbf{T}_{\mu}^{x}\cap\lbrack t^{\prime},\infty),\{t+zT|z\in
\mathbf{Z}\}\cap\lbrack t^{\prime},\infty)\subset\mathbf{T}_{\mu}^{x}
\label{pre58}%
\end{equation}
fulfilled, thus the property of periodicity of $\mu$ with the period $T$ is
true. For $t^{\prime}<t_{3}-T,$ let us take an arbitrary $t\in\lbrack
\max\{t^{\prime},t_{2}-T\},t_{3}-T).$ On one hand $t\in\mathbf{T}_{\mu}%
^{x}\cap\lbrack t^{\prime},\infty)$ and on the other hand the truth of%
\[
t+T\in\{t+zT|z\in\mathbf{Z}\}\cap\lbrack t^{\prime},\infty)\overset
{(\ref{pre58})}{\subset}\mathbf{T}_{\mu}^{x}%
\]
should indicate that $x(t+T)=\mu$. But this is false, since $t+T\in\lbrack
t_{2},t_{3}).$ We have shown that $L_{\mu}^{x}=[t_{3}-T,\infty).$ We notice
also that choosing $t^{\prime}\geq t_{0}$ is not possible, since
$I^{x}=(-\infty,t_{0})$. We conclude that the exact bounds of the initial
time=limit of periodicity $t^{\prime}$ are given by $t^{\prime}\in I^{x}\cap
L_{\mu}^{x}=$ $[t_{3}-T,t_{0}).$
\end{example}

\begin{theorem}
\label{The5}Let the non constant signal $x$ be given, together with $\mu\in
Or(x),$ $T>0$ and $t^{\prime}\in I^{x}$ having the property that%
\begin{equation}
\forall t\in\mathbf{T}_{\mu}^{x}\cap\lbrack t^{\prime},\infty),\{t+zT|z\in
\mathbf{Z}\}\cap\lbrack t^{\prime},\infty)\subset\mathbf{T}_{\mu}^{x}.
\label{pre60}%
\end{equation}
Then $t_{0}^{\prime},t_{0}\in\mathbf{R}$ exist, $t_{0}^{\prime}<t_{0}$ such
that $\forall t^{\prime\prime}\in\lbrack t_{0}^{\prime},t_{0}),$ we have that
$t^{\prime\prime}\in I^{x},$%
\begin{equation}
\forall t\in\mathbf{T}_{\mu}^{x}\cap\lbrack t^{\prime\prime},\infty
),\{t+zT|z\in\mathbf{Z}\}\cap\lbrack t^{\prime\prime},\infty)\subset
\mathbf{T}_{\mu}^{x} \label{pre62}%
\end{equation}
hold and for any $t^{\prime\prime}\notin\lbrack t_{0}^{\prime},t_{0}),$ at
least one of $t^{\prime\prime}\in I^{x},$ (\ref{pre62}) is false. In other
words $[t_{0}^{\prime},t_{0})=I^{x}\cap L_{\mu}^{x}.$
\end{theorem}

\begin{remark}
We give two proofs of the previous Theorem for reasons that will become clear later.
\end{remark}

\begin{proof}
The first proof of Theorem \ref{The5}.

From the fact that $x$ is not constant we get the existence of $t_{0}%
\in\mathbf{R}$ with $I^{x}=(-\infty,t_{0}).$ From $\mu\in Or(x)$ and
$t^{\prime}\in I^{x}$ we have that $\mathbf{T}_{\mu}^{x}\cap\lbrack t^{\prime
},\infty)\neq\varnothing$ and this, taking into account (\ref{pre60}) also,
implies $\mu\in\omega(x).$ The existence of $t^{\prime}\in\mathbf{R}$ such
that (\ref{pre60}) holds shows the fact that $L_{\mu}^{x}\neq\varnothing,$
thus we can apply Theorem \ref{The117}, page \pageref{The117}. The existence
of $t_{0}^{\prime}\in R$ has resulted with $L_{\mu}^{x}=[t_{0}^{\prime}%
,\infty).$

The existence of $t^{\prime}\in I^{x}$ making (\ref{pre60}) true shows
furthermore that $t_{0}^{\prime}<t_{0}$ and $[t_{0}^{\prime},t_{0})=I^{x}\cap
L_{\mu}^{x}.$
\end{proof}

\begin{proof}
The second proof of Theorem \ref{The5}.

The fact that $x$ is not constant shows the existence of $t_{0}$ that is
defined by%
\begin{equation}
\forall t<t_{0},x(t)=x(-\infty+0), \label{pre63}%
\end{equation}%
\begin{equation}
x(t_{0})\neq x(-\infty+0). \label{pre64}%
\end{equation}
From (\ref{pre63}), (\ref{pre64}), $t^{\prime}\in I^{x}$ we infer that
$I^{x}=(-\infty,t_{0}),$ $t^{\prime}<t_{0}$ hold$.$ We have the following possibilities.

a) Case $\mu=x(-\infty+0)$

We show first that $[t^{\prime}+T,t_{0}+T)\subset\mathbf{T}_{\mu}^{x}$ and let
for this an arbitrary $t\in\lbrack t^{\prime}+T,t_{0}+T).$ We have $t-T\geq
t^{\prime},t\geq t^{\prime}$ and $t-T\in\lbrack t^{\prime},t_{0}%
)\subset\mathbf{T}_{\mu}^{x},$ so that we can apply (\ref{pre60}):%
\[
t\in\{t-T+zT|z\in\mathbf{Z}\}\cap\lbrack t^{\prime},\infty)\subset
\mathbf{T}_{\mu}^{x}.
\]
The inclusion $[t^{\prime}+T,t_{0}+T)\subset\mathbf{T}_{\mu}^{x}$ is proved.

We get the existence of $t_{1}\leq t^{\prime}+T$ with%
\begin{equation}
\lbrack t_{1},t_{0}+T)\subset\mathbf{T}_{\mu}^{x}, \label{pre64_}%
\end{equation}%
\begin{equation}
x(t_{1}-0)\neq\mu. \label{pre65}%
\end{equation}
We have $t_{1}>t_{0},$ because the other possibility $t_{0}\geq t_{1}$ is in
contradiction with (\ref{pre63}). The conclusion is that%
\begin{equation}
t_{1}-T\leq t^{\prime}<t_{0}<t_{1}<t_{0}+T. \label{pre66}%
\end{equation}
From Lemma \ref{Lem28}, page \pageref{Lem28} and (\ref{pre63}) we infer
\begin{equation}
(-\infty,t_{0})\cup\lbrack t_{1},t_{0}+T)\cup\lbrack t_{1}+T,t_{0}%
+2T)\cup\lbrack t_{1}+2T,t_{0}+3T)\cup...\subset\mathbf{T}_{\mu}^{x}.
\label{pre985}%
\end{equation}
We claim that $t_{0}^{\prime}=t_{1}-T$ fulfills the statement of the Theorem,
in particular that
\begin{equation}
(-\infty,t_{1}-T]\subset\mathbf{T}_{x(-\infty+0)}^{x}, \label{pre986}%
\end{equation}%
\begin{equation}
\forall t\in\mathbf{T}_{\mu}^{x}\cap\lbrack t_{1}-T,\infty),\{t+zT|z\in
\mathbf{Z}\}\cap\lbrack t_{1}-T,\infty)\subset\mathbf{T}_{\mu}^{x}
\label{pre987}%
\end{equation}
hold. We notice that the truth of (\ref{pre986}) is trivial (from
(\ref{pre63}) and (\ref{pre66})) and, in order to prove the satisfaction of
(\ref{pre987}), let $t\in\mathbf{T}_{\mu}^{x}\cap\lbrack t_{1}-T,\infty)$
arbitrary. We have the following sub-cases.

a.1) Case $t\in\lbrack t_{1}-T,t_{0})\cup\lbrack t_{1},t_{0}+T)\cup\lbrack
t_{1}+T,t_{0}+2T)\cup...$

Some $k\in\mathbf{N}_{\_}$ exists with $t\in\lbrack t_{1}+kT,t_{0}+(k+1)T). $
Then%
\[
\{t+zT|z\in\mathbf{Z}\}\cap\lbrack t_{1}-T,\infty
)=\{t+(-k-1)T,t+(-k)T,t+(-k+1)T,...\}
\]%
\[
\subset\lbrack t_{1}-T,t_{0})\cup\lbrack t_{1},t_{0}+T)\cup\lbrack
t_{1}+T,t_{0}+2T)\cup\lbrack t_{1}+2T,t_{0}+3T)\cup...\overset{(\ref{pre985}%
)}{\subset}\mathbf{T}_{\mu}^{x}.
\]

a.2) Case $t\in\mathbf{T}_{\mu}^{x}\cap([t_{0},t_{1})\cup\lbrack t_{0}%
+T,t_{1}+T)\cup\lbrack t_{0}+2T,t_{1}+2T)\cup...)$

Then $t\in\mathbf{T}_{\mu}^{x}\cap\lbrack t^{\prime},\infty)$ and
$k\in\mathbf{N}$ exists such that $t\in\lbrack t_{0}+kT,t_{1}+kT).$ We have,
since $t+(-k-1)T<t_{1}-T,$ that%
\[
\{t+zT|z\in\mathbf{Z}\}\cap\lbrack t_{1}-T,\infty
)=\{t+(-k)T,t+(-k+1)T,t+(-k+2)T,...\}
\]%
\[
=\{t+zT|z\in\mathbf{Z}\}\cap\lbrack t^{\prime},\infty)\overset{(\ref{pre60}%
)}{\subset}\mathbf{T}_{\mu}^{x}.
\]

This ends proving the truth of (\ref{pre987}). For any $t^{\prime\prime}%
\in\lbrack t_{0}^{\prime},t_{0}),$ we have that $t^{\prime\prime}\in I^{x},$
(\ref{pre62}) are fulfilled, see also Lemma \ref{Lem30}, page \pageref{Lem30}
(the statement $\mu\in\omega(x)$ from the hypothesis of the Lemma results from
$t^{\prime}\in I^{x}$, giving $\mathbf{T}_{\mu}^{x}\cap\lbrack t^{\prime
},\infty)\neq\varnothing,$ and from (\ref{pre60})).

In order to prove the last statement of the Theorem, let $t^{\prime\prime}\in
I^{x},$ (\ref{pre62}) be true with arbitrary, fixed $t^{\prime\prime}$. We
suppose against all reason that $t^{\prime\prime}<t_{1}-T$ and let
$\varepsilon>0$ with the property that%
\begin{equation}
\forall\xi\in(t_{1}-\varepsilon,t_{1}),x(\xi)=x(t_{1}-0). \label{pre67}%
\end{equation}
We take an arbitrary $t\in(\max\{t^{\prime\prime},t_{1}-T-\varepsilon
\},t_{1}-T)$ for which we can write that $x(t)=\mu$ and, on the other hand,%
\[
t+T\in\{t+zT|z\in\mathbf{Z}\}\cap\lbrack t^{\prime\prime},\infty
)\overset{(\ref{pre62})}{\subset}\mathbf{T}_{\mu}^{x},
\]
thus $x(t+T)=\mu.$ But $t+T\in(t_{1}-\varepsilon,t_{1}),$ wherefrom
$x(t+T)\overset{(\ref{pre67})}{=}x(t_{1}-0)$ and finally $x(t_{1}-0)=\mu,$
contradiction with (\ref{pre65}). We have obtained that $L_{\mu}^{x}%
=[t_{1}-T,\infty).$ The supposition that $t^{\prime\prime}\geq t_{0}$ is in
contradiction with the hypothesis $t^{\prime\prime}\in I^{x},$ since
$I^{x}=(-\infty,t_{0}).$

b) Case $\mu\neq x(-\infty+0)$

We show that $[t^{\prime}+T,t_{0}+T)\cap\mathbf{T}_{\mu}^{x}=\varnothing$ and
we suppose against all reason that $t\in\lbrack t^{\prime}+T,t_{0}%
+T)\cap\mathbf{T}_{\mu}^{x}$ exists, thus $t-T\in\lbrack t^{\prime},t_{0})$
and $t\geq t^{\prime}$ hold$.$ We infer%
\begin{equation}
t-T\in\{t+zT|z\in\mathbf{Z}\}\cap\lbrack t^{\prime},\infty)\overset
{(\ref{pre60})}{\subset}\mathbf{T}_{\mu}^{x} \label{pre75}%
\end{equation}
wherefrom the contradiction%
\[
x(-\infty+0)\overset{(\ref{pre63})}{=}x(t-T)\overset{(\ref{pre75})}{=}\mu\neq
x(-\infty+0).
\]
We infer from here, taking into account Theorem \ref{Lem1}, page
\pageref{Lem1} also (the statement $\mu\in\omega(x)$ from the hypothesis of
the Theorem follows, like previously, from $t^{\prime}\in I^{x}$, implying
that $\mathbf{T}_{\mu}^{x}\cap\lbrack t^{\prime},\infty)\neq\varnothing,$ and
from (\ref{pre60})), written for $t=t_{0},$ stating that $\mathbf{T}_{\mu}%
^{x}\cap\lbrack t_{0},t_{0}+T)\neq\varnothing,$ the existence of $t_{1}%
,t_{2}\in\mathbf{R}$ with%
\begin{equation}
t_{2}-T\leq t^{\prime}<t_{0}\leq t_{1}<t_{2}\leq t^{\prime}+T<t_{0}+T,
\label{pre76}%
\end{equation}%
\begin{equation}
x(t_{1}-0)\neq\mu, \label{pre77}%
\end{equation}%
\begin{equation}
\lbrack t_{1},t_{2})\subset\mathbf{T}_{\mu}^{x}, \label{pre78}%
\end{equation}%
\begin{equation}
\lbrack t_{2},t_{0}+T)\cap\mathbf{T}_{\mu}^{x}=\varnothing. \label{pre79}%
\end{equation}
We claim that the statement of the Theorem is fulfilled by $t_{0}^{\prime
}=t_{2}-T$ and in particular that
\begin{equation}
(-\infty,t_{2}-T]\subset\mathbf{T}_{x(-\infty+0)}^{x}, \label{p62}%
\end{equation}%
\begin{equation}
\forall t\in\mathbf{T}_{\mu}^{x}\cap\lbrack t_{2}-T,\infty),\{t+zT|z\in
\mathbf{Z}\}\cap\lbrack t_{2}-T,\infty)\subset\mathbf{T}_{\mu}^{x} \label{p63}%
\end{equation}
hold. We notice that (\ref{p62}) results from (\ref{pre63}) and (\ref{pre76}).
Let $t\in\mathbf{T}_{\mu}^{x}\cap\lbrack t_{2}-T,\infty)$ arbitrary. We easily
see that $((-\infty,t_{0})\cup\lbrack t_{2},t_{0}+T)\cup\lbrack t_{2}%
+T,t_{0}+2T)\cup\lbrack t_{2}+2T,t_{0}+3T)\cup...)\cap\mathbf{T}_{\mu}%
^{x}=\varnothing$ since by supposing against all reason that this is not true
we get a contradiction, thus $\mathbf{T}_{\mu}^{x}\subset\lbrack t_{0}%
,t_{2})\cup\lbrack t_{0}+T,t_{2}+T)\cup\lbrack t_{0}+2T,t_{2}+2T)\cup...$ and
let $k\in\mathbf{N}$ with $t\in\lbrack t_{0}+kT,t_{2}+kT).$ This means that,
on one hand $t\geq t_{0}>t^{\prime}$ and on the other hand%
\[
t+(-k-1)T<t_{2}-T\leq t^{\prime}<t_{0}\leq t-kT<t_{2},
\]
thus%
\[
\{t+zT|z\in\mathbf{Z}\}\cap\lbrack t_{2}-T,\infty
)=\{t+(-k)T,t+(-k+1)T,t+(-k+2)T,...\}
\]%
\[
=\{t+zT|z\in\mathbf{Z}\}\cap\lbrack t^{\prime},\infty)\overset{(\ref{pre60}%
)}{\subset}\mathbf{T}_{\mu}^{x}.
\]
This ends proving (\ref{p63}). For any $t^{\prime\prime}\in\lbrack
t_{0}^{\prime},t_{0}),$ we have that $t^{\prime\prime}\in I^{x},$
(\ref{pre62}) are true, see also Lemma \ref{Lem30}, page \pageref{Lem30}.

The supposition that $t^{\prime\prime}<t_{2}-T$ makes, from (\ref{pre78}),
that (\ref{pre62}) is false, thus $L_{\mu}^{x}=[t_{2}-T,\infty).$ The
supposition that $t^{\prime\prime}\geq t_{0}$ makes $t^{\prime\prime}\in
I^{x}$ be false.
\end{proof}

\begin{remark}
We use to think that the property of periodicity of $\mu\in Or(x)$ is
independent on the choice of the initial time=limit of periodicity in the
terms given by Theorem \ref{The5}.
\end{remark}

\begin{corollary}
\label{Cor3}We suppose that $\mu$ is a periodic point of the non constant
signal $x,$ that $T>0$ is its period and that $t^{\prime}$ is the initial time
of $x$ and the limit of periodicity of $\mu$ at the same time.

a) If $\mu=x(-\infty+0)$ and $t_{0}<t_{1}$ are defined by%
\begin{equation}
\forall t<t_{0},x(t)=x(-\infty+0), \label{pre80}%
\end{equation}%
\begin{equation}
x(t_{0})\neq x(-\infty+0), \label{pre81}%
\end{equation}%
\begin{equation}
\lbrack t_{1},t_{0}+T)\subset\mathbf{T}_{\mu}^{x}, \label{pre93}%
\end{equation}%
\begin{equation}
x(t_{1}-0)\neq\mu, \label{pre94}%
\end{equation}
then $t^{\prime}\in\lbrack t_{1}-T,t_{0});$

b) if $\mu\neq x(-\infty+0)$ and $t_{0}<t_{2}$ are defined by (\ref{pre80}),
(\ref{pre81}),
\begin{equation}
x(t_{2}-0)=\mu, \label{pre95}%
\end{equation}%
\begin{equation}
\lbrack t_{2},t_{0}+T)\cap\mathbf{T}_{\mu}^{x}=\varnothing, \label{pre96}%
\end{equation}
then $t^{\prime}\in\lbrack t_{2}-T,t_{0}).$
\end{corollary}

\begin{proof}
These are consequences of the second proof of Theorem \ref{The5}, page
\pageref{The5}.
\end{proof}

\section{A property of constancy}

\begin{theorem}
\label{The32}The signals $\widehat{x}\in\widehat{S}^{(n)},x\in S^{(n)}$ are considered.

a) If $\mu\in\widehat{Or}(\widehat{x})$ and the statement%
\begin{equation}
\forall k\in\widehat{\mathbf{T}}_{\mu}^{\widehat{x}},\{k+zp|z\in
\mathbf{Z}\}\cap\mathbf{N}_{\_}\subset\widehat{\mathbf{T}}_{\mu}^{\widehat{x}}
\label{per367}%
\end{equation}
is true for $p=1,$ then we have%
\begin{equation}
\forall k\in\mathbf{N}_{\_},\widehat{x}(k)=\mu\label{per369}%
\end{equation}
and (\ref{per367}) is true for any $p\geq1$.

b) Let $\mu\in Or(x)$ be some point and we suppose that $t_{0}\in
\mathbf{R},h>0$ exist such that $x$ has the form%
\begin{equation}%
\begin{array}
[c]{c}%
x(t)=x(-\infty+0)\cdot\chi_{(-\infty,t_{0})}(t)\oplus x(t_{0})\cdot
\chi_{\lbrack t_{0},t_{0}+h)}(t)\oplus...\\
...\oplus x(t_{0}+kh)\cdot\chi_{\lbrack t_{0}+kh,t_{0}+(k+1)h)}(t)\oplus...
\end{array}
\label{per484_}%
\end{equation}
If the statement%
\begin{equation}
\forall t\in\mathbf{T}_{\mu}^{x}\cap\lbrack t^{\prime},\infty),\{t+zT|z\in
\mathbf{Z}\}\cap\lbrack t^{\prime},\infty)\subset\mathbf{T}_{\mu}^{x}
\label{per368}%
\end{equation}
is true for some $t^{\prime}\in I^{x}$, $T\in(0,h)\cup(h,2h)\cup
...\cup(qh,(q+1)h)\cup...,$ then%
\begin{equation}
\forall t\in\mathbf{R},x(t)=\mu\label{per370}%
\end{equation}
holds and (\ref{per368}) is true for any $t^{\prime}\in\mathbf{R}$ and any
$T>0$.

c) If (\ref{per484_}) is true under the form%
\begin{equation}%
\begin{array}
[c]{c}%
x(t)=\widehat{x}(-1)\cdot\chi_{(-\infty,t_{0})}(t)\oplus\widehat{x}%
(0)\cdot\chi_{\lbrack t_{0},t_{0}+h)}(t)\oplus...\\
...\oplus\widehat{x}(k)\cdot\chi_{\lbrack t_{0}+kh,t_{0}+(k+1)h)}(t)\oplus...
\end{array}
\label{per366}%
\end{equation}
and $\mu\in\widehat{Or}(\widehat{x})=Or(x)$\footnote{The fact that
(\ref{per366}) implies $\widehat{Or}(\widehat{x})=Or(x)$ is proved at Theorem
\ref{The114} a), page \pageref{The114}.} is arbitrary, then

c.1) the satisfaction of (\ref{per367}) for $p=1$ implies that (\ref{per369}),
(\ref{per370}) are true, (\ref{per367}) holds for any $p\geq1$ and
(\ref{per368}) holds for any $t^{\prime}\in\mathbf{R}$ and any $T>0$;

c.2) the satisfaction of (\ref{per368}) for some $t^{\prime}\in I^{x}$,
$T\in(0,h)\cup(h,2h)\cup...\cup(qh,(q+1)h)\cup...$ implies also that
(\ref{per369}), (\ref{per370}) are true, (\ref{per367}) holds for any $p\geq1
$ and (\ref{per368}) holds for any $t^{\prime}\in\mathbf{R}$ and any $T>0.$
\end{theorem}

\begin{proof}
a) The statement (\ref{per367}) written for $p=1,$%
\begin{equation}
\forall k,\forall z,(k\in\widehat{\mathbf{T}}_{\mu}^{\widehat{x}}\text{ and
}z\in\mathbf{Z}\text{ and }k+z\geq-1)\Longrightarrow k+z\in\widehat
{\mathbf{T}}_{\mu}^{\widehat{x}}%
\end{equation}
together with $\widehat{\mathbf{T}}_{\mu}^{\widehat{x}}\neq\varnothing$ (since
$\mu\in\widehat{Or}(\widehat{x})$) implies that $\widehat{\mathbf{T}}_{\mu
}^{\widehat{x}}=\mathbf{N}_{\_},$ meaning the truth of (\ref{per369}). In
these circumstances (\ref{per367}) is true for any $p\geq1.$

b) If $\mu\in Or(x)$ and $t^{\prime}\in I^{x},$ then $\mathbf{T}_{\mu}^{x}%
\cap\lbrack t^{\prime},\infty)\neq\varnothing,$ and from (\ref{per368}) we
have $\mu\in\omega(x).$ The hypothesis asks furthermore that $T\in
(0,h)\cup(h,2h)\cup...\cup(qh,(q+1)h)\cup...$ and $t_{0}\in\mathbf{R},h>0$
exist making (\ref{per484_}) true. In this situation, Theorem \ref{The19},
page \pageref{The19} states that%
\begin{equation}
\forall t\geq t^{\prime},x(t)=\mu,
\end{equation}
and on the other hand we have%
\begin{equation}
\forall t\leq t^{\prime},x(t)=\mu.
\end{equation}
The statement (\ref{per370}) is true. In these conditions $I^{x}=\mathbf{R},$
$P_{\mu}^{x}=(0,\infty),$ thus (\ref{per368}) holds for any $t^{\prime}%
\in\mathbf{R}$ and any $T>0.$

c) This is a consequence of a) and b).
\end{proof}

\section{Discrete time vs real time}

\begin{theorem}
\label{The20_}Let the non constant signals $\widehat{x}\in\widehat{S}%
^{(n)},x\in S^{(n)}$ be related by%
\begin{equation}%
\begin{array}
[c]{c}%
x(t)=\widehat{x}(-1)\cdot\chi_{(-\infty,t_{0})}(t)\oplus\widehat{x}%
(0)\cdot\chi_{\lbrack t_{0},t_{0}+h)}(t)\oplus...\\
...\oplus\widehat{x}(k)\cdot\chi_{\lbrack t_{0}+kh,t_{0}+(k+1)h)}(t)\oplus...
\end{array}
\label{per141}%
\end{equation}
where $t_{0}\in\mathbf{R},h>0$ and let $\mu\in\widehat{Or}(\widehat
{x})=Or(x).$

a) If $p\geq1$ exists such that%
\begin{equation}
\forall k\in\widehat{\mathbf{T}}_{\mu}^{\widehat{x}},\{k+zp|z\in
\mathbf{Z}\}\cap\mathbf{N}_{\_}\subset\widehat{\mathbf{T}}_{\mu}^{\widehat{x}}
\label{per144*}%
\end{equation}
is true, then%
\begin{equation}
\exists t^{\prime}\in I^{x},\forall t\in\mathbf{T}_{\mu}^{x}\cap\lbrack
t^{\prime},\infty),\{t+zT|z\in\mathbf{Z}\}\cap\lbrack t^{\prime}%
,\infty)\subset\mathbf{T}_{\mu}^{x} \label{per145*}%
\end{equation}
holds for $T=ph.$

b) We presume that (\ref{per145*}) is true for some $T>0$. Then $\frac{T}%
{h}\in\{1,2,3,...\}$ and $k^{\prime}\in\mathbf{N}_{\_}$ exists such that%
\begin{equation}
\widehat{\mathbf{T}}_{\mu}^{\widehat{x}}\cap\{k^{\prime},k^{\prime
}+1,k^{\prime}+2,...\}\neq\varnothing, \label{p96}%
\end{equation}%
\begin{equation}
\left\{
\begin{array}
[c]{c}%
\forall k\in\widehat{\mathbf{T}}_{\mu}^{\widehat{x}}\cap\{k^{\prime}%
,k^{\prime}+1,k^{\prime}+2,...\},\\
\{k+zp|z\in\mathbf{Z}\}\cap\{k^{\prime},k^{\prime}+1,k^{\prime}+2,...\}\subset
\widehat{\mathbf{T}}_{\mu}^{\widehat{x}}%
\end{array}
\right.  \label{per527}%
\end{equation}
hold for $p=\frac{T}{h}.$
\end{theorem}

\begin{proof}
a) The existence of $p\geq1$ such that (\ref{per144*}) is true shows that
$\mu\in\widehat{\omega}(\widehat{x}),$ thus, as far as $\widehat{\omega
}(\widehat{x})=\omega(x),$ we infer $\mu\in\omega(x).$ The fact that $\mu
\in\widehat{\omega}(\widehat{x})$ is eventually periodic with the period $p $
implies, from Theorem \ref{The24}, page \pageref{The24}, that $\mu$ is
eventually periodic with the period $T=ph.$ As $x$ is not constant and
$L_{\mu}^{x}\neq\varnothing,$ we have, see Theorem \ref{The117}, page
\pageref{The117} the existence of $t_{0}^{\prime}\in\mathbf{R}$ with $L_{\mu
}^{x}=[t_{0}^{\prime},\infty).$ We claim that $t_{0}-h\geq t_{0}^{\prime}.$

Let us suppose against all reason that this is not the case, i.e. that%
\begin{equation}
\forall t\in\mathbf{T}_{\mu}^{x}\cap\lbrack t_{0}-h,\infty),\{t+zT|z\in
\mathbf{Z}\}\cap\lbrack t_{0}-h,\infty)\subset\mathbf{T}_{\mu}^{x}%
\end{equation}
is false. This means the existence of $t_{1}\in\mathbf{T}_{\mu}^{x},z_{1}%
\in\mathbf{Z}$ with $t_{1}\geq t_{0}-h,$ $t_{1}+z_{1}ph\geq t_{0}-h$ and
$t_{1}+z_{1}ph\notin\mathbf{T}_{\mu}^{x}.$ Then $k_{1}\in\mathbf{N}_{\_}$
exists such that $t_{1}\in\lbrack t_{0}+k_{1}h,t_{0}+(k_{1}+1)h),$
$t_{1}+z_{1}ph\in\lbrack t_{0}+(k_{1}+z_{1}p)h,t_{0}+(k_{1}+z_{1}p+1)h).$ We
have, as far as $k_{1}+z_{1}p\geq-1:$%
\[
\mu=x(t_{1})=\widehat{x}(k_{1})\overset{(\ref{per144*})}{=}\widehat{x}%
(k_{1}+z_{1}p)=x(t_{1}+z_{1}ph)=x(t_{1}+z_{1}T),
\]
contradiction with the way that $t_{1}$ was defined.

As $t_{0}-h\geq t_{0}^{\prime}$ we get the truth of a), since $(-\infty
,t_{0})\subset I^{x},[t_{0}-h,\infty)\subset L_{\mu}^{x}$ and $\varnothing
\neq\lbrack t_{0}-h,t_{0})\subset I^{x}\cap L_{\mu}^{x}.$

b) If $\mu\in Or(x)$ satisfies (\ref{per145*}), then $\mu\in\omega
(x)=\widehat{\omega}(\widehat{x}).$ The fact that $\frac{T}{h}\in
\{1,2,3,...\}$ and the existence of $k^{\prime}\in\mathbf{N}_{\_}$ such that
(\ref{p96}), (\ref{per527}) are fulfilled for $p=\frac{T}{h}$ result from
Theorem \ref{The24}, page \pageref{The24}.
\end{proof}

\begin{remark}
Theorem \ref{The20_} states, in a manner that updates Theorem \ref{The24} to
periodic points, that the discrete time and the real time periodicity of the
points are not equivalent when (\ref{per141}) is true: if $\mu\in\widehat
{Or}(\widehat{x})$ is periodic with the period $p$, then $\mu\in Or(x)$ is
periodic with the period $ph,$ while the converse implication takes place
under the form: if $\mu\in Or(x)$ is periodic with the period $T$, then
$\frac{T}{h}\in\{1,2,3,...\}$ and $\mu\in\widehat{Or}(\widehat{x})$ is
eventually periodic with the period $\frac{T}{h}.$
\end{remark}

\section{Support sets vs sets of periods}

\begin{remark}
Let the signals $\widehat{x},\widehat{y}\in\widehat{S}^{(n)},x,y\in S^{(n)}$
with $\mu\in\widehat{Or}(\widehat{x})\cap\widehat{Or}(\widehat{y}).$ The
implications%
\begin{equation}
\widehat{\mathbf{T}}_{\mu}^{\widehat{x}}=\widehat{\mathbf{T}}_{\mu}%
^{\widehat{y}}\Longrightarrow\widehat{P}_{\mu}^{\widehat{x}}=\widehat{P}_{\mu
}^{\widehat{y}},
\end{equation}%
\begin{equation}
\mathbf{T}_{\mu}^{x}=\mathbf{T}_{\mu}^{y}\Longrightarrow P_{\mu}^{x}=P_{\mu
}^{y}%
\end{equation}
are not true, in the sense given by Example \ref{Exa14}, page \pageref{Exa14}
and its discrete time counterpart: $\widehat{\mathbf{T}}_{\mu}^{\widehat{x}%
}=\widehat{\mathbf{T}}_{\mu}^{\widehat{y}}$ may take place and $\mu$ may be a
periodic point of $\widehat{x}$ and an eventually periodic point of
$\widehat{y}.$ If so, the equality $\widehat{P}_{\mu}^{\widehat{x}}%
=\widehat{P}_{\mu}^{\widehat{y}}$ refers to eventual periodicity, not to periodicity.
\end{remark}

\begin{remark}
Similarly, the implications%
\begin{equation}
\widehat{P}_{\mu}^{\widehat{x}}=\widehat{P}_{\mu}^{\widehat{y}}\Longrightarrow
\widehat{\mathbf{T}}_{\mu}^{\widehat{x}}=\widehat{\mathbf{T}}_{\mu}%
^{\widehat{y}}, \label{p226}%
\end{equation}%
\begin{equation}
P_{\mu}^{x}=P_{\mu}^{y}\Longrightarrow\mathbf{T}_{\mu}^{x}=\mathbf{T}_{\mu
}^{y} \label{p227}%
\end{equation}
are not true. For this, we take $x,y\in S^{(1)},$%
\[
x(t)=\chi_{\lbrack4,5)}(t)\oplus\chi_{\lbrack9,10)}(t)\oplus\chi
_{\lbrack14,15)}(t)\oplus...
\]%
\[
y(t)=\chi_{\lbrack2,3)}(t)\oplus\chi_{\lbrack4,5)}(t)\oplus\chi_{\lbrack
7,8)}(t)\oplus\chi_{\lbrack9,10)}(t)\oplus\chi_{\lbrack12,13)}(t)\oplus...
\]
$\mu=1$ is a periodic point of both $x,y$ with $I^{x}=(-\infty,4),I^{y}%
=(-\infty,2),P_{\mu}^{x}=P_{\mu}^{y}=\{5,10,15,...\}$ and $L_{\mu}^{x}=L_{\mu
}^{y}=[0,\infty).$ In $\mathbf{T}_{\mu}^{x}$ the interval $[4,5)$ repeats
within a period and in $\mathbf{T}_{\mu}^{y}$ the intervals $[2,3),[4,5)$
repeat within a period. The periods $T$ coincide for $x$ and $y$ and
(\ref{p226}), (\ref{p227}) are false in general.
\end{remark}

\section{Sums, differences and multiples of periods}

\begin{theorem}
\label{The3}The signals $\widehat{x},x$ are considered.

a) Let $p,p^{\prime}\geq1,$ $\mu\in\widehat{Or}(\widehat{x})$ and we ask that%
\begin{equation}
\forall k\in\widehat{\mathbf{T}}_{\mu}^{\widehat{x}},\{k+zp|z\in
\mathbf{Z}\}\cap\mathbf{N}_{\_}\subset\widehat{\mathbf{T}}_{\mu}^{\widehat{x}%
}, \label{per977}%
\end{equation}%
\begin{equation}
\forall k\in\widehat{\mathbf{T}}_{\mu}^{\widehat{x}},\{k+zp^{\prime}%
|z\in\mathbf{Z}\}\cap\mathbf{N}_{\_}\subset\widehat{\mathbf{T}}_{\mu
}^{\widehat{x}} \label{per978}%
\end{equation}
hold. We have $p+p^{\prime}\geq1,$%
\begin{equation}
\forall k\in\widehat{\mathbf{T}}_{\mu}^{\widehat{x}},\{k+z(p+p^{\prime}%
)|z\in\mathbf{Z}\}\cap\mathbf{N}_{\_}\subset\widehat{\mathbf{T}}_{\mu
}^{\widehat{x}} \label{per970}%
\end{equation}
and if $p>p^{\prime},$ then $p-p^{\prime}\geq1,$%
\begin{equation}
\forall k\in\widehat{\mathbf{T}}_{\mu}^{\widehat{x}},\{k+z(p-p^{\prime}%
)|z\in\mathbf{Z}\}\cap\mathbf{N}_{\_}\subset\widehat{\mathbf{T}}_{\mu
}^{\widehat{x}}. \label{per971}%
\end{equation}

b) Let $T,T^{\prime}>0,$ $t^{\prime}\in I^{x}\mathbf{,}$ $\mu\in Or(x)$ be
arbitrary with%
\begin{equation}
\forall t\in\mathbf{T}_{\mu}^{x}\cap\lbrack t^{\prime},\infty),\{t+zT|z\in
\mathbf{Z}\}\cap\lbrack t^{\prime},\infty)\subset\mathbf{T}_{\mu}^{x},
\label{per972}%
\end{equation}%
\begin{equation}
\forall t\in\mathbf{T}_{\mu}^{x}\cap\lbrack t^{\prime},\infty),\{t+zT^{\prime
}|z\in\mathbf{Z}\}\cap\lbrack t^{\prime},\infty)\subset\mathbf{T}_{\mu}^{x}
\label{per973}%
\end{equation}
fulfilled. We have on one hand that $T+T^{\prime}>0$ and
\begin{equation}
\forall t\in\mathbf{T}_{\mu}^{x}\cap\lbrack t^{\prime},\infty
),\{t+z(T+T^{\prime})|z\in\mathbf{Z}\}\cap\lbrack t^{\prime},\infty
)\subset\mathbf{T}_{\mu}^{x}%
\end{equation}
and on the other hand that $T>T^{\prime}$ implies $T-T^{\prime}>0$ and
\begin{equation}
\forall t\in\mathbf{T}_{\mu}^{x}\cap\lbrack t^{\prime},\infty
),\{t+z(T-T^{\prime})|z\in\mathbf{Z}\}\cap\lbrack t^{\prime},\infty
)\subset\mathbf{T}_{\mu}^{x}.
\end{equation}

\end{theorem}

\begin{proof}
This is the special case of Theorem \ref{The68}, page \pageref{The68} when the
eventually periodic points are periodic.
\end{proof}

\begin{theorem}
\label{The2}We consider the signals $\widehat{x},x.$

a) Let $p,k^{\prime}\geq1$ and $\mu\in\widehat{Or}(\widehat{x}).$ Then
$p^{\prime}=k^{\prime}p$ fulfills $p^{\prime}\geq1$ and%
\begin{equation}
\forall k\in\widehat{\mathbf{T}}_{\mu}^{\widehat{x}},\{k+zp|z\in
\mathbf{Z}\}\cap\mathbf{N}_{\_}\subset\widehat{\mathbf{T}}_{\mu}^{\widehat{x}}
\label{pre967}%
\end{equation}
implies%
\begin{equation}
\forall k\in\widehat{\mathbf{T}}_{\mu}^{\widehat{x}},\{k+zp^{\prime}%
|z\in\mathbf{Z}\}\cap\mathbf{N}_{\_}\subset\widehat{\mathbf{T}}_{\mu
}^{\widehat{x}}.
\end{equation}

b) Let $T>0,t^{\prime}\in I^{x}\mathbf{,}$ $k^{\prime}\geq1$ and $\mu\in
Or(x)$ be arbitrary. Then $T^{\prime}=k^{\prime}T$ fulfills $T^{\prime}>0$ and%
\[
\forall t\in\mathbf{T}_{\mu}^{x}\cap\lbrack t^{\prime},\infty),\{t+zT|z\in
\mathbf{Z}\}\cap\lbrack t^{\prime},\infty)\subset\mathbf{T}_{\mu}^{x}%
\]
implies%
\[
\forall t\in\mathbf{T}_{\mu}^{x}\cap\lbrack t^{\prime},\infty),\{t+zT^{\prime
}|z\in\mathbf{Z}\}\cap\lbrack t^{\prime},\infty)\subset\mathbf{T}_{\mu}^{x}.
\]

\end{theorem}

\begin{proof}
This is a direct consequence of Theorem \ref{The3}.
\end{proof}

\begin{remark}
Another way of expressing the statements of Theorem \ref{The2} is: if
$p\in\widehat{P}_{\mu}^{\widehat{x}},$ then $\{p,2p,3p,...\}\subset\widehat
{P}_{\mu}^{\widehat{x}}$ and if $T\in P_{\mu}^{x},$ then
$\{T,2T,3T,...\}\subset P_{\mu}^{x}.$
\end{remark}

\section{The set of the periods}

\begin{theorem}
\label{The26_}a) Let the signal $\widehat{x}\in\widehat{S}^{(n)}$ and $\mu
\in\widehat{Or}(\widehat{x}).$ We ask that $\mu$ is a periodic point of
$\widehat{x}$. Then some $\widetilde{p}\geq1$ exists such that%
\[
\widehat{P}_{\mu}^{\widehat{x}}=\{\widetilde{p},2\widetilde{p},3\widetilde
{p},...\}.
\]

b) We suppose that the signal $x\in S^{(n)}$ is not constant and we take some
$\mu\in$ $Or(x).$ We ask that $\mu$ is a periodic point of $x$. Then there is
$\widetilde{T}>0$ such that%
\[
P_{\mu}^{x}=\{\widetilde{T},2\widetilde{T},3\widetilde{T},...\}.
\]

\end{theorem}

\begin{proof}
This is a special case of Theorem \ref{The70}, page \pageref{The70}.
\end{proof}

\begin{remark}
An asymmetry occurs here, we have not asked in the hypothesis of Theorem
\ref{The26_}, item a) that $\widehat{x}$ is not constant; when $\widehat{x}$
is constant and equal with $\mu$ we have $\widetilde{p}=1$ and $\widehat
{P}_{\mu}^{\widehat{x}}=\{1,2,3,...\},$ thus the Theorem is still true. Like
in the case of the eventually periodic points, item b) of the Theorem does not
hold if $x$ is constant and equal with $\mu$, since in that case $P_{\mu}%
^{x}=(0,\infty).$
\end{remark}

\begin{theorem}
\label{The115}We suppose that the relation between $\widehat{x}$ and $x$ is
given by%
\begin{equation}%
\begin{array}
[c]{c}%
x(t)=\widehat{x}(-1)\cdot\chi_{(-\infty,t_{0})}(t)\oplus\widehat{x}%
(0)\cdot\chi_{\lbrack t_{0},t_{0}+h)}(t)\oplus\widehat{x}(1)\cdot\chi_{\lbrack
t_{0}+h,t_{0}+2h)}(t)\oplus...\\
...\oplus\widehat{x}(k)\cdot\chi_{\lbrack t_{0}+kh,t_{0}+(k+1)h)}(t)\oplus...
\end{array}
\end{equation}
where $t_{0}\in\mathbf{R}$ and $h>0$ and that $\mu\in\widehat{\omega}%
(\widehat{x})=\omega(x)$ is a periodic point of any of $\widehat{x},x.$ Then
two possibilities exist:

a) $\widehat{x},$ $x$ are both constant, $\widehat{P}_{\mu}^{\widehat{x}%
}=\{1,2,3,...\}$ and $P_{\mu}^{x}=(0,\infty);$

b) none of $\widehat{x},$ $x$ is constant, $\min\widehat{P}_{\mu}^{\widehat
{x}}=p>1$ and $\min P_{\mu}^{x}=T=ph$.
\end{theorem}

\begin{proof}
The fact that $\widehat{x},x$ are simultaneously constant or non constant is
obvious. We suppose that they are both non constant and we prove b). From
Theorem \ref{The20_}, page \pageref{The20_} we have that $p\in\widehat{P}%
_{\mu}^{\widehat{x}}\Longrightarrow T=ph\in P_{\mu}^{x}$ and conversely, $T\in
P_{\mu}^{x}\Longrightarrow p=\frac{T}{h}\in\widehat{P}_{\mu}^{\widehat{x}}.$
In particular, if $\widehat{P}_{\mu}^{\widehat{x}}=\{p,2p,3p,...\}$ and
$P_{\mu}^{x}=\{T,2T,3T,...\}$ (from Theorem \ref{The26_}, page
\pageref{The26_}), then $T=ph.$
\end{proof}

\section{Necessity conditions of periodicity}

\begin{theorem}
\label{The54}Let $\widehat{x}\in\widehat{S}^{(n)}$ non constant. For $\mu
\in\widehat{Or}(\widehat{x}),p\geq1$ we suppose that%
\begin{equation}
\forall k\in\widehat{\mathbf{T}}_{\mu}^{\widehat{x}},\{k+zp|z\in
\mathbf{Z}\}\cap\mathbf{N}_{\_}\subset\widehat{\mathbf{T}}_{\mu}^{\widehat{x}}
\label{pre97}%
\end{equation}
takes place. Then $n_{1},n_{2},...,n_{k_{1}}\in\{-1,0,...,p-2\},$ $k_{1}%
\geq1,$ exist such that%
\begin{equation}
\widehat{\mathbf{T}}_{\mu}^{\widehat{x}}=\underset{k\in\mathbf{N}}{%
{\displaystyle\bigcup}
}\{n_{1}+kp,n_{2}+kp,...,n_{k_{1}}+kp\} \label{pre99}%
\end{equation}
holds.
\end{theorem}

\begin{proof}
$\mu\in\widehat{Or}(\widehat{x})$ and (\ref{pre97}) imply $\mu\in
\widehat{\omega}(\widehat{x}).$ We apply Theorem \ref{The69}, page
\pageref{The69} written for $k^{\prime}=-1.$
\end{proof}

\begin{remark}
If $\widehat{x}$ is constant, then the previous Theorem takes the form
$\widehat{Or}(\widehat{x})=\{\mu\},p=1,k_{1}=1,n_{1}=-1$ and (\ref{pre99})
becomes $\widehat{\mathbf{T}}_{\mu}^{\widehat{x}}=\underset{k\in\mathbf{N}}{%
{\displaystyle\bigcup}
}\{-1+k\}=\mathbf{N}_{\_}.$
\end{remark}

\begin{theorem}
\label{The51_}The non constant signal $x\in S^{(n)}$ is considered and let the
point $\mu=x(-\infty+0)$ be given, together with $T>0,t^{\prime}\in I^{x}$
such that
\begin{equation}
\forall t\in\mathbf{T}_{\mu}^{x}\cap\lbrack t^{\prime},\infty),\{t+zT|z\in
\mathbf{Z}\}\cap\lbrack t^{\prime},\infty)\subset\mathbf{T}_{\mu}^{x}
\label{per8741}%
\end{equation}
holds. Then $t_{0},a_{1},b_{1},a_{2},b_{2},...,a_{k_{1}},b_{k_{1}}%
\in\mathbf{R,}$ $k_{1}\geq1$ exist such that%
\begin{equation}
\forall t<t_{0},x(t)=\mu, \label{p52}%
\end{equation}%
\begin{equation}
x(t_{0})\neq\mu, \label{p53}%
\end{equation}%
\begin{equation}
t_{0}<a_{1}<b_{1}<a_{2}<b_{2}<...<a_{k_{1}}<b_{k_{1}}=t_{0}+T, \label{p51}%
\end{equation}%
\begin{equation}
\lbrack a_{1},b_{1})\cup\lbrack a_{2},b_{2})\cup...\cup\lbrack a_{k_{1}%
},b_{k_{1}})=\mathbf{T}_{\mu}^{x}\cap\lbrack t_{0},t_{0}+T), \label{per8791}%
\end{equation}%
\begin{equation}
\mathbf{T}_{\mu}^{x}=%
\begin{array}
[c]{c}%
(-\infty,t_{0})\cup\underset{k\in\mathbf{N}}{%
{\displaystyle\bigcup}
}([a_{1}+kT,b_{1}+kT)\cup\lbrack a_{2}+kT,b_{2}+kT)\cup...\\
...\cup\lbrack a_{k_{1}}+kT,b_{k_{1}}+kT))
\end{array}
\label{p54}%
\end{equation}
hold.
\end{theorem}

\begin{proof}
A $t_{0}$ like at (\ref{p52}), (\ref{p53}) exists because $x$ is not constant
and we infer $I^{x}=(-\infty,t_{0}),$ $t^{\prime}<t_{0}.$ We have from Lemma
\ref{Lem36}, page \pageref{Lem36} that $\mathbf{T}_{\mu}^{x}\cap\lbrack
t^{\prime},\infty)\neq\varnothing,$ thus $\mu\in\omega(x)$ from (\ref{per8741}%
) and the fact that $\mathbf{T}_{\mu}^{x}\cap\lbrack t_{0},t_{0}%
+T)\neq\varnothing$ follows from Theorem \ref{Lem1}, page \pageref{Lem1}.

We have on one hand the existence of $\varepsilon>0$ with
\begin{equation}
\forall t\in\lbrack t_{0},t_{0}+\varepsilon),x(t)=x(t_{0})\overset
{(\ref{p53})}{\neq}\mu, \label{p58}%
\end{equation}
showing that $a_{1}=\min\mathbf{T}_{\mu}^{x}\cap\lbrack t_{0},t_{0}+T)>t_{0}.$
On the other hand we must show the existence of $b_{k_{1}}$ like at
(\ref{p51}), (\ref{per8791}). Indeed, we suppose against all reason that such
a $b_{k_{1}}$ does not exist and consequently that $a_{k_{1}}<b_{k_{1}}%
<t_{0}+T,$ $[a_{k_{1}},b_{k_{1}})\subset\mathbf{T}_{\mu}^{x}$ and $[b_{k_{1}%
},t_{0}+T)\cap\mathbf{T}_{\mu}^{x}=\varnothing.$ Let then $t\in\lbrack
\max\{b_{k_{1}},t^{\prime}+T\},t_{0}+T)$ arbitrary. We get
\[
b_{k_{1}}\leq\max\{b_{k_{1}},t^{\prime}+T\}\leq t<t_{0}+T
\]
i.e. $t\notin\mathbf{T}_{\mu}^{x}.$ We have also $t>t-T\geq t^{\prime}$ and
$t-T\in\lbrack t^{\prime},t_{0})\subset\mathbf{T}_{\mu}^{x},$ thus%
\[
t\in\{t-T+zT|z\in\mathbf{Z}\}\cap\lbrack t^{\prime},\infty)\overset
{(\ref{per8741})}{\subset}\mathbf{T}_{\mu}^{x},
\]
contradiction. The existence of $t_{0},a_{1},b_{1},a_{2},b_{2},...,a_{k_{1}%
},b_{k_{1}}$ like at (\ref{p52}),...,(\ref{per8791}) is proved.

We prove $\mathbf{T}_{\mu}^{x}\subset(-\infty,t_{0})\cup\underset
{k\in\mathbf{N}}{%
{\displaystyle\bigcup}
}([a_{1}+kT,b_{1}+kT)\cup\lbrack a_{2}+kT,b_{2}+kT)\cup...\cup\lbrack
a_{k_{1}}+kT,b_{k_{1}}+kT))$ and let $t\in\mathbf{T}_{\mu}^{x}$ arbitrary. If
$t<t_{0}$ the inclusion is obvious (from (\ref{p52})), so we can suppose now
that $t\geq t_{0}.$ We get from (\ref{per8741}) the existence of a finite
sequence $t,t-T,...,t-\overline{k}T\in\mathbf{T}_{\mu}^{x},\overline{k}\geq0$
with the property that $t-\overline{k}T\in\lbrack t_{0},t_{0}+T).$ We infer
from (\ref{per8791}) the existence of $j\in\{1,...,k_{1}\}$ with
$t-\overline{k}T\in\lbrack a_{j},b_{j})$ and we conclude that $t\in\lbrack
a_{j}+\overline{k}T,b_{j}+\overline{k}T)\in(-\infty,t_{0})\cup\underset
{k\in\mathbf{N}}{%
{\displaystyle\bigcup}
}([a_{1}+kT,b_{1}+kT)\cup\lbrack a_{2}+kT,b_{2}+kT)\cup...\cup\lbrack
a_{k_{1}}+kT,b_{k_{1}}+kT)).$

We prove $(-\infty,t_{0})\cup\underset{k\in\mathbf{N}}{%
{\displaystyle\bigcup}
}([a_{1}+kT,b_{1}+kT)\cup\lbrack a_{2}+kT,b_{2}+kT)\cup...\cup\lbrack
a_{k_{1}}+kT,b_{k_{1}}+kT))\subset\mathbf{T}_{\mu}^{x}$. The fact that
$(-\infty,t_{0})\subset\mathbf{T}_{\mu}^{x}$ coincides with (\ref{p52}) and we
take an arbitrary $t\in\underset{k\in\mathbf{N}}{%
{\displaystyle\bigcup}
}([a_{1}+kT,b_{1}+kT)\cup\lbrack a_{2}+kT,b_{2}+kT)\cup...\cup\lbrack
a_{k_{1}}+kT,b_{k_{1}}+kT)).$ Some $k\in\mathbf{N}$ and some $j\in
\{1,...,k_{1}\}$ exist with $t\in\lbrack a_{j}+kT,b_{j}+kT),$ thus
$t-kT\in\lbrack a_{j},b_{j})\subset\mathbf{T}_{\mu}^{x}\cap\lbrack t_{0}%
,t_{0}+T)\subset\mathbf{T}_{\mu}^{x}\cap\lbrack t^{\prime},\infty). $ In
particular we can see that $t\geq t-kT\geq t^{\prime}.$ We have%
\[
t\in\{t-kT+zT|z\in\mathbf{Z}\}\cap\lbrack t^{\prime},\infty)\overset
{(\ref{per8741})}{\subset}\mathbf{T}_{\mu}^{x},
\]
wherefrom we get $t\in\mathbf{T}_{\mu}^{x}.$ (\ref{p54}) is proved.
\end{proof}

\begin{theorem}
\label{The52 copy(1)}The signal $x\in S^{(n)}$ is not constant and let the
point $\mu\in Or(x),$ $\mu\neq x(-\infty+0)$, as well as $T>0,t^{\prime}\in
I^{x}$ with%
\begin{equation}
\forall t\in\mathbf{T}_{\mu}^{x}\cap\lbrack t^{\prime},\infty),\{t+zT|z\in
\mathbf{Z}\}\cap\lbrack t^{\prime},\infty)\subset\mathbf{T}_{\mu}^{x}
\label{pre30}%
\end{equation}
fulfilled. Then $t_{0},a_{1},b_{1},a_{2},b_{2},...,a_{k_{1}},b_{k_{1}}%
\in\mathbf{R,}$ $k_{1}\geq1$ exist such that%
\begin{equation}
\forall t<t_{0},x(t)=x(-\infty+0), \label{p55}%
\end{equation}%
\begin{equation}
x(t_{0})\neq x(-\infty+0), \label{p56}%
\end{equation}%
\begin{equation}
t_{0}\leq a_{1}<b_{1}<a_{2}<b_{2}<...<a_{k_{1}}<b_{k_{1}}<t_{0}+T,
\label{pre31}%
\end{equation}%
\begin{equation}
\lbrack a_{1},b_{1})\cup\lbrack a_{2},b_{2})\cup...\cup\lbrack a_{k_{1}%
},b_{k_{1}})=\mathbf{T}_{\mu}^{x}\cap\lbrack t_{0},t_{0}+T), \label{p57}%
\end{equation}%
\begin{equation}
\mathbf{T}_{\mu}^{x}=%
\begin{array}
[c]{c}%
\underset{k\in\mathbf{N}}{%
{\displaystyle\bigcup}
}([a_{1}+kT,b_{1}+kT)\cup\lbrack a_{2}+kT,b_{2}+kT)\cup...\\
...\cup\lbrack a_{k_{1}}+kT,b_{k_{1}}+kT))
\end{array}
\label{pre34}%
\end{equation}
are fulfilled.
\end{theorem}

\begin{proof}
As $x$ is not constant we get the existence of $t_{0}$ like in (\ref{p55}),
(\ref{p56}) and if we take in consideration that $I^{x}=(-\infty,t_{0})$, we
get $t^{\prime}<t_{0}.$

We have from Lemma \ref{Lem36}, page \pageref{Lem36} that $\mathbf{T}_{\mu
}^{x}\cap\lbrack t^{\prime},\infty)\neq\varnothing$ and, as $\mu\in\omega(x)$
(from (\ref{pre30})), the fact that $\mathbf{T}_{\mu}^{x}\cap\lbrack
t_{0},t_{0}+T)\neq\varnothing$ results from Theorem \ref{Lem1}, page
\pageref{Lem1}. We show that $b_{k_{1}}<t_{0}+T$ and for this we suppose
against all reason that $b_{k_{1}}=t_{0}+T.$ Let $t\in\lbrack\max\{a_{k_{1}%
},t^{\prime}+T\},t_{0}+T)$ arbitrary, fixed. We have $t>t-T\geq t^{\prime}$
and $t\in\lbrack a_{k_{1}},t_{0}+T)\subset\mathbf{T}_{\mu}^{x},$ thus we can
apply (\ref{pre30}):%
\[
t-T\in\{t+zT|z\in\mathbf{Z}\}\cap\lbrack t^{\prime},\infty)\subset
\mathbf{T}_{\mu}^{x}.
\]
Since $t-T\in\lbrack t^{\prime},t_{0}),$ we have reached the contradiction%
\[
\mu=x(t)=x(t-T)=x(-\infty+0).
\]

The fact that $a_{1},b_{1},a_{2},b_{2},...,a_{k_{1}},b_{k_{1}}$ exist making
(\ref{pre31}), (\ref{p57}) true is proved.

The proof of the equation (\ref{pre34}) is made like in the proof of Theorem
\ref{The51_}.
\end{proof}

\begin{remark}
The proofs of Theorem \ref{The51_} and Theorem \ref{The52 copy(1)} are similar
with the proof of Theorem \ref{The71}, page \pageref{The71} stating necessary
conditions of eventual periodicity of the points $\mu\in\omega(x),\mu
\in\widehat{\omega}(\widehat{x}).$
\end{remark}

\begin{remark}
Theorem \ref{The51_} and Theorem \ref{The52 copy(1)} are not special cases,
written for periodicity, of Theorem \ref{The71}, but rather versions of that
Theorem. To be compared (\ref{p50})$_{page\;\pageref{p50}}$ with
(\ref{p51})$_{page\;\pageref{p51}}$ and (\ref{pre31})$_{page\;\pageref{pre31}%
}$.
\end{remark}

\begin{example}
\label{Exa1}We take $x\in S^{(1)},$%
\[
x(t)=\chi_{(-\infty,0)}(t)\oplus\chi_{\lbrack1,2)}(t)\oplus\chi_{\lbrack
3,5)}(t)\oplus\chi_{\lbrack6,7)}(t)\oplus\chi_{\lbrack8,10)}(t)\oplus
\chi_{\lbrack11,12)}(t)\oplus...
\]
In this example, see Theorem \ref{The51_}, $\mu=1,t_{0}=0,k_{1}=2,T=5$ and
$t^{\prime}\in\lbrack-2,0).$
\end{example}

\section{Sufficiency conditions of periodicity}

\begin{theorem}
\label{The55}Let $\widehat{x}\in\widehat{S}^{(n)}$, $\mu\in\widehat
{Or}(\widehat{x}),$ $p\geq1$ and $n_{1},n_{2},...,n_{k_{1}}\in
\{-1,0,...,p-2\},$ $k_{1}\geq1,$ such that%
\begin{equation}
\widehat{\mathbf{T}}_{\mu}^{\widehat{x}}=\underset{k\in\mathbf{N}}{%
{\displaystyle\bigcup}
}\{n_{1}+kp,n_{2}+kp,...,n_{k_{1}}+kp\}. \label{pre102}%
\end{equation}
We have%
\begin{equation}
\forall k\in\widehat{\mathbf{T}}_{\mu}^{\widehat{x}},\{k+zp|z\in
\mathbf{Z}\}\cap\mathbf{N}_{\_}\subset\widehat{\mathbf{T}}_{\mu}^{\widehat{x}%
}. \label{pre100}%
\end{equation}

\end{theorem}

\begin{proof}
This is a special case of Theorem \ref{The72}, page \pageref{The72} written
for $k^{\prime}=-1.$
\end{proof}

\begin{theorem}
\label{The11}The signal $x\in S^{(n)}$ is given with $\mu=x(-\infty+0),$ $T>0$
and the numbers $t_{0},a_{1},b_{1},a_{2},b_{2},...,a_{k_{1}},b_{k_{1}}%
\in\mathbf{R}$, $k_{1}\geq1$ that fulfill
\begin{equation}
t_{0}<a_{1}<b_{1}<a_{2}<b_{2}<...<a_{k_{1}}<b_{k_{1}}=t_{0}+T, \label{p126}%
\end{equation}%
\begin{equation}
\mathbf{T}_{\mu}^{x}=%
\begin{array}
[c]{c}%
(-\infty,t_{0})\cup\underset{k\in\mathbf{N}}{%
{\displaystyle\bigcup}
}([a_{1}+kT,b_{1}+kT)\cup\lbrack a_{2}+kT,b_{2}+kT)\cup...\\
...\cup\lbrack a_{k_{1}}+kT,b_{k_{1}}+kT)).
\end{array}
\label{p64}%
\end{equation}
For any $t^{\prime}\in\lbrack a_{k_{1}}-T,t_{0}),$ the properties $t^{\prime
}\in I^{x},$%
\begin{equation}
\forall t\in\mathbf{T}_{\mu}^{x}\cap\lbrack t^{\prime},\infty),\{t+zT|z\in
\mathbf{Z}\}\cap\lbrack t^{\prime},\infty)\subset\mathbf{T}_{\mu}^{x}
\label{pre42}%
\end{equation}
hold.
\end{theorem}

\begin{proof}
Let $t^{\prime}\in\lbrack a_{k_{1}}-T,t_{0})$ arbitrary. From (\ref{p126}),
(\ref{p64}) we get $I^{x}=(-\infty,t_{0}),$ thus $t^{\prime}\in I^{x}.$

We infer%
\[
\mathbf{T}_{\mu}^{x}\cap\lbrack t^{\prime},\infty)=[t^{\prime},t_{0}%
)\cup\lbrack a_{1},b_{1})\cup...\cup\lbrack a_{k_{1}},b_{k_{1}})\cup\lbrack
a_{1}+T,b_{1}+T)\cup...
\]
and we take an arbitrary $t\in\mathbf{T}_{\mu}^{x}\cap\lbrack t^{\prime
},\infty).$ We have several possibilities.

a) Case $t\in\lbrack t^{\prime},t_{0}),$ when%
\[
\{t+zT|z\in\mathbf{Z}\}\cap\lbrack t^{\prime},\infty
)=\{t,t+T,t+2T,...\}\subset
\]%
\[
\subset\lbrack t^{\prime},t_{0})\cup\lbrack t^{\prime}+T,b_{k_{1}})\cup\lbrack
t^{\prime}+2T,b_{k_{1}}+T)\cup...\subset
\]%
\[
\subset(-\infty,t_{0})\cup\lbrack a_{k_{1}},b_{k_{1}})\cup\lbrack a_{k_{1}%
}+T,b_{k_{1}}+T)\cup...\subset\mathbf{T}_{\mu}^{x}.
\]
We have used the fact that%
\[
t-T<t_{0}-T<a_{k_{1}}-T\leq t^{\prime}\leq t<t_{0}.
\]

b) Case $t\in\lbrack a_{j}+kT,b_{j}+kT),k\geq0,j\in\{1,2,...,k_{1}-1\},$%
\[
\{t+zT|z\in\mathbf{Z}\}\cap\lbrack t^{\prime},\infty
)=\{t+(-k)T,t+(-k+1)T,t+(-k+2)T,...\}\subset
\]%
\[
\subset\lbrack a_{j},b_{j})\cup\lbrack a_{j}+T,b_{j}+T)\cup\lbrack
a_{j}+2T,b_{j}+2T)\cup...\subset\mathbf{T}_{\mu}^{x}%
\]
and we have used%
\[
t+(-k-1)T<t^{\prime}<t_{0}<a_{j}\leq t+(-k)T<b_{j}<t^{\prime}+T.
\]

c) Case $t\in\lbrack a_{k_{1}}+kT,b_{k_{1}}+kT),$ $k\geq0$ when there are two sub-cases,

c.1) Case $t\in\lbrack t^{\prime}+(k+1)T,b_{k_{1}}+kT),$%
\[
\{t+zT|z\in\mathbf{Z}\}\cap\lbrack t^{\prime},\infty
)=\{t+(-k-1)T,t+(-k)T,t+(-k+1)T,...\}\subset
\]%
\[
\subset\lbrack t^{\prime},t_{0})\cup\lbrack t^{\prime}+T,b_{k_{1}})\cup\lbrack
t^{\prime}+2T,b_{k_{1}}+T)\cup...\subset
\]%
\[
\subset(-\infty,t_{0})\cup\lbrack a_{k_{1}},b_{k_{1}})\cup\lbrack a_{k_{1}%
}+T,b_{k_{1}}+T)\cup...\subset\mathbf{T}_{\mu}^{x}%
\]
and we have used the fact that%
\[
t+(-k-2)T<t_{0}-T<a_{k_{1}}-T\leq t^{\prime}\leq t+(-k-1)T<t_{0}.
\]

c.2) Case $t\in\lbrack a_{k_{1}}+kT,t^{\prime}+(k+1)T),$%
\[
\{t+zT|z\in\mathbf{Z}\}\cap\lbrack t^{\prime},\infty
)=\{t+(-k)T,t+(-k+1)T,t+(-k+2)T,...\}\subset
\]%
\[
\subset\lbrack a_{k_{1}},t^{\prime}+T)\cup\lbrack a_{k_{1}}+T,t^{\prime
}+2T)\cup\lbrack a_{k_{1}}+2T,t^{\prime}+3T)\cup...\subset
\]%
\[
\subset\lbrack a_{k_{1}},b_{k_{1}})\cup\lbrack a_{k_{1}}+T,b_{k_{1}}%
+T)\cup\lbrack a_{k_{1}}+2T,b_{k_{1}}+2T)\cup...\subset\mathbf{T}_{\mu}^{x}%
\]
and we have used%
\[
t+(-k-1)T<t^{\prime}<t_{0}<a_{k_{1}}\leq t+(-k)T<t^{\prime}+T.
\]

(\ref{pre42}) holds.
\end{proof}

\begin{theorem}
\label{The11_}Let $x,$ $\mu\in Or(x),$ $\mu\neq x(-\infty+0),$ $T>0$ and the
numbers $t_{0},a_{1},b_{1},a_{2},b_{2},...,a_{k_{1}},b_{k_{1}}\in\mathbf{R,}$
$k_{1}\geq1,$ with the property that%
\begin{equation}
\forall t<t_{0},x(t)=x(-\infty+0), \label{p127}%
\end{equation}%
\begin{equation}
x(t_{0})\neq x(-\infty+0), \label{p128}%
\end{equation}%
\begin{equation}
b_{k_{1}}-T<t_{0}\leq a_{1}<b_{1}<a_{2}<b_{2}<...<a_{k_{1}}<b_{k_{1}},
\label{pre45}%
\end{equation}%
\begin{equation}
\mathbf{T}_{\mu}^{x}=\underset{k\in\mathbf{N}}{%
{\displaystyle\bigcup}
}([a_{1}+kT,b_{1}+kT)\cup\lbrack a_{2}+kT,b_{2}+kT)\cup...\cup\lbrack
a_{k_{1}}+kT,b_{k_{1}}+kT)). \label{pre46}%
\end{equation}
For any $t^{\prime}\in\lbrack b_{k_{1}}-T,t_{0}),$ we have $t^{\prime}\in
I^{x},$%
\begin{equation}
\forall t\in\mathbf{T}_{\mu}^{x}\cap\lbrack t^{\prime},\infty),\{t+zT|z\in
\mathbf{Z}\}\cap\lbrack t^{\prime},\infty)\subset\mathbf{T}_{\mu}^{x}.
\label{pre48}%
\end{equation}

\end{theorem}

\begin{proof}
Let $t^{\prime}\in\lbrack b_{k_{1}}-T,t_{0})$ be arbitrary. From (\ref{p127}),
(\ref{p128}) we infer $I^{x}=(-\infty,t_{0}),$ thus $t^{\prime}\in I^{x}.$

We get $\mathbf{T}_{\mu}^{x}\cap\lbrack t^{\prime},\infty)=\mathbf{T}_{\mu
}^{x}$ and we take an arbitrary $t\in\mathbf{T}_{\mu}^{x}\cap\lbrack
t^{\prime},\infty).$ Then $k\geq0$ and $j\in\{1,2,...,k_{1}\}$ exist such that
$t\in\lbrack a_{j}+kT,b_{j}+kT).$ We have:%
\[
\{t+zT|z\in\mathbf{Z}\}\cap\lbrack t^{\prime},\infty
)=\{t+(-k)T,t+(-k+1)T,t+(-k+2)T,...\}\subset
\]%
\[
\subset\lbrack a_{j},b_{j})\cup\lbrack a_{j}+T,b_{j}+T)\cup\lbrack
a_{j}+2T,b_{j}+2T)\cup...\subset\mathbf{T}_{\mu}^{x},
\]
where%
\[
t+(-k-1)T<t^{\prime}<t_{0}\leq a_{j}\leq t+(-k)T<b_{j}\leq t^{\prime}+T.
\]
(\ref{pre48}) holds.
\end{proof}

\begin{remark}
The proofs of Theorem \ref{The55}, page \pageref{The55}, Theorem \ref{The11},
page \pageref{The11} and Theorem \ref{The11_}, page \pageref{The11_} are
similar with the proofs of Theorem \ref{The69}, page \pageref{The69} and
Theorem \ref{The71}, page \pageref{The71} that state sufficient conditions of
eventual periodicity of the points $\mu\in\widehat{\omega}(\widehat{x}),\mu
\in\omega(x).$
\end{remark}

\section{A special case}

\begin{theorem}
\label{The120}Let $\widehat{x}\in\widehat{S}^{(n)}$, $\mu\in\widehat
{Or}(\widehat{x}),$ $p\geq1$ and $n_{1}\in\{-1,0,...,p-2\}$ such that%
\begin{equation}
\widehat{\mathbf{T}}_{\mu}^{\widehat{x}}=\{n_{1},n_{1}+p,n_{1}+2p,n_{1}%
+3p,...\}. \label{p15}%
\end{equation}
Then

a) $\mu$ is a periodic point of $\widehat{x}$ with the period $p:$%
\begin{equation}
\forall k\in\widehat{\mathbf{T}}_{\mu}^{\widehat{x}},\{k+zp|z\in
\mathbf{Z}\}\cap\mathbf{N}_{\_}\subset\widehat{\mathbf{T}}_{\mu}^{\widehat{x}%
};
\end{equation}

b) $p$ is the prime period of $\mu:$ for any $p^{\prime}$ with%
\begin{equation}
\forall k\in\widehat{\mathbf{T}}_{\mu}^{\widehat{x}},\{k+zp^{\prime}%
|z\in\mathbf{Z}\}\cap\mathbf{N}_{\_}\subset\widehat{\mathbf{T}}_{\mu
}^{\widehat{x}},
\end{equation}
we infer $p^{\prime}\in\{p,2p,3p,...\}.$
\end{theorem}

\begin{proof}
a) This is a special case of Theorem \ref{The55}, page \pageref{The55},
written for $k_{1}=1.$

b) We suppose against all reason that $p^{\prime}\in\widehat{P}_{\mu
}^{\widehat{x}}$ exists with $p^{\prime}<p.$ As $n_{1}\in\widehat{\mathbf{T}%
}_{\mu}^{\widehat{x}},$ we obtain that $n_{1}+p^{\prime}\in\widehat
{\mathbf{T}}_{\mu}^{\widehat{x}},$ contradiction with (\ref{p15}). Thus any
$p^{\prime}\in\widehat{P}_{\mu}^{\widehat{x}}$ fulfills $p^{\prime}\geq p. $
We apply Theorem \ref{The26_}, page \pageref{The26_}.
\end{proof}

\begin{theorem}
\label{The75}Let $x\in S^{(n)},$ $\mu=x(-\infty+0),$ $T>0$ and the points
$t_{0},a_{1},b_{1}\in\mathbf{R}$ having the property that%
\begin{equation}
t_{0}<a_{1}<b_{1}=t_{0}+T, \label{per902}%
\end{equation}%
\begin{equation}
\mathbf{T}_{\mu}^{x}=(-\infty,t_{0})\cup\lbrack a_{1},b_{1})\cup\lbrack
a_{1}+T,b_{1}+T)\cup\lbrack a_{1}+2T,b_{1}+2T)\cup... \label{per501}%
\end{equation}
hold$.$

a) For any $t^{\prime}\in\lbrack a_{1}-T,t_{0}),$ the properties $t^{\prime
}\in I^{x},$%
\begin{equation}
\forall t\in\mathbf{T}_{\mu}^{x}\cap\lbrack t^{\prime},\infty),\{t+zT|z\in
\mathbf{Z}\}\cap\lbrack t^{\prime},\infty)\subset\mathbf{T}_{\mu}^{x}
\label{per483}%
\end{equation}
are fulfilled.

b) Let $t^{\prime\prime}\in\lbrack a_{1}-T,t_{0})$ arbitrary. For any
$T^{\prime}>0$ such that%
\begin{equation}
\forall t\in\mathbf{T}_{\mu}^{x}\cap\lbrack t^{\prime\prime},\infty
),\{t+zT^{\prime}|z\in\mathbf{Z}\}\cap\lbrack t^{\prime\prime},\infty
)\subset\mathbf{T}_{\mu}^{x} \label{per504}%
\end{equation}
holds, we have $T^{\prime}\in\{T,2T,3T,...\}.$
\end{theorem}

\begin{proof}
a) This is a special case of Theorem \ref{The11}, page \pageref{The11},
written for $k_{1}=1.$

b) We suppose against all reason that $T^{\prime}<T.$ Let us note in the
beginning that%
\[
\max\{a_{1},b_{1}-T^{\prime}\}<\min\{b_{1},a_{1}+T-T^{\prime}\}
\]
is true, since all of $a_{1}<b_{1},a_{1}<a_{1}+T-T^{\prime},b_{1}-T^{\prime
}<b_{1},b_{1}-T^{\prime}<a_{1}+T-T^{\prime}$ hold. We infer that any
$t\in\lbrack\max\{a_{1},b_{1}-T^{\prime}\},\min\{b_{1},a_{1}+T-T^{\prime}\})$
fulfills $t\in\lbrack a_{1},b_{1})\subset\mathbf{T}_{\mu}^{x}\cap\lbrack
t^{\prime\prime},\infty)$ and%
\[
t+T^{\prime}\in\{t+zT^{\prime}|z\in\mathbf{Z}\}\cap\lbrack t^{\prime\prime
},\infty)\overset{(\ref{per504})}{\subset}\mathbf{T}_{\mu}^{x},
\]
and on the other hand we have%
\[
b_{1}\leq\max\{a_{1}+T^{\prime},b_{1}\}\leq t+T^{\prime}<\min\{b_{1}%
+T^{\prime},a_{1}+T\}\leq a_{1}+T,
\]
meaning that $t+T^{\prime}\notin\mathbf{T}_{\mu}^{x},$ contradiction. We
conclude that $T^{\prime}\geq T.$

We get $T=\min P_{\mu}^{x}$ and, as $P_{\mu}^{x}=\{T,2T,3T,...\}$ from Theorem
\ref{The26_}, page \pageref{The26_}, we have that $T^{\prime}\in
\{T,2T,3T,...\}.$
\end{proof}

\begin{theorem}
\label{The75__}Let $x\in S^{(n)},$ $\mu\in Or(x),$ $\mu\neq x(-\infty+0),$
$T>0$ and the points $t_{0},a_{1},b_{1}\in\mathbf{R}$ with the property that%
\begin{equation}
\forall t<t_{0},x(t)=x(-\infty+0), \label{per894}%
\end{equation}%
\begin{equation}
x(t_{0})\neq x(-\infty+0), \label{per895}%
\end{equation}%
\begin{equation}
b_{1}-T<t_{0}\leq a_{1}<b_{1}, \label{per896}%
\end{equation}%
\begin{equation}
\mathbf{T}_{\mu}^{x}=[a_{1},b_{1})\cup\lbrack a_{1}+T,b_{1}+T)\cup\lbrack
a_{1}+2T,b_{1}+2T)\cup... \label{per897}%
\end{equation}
hold$.$

a) For any $t^{\prime}\in\lbrack b_{1}-T,t_{0}),$ the following properties:
$t^{\prime}\in I^{x},$%
\begin{equation}
\forall t\in\mathbf{T}_{\mu}^{x}\cap\lbrack t^{\prime},\infty),\{t+zT|z\in
\mathbf{Z}\}\cap\lbrack t^{\prime},\infty)\subset\mathbf{T}_{\mu}^{x}
\label{per899}%
\end{equation}
are fulfilled.

b) Let $t^{\prime\prime}\in\lbrack b_{1}-T,t_{0})$ arbitrary. For any
$T^{\prime}>0$ such that%
\begin{equation}
\forall t\in\mathbf{T}_{\mu}^{x}\cap\lbrack t^{\prime\prime},\infty
),\{t+zT^{\prime}|z\in\mathbf{Z}\}\cap\lbrack t^{\prime\prime},\infty
)\subset\mathbf{T}_{\mu}^{x} \label{per901}%
\end{equation}
is true, we have $T^{\prime}\in\{T,2T,3T,...\}.$
\end{theorem}

\begin{proof}
a) This is a special case of Theorem \ref{The11_}, page \pageref{The11_},
written for $k_{1}=1.$

b) We suppose against all reason now that $T^{\prime}<T.$ Let us notice the
truth of%
\[
\max\{a_{1},b_{1}-T^{\prime}\}<\min\{b_{1},a_{1}+T-T^{\prime}\}.
\]
We infer that $t\in\lbrack\max\{a_{1},b_{1}-T^{\prime}\},\min\{b_{1}%
,a_{1}+T-T^{\prime}\})$ satisfies $t\in\lbrack a_{1},b_{1})\subset
\mathbf{T}_{\mu}^{x}\cap\lbrack t^{\prime\prime},\infty)$ and%
\[
t+T^{\prime}\in\{t+zT^{\prime}|z\in\mathbf{Z}\}\cap\lbrack t^{\prime\prime
},\infty)\overset{(\ref{per901})}{\subset}\mathbf{T}_{\mu}^{x},
\]
thus $t+T^{\prime}\in\mathbf{T}_{\mu}^{x};$ on the other hand%
\[
b_{1}\leq\max\{a_{1}+T^{\prime},b_{1}\}\leq t+T^{\prime}<\min\{b_{1}%
+T^{\prime},a_{1}+T\}\leq a_{1}+T,
\]
wherefrom $t+T^{\prime}\notin\mathbf{T}_{\mu}^{x},$ contradiction. We have
proved that $T^{\prime}\geq T.$

We get that $T=\min P_{\mu}^{x}.$ Theorem \ref{The26_}, page \pageref{The26_}
shows that $P_{\mu}^{x}=\{T,2T,3T,...\},$ wherefrom $T^{\prime}\in
\{T,2T,3T,...\}.$
\end{proof}

\begin{remark}
Theorems \ref{The75}, \ref{The75__} represent the same phenomenon and their
proof is formally the same: when $\mathbf{T}_{\mu}^{x}$ has one of the forms
(\ref{per501}), (\ref{per897}), the prime period of $\mu$ is $T$. The
difference between the Theorems is given by the fact that $\mu=x(-\infty+0) $
in the first case and $\mu\neq x(-\infty+0)$ in the second case.
\end{remark}

\section{Periodic points vs. eventually periodic points}

\begin{theorem}
\label{The84}a) Let $\widehat{x}$ and the periodic point $\mu\in\widehat
{Or}(\widehat{x});$ for any $\widetilde{k}\in\mathbf{N,}$ we have $\widehat
{P}_{\mu}^{\widehat{x}}=\widehat{P}_{\mu}^{\widehat{\sigma}^{\widetilde{k}%
}(\widehat{x})}.$

b) We consider $x$ and the periodic point $\mu\in Or(x);$ for any
$\widetilde{t}\in\mathbf{R,}$ we have $P_{\mu}^{x}=P_{\mu}^{\sigma
^{\widetilde{t}}(x)}.$
\end{theorem}

\begin{proof}
a) The hypothesis states that $\widehat{P}_{\mu}^{\widehat{x}}\neq\varnothing$
and let $\widetilde{k}\in\mathbf{N}$ arbitrary.

We prove $\widehat{P}_{\mu}^{\widehat{x}}\subset\widehat{P}_{\mu}%
^{\widehat{\sigma}^{\widetilde{k}}(\widehat{x})}.$ Let $p\in\widehat{P}_{\mu
}^{\widehat{x}}$ arbitrary, thus%
\begin{equation}
\forall k\in\widehat{\mathbf{T}}_{\mu}^{\widehat{x}},\{k+zp|z\in
\mathbf{Z}\}\cap\mathbf{N}_{\_}\subset\widehat{\mathbf{T}}_{\mu}^{\widehat{x}}
\label{pre344}%
\end{equation}
holds and we must show that $\mu\in\widehat{Or}(\widehat{\sigma}%
^{\widetilde{k}}(\widehat{x}))$ and%
\begin{equation}
\forall k\in\widehat{\mathbf{T}}_{\mu}^{\widehat{\sigma}^{\widetilde{k}%
}(\widehat{x})},\{k+zp|z\in\mathbf{Z}\}\cap\mathbf{N}_{\_}\subset
\widehat{\mathbf{T}}_{\mu}^{\widehat{\sigma}^{\widetilde{k}}(\widehat{x})}.
\label{pre345}%
\end{equation}

If $\mu\in\widehat{Or}(\widehat{x}),$ then $\widehat{\mathbf{T}}_{\mu
}^{\widehat{x}}\neq\varnothing$ and from (\ref{pre344}) we infer that $\mu
\in\widehat{\omega}(\widehat{x}).$ Theorem \ref{The12}, page \pageref{The12}
shows that $\widehat{\omega}(\widehat{\sigma}^{\widetilde{k}}(\widehat
{x}))=\widehat{\omega}(\widehat{x}),$ hence $\mu\in\widehat{\omega}%
(\widehat{\sigma}^{\widetilde{k}}(\widehat{x}))\subset\widehat{Or}%
(\widehat{\sigma}^{\widetilde{k}}(\widehat{x}))$ and $\widehat{\mathbf{T}%
}_{\mu}^{\widehat{\sigma}^{\widetilde{k}}(\widehat{x})}\neq\varnothing.$

Let now $k\in\widehat{\mathbf{T}}_{\mu}^{\widehat{\sigma}^{\widetilde{k}%
}(\widehat{x})}$ and $z\in\mathbf{Z}$ with $k+zp\geq-1,$ meaning that
$\widehat{x}(k+\widetilde{k})=\mu.$ We have $k+\widetilde{k}\in\widehat
{\mathbf{T}}_{\mu}^{\widehat{x}}$ and $k+\widetilde{k}+zp\geq-1,$ thus we can
apply (\ref{pre344}). We infer that $k+\widetilde{k}+zp\in\widehat{\mathbf{T}%
}_{\mu}^{\widehat{x}},$ wherefrom $\mu=\widehat{x}(k+\widetilde{k}%
+zp)=\widehat{\sigma}^{\widetilde{k}}(\widehat{x})(k+zp)$ and, finally,
$k+zp\in\widehat{\mathbf{T}}_{\mu}^{\widehat{\sigma}^{\widetilde{k}}%
(\widehat{x})}.$

We prove $\widehat{P}_{\mu}^{\widehat{\sigma}^{\widetilde{k}}(\widehat{x}%
)}\subset\widehat{P}_{\mu}^{\widehat{x}}.$ Let $p\in\widehat{P}_{\mu
}^{\widehat{x}},$ thus (\ref{pre344}) is true. We suppose against all reason
that $\widehat{P}_{\mu}^{\widehat{\sigma}^{\widetilde{k}}(\widehat{x})}%
\subset\widehat{P}_{\mu}^{\widehat{x}}$ is false, i.e. some $p^{\prime}%
\in\widehat{P}_{\mu}^{\widehat{\sigma}^{\widetilde{k}}(\widehat{x}%
)}\smallsetminus\widehat{P}_{\mu}^{\widehat{x}}$ exists. This means the truth
of%
\begin{equation}
\forall k\in\widehat{\mathbf{T}}_{\mu}^{\widehat{x}}\cap\{\widetilde
{k}-1,\widetilde{k},\widetilde{k}+1,...\},\{k+zp^{\prime}|z\in\mathbf{Z}%
\}\cap\{\widetilde{k}-1,\widetilde{k},\widetilde{k}+1,...\}\subset
\widehat{\mathbf{T}}_{\mu}^{\widehat{x}}, \label{p84}%
\end{equation}%
\begin{equation}
\exists k_{1}\in\widehat{\mathbf{T}}_{\mu}^{\widehat{x}},\exists z_{1}%
\in\mathbf{Z},k_{1}+z_{1}p^{\prime}\geq-1\text{ and }k_{1}+z_{1}p^{\prime
}\notin\widehat{\mathbf{T}}_{\mu}^{\widehat{x}}. \label{p85}%
\end{equation}
Let $\overline{k}\in\mathbf{N}$ having the property that $k_{1}+\overline
{k}p\geq\widetilde{k}-1,$ $k_{1}+\overline{k}p+z_{1}p^{\prime}\geq
\widetilde{k}-1.$ We have%
\begin{equation}
\mu=\widehat{x}(k_{1})\overset{(\ref{pre344})}{=}\widehat{x}(k_{1}%
+\overline{k}p)\overset{(\ref{p84})}{=}\widehat{x}(k_{1}+\overline{k}%
p+z_{1}p^{\prime})\overset{(\ref{pre344})}{=}\widehat{x}(k_{1}+z_{1}p^{\prime
}). \label{p86}%
\end{equation}
The statements (\ref{p85}), (\ref{p86}) are contradictory.

b) We suppose that $P_{\mu}^{x}\neq\varnothing$ and let $\widetilde{t}%
\in\mathbf{R}$ arbitrary, fixed.

We prove $P_{\mu}^{x}\subset P_{\mu}^{\sigma^{\widetilde{t}}(x)}.$ Let $T\in
P_{\mu}^{x}$ arbitrary$,$ thus $t^{\prime}\in I^{x}$ exists such that%
\begin{equation}
\forall t\in\mathbf{T}_{\mu}^{x}\cap\lbrack t^{\prime},\infty),\{t+zT|z\in
\mathbf{Z}\}\cap\lbrack t^{\prime},\infty)\subset\mathbf{T}_{\mu}^{x}.
\label{pre346}%
\end{equation}
We must show that $\mu\in Or(\sigma^{\widetilde{t}}(x))$ and $t^{\prime\prime
}\in I^{\sigma^{\widetilde{t}}(x)}$ exists such that%
\begin{equation}
\forall t\in\mathbf{T}_{\mu}^{\sigma^{\widetilde{t}}(x)}\cap\lbrack
t^{\prime\prime},\infty),\{t+zT|z\in\mathbf{Z}\}\cap\lbrack t^{\prime\prime
},\infty)\subset\mathbf{T}_{\mu}^{\sigma^{\widetilde{t}}(x)}. \label{pre347}%
\end{equation}

From $\mu\in Or(x),$ $t^{\prime}\in I^{x}$ and Lemma \ref{Lem36}, page
\pageref{Lem36} we have that $\mathbf{T}_{\mu}^{x}\cap\lbrack t^{\prime
},\infty)\neq\varnothing;$ from (\ref{pre346}) we infer that $\mathbf{T}_{\mu
}^{x}$ is superiorly unbounded, wherefrom we have that $\mu\in\omega(x).$

Theorem \ref{The12}, page \pageref{The12} shows that $\omega(x)=\omega
(\sigma^{\widetilde{t}}(x))$ and we infer $\mu\in\omega(\sigma^{\widetilde{t}%
}(x))\subset Or(\sigma^{\widetilde{t}}(x)).$ In particular $\mathbf{T}_{\mu
}^{\sigma^{\widetilde{t}}(x)}$ is superiorly unbounded and $\mathbf{T}_{\mu
}^{\sigma^{\widetilde{t}}(x)}\cap\lbrack t^{\prime\prime},\infty
)\neq\varnothing$ is true for any $t^{\prime\prime}\in\mathbf{R}.$

If $x$ is constant, then $\sigma^{\widetilde{t}}(x)=x$ and $t^{\prime\prime
}\in I^{\sigma^{\widetilde{t}}(x)}$, (\ref{pre347}) take place trivially for
any $t^{\prime\prime},$ thus we shall suppose from now that $x$ is not
constant and consequently some $t_{0}\in\mathbf{R}$ exists with%
\begin{equation}
\forall t<t_{0},x(t)=x(-\infty+0), \label{p93}%
\end{equation}%
\begin{equation}
x(t_{0})\neq x(-\infty+0). \label{p94}%
\end{equation}
From (\ref{p93}), (\ref{p94}) we have $I^{x}=(-\infty,t_{0})$ and since
$t^{\prime}\in I^{x},$ we get $t^{\prime}<t_{0}.$ Two possibilities exist.

Case $\widetilde{t}\leq t_{0}$

In this situation $\sigma^{\widetilde{t}}(x)=x$ and $t^{\prime\prime}\in
I^{\sigma^{\widetilde{t}}(x)}$, (\ref{pre347}) take place with $t^{\prime
\prime}=t^{\prime}.$

Case $\widetilde{t}>t_{0}$

Some $\varepsilon>0$ exists with $\forall t\in(\widetilde{t}-\varepsilon
,\widetilde{t}),x(t)=x(\widetilde{t}-0)$ and we infer from here that
$\widetilde{t}-\varepsilon\geq t_{0}>t^{\prime}.$ We take $t^{\prime\prime}%
\in(\widetilde{t}-\varepsilon,\widetilde{t})$ arbitrary, fixed. We have%
\[
\sigma^{\widetilde{t}}(x)(t)=\left\{
\begin{array}
[c]{c}%
x(\widetilde{t}-0),t<\widetilde{t}\\
x(t),t\geq t^{\prime\prime}%
\end{array}
\right.  .
\]
The statement $t^{\prime\prime}\in I^{\sigma^{\widetilde{t}}(x)}$ is true. In
order to prove the fulfillment of (\ref{pre347}), let $t\in\mathbf{T}_{\mu
}^{\sigma^{\widetilde{t}}(x)}\cap\lbrack t^{\prime\prime},\infty)$ and
$z\in\mathbf{Z}$ arbitrary, with $t+zT\geq t^{\prime\prime}.$ We have
$t\in\mathbf{T}_{\mu}^{x},$ $t\geq t^{\prime\prime}>t_{0}>t^{\prime}$ and
$t+zT\geq t^{\prime\prime}>t_{0}>t^{\prime}$ thus (\ref{pre346}) can be
applied. We get $t+zT\in\mathbf{T}_{\mu}^{x}.$ As far as $\mu=x(t+zT)=\sigma
^{\widetilde{t}}(x)(t+zT),$ we conclude that $t+zT\in\mathbf{T}_{\mu}%
^{\sigma^{\widetilde{t}}(x)}.$

We prove $P_{\mu}^{\sigma^{\widetilde{t}}(x)}\subset P_{\mu}^{x}.$ Let $T\in
P_{\mu}^{x}$ arbitrary$,$ thus $t^{\prime}\in I^{x}$ exists with
(\ref{pre346}) fulfilled. We suppose against all reason that $P_{\mu}%
^{\sigma^{\widetilde{t}}(x)}\subset P_{\mu}^{x}$ is false, i.e. $T^{\prime}\in
P_{\mu}^{\sigma^{\widetilde{t}}(x)}\smallsetminus P_{\mu}^{x}$ exists. This
means the existence of $t^{\prime\prime}\in I^{\sigma^{\widetilde{t}}(x)}$
with%
\begin{equation}
\forall t\in\mathbf{T}_{\mu}^{\sigma^{\widetilde{t}}(x)}\cap\lbrack
t^{\prime\prime},\infty),\{t+zT^{\prime}|z\in\mathbf{Z}\}\cap\lbrack
t^{\prime\prime},\infty)\subset\mathbf{T}_{\mu}^{\sigma^{\widetilde{t}}(x)},
\label{p90}%
\end{equation}%
\begin{equation}
\forall t^{\prime\prime\prime}\in I^{x},\exists t_{1}\in\mathbf{T}_{\mu}%
^{x}\cap\lbrack t^{\prime\prime\prime},\infty),\exists z_{1}\in\mathbf{Z}%
,t_{1}+z_{1}T^{\prime}\geq t^{\prime\prime\prime}\text{ and }t_{1}%
+z_{1}T^{\prime}\notin\mathbf{T}_{\mu}^{x}. \label{p91}%
\end{equation}
Let $t^{\prime\prime\prime}\geq t^{\prime}$ arbitrary such that $(-\infty
,t^{\prime\prime\prime}]\subset\mathbf{T}_{x(-\infty+0)}^{x}$ and
$\overline{k}\in N$ with the property that $t_{1}+\overline{k}T\geq
\max\{t^{\prime\prime},\widetilde{t}\},t_{1}+\overline{k}T+z_{1}T^{\prime}%
\geq\max\{t^{\prime\prime},\widetilde{t}\}.$ We have%
\begin{equation}%
\begin{array}
[c]{c}%
\mu=x(t_{1})\overset{(\ref{pre346})}{=}x(t_{1}+\overline{k}T)=\sigma
^{\widetilde{t}}(x)(t_{1}+\overline{k}T)\\
\overset{(\ref{p90})}{=}\sigma^{\widetilde{t}}(x)(t_{1}+\overline{k}%
T+z_{1}T^{\prime})=x(t_{1}+\overline{k}T+z_{1}T^{\prime})\overset
{(\ref{pre346})}{=}x(t_{1}+z_{1}T^{\prime}).
\end{array}
\label{p92}%
\end{equation}
The statements (\ref{p91}), (\ref{p92}) are contradictory.
\end{proof}

\begin{remark}
\label{Cor5}In Theorem \ref{The84}, the statements about the eventual
periodicity of $\mu\in\widehat{Or}(\widehat{x}),\mu\in Or(x)$ are in fact
statements about the periodicity of $\mu\in\widehat{Or}(\widehat{\sigma
}^{\widetilde{k}}(\widehat{x})),\mu\in Or(\sigma^{\widetilde{t}}(x)).$
\end{remark}

\begin{theorem}
\label{The131}a) If $\mu\in\widehat{\omega}(\widehat{x})$ is an eventually
periodic point of $\widehat{x}:\exists p\geq1,\exists p^{\prime}\geq1,\exists
k^{\prime}\in\mathbf{N},\exists k^{\prime\prime}\in\mathbf{N}$ with%
\begin{equation}
\forall k\in\widehat{\mathbf{T}}_{\mu}^{\widehat{\sigma}^{k^{\prime}}%
(\widehat{x})},\{k+zp|z\in\mathbf{Z}\}\cap\mathbf{N}_{\_}\subset
\widehat{\mathbf{T}}_{\mu}^{\widehat{\sigma}^{k^{\prime}}(\widehat{x})},
\label{p18}%
\end{equation}%
\begin{equation}
\forall k\in\widehat{\mathbf{T}}_{\mu}^{\widehat{\sigma}^{k^{\prime\prime}%
}(\widehat{x})},\{k+zp^{\prime}|z\in\mathbf{Z}\}\cap\mathbf{N}_{\_}%
\subset\widehat{\mathbf{T}}_{\mu}^{\widehat{\sigma}^{k^{\prime\prime}%
}(\widehat{x})} \label{p19}%
\end{equation}
fulfilled, then $\widehat{P}_{\mu}^{\widehat{\sigma}^{k^{\prime}}(\widehat
{x})}=\widehat{P}_{\mu}^{\widehat{\sigma}^{k^{\prime\prime}}(\widehat{x})}.$

b) If $\mu\in\omega(x)$ is an eventually periodic point of $x:\exists
T>0,\exists T^{\prime}>0,\exists t^{\prime}\in\mathbf{R,}\exists
t^{\prime\prime}\in\mathbf{R}$ such that%
\begin{equation}
\forall t\in\mathbf{T}_{\mu}^{x}\cap\lbrack t^{\prime},\infty),\{t+zT|z\in
\mathbf{Z}\}\cap\lbrack t^{\prime},\infty)\subset\mathbf{T}_{\mu}^{x},
\end{equation}%
\begin{equation}
\forall t\in\mathbf{T}_{\mu}^{x}\cap\lbrack t^{\prime\prime},\infty
),\{t+zT^{\prime}|z\in\mathbf{Z}\}\cap\lbrack t^{\prime\prime},\infty
)\subset\mathbf{T}_{\mu}^{x}%
\end{equation}
are true, then $P_{\mu}^{\sigma^{t^{\prime}}(x)}=P_{\mu}^{\sigma
^{t^{\prime\prime}}(x)}.$
\end{theorem}

\begin{proof}
a) Both (\ref{p18}) and (\ref{p19}) are equivalent with $\widehat{P}_{\mu
}^{\widehat{x}}\neq\varnothing.$ If they are fulfilled, then $\widehat{P}%
_{\mu}^{\widehat{\sigma}^{k^{\prime}}(\widehat{x})}=\widehat{P}_{\mu
}^{\widehat{x}}=\widehat{P}_{\mu}^{\widehat{\sigma}^{k^{\prime\prime}%
}(\widehat{x})}.$
\end{proof}

\section{Further research}

\begin{remark}
Item b) in Theorem \ref{The75}, page \pageref{The75} does not work in the
general case, when $\mathbf{T}_{\mu}^{x}$ is given by (\ref{p64}%
)$_{page\;\pageref{p64}},$ instead of (\ref{per501})$_{page\;\pageref{per501}%
}.$ In order to understand the phenomenon, one may consider the case of $x$
from Example \ref{Exa1}, page \pageref{Exa1}: $x\in S^{(1)},$%
\[
x(t)=\chi_{(-\infty,0)}(t)\oplus\chi_{\lbrack1,2)}(t)\oplus\chi_{\lbrack
3,5)}(t)\oplus\chi_{\lbrack6,7)}(t)\oplus\chi_{\lbrack8,10)}(t)\oplus
\chi_{\lbrack11,12)}(t)\oplus...
\]
with the periodic point $\mu=1,$ $p\in\{5,10,15,...\}$ and $k_{1}%
\in\{2,4,6,...\}.$ The same is true for item b) in Theorem \ref{The75__}, page
\pageref{The75__} as we can see by observing the behavior of the periodic
point $\mu=0$ of the previous function for $p\in\{5,10,15,...\}$ and $k_{1}%
\in\{2,4,6,...\}.$ A generalization of Theorem \ref{The75} b) and Theorem
\ref{The75__} b) is required.
\end{remark}

\begin{remark}
Theorem \ref{The20_}, page \pageref{The20_} referring to the periodic point
$\mu\in Or(x)$ is continued by Theorem \ref{The43}, page \pageref{The43} and
Theorem \ref{The18} to follow. The statement (\ref{per435}%
)$_{page\;\pageref{per435}}$ of Theorem \ref{The18} suggests that Theorem
\ref{The20_} b) can be strengthened.
\end{remark}

\begin{remark}
\label{Rem21}Let the non constant signals $\widehat{x},x$ and we think if the
compatibility properties%
\[
\forall\mu\in\widehat{Or}(\widehat{x}),\forall\mu^{\prime}\in\widehat
{Or}(\widehat{x}),(\widehat{P}_{\mu}^{\widehat{x}}\neq\varnothing\text{ and
}\widehat{P}_{\mu^{\prime}}^{\widehat{x}}\neq\varnothing)\Longrightarrow
\widehat{P}_{\mu}^{\widehat{x}}\cap\widehat{P}_{\mu^{\prime}}^{\widehat{x}%
}\neq\varnothing,
\]%
\[
\forall\mu\in Or(x),\forall\mu^{\prime}\in Or(x),(P_{\mu}^{x}\neq
\varnothing\text{ and }P_{\mu^{\prime}}^{x}\neq\varnothing)\Longrightarrow
P_{\mu}^{x}\cap P_{\mu^{\prime}}^{x}\neq\varnothing
\]
hold, see Remark \ref{Rem23}, page \pageref{Rem23}. Proving the first one is
trivial, while the second one has no proof so far. Taking into account the
form of the sets of periods, the above statements give the suggestions that,
see Theorem \ref{The26}, page \pageref{The26}:

a) $\widehat{P}_{\mu}^{\widehat{x}}=\{p,2p,3p,...\},\widehat{P}_{\mu^{\prime}%
}^{\widehat{x}}=\{p^{\prime},2p^{\prime},3p^{\prime},...\}$ imply the
existence of $n_{1},n_{2}\geq1$ relatively prime such that $n_{1}%
p=n_{2}p^{\prime}$ and if we denote this value with $p^{\prime\prime},$ then
$\widehat{P}_{\mu}^{\widehat{x}}\cap\widehat{P}_{\mu^{\prime}}^{\widehat{x}%
}=\{p^{\prime\prime},2p^{\prime\prime},3p^{\prime\prime},...\};$

b) $P_{\mu}^{x}=\{T,2T,3T,...\},P_{\mu\prime}^{x}=\{T^{\prime},2T^{\prime
},3T^{\prime},...\}\Longrightarrow\exists n_{1}\geq1,\exists n_{2}\geq1$
relatively prime with $n_{1}T=n_{2}T^{\prime}$ and for $T^{\prime\prime}$
equal with the previous value we get $P_{\mu}^{x}\cap P_{\mu\prime}%
^{x}=\{T^{\prime\prime},2T^{\prime\prime},3T^{\prime\prime},...\}.$

The limit case consists in signals that have all their points periodic, the
periodic signals.
\end{remark}

\chapter{\label{Cha5}Periodic signals}

We give in Section 1 and Section 2 properties that are equivalent with the
periodicity of the signals, structured in two groups.

The purpose of Section 3 is that of showing that all the values of the orbit
of a periodic signal are accessible in an interval with the length of a period.

Section 4 proves the independence of periodicity on the choice of $t^{\prime
}=$initial time of $x$ and limit of periodicity of $x$ and gives the bounds of
$t^{\prime}$.

The property of constancy from Section 5 is interesting by itself and it is
also a useful result in the exposure. The discussion from Section 6 shows the
relation between stating the constancy of a signal and the corresponding
statement that refers to the periodicity of its points.

When the relation between $\widehat{x}$ and $x\,$\ is
\begin{align*}
x(t)  &  =\widehat{x}(-1)\cdot\chi_{(-\infty,t_{0})}(t)\oplus\widehat
{x}(0)\cdot\chi_{\lbrack t_{0},t_{0}+h)}(t)\oplus...\\
&  ...\oplus\widehat{x}(k)\cdot\chi_{\lbrack t_{0}+kh,t_{0}+(k+1)h)}%
(t)\oplus...
\end{align*}
we are interested to see how the periodicity of $\widehat{x}$ determines the
periodicity of $x$ and vice versa. This is made in Section 7.

The fact that the sums, the differences and the multiples of the periods are
periods is proved in Section 8.

Section 9 characterizes the form of $\widehat{P}^{\widehat{x}},P^{x},$ in
particular the existence of the prime period is proved.

Sections 10, 11 give necessary and sufficient conditions of periodicity,
stated in terms of support sets. These conditions are inspired by those of the
periodic points and use the fact that if all the values of a signal are
periodic with the same period, then the signal is periodic.

A special case of periodicity is presented in Section 12. In this case the
exact value of the prime period is known.

By forgetting some first values of the periodic signals $\widehat{x},x$ we get
signals with the same period. This is the topic of Section 13.

In Section 14 we put the problem of changing the order of some quantifiers in
stating the periodicity of the signals.

\section{The first group of periodicity properties}

\begin{remark}
These properties involve the periodicity and the eventual periodicity of all
the points $\mu\in\widehat{Or}(\widehat{x}),\mu\in Or(x).$ The properties
(\ref{per152}),...,(\ref{pre543}) are associated with the periodicity
properties (\ref{per144})$_{page\;\pageref{per144}}$,...,(\ref{pre536}%
)$_{page\;\pageref{pre536}}$ and the properties (\ref{pre563}%
),...,(\ref{pre568}) are associated with (\ref{per145}%
)$_{page\;\pageref{per145}}$,...,(\ref{pre562})$_{page\;\pageref{pre562}}$
from Theorem \ref{The49}, page \pageref{The49}. One should compare also these
properties with (\ref{per185})$_{page\;\pageref{per185}}$,...,(\ref{pre537}%
)$_{page\;\pageref{pre537}}$ and (\ref{per186})$_{page\;\pageref{per186}}%
$,...,(\ref{pre552})$_{page\;\pageref{pre552}}$ from Theorem \ref{The95}, page
\pageref{The95}.
\end{remark}

\begin{theorem}
\label{The23}The signals $\widehat{x}\in\widehat{S}^{(n)},x\in S^{(n)}$ are given.

a) The following statements are equivalent for any $p\geq1$:%
\begin{equation}
\forall\mu\in\widehat{Or}(\widehat{x}),\forall k\in\widehat{\mathbf{T}}_{\mu
}^{\widehat{x}},\{k+zp|z\in\mathbf{Z}\}\cap\mathbf{N}_{\_}\subset
\widehat{\mathbf{T}}_{\mu}^{\widehat{x}}, \label{per152}%
\end{equation}%
\begin{equation}
\left\{
\begin{array}
[c]{c}%
\forall\mu\in\widehat{Or}(\widehat{x}),\forall k^{\prime}\in\mathbf{N}%
_{\_},\forall k\in\widehat{\mathbf{T}}_{\mu}^{\widehat{x}}\cap\{k^{\prime
},k^{\prime}+1,k^{\prime}+2,...\},\\
\{k+zp|z\in\mathbf{Z}\}\cap\{k^{\prime},k^{\prime}+1,k^{\prime}+2,...\}\subset
\widehat{\mathbf{T}}_{\mu}^{\widehat{x}},
\end{array}
\right.  \label{pre539}%
\end{equation}%
\begin{equation}
\forall\mu\in\widehat{Or}(\widehat{x}),\forall k^{\prime\prime}\in
\mathbf{N},\forall k\in\widehat{\mathbf{T}}_{\mu}^{\widehat{\sigma}%
^{k^{\prime\prime}}(\widehat{x})},\{k+zp|z\in\mathbf{Z}\}\cap\mathbf{N}%
_{\_}\subset\widehat{\mathbf{T}}_{\mu}^{\widehat{\sigma}^{k^{\prime\prime}%
}(\widehat{x})}, \label{pre540}%
\end{equation}%
\begin{equation}
\left\{
\begin{array}
[c]{c}%
\forall\mu\in\widehat{Or}(\widehat{x}),\forall k\in\mathbf{N}_{\_},\widehat
{x}(k)=\mu\Longrightarrow\\
\Longrightarrow(\widehat{x}(k)=\widehat{x}(k+p)\text{ and }k-p\geq
-1\Longrightarrow\widehat{x}(k)=\widehat{x}(k-p)),
\end{array}
\right.  \label{pre541}%
\end{equation}%
\begin{equation}
\left\{
\begin{array}
[c]{c}%
\forall\mu\in\widehat{Or}(\widehat{x}),\forall k^{\prime}\in\mathbf{N}%
_{\_},\forall k\geq k^{\prime},\widehat{x}(k)=\mu\Longrightarrow\\
\Longrightarrow(\widehat{x}(k)=\widehat{x}(k+p)\text{ and }k-p\geq k^{\prime
}\Longrightarrow\widehat{x}(k)=\widehat{x}(k-p)),
\end{array}
\right.  \label{pre542}%
\end{equation}%
\begin{equation}
\left\{
\begin{array}
[c]{c}%
\forall\mu\in\widehat{Or}(\widehat{x}),\forall k^{\prime\prime}\in
\mathbf{N},\forall k\in\mathbf{N}_{\_},\widehat{\sigma}^{k^{\prime\prime}%
}(\widehat{x})(k)=\mu\Longrightarrow\\
\Longrightarrow(\widehat{\sigma}^{k^{\prime\prime}}(\widehat{x})(k)=\widehat
{\sigma}^{k^{\prime\prime}}(\widehat{x})(k+p)\text{ and }\\
\text{and }k-p\geq-1\Longrightarrow\widehat{\sigma}^{k^{\prime\prime}%
}(\widehat{x})(k)=\widehat{\sigma}^{k^{\prime\prime}}(\widehat{x}%
)(k-p))\text{.}%
\end{array}
\right.  \label{pre543}%
\end{equation}

b) The following statements are also equivalent for any $T>0$:%
\begin{equation}
\forall\mu\in Or(x),\exists t^{\prime}\in I^{x},\forall t\in\mathbf{T}_{\mu
}^{x}\cap\lbrack t^{\prime},\infty),\{t+zT|z\in\mathbf{Z}\}\cap\lbrack
t^{\prime},\infty)\subset\mathbf{T}_{\mu}^{x}, \label{pre563}%
\end{equation}%
\begin{equation}
\left\{
\begin{array}
[c]{c}%
\forall\mu\in Or(x),\exists t^{\prime}\in I^{x},\\
\forall t_{1}^{\prime}\geq t^{\prime},\forall t\in\mathbf{T}_{\mu}^{x}%
\cap\lbrack t_{1}^{\prime},\infty),\{t+zT|z\in\mathbf{Z}\}\cap\lbrack
t_{1}^{\prime},\infty)\subset\mathbf{T}_{\mu}^{x},
\end{array}
\right.  \label{pre564}%
\end{equation}%
\begin{equation}
\left\{
\begin{array}
[c]{c}%
\forall\mu\in Or(x),\forall t^{\prime\prime}\in\mathbf{R},\exists t^{\prime
}\in I^{\sigma^{t^{\prime\prime}}(x)},\\
\forall t\in\mathbf{T}_{\mu}^{\sigma^{t^{\prime\prime}}(x)}\cap\lbrack
t^{\prime},\infty),\{t+zT|z\in\mathbf{Z}\}\cap\lbrack t^{\prime}%
,\infty)\subset\mathbf{T}_{\mu}^{\sigma^{t^{\prime\prime}}(x)},
\end{array}
\right.  \label{pre565}%
\end{equation}%
\begin{equation}
\left\{
\begin{array}
[c]{c}%
\forall\mu\in Or(x),\exists t^{\prime}\in I^{x},\forall t\geq t^{\prime},\\
x(t)=\mu\Longrightarrow(x(t)=x(t+T)\text{ and }t-T\geq t^{\prime
}\Longrightarrow x(t)=x(t-T)),
\end{array}
\right.  \label{pre566}%
\end{equation}%
\begin{equation}
\left\{
\begin{array}
[c]{c}%
\forall\mu\in Or(x),\exists t^{\prime}\in I^{x},\forall t_{1}^{\prime}\geq
t^{\prime},\forall t\geq t_{1}^{\prime},x(t)=\mu\Longrightarrow\\
\Longrightarrow(x(t)=x(t+T)\text{ and }t-T\geq t_{1}^{\prime}\Longrightarrow
x(t)=x(t-T)),
\end{array}
\right.  \label{pre567}%
\end{equation}%
\begin{equation}
\left\{
\begin{array}
[c]{c}%
\forall\mu\in Or(x),\forall t^{\prime\prime}\in\mathbf{R},\exists t^{\prime
}\in I^{\sigma^{t^{\prime\prime}}(x)},\\
\forall t\geq t^{\prime},\sigma^{t^{\prime\prime}}(x)(t)=\mu\Longrightarrow
(\sigma^{t^{\prime\prime}}(x)(t)=\sigma^{t^{\prime\prime}}(x)(t+T)\text{
and}\\
\text{and }t-T\geq t^{\prime}\Longrightarrow\sigma^{t^{\prime\prime}%
}(x)(t)=\sigma^{t^{\prime\prime}}(x)(t-T))\text{.}%
\end{array}
\right.  \label{pre568}%
\end{equation}

\end{theorem}

\begin{proof}
a) The proof of the implications%
\[
(\ref{per152})\Longrightarrow(\ref{pre539})\Longrightarrow(\ref{pre540}%
)\Longrightarrow(\ref{pre541})\Longrightarrow(\ref{pre542})\Longrightarrow
(\ref{pre543})
\]
follows from Theorem \ref{The95}.

(\ref{pre543})$\Longrightarrow$(\ref{per152}) Let $\mu\in\widehat{Or}%
(\widehat{x}),k\in\widehat{\mathbf{T}}_{\mu}^{\widehat{x}}$ and $z\in
\mathbf{Z}$ arbitrary such that $k+zp\geq-1.$ (\ref{pre543}) written for
$k^{\prime\prime}=0$ gives%
\begin{equation}
\left\{
\begin{array}
[c]{c}%
\forall k_{1}\in\mathbf{N}_{\_},\widehat{x}(k_{1})=\mu\Longrightarrow\\
\Longrightarrow(\widehat{x}(k_{1})=\widehat{x}(k_{1}+p)\text{ and }k_{1}%
-p\geq-1\Longrightarrow\widehat{x}(k_{1})=\widehat{x}(k_{1}-p)).
\end{array}
\right.  \label{pre989}%
\end{equation}

Case $z<0,$%
\[
\mu=\widehat{x}(k)\overset{(\ref{pre989})}{=}\widehat{x}(k-p)\overset
{(\ref{pre989})}{=}\widehat{x}(k-2p)\overset{(\ref{pre989})}{=}...\overset
{(\ref{pre989})}{=}\widehat{x}(k+zp);
\]

Case $z=0,$%
\[
\mu=\widehat{x}(k+zp);
\]

Case $z>0,$%
\[
\mu=\widehat{x}(k)\overset{(\ref{pre989})}{=}\widehat{x}(k+p)\overset
{(\ref{pre989})}{=}\widehat{x}(k+2p)\overset{(\ref{pre989})}{=}...\overset
{(\ref{pre989})}{=}\widehat{x}(k+zp).
\]
We infer in all the three cases that $k+zp\in\widehat{\mathbf{T}}_{\mu
}^{\widehat{x}}.$

b) The proof of the following implications%
\[
(\ref{pre563})\Longrightarrow(\ref{pre564})\Longrightarrow(\ref{pre565}%
)\Longrightarrow(\ref{pre566})\Longrightarrow(\ref{pre567})\Longrightarrow
(\ref{pre568})
\]
follows from Theorem \ref{The95}.

(\ref{pre568})$\Longrightarrow$(\ref{pre563}) Let $\mu\in Or(x).$
(\ref{pre568}) written for $t^{\prime\prime}$ sufficiently small in order that
$\sigma^{t^{\prime\prime}}(x)=x$ gives the existence of $t^{\prime}\in I^{x}$
with%
\begin{equation}
\left\{
\begin{array}
[c]{c}%
\forall t_{1}\geq t^{\prime},x(t_{1})=\mu\Longrightarrow\\
\Longrightarrow(x(t_{1})=x(t_{1}+T)\text{ and }t_{1}-T\geq t^{\prime
}\Longrightarrow x(t_{1})=x(t_{1}-T)).
\end{array}
\right.  \label{pre990}%
\end{equation}
From Lemma \ref{Lem36}, page \pageref{Lem36} we have $\mathbf{T}_{\mu}^{x}%
\cap\lbrack t^{\prime},\infty)\neq\varnothing.$ We take $t\in\mathbf{T}_{\mu
}^{x}\cap\lbrack t^{\prime},\infty)$ and $z\in\mathbf{Z}$ arbitrary with
$t+zT\geq t^{\prime}$ and we have the following possibilities.

Case $z<0,$%
\[
\mu=x(t)\overset{(\ref{pre990})}{=}x(t-T)\overset{(\ref{pre990})}%
{=}x(t-2T)\overset{(\ref{pre990})}{=}...\overset{(\ref{pre990})}{=}x(t+zT);
\]

Case $z=0,$%
\[
\mu=x(t+zT);
\]

Case $z>0,$%
\[
\mu=x(t)\overset{(\ref{pre990})}{=}x(t+T)\overset{(\ref{pre990})}%
{=}x(t+2T)\overset{(\ref{pre990})}{=}...\overset{(\ref{pre990})}{=}x(t+zT).
\]
We have obtained in all these cases that $t+zT\in\mathbf{T}_{\mu}^{x}.$ We
infer the truth of (\ref{pre563}).
\end{proof}

\section{The second group of periodicity properties}

\begin{remark}
The properties (\ref{per50}),...,(\ref{pre545}) and (\ref{pre569}%
),...,(\ref{pre571}) from this group have occurred for the first time as
(\ref{per75})$_{page\;\pageref{per75}}$,...,(\ref{pre509}%
)$_{page\;\pageref{pre509}}$ and (\ref{per76})$_{page\;\pageref{per76}}$,...,
(\ref{pre513})$_{page\;\pageref{pre513}}$ in Theorem \ref{The96}, page
\pageref{The96}. These properties refer to the signals themselves, and not to
their values.
\end{remark}

\begin{theorem}
\label{The101}The signals $\widehat{x}\in\widehat{S}^{(n)},x\in S^{(n)}$ are given.

a) The following properties are equivalent, for any $p\geq1$, with any of
(\ref{per152}),..., (\ref{pre543})$:$%
\begin{equation}
\forall k\in\mathbf{N}_{\_},\widehat{x}(k)=\widehat{x}(k+p), \label{per50}%
\end{equation}%
\begin{equation}
\forall k^{\prime}\in\mathbf{N}_{\_},\forall k\geq k^{\prime},\widehat
{x}(k)=\widehat{x}(k+p), \label{pre544}%
\end{equation}%
\begin{equation}
\forall k^{\prime\prime}\in\mathbf{N},\forall k\in\mathbf{N}_{\_}%
,\widehat{\sigma}^{k^{\prime\prime}}(\widehat{x})(k)=\widehat{\sigma
}^{k^{\prime\prime}}(\widehat{x})(k+p). \label{pre545}%
\end{equation}

b) For any $T>0,$ the following properties are equivalent with any of
(\ref{pre563}),..., (\ref{pre568}):%
\begin{equation}
\exists t^{\prime}\in I^{x}\text{,}\forall t\geq t^{\prime},x(t)=x(t+T),
\label{pre569}%
\end{equation}%
\begin{equation}
\exists t^{\prime}\in I^{x},\forall t_{1}^{\prime}\geq t^{\prime},\forall
t\geq t_{1}^{\prime},x(t)=x(t+T), \label{pre570}%
\end{equation}%
\begin{equation}
\forall t^{\prime\prime}\in\mathbf{R},\exists t^{\prime}\in I^{\sigma
^{t^{\prime\prime}}(x)},\forall t\geq t^{\prime},\sigma^{t^{\prime\prime}%
}(x)(t)=\sigma^{t^{\prime\prime}}(x)(t+T). \label{pre571}%
\end{equation}

\end{theorem}

\begin{proof}
a) The proof of (\ref{per50})$\Longrightarrow$(\ref{pre544})$\Longrightarrow
$(\ref{pre545}) follows from Theorem \ref{The96}.

(\ref{per152})$\Longrightarrow$(\ref{per50}) Let $k\in\mathbf{N}_{\_}$
arbitrary, fixed and we choose $\mu\in\widehat{Or}(\widehat{x})$ with the
property $\widehat{x}(k)=\mu.$ We infer%
\[
k+p\in\{k+zp|z\in\mathbf{Z}\}\cap\mathbf{N}_{\_}\overset{(\ref{per152}%
)}{\subset}\widehat{\mathbf{T}}_{\mu}^{\widehat{x}},
\]
thus $\widehat{x}(k+p)=\mu=\widehat{x}(k).$

(\ref{pre545})$\Longrightarrow$(\ref{per152}) Let $\mu\in\widehat{Or}%
(\widehat{x}),k\in\widehat{\mathbf{T}}_{\mu}^{\widehat{x}}$ and $z\in
\mathbf{Z}$ arbitrary with $k+zp\geq-1.$ We apply (\ref{pre545}) written for
$k^{\prime\prime}=0,$%
\begin{equation}
\forall k_{1}\in\mathbf{N}_{\_},\widehat{x}(k_{1})=\widehat{x}(k_{1}+p)
\label{pre994}%
\end{equation}
and we have the following cases:

Case $z>0,$%
\[
\mu=\widehat{x}(k)\overset{(\ref{pre994})}{=}\widehat{x}(k+p)\overset
{(\ref{pre994})}{=}\widehat{x}(k+2p)\overset{(\ref{pre994})}{=}%
\]%
\[
\overset{(\ref{pre994})}{=}...\overset{(\ref{pre994})}{=}\widehat{x}(k+zp);
\]

Case $z=0,$%
\[
\mu=\widehat{x}(k)=\widehat{x}(k+zp);
\]

Case $z<0,$%
\[
\widehat{x}(k+zp)\overset{(\ref{pre994})}{=}\widehat{x}(k+(z+1)p)\overset
{(\ref{pre994})}{=}%
\]%
\[
\overset{(\ref{pre994})}{=}\widehat{x}(k+(z+2)p)\overset{(\ref{pre994})}%
{=}...\overset{(\ref{pre994})}{=}\widehat{x}(k)=\mu.
\]
In all these cases we have obtained that $k+zp\in\widehat{\mathbf{T}}_{\mu
}^{\widehat{x}}.$

b) The proof of the implications (\ref{pre569})$\Longrightarrow$%
(\ref{pre570})$\Longrightarrow$(\ref{pre571}) follows from Theorem \ref{The96}.

(\ref{pre563})$\Longrightarrow$(\ref{pre569}) We suppose that $Or(x)=\{\mu
^{1},...,\mu^{s}\}$ and from (\ref{pre563}) we have the existence $\forall
i\in\{1,...,s\}$ of $t_{i}^{\prime}\in I^{x}$ with%
\begin{equation}
\forall t_{1}\in\mathbf{T}_{\mu^{i}}^{x}\cap\lbrack t_{i}^{\prime}%
,\infty),\{t_{1}+zT|z\in\mathbf{Z}\}\cap\lbrack t_{i}^{\prime},\infty
)\subset\mathbf{T}_{\mu^{i}}^{x} \label{pre876}%
\end{equation}
fulfilled. With the notation $t^{\prime}=\max\{t_{1}^{\prime},...,t_{s}%
^{\prime}\},$ we get the truth of $t^{\prime}\in I^{x},$%
\begin{equation}
\forall t\in\mathbf{T}_{\mu^{i}}^{x}\cap\lbrack t^{\prime},\infty
),\{t+zT|z\in\mathbf{Z}\}\cap\lbrack t^{\prime},\infty)\subset\mathbf{T}%
_{\mu^{i}}^{x} \label{pre878}%
\end{equation}
for any $i\in\{1,...,s\},$ see Lemma \ref{Lem30}, page \pageref{Lem30}$.$ Let
now $t\geq t^{\prime}$ arbitrary. Some $i\in\{1,...,s\}$ exists such that
$x(t)=\mu^{i},$ for which we can write%
\[
t+T\in\{t+zT|z\in\mathbf{Z}\}\cap\lbrack t^{\prime},\infty)\overset
{(\ref{pre878})}{\subset}\mathbf{T}_{\mu^{i}}^{x},
\]
in other words $x(t+T)=\mu^{i}=x(t).$

(\ref{pre571})$\Longrightarrow$(\ref{pre563}) We take in (\ref{pre571})
$t^{\prime\prime}\in\mathbf{R}$ sufficiently small so that $\sigma
^{t^{\prime\prime}}(x)=x$ and the existence of $t^{\prime}\in I^{x}$ results
with%
\begin{equation}
\forall t_{1}\geq t^{\prime},x(t_{1})=x(t_{1}+T). \label{pre883}%
\end{equation}
Let $\mu\in Or(x)$ arbitrary. We have from Lemma \ref{Lem36}, page
\pageref{Lem36} that $\mathbf{T}_{\mu}^{x}\cap\lbrack t^{\prime},\infty
)\neq\varnothing$ and we take $t\in\mathbf{T}_{\mu}^{x}\cap\lbrack t^{\prime
},\infty),z\in\mathbf{Z}$ arbitrary such that $t+zT\geq t^{\prime}.$ The
following possibilities exist:

Case $z>0,$%
\[
\mu=x(t)\overset{(\ref{pre883})}{=}x(t+T)\overset{(\ref{pre883})}%
{=}x(t+2T)\overset{(\ref{pre883})}{=}...\overset{(\ref{pre883})}{=}x(t+zT);
\]

Case $z=0,$%
\[
\mu=x(t)=x(t+zT);
\]

Case $z<0,$%
\[
x(t+zT)\overset{(\ref{pre883})}{=}x(t+(z+1)T)\overset{(\ref{pre883})}%
{=}x(t+(z+2)T)\overset{(\ref{pre883})}{=}...\overset{(\ref{pre883})}%
{=}x(t)=\mu.
\]
In all these cases, the satisfaction of $t+zT\in\mathbf{T}_{\mu}^{x}$ is proved.
\end{proof}

\begin{remark}
All the points of the orbit of a periodic signal are periodic and they have a
common period $p,T$ and vice versa, if all the points of the orbit of a signal
are periodic and have a common period $p,T$, then the signal is periodic:
\[
(\ref{per152})_{page\;\pageref{per152}}\Longleftrightarrow\forall\mu
\in\widehat{Or}(\widehat{x}),(\ref{per144})_{page\;\pageref{per144}},
\]%
\[
(\ref{pre563})_{page\;\pageref{pre563}}\Longleftrightarrow\forall\mu\in
Or(x),(\ref{per145})_{page\;\pageref{per145}}%
\]
hold.
\end{remark}

\begin{remark}
(\ref{per152}),...,(\ref{pre543}) and (\ref{pre563}),...,(\ref{pre568}) refer
to left-and-right time shifts, (\ref{per50}),...,(\ref{pre545}) and
(\ref{pre569}),...,(\ref{pre571}) refer to right time shifts only.
\end{remark}

\begin{theorem}
\label{The130}Let the periodic signals $\widehat{x}\in\widehat{S}^{(n)},x\in
S^{(n)}.$ We have $\widehat{\omega}(\widehat{x})=\widehat{Or}(\widehat
{x}),\omega(x)=Or(x).$
\end{theorem}

\begin{proof}
In order to prove the real time statement, we suppose that $T>0,t^{\prime}\in
I^{x}$ exist such that%
\[
\forall\mu\in Or(x),\forall t\in\mathbf{T}_{\mu}^{x}\cap\lbrack t^{\prime
},\infty),\{t+zT|z\in\mathbf{Z}\}\cap\lbrack t^{\prime},\infty)\subset
\mathbf{T}_{\mu}^{x}%
\]
is true and let $\mu\in Or(x)$ arbitrary. The fact that $\mathbf{T}_{\mu}^{x}$
is superiorly unbounded shows that $\mu\in\omega(x),$ wherefrom the conclusion
that $Or(x)\subset\omega(x).$ As far as the inclusion $\omega(x)\subset Or(x)$
is always true, we infer that $\omega(x)=Or(x).$
\end{proof}

\section{The accessibility of the orbit}

\begin{theorem}
a) If $\widehat{x}\in\widehat{S}^{(n)},$ then%
\begin{equation}
\widehat{P}^{\widehat{x}}\neq\varnothing\Longrightarrow\forall k^{\prime}%
\in\mathbf{N}_{\_},\widehat{Or}(\widehat{x})=\{\widehat{x}(k)|k\geq k^{\prime
}\}.
\end{equation}

b) For $x\in S^{(n)}$ we have%
\begin{equation}
P^{x}\neq\varnothing\Longrightarrow\forall t^{\prime}\in\mathbf{R}%
,Or(x)=\{x(t)|t\geq t^{\prime}\}. \label{p263}%
\end{equation}

\end{theorem}

\begin{proof}
In the case of the periodicity of $\widehat{x},x$ we have $\widehat{\omega
}(\widehat{x})=\widehat{Or}(\widehat{x}),\omega(x)=Or(x).$ These statements
follow from Theorem \ref{The140}, page \pageref{The140} where $\widehat
{L}^{\widehat{x}}=\underset{\mu\in\widehat{\omega}(\widehat{x})}{%
{\displaystyle\bigcap}
}\widehat{L}_{\mu}^{\widehat{x}}=\mathbf{N}_{\_}$ at a) and at b) notice that
$P^{x}\neq\varnothing,$ $L^{x}=\underset{\mu\in\omega(x)}{%
{\displaystyle\bigcap}
}L_{\mu}^{x}$ and
\begin{equation}
\forall t^{\prime}\in L^{x},Or(x)=\{x(t)|t\geq t^{\prime}\} \label{p305}%
\end{equation}
imply
\[
\forall t^{\prime}\in\mathbf{R},Or(x)=\{x(t)|t\geq t^{\prime}\}.
\]
Indeed, let $t^{\prime}\in\mathbf{R}$ arbitrary. If $t^{\prime}\in L^{x}$ then
(\ref{p263}) is true from (\ref{p305}) and if $t^{\prime}\in\mathbf{R}%
\setminus L^{x}$ then for any $t^{\prime\prime}\in L^{x}$ we have $t^{\prime
}<t^{\prime\prime}$ and we can write that
\[
Or(x)=\{x(t)|t\geq t^{\prime\prime}\}\subset\{x(t)|t\geq t^{\prime}\}\subset
Or(x).
\]

\end{proof}

\begin{theorem}
\label{The126}a) We suppose that $\widehat{x}$ is periodic, with the period
$p\geq1:$%
\begin{equation}
\forall k\in\mathbf{N}_{\_},\widehat{x}(k)=\widehat{x}(k+p).
\end{equation}
Then%
\begin{equation}
\forall k\in\mathbf{N}_{\_},\widehat{Or}(\widehat{x})=\{\widehat{x}%
(i)|i\in\{k,k+1,...,k+p-1\}\}.
\end{equation}

b) If $x$ is periodic with the period $T>0:$ $t^{\prime}\in I^{x}$ exists with%
\begin{equation}
\forall t\geq t^{\prime},x(t)=x(t+T),
\end{equation}
then%
\begin{equation}
\forall t\geq t^{\prime},Or(x)=\{x(\xi)|\xi\in\lbrack t,t+T)\}.
\end{equation}

\end{theorem}

\begin{proof}
We apply Theorem \ref{The125}, page \pageref{The125} with $\widehat{\omega
}(\widehat{x})=\widehat{Or}(\widehat{x}),$ $k^{\prime}=-1$ and $\omega
(x)=Or(x),$ $t^{\prime}\in I^{x}.$
\end{proof}

\begin{remark}
The previous Theorem states the property that, in the case of the periodic
signals, all the points of the orbit are accessible in a time interval with
the length of a period.
\end{remark}

\section{The limit of periodicity}

\begin{theorem}
\label{The144}If $\widehat{x}$ is periodic, then%
\begin{equation}
\forall\mu\in\widehat{Or}(\widehat{x}),\widehat{L}^{\widehat{x}}=\widehat
{L}_{\mu}^{\widehat{x}}=\mathbf{N}_{\_}. \label{p264}%
\end{equation}

\end{theorem}

\begin{proof}
The fact that the periodicity of $\widehat{x}$ implies $\widehat{L}%
^{\widehat{x}}=\mathbf{N}_{\_}$ is obvious and the fact that $\forall\mu
\in\widehat{Or}(\widehat{x}),\widehat{L}^{\widehat{x}}\subset\widehat{L}_{\mu
}^{\widehat{x}}$ results from Theorem \ref{The139}, page \pageref{The139},
where $\widehat{Or}(\widehat{x})=\widehat{\omega}(\widehat{x}).$ For any
$\mu\in\widehat{Or}(\widehat{x}),$ we infer%
\[
\mathbf{N}_{\_}=\widehat{L}^{\widehat{x}}\subset\widehat{L}_{\mu}^{\widehat
{x}}\subset\mathbf{N}_{\_}.
\]

\end{proof}

\begin{example}
Let $\mu,\mu^{\prime},\mu^{\prime\prime}\in\mathbf{B}^{n}$ distinct and $x\in
S^{(n)}$ defined this way:%
\[
x(t)=\mu\cdot\chi_{(-\infty,0)}(t)\oplus\mu^{\prime\prime}\cdot\chi
_{\lbrack0,1)}(t)\oplus\mu^{\prime}\cdot\chi_{\lbrack1,2)}(t)\oplus\mu
^{\prime\prime}\cdot\chi_{\lbrack2,4)}(t)\oplus\mu^{\prime}\cdot\chi
_{\lbrack4,5)}(t)
\]%
\[
\oplus\mu\cdot\chi_{\lbrack5,6)}(t)\oplus\mu^{\prime\prime}\cdot\chi
_{\lbrack6,7)}(t)\oplus\mu^{\prime}\cdot\chi_{\lbrack7,8)}(t)\oplus\mu
^{\prime\prime}\cdot\chi_{\lbrack8,10)}(t)\oplus\mu^{\prime}\cdot\chi
_{\lbrack10,11)}(t)
\]%
\[
\oplus\mu\cdot\chi_{\lbrack11,12)}(t)\oplus\mu^{\prime\prime}\cdot
\chi_{\lbrack12,13)}(t)\oplus\mu^{\prime}\cdot\chi_{\lbrack13,14)}(t)\oplus
\mu^{\prime\prime}\cdot\chi_{\lbrack14,16)}(t)\oplus\mu^{\prime}\cdot
\chi_{\lbrack16,17)}(t)\oplus...
\]
In this example $Or(x)=\{\mu,\mu^{\prime},\mu^{\prime\prime}\};$ $L_{\mu}%
^{x}=[-1,\infty),L_{\mu^{\prime}}^{x}=[-1,\infty),L_{\mu^{\prime\prime}}%
^{x}=[-2,\infty),$ $L^{x}=[-1,\infty);$ $P_{\mu}^{x}=\{6,12,18,...\},P_{\mu
^{\prime}}^{x}=\{3,6,9,...\},P_{\mu^{\prime\prime}}^{x}=\{6,12,18,...\},P^{x}%
=\{6,12,18,...\}.$ We notice the falsity of (\ref{p264}) in the real time
case, expressed under the form $L^{x}\neq L_{\mu^{\prime\prime}}^{x}.$
\end{example}

\begin{theorem}
\label{The53}The non constant signal $x$ is given, together with $T>0$ and we
suppose that $t^{\prime}\in I^{x}$ exists with the property that%
\begin{equation}
\forall t\geq t^{\prime},x(t)=x(t+T). \label{per989}%
\end{equation}
Then $t_{0}^{\prime},t_{0}\in\mathbf{R}$ exist, $t_{0}^{\prime}<t_{0}$ such
that $\forall t^{\prime\prime}\in\lbrack t_{0}^{\prime},t_{0}),$ we have
$t^{\prime\prime}\in I^{x},$%
\begin{equation}
\forall t\geq t^{\prime\prime},x(t)=x(t+T) \label{per991}%
\end{equation}
and if $t^{\prime\prime}\notin\lbrack t_{0}^{\prime},t_{0}),$ then at least
one of $t^{\prime\prime}\in I^{x},$ (\ref{per991}) is false. In other words
$[t_{0}^{\prime},t_{0})=I^{x}\cap L^{x}.$
\end{theorem}

\begin{proof}
The first proof. As $x$ is not constant, $t_{0}\in\mathbf{R}$ exists with
$I^{x}=(-\infty,t_{0}).$ We suppose that $Or(x)=\{\mu^{1},...,\mu^{s}%
\},s\geq2$ and from the periodicity of $x$ we get $\omega(x)=\{\mu^{1}%
,...,\mu^{s}\}.$ From Theorem \ref{The117}, page \pageref{The117} we infer the
existence of $t_{1}^{\prime},...,t_{s}^{\prime}$ with $L_{\mu^{1}}^{x}%
=[t_{1}^{\prime},\infty),...,L_{\mu^{s}}^{x}=[t_{s}^{\prime},\infty).$ The
periodicity of $\mu^{1},...,\mu^{s}$ implies $I^{x}\cap L_{\mu^{1}}^{x}%
\neq\varnothing,...,I^{x}\cap L_{\mu^{s}}^{x}\neq\varnothing$ and the eventual
periodicity of $x$ shows, from Theorem \ref{The139}, page \pageref{The139}
that $L^{x}=L_{\mu^{1}}^{x}\cap...\cap L_{\mu^{s}}^{x}.$ It has resulted the
fact that $t_{0}^{\prime}=\max\{t_{1}^{\prime},...,t_{s}^{\prime}\}$ satisfies
$L^{x}=[t_{0}^{\prime},\infty),I^{x}\cap L^{x}=[t_{0}^{\prime},t_{0}%
)\neq\varnothing.$
\end{proof}

\begin{proof}
The second proof. We define $t_{0}$ in the following way:%
\begin{equation}
\forall t<t_{0},x(t)=x(-\infty+0), \label{per992}%
\end{equation}%
\begin{equation}
x(t_{0})\neq x(-\infty+0) \label{per993}%
\end{equation}
and this is possible since $x$ is not constant. From (\ref{per992}),
(\ref{per993}) we have $I^{x}=(-\infty,t_{0})$ and since $t^{\prime}\in
I^{x},$ we infer that $t^{\prime}<t_{0}.$ We have from (\ref{per989}):%
\begin{equation}
\forall t\in\lbrack t^{\prime},t_{0}),x(-\infty+0)=x(t)=x(t+T), \label{per994}%
\end{equation}
thus%
\begin{equation}
\forall t\in\lbrack t^{\prime}+T,t_{0}+T),x(t)=x(-\infty+0), \label{per995}%
\end{equation}%
\begin{equation}
x(t_{0}+T)=x(t_{0})\overset{(\ref{per993})}{\neq}x(-\infty+0). \label{per996}%
\end{equation}
We can see that $(-\infty,t_{0})\cup\lbrack t^{\prime}+T,t_{0}+T)\subset
\mathbf{T}_{x(-\infty+0)}^{x},$ $t_{0},t_{0}+T\notin\mathbf{T}_{x(-\infty
+0)}^{x},$ where $t_{0}<t^{\prime}+T$ is the only possibility, thus%
\[
t^{\prime}<t_{0}<t^{\prime}+T<t_{0}+T
\]
is true. Then $t_{0}^{\prime}\leq t^{\prime}$ exists such that $t_{0}^{\prime
}+T>t_{0}$ and%
\begin{equation}
\forall t\in\lbrack t_{0}^{\prime}+T,t_{0}+T),x(t)=x(-\infty+0),
\label{per997}%
\end{equation}%
\begin{equation}
x(t_{0}^{\prime}+T-0)\neq x(-\infty+0). \label{per998}%
\end{equation}
We take some arbitrary $t^{\prime\prime}\in\lbrack t_{0}^{\prime},t_{0}), $
some arbitrary $t\in\mathbf{R}$ and we have the following possibilities.

a) Case $t^{\prime\prime}\in\lbrack t^{\prime},t_{0})$

a.1) Case $t\leq t^{\prime\prime},$ when $x(t)=x(-\infty+0),$

a.2) Case $t\geq t^{\prime\prime},$ when $t\geq t^{\prime}$ and $x(t)\overset
{(\ref{per989})}{=}x(t+T);$

b) Case $t^{\prime\prime}\in\lbrack t_{0}^{\prime},t^{\prime})$

b.1) Case $t\leq t^{\prime\prime},$ when $x(t)=x(-\infty+0),$

b.2) Case $t\geq t^{\prime\prime}$

b.2.1) Case $t\in\lbrack t^{\prime\prime},t^{\prime}),$ when%
\[
t_{0}^{\prime}+T\leq t^{\prime\prime}+T\leq t+T<t^{\prime}+T<t_{0}+T
\]
and $x(t)=x(-\infty+0)\overset{(\ref{per997})}{=}x(t+T),$

b.2.2) Case $t\geq t^{\prime},$ when $x(t)\overset{(\ref{per989})}{=}x(t+T). $

In all these cases $t^{\prime\prime}\in I^{x}$ and (\ref{per991}) hold.

We suppose now, against all reason, that $t^{\prime\prime}\in I^{x}$,
(\ref{per991}) hold and $t^{\prime\prime}\notin\lbrack t_{0}^{\prime},t_{0}).$
The following possibilities exist.

i) Case $t^{\prime\prime}<t_{0}^{\prime}$

Some $\varepsilon_{1}>0$ exists such that%
\begin{equation}
\forall t\in(t_{0}^{\prime}+T-\varepsilon_{1},t_{0}^{\prime}+T),x(t)=x(t_{0}%
^{\prime}+T-0) \label{per999}%
\end{equation}
and let $\varepsilon\in(0,\min\{t_{0}^{\prime}-t^{\prime\prime},\varepsilon
_{1}\}),$ for which%
\begin{equation}
(t_{0}^{\prime}-\varepsilon,t_{0}^{\prime})\subset\lbrack t^{\prime\prime
},\infty), \label{pre1}%
\end{equation}%
\begin{equation}
(t_{0}^{\prime}+T-\varepsilon,t_{0}^{\prime}+T)\subset(t_{0}^{\prime
}+T-\varepsilon_{1},t_{0}^{\prime}+T). \label{pre2}%
\end{equation}
We take an arbitrary $t\in(t_{0}^{\prime}-\varepsilon,t_{0}^{\prime}).$ We
have the contradiction%
\[
x(-\infty+0)=x(t)\overset{(\ref{per991}),(\ref{pre1})}{=}x(t+T)\overset
{(\ref{per999}),(\ref{pre2})}{=}x(t_{0}^{\prime}+T-0)\overset{(\ref{per998}%
)}{\neq}x(-\infty+0).
\]

ii) Case $t^{\prime\prime}\geq t_{0}$%
\[
x(-\infty+0)=x(t^{\prime\prime})=x(t_{0})\overset{(\ref{per993})}{\neq
}x(-\infty+0),
\]
contradiction.
\end{proof}

\begin{corollary}
\label{Cor4}Let $x$ be not constant, $Or(x)=\{\mu^{1},\mu^{2},...,\mu
^{s}\},s\geq2,$ with $\mu^{1}=x(-\infty+0).$ We suppose that $x$ has the
period $T>0$ and we consider the statements%
\begin{equation}
\forall t<t_{0},x(t)=x(-\infty+0), \label{pre82}%
\end{equation}%
\begin{equation}
x(t_{0})\neq x(-\infty+0), \label{pre83}%
\end{equation}%
\begin{equation}
\forall t\in\lbrack t_{0}^{\prime}+T,t_{0}+T),x(t)=x(-\infty+0), \label{pre84}%
\end{equation}%
\begin{equation}
x(t_{0}^{\prime}+T-0)\neq x(-\infty+0), \label{pre85}%
\end{equation}%
\begin{equation}
\lbrack t_{1},t_{0}+T)\subset\mathbf{T}_{\mu^{1}}^{x}, \label{pre86}%
\end{equation}%
\begin{equation}
x(t_{1}-0)\neq\mu^{1}, \label{pre87}%
\end{equation}%
\begin{equation}
\lbrack t_{2},t_{0}+T)\cap\mathbf{T}_{\mu^{2}}^{x}=\varnothing, \label{pre88}%
\end{equation}%
\begin{equation}
x(t_{2}-0)=\mu^{2}, \label{pre89}%
\end{equation}%
\[
...
\]%
\begin{equation}
\lbrack t_{s},t_{0}+T)\cap\mathbf{T}_{\mu^{s}}^{x}=\varnothing, \label{pre90}%
\end{equation}%
\begin{equation}
x(t_{s}-0)=\mu^{s}; \label{pre91}%
\end{equation}
(\ref{pre82}), (\ref{pre83}) define $t_{0}$, (\ref{pre84}), (\ref{pre85})
define $t_{0}^{\prime}$, (\ref{pre86}), (\ref{pre87}) define $t_{1}$,
(\ref{pre88}), (\ref{pre89}) define $t_{2}$,..., (\ref{pre90}), (\ref{pre91})
define $t_{s}.$ The bounds $[t_{0}^{\prime},t_{0})$ of the initial time= limit
of periodicity of $x$ and the bounds $[t_{1}-T,t_{0}),[t_{2}-T,t_{0}%
),...,[t_{s}-T,t_{0})$ of the initial time of $x$=limits of periodicity of
$\mu^{1},\mu^{2},...,\mu^{s},$ considered as periodic points of $x$ with the
period $T$ fulfill%
\begin{equation}
\lbrack t_{0}^{\prime},t_{0})=[t_{1}-T,t_{0})=[t_{1}-T,t_{0})\cap\lbrack
t_{2}-T,t_{0})\cap...\cap\lbrack t_{s}-T,t_{0}). \label{pre92}%
\end{equation}

\end{corollary}

\begin{proof}
(\ref{pre82}), (\ref{pre83}) coincide with (\ref{per992}%
)$_{page\;\pageref{per992}}$, (\ref{per993})$_{page\;\pageref{per993}}$ and
(\ref{pre84}), (\ref{pre85}) coincide with (\ref{per997}%
)$_{page\;\pageref{per997}}$, (\ref{per998})$_{page\;\pageref{per998}}$ from
Theorem \ref{The53}, page \pageref{The53}. From the Theorem we have that
$[t_{0}^{\prime},t_{0})$ gives the bounds of the initial time=limit of
periodicity of $x$. The periodicity of $x $ with the period $T$ implies the
periodicity of its values $\mu^{1},\mu^{2},...,\mu^{s}$ with the period $T$
and the bounds of the initial time=limits of periodicity are given by
$[t_{1}-T,t_{0}),[t_{2}-T,t_{0}),...,[t_{s}-T,t_{0}),$ where (\ref{pre86}),
(\ref{pre87}) coincide with (\ref{pre93})$_{page\;\pageref{pre93}}$,
(\ref{pre94})$_{page\;\pageref{pre94}}$, while (\ref{pre88}), (\ref{pre89}%
),...,(\ref{pre90}), (\ref{pre91}) coincide with (\ref{pre96}%
)$_{page\;\pageref{pre96}}$, (\ref{pre95})$_{page\;\pageref{pre95}}$ from
Corollary \ref{Cor3}, page \pageref{Cor3}. We note also the coincidence of
(\ref{pre82}), (\ref{pre83}) with (\ref{pre80})$_{page\;\pageref{pre80}}$,
(\ref{pre81})$_{page\;\pageref{pre81}}.$ The truth of (\ref{pre92}) results
from the remark that $t_{1}=\max\{t_{1},t_{2},...,t_{s}\}$ and $t_{0}^{\prime
}+T=t_{1}$ (the last equality results from (\ref{pre84}), (\ref{pre85}) and
(\ref{pre86}), (\ref{pre87})).
\end{proof}

\begin{remark}
When we state the property of periodicity of a non constant signal $x,$ the
initial time=limit of periodicity $t^{\prime}$ belongs to some interval
$[t_{0}^{\prime},t_{0});$ outside this interval, any choice of $t^{\prime}$
makes the periodicity property of $x$ be false.
\end{remark}

\section{A property of constancy}

\begin{theorem}
\label{The31}Let the signals $\widehat{x},x.$

a) If the statement%
\begin{equation}
\forall k\in\mathbf{N}_{\_},\widehat{x}(k)=\widehat{x}(k+p) \label{per339}%
\end{equation}
is true for $p=1,$ then $\mu\in\widehat{Or}(\widehat{x})$ exists such that
\begin{equation}
\forall k\in\mathbf{N}_{\_},\widehat{x}(k)=\mu\label{per341}%
\end{equation}
and (\ref{per339}) is true for any $p\geq1$.

b) We suppose that $t_{0}\in\mathbf{R},h>0$ exist such that $x$ is of the
form
\begin{equation}%
\begin{array}
[c]{c}%
x(t)=x(-\infty+0)\cdot\chi_{(-\infty,t_{0})}(t)\oplus x(t_{0})\cdot
\chi_{\lbrack t_{0},t_{0}+h)}(t)\oplus...\\
...\oplus x(t_{0}+kh)\cdot\chi_{\lbrack t_{0}+kh,t_{0}+(k+1)h)}(t)\oplus...
\end{array}
\label{per485}%
\end{equation}
If the statement%
\begin{equation}
\forall t\geq t^{\prime},x(t)=x(t+T) \label{per340}%
\end{equation}
is true for some $t^{\prime}\in I^{x}\mathbf{,}$ $T\in(0,h)\cup(h,2h)\cup
...\cup(qh,(q+1)h)\cup...,$ then some $\mu\in Or(x)$ exists such that%
\begin{equation}
\forall t\in\mathbf{R},x(t)=\mu\label{per342}%
\end{equation}
and (\ref{per340}) is true for any $t^{\prime}\in\mathbf{R,}$ $T>0.$

c) We presume that (\ref{per485}) is true under the form%
\begin{equation}%
\begin{array}
[c]{c}%
x(t)=\widehat{x}(-1)\cdot\chi_{(-\infty,t_{0})}(t)\oplus\widehat{x}%
(0)\cdot\chi_{\lbrack t_{0},t_{0}+h)}(t)\oplus...\\
...\oplus\widehat{x}(k)\cdot\chi_{\lbrack t_{0}+kh,t_{0}+(k+1)h)}(t)\oplus...
\end{array}
\label{per338}%
\end{equation}
Then:

c.1) the fulfillment of (\ref{per339}) for $p=1$ implies that $\mu\in
\widehat{Or}(\widehat{x})=Or(x)$ exists such that (\ref{per341}),
(\ref{per342}) are true, (\ref{per339}) holds for any $p\geq1$ and
(\ref{per340}) holds for any $t^{\prime}\in\mathbf{R}$ and any $T>0;$

c.2) the satisfaction of the statement (\ref{per340}) for some $t^{\prime}\in
I^{x}\mathbf{,}T\in(0,h)\cup(h,2h)\cup...\cup(qh,(q+1)h)\cup...$ implies the
existence of $\mu\in\widehat{Or}(\widehat{x})=Or(x)$ such that (\ref{per341}),
(\ref{per342}) are true, (\ref{per339}) holds for any $p\geq1$ and
(\ref{per340}) holds for any $t^{\prime}\in\mathbf{R}$ and any $T>0. $
\end{theorem}

\begin{proof}
a) From (\ref{per339}) written for $p=1$ we get the existence of $\mu
=\widehat{x}(-1)$ such that (\ref{per341}) is true. Moreover, as far as
$\widehat{x}$ is the constant function, (\ref{per339}) holds for any $p\geq1.
$

b) We suppose against all reason that $x$ is not constant, thus $t_{0}%
^{\prime}\in\mathbf{R}$ exists such that $I^{x}=(-\infty,t_{0}^{\prime}).$ The
hypothesis states the existence of $t^{\prime}\in I^{x}$ with the property
that (\ref{per340}) is true for $T\in(0,h)\cup(h,2h)\cup...\cup(qh,(q+1)h)\cup
...$ In these conditions, Theorem \ref{The20}, page \pageref{The20} shows the
existence of $\mu\in\omega(x)$ such that%
\[
\forall t\geq t^{\prime},x(t)=\mu,
\]%
\[
\forall t\leq t^{\prime},x(t)=x(-\infty+0)
\]
are true. We have obtained that $\mu=x(-\infty+0),$ contradiction with our
supposition that $x$ is not constant. (\ref{per342}) holds. As in this
situation $I^{x}=L^{x}=\mathbf{R}$ and $P^{x}=(0,\infty)$ are true, b) is proved.

c) This is a consequence of a) and b).
\end{proof}

\section{Discussion on constancy}

\begin{remark}
Theorem \ref{The32}, page \pageref{The32} (concerning the periodic points) and
Theorem \ref{The31}, page \pageref{The31} (concerning the periodic signals)
express essentially the same idea, namely that in the situation when
$\widehat{x},x$ are related by%
\[%
\begin{array}
[c]{c}%
x(t)=\widehat{x}(-1)\cdot\chi_{(-\infty,t_{0})}(t)\oplus\widehat{x}%
(0)\cdot\chi_{\lbrack t_{0},t_{0}+h)}(t)\oplus...\\
...\oplus\widehat{x}(k)\cdot\chi_{\lbrack t_{0}+kh,t_{0}+(k+1)h)}(t)\oplus...
\end{array}
\]
any of a)
\[
\forall k\in\widehat{\mathbf{T}}_{\mu}^{\widehat{x}},\{k+zp|z\in
\mathbf{Z}\}\cap\mathbf{N}_{\_}\subset\widehat{\mathbf{T}}_{\mu}^{\widehat{x}}%
\]
or%
\[
\forall k\in\mathbf{N}_{\_},\widehat{x}(k)=\widehat{x}(k+p)
\]
true for $p=1,$

b)
\[
\exists t^{\prime}\in I^{x}\mathbf{,}\forall t\in\mathbf{T}_{\mu}^{x}%
\cap\lbrack t^{\prime},\infty),\{t+zT|z\in\mathbf{Z}\}\cap\lbrack t^{\prime
},\infty)\subset\mathbf{T}_{\mu}^{x}%
\]
or%
\[
\exists t^{\prime}\in I^{x}\mathbf{,}\forall t\geq t^{\prime},x(t)=x(t+T)
\]
true for $T\in(0,h)\cup(h,2h)\cup...\cup(qh,(q+1)h)\cup...$

implies the truth of
\[
\forall k\in\mathbf{N}_{\_},\widehat{x}(k)=\mu,
\]%
\[
\forall t\in\mathbf{R},x(t)=\mu
\]
thus $\widehat{x},x$ are equal with the same constant $\mu$. The validity
$\forall p\geq1,$ of%
\[
\forall\mu\in\widehat{Or}(\widehat{x}),\forall k\in\widehat{\mathbf{T}}_{\mu
}^{\widehat{x}},\{k+zp|z\in\mathbf{Z}\}\cap\mathbf{N}_{\_}\subset
\widehat{\mathbf{T}}_{\mu}^{\widehat{x}}%
\]%
\[
\overset{\text{Theorem \ref{The101}, page\ \pageref{The101}}}%
{\Longleftrightarrow}\forall k\in\mathbf{N}_{\_},\widehat{x}(k)=\widehat
{x}(k+p)
\]
and $\forall T>0,$ of%
\[
\exists t^{\prime}\in I^{x},\forall\mu\in Or(x),\forall t\in\mathbf{T}_{\mu
}^{x}\cap\lbrack t^{\prime},\infty),\{t+zT|z\in\mathbf{Z}\}\cap\lbrack
t^{\prime},\infty)\subset\mathbf{T}_{\mu}^{x}%
\]%
\[
\overset{\text{Theorem \ref{The101}, page \pageref{The101}}}%
{\Longleftrightarrow}\exists t^{\prime}\in I^{x},\forall t\geq t^{\prime
},x(t)=x(t+T)
\]
shows the fact that the common conclusion of Theorem \ref{The32}, page
\pageref{The32}, and Theorem \ref{The31}, page \pageref{The31}, is not
surprising.\footnote{Note that at (\ref{pre563}) we have the order of the
quantifiers $\forall\mu\in Or(x),\exists t^{\prime}\in I^{x}$ but in the proof
(\ref{pre563})$\Longrightarrow$(\ref{pre569}) from Theorem \ref{The101}, page
\pageref{The101} we could make use of $\exists t^{\prime}\in I^{x}%
\mathbf{,}\forall\mu\in Or(x)$, thus the previous argument is correct. We
shall refer again to the possibility of changing the order of some quantifiers
in stating periodicity properties in Section 14 of this Chapter.}
\end{remark}

\section{Discrete time vs real time}

\begin{theorem}
\label{The43}We presume that $\widehat{x},x$ satisfy%
\begin{equation}%
\begin{array}
[c]{c}%
x(t)=\widehat{x}(-1)\cdot\chi_{(-\infty,t_{0})}(t)\oplus\widehat{x}%
(0)\cdot\chi_{\lbrack t_{0},t_{0}+h)}(t)\oplus...\\
...\oplus\widehat{x}(k)\cdot\chi_{\lbrack t_{0}+kh,t_{0}+(k+1)h)}(t)\oplus...
\end{array}
\label{per350}%
\end{equation}
for some $t_{0}\in\mathbf{R}$ and $h>0.$ Then the existence of $p\geq1$ such
that%
\begin{equation}
\forall k\in\mathbf{N}_{\_},\widehat{x}(k)=\widehat{x}(k+p) \label{per596}%
\end{equation}
implies that, for $T=ph$ we have $\exists t^{\prime}\in I^{x},$%
\begin{equation}
\forall t\geq t^{\prime},x(t)=x(t+T). \label{per598}%
\end{equation}

\end{theorem}

\begin{proof}
The hypothesis states that $t_{0}\in\mathbf{R},h>0$ and $p\geq1$ exist such
that (\ref{per350}), (\ref{per596}) hold and we denote $t^{\prime}=t_{0}-h.$
We have%
\begin{equation}
x(-\infty+0)=\widehat{x}(-1), \label{per323}%
\end{equation}%
\begin{equation}
\forall t\leq t^{\prime},x(t)=x(-\infty+0). \label{per324}%
\end{equation}
We fix some arbitrary $t\geq t^{\prime}.$ Then $k\in\mathbf{N}_{\_}$ exists
with $t\in\lbrack t_{0}+kh,t_{0}+(k+1)h)$ wherefrom, for $T=ph$ we infer%
\[
t+T\in\lbrack t_{0}+kh+T,t_{0}+(k+1)h+T)=[t_{0}+(k+p)h,t_{0}+(k+p+1)h).
\]
We finally get%
\[
x(t)=\widehat{x}(k)\overset{(\ref{per596})}{=}\widehat{x}(k+p)=x(t+T)
\]
and, by taking into account (\ref{per324}) also, we infer that $t^{\prime}\in
I^{x}$ exists such that (\ref{per598}) is true.
\end{proof}

\begin{theorem}
\label{The18}If $\widehat{x},x$ are not constant and

i) $t_{0}\in\mathbf{R,}$ $h>0$ exist such that $\widehat{x},x$ fulfill
(\ref{per350}),

ii) $T>0,$ $t^{\prime}\in I^{x}$ exist such that $x$ fulfills (\ref{per598})

then $\frac{T}{h}\in\{1,2,3,...\}$ and $k^{\prime}\in\mathbf{N}_{\_}$ exists
making%
\begin{equation}
\forall k\geq k^{\prime},\widehat{x}(k)=\widehat{x}(k+p), \label{per434}%
\end{equation}%
\begin{equation}
\forall k\in\{-1,0,...,k^{\prime}\},\widehat{x}(k)=\widehat{x}(k^{\prime})
\label{per435}%
\end{equation}
true for $p=\frac{T}{h}.$
\end{theorem}

\begin{proof}
We presume that $t_{0}\in\mathbf{R},$ $h>0$ exist such that (\ref{per350})
holds. We have also the existence of $T>0,$ $t^{\prime}\in I^{x}$ such that
(\ref{per598}) is true.

If $T\in(0,h)\cup(h,2h)\cup...\cup(qh,(q+1)h)\cup...$ then $\widehat{x},x$ are
both constant from Theorem \ref{The31} b), page \pageref{The31}, contradiction
with the hypothesis. We suppose at this moment that $T\in\{h,2h,3h,...\}$ and
let $p=\frac{T}{h},$ $p\geq1.$ Let $k_{1}\in\mathbf{Z} $ be the number that
fulfills $t^{\prime}\in\lbrack t_{0}+k_{1}h,t_{0}+(k_{1}+1)h)$ and we define
$k^{\prime}=\max\{k_{1},-1\}. $ We have the following possibilities.

Case $k_{1}\leq-2.$

We have $k^{\prime}=-1.$ Let an arbitrary $k\geq k^{\prime}$ for which an
arbitrary, fixed $t\in\lbrack t_{0}+kh,t_{0}+(k+1)h)$ fulfills%
\[
t+T\in\lbrack t_{0}+kh+T,t_{0}+(k+1)h+T)=[t_{0}+(k+p)h,t_{0}+(k+p+1)h)
\]
and moreover%
\begin{equation}
t\geq t_{0}+kh\geq t_{0}+k^{\prime}h=t_{0}-h>t^{\prime}. \label{per745}%
\end{equation}
We can write that%
\[
\widehat{x}(k)=x(t)\overset{(\ref{per598}),(\ref{per745})}{=}x(t+T)=x(t_{0}%
+(k+p)h)=\widehat{x}(k+p),
\]
thus (\ref{per434}) is true and (\ref{per435}) is also true.

Case $k_{1}\geq-1.$

In this situation we have $k^{\prime}=k_{1}.$ Let us take some arbitrary,
fixed $k\geq k^{\prime}.$ We can write%
\begin{equation}
t_{0}+kh\leq t^{\prime}+(k-k_{1})h<t_{0}+(k+1)h, \label{per746}%
\end{equation}%
\begin{equation}
t_{0}+(k+p)h\leq t^{\prime}+(k-k_{1}+p)h<t_{0}+(k+p+1)h, \label{per747}%
\end{equation}
where%
\begin{equation}
t^{\prime}+(k-k_{1})h\geq t^{\prime}. \label{per748}%
\end{equation}
We infer%
\[
\widehat{x}(k)=x(t_{0}+kh)\overset{(\ref{per746})}{=}x(t^{\prime}%
+(k-k_{1})h)\overset{(\ref{per598}),(\ref{per748})}{=}x(t^{\prime}%
+(k-k_{1})h+T)
\]%
\[
=x(t^{\prime}+(k-k_{1}+p)h)\overset{(\ref{per747})}{=}x(t_{0}+(k+p)h)=\widehat
{x}(k+p).
\]
We have proved the truth of (\ref{per434}) and the truth of (\ref{per435})
results from the fact that%
\[
\forall t\leq t^{\prime},\forall k\in\{-1,0,...,k^{\prime}\},\widehat
{x}(k)=x(t)=x(-\infty+0).
\]

\end{proof}

\begin{example}
We define $\widehat{x}\in\widehat{S}^{(1)}$ by
\[
\forall k\in\mathbf{N}_{\_},\widehat{x}(k)=\left\{
\begin{array}
[c]{c}%
1,if\;k\in\{2,4,6,8,...\}\\
0,otherwise
\end{array}
\right.
\]
and $x\in S^{(1)}$ respectively by%
\[
x(t)=\widehat{x}(-1)\cdot\chi_{(-\infty,-4)}(t)\oplus\widehat{x}(0)\cdot
\chi_{\lbrack-4,-2)}(t)\oplus
\]%
\[
\oplus\widehat{x}(1)\cdot\chi_{\lbrack-2,0)}(t)\oplus\widehat{x}(2)\cdot
\chi_{\lbrack0,2)}(t)\oplus...
\]
We have $I^{x}=(-\infty,0),L^{x}=[-2,\infty),$%
\[
\forall t\leq-2,x(t)=0,
\]%
\[
\forall t\geq-2,x(t)=x(t+4),
\]%
\[
\forall k\geq1,\widehat{x}(k)=\widehat{x}(k+2),
\]%
\[
\forall k\in\{-1,0,1\},\widehat{x}(k)=0
\]
thus (\ref{per598})$_{page\;\pageref{per598}}$ is fulfilled with
$T=4,t^{\prime}=-2$ and (\ref{per434}), (\ref{per435}) are true with
$p=2,k^{\prime}=1.$ Furthermore, in this example $h=2.$
\end{example}

\begin{remark}
In Theorem \ref{The18}, the conjunction of (\ref{per434}) with (\ref{per435})
gives a special case of eventual periodicity.
\end{remark}

\begin{remark}
Theorem \ref{The43} and Theorem \ref{The18} represent the periodic version of
Theorem \ref{The21}, page \pageref{The21} and Theorem \ref{The118}, page
\pageref{The118} that refer to eventual periodicity.
\end{remark}

\section{Sums, differences and multiples of periods}

\begin{theorem}
\label{The17_}Let the signals $\widehat{x},x.$

a) We suppose that $\widehat{x}$ has the periods $p,p^{\prime}\geq1,$%
\begin{equation}
\forall k\in\mathbf{N}_{\_},\widehat{x}(k)=\widehat{x}(k+p), \label{pre9}%
\end{equation}%
\begin{equation}
\forall k\in\mathbf{N}_{\_},\widehat{x}(k)=\widehat{x}(k+p^{\prime}).
\label{pre10}%
\end{equation}
Then $p+p^{\prime}\geq1$, $\widehat{x}$ has the period $p+p^{\prime},$%
\begin{equation}
\forall k\in\mathbf{N}_{\_},\widehat{x}(k)=\widehat{x}(k+p+p^{\prime})
\label{pre11}%
\end{equation}
and if $p>p^{\prime},$ then $p-p^{\prime}\geq1$ and $\widehat{x}$ has the
period $p-p^{\prime},$%
\begin{equation}
\forall k\in\mathbf{N}_{\_},\widehat{x}(k)=\widehat{x}(k+p-p^{\prime}).
\label{pre12}%
\end{equation}

b) Let $T,T^{\prime}>0,$ $t^{\prime}\in I^{x}$ arbitrary with%
\begin{equation}
\forall t\geq t^{\prime},x(t)=x(t+T), \label{pre13}%
\end{equation}%
\begin{equation}
\forall t\geq t^{\prime},x(t)=x(t+T^{\prime}) \label{pre14}%
\end{equation}
fulfilled. We have on one hand that $T+T^{\prime}>0$ and%
\begin{equation}
\forall t\geq t^{\prime},x(t)=x(t+T+T^{\prime}), \label{pre15}%
\end{equation}
and on the other hand that $T>T^{\prime}$ implies $T-T^{\prime}>0$ and%
\begin{equation}
\forall t\geq t^{\prime},x(t)=x(t+T-T^{\prime}). \label{pre16}%
\end{equation}

\end{theorem}

\begin{proof}
This is a special case of Theorem \ref{The79}, page \pageref{The79} with the
limit of periodicity $k^{\prime}=-1$ at a) and $t^{\prime}\in I^{x}$ at b).
\end{proof}

\begin{theorem}
\label{The17}We consider the signals $\widehat{x},x.$

a) Let $p,k_{1}\geq1.$ We have that $p^{\prime}=k_{1}p$ fulfills $p^{\prime
}\geq1$ and%
\begin{equation}
\forall k\in\mathbf{N}_{\_},\widehat{x}(k)=\widehat{x}(k+p) \label{per635}%
\end{equation}
implies%
\begin{equation}
\forall k\in\mathbf{N}_{\_},\widehat{x}(k)=\widehat{x}(k+p^{\prime}).
\label{per636}%
\end{equation}

b) We suppose that $T>0,$ $t^{\prime}\in I^{x}\mathbf{,}$ $k_{1}\geq1$ are
given. We infer that $T^{\prime}=k_{1}T$ fulfills $T^{\prime}>0$ and%
\begin{equation}
\forall t\geq t^{\prime},x(t)=x(t+T) \label{per639}%
\end{equation}
implies%
\begin{equation}
\forall t\geq t^{\prime},x(t)=x(t+T^{\prime}). \label{per640}%
\end{equation}

\end{theorem}

\begin{proof}
This is a consequence of Theorem \ref{The17_}, the first assertion from a), b)
that refers to the addition.
\end{proof}

\begin{remark}
We can express the statements of Theorem \ref{The17} in an equivalent way
under the form: if $p\in\widehat{P}^{\widehat{x}}$ then
$\{p,2p,3p,...\}\subset\widehat{P}^{\widehat{x}}$ and if $T\in P^{x}$ then
$\{T,2T,3T,...\}\subset P^{x}.$
\end{remark}

\section{The set of the periods}

\begin{theorem}
\label{The44}a) We suppose that for $\widehat{x}\in\widehat{S}^{(n)},$ the set
$\widehat{P}^{\widehat{x}}$ is non empty. Some $\widetilde{p}\geq1$ exists
then with the property%
\begin{equation}
\widehat{P}^{\widehat{x}}=\{\widetilde{p},2\widetilde{p},3\widetilde{p},...\}.
\label{per222}%
\end{equation}

b) Let $x\in S^{(n)}$ be not constant and we suppose that the set $P^{x}$ is
not empty. Then $\widetilde{T}>0$ exists such that%
\begin{equation}
P^{x}=\{\widetilde{T},2\widetilde{T},3\widetilde{T},...\}. \label{per224}%
\end{equation}

\end{theorem}

\begin{proof}
This is a special case of Theorem \ref{The124}, page \pageref{The124} with
$\widehat{L}^{\widehat{x}}=\mathbf{N}_{\_}$ at a) and $I^{x}\cap L^{x}%
\neq\varnothing$ at b).
\end{proof}

\begin{remark}
If in Theorem \ref{The44}, item a) $\widehat{x}$ is constant, then
$\widetilde{p}=1$ and $\widehat{P}^{\widehat{x}}=\{1,2,3,...\},$ thus
(\ref{per222}) is still true. Unlike this situation, if $x$ is constant, then
(\ref{per224}) is false and we get $P^{x}=(0,\infty)$ instead.
\end{remark}

\begin{theorem}
We suppose that the relation between $\widehat{x}$ and $x$ is given by%
\begin{equation}%
\begin{array}
[c]{c}%
x(t)=\widehat{x}(-1)\cdot\chi_{(-\infty,t_{0})}(t)\oplus\widehat{x}%
(0)\cdot\chi_{\lbrack t_{0},t_{0}+h)}(t)\oplus\\
\oplus\widehat{x}(1)\cdot\chi_{\lbrack t_{0}+h,t_{0}+2h)}(t)\oplus
...\oplus\widehat{x}(k)\cdot\chi_{\lbrack t_{0}+kh,t_{0}+(k+1)h)}(t)\oplus...
\end{array}
\end{equation}
where $t_{0}\in\mathbf{R}$ and $h>0.$ If $\widehat{x},$ $x$ are periodic, two
possibilities exist:

a) $\widehat{x},$ $x$ are both constant, $\widehat{P}^{\widehat{x}%
}=\{1,2,3,...\}$ and $P^{x}=(0,\infty)$;

b) $\widehat{x},$ $x$ are both non-constant, $\min\widehat{P}^{\widehat{x}%
}=p>1$ is the prime period of $\widehat{x}$ and $\min P^{x}=T=ph$ is the prime
period of $x$.
\end{theorem}

\begin{proof}
The proof is analogue with the proof of Theorem \ref{The115}, page
\pageref{The115}.

b) Theorem \ref{The43}, page \pageref{The43} shows that $p\in\widehat
{P}^{\widehat{x}}\Longrightarrow T=ph\in P^{x}$ and from Theorem \ref{The18},
page \pageref{The18} we get that $T\in P^{x}\Longrightarrow p=\frac{T}{h}%
\in\widehat{P}^{\widehat{x}}.$ We suppose, see Theorem \ref{The44}, that
$\widehat{P}^{\widehat{x}}=\{\widetilde{p},2\widetilde{p},3\widetilde{p},...\}
$ and $P^{x}=\{\widetilde{T},2\widetilde{T},3\widetilde{T},...\}.$ Then
$\widetilde{T}=\widetilde{p}h.$
\end{proof}

\section{Necessity conditions of periodicity}

\begin{theorem}
\label{The58}Let $\widehat{x}\in\widehat{S}^{(n)}$ with $\widehat{Or}%
(\widehat{x})=\{\mu^{1},...,\mu^{s}\},s\geq2$ be periodic with the period
$p\geq1.$ We have the existence, for any $i\in\{1,...,s\},$ of $n_{1}%
^{i},n_{2}^{i},...,n_{k_{i}}^{i}\in\{-1,0,...,p-2\}$, $k_{i}\geq1,$ such that%
\begin{equation}
\widehat{\mathbf{T}}_{\mu^{i}}^{\widehat{x}}=\underset{k\in\mathbf{N}}{%
{\displaystyle\bigcup}
}\{n_{1}^{i}+kp,n_{2}^{i}+kp,...,n_{k_{i}}^{i}+kp\}. \label{pre105}%
\end{equation}

\end{theorem}

\begin{proof}
If $\widehat{x}$ is periodic with the period $p,$ then we can apply for any
$i\in\{1,...,s\}$ Theorem \ref{The54}, page \pageref{The54}, since any
$\mu^{i}$ is periodic with the period $p$.
\end{proof}

\begin{theorem}
\label{The57}Let $x\in S^{(n)}$ non constant with $Or(x)=\{\mu^{1},...,\mu
^{s}\}$ and we denote $\mu^{1}=x(-\infty+0).$ We suppose that $x$ is periodic
with the period $T>0$. Then $t_{0}$ and $a_{1}^{i},b_{1}^{i},a_{2}^{i}%
,b_{2}^{i},...,a_{k_{i}}^{i},b_{k_{i}}^{i}\in\mathbf{R}$ exist, $k_{i}\geq1,$
$i\in\{1,...,s\}$ such that%
\begin{equation}
\forall t<t_{0},x(t)=\mu^{1}, \label{pre109}%
\end{equation}%
\begin{equation}
x(t_{0})\neq\mu^{1}, \label{pre110}%
\end{equation}%
\begin{equation}
t_{0}<a_{1}^{1}<b_{1}^{1}<a_{2}^{1}<b_{2}^{1}<...<a_{k_{1}}^{1}<b_{k_{1}}%
^{1}=t_{0}+T, \label{pre108}%
\end{equation}%
\begin{equation}
\lbrack a_{1}^{1},b_{1}^{1})\cup\lbrack a_{2}^{1},b_{2}^{1})\cup...\cup\lbrack
a_{k_{1}}^{1},b_{k_{1}}^{1})=\mathbf{T}_{\mu^{1}}^{x}\cap\lbrack t_{0}%
,t_{0}+T),
\end{equation}%
\begin{equation}
\mathbf{T}_{\mu^{1}}^{x}=%
\begin{array}
[c]{c}%
(-\infty,t_{0})\cup\underset{k\in\mathbf{N}}{%
{\displaystyle\bigcup}
}([a_{1}^{1}+kT,b_{1}^{1}+kT)\cup\lbrack a_{2}^{1}+kT,b_{2}^{1}+kT)\cup...\\
...\cup\lbrack a_{k_{1}}^{1}+kT,b_{k_{1}}^{1}+kT))
\end{array}
\label{pre111}%
\end{equation}
and for any $i\in\{2,...,s\},$%
\begin{equation}
t_{0}\leq a_{1}^{i}<b_{1}^{i}<a_{2}^{i}<b_{2}^{i}<...<a_{k_{i}}^{i}<b_{k_{i}%
}^{i}<t_{0}+T, \label{pre114}%
\end{equation}%
\begin{equation}
\lbrack a_{1}^{i},b_{1}^{i})\cup\lbrack a_{2}^{i},b_{2}^{i})\cup...\cup\lbrack
a_{k_{i}}^{i},b_{k_{i}}^{i})=\mathbf{T}_{\mu^{i}}^{x}\cap\lbrack t_{0}%
,t_{0}+T),
\end{equation}%
\begin{equation}
\mathbf{T}_{\mu^{i}}^{x}=\underset{k\in\mathbf{N}}{%
{\displaystyle\bigcup}
}([a_{1}^{i}+kT,b_{1}^{i}+kT)\cup\lbrack a_{2}^{i}+kT,b_{2}^{i}+kT)\cup
...\cup\lbrack a_{k_{1}}^{i}+kT,b_{k_{1}}^{i}+kT)) \label{pre117}%
\end{equation}
are fulfilled.
\end{theorem}

\begin{proof}
As $x$ is non constant, periodic with the period $T,$ all of $\mu^{1}%
,...,\mu^{s}$ are periodic, with the period $T$. Theorem \ref{The51_}, page
\pageref{The51_} shows then the existence of $t_{0},a_{1}^{1},b_{1}^{1}%
,a_{2}^{1},b_{2}^{1},...,a_{k_{1}}^{1},b_{k_{1}}^{1}\in\mathbf{R}$, $k_{1}%
\geq1$ such that (\ref{pre109}),...,(\ref{pre111}) are true. The existence of
$t_{0},a_{1}^{i},b_{1}^{i},a_{2}^{i},b_{2}^{i},...,$ $a_{k_{i}}^{i},b_{k_{i}%
}^{i}\in\mathbf{R}$, $k_{i}\geq1$ for $i\in\{2,...,s\}$ such that
(\ref{pre109}), (\ref{pre110}), (\ref{pre114}),..., (\ref{pre117}) hold
results from Theorem \ref{The52 copy(1)}, page \pageref{The52 copy(1)}.
\end{proof}

\begin{example}
\label{Exa11}The periodic signal $x\in S^{(1)},$%
\[
x(t)=\chi_{(-\infty,0)}(t)\oplus\chi_{\lbrack1,2)}(t)\oplus\chi_{\lbrack
3,5)}(t)\oplus\chi_{\lbrack6,7)}(t)\oplus\chi_{\lbrack8,10)}(t)\oplus
\chi_{\lbrack11,12)}(t)\oplus...
\]
fulfills $\mu^{1}=1,\mu^{2}=0,k_{1}=k_{2}=2,T=5,t_{0}=a_{1}^{2}=0,a_{1}%
^{1}=b_{1}^{2}=1,b_{1}^{1}=a_{2}^{2}=2,a_{2}^{1}=b_{2}^{2}=3,b_{2}^{1}%
=t_{0}+5=5. $
\end{example}

\section{Sufficiency conditions of periodicity}

\begin{theorem}
\label{The59}Let $\widehat{x}\in\widehat{S}^{(n)}$, $\widehat{Or}(\widehat
{x})=\{\mu^{1},...,\mu^{s}\}$ and $p\geq1.$ We suppose that $\forall
i\in\{1,...,s\},$ $n_{1}^{i},n_{2}^{i},...,n_{k_{i}}^{i}\in
\{-1,0,...,p-2\},k_{i}\geq1$ exist such that%
\begin{equation}
\widehat{\mathbf{T}}_{\mu^{i}}^{\widehat{x}}=\underset{k\in\mathbf{N}}{%
{\displaystyle\bigcup}
}\{n_{1}^{i}+kp,n_{2}^{i}+kp,...,n_{k_{i}}^{i}+kp\}. \label{pre119}%
\end{equation}
Then $\forall i\in\{1,...,s\},$
\begin{equation}
\forall k\in\widehat{\mathbf{T}}_{\mu^{i}}^{\widehat{x}},\{k+zp|z\in
\mathbf{Z}\}\cap\mathbf{N}_{\_}\subset\widehat{\mathbf{T}}_{\mu^{i}}%
^{\widehat{x}}. \label{pre120}%
\end{equation}

\end{theorem}

\begin{proof}
If $\forall i\in\{1,...,s\},$ $n_{1}^{i},n_{2}^{i},...,n_{k_{i}}^{i}%
\in\{-1,0,...,p-2\},k_{i}\geq1$ exist such that (\ref{pre119}) holds, we infer
from Theorem \ref{The55}, page \pageref{The55} that all of $\mu^{1}%
,...,\mu^{s}$ are periodic with the period $p$, i.e. $\widehat{x}$ is periodic
with the period $p$.
\end{proof}

\begin{theorem}
\label{The60}The signal $x\in S^{(n)}$ is given, such that $Or(x)=\{\mu
^{1},...,\mu^{s}\},s\geq2$ and we suppose that the initial value of $x$ is
$\mu^{1}.$ We ask that $T>0$ and the points $t_{0},a_{1}^{i},b_{1}^{i}%
,a_{2}^{i},b_{2}^{i},...,a_{k_{i}}^{i},b_{k_{i}}^{i}\in\mathbf{R,}$ $k_{i}%
\geq1$ exist, $i\in\{1,...,s\}$ such that%
\begin{equation}
t_{0}<a_{1}^{1}<b_{1}^{1}<a_{2}^{1}<b_{2}^{1}<...<a_{k_{1}}^{1}<b_{k_{1}}%
^{1}=t_{0}+T, \label{pre121}%
\end{equation}%
\begin{equation}
\mathbf{T}_{\mu^{1}}^{x}=%
\begin{array}
[c]{c}%
(-\infty,t_{0})\cup\underset{k\in\mathbf{N}}{%
{\displaystyle\bigcup}
}([a_{1}^{1}+kT,b_{1}^{1}+kT)\cup\lbrack a_{2}^{1}+kT,b_{2}^{1}+kT)\cup...\\
...\cup\lbrack a_{k_{1}}^{1}+kT,b_{k_{1}}^{1}+kT))
\end{array}
\label{pre122}%
\end{equation}
and for any $i\in\{2,...,s\},$%
\begin{equation}
b_{k_{i}}^{i}-T<t_{0}\leq a_{1}^{i}<b_{1}^{i}<a_{2}^{i}<b_{2}^{i}%
<...<a_{k_{i}}^{i}<b_{k_{i}}^{i}, \label{pre125}%
\end{equation}%
\begin{equation}
\mathbf{T}_{\mu^{i}}^{x}=\underset{k\in\mathbf{N}}{%
{\displaystyle\bigcup}
}([a_{1}^{i}+kT,b_{1}^{i}+kT)\cup\lbrack a_{2}^{i}+kT,b_{2}^{i}+kT)\cup
...\cup\lbrack a_{k_{1}}^{i}+kT,b_{k_{1}}^{i}+kT)). \label{pre126}%
\end{equation}
For any $t^{\prime}\in\lbrack a_{k_{1}}^{1}-T,t_{0}),$ we have $t^{\prime}\in
I^{x}$ and%
\begin{equation}
\forall i\in\{1,...,s\},\forall t\in\mathbf{T}_{\mu^{i}}^{x}\cap\lbrack
t^{\prime},\infty),\{t+zT|z\in\mathbf{Z}\}\cap\lbrack t^{\prime}%
,\infty)\subset\mathbf{T}_{\mu^{i}}^{x}. \label{pre128}%
\end{equation}

\end{theorem}

\begin{proof}
We suppose that $T$ and $t_{0},a_{1}^{1},b_{1}^{1},a_{2}^{1},b_{2}%
^{1},...,a_{k_{1}}^{1},b_{k_{1}}^{1}, $ $k_{1}\geq1$ exist such that
(\ref{pre121}), (\ref{pre122}) are true and let $t^{\prime}\in\lbrack
a_{k_{1}}^{1}-T,t_{0})$ arbitrary, fixed. From Theorem \ref{The11}, page
\pageref{The11}, we have that for an arbitrary $t_{1}^{\prime}\in\lbrack
a_{k_{1}}^{1}-T,t_{0}),$ the statements $t_{1}^{\prime}\in I^{x},$%
\begin{equation}
\forall t\in\mathbf{T}_{\mu^{1}}^{x}\cap\lbrack t_{1}^{\prime},\infty
),\{t+zT|z\in\mathbf{Z}\}\cap\lbrack t_{1}^{\prime},\infty)\subset
\mathbf{T}_{\mu^{1}}^{x} \label{p129}%
\end{equation}
hold. We suppose furthermore that $\forall i\in\{2,...,s\},$ $a_{1}^{i}%
,b_{1}^{i},a_{2}^{i},b_{2}^{i},...,a_{k_{i}}^{i},b_{k_{i}}^{i},$ $k_{i}\geq1$
exist with the property that (\ref{pre125}), (\ref{pre126}) are true. As far
as%
\[
\forall t<t_{0},x(t)=\mu^{1},
\]%
\[
x(t_{0})\overset{(\ref{pre121}),(\ref{pre122})}{\neq}\mu^{1},
\]
Theorem \ref{The11_}, page \pageref{The11_} shows that for any $t_{i}^{\prime
}\in\lbrack b_{k_{i}}^{i}-T,t_{0}),$ the statements $t_{i}^{\prime}\in I^{x},$%
\begin{equation}
\forall t\in\mathbf{T}_{\mu^{i}}^{x}\cap\lbrack t_{i}^{\prime},\infty
),\{t+zT|z\in\mathbf{Z}\}\cap\lbrack t_{i}^{\prime},\infty)\subset
\mathbf{T}_{\mu^{i}}^{x} \label{p70}%
\end{equation}
hold also, i.e. $\mu^{1},...,\mu^{s}$ are all periodic with the same period
$T.$ Corollary \ref{Cor4}, page \pageref{Cor4}, states basically that%
\[
\lbrack a_{k_{1}}^{1}-T,t_{0})=[a_{k_{1}}^{1}-T,t_{0})\cap\lbrack b_{k_{2}%
}^{2}-T,t_{0})\cap...\cap\lbrack b_{k_{s}}^{s}-T,t_{0}),
\]
i.e. in (\ref{p129}),(\ref{p70}) we can make the choice $t_{i}^{\prime
}=t^{\prime},i\in\{1,...,s\}.$ In such circumstances $t^{\prime}\in I^{x}$ and
(\ref{pre128}) holds$.$
\end{proof}

\section{A special case}

\begin{theorem}
\label{The59 copy(1)}Let $\widehat{x}\in\widehat{S}^{(n)}$ with $\widehat
{Or}(\widehat{x})=\{\mu^{1},...,\mu^{s}\}$ and $p\geq1.$ We suppose that
$\forall i\in\{1,...,s\},$ $n^{i}\in\{-1,0,...,p-2\}$ exists such that%
\begin{equation}
\widehat{\mathbf{T}}_{\mu^{i}}^{\widehat{x}}=\{n^{i},n^{i}+p,n^{i}+2p,...\}.
\end{equation}

a) We have $\forall i\in\{1,...,s\},$
\begin{equation}
\forall k\in\widehat{\mathbf{T}}_{\mu^{i}}^{\widehat{x}},\{k+zp|z\in
\mathbf{Z}\}\cap\mathbf{N}_{\_}\subset\widehat{\mathbf{T}}_{\mu^{i}}%
^{\widehat{x}}.
\end{equation}

b) Any $p^{\prime}\geq1$ making $\forall i\in\{1,...,s\},$
\begin{equation}
\forall k\in\widehat{\mathbf{T}}_{\mu^{i}}^{\widehat{x}},\{k+zp^{\prime}%
|z\in\mathbf{Z}\}\cap\mathbf{N}_{\_}\subset\widehat{\mathbf{T}}_{\mu^{i}%
}^{\widehat{x}}%
\end{equation}
true fulfills $p^{\prime}\in\{p,2p,3p,...\},$ i.e. $p$ is the prime period of
$\widehat{x}$.
\end{theorem}

\begin{proof}
Item a) is a special case of Theorem \ref{The59}, page \pageref{The59} with
$k_{i}=1,i=\overline{1,s}.$ The theorem is also a special case of Theorem
\ref{The134}, page \pageref{The134}.
\end{proof}

\begin{theorem}
\label{The75 copy(1)}Let $x\in S^{(n)}$ with $Or(x)=\{\mu^{1},...,\mu^{s}\},$
$s\geq2,$ $T>0$ and the points $t_{0},t_{1},...,t_{s-1}\in\mathbf{R}$ with the
following property:%
\begin{equation}
t_{0}<t_{1}<...<t_{s-1}<t_{0}+T, \label{pre3}%
\end{equation}%
\begin{equation}
\mathbf{T}_{\mu^{1}}^{x}=(-\infty,t_{0})\cup\lbrack t_{s-1},t_{0}%
+T)\cup\lbrack t_{s-1}+T,t_{0}+2T)\cup\lbrack t_{s-1}+2T,t_{0}+3T)\cup...
\label{pre4}%
\end{equation}%
\begin{equation}
\mathbf{T}_{\mu^{2}}^{x}=[t_{0},t_{1})\cup\lbrack t_{0}+T,t_{1}+T)\cup\lbrack
t_{0}+2T,t_{1}+2T)\cup... \label{pre5}%
\end{equation}%
\[
...
\]%
\begin{equation}
\mathbf{T}_{\mu^{s}}^{x}=[t_{s-2},t_{s-1})\cup\lbrack t_{s-2}+T,t_{s-1}%
+T)\cup\lbrack t_{s-2}+2T,t_{s-1}+2T)\cup... \label{pre6}%
\end{equation}

a) For any $t^{\prime}\in\lbrack t_{s-1}-T,t_{0}),$ the properties $t^{\prime
}\in I^{x},$%
\begin{equation}
\forall i\in\{1,...,s\},\forall t\in\mathbf{T}_{\mu^{i}}^{x}\cap\lbrack
t^{\prime},\infty),\{t+zT|z\in\mathbf{Z}\}\cap\lbrack t^{\prime}%
,\infty)\subset\mathbf{T}_{\mu^{i}}^{x} \label{pre8}%
\end{equation}
are fulfilled.

b) Let $t^{\prime\prime}\in\lbrack t_{s-1}-T,t_{0})$ arbitrary. For any
$T^{\prime}>0$ such that%
\begin{equation}
\forall i\in\{1,...,s\},\forall t\in\mathbf{T}_{\mu^{i}}^{x}\cap\lbrack
t^{\prime\prime},\infty),\{t+zT^{\prime}|z\in\mathbf{Z}\}\cap\lbrack
t^{\prime\prime},\infty)\subset\mathbf{T}_{\mu^{i}}^{x},
\end{equation}
we have $T^{\prime}\in\{T,2T,3T,...\}.$
\end{theorem}

\begin{proof}
Item a) is a special case of Theorem \ref{The60}, page \pageref{The60}. By
definition we have $t_{0}=a_{1}^{s},t_{1}=a_{1}^{s-1}=b_{1}^{s},...,t_{s-2}%
=a_{1}^{2}=b_{1}^{3},t_{s-1}=a_{1}^{1}=b_{1}^{2},t_{0}+T=b_{1}^{1}. $ The
Theorem is also a special case of Theorem \ref{The135}, page \pageref{The135}.
\end{proof}

\section{Periodicity vs eventual periodicity}

\begin{theorem}
\label{The121}a) If $\widehat{x}$ is periodic, then for any $\widetilde{k}%
\in\mathbf{N}$, $\widehat{\sigma}^{\widetilde{k}}(\widehat{x})$ is periodic
and $\widehat{P}^{\widehat{x}}=\widehat{P}^{\widehat{\sigma}^{\widetilde{k}%
}(\widehat{x})}.$

b) We suppose that $x$ is periodic. For arbitrary $\widetilde{t}\in
\mathbf{R,}$ we have that $\sigma^{\widetilde{t}}(x)$ is periodic and
$P^{x}=P^{\sigma^{\widetilde{t}}(x)}.$
\end{theorem}

\begin{proof}
a) We suppose that $\widehat{P}^{\widehat{x}}\neq\varnothing$ and let
$\widetilde{k}\in\mathbf{N}$ arbitrary.

We prove $\widehat{P}^{\widehat{x}}\subset\widehat{P}^{\widehat{\sigma
}^{\widetilde{k}}(\widehat{x})}.$ We take an arbitrary $p\in\widehat
{P}^{\widehat{x}},$ meaning that
\begin{equation}
\forall k\in\mathbf{N}_{\_},\widehat{x}(k)=\widehat{x}(k+p) \label{pre995}%
\end{equation}
holds and we show that%
\begin{equation}
\forall k\in\mathbf{N}_{\_},\widehat{\sigma}^{\widetilde{k}}(\widehat
{x})(k)=\widehat{\sigma}^{\widetilde{k}}(\widehat{x})(k+p) \label{p3}%
\end{equation}
is true. Indeed, for any $k\in\mathbf{N}_{\_}$ we get%
\[
\widehat{\sigma}^{\widetilde{k}}(\widehat{x})(k)=\widehat{x}(k+\widetilde
{k})\overset{(\ref{pre995})}{=}\widehat{x}(k+\widetilde{k}+p)=\widehat{\sigma
}^{\widetilde{k}}(\widehat{x})(k+p).
\]

We prove $\widehat{P}^{\widehat{\sigma}^{\widetilde{k}}(\widehat{x})}%
\subset\widehat{P}^{\widehat{x}}.$ Let $p\in\widehat{P}^{\widehat{x}},$ thus
(\ref{pre995}) is true. We suppose against all reason that $\widehat
{P}^{\widehat{\sigma}^{\widetilde{k}}(\widehat{x})}\subset\widehat
{P}^{\widehat{x}}$ is false, i.e. some $p^{\prime}\in\widehat{P}%
^{\widehat{\sigma}^{\widetilde{k}}(\widehat{x})}\smallsetminus\widehat
{P}^{\widehat{x}}$ exists. This means, by rewriting (\ref{p3}) under the form%
\begin{equation}
\forall k\geq\widetilde{k}-1,\widehat{x}(k)=\widehat{x}(k+p^{\prime}),
\label{p76}%
\end{equation}
that%
\begin{equation}
\exists k_{1}\in\{-1,0,...,\widetilde{k}-2\},\widehat{x}(k_{1})\neq\widehat
{x}(k_{1}+p^{\prime}). \label{p77}%
\end{equation}
Let $\overline{k}\in\mathbf{N}$ with the property that $k_{1}+\overline
{k}p\geq\widetilde{k}-1.$ We infer:%
\begin{equation}
\widehat{x}(k_{1})\overset{(\ref{pre995})}{=}\widehat{x}(k_{1}+\overline
{k}p)\overset{(\ref{p76})}{=}\widehat{x}(k_{1}+\overline{k}p+p^{\prime
})\overset{(\ref{pre995})}{=}\widehat{x}(k_{1}+p^{\prime}). \label{p78}%
\end{equation}
The statements (\ref{p77}) and (\ref{p78}) are contradictory.

b) We suppose that $P^{x}\neq\varnothing$ and we take $\widetilde{t}%
\in\mathbf{R}$ arbitrarily.

We prove $P^{x}\subset P^{\sigma^{\widetilde{t}}(x)}.$ Let $T\in P^{x}$
arbitrary$,$ for which $t^{\prime}\in I^{x}$ exists such that%
\begin{equation}
\forall t\geq t^{\prime},x(t)=x(t+T) \label{pre998}%
\end{equation}
holds. We must prove the existence of $t^{\prime\prime}\in I^{\sigma
^{\widetilde{t}}(x)}$ making%
\begin{equation}
\forall t\geq t^{\prime\prime},\sigma^{\widetilde{t}}(x)(t)=\sigma
^{\widetilde{t}}(x)(t+T) \label{pre996}%
\end{equation}
true. If $x$ is constant, then $\sigma^{\widetilde{t}}(x)=x$ and
$t^{\prime\prime}\in I^{\sigma^{\widetilde{t}}(x)}$, (\ref{pre996}) are
trivially true for any $t^{\prime\prime}\in R$ so that we shall suppose from
now that $x$ is not constant. Some $t_{0}\in\mathbf{R}$ exists with%
\begin{equation}
\forall t<t_{0},x(t)=x(-\infty+0), \label{pre999}%
\end{equation}%
\begin{equation}
x(t_{0})\neq x(-\infty+0). \label{p1}%
\end{equation}
From (\ref{pre999}), (\ref{p1}) we get $I^{x}=(-\infty,t_{0})$ and since
$t^{\prime}\in I^{x},$ we infer $t^{\prime}<t_{0}.$ Two possibilities exist.

Case $\widetilde{t}\leq t_{0}$

In this case $\sigma^{\widetilde{t}}(x)=x$ and $t^{\prime\prime}\in
I^{\sigma^{\widetilde{t}}(x)}$, (\ref{pre996}) are true again for
$t^{\prime\prime}=t^{\prime}.$

Case $\widetilde{t}>t_{0}$

Some $\varepsilon>0$ exists with the property $\forall\xi\in(\widetilde
{t}-\varepsilon,\widetilde{t}),x(\xi)=x(\widetilde{t}-0).$ We infer from here
that $\widetilde{t}-\varepsilon\geq t_{0}>t^{\prime}$ and for $t^{\prime
\prime}\in(\widetilde{t}-\varepsilon,\widetilde{t})$ arbitrary, fixed we have%
\begin{equation}
\sigma^{\widetilde{t}}(x)(t)=\left\{
\begin{array}
[c]{c}%
x(t),t\geq t^{\prime\prime}\\
x(\widetilde{t}-0),t<\widetilde{t}%
\end{array}
\right.  . \label{p2}%
\end{equation}
We notice that $t^{\prime\prime}\in I^{\sigma^{\widetilde{t}}(x)}$ and, on the
other hand, we can write for any $t\geq t^{\prime\prime}$ that%
\[
\sigma^{\widetilde{t}}(x)(t)\overset{(\ref{p2})}{=}x(t)\overset{(\ref{pre998}%
)}{=}x(t+T)\overset{(\ref{p2})}{=}\sigma^{\widetilde{t}}(x)(t+T),
\]
wherefrom the truth of (\ref{pre996}).

We prove $P^{\sigma^{\widetilde{t}}(x)}\subset P^{x}.$ Let $T\in P^{x}$
arbitrary$,$ thus $t^{\prime}\in I^{x}$ exists such that (\ref{pre998}) takes
place. We suppose against all reason that $P^{\sigma^{\widetilde{t}}%
(x)}\subset P^{x}$ is false, i.e. some $T^{\prime}\in P^{\sigma^{\widetilde
{t}}(x)}\smallsetminus P^{x}$ exists. This means the existence of
$t^{\prime\prime}\in I^{\sigma^{\widetilde{t}}(x)}$ with%
\begin{equation}
\forall t\geq t^{\prime\prime},\sigma^{\widetilde{t}}(x)(t)=\sigma
^{\widetilde{t}}(x)(t+T^{\prime}), \label{p81}%
\end{equation}%
\begin{equation}
\forall t^{\prime\prime\prime}\in I^{x},\exists t_{1}\geq t^{\prime
\prime\prime},x(t_{1})\neq x(t_{1}+T^{\prime}). \label{p82}%
\end{equation}
Let $t^{\prime\prime\prime}\in I^{x}\cap\lbrack t^{\prime},\infty)$ arbitrary
and we take $\overline{k}\in\mathbf{N}$ such that for $t_{1}\geq
t^{\prime\prime\prime}$ whose existence is stated in (\ref{p82}), we have
$t_{1}+\overline{k}T\geq\max\{t^{\prime\prime},\widetilde{t}\}.$ We get%
\begin{equation}%
\begin{array}
[c]{c}%
x(t_{1})\overset{(\ref{pre998})}{=}x(t_{1}+\overline{k}T)=\sigma
^{\widetilde{t}}(x)(t_{1}+\overline{k}T)\\
\overset{(\ref{p81})}{=}\sigma^{\widetilde{t}}(x)(t_{1}+\overline
{k}T+T^{\prime})=x(t_{1}+\overline{k}T+T^{\prime})\overset{(\ref{pre998})}%
{=}x(t_{1}+T^{\prime}),
\end{array}
\label{p83}%
\end{equation}
and (\ref{p82}), (\ref{p83}) are contradictory.
\end{proof}

\begin{remark}
In Theorem \ref{The121}, the statements about the eventual periodicity of
$\widehat{x},x$ are statements about the periodicity of $\widehat{\sigma
}^{\widetilde{k}}(\widehat{x}),\sigma^{\widetilde{t}}(x).$
\end{remark}

\section{Changing the order of the quantifiers}

\begin{theorem}
\label{The31_}a) The statements%
\begin{equation}
\exists p\geq1,\forall\mu\in\widehat{Or}(\widehat{x}),\forall k\in
\widehat{\mathbf{T}}_{\mu}^{\widehat{x}},\{k+zp|z\in\mathbf{Z}\}\cap
\mathbf{N}_{\_}\subset\widehat{\mathbf{T}}_{\mu}^{\widehat{x}}, \label{pre991}%
\end{equation}%
\begin{equation}
\forall\mu\in\widehat{Or}(\widehat{x}),\exists p\geq1,\forall k\in
\widehat{\mathbf{T}}_{\mu}^{\widehat{x}},\{k+zp|z\in\mathbf{Z}\}\cap
\mathbf{N}_{\_}\subset\widehat{\mathbf{T}}_{\mu}^{\widehat{x}} \label{per148}%
\end{equation}
are equivalent.

b) The real time statements%
\begin{equation}
\left\{
\begin{array}
[c]{c}%
\exists T>0,\exists t^{\prime}\in I^{x},\forall\mu\in Or(x),\\
\forall t\in\mathbf{T}_{\mu}^{x}\cap\lbrack t^{\prime},\infty),\{t+zT|z\in
\mathbf{Z}\}\cap\lbrack t^{\prime},\infty)\subset\mathbf{T}_{\mu}^{x},
\end{array}
\right.  \label{pre992}%
\end{equation}%
\begin{equation}
\left\{
\begin{array}
[c]{c}%
\exists T>0,\forall\mu\in Or(x),\exists t^{\prime}\in I^{x},\\
\forall t\in\mathbf{T}_{\mu}^{x}\cap\lbrack t^{\prime},\infty),\{t+zT|z\in
\mathbf{Z}\}\cap\lbrack t^{\prime},\infty)\subset\mathbf{T}_{\mu}^{x},
\end{array}
\right.  \label{per154}%
\end{equation}%
\begin{equation}
\left\{
\begin{array}
[c]{c}%
\exists t^{\prime}\in I^{x},\forall\mu\in Or(x),\exists T>0,\\
\forall t\in\mathbf{T}_{\mu}^{x}\cap\lbrack t^{\prime},\infty),\{t+zT|z\in
\mathbf{Z}\}\cap\lbrack t^{\prime},\infty)\subset\mathbf{T}_{\mu}^{x},
\end{array}
\right.  \label{per749}%
\end{equation}%
\begin{equation}
\left\{
\begin{array}
[c]{c}%
\forall\mu\in Or(x),\exists T>0,\exists t^{\prime}\in I^{x},\\
\forall t\in\mathbf{T}_{\mu}^{x}\cap\lbrack t^{\prime},\infty),\{t+zT|z\in
\mathbf{Z}\}\cap\lbrack t^{\prime},\infty)\subset\mathbf{T}_{\mu}^{x}%
\end{array}
\right.  \label{per149}%
\end{equation}
are also equivalent.
\end{theorem}

\begin{proof}
a) Since (\ref{pre991})$\Longrightarrow$(\ref{per148}) is obvious, we prove
(\ref{per148})$\Longrightarrow$(\ref{pre991}). We denote $\widehat
{Or}(\widehat{x})=\{\mu^{1},...,\mu^{s}\}.$ From (\ref{per148}) we get that
for any $i\in\{1,...,s\},$ some $p_{i}\geq1$ exists such that%
\[
\forall k\in\widehat{\mathbf{T}}_{\mu^{i}}^{\widehat{x}},\{k+zp_{i}%
|z\in\mathbf{Z}\}\cap\mathbf{N}_{\_}\subset\widehat{\mathbf{T}}_{\mu^{i}%
}^{\widehat{x}}.
\]
The number $p=p_{1}\cdot...\cdot p_{s}\geq1$ fulfills the property that
$\forall i\in\{1,...,s\},$ we have%
\begin{equation}
\forall k\in\widehat{\mathbf{T}}_{\mu^{i}}^{\widehat{x}},\{k+zp|z\in
\mathbf{Z}\}\cap\mathbf{N}_{\_}\subset\{k+zp_{i}|z\in\mathbf{Z}\}\cap
\mathbf{N}_{\_}\subset\widehat{\mathbf{T}}_{\mu^{i}}^{\widehat{x}},
\label{per167}%
\end{equation}
i.e. (\ref{pre991}) is true.

b) The implications (\ref{pre992})$\Longrightarrow$(\ref{per154}%
)$\Longrightarrow$(\ref{per149}) and (\ref{pre992})$\Longrightarrow
$(\ref{per749})$\Longrightarrow$(\ref{per149}) are obvious.

The implication (\ref{per149})$\Longrightarrow$(\ref{pre992}) has no proof.
\end{proof}

\begin{remark}
In the previous Theorem, where the proof of the implication (\ref{per149}%
)$\Longrightarrow$(\ref{pre992}) is missing, we address the problem of
changing the order of the quantifiers in stating periodicity properties of the
signals. The importance of this aspect is given by the fact that we are
tempted to define the periodic signals by (\ref{per148}), (\ref{per149}) (all
the points of $\widehat{Or}(\widehat{x}),$ $Or(x)$ are periodic) and to use
(\ref{pre991}), (\ref{pre992}) instead.
\end{remark}

\section{Further research}

\begin{remark}
Remarks \ref{Rem22} and \ref{Rem23}, page \pageref{Rem22} from the eventually
periodic signals, as well as Section \ref{Sec1}, page \pageref{Sec1} may be
restated in the case of the periodic signals also, where they are still interesting.
\end{remark}

\chapter{Examples}

We sketch in this Chapter some constructions that either weaken, in discrete
time and real time, the periodicity of the points to eventual periodicity, or
change the sets of periods.

\section{Discrete time, periodic points}

\begin{remark}
In the following, the signal $\widehat{x}\in\widehat{S}^{(n)}$ and the
periodic point $\mu\in\widehat{Or}(\widehat{x})$ with the period
$\widetilde{p}\geq1$ are considered. The sets $\widehat{\mathbf{T}}_{\mu
}^{\widehat{x}}$ and $\widehat{P}_{\mu}^{\widehat{x}}$ are given by%
\[
\widehat{\mathbf{T}}_{\mu}^{\widehat{x}}=\{k_{0},k_{0}+\widetilde{p}%
,k_{0}+2\widetilde{p},...\},k_{0}\in\{-1,0,...,\widetilde{p}-2\},
\]%
\[
\widehat{P}_{\mu}^{\widehat{x}}=\{\widetilde{p},2\widetilde{p},3\widetilde
{p},...\}.
\]
This simple form of $\widehat{\mathbf{T}}_{\mu}^{\widehat{x}},$ corresponding
to the special case from Theorem \ref{The120}, page \pageref{The120} does not
restrict the generality of the exposure.
\end{remark}

\begin{example}
Let $\mu^{\prime}\in\mathbf{B}^{n},\mu^{\prime}\neq\mu$ and the time instant
$k^{\prime}\in\widehat{\mathbf{T}}_{\mu}^{\widehat{x}}.$ We define
$\widehat{y}\in\widehat{S}^{(n)}$ by $\forall k\in\mathbf{N}_{\_},$%
\[
\widehat{y}(k)=\left\{
\begin{array}
[c]{c}%
\widehat{x}(k),k\neq k^{\prime},\\
\mu^{\prime},k=k^{\prime}.
\end{array}
\right.
\]
We say that $\widehat{y}$ is obtained by removing from $\widehat{x}$ the
instant $k^{\prime}$ of periodicity of $\mu$ and we interpret the fact that
$\widehat{y}(k^{\prime})=\mu^{\prime}\neq\mu=\widehat{x}(k^{\prime})$ instead
of $\widehat{y}(k^{\prime})=\mu=\widehat{x}(k^{\prime})$ as representing an
error, or a perturbation of the periodicity of $\mu$. We have: after removing
from $\widehat{x}$ the instant $k^{\prime}$ of periodicity of $\mu$ we loose
the periodicity of $\mu$, but we still have eventual periodicity with the same
sets of periods $\widehat{P}_{\mu}^{\widehat{y}}=\widehat{P}_{\mu}%
^{\widehat{x}}$ and with the limit of periodicity $k^{\prime}+1.$
\end{example}

\begin{example}
The previous example is continued by taking the points $\mu^{1},...,\mu^{s}%
\in\mathbf{B}^{n}$ that are not necessarily distinct, but they differ from
$\mu$: $\mu^{1}\neq\mu,...,\mu^{s}\neq\mu$ and also the distinct time instants
$k^{1},...,k^{s}\in\widehat{\mathbf{T}}_{\mu}^{\widehat{x}}.$ The signal
$\widehat{y}\in\widehat{S}^{(n)}$ is defined in the following way: $\forall
k\in\mathbf{N}_{\_},$%
\[
\widehat{y}(k)=\left\{
\begin{array}
[c]{c}%
\widehat{x}(k),k\notin\{k^{1},...,k^{s}\},\\
\mu^{j},\exists j\in\{1,...,s\},k=k^{j}.
\end{array}
\right.
\]
We use to say that $\widehat{y}$ is obtained by the removal from $\widehat{x}
$ of the instants $k^{1},...,k^{s}$ of periodicity of $\mu.$ Then $\widehat
{P}_{\mu}^{\widehat{y}}=\widehat{P}_{\mu}^{\widehat{x}}$ again, where $\mu$ is
an eventually periodic point of $\widehat{y}$ and $1+\max\{k^{1},...,k^{s}\}$
is its limit of periodicity.
\end{example}

\begin{example}
\label{Exa9}We give the countable version of the construction. The sequence
$\mu^{j}\in\mathbf{B}^{n},j\in\mathbf{N}_{\_}$ is considered with $\forall
j\in\mathbf{N}_{\_}\mathbf{,}\mu^{j}\neq\mu$ and also the sequence $k^{j}%
\in\widehat{\mathbf{T}}_{\mu}^{\widehat{x}},j\in\mathbf{N}_{\_}$ of distinct
time instants. We define $\widehat{y}\in\widehat{S}^{(n)}$: $\forall
k\in\mathbf{N}_{\_},$%
\[
\widehat{y}(k)=\left\{
\begin{array}
[c]{c}%
\widehat{x}(k),\forall j\in\mathbf{N}_{\_},k\neq k^{j},\\
\mu^{j},\exists j\in\mathbf{N}_{\_},k=k^{j},
\end{array}
\right.
\]
thus $\widehat{y}$ is obtained by removing from $\widehat{x}$ the instants
$(k^{j})$ of periodicity of $\mu.$ We get several possibilities that result
from this construction, we give here only one of these possibilities:%
\[
\mu^{j}=\mu^{\prime}\in\mathbf{B}^{n}\setminus\widehat{Or}(\widehat{x}%
),j\in\mathbf{N}_{\_},
\]%
\[
(k^{-1},k^{0},k^{1},...)=(k_{0},k_{0}+2\widetilde{p},k_{0}+4\widetilde
{p},...)
\]
and the periodic point $\mu$ gives birth to two periodic points, $\mu$ and
$\mu^{\prime},$ with%
\[
\widehat{P}_{\mu}^{\widehat{y}}=\widehat{P}_{\mu^{\prime}}^{\widehat{y}%
}=\{2\widetilde{p},4\widetilde{p},6\widetilde{p},...\},
\]%
\[
\widehat{\mathbf{T}}_{\mu}^{\widehat{y}}=\{k_{0}+\widetilde{p},k_{0}%
+3\widetilde{p},k_{0}+5\widetilde{p},...\},
\]%
\[
\widehat{\mathbf{T}}_{\mu^{\prime}}^{\widehat{y}}=\{k_{0},k_{0}+2\widetilde
{p},k_{0}+4\widetilde{p},...\}.
\]

\end{example}

\begin{example}
For the time instant $k^{\prime}\in\mathbf{N}_{\_}\setminus\widehat
{\mathbf{T}}_{\mu}^{\widehat{x}},$ we define $\widehat{y}\in\widehat{S}^{(n)}$
like that: $\forall k\in\mathbf{N}_{\_},$%
\[
\widehat{y}(k)=\left\{
\begin{array}
[c]{c}%
\widehat{x}(k),k\neq k^{\prime},\\
\mu,k=k^{\prime}.
\end{array}
\right.
\]
We say that $\widehat{y}$ is obtained by adding to $\widehat{x}$ the instant
$k^{\prime}$ of equality with $\mu.$ We notice that after adding to
$\widehat{x}$ the instant $k^{\prime}$ of equality with $\mu$ we loose the
periodicity of $\mu$, but we still have eventual periodicity with the same
sets of periods and with the limit of periodicity $k^{\prime}+1.$
\end{example}

\begin{example}
We take now the distinct time instants $k^{1},...,k^{s}\in\mathbf{N}%
_{\_}\setminus\widehat{\mathbf{T}}_{\mu}^{\widehat{x}}$ and we define
$\widehat{y}\in\widehat{S}^{(n)}$ by: $\forall k\in\mathbf{N}_{\_},$%
\[
\widehat{y}(k)=\left\{
\begin{array}
[c]{c}%
\widehat{x}(k),k\notin\{k^{1},...,k^{s}\},\\
\mu,\exists j\in\{1,...,s\},k=k^{j}.
\end{array}
\right.
\]
In this situation we say that $\widehat{y}$ is obtained from $\widehat{x}$ by
addition of the instants $k^{1},...,k^{s}$ of equality with $\mu.$ We get the
same sets of periods but eventual periodicity of $\mu~$only, with the limit of
periodicity $1+\max\{k^{1},...,k^{s}\}.$
\end{example}

\begin{example}
We consider the sequence of distinct time instants $k^{j}\in\mathbf{N}%
_{\_}\setminus\widehat{\mathbf{T}}_{\mu}^{\widehat{x}},j\in\mathbf{N}_{\_} $.
We define $\widehat{y}\in\widehat{S}^{(n)}$ by $\forall k\in\mathbf{N}_{\_},$%
\[
\widehat{y}(k)=\left\{
\begin{array}
[c]{c}%
\widehat{x}(k),\forall j\in\mathbf{N}_{\_},k\neq k^{j},\\
\mu,\exists j\in\mathbf{N}_{\_},k=k^{j}%
\end{array}
\right.
\]
meaning that we have constructed $\widehat{y}$ by addition to $\widehat{x}$ of
the instants $(k^{j})$ of equality with $\mu.$ Several possibilities exist in
this construction, we point out the following two situations only, given by%
\[
(k^{-1},k^{0},k^{1},...)=(k_{0}+2,k_{0}+2+\widetilde{p},k_{0}+2+2\widetilde
{p},...),k_{0}+2\in\{-1,0,...,\widetilde{p}-2\},
\]
when, after the passage from $\widehat{x}$ to $\widehat{y}$, the point $\mu$
is still periodic and

i) it keeps its prime period if $\widetilde{p}=5,k_{0}=-1;$ during a period
interval, by adding instants of equality with $\mu,$ instead of one occurrence
of $\mu$ at $k_{0}+k5,$ we have two occurrences, $k_{0}+k5$ and $k_{0}%
+2+k5,k\in\mathbf{N};$

ii) it doubles its prime period if $\widetilde{p}=4,k_{0}=-1,$ with one
occurrence of $\mu$ only during a period interval, at $k_{0}+k2,k\in
\mathbf{N}.$
\end{example}

\section{Real time, periodic points}

\begin{remark}
In this Section $x\in S^{(n)},$ $t_{0},t_{1}\in\mathbf{R,}$ $\widetilde{T}>0 $
and the periodic point $\mu\in Or(x)$ are given, such that%
\[
t_{0}<t_{1}<t_{0}+\widetilde{T},
\]%
\[
\mathbf{T}_{\mu}^{x}=(-\infty,t_{0})\cup\lbrack t_{1},t_{0}+\widetilde{T}%
)\cup\lbrack t_{1}+\widetilde{T},t_{0}+2\widetilde{T})\cup\lbrack
t_{1}+2\widetilde{T},t_{0}+3\widetilde{T})\cup...
\]%
\[
P_{\mu}^{x}=\{\widetilde{T},2\widetilde{T},3\widetilde{T},...\}.
\]
We notice that $\mu=x(-\infty+0)$ and we are in the special case of
periodicity from Theorem \ref{The75}, page \pageref{The75} but a different
choice of $\mu~$or of $\mathbf{T}_{\mu}^{x}$ does not change things significantly.
\end{remark}

\begin{definition}
We define the function $t\longmapsto\underline{t}$ that associates to each
real number $t\in\mathbf{R}$ an interval $\underline{t}\subset\mathbf{R}$ in
the following way:%
\[
\underline{t}=\left\{
\begin{array}
[c]{c}%
I^{x},\text{ if }t\in I^{x},\\
\lbrack a,b),\text{ if }t\in\lbrack a,b),[a,b)\subset\mathbf{T}_{x(t)}%
^{x},x(a-0)\neq x(t),x(b)\neq x(t).
\end{array}
\right.
\]

\end{definition}

\begin{remark}
We notice that $\underline{t}$ is the greatest interval that contains $t$ and
where $x$ has the constant value $x(t)$.
\end{remark}

\begin{remark}
The definition of $\underline{t}$ is possible since $x$ is not constant; the
non constancy of $x$ is inferred from the form of $\mathbf{T}_{\mu}^{x}.$
\end{remark}

\begin{example}
Let now $\mu^{\prime}\in\mathbf{B}^{n}$ with $\mu^{\prime}\neq\mu$ and let
also the time instant $t^{\prime\prime}\in\mathbf{T}_{\mu}^{x}.$ We define the
signal%
\[
y(t)=\left\{
\begin{array}
[c]{c}%
x(t),t\notin\underline{t^{\prime\prime}},\\
\mu^{\prime},t\in\underline{t^{\prime\prime}}%
\end{array}
\right.  .
\]
We say that $y$ is obtained by removing from $x$ the interval $\underline
{t^{\prime\prime}}$ of periodicity of $\mu.$ After the removal of
$\underline{t^{\prime\prime}},$ the periodicity of $\mu$ is lost, but eventual
periodicity still holds; the set of the periods is the same $P_{\mu}%
^{y}=P_{\mu}^{x}$ and the limit of periodicity is $\sup$ $\underline
{t^{\prime\prime}}$.
\end{example}

\begin{example}
The points $\mu^{1},...,\mu^{s}\in\mathbf{B}^{n}$ are taken and they are not
required to be distinct, but we ask that they are distinct from $\mu:\mu
^{1}\neq\mu,...,\mu^{s}\neq\mu$ and we also take the time instants
$t_{1}^{\prime},...,t_{s}^{\prime}\in\mathbf{T}_{\mu}^{x}$ with the property
that the intervals $\underline{t_{1}^{\prime}},...,\underline{t_{s}^{\prime}}$
are disjoint. We define%
\[
y(t)=\left\{
\begin{array}
[c]{c}%
x(t),t\notin\underline{t_{1}^{\prime}}\cup...\cup\underline{t_{s}^{\prime},}\\
\mu^{j},\exists j\in\{1,...,s\},t\in\underline{t_{j}^{\prime}}.
\end{array}
\right.
\]
Obviously the phenomenon is the same, $\mu$ is not periodic any longer, but it
is eventually periodic with $P_{\mu}^{y}=P_{\mu}^{x}$ and the limit of
periodicity is $max\{\sup$ $\underline{t_{1}^{\prime}},...,\sup\underline
{t_{s}^{\prime}}\}.$
\end{example}

\begin{example}
We consider the sequence $\mu^{j}\in\mathbf{B}^{n}\setminus\{\mu
\},j\in\mathbf{N}$ and also the time instants $t_{j}^{\prime}\in
\mathbf{T}_{\mu}^{x},j\in\mathbf{N}$ having the property that $\underline
{t_{0}^{\prime}},\underline{t_{1}^{\prime}},\underline{t_{2}^{\prime}},...$
are disjoint. We define%
\[
y(t)=\left\{
\begin{array}
[c]{c}%
x(t),t\notin\underline{t_{0}^{\prime}}\cup t_{1}^{\prime}\cup t_{2}^{\prime
}\cup...\\
\mu^{j},\exists j\in\mathbf{N},t\in\underline{t_{j}^{\prime}}%
\end{array}
\right.  .
\]
Similarly with Example \ref{Exa9}, page \pageref{Exa9} several possibilities
may occur, for example the periodic point $\mu$ gives birth to the periodic
points $\mu,\mu^{\prime},\mu^{\prime\prime}$ with%
\[
P_{\mu}^{y}=P_{\mu\prime}^{y}=P_{\mu^{\prime\prime}}^{y}=\{3\widetilde
{T},6\widetilde{T},9\widetilde{T},...\},
\]%
\[
\mathbf{T}_{\mu}^{y}=(-\infty,t_{0})\cup\lbrack t_{1}+2\widetilde{T}%
,t_{0}+3\widetilde{T})\cup\lbrack t_{1}+5\widetilde{T},t_{0}+6\widetilde
{T})\cup\lbrack t_{1}+8\widetilde{T},t_{0}+9\widetilde{T})\cup...
\]%
\[
\mathbf{T}_{\mu^{\prime}}^{y}=[t_{1},t_{0}+\widetilde{T})\cup\lbrack
t_{1}+3\widetilde{T},t_{0}+4\widetilde{T})\cup\lbrack t_{1}+6\widetilde
{T},t_{0}+7\widetilde{T})\cup...
\]%
\[
\mathbf{T}_{\mu^{\prime\prime}}^{y}=[t_{1}+\widetilde{T},t_{0}+2\widetilde
{T})\cup\lbrack t_{1}+4\widetilde{T},t_{0}+5\widetilde{T})\cup\lbrack
t_{1}+7\widetilde{T},t_{0}+8\widetilde{T})\cup...
\]

\end{example}

\begin{example}
Let us increase now the initial time instant $t_{0}$ to $t_{0}+\varepsilon,$
where $\varepsilon\in(0,t_{1}-t_{0}).$ We have:%
\[
\mathbf{T}_{\mu}^{y}=(-\infty,t_{0}+\varepsilon)\cup\lbrack t_{1}%
,t_{0}+\widetilde{T})\cup\lbrack t_{1}+\widetilde{T},t_{0}+2\widetilde{T}%
)\cup\lbrack t_{1}+2\widetilde{T},t_{0}+3\widetilde{T})\cup...
\]
The periodicity of $\mu$ has become eventual periodicity, the two sets of
periods are equal $P_{\mu}^{y}=P_{\mu}^{x}$ and the limit of periodicity is
$t_{0}+\varepsilon.$

If this increase with $\varepsilon\in(0,t_{1}-t_{0})$ is applied at the time
instant $t_{0}+k\widetilde{T},$ then%
\[
\mathbf{T}_{\mu}^{y}=(-\infty,t_{0})\cup\lbrack t_{1},t_{0}+\widetilde{T}%
)\cup...\cup\lbrack t_{1}+(k-1)\widetilde{T},t_{0}+k\widetilde{T}%
+\varepsilon)\cup\lbrack t_{1}+k\widetilde{T},t_{0}+(k+1)\widetilde{T}%
)\cup...,
\]
$P_{\mu}^{y}=P_{\mu}^{x}$ and the limit of periodicity is $t_{0}%
+k\widetilde{T}+\varepsilon.$ In this construction we say that we have added
the intervals $[t_{0},t_{0}+\varepsilon),[t_{0}+k\widetilde{T},t_{0}%
+k\widetilde{T}+\varepsilon)$ of equality with $\mu.$
\end{example}

\begin{example}
Let $\varepsilon>0$ be arbitrary and we decrease the initial time instant from
$t_{0}$ to $t_{0}-\varepsilon.$ Then%
\[
\mathbf{T}_{\mu}^{y}=(-\infty,t_{0}-\varepsilon)\cup\lbrack t_{1}%
,t_{0}+\widetilde{T})\cup\lbrack t_{1}+\widetilde{T},t_{0}+2\widetilde{T}%
)\cup\lbrack t_{1}+2\widetilde{T},t_{0}+3\widetilde{T})\cup...
\]
$\mu~$is eventually periodic, $P_{\mu}^{y}=P_{\mu}^{x}$ and the limit of
periodicity is $t_{0}.$ If we decrease however $t_{0}+k\widetilde{T}$ to
$t_{0}+k\widetilde{T}-\varepsilon,$ with $\varepsilon\in(0,t_{0}%
-t_{1}+\widetilde{T})$, we see that%
\[
\mathbf{T}_{\mu}^{y}=(-\infty,t_{0})\cup\lbrack t_{1},t_{0}+\widetilde{T}%
)\cup...\cup\lbrack t_{1}+(k-1)\widetilde{T},t_{0}+k\widetilde{T}%
-\varepsilon)\cup\lbrack t_{1}+k\widetilde{T},t_{0}+(k+1)\widetilde{T}%
)\cup...
\]
$\mu$ is eventually periodic, $P_{\mu}^{y}=P_{\mu}^{x}$ and the limit of
periodicity is $t_{0}+k\widetilde{T}.$ We have removed from $x$ the time
intervals $[t_{0}-\varepsilon,t_{0}),[t_{0}+k\widetilde{T}-\varepsilon
,t_{0}+k\widetilde{T})$ of equality with $\mu.$
\end{example}

\backmatter\appendix

\chapter{Notations}

$\mathbf{B},$ Notation \ref{Not1}, page \pageref{Not1}

$\mathbf{N},$ $\mathbf{Z},$ $\mathbf{R}$, $\mathbf{N}_{\_},$ Notation
\ref{Not2}, page \pageref{Not2}

$\widehat{Seq},Seq,$ Notation \ref{Not3}, page \pageref{Not3}

$\chi_{A},$ Notation \ref{Not4}, page \pageref{Not4}

$\widehat{S}^{(n)},S^{(n)},$ Notation \ref{Def2}, page \pageref{Def2}

$x(t-0),x(t+0),$ Notation \ref{Def8}, page \pageref{Def8}

$x(-\infty+0),\underset{t\rightarrow-\infty}{\lim}x(t),$ Notation \ref{Not14},
page \pageref{Not14}

$I^{x},$ Notation \ref{Not15}, page \pageref{Not15}

$\widehat{x}(\infty-0),\underset{k\rightarrow\infty}{\lim}\widehat
{x}(k),x(\infty-0),\underset{t\rightarrow\infty}{\lim}x(t),$ Notation
\ref{Not16}, page \pageref{Not16}

$\widehat{F}^{\widehat{x}},F^{x},$ Notation \ref{Not17}, page \pageref{Not17}

$\widehat{\sigma}^{k^{\prime}}$,$\sigma^{t^{\prime}},$ Definition \ref{Def26},
page \pageref{Def26}

$\widehat{Or}(\widehat{x}),Or(x),$ Definition \ref{Def9}, page \pageref{Def9}

$\widehat{\omega}(\widehat{x}),\omega(x),$ Definition \ref{Def10_}, page
\pageref{Def10_}

$\widehat{\mathbf{T}}_{\mu}^{\widehat{x}}$,$\mathbf{T}_{\mu}^{x},$ Notation
\ref{Not5}, page \pageref{Not5}

$\widehat{P}_{\mu}^{\widehat{x}},P_{\mu}^{x},$ Notation \ref{Not12}, page
\pageref{Not12}

$\widehat{L}_{\mu}^{\widehat{x}},L_{\mu}^{x},$ Notation \ref{Not18}, page
\pageref{Not18}

$\widehat{P}^{\widehat{x}}$,$P^{x},$ Notation \ref{Not19}, page
\pageref{Not19}

$\widehat{L}^{\widehat{x}},L^{x},$ Notation \ref{Not20}, page \pageref{Not20}

$\widehat{\Pi}_{n}^{\prime},\Pi_{n}^{\prime},$ Definition \ref{Def38}, page
\pageref{Def38}

$I^{\rho},$ Definition \ref{Def40}, page \pageref{Def40}

$\widehat{Seq^{\prime}},$ Notation \ref{Not8}, page \pageref{Not8}

$\widehat{Or}(\alpha)$,$Or(\rho),$ Definition \ref{Def9_}, page
\pageref{Def9_}

$\widehat{\omega}(\alpha)$,$\omega(\rho),$ Definition \ref{Def41}, page
\pageref{Def41}

$\widehat{P}^{\alpha},P^{\rho},$ Notation \ref{Not21}, page \pageref{Not21}

$\widehat{L}^{\alpha},L^{\rho},$ Notation \ref{Not22}, page \pageref{Not22}

$\widehat{\sigma}^{k^{\prime}},\sigma^{t^{\prime}},$ Definition \ref{Def27},
page \pageref{Def27}

$\widehat{\Pi}_{n},\Pi_{n},$ Definition \ref{Def42}, page \pageref{Def42}

$\Phi^{\lambda},$ Definition \ref{Def41}, page \pageref{Def41}

$\Phi^{\alpha^{0}...\alpha^{k}},$ Definition \ref{Def12}, page \pageref{Def12}

$\widehat{\Phi}^{\alpha}(\mu,\cdot)$,$\Phi^{\rho}(\mu,\cdot),$ Definition
\ref{Def13}, page \pageref{Def13}

$I_{\mu}^{\rho},$ Notation \ref{Not13}, page \pageref{Not13}

$\widehat{Or}^{\alpha}(\mu),Or^{\rho}(\mu),$ Notation \ref{Not9}, page
\pageref{Not9}

$\widehat{\omega}^{\alpha}(\mu),$ $\omega^{\rho}(\mu),$ Notation \ref{Not10},
page \pageref{Not10}

$\widehat{\mathbf{T}}_{\mu,\mu^{\prime}}^{\alpha},\mathbf{T}_{\mu,\mu^{\prime
}}^{\rho},$ Notation \ref{Not11}, page \pageref{Not11}

$supp$ $\alpha,supp$ $\rho,$ Notation \ref{Not23}, page \pageref{Not23}

$\Psi^{\lambda},$ Definition \ref{Def43}, page \pageref{Def43}

$\widehat{\Psi}^{-\alpha}(\mu,\widehat{u},k),\widehat{\Psi}^{\alpha}%
(\mu,\widehat{u},k)$,$\Psi^{-\rho}(\mu,u,t),\Psi^{\rho}(\mu,u,t),$ Definition
\ref{Def36}, page \pageref{Def36}

$P^{\ast}(H),$ Notation \ref{Not24}, page \pageref{Not24}

$\widehat{\Xi}_{\Psi}$,$\Xi_{\Psi},$ Definition \ref{Def31}, page
\pageref{Def31}

$\widehat{i}_{0}(\widehat{u})$,$i_{0}(u),$ Definition \ref{Def32}, page
\pageref{Def32}

$\widehat{P}_{\mu,\mu^{\prime}}^{\alpha},$ $P_{\mu,\mu^{\prime}}^{\rho},$
Notation \ref{Not25}, page \pageref{Not25}

$\widehat{P}_{\mu}^{\alpha},P_{\mu}^{\rho},$ Notation \ref{Not26}, page
\pageref{Not26}

\chapter{Index}

\textbf{A}

asynchronous system, Definition \ref{Def30}, page \pageref{Def30}

autonomous system, Definition \ref{Def35}, page \pageref{Def35}

\textbf{C}

computation function, Definition \ref{Def38}, page \pageref{Def38}

\textbf{D}

deterministic system, Definition \ref{Def34}, page \pageref{Def34}

double eventually periodic flow, Definition \ref{Def5_}, page \pageref{Def5_}

double periodic flow, Definition \ref{Def5}, page \pageref{Def5}

\textbf{E}

equilibrium point (of a flow), Definition \ref{Def53}, page \pageref{Def53}

equivalent computation functions, Definition \ref{Def51}, page \pageref{Def51}

eventually constant signal, Definition \ref{Def15}, page \pageref{Def15}

eventually equilibrium point (of a flow), Definition \ref{Def4}, page
\pageref{Def4}

eventually fixed point (of a Boolean function), Definition \ref{Def4}, page
\pageref{Def4}

eventually periodic computation function, Definition \ref{Def25}, page
\pageref{Def25}

eventually periodic point, Definition \ref{Def19}, page \pageref{Def19}

eventually periodic signal, Definition \ref{Def28}, page \pageref{Def28}

eventually rest point (of a flow), Definition \ref{Def4}, page \pageref{Def4}

evolution function, Definition \ref{Def13}, page \pageref{Def13}

evolution function, Definition \ref{Def36}, page \pageref{Def36}

\textbf{F}

final time instant (of a signal), Definition \ref{Def44}, page \pageref{Def44}

final value (of a signal), Definition \ref{Def7}, page \pageref{Def7}

fixed point (of a Boolean function), Definition \ref{Def53}, page
\pageref{Def53}

flow, Definition \ref{Def50}, page \pageref{Def50} and Definition \ref{Def52},
page \pageref{Def52}

forgetful function, Definition \ref{Def26}, page \pageref{Def26}

forgetful function, Definition \ref{Def27}, page \pageref{Def27}

\textbf{G}

generator function (of a system), Definition \ref{Def33}, page \pageref{Def33}

\textbf{H}

hypothesis $P$, Definition \ref{Def48}, page \pageref{Def48}

\textbf{I}

initial value (of a computation function), Definition \ref{Def39}, page
\pageref{Def39}

initial value (of a signal), Definition \ref{Def6}, page \pageref{Def6}

initial state function (of a system), Definition \ref{Def32}, page
\pageref{Def32}

initial time instant (of a computation function), Definition \ref{Def40}, page
\pageref{Def40}

initial time instant (of a signal), Definition \ref{Def29}, page
\pageref{Def29}

initialized system, Definition \ref{Def54}, page \pageref{Def54}

\textbf{L}

(left) limit function (of a real time signal), Definition \ref{Def8}, page
\pageref{Def8}

(right) limit function (of a real time signal), Definition \ref{Def8}, page
\pageref{Def8}

limit of constancy, Definition \ref{Def15}, page \pageref{Def15}

limit of equilibrium, Definition \ref{Def15}, page \pageref{Def15}

limit of periodicity (of a computation function), Definition \ref{Def25}, page
\pageref{Def25}

limit of periodicity (of a point), Definition \ref{Def19}, page
\pageref{Def19}

limit of periodicity (of a signal), Definition \ref{Def28}, page
\pageref{Def28}

\textbf{N}

next state function, Definition \ref{Def13}, page \pageref{Def13}

\textbf{O}

orbit (of a computation function), Definition \ref{Def9_}, page
\pageref{Def9_}

orbit (of a signal), Definition \ref{Def9}, page \pageref{Def9}

omega limit point, Definition \ref{Def10_}, page \pageref{Def10_}

omega limit set, Definition \ref{Def10_}, page \pageref{Def10_}

omega limit set (of a computation function), Definition \ref{Def41}, page
\pageref{Def41}

\textbf{P}

period (of a computation function), Definition \ref{Def25}, page
\pageref{Def25}

period (of a point), Definition \ref{Def19}, page \pageref{Def19}

period (of a signal), Definition \ref{Def28}, page \pageref{Def28}

periodic computation function, Definition \ref{Def49}, page \pageref{Def49}

periodic point, Definition \ref{Def47}, page \pageref{Def47}

periodic signal, Definition \ref{Def18}, page \pageref{Def18}

prime limit of periodicity (of a computation function), Definition
\ref{Def37}, page \pageref{Def37}

prime limit of periodicity (of a point), Definition \ref{Def45}, page
\pageref{Def45}

prime limit of periodicity (of a signal), Definition \ref{Def46}, page
\pageref{Def46}

prime period (of a computation function), Definition \ref{Def37}, page
\pageref{Def37}

prime period (of a point), Definition \ref{Def45}, page \pageref{Def45}

prime period (of a signal), Definition \ref{Def46}, page \pageref{Def46}

progressive computation function, Definition \ref{Def42}, page \pageref{Def42}

\textbf{R}

regular system, Definition \ref{Def33}, page \pageref{Def33}

rest point (of a flow), Definition \ref{Def53}, page \pageref{Def53}

\textbf{S}

semi-flow, Definition \ref{Def50}, page \pageref{Def50}

signal, Definition \ref{Def2}, page \pageref{Def2}

support set, Definition \ref{Not5}, page \pageref{Not5}

\textbf{T}

(state) transition function, Definition \ref{Def13}, page \pageref{Def13}

\textbf{U}

universal regular asynchronous system, Definition \ref{Def31}, page
\pageref{Def31}

\chapter{Lemmas}

The purpose of this Appendix is that of presenting results that are necessary
in the proofs of some Theorems. Several Lemmas are interesting by themselves too.

\begin{lemma}
\label{Lem37}Let the signals $\widehat{x}\in\widehat{S}^{(n)},x\in S^{(n)}$
and we consider the following statements:%
\begin{equation}
\mu\in\widehat{Or}(\widehat{x}), \label{p108}%
\end{equation}%
\begin{equation}
\mu\in\widehat{\omega}(\widehat{x}), \label{p109}%
\end{equation}%
\begin{equation}
\widehat{\mathbf{T}}_{\mu}^{\widehat{x}}\cap\{k^{\prime},k^{\prime
}+1,k^{\prime}+2,...\}\neq\varnothing, \label{p110}%
\end{equation}%
\begin{equation}
\forall k\in\widehat{\mathbf{T}}_{\mu}^{\widehat{x}}\cap\{k^{\prime}%
,k^{\prime}+1,k^{\prime}+2,...\},\{k+zp|z\in\mathbf{Z}\}\cap\{k^{\prime
},k^{\prime}+1,k^{\prime}+2,...\}\subset\widehat{\mathbf{T}}_{\mu}%
^{\widehat{x}}, \label{p111}%
\end{equation}%
\begin{equation}
\mu\in Or(x), \label{p112}%
\end{equation}%
\begin{equation}
\mu\in\omega(x), \label{p113}%
\end{equation}%
\begin{equation}
\mathbf{T}_{\mu}^{x}\cap\lbrack t^{\prime},\infty)\neq\varnothing,
\label{p114}%
\end{equation}%
\begin{equation}
\forall t\in\mathbf{T}_{\mu}^{x}\cap\lbrack t^{\prime},\infty),\{t+zT|z\in
\mathbf{Z}\}\cap\lbrack t^{\prime},\infty)\subset\mathbf{T}_{\mu}^{x},
\label{p115}%
\end{equation}
where $k^{\prime}\in\mathbf{N}_{\_}$ and $t^{\prime}\in\mathbf{R}.$ We have
the equivalencies:%
\begin{equation}
((\ref{p108})\text{ and }(\ref{p110})\text{ and }(\ref{p111}%
))\Longleftrightarrow((\ref{p109})\text{ and }(\ref{p111})), \label{p116}%
\end{equation}%
\begin{equation}
((\ref{p112})\text{ and }(\ref{p114})\text{ and }(\ref{p115}%
))\Longleftrightarrow((\ref{p113})\text{ and }(\ref{p115})). \label{p117}%
\end{equation}

\end{lemma}

\begin{proof}
If (\ref{p108}) and (\ref{p110}) and (\ref{p111}) hold, we can take some
$k\in\widehat{\mathbf{T}}_{\mu}^{\widehat{x}}\cap\{k^{\prime},k^{\prime
}+1,k^{\prime}+2,...\}$ arbitrarily. From (\ref{p111}) we get
$k,k+p,k+2p,...\in\widehat{\mathbf{T}}_{\mu}^{\widehat{x}},$ thus
$\widehat{\mathbf{T}}_{\mu}^{\widehat{x}}$ is infinite and $\mu\in
\widehat{\omega}(\widehat{x}).$

Conversely, we suppose that (\ref{p109}) and (\ref{p111}) hold. (\ref{p108})
is trivially fulfilled and as far as $\widehat{\mathbf{T}}_{\mu}^{\widehat{x}%
}$ is infinite, (\ref{p110}) is true also.

(\ref{p116}) is proved and the proof of (\ref{p117}) is similar.
\end{proof}

\begin{remark}
Equivalent forms of Lemma \ref{Lem30}, Lemma \ref{Lem28}, Lemma \ref{Lem25},
Lemma \ref{Lem10}, and Lemma \ref{Lem35} exist, due to Lemma \ref{Lem37}; we
can replace in their hypothesis $\mu\in\widehat{\omega}(\widehat{x})$ with
$\mu\in\widehat{Or}(\widehat{x}),$ $\widehat{\mathbf{T}}_{\mu}^{\widehat{x}%
}\cap\{k^{\prime},k^{\prime}+1,k^{\prime}+2,...\}\neq\varnothing$ and we can
also replace $\mu\in\omega(x)$ with $\mu\in Or(x),$ $\mathbf{T}_{\mu}^{x}%
\cap\lbrack t^{\prime},\infty)\neq\varnothing.$
\end{remark}

\begin{lemma}
\label{Lem36}Let $\mu\in Or(x)$ and $t^{\prime}\in I^{x}.$ Then $\mathbf{T}%
_{\mu}^{x}\cap\lbrack t^{\prime},\infty)\neq\varnothing.$
\end{lemma}

\begin{proof}
The hypothesis states that $(-\infty,t^{\prime}]\subset\mathbf{T}%
_{x(-\infty+0)}^{x}$ is true. If $\mu=x(-\infty+0),$ when $t^{\prime}%
\in\mathbf{T}_{\mu}^{x},$ we have $\mathbf{T}_{\mu}^{x}\cap\lbrack t^{\prime
},\infty)\neq\varnothing$ true. And if $\mu\neq x(-\infty+0),$ when
$\mathbf{T}_{\mu}^{x}\cap(-\infty,t^{\prime}]=\varnothing,\mathbf{T}_{\mu}%
^{x}\neq\varnothing,$ we get $\mathbf{T}_{\mu}^{x}\subset(t^{\prime},\infty),$
thus $\mathbf{T}_{\mu}^{x}\cap\lbrack t^{\prime},\infty)\neq\varnothing.$
\end{proof}

\begin{lemma}
\label{Lem30}a) $\widehat{x}\in\widehat{S}^{(n)}$ is given and we suppose that
$\mu\in\widehat{\omega}(\widehat{x})$ is eventually periodic with the period
$p\geq1$ and with the limit of periodicity $k^{\prime}\in\mathbf{N}_{\_}.$ If
$k^{\prime\prime}\geq k^{\prime},$ then $\mu$ is eventually periodic with the
period $p$ and with the limit of periodicity $k^{\prime\prime}.$

b) Let $x\in S^{(n)}$ and we suppose that $\mu\in\omega(x)$ is eventually
periodic with the period $T>0$ and with the limit of periodicity $t^{\prime
}\in\mathbf{R.}$ If $t^{\prime\prime}\geq t^{\prime},$ then $\mu$ is
eventually periodic with the period $T$ and with the limit of periodicity
$t^{\prime\prime}.$
\end{lemma}

\begin{proof}
b) The hypothesis states that%
\begin{equation}
\forall t\in\mathbf{T}_{\mu}^{x}\cap\lbrack t^{\prime},\infty),\{t+zT|z\in
\mathbf{Z}\}\cap\lbrack t^{\prime},\infty)\subset\mathbf{T}_{\mu}^{x}
\label{pre952}%
\end{equation}
is true and we must prove%
\begin{equation}
\forall t\in\mathbf{T}_{\mu}^{x}\cap\lbrack t^{\prime\prime},\infty
),\{t+zT|z\in\mathbf{Z}\}\cap\lbrack t^{\prime\prime},\infty)\subset
\mathbf{T}_{\mu}^{x} \label{pre953}%
\end{equation}
for an arbitrary $t^{\prime\prime}\geq t^{\prime}.$ Indeed, we take some
arbitrary $t\in\mathbf{T}_{\mu}^{x}\cap\lbrack t^{\prime\prime},\infty) $ and
$z\in\mathbf{Z}$ such that $t+zT\geq t^{\prime\prime}$ holds. Then
$t\in\mathbf{T}_{\mu}^{x}\cap\lbrack t^{\prime},\infty)$ and $t+zT\geq
t^{\prime}$ are true, thus we can apply (\ref{pre952}). We have obtained that
$t+zT\in\mathbf{T}_{\mu}^{x},$ i.e. (\ref{pre953}) is fulfilled.
\end{proof}

\begin{lemma}
\label{Lem28}a) Let $\widehat{x},\mu\in\widehat{\omega}(\widehat{x})$ that is
an eventually periodic point of $\widehat{x}$ with the period $p\geq1$ and the
limit of periodicity $k^{\prime}\in\mathbf{N}_{\_}$ and let also $k\in
\widehat{\mathbf{T}}_{\mu}^{\widehat{x}}\cap\{k^{\prime},k^{\prime
}+1,k^{\prime}+2,...\}.$ Then $\{k,k+p,k+2p,...\}\subset\widehat{\mathbf{T}%
}_{\mu}^{\widehat{x}}.$

b) Let $x\in S^{(n)},\mu\in\omega(x)$ that is eventually periodic with the
period $T>0$ and the limit of periodicity $t^{\prime}\in\mathbf{R}$ and we
suppose that $t_{1}<t_{2}$ fulfill $[t_{1},t_{2})\subset\mathbf{T}_{\mu}%
^{x}\cap\lbrack t^{\prime},\infty\dot{)}$. Then%
\begin{equation}
\lbrack t_{1},t_{2})\cup\lbrack t_{1}+T,t_{2}+T)\cup\lbrack t_{1}%
+2T,t_{2}+2T)\cup...\subset\mathbf{T}_{\mu}^{x}. \label{pre947}%
\end{equation}

\end{lemma}

\begin{proof}
b) We have the truth of%
\begin{equation}
\forall t\in\mathbf{T}_{\mu}^{x}\cap\lbrack t^{\prime},\infty),\{t+zT|z\in
\mathbf{Z}\}\cap\lbrack t^{\prime},\infty)\subset\mathbf{T}_{\mu}^{x}.
\label{pre946}%
\end{equation}
Let $k\in\mathbf{N}$ and $t\in\lbrack t_{1}+kT,t_{2}+kT)$ arbitrary. We infer
$t-kT\in\lbrack t_{1},t_{2}),$ thus%
\[
\mu=x(t-kT)\overset{(\ref{pre946})}{=}x(t),
\]
wherefrom $t\in\mathbf{T}_{\mu}^{x}.$ (\ref{pre947}) is proved.
\end{proof}

\begin{lemma}
\label{Lem25}a) $\widehat{x}\in\widehat{S}^{(n)},\mu\in\widehat{\omega
}(\widehat{x})$ are given with the property that $\mu$ is eventually periodic
with the period $p\geq1$ and the limit of periodicity $k^{\prime}\in
\mathbf{N}_{\_}.$ If $k_{1}\geq k^{\prime}$ and%
\begin{equation}
\widehat{x}(k_{1})\neq\mu, \label{p105}%
\end{equation}
then $\forall k\in\mathbf{N},$%
\begin{equation}
\widehat{x}(k_{1}+kp)\neq\mu. \label{p106}%
\end{equation}

b) We suppose that $x\in S^{(n)},$ $\mu\in\omega(x)$ are given and $\mu$ is
eventually periodic with the period $T>0$ and the limit of periodicity
$t^{\prime}\in\mathbf{R.}$ If $t_{1}>t^{\prime}$ and%
\begin{equation}
x(t_{1}-0)\neq\mu, \label{pre770}%
\end{equation}
then $\forall k\in\mathbf{N},$%
\begin{equation}
x(t_{1}+kT-0)\neq\mu; \label{pre772}%
\end{equation}
if $t_{2}\geq t^{\prime}$ and%
\begin{equation}
x(t_{2})\neq\mu, \label{pre771}%
\end{equation}
then $\forall k\in\mathbf{N},$%
\begin{equation}
x(t_{2}+kT)\neq\mu. \label{pre773}%
\end{equation}

\end{lemma}

\begin{proof}
a) The hypothesis states the truth of%
\begin{equation}
\forall k\in\widehat{\mathbf{T}}_{\mu}^{\widehat{x}}\cap\{k^{\prime}%
,k^{\prime}+1,k^{\prime}+2,...\},\{k+zp|z\in\mathbf{Z}\}\cap\{k^{\prime
},k^{\prime}+1,k^{\prime}+2,...\}\subset\widehat{\mathbf{T}}_{\mu}%
^{\widehat{x}}. \label{p107}%
\end{equation}
Let $k\in\mathbf{N}$ arbitrary and we suppose against all reason that
(\ref{p106}) is false. We obtain the contradiction:%
\[
\mu=\widehat{x}(k_{1}+kp)\overset{(\ref{p107})}{=}\widehat{x}(k_{1}%
)\overset{(\ref{p105})}{\neq}\mu.
\]

b) We have from the hypothesis that%
\begin{equation}
\forall t\in\mathbf{T}_{\mu}^{x}\cap\lbrack t^{\prime},\infty),\{t+zT|z\in
\mathbf{Z}\}\cap\lbrack t^{\prime},\infty)\subset\mathbf{T}_{\mu}^{x}
\label{pre769}%
\end{equation}
holds. Let $k\in\mathbf{N}$ arbitrary$.$ We get the existence of
$\varepsilon_{1}>0$ such that $t_{1}-\varepsilon_{1}\geq t^{\prime}$ and
\begin{equation}
\forall t\in(t_{1}-\varepsilon_{1},t_{1}),x(t)=x(t_{1}-0) \label{pre774}%
\end{equation}
and respectively the existence of $\varepsilon_{2}>0$ such that $t_{1}%
+kT-\varepsilon_{2}\geq t^{\prime}$ and%
\begin{equation}
\forall t\in(t_{1}+kT-\varepsilon_{2},t_{1}+kT),x(t)=x(t_{1}+kT-0).
\label{pre775}%
\end{equation}
We denote $\varepsilon=\min\{\varepsilon_{1},\varepsilon_{2}\}$ and we suppose
against all reason that $t^{\prime\prime}\in(t_{1}+kT-\varepsilon,t_{1}+kT)$
exists with $x(t^{\prime\prime})=\mu.$ We infer%
\begin{equation}
t^{\prime}\leq t_{1}+kT-\varepsilon_{2}\leq t_{1}+kT-\varepsilon
<t^{\prime\prime}<t_{1}+kT, \label{pre942}%
\end{equation}%
\begin{equation}
\mu=x(t^{\prime\prime})\overset{(\ref{pre775}),(\ref{pre942})}{=}%
x(t_{1}+kT-0), \label{pre944}%
\end{equation}%
\begin{equation}
t^{\prime}\leq t_{1}-\varepsilon_{1}\leq t_{1}-\varepsilon<t^{\prime\prime
}-kT<t_{1}, \label{pre943}%
\end{equation}%
\begin{equation}
\mu=x(t^{\prime\prime})\overset{(\ref{pre769}),(\ref{pre942}),(\ref{pre943}%
)}{=}x(t^{\prime\prime}-kT)\overset{(\ref{pre774}),(\ref{pre943})}{=}%
x(t_{1}-0)\overset{(\ref{pre770})}{\neq}\mu, \label{pre945}%
\end{equation}
contradiction. We have obtained that $\forall t\in(t_{1}+kT-\varepsilon
,t_{1}+kT),x(t)\neq\mu,$ i.e. (\ref{pre772}) holds.

Furthermore, if (\ref{pre773}) is false, against all reason, then we get%
\[
\mu=x(t_{2}+kT)\overset{(\ref{pre769})}{=}x(t_{2})\overset{(\ref{pre771}%
)}{\neq}\mu,
\]
contradiction. (\ref{pre773}) holds.
\end{proof}

\begin{remark}
Lemma \ref{Lem25} refers to eventually periodic points $\mu$ and makes a
weaker statement than the appropriate one of the eventually periodic signals.
We cannot draw the conclusion here, like at the eventually periodic signals,
that $\widehat{x}(k_{1})=\widehat{x}(k_{1}+kp),x(t_{1}-0)=x(t_{1}%
+kT-0),x(t_{2})=x(t_{2}+kT),$ but we can state that $\widehat{x}(k_{1})\neq
\mu,x(t_{1}-0)\neq\mu,x(t_{2})\neq\mu$ imply $\widehat{x}(k_{1}+kp)\neq
\mu,x(t_{1}+kT-0)\neq\mu,x(t_{2}+kT)\neq\mu.$
\end{remark}

\begin{lemma}
\label{Lem29}Let $t_{1}<t_{2},t_{1}^{\prime}<t_{2}^{\prime},T>0$ and
$T^{\prime}\in(0,t_{2}-t_{1}).$ Then
\[
([t_{1},t_{2})\cup\lbrack t_{1}+T,t_{2}+T)\cup\lbrack t_{1}+2T,t_{2}%
+2T)\cup...)\cap
\]%
\[
\cap([t_{1}^{\prime},t_{2}^{\prime})\cup\lbrack t_{1}^{\prime}+T^{\prime
},t_{2}^{\prime}+T^{\prime})\cup\lbrack t_{1}^{\prime}+2T^{\prime}%
,t_{2}^{\prime}+2T^{\prime})\cup...)\neq\varnothing.
\]

\end{lemma}

\begin{proof}
Let $k_{1}\in\mathbf{N}$ such that $t_{1}+k_{1}T>t_{1}^{\prime}.$ In the
sequence $t_{1}^{\prime},t_{1}^{\prime}+T^{\prime},t_{1}^{\prime}+2T^{\prime
},...$ some $k_{2}\in\mathbf{N}$ exists with%
\begin{equation}
t_{1}^{\prime}+k_{2}T^{\prime}<t_{1}+k_{1}T, \label{pre948}%
\end{equation}%
\begin{equation}
t_{1}^{\prime}+(k_{2}+1)T^{\prime}\geq t_{1}+k_{1}T. \label{pre949}%
\end{equation}
We get from here that%
\begin{equation}
t_{1}^{\prime}+(k_{2}+1)T^{\prime}\overset{(\ref{pre948})}{<}t_{1}%
+k_{1}T+T^{\prime}<t_{1}+k_{1}T+t_{2}-t_{1}=t_{2}+k_{1}T. \label{pre950}%
\end{equation}
From (\ref{pre949}) and (\ref{pre950}) we infer that $t_{1}^{\prime}%
+(k_{2}+1)T^{\prime}\in\lbrack t_{1}+k_{1}T,t_{2}+k_{1}T).$
\end{proof}

\begin{lemma}
\label{Lem38}a) Let $\widehat{x}$ that is not eventually constant and $\mu
\in\widehat{\omega}(\widehat{x}).$ Then%
\[
\forall k\in\mathbf{N},\exists k^{\prime}>k,\widehat{x}(k^{\prime}%
-1)\neq\widehat{x}(k^{\prime})=\mu.
\]

b) We suppose that $x$ is not eventually constant and we take $\mu\in
\omega(x).$ We have%
\[
\forall t\in\mathbf{R},\exists t^{\prime}>t,x(t^{\prime}-0)\neq x(t^{\prime
})=\mu.
\]

\end{lemma}

\begin{proof}
a) We suppose that $\widehat{\omega}(\widehat{x})=\{\mu^{1},...,\mu
^{s}\},s\geq2,$ that $\mu=\mu^{1}$ and let $k\in\mathbf{N}_{\_}$ arbitrary.
$\widehat{\mathbf{T}}_{\mu^{1}}^{\widehat{x}},...,\widehat{\mathbf{T}}%
_{\mu^{s}}^{\widehat{x}}$ are all infinite and we define%
\[
k_{1}=\min(\widehat{\mathbf{T}}_{\mu^{2}}^{\widehat{x}}\cup...\cup
\widehat{\mathbf{T}}_{\mu^{s}}^{\widehat{x}})\cap\{k,k+1,k+2,...\},
\]%
\[
k^{\prime}=\min\widehat{\mathbf{T}}_{\mu^{1}}^{\widehat{x}}\cap\{k_{1}%
,k_{1}+1,k_{1}+2,...\}.
\]
We have $k^{\prime}>k_{1}\geq k$ and%
\[
\{\mu^{2},...,\mu^{s}\}\ni\widehat{x}(k^{\prime}-1)\neq\widehat{x}(k^{\prime
})=\mu^{1}.
\]

b) We put $\omega(x)$ under the form $\omega(x)=\{\mu^{1},...,\mu^{2}\},$
where $s\geq2$ and $\mu=\mu^{1}.$ Let $t\in\mathbf{R}$ arbitrary. The support
sets $\mathbf{T}_{\mu^{1}}^{x},...,\mathbf{T}_{\mu^{s}}^{x}$ are all
superiorly unbounded and we define:%
\[
t_{1}=\min(\mathbf{T}_{\mu^{2}}^{x}\cup...\cup\mathbf{T}_{\mu^{s}}^{x}%
)\cap\lbrack t,\infty),
\]%
\[
t^{\prime}=\min\mathbf{T}_{\mu^{1}}^{x}\cap\lbrack t_{1},\infty).
\]
The sets $(\mathbf{T}_{\mu^{2}}^{x}\cup...\cup\mathbf{T}_{\mu^{s}}^{x}%
)\cap\lbrack t,\infty),\mathbf{T}_{\mu^{1}}^{x}\cap\lbrack t_{1},\infty)$ are
of the form $[a,b)\cup\lbrack c,d)\cup...$, thus their minimum exists. We have
$t^{\prime}>t_{1}\geq t$ and%
\[
\{\mu^{2},...,\mu^{s}\}\ni x(t^{\prime}-0)\neq x(t^{\prime})=\mu^{1}.
\]

\end{proof}

\begin{lemma}
\label{Lem10}Let $x\in S^{(n)},\mu\in\omega(x)$ be an eventually periodic
point with the period $T>0$ and the limit of periodicity $t^{\prime}%
\in\mathbf{R.}$ For any $t_{1}>t^{\prime},$%
\begin{equation}
x(t_{1}-0)\neq x(t_{1})=\mu\label{per937}%
\end{equation}
implies%
\begin{equation}
x(t_{1}+T-0)\neq x(t_{1}+T)=\mu. \label{per938}%
\end{equation}

\end{lemma}

\begin{proof}
This is a special case of Lemma \ref{Lem25} b) when $x(t_{1})=\mu$ and $k=1. $
\end{proof}

\begin{lemma}
\label{Lem4_}Let $x\in S^{(n)}$ and the sequence $T_{k}\in\mathbf{R}%
,k\in\mathbf{N}$ that is strictly decreasingly convergent to $T\in\mathbf{R.}$
Then $\exists N\in\mathbf{N},\forall k\geq N,$%
\begin{equation}
x(T_{k}-0)=x(T_{k})=x(T). \label{per156}%
\end{equation}

\end{lemma}

\begin{proof}
Some $\delta>0$ exists with the property that%
\begin{equation}
\forall\xi\in\lbrack T,T+\delta),x(\xi)=x(T). \label{per157}%
\end{equation}
As $T_{k}\rightarrow T$ strictly decreasingly, $N_{\delta}\in\mathbf{N}$
exists such that%
\begin{equation}
\forall k\geq N_{\delta},T<T_{k}<T+\delta. \label{per158}%
\end{equation}
We fix an arbitrary $k\geq N_{\delta}.$ If we take $\varepsilon\in(0,T_{k}-T)$
arbitrary also$,$ we have%
\begin{equation}
T-T_{k}<-\varepsilon<0. \label{per159}%
\end{equation}
We add $T_{k}$ to the terms of (\ref{per159}) and we obtain, taking into
account (\ref{per158}) too:%
\begin{equation}
T<T_{k}-\varepsilon<T_{k}<T+\delta. \label{per160}%
\end{equation}
We conclude on one hand that%
\[
\forall\xi\in(T_{k}-\varepsilon,T_{k}),x(\xi)\overset{(\ref{per157}%
),(\ref{per160})}{=}x(T),
\]
thus%
\begin{equation}
x(T_{k}-0)=x(T) \label{per161}%
\end{equation}
and on the other hand that%
\begin{equation}
x(T_{k})\overset{(\ref{per157}),(\ref{per160})}{=}x(T). \label{per162}%
\end{equation}
By comparing (\ref{per161}) with (\ref{per162}) we infer (\ref{per156}).
\end{proof}

\begin{remark}
In Lemma \ref{Lem4_} $T_{k}>0,k\in\mathbf{N}$ and $T\geq0$ are not necessarily
related with any property of periodicity of $x$.
\end{remark}

\begin{lemma}
\label{Lem35}a) We consider $\widehat{x},$ $p\geq1,$ $k^{\prime}\in
\mathbf{N}_{\_}$ and $\mu\in\widehat{\omega}(\widehat{x})$ such that%
\begin{equation}
\forall k\in\widehat{\mathbf{T}}_{\mu}^{\widehat{x}}\cap\{k^{\prime}%
,k^{\prime}+1,k^{\prime}+2,...\},\{k+zp|z\in\mathbf{Z}\}\cap\{k^{\prime
},k^{\prime}+1,k^{\prime}+2,...\}\subset\widehat{\mathbf{T}}_{\mu}%
^{\widehat{x}} \label{p37}%
\end{equation}
holds. We define $n_{1},n_{2},...,n_{k_{1}}\in\mathbf{N}_{\_},k_{1}\geq1$ by%
\begin{equation}
\{n_{1},n_{2},...,n_{k_{1}}\}=\widehat{\mathbf{T}}_{\mu}^{\widehat{x}}%
\cap\{k^{\prime},k^{\prime}+1,...,k^{\prime}+p-1\}. \label{p38}%
\end{equation}
For any $k^{\prime\prime}\geq k^{\prime},$ with $n_{1}^{\prime},n_{2}^{\prime
},...,n_{p_{1}}^{\prime}\in\mathbf{N}_{\_},p_{1}\geq1$ defined by%
\begin{equation}
\{n_{1}^{\prime},n_{2}^{\prime},...,n_{p_{1}}^{\prime}\}=\widehat{\mathbf{T}%
}_{\mu}^{\widehat{x}}\cap\{k^{\prime\prime},k^{\prime\prime}+1,...,k^{\prime
\prime}+p-1\}, \label{p39}%
\end{equation}
we have $k_{1}=p_{1}$ and
\[%
\begin{array}
[c]{c}%
\underset{k\in\mathbf{N}}{%
{\displaystyle\bigcup}
}\{n_{1}+kp,n_{2}+kp,...,n_{k_{1}}+kp\}=\\
=\underset{z\in\mathbf{Z}}{%
{\displaystyle\bigcup}
}\{n_{1}^{\prime}+zp,n_{2}^{\prime}+zp,...,n_{k_{1}}^{\prime}+zp\}\cap
\{k^{\prime},k^{\prime}+1,k^{\prime}+2,...\}.
\end{array}
\]

b) Let $x,T>0,t^{\prime}\in\mathbf{R}$ and $\mu\in\omega(x)$ with%
\begin{equation}
\forall t\in\mathbf{T}_{\mu}^{x}\cap\lbrack t^{\prime},\infty),\{t+zT|z\in
\mathbf{Z}\}\cap\lbrack t^{\prime},\infty)\subset\mathbf{T}_{\mu}^{x}.
\label{p44}%
\end{equation}
The disjoint intervals $[a_{1},b_{1}),[a_{2},b_{2}),...,[a_{k_{1}},b_{k_{1}%
}),$ $k_{1}\geq1$ are defined by%
\begin{equation}
\lbrack a_{1},b_{1})\cup\lbrack a_{2},b_{2})\cup...\cup\lbrack a_{k_{1}%
},b_{k_{1}})=\mathbf{T}_{\mu}^{x}\cap\lbrack t^{\prime},t^{\prime}+T).
\label{p45}%
\end{equation}
For any $t^{\prime\prime}\geq t^{\prime},$ we define the disjoint intervals
$[a_{1}^{\prime},b_{1}^{\prime}),[a_{2}^{\prime},b_{2}^{\prime}),...,[a_{p_{1}%
}^{\prime},b_{p_{1}}^{\prime}),$ $p_{1}\geq1,$ by%
\begin{equation}
\lbrack a_{1}^{\prime},b_{1}^{\prime})\cup\lbrack a_{2}^{\prime},b_{2}%
^{\prime})\cup...\cup\lbrack a_{p_{1}}^{\prime},b_{p_{1}}^{\prime}%
)=\mathbf{T}_{\mu}^{x}\cap\lbrack t^{\prime\prime},t^{\prime\prime}+T).
\label{p46}%
\end{equation}
Then we have%
\[%
\begin{array}
[c]{c}%
\underset{k\in\mathbf{N}}{%
{\displaystyle\bigcup}
}([a_{1}+kT,b_{1}+kT)\cup\lbrack a_{2}+kT,b_{2}+kT)\cup...\cup\lbrack
a_{k_{1}}+kT,b_{k_{1}}+kT))=\\
=\underset{z\in\mathbf{Z}}{%
{\displaystyle\bigcup}
}([a_{1}^{\prime}+zT,b_{1}^{\prime}+zT)\cup\lbrack a_{2}^{\prime}%
+zT,b_{2}^{\prime}+zT)\cup...\cup\lbrack a_{p_{1}}^{\prime}+zT,b_{p_{1}%
}^{\prime}+zT))\cap\lbrack t^{\prime},\infty).
\end{array}
\]

\end{lemma}

\begin{proof}
a) Let $k^{\prime\prime}\geq k^{\prime}$ arbitrary. As $\mu\in\widehat{\omega
}(\widehat{x}),$ we get that $\widehat{\mathbf{T}}_{\mu}^{\widehat{x}}$ is
infinite, thus $\widehat{\mathbf{T}}_{\mu}^{\widehat{x}}\cap\{k^{\prime
},k^{\prime}+1,k^{\prime}+2,...\}\neq\varnothing$ and we can apply Theorem
\ref{Lem1}, page \pageref{Lem1}, wherefrom $\widehat{\mathbf{T}}_{\mu
}^{\widehat{x}}\cap\{k^{\prime},k^{\prime}+1,...,k^{\prime}+p-1\}\neq
\varnothing,$ $\widehat{\mathbf{T}}_{\mu}^{\widehat{x}}\cap\{k^{\prime\prime
},k^{\prime\prime}+1,...,k^{\prime\prime}+p-1\}\neq\varnothing$ hence the
definitions (\ref{p38}), (\ref{p39}) of $n_{1},n_{2},...,n_{k_{1}}$ and
$n_{1}^{\prime},n_{2}^{\prime},...,n_{p_{1}}^{\prime}$ make sense.

Let $j\in\{1,...,k_{1}\}$ arbitrary. We claim that exactly one term of the
sequence $n_{j},n_{j}+p,n_{j}+2p,...$ belongs to $\widehat{\mathbf{T}}_{\mu
}^{\widehat{x}}\cap\{k^{\prime\prime},k^{\prime\prime}+1,...,k^{\prime\prime
}+p-1\}.$ Indeed, let us suppose against all reason that no term belongs to
$\widehat{\mathbf{T}}_{\mu}^{\widehat{x}}\cap\{k^{\prime\prime},k^{\prime
\prime}+1,...,k^{\prime\prime}+p-1\}.$ As (\ref{p37}) implies $\forall
k\in\mathbf{N},n_{j}+kp\in\widehat{\mathbf{T}}_{\mu}^{\widehat{x}},$ we infer
the existence of $k$ having the property that%
\[
n_{j}+kp\leq k^{\prime\prime}-1,
\]%
\[
n_{j}+(k+1)p\geq k^{\prime\prime}+p.
\]
We infer the contradiction
\[
k^{\prime\prime}+p\leq n_{j}+(k+1)p\leq k^{\prime\prime}-1+p.
\]

We suppose against all reason that several terms of the sequence $n_{j}%
,n_{j}+p,n_{j}+2p,...$ belong to $\widehat{\mathbf{T}}_{\mu}^{\widehat{x}}%
\cap\{k^{\prime\prime},k^{\prime\prime}+1,...,k^{\prime\prime}+p-1\}.$ This
fact implies the existence of $k\in\mathbf{N}$ with%
\[
n_{j}+kp\geq k^{\prime\prime},
\]%
\[
n_{j}+(k+1)p\leq k^{\prime\prime}+p-1,
\]
wherefrom we get the contradiction
\[
k^{\prime\prime}+p\leq n_{j}+(k+1)p\leq k^{\prime\prime}+p-1.
\]

We have shown the existence of the function $\widehat{\Lambda}:\widehat
{\mathbf{T}}_{\mu}^{\widehat{x}}\cap\{k^{\prime},k^{\prime}+1,...,k^{\prime
}+p-1\}\longrightarrow\widehat{\mathbf{T}}_{\mu}^{\widehat{x}}\cap
\{k^{\prime\prime},k^{\prime\prime}+1,...,k^{\prime\prime}+p-1\},$
$\widehat{\mathbf{T}}_{\mu}^{\widehat{x}}\cap\{k^{\prime},k^{\prime
}+1,...,k^{\prime}+p-1\}\ni n_{j}\longrightarrow n_{j}+kp\in\widehat
{\mathbf{T}}_{\mu}^{\widehat{x}}\cap\{k^{\prime\prime},k^{\prime\prime
}+1,...,k^{\prime\prime}+p-1\},$ where $k$ may depend on $j$ and it is chosen conveniently.

We show that $\widehat{\Lambda}$ is injective and we suppose for this, against
all reason, that $j_{1},j_{2}\in\{1,...,k_{1}\},j_{1}\neq j_{2},$ $k_{1}%
,k_{2}\in\mathbf{N}$ exist such that $n_{j_{1}}+k_{1}p=n_{j_{2}}+k_{2}p $
where, without loss of generality, we have%
\[
k^{\prime}\leq n_{j_{1}}<n_{j_{2}}\leq k^{\prime}+p-1.
\]
On one hand we have $n_{j_{2}}-n_{j_{1}}=(k_{1}-k_{2})p\in\{p,2p,3p,...\},$
thus $n_{j_{2}}-n_{j_{1}}\geq p$ and on the other hand we obtain $n_{j_{2}%
}-n_{j_{1}}\leq k^{\prime}+p-1-k^{\prime},$ thus $n_{j_{2}}-n_{j_{1}}\leq
p-1.$ The contradiction that we have obtained completes the proof that
$\widehat{\Lambda}$ is injective.

We show that $\widehat{\Lambda}$ is surjective and let $j^{\prime}%
\in\{1,...,p_{1}\}$ arbitrary. The fact that exactly one term of the sequence
$n_{j^{\prime}}^{\prime},n_{j^{\prime}}^{\prime}-p,n_{j^{\prime}}^{\prime
}-2p,...$ belongs to $\widehat{\mathbf{T}}_{\mu}^{\widehat{x}}\cap\{k^{\prime
},k^{\prime}+1,...,k^{\prime}+p-1\}$ is proved similarly with the proof of
existence of $\widehat{\Lambda}$ and let this term be $n_{j^{\prime}}^{\prime
}-\overline{k}p=n_{j}.$ Obviously $\widehat{\Lambda}(n_{j})=n_{j^{\prime}%
}^{\prime}.$

It has resulted that $\widehat{\Lambda}$ is bijective and $k_{1}=p_{1}.$

We prove $\underset{k\in\mathbf{N}}{%
{\displaystyle\bigcup}
}\{n_{1}+kp,n_{2}+kp,...,n_{k_{1}}+kp\}\subset\underset{z\in\mathbf{Z}}{%
{\displaystyle\bigcup}
}\{n_{1}^{\prime}+zp,n_{2}^{\prime}+zp,...,n_{k_{1}}^{\prime}+zp\}\cap
\{k^{\prime},k^{\prime}+1,k^{\prime}+2,...\}$ and let $\widetilde{k}%
\in\underset{k\in\mathbf{N}}{%
{\displaystyle\bigcup}
}\{n_{1}+kp,n_{2}+kp,...,n_{k_{1}}+kp\}$ arbitrary. Some $k\in\mathbf{N}$ and
some $j\in\{1,...,k_{1}\}$ exist with $\widetilde{k}=n_{j}+kp.$ But $kp\geq0$
and $\widetilde{k}\geq n_{j}\geq k^{\prime}.$ Some $j^{\prime}\in
\{1,...,k_{1}\}$ and some $\overline{k}\in\mathbf{N}$ exist such that
$\widehat{\Lambda}(n_{j})=n_{j}+\overline{k}p=n_{j^{\prime}}^{\prime},$ in
other words $\widetilde{k}=n_{j^{\prime}}^{\prime}-\overline{k}p+kp.$ We have
proved that $\widetilde{k}\in\underset{z\in\mathbf{Z}}{%
{\displaystyle\bigcup}
}\{n_{1}^{\prime}+zp,n_{2}^{\prime}+zp,...,n_{k_{1}}^{\prime}+zp\}\cap
\{k^{\prime},k^{\prime}+1,k^{\prime}+2,...\}.$

We prove that $\underset{z\in\mathbf{Z}}{%
{\displaystyle\bigcup}
}\{n_{1}^{\prime}+zp,n_{2}^{\prime}+zp,...,n_{k_{1}}^{\prime}+zp\}\cap
\{k^{\prime},k^{\prime}+1,k^{\prime}+2,...\}\subset\underset{k\in\mathbf{N}}{%
{\displaystyle\bigcup}
}\{n_{1}+kp,n_{2}+kp,...,n_{k_{1}}+kp\}$ and let for this $\widetilde{k}%
\in\underset{z\in\mathbf{Z}}{%
{\displaystyle\bigcup}
}\{n_{1}^{\prime}+zp,n_{2}^{\prime}+zp,...,n_{k_{1}}^{\prime}+zp\}\cap
\{k^{\prime},k^{\prime}+1,k^{\prime}+2,...\}$ arbitrary. Some $j^{\prime}%
\in\{1,...,k_{1}\}$ and some $z\in\mathbf{Z}$ exist with $\widetilde
{k}=n_{j^{\prime}}^{\prime}+zp\geq k^{\prime}.$ We have the existence of
$j\in\{1,...,k_{1}\}$ and $\overline{k}\in\mathbf{N}$ with $n_{j^{\prime}%
}^{\prime}-\overline{k}p=\widehat{\Lambda}^{-1}(n_{j^{\prime}}^{\prime}%
)=n_{j},$ thus $\widetilde{k}=n_{j}+(\overline{k}+z)p.$ As $n_{j}-p\leq
k^{\prime}-1,$ the condition $n_{j}+(\overline{k}+z)p\geq k^{\prime}$ implies
$\overline{k}+z\in\mathbf{N},$ in other words $\widetilde{k}\in\underset
{k\in\mathbf{N}}{%
{\displaystyle\bigcup}
}\{n_{1}+kp,n_{2}+kp,...,n_{k_{1}}+kp\}.$

b) We take an arbitrary $t^{\prime\prime}\geq t^{\prime}.$ As far as $\mu
\in\omega(x),$ we have that $\mathbf{T}_{\mu}^{x}$ is superiorly unbounded and
consequently $\mathbf{T}_{\mu}^{x}\cap\lbrack t^{\prime},\infty)\neq
\varnothing.$ We can apply Theorem \ref{Lem1}, page \pageref{Lem1} and we get
$\mathbf{T}_{\mu}^{x}\cap\lbrack t^{\prime},t^{\prime}+T)\neq\varnothing
,\mathbf{T}_{\mu}^{x}\cap\lbrack t^{\prime\prime},t^{\prime\prime}%
+T)\neq\varnothing,$ hence the definitions (\ref{p45}) of the disjoint
intervals $[a_{1},b_{1}),[a_{2},b_{2}),...,[a_{k_{1}},b_{k_{1}})$ and
(\ref{p46}) of the disjoint intervals $[a_{1}^{\prime},b_{1}^{\prime}%
),[a_{2}^{\prime},b_{2}^{\prime}),...,[a_{p_{1}}^{\prime},b_{p_{1}}^{\prime})$
make sense.

Let $t\in\mathbf{T}_{\mu}^{x}\cap\lbrack t^{\prime},t^{\prime}+T)$
arbitrary$.$ We have from (\ref{p44}) that $\forall k\in\mathbf{N}%
,t+kT\in\mathbf{T}_{\mu}^{x}$ and we claim that exactly one term of the
sequence $t,t+T,t+2T,...$ belongs to $\mathbf{T}_{\mu}^{x}\cap\lbrack
t^{\prime\prime},t^{\prime\prime}+T).$ This is proved similarly with a), the
supposition that no term of the sequence belongs to $\mathbf{T}_{\mu}^{x}%
\cap\lbrack t^{\prime\prime},t^{\prime\prime}+T)$ and the supposition that
several terms of the sequence belong to $\mathbf{T}_{\mu}^{x}\cap\lbrack
t^{\prime\prime},t^{\prime\prime}+T)$ give contradictions. The reasoning shows
the existence of a function $\Lambda:\mathbf{T}_{\mu}^{x}\cap\lbrack
t^{\prime},t^{\prime}+T)\longrightarrow\mathbf{T}_{\mu}^{x}\cap\lbrack
t^{\prime\prime},t^{\prime\prime}+T),\mathbf{T}_{\mu}^{x}\cap\lbrack
t^{\prime},t^{\prime}+T)\ni t\longrightarrow t+kT\in\mathbf{T}_{\mu}^{x}%
\cap\lbrack t^{\prime\prime},t^{\prime\prime}+T),$ where $k$ may depend on $t$
and it is chosen conveniently.

We prove that $\Lambda$ is injective and let us suppose against all reason
that $t_{1},t_{2}\in\mathbf{T}_{\mu}^{x}\cap\lbrack t^{\prime},t^{\prime
}+T),k_{1},k_{2}\in\mathbf{N}$ exist with the property that $t_{1}%
+k_{1}T=t_{2}+k_{2}T.$ We can suppose without loosing the generality that%
\[
t^{\prime}\leq t_{1}<t_{2}<t^{\prime}+T.
\]
On one hand $t_{2}-t_{1}=(k_{1}-k_{2})T\in\{T,2T,3T,...\},$ thus $t_{2}%
-t_{1}\geq T$ and on the other hand $t_{2}-t_{1}<t^{\prime}+T-t^{\prime}=T,$ contradiction.

The proof that $\Lambda$ is surjective is made by taking $\widetilde{t}%
\in\mathbf{T}_{\mu}^{x}\cap\lbrack t^{\prime\prime},t^{\prime\prime}+T)$
arbitrarily and showing, by making use of (\ref{p44}), that exactly one term
of the sequence $\widetilde{t},\widetilde{t}-T,\widetilde{t}-2T,...$ belongs
to $\mathbf{T}_{\mu}^{x}\cap\lbrack t^{\prime},t^{\prime}+T).$

The conclusion is that $\Lambda$ is bijective.

We prove $\underset{k\in\mathbf{N}}{%
{\displaystyle\bigcup}
}([a_{1}+kT,b_{1}+kT)\cup\lbrack a_{2}+kT,b_{2}+kT)\cup...\cup\lbrack
a_{k_{1}}+kT,b_{k_{1}}+kT))\subset\underset{z\in\mathbf{Z}}{%
{\displaystyle\bigcup}
}([a_{1}^{\prime}+zT,b_{1}^{\prime}+zT)\cup\lbrack a_{2}^{\prime}%
+zT,b_{2}^{\prime}+zT)\cup...\cup\lbrack a_{p_{1}}^{\prime}+zT,b_{p_{1}%
}^{\prime}+zT))\cap\lbrack t^{\prime},\infty)$ and let $\widetilde{t}%
\in\underset{k\in\mathbf{N}}{%
{\displaystyle\bigcup}
}([a_{1}+kT,b_{1}+kT)\cup\lbrack a_{2}+kT,b_{2}+kT)\cup...\cup\lbrack
a_{k_{1}}+kT,b_{k_{1}}+kT))$ arbitrary. We get the existence of $k\in
\mathbf{N}$ and $j\in\{1,...,k_{1}\}$ such that $\widetilde{t}\in\lbrack
a_{j}+kT,b_{j}+kT),$ thus $\widetilde{t}-kT\in\lbrack a_{j},b_{j})$ and
$\widetilde{t}-kT\geq t^{\prime}.$ Furthermore, a unique $\overline{k}%
\in\mathbf{N}$ exists with the property $\Lambda(\widetilde{t}-kT)=\widetilde
{t}-kT+\overline{k}T\in\mathbf{T}_{\mu}^{x}\cap\lbrack t^{\prime\prime
},t^{\prime\prime}+T)$ and a unique $j^{\prime}\in\{1,...,p_{1}\}$ exists also
with $\widetilde{t}-kT+\overline{k}T\in\lbrack a_{j^{\prime}}^{\prime
},b_{j^{\prime}}^{\prime}),$ i.e. $\widetilde{t}\in\lbrack a_{j^{\prime}%
}^{\prime}+(k-\overline{k})T,b_{j^{\prime}}^{\prime}+(k-\overline{k})T).$ As
$k-\overline{k}\in\mathbf{Z}$ and $\widetilde{t}\geq a_{j}+kT\geq a_{j}\geq
t^{\prime},$ it has resulted that $\widetilde{t}\in\underset{z\in\mathbf{Z}}{%
{\displaystyle\bigcup}
}([a_{1}^{\prime}+zT,b_{1}^{\prime}+zT)\cup\lbrack a_{2}^{\prime}%
+zT,b_{2}^{\prime}+zT)\cup...\cup\lbrack a_{p_{1}}^{\prime}+zT,b_{p_{1}%
}^{\prime}+zT))\cap\lbrack t^{\prime},\infty).$

We prove the inclusion $\underset{z\in\mathbf{Z}}{%
{\displaystyle\bigcup}
}([a_{1}^{\prime}+zT,b_{1}^{\prime}+zT)\cup\lbrack a_{2}^{\prime}%
+zT,b_{2}^{\prime}+zT)\cup...\cup\lbrack a_{p_{1}}^{\prime}+zT,b_{p_{1}%
}^{\prime}+zT))\cap\lbrack t^{\prime},\infty)\subset\underset{k\in\mathbf{N}}{%
{\displaystyle\bigcup}
}([a_{1}+kT,b_{1}+kT)\cup\lbrack a_{2}+kT,b_{2}+kT)\cup...\cup\lbrack
a_{k_{1}}+kT,b_{k_{1}}+kT))$ and let $\widetilde{t}\in\underset{z\in
\mathbf{Z}}{%
{\displaystyle\bigcup}
}([a_{1}^{\prime}+zT,b_{1}^{\prime}+zT)\cup\lbrack a_{2}^{\prime}%
+zT,b_{2}^{\prime}+zT)\cup...\cup\lbrack a_{p_{1}}^{\prime}+zT,b_{p_{1}%
}^{\prime}+zT))\cap\lbrack t^{\prime},\infty)$ arbitrary. We have the
existence of $z\in\mathbf{Z}$ and $j^{\prime}\in\{1,...,p_{1}\}$ such that
$\widetilde{t}\in\lbrack a_{j^{\prime}}^{\prime}+zT,b_{j^{\prime}}^{\prime
}+zT),\widetilde{t}\geq t^{\prime}.$ Thus $\widetilde{t}-zT\in\lbrack
a_{j^{\prime}}^{\prime},b_{j^{\prime}}^{\prime}).$ Some $\overline{k}%
\in\mathbf{N}$ exists with $\Lambda^{-1}(\widetilde{t}-zT)=\widetilde
{t}-zT-\overline{k}T\in\mathbf{T}_{\mu}^{x}\cap\lbrack t^{\prime},t^{\prime
}+T),$ thus $j\in\{1,...,k_{1}\}$ exists with $\widetilde{t}-zT-\overline
{k}T\in\lbrack a_{j},b_{j})$ and finally $\widetilde{t}\in\lbrack
a_{j}+(z+\overline{k})T,b_{j}+(z+\overline{k})T).$ Because $b_{j}-T<t^{\prime
}$ and $\widetilde{t}\geq t^{\prime},$ we infer that $z+\overline{k}%
\in\mathbf{N}.$ It has resulted that $\widetilde{t}\in\underset{k\in
\mathbf{N}}{%
{\displaystyle\bigcup}
}([a_{1}+kT,b_{1}+kT)\cup\lbrack a_{2}+kT,b_{2}+kT)\cup...\cup\lbrack
a_{k_{1}}+kT,b_{k_{1}}+kT)).$
\end{proof}

\begin{example}
We show that in the previous Lemma, item b) we have in general $k_{1}\neq
p_{1}$ (unlike item a)) and we consider $x\in S^{(1)},$%
\[
x(t)=\chi_{\lbrack0,1)}(t)\oplus\chi_{\lbrack2,3)}(t)\oplus\chi_{\lbrack
4,5)}(t)\oplus...
\]
In this case we have $\mu=1,T=2,\mathbf{T}_{\mu}^{x}=[0,1)\cup\lbrack
2,3)\cup\lbrack4,5)\cup...$ and

- for $t\in\lbrack-1,0]\cup\lbrack1,2]\cup\lbrack3,4]\cup...$ we have
$k_{1}=1,$ for example $t=-0.5$ when $[a_{1},b_{1})=[0,1);$

- for $t\in(0,1)\cup(2,3)\cup(4,5)\cup...$ we have $k_{1}=2,$ for example
$t=0.5$ when $[a_{1},b_{1})\cup\lbrack a_{2},b_{2})=[0.5,1)\cup\lbrack2,2.5).$
\end{example}

\begin{lemma}
\label{Lem39}For any $a,b\in\mathbf{R}$ with $a<b$ and any $\rho\in\Pi
_{n}^{\prime},$ the set $\{t|t\in\lbrack a,b),\rho(t)\neq(0,...,0)\}$ is finite.
\end{lemma}

\begin{proof}
Let $\alpha\in\widehat{\Pi}_{n}^{\prime}$ and $(t_{k})\in Seq$ such that%
\[
\rho(t)=\alpha^{0}\cdot\chi_{\{t_{0}\}}(t)\oplus\alpha^{1}\cdot\chi
_{\{t_{1}\}}(t)\oplus...\oplus\alpha^{k}\cdot\chi_{\{t_{k}\}}(t)\oplus...
\]

We notice first that the set $\{t_{k}|k\in\mathbf{N},t_{k}\in\lbrack a,b)\} $
is finite (we consider that the empty set is finite). Indeed, if we suppose
against all reason that $\{t_{k}|k\in\mathbf{N},t_{k}\in\lbrack a,b)\}$ is
infinite, as $(t_{k})$ is strictly increasing, we infer that $\forall
k\in\mathbf{N},t_{k}<b,$ contradiction with the fact that $(t_{k})$ is
superiorly unbounded.

We can infer now, as far as the right hand set of the following inclusion%
\[
\{t|t\in\lbrack a,b),\rho(t)\neq(0,...,0)\}\subset\{t_{k}|k\in\mathbf{N}%
,t_{k}\in\lbrack a,b)\}
\]
is finite, that the left hand set is finite too.
\end{proof}

\end{document}